\documentclass[11pt]{article} 
\usepackage[utf8]{inputenc} 
\usepackage[margin=1in]{geometry} 
\usepackage{graphicx} 

\usepackage{booktabs} 
\usepackage{array} 
\usepackage{paralist} 
\usepackage{verbatim} 
\usepackage{mathrsfs}
\usepackage{amssymb}
\usepackage{amsthm}
\usepackage{amsmath,amsfonts,amssymb}
\usepackage{esint}
\usepackage{graphics}
\usepackage{enumerate}
\usepackage{mathtools}
\usepackage{xfrac}
\usepackage{nicefrac}
\usepackage{subcaption}
\usepackage[normalem]{ulem}
\usepackage{cancel}
\usepackage{enumitem}
\usepackage{fancyvrb}

\let\olddiamond\diamond
\let\oldsquare\square 

\usepackage{mathabx}

\renewcommand{\square}{\oldsquare}
\renewcommand{\diamond}{\olddiamond}

\usepackage[usenames,dvipsnames]{xcolor}
\usepackage[colorlinks=true, pdfstartview=FitV, linkcolor=blue, citecolor=blue, urlcolor=blue]{hyperref}
\usepackage[normalem]{ulem}

\usepackage{tikz}
\usetikzlibrary{calc}
\usepackage{pgf}
\usetikzlibrary{external}
\numberwithin{equation}{section}
\numberwithin{figure}{section}

\newtheorem{theorem}{Theorem}[section]

\newtheorem{corollary}[theorem]{Corollary}
\newtheorem{proposition}[theorem]{Proposition}
\newtheorem{lemma}[theorem]{Lemma}
\newtheorem{theoremA}{Theorem}

\theoremstyle{definition}
\newtheorem{definition}[theorem]{Definition}

\newtheorem{remark}[theorem]{Remark}

\newtheorem{assertion}{Assertion}

\newcommand*{\supp}{\ensuremath{\mathrm{supp\,}}}
\newcommand*{\Id}{\ensuremath{\mathrm{I}_d}}

\newcommand*{\Itwod}{\ensuremath{\mathrm{I}_{2d}}}
\newcommand{\nf}{\nicefrac}
\newcommand*{\tr}{\ensuremath{\mathrm{trace\,}}}

\newcommand*{\N}{\ensuremath{\mathbb{N}}}

\newcommand*{\Z}{\ensuremath{\mathbb{Z}}}

\newcommand*{\R}{\ensuremath{\mathbb{R}}}
\newcommand*{\Zd}{\ensuremath{\mathbb{Z}^d}}
\newcommand*{\Rd}{\ensuremath{\mathbb{R}^d}}

\newcommand{\eps}{\varepsilon}

\renewcommand*{\tilde}{\widetilde}

\renewcommand{\P}{\ensuremath{\mathbb{P}}}
\renewcommand{\O}{\ensuremath{\mathcal{O}}}
\newcommand{\X}{\ensuremath{\mathcal{X}}}

\renewcommand{\b}{\ensuremath{\mathbf{b}}}

\newcommand{\qand}{\quad \mbox{and} \quad }

\newcommand{\sym}{\mathrm{sym}}
\renewcommand{\skew}{\mathrm{skew}}
\newcommand{\f}{\mathbf{f}}

\newcommand{\g}{\mathbf{g}}
\newcommand{\h}{\mathbf{h}}
\newcommand{\s}{\mathbf{s}}
\renewcommand{\tt}{\mathbf{t}}

\newcommand{\Lat}{\ensuremath{\mathbb{L}_0}}
\newcommand{\ep}{\eps}

\newcommand{\A}{\mathcal{A}}

\renewcommand{\S}{\mathcal{S}}
\newcommand{\G}{\mathbf{G}}
\DeclareMathOperator{\dist}{dist}
\DeclareMathOperator{\diam}{diam}
\DeclareMathOperator{\intr}{int}

\DeclareMathOperator*{\var}{var}\DeclareMathOperator*{\cov}{cov}

\newcommand{\E}{\mathbb{E}}

\newcommand{\XiDet}{\Xi}

\DeclareSymbolFont{boldoperators}{OT1}{cmr}{bx}{n}
\SetSymbolFont{boldoperators}{bold}{OT1}{cmr}{bx}{n}
\usepackage{accents}
\newcommand\thickbar[1]{\accentset{\rule{.45em}{.6pt}}{#1}}
\renewcommand{\bar}{\thickbar}
\renewcommand{\a}{\mathbf{a}}
\renewcommand{\k}{\mathbf{k}}
\newcommand{\m}{\mathbf{m}}

\newcommand{\ahom}{\bar{\a}}
\newcommand{\bhom}{\bar{\mathbf{b}}}
\newcommand{\shom}{\bar{\mathbf{s}}}
\newcommand{\khom}{\bar{\mathbf{k}}}

\newcommand{\hhom}{\bar{\mathbf{h}}}
\newcommand{\thom}{\bar{\mathbf{t}}}
\newcommand{\gammafun}{\mathbf{\Gamma}}

\newcommand{\Lsolo}{L^2_{\a,\mathrm{sol,0}}} 
 
\newcommand{\Lpoto}{L^2_{\a,\mathrm{pot,0}}} 
\newcommand{\sol}{\mathrm{sol}}
\newcommand{\pot}{\mathrm{pot}}
\newcommand{\CFS}{\mathsf{CFS}}

\newcommand{\bfA}{\mathbf{A}}
\newcommand{\bfAhom}{\overline{\mathbf{A}}}
\newcommand{\bfB}{\mathbf{B}}

\newcommand{\bfJ}{\mathbf{J}}
\newcommand{\bfE}{\mathbf{E}_0}
\newcommand{\uhom}{u_{\mathrm{hom}}}

\makeatletter 
\newcommand{\negphantom}{\v@true\h@true\negph@nt} 
\newcommand{\neghphantom}{\v@false\h@true\negph@nt} 
\newcommand{\negph@nt}{\ifmmode\expandafter\mathpalette 
  \expandafter\mathnegph@nt\else\expandafter\makenegph@nt\fi} 
\newcommand{\makenegph@nt}[1]{%
  \setbox\z@\hbox{\color@begingroup#1\color@endgroup}\finnegph@nt} 
\newcommand{\finnegph@nt}{%
  \setbox\tw@\null 
  \ifv@ \ht\tw@\ht\z@\dp\tw@\dp\z@\fi \ifh@\wd\tw@-\wd\z@\fi\box\tw@} 
\newcommand{\mathnegph@nt}[2]{%
  \setbox\z@\hbox{$\m@th #1{#2}$}\finnegph@nt} 
\makeatother

\newcommand{\Hminushat}[1]{\hat{\phantom{H}}\negphantom{H}{H}^{#1}}

\newcommand{\Hminusuls}[1]{\hat{\phantom{H}}\negphantom{H}\underline{H}^{#1}}

\newcommand{\Wminusul}[2]{\hat{\phantom{W}}\negphantom{W}\underline{W}^{#1,#2}}

\newcommand{\Besov}[3]{\mathring{\phantom{B}}\negphantom{B}\underline{B}^{#1}_{#2,#3}}

\newcommand{\Besovnoul}[3]{\mathring{\phantom{B}}\negphantom{B}{B}^{#1}_{#2,#3}}

\newcommand{\Bhatminusul}[3]{\hat{\phantom{B}}\negphantom{B}\underline{B}^{#1}_{#2,#3}}

\newcommand{\cs}[1][s]{\mathfrak{c}_{#1}}
\newcommand{\css}[1]{\mathfrak{c}_{#1}} 

\def\Xint#1{\mathchoice
{\XXint\displaystyle\textstyle{#1}}%
{\XXint\textstyle\scriptstyle{#1}}%
{\XXint\scriptstyle\scriptscriptstyle{#1}}%
{\XXint\scriptscriptstyle\scriptscriptstyle{#1}}%
\!\int}
\def\XXint#1#2#3{{\setbox0=\hbox{$#1{#2#3}{\int}$}
\vcenter{\hbox{$#2#3$}}\kern-.5\wd0}}

\def\fint{\Xint-}

\newcommand{\avsum}{\mathop{\mathpalette\avsuminner\relax}\displaylimits}

\makeatletter
\newcommand\avsuminner[2]{%
  {\sbox0{$\m@th#1\sum$}%
   \vphantom{\usebox0}%
   \ooalign{%
     \hidewidth
     \smash{\,\rule[.23em]{8.8pt}{1.1pt} \relax}%
     \hidewidth\cr
     $\m@th#1\sum$\cr
   }%
  }%
}
\makeatother

\makeatletter
\newcommand\avsuminnerr[2]{%
  {\sbox0{$\m@th#1\sum$}%
   \vphantom{\usebox0}%
   \ooalign{%
     \hidewidth
     \smash{\,\rule[.23em]{6pt}{0.7pt} \relax}%
     \hidewidth\cr
     $\m@th#1\sum$\cr
   }%
  }%
}
\makeatother

\let\originalleft\left
\let\originalright\right
\renewcommand{\left}{\mathopen{}\mathclose\bgroup\originalleft}
\renewcommand{\right}{\aftergroup\egroup\originalright}

\newcommand{\cu}{\square}

\newcommand{\cus}{\mathord{\lozenge}}

\newcommand{\indc}{{\boldsymbol{1}}}
\renewcommand{\hat}{\widehat}

\usepackage{titlesec}

\newcommand{\addperiod}[1]{#1.}
\titleformat{\section}
   {\centering\normalfont\bfseries\Large}{\thesection.}{0.5em}{}
\titleformat{\subsection}[runin]
  {\normalfont\bfseries\large}
  {\thesubsection.}
  {0.5em}
  {\addperiod}
\titleformat{\subsubsection}[runin]
  {\normalfont\bfseries}
  {\thesubsubsection.}
  {0.5em}
  {\addperiod}
\titleformat*{\subsubsection}{\normalfont\itshape}
\titleformat*{\paragraph}{\bfseries}
\titleformat*{\subparagraph}{\large\bfseries}

\title{Renormalization group and elliptic homogenization in high contrast}

\author{Scott Armstrong
\thanks{Courant Institute of Mathematical Sciences, New York University.
{\footnotesize \href{mailto:scotta@cims.nyu.edu}{scotta@cims.nyu.edu}.}
}
\and
Tuomo Kuusi
\thanks{Department of Mathematics and Statistics, University of Helsinki.
{\footnotesize \href{mailto:tuomo.kuusi@helsinki.fi}{tuomo.kuusi@helsinki.fi}.}
}
}

\date{September 5, 2025}

\usepackage[nottoc,notlot,notlof]{tocbibind}
\usepackage{fancyhdr}
\usepackage{extramarks}
\begin{document}

\maketitle

\begin{abstract}
We prove a quantitative estimate for the homogenization length scale in terms of the ellipticity ratio~$\Lambda/\lambda$ of the coefficient field. 
This upper bound applies to high-contrast elliptic equations exhibiting near-critical behavior. Specifically, we show, assuming a suitable decay of correlations, the length scale at which homogenization occurs is at most $\exp(C \log^2(1+\Lambda/\lambda))$. The proof introduces the new concept of coarse-grained ellipticity, which measures the effective ellipticity ratio of the equation---and thus the strength of the disorder---after integrating out smaller scales. By a direct analytic argument, we derive an approximate differential inequality for this coarse-grained ellipticity as a function of the length scale. This approach may be viewed as a rigorous renormalization group argument and provides a quantitative framework for homogenization that can be iteratively applied across an arbitrary number of length scales.
\end{abstract}

\setcounter{tocdepth}{2}
\renewcommand{\baselinestretch}{0.95}\normalsize
\tableofcontents
\renewcommand{\baselinestretch}{1.0}\normalsize

\section{Introduction}

\subsection{Homogenization in high contrast}

We consider the elliptic equation
\begin{equation}
\label{e.pde}
-\nabla \cdot \a(x)\nabla u = 0 \quad \mbox{in} \ U \subseteq \Rd\,,\end{equation}
where~$\a(x)$ is a~$\Zd$--stationary random field in dimension~$d\geq 2$, taking values in~$\R^{d\times d}$. 
Under appropriate ellipticity and decorrelation assumptions on the coefficient field~$\a(x)$, the equation~\eqref{e.pde} homogenizes on large scales. This essentially means that its solutions will be close to those of the equation~$-\nabla \cdot \ahom \nabla \uhom = 0$, for a constant and deterministic matrix~$\ahom$ which depends in a complicated way on the law of~$\a(x)$ and describes the macroscopic behavior of the system. The phenomenon of elliptic homogenization is observed across a wide spectrum of statistical physics problems in which diffusive limits on macroscopic scales are expected to emerge from microscopic disorder. 

\smallskip

Given an error tolerance~$\delta>0$, the \emph{homogenization length scale}~$\X$ is a random variable that is defined, roughly, as the minimal scale above which the relative homogenization error---the ratio of the size of the error~$u-\uhom$ to the size of~$\uhom$---is smaller than~$\delta$. This length scale characterizes the transition beyond which the local fluctuations in the solutions have averaged out, and the macroscopic behavior is dominant. As such, the stochastic moments of~$\X$ measure the extent to which local randomness in the field~$\a(x)$ affects the large-scale properties of the solutions.

\smallskip

In this paper, we establish quantitative upper bounds on the homogenization length scale for very general coefficient fields. These estimates depend explicitly on the ellipticity of the field~$\a(x)$, as well as its rate of decorrelation, but otherwise require no structural assumption on the law of~$\a(x)$. 
Of particular interest is the dependence of~$\X$ on the ellipticity of the coefficient field. We obtain estimates which are completely new even in the special case that the field~$\a(x)$ satisfies (almost surely) a uniform ellipticity condition of the form
\begin{equation}
\label{e.UE.intro}
\exists \lambda,\Lambda>0 \,, \quad
\lambda \leq \Lambda\,, \quad
\lambda |e|^2 \leq e \cdot \a (x)e
\qand
\Lambda^{-1} |e|^2 \leq e \cdot \a^{-1} (x)e \,,
\quad \forall x,e\in\Rd\,.
\end{equation}
Any upper bound on~$\X$ for general fields satisfying~\eqref{e.UE.intro} must diverge as the \emph{ellipticity ratio}~$\Lambda/\lambda$ tends to infinity---even under the strongest possible mixing assumptions, such as a finite range of dependence. Indeed, as~$\Lambda/\lambda$ increases, the dependence of the solutions on the coefficient field becomes more singular, necessitating, in general, a larger scale separation for the macroscopic behavior to manifest. This divergence of the homogenization length scale mirrors critical phenomena widely observed in statistical physics, in which systems nearing critical points exhibit diverging correlation lengths and increased sensitivity to external parameters.

\smallskip

An example that exhibits a diverging homogenization length scale, and to which our results apply, is a continuum version of the conductance model at criticality. Consider a Poisson point process on~$\Rd$ with intensity~$\gamma>0$ and let~$A \subseteq\Rd$ be the union of all balls of radius one centered at a point in the point cloud. It is well-known that there is a critical value~$\gamma_c\in (0,\infty)$ such that~$A$ has an infinite connected component (which is necessarily unique) if~$\gamma > \gamma_c$, and no infinite component if~$\gamma < \gamma_c$. This is often referred to as the \emph{continuum percolation model}, and we associate an elliptic coefficient field to it by setting~$\a \coloneqq  \Id \indc_{A} + \lambda \indc_{\Rd \setminus A}$, where~$0< \lambda \ll 1$ is a small parameter. The scalar field~$\a(x)$ has the physical interpretation of the conductivity of a random material at the point~$x$. It satisfies~\eqref{e.UE.intro} with~$(\Lambda,\lambda)=(1,\lambda)$ and so its ellipticity ratio is~$\lambda^{-1}$. Smaller values of~$\lambda$ mean the resistance of the vacant set~$\Rd \setminus A$ is larger, and the flux of the solutions of~\eqref{e.pde} will therefore be more concentrated on the set~$A$. The connectedness (or lack thereof) of the set~$A$ becomes the main driver of large-scale behavior of solutions, and in the limit~$\lambda \to 0$, we should expect the homogenization length scale~$\X$ to be roughly of the same order as the correlation length of the underlying percolation problem.\footnote{Indeed, for the noncritical case~$\gamma\neq\gamma_c$, this can be made rigorous in the limit~$\lambda \to 0$ by a simple soft argument; the  convergence rate would however depend on a lower bound for~$|\gamma-\gamma_c|$.} However, since~$\gamma = \gamma_c$, this is infinite. It is, therefore, natural to wonder how large we should expect~$\X$ to be as a function of~$\lambda^{-1}$. 

\smallskip

The following theorem is the first general quantitative estimate in \emph{high contrast homogenization}. Assuming only the uniform ellipticity condition~\eqref{e.UE.intro} and a unit range of dependence, it provides an upper bound estimate on the homogenization length scale which states roughly that
\begin{equation}
\label{e.X.upperbound}
\X \lesssim \exp \bigl( C \log^2 (1+\Lambda /\lambda) \bigr)
\simeq
\Bigl( \frac{\Lambda}{\lambda} \Bigr)^{C \log (1+\Lambda /\lambda)}
\end{equation}
for a prefactor constant~$C(\delta,d)<\infty$ which depends only on the dimension~$d$ and tolerance~$\delta$. This estimate is a special case of the main results in the paper, presented below in Section~\ref{ss.mainresults}, which apply to more general coefficient fields (very singular and/or degenerate fields are allowed, as are those with much weaker decay of correlations) and provide stronger, more extensive quantitative estimates in their conclusions.

\begin{theoremA}[Quantitative homogenization in high contrast]
\label{t.HC.intro}
Let~$\P$ be the law of a~$\Zd$--stationary random field~$\a(\cdot)$, valued in~$\R^{d\times d}$, such that:
\begin{itemize}
\item $\a(\cdot)$ satisfies the uniform ellipticity condition~\eqref{e.UE.intro} with constants~$0<\lambda\leq1\leq \Lambda<\infty$, $\P$--a.s. 

\item $\a(\cdot)$ has a unit range of dependence. 

\end{itemize}
Let~$\ahom$ denote the corresponding homogenized matrix,~$\shom \coloneqq  \frac12(\ahom+\ahom^t)$ be the symmetric part of~$\ahom$ and~$\overline\lambda$ be the smallest eigenvalue of~$\shom$. 
Denote the family~$\{ E_r \}_{r\geq 0}$ of ellipsoids adapted to~$\shom$ by
\begin{equation*}
E_r  \coloneqq  \bigl \{ x\in \Rd \,:\, x\cdot \shom^{-1} x \leq \overline{\lambda}^{\,-1} r^2 \bigr\} \,.
\end{equation*}
Then, for every~$\delta>0$, there exists~$C(\delta,d)<\infty$ and a nonnegative random variable~$\X$ satisfying 
\begin{equation}
\P \bigl[ \X \geq t \exp \bigl( C \log^2 (1+\Lambda /\lambda) \bigr)  \bigr] 
\leq 
\exp ( - t^d )\,, \quad \forall t \in [1,\infty)\,,
\label{e.thm.A.estimate.for.X}
\end{equation}
such that the following statements are valid:
\begin{itemize}

\item \underline{Homogenization of the Dirichlet problem}. For every~$r\geq \X$,~$f\in L^2(E_r)$ and~$g\in H^1(E_r)$, if we let~$u,\uhom\in H^1(E_r)$ be the solutions of the Dirichlet problems
\begin{equation}
\label{e.BVPs}
\left\{ \begin{aligned} & -\nabla \cdot \a \nabla u =  f & \mbox{in} & \ E_r \,, \\ & u = g &\mbox{on} &\ \partial E_r \,, \end{aligned} \right.
\qquad \mbox{and} \qquad  
\left\{ \begin{aligned} & -\nabla \cdot \ahom \nabla \uhom = f & \mbox{in} & \ E_r \,, \\ & \uhom = g &\mbox{on} &\ \partial E_r \,, \end{aligned} \right.
\end{equation}
then we have the estimate
\begin{equation}
\label{e.intro.homog}
\| u - \uhom \|_{L^2(E_r)} 
\leq 
\delta \bigl( r \| \nabla g \|_{L^2(E_r)} + r^2 \| f \|_{L^2(E_r)} \bigr) 
\,.
\end{equation}

\item \underline{Large-scale $C^{0,1}$ regularity}. 
For every~$R \geq \X$ and solution~$u\in H^1(E_R)$ of the equation
\begin{equation*}
 -\nabla \cdot \a \nabla u = 0 \quad \mbox{in} \ E_R \,, 
\end{equation*}
we have the estimate
\begin{equation*}
\sup_{ r \in [ \X , R ] } 
\fint_{E_r} \nabla u \cdot \a\nabla u 
\leq 
C 
\fint_{E_R} \nabla u \cdot \a\nabla u \,. 
\end{equation*}
\smallskip
\end{itemize}
\end{theoremA}

It is widely believed that the divergences of correlation lengths near critical points, as well as that of other physical quantities, are described by power laws. That is, one expects a correlation length to be a power of the underlying macroscopic control parameters driving the phase transition. 
Typical examples include percolation and Ising models, in which one expects the correlation length~$\xi$ to be of order~$|p-p_c|^{-\nu}$ and~$|T/T_c-1|^{-\nu}$, respectively, where~$p_c$ is the critical percolation probability,~$T_c$ is the critical temperature for the Ising model, and the occupation probability~$p\in (0,1)$ and temperature~$T>0$ are the control parameters. In the continuum percolation model described above, we expect the correlation length~$\xi$ to be of order~$|\gamma-\gamma_c|^{-\nu}$. The value of~$\nu$, as well as that of other critical exponents, is expected to be universal in the sense that it should depend only on the dimension~$d$ and the type of model but not on the particular microscopic details of the model. For instance, the value of~$\nu$ for a bond percolation model on the lattice~$\Zd$ is expected to be the same as the exponent for continuum percolation.

\smallskip

The existence of critical exponents is of great physical interest, and there is a large body of literature devoted to estimating and computing them, with predictions of their exact values available for certain models. This is typically accomplished by heuristic renormalization group arguments, with rigorous results being comparatively rare. Some models for which critical exponents are known rigorously include certain two-dimension models in which conformal invariance can be exploited (such as site percolation on the~$2d$ triangular lattice~\cite{SW}), some exactly integrable models, and in sufficiently large dimensions where mean field methods are applicable (see~\cite{Hara} in the case of bond percolation). For most models, polynomial upper bounds---much less the existence of critical exponents---have not been rigorously demonstrated. For instance, to our knowledge, the best upper bound for the correlation length for Bernoulli bond percolation on the lattice~$\Z^2$ is 
\begin{equation}
\label{e.DKT}
\xi \leq \exp\bigl(C |p-p_c|^{-2} \bigr)\,. 
\end{equation}
This result was proved in~\cite{DKT} for~$\Zd$ in general dimension~$d \geq 2$ and is obviously far from the expected polynomial-type dependence.

\smallskip

In the context of the elliptic equation~\eqref{e.pde}, the ellipticity ratio~$\Lambda/\lambda$ plays the role of a control parameter. So the natural conjecture is that the homogenization length scale should satisfy an upper bound of the form~$\X \lesssim (\Lambda / \lambda)^{\nu}$ for some finite exponent~$\nu(d)<\infty$. Proving such an upper bound estimate is perhaps the most important open problem in quantitative homogenization. Apart from its intrinsic interest, such an estimate would have immediate and important consequences in mathematical physics and probability. Theorem~\ref{t.HC.intro} does not provide such an estimate. 
However, the upper bound in~\eqref{e.X.upperbound} is close to a polynomial-type bound in the sense that the desired fixed exponent~$\nu=C$ is replaced by one that is only logarithmically diverging,~$\nu=C\log (1+\Lambda/\lambda)$. This is evidently significantly better than a bound which is exponential in the control parameter, like~\eqref{e.DKT}, and in fact it is to our knowledge the best rigorous upper bound obtained for a general class of models in low dimensions. Note that our results do not however yield an improvement of~\eqref{e.DKT}, since we are not able to pass quantitative information from the conductance model to the bond percolation model.

\smallskip

To prove Theorem~\ref{t.HC.intro}, we study certain coarse-grained diffusion matrices that are defined at a given scale and in a particular region of space. Based on these objects, we introduce the new concept of \emph{coarse-grained ellipticity}, which is a relaxation of the usual uniform ellipticity ratio. This quantity is a softer and more flexible notion of ellipticity compared to uniform ellipticity and, in particular, permits certain degenerate and unbounded coefficient fields. We view the process of homogenization as a \emph{flow} of the coarse-grained ellipticity, from a possibly very large number at small length scales to unity in the large-scale limit. Indeed, as we show in the paper, the homogenization error can be controlled, in a deterministic way, by the coarse-grained ellipticity. The homogenization length scale~$\X$ is then, roughly, the scale at which the coarse-grained ellipticity is smaller than~$1+\delta$. At the heart of this paper is an analytic argument that obtains a differential inequality for coarse-grained ellipticity as a function of (the logarithm of) the length scale, which then implies the desired bound on the homogenization length scale. This argument is notable for being entirely \emph{renormalizable} in the sense that its outputs (bounds on the coarse-grained ellipticity) are of the same form as its inputs. It is, therefore, possible to iterate it, and indeed, the proof Theorem~\ref{t.HC.intro} relies on such an iteration. 

\smallskip

In a concurrent joint work with Bou-Rabee~\cite{ABK.SD}, we prove a superdiffusive central limit theorem for a Brownian particle in a critically correlated, divergence-free drift. The high contrast homogenization estimates proved in this present paper played an important role in the arguments in~\cite{ABK.SD}. In particular, the fact that the exponent in our estimate for the homogenization length scale in~\eqref{e.X.upperbound} grows logarithmically, rather than like a power of~$\Lambda/\lambda$, is crucial. In the context of that paper, homogenization estimates must be iterated an infinite number of times as a way of formalizing a renormalization group argument. One difficulty encountered is that the ellipticity ratio is also growing as a function of the scale, and there is an apparent ``race'' between homogenization and an accumulation of disorder. A quantitative estimate like~\eqref{e.X.upperbound} is needed to ensure that the randomness at each scale can be integrated out before interacting in an unexpected way with the other, larger scales. Of course, this kind of phenomenon is not unique to this particular problem, and we expect that the methods developed here will find similar applications to other critical models in mathematical physics.

\smallskip

Quantitative estimates for elliptic homogenization problems have been extensively studied in the regime of \emph{moderate} ellipticity contrast in recent years. By this, we mean that the ellipticity ratio~$\Lambda/\lambda$ is held fixed, and the goal, broadly speaking, is to obtain estimates for the homogenization error as a function of the scale separation ratio as it asymptotically approaches infinity. Originating in the pioneering works~\cite{GO1,GO2}, this theory has recently reached maturity, and there are now very detailed and precise quantitative estimates available (an overview and further references can be found in our monograph~\cite{AK.Book}). Each of the several independent approaches to this theory uses constructive arguments and produces constants that are explicitly computable. While the dependence on the ellipticity ratio has been kept implicit in this literature, it is possible to extract an estimate for the homogenization length scale~$\X$ by tracking the dependence of~$\Lambda/\lambda$ through the whole theory. Prior to this work, such a bookkeeping exercise would reveal, in all cases, an exponential upper bound, comparable to~\eqref{e.DKT}, of the form~$\X \lesssim\exp ( C (\Lambda/\lambda)^p)$ for an exponent~$p$ which is at least~$\nicefrac12$ and, we expect, typically much larger (see the discussion below~\eqref{e.naive} for more).

\smallskip 

This paper also develops the first systematic quantitative homogenization theory for a broad class of degenerate equations. Previous works on quantitative homogenization in non-uniformly elliptic settings have addressed finite difference equations on supercritical bond percolation clusters~\cite{AD,D,DG}, domains with random inclusions~\cite{DGl}, certain unbounded coefficients with local integrability conditions~\cite{BellaK,Aya} and, more recently, log-normal coefficients with an integrable covariance function~\cite{CGQ} (cf.~Proposition~\ref{p.example.log-normal}, below). These works extend techniques from the moderate contrast, uniformly elliptic theory while managing specific degeneracies in an ad hoc manner. Each considers only ``far from critical'' cases.\footnote{For example, if the parameter~$\gamma$ in the conductance model mentioned above is either very large or very small, then the system does not display critical behavior and can be analyzed quantitatively for all values of~$\lambda$ by arguments which are comparatively much simpler than those deployed here. Similar comments apply to the papers~\cite{AD,D,DG}, which analyze harmonic functions on supercritical bond percolation clusters on~$\Zd$, for any~$p>p_c$, and prove sharp quantitative homogenization estimates. The constants in these estimates, however, depend on~$p-p_c$ like in~\eqref{e.DKT}, as the correlation length scale for supercritical percolation is used as input to the homogenization argument. As such, the results in these papers should not be seen as analyzing near-critical phenomena.} In contrast, our introduction of the concept of \emph{coarse-grained ellipticity} leads to a quantitative theory covering a broad class of degenerate equations. Note that while the estimates stated in this paper do not give sharp scaling exponents for the homogenization error in the regime of large-scale separation, such estimates can be straightforwardly obtained by combining known techniques from the moderate contrast theory with Theorem~\ref{t.main} below.

\smallskip 

In the following subsection, we state our main results, which are a great deal more general than Theorem~\ref{t.HC.intro}. Since our methods are based on renormalizing the coarse-grained ellipticity, the uniform ellipticity condition can be replaced by the assumption that the coarse-grained ellipticity ratio is bounded on sufficiently large scales. This allows us to consider very general coefficient fields that may be very degenerate and/or singular, including some of those mentioned in the previous paragraph and others that have not been previously analyzed.

\subsection{Coarse-grained ellipticity and the statement of the main results}
\label{ss.mainresults}

In this subsection, we state the main result, which is a general quantitative homogenization result for elliptic equations with high contrast coefficients. 

\smallskip

We begin by introducing some notation. 
Throughout the paper, we denote, for~$m\in\Z$, triadic cubes of size~$3^m$ by
\begin{equation*}
\cu_m  \coloneqq  \Bigl( -\frac12 3^m , \frac12 3^m \Bigr)^d \,.
\end{equation*}
The set of~$m$-by-$n$ matrices with real entries is denoted by~$\R^{m\times n}$. The transpose of a matrix~$A$ is denoted by~$A^t$. 
The sets of~$m$-by-$m$ symmetric and anti-symmetric matrices are denoted, respectively, by 
\begin{equation*}
\R^{d\times d}_{\mathrm{sym}}  \coloneqq 
\bigl\{ A \in \R^{d\times d} \,:\, A = A^t \bigr\}
\quad \mbox{and} \quad 
\R^{d\times d}_{\mathrm{skew}} \coloneqq 
\bigl\{ A \in \R^{d\times d} \,:\, A = -A^t \bigr\}\,.
\end{equation*}
We also define the cone of matrices with positive symmetric parts by
\begin{equation*}
\R^{d \times d}_+  \coloneqq  
\bigl \{ 
A \in \R^{d \times d} \,:\, e \cdot A e \geq 0 \,, \ \forall e \in \Rd
\bigl \}\,.
\end{equation*}
Let~$\Omega$ be the collection of all Lebesgue measurable maps~$\a:\Rd \to\R^{d \times d}_+$
such that, if we split~$\a(\cdot)$ into its symmetric and anti-symmetric parts by writing~$\a(x) = \s(x) + \k(x)$, where we define
\begin{equation}
\label{e.sk.def}
\s(x) \coloneqq  \frac12(\a(x) + \a^t(x)) \in \R^{d\times d}_{\mathrm{sym}}
\quad 
\mbox{and} \quad
\k(x) \coloneqq  \frac12(\a(x) - \a^t(x)) \in \R^{d\times d}_{\mathrm{skew}}
\,, \quad x\in\Rd\,,
\end{equation}
then we have that 
\begin{equation}
\label{e.qual.ellipticity}
\s ,\, \s^{-1} , \, \k^t \s^{-1} \k \in L^1_{\mathrm{loc}}(\Rd; \R^{d\times d})
\,.
\end{equation}
The condition~\eqref{e.qual.ellipticity} represents the minimal qualitative ellipticity we require of our coefficient fields. It is natural in view of the fact that the uniform ellipticity condition~\eqref{e.UE.intro} is equivalent to 
\begin{equation}
\label{e.ellipticity.nonsymm}
\s^{-1} \leq \lambda^{-1}  \Id 
\quad \mbox{and} \quad
\s + \k^t\s^{-1}\k \leq \Lambda  \Id\,.
\end{equation}
Indeed, replacing $e$ by~$\a e$ in the second condition in~\eqref{e.UE.intro}, we see that it is equivalent to~$\a^t \a \leq \Lambda \s$, which is equivalent to~$\Lambda  \Id \geq \a^t \s^{-1} \a  =  
 (\s-\k) \s^{-1} (\s+\k) = \s + \k^t \s^{-1} \k$. In view of~\eqref{e.ellipticity.nonsymm}, we see that~\eqref{e.qual.ellipticity} is a relaxation in ellipticity from~$L^\infty$ to~$L^1$. 

\smallskip
 
We define a~$\sigma$-field~$\mathcal{F}(U)$, for each Borel subset~$U\subseteq \Rd$, as the one generated by the family random variables of the form 
\begin{equation*}
\a \mapsto \int_{\Rd} e' \cdot \a(x) e \varphi(x) \,dx\,
\quad
e,e' \in \Rd, \
\varphi \in C^\infty_c(U)\,.
\end{equation*}
We also denote~$\mathcal{F} \coloneqq  \mathcal{F}(\Rd)$. 
We let~$\{ T_y \, : \, y \in \Rd\}$ denote the group of~$\Rd$ translations acting on~$\Omega$. That is,~$T_y:\Omega \to \Omega$ is given by~$T_y \a = \a(\cdot+y)$. We extend this group action to~$\mathcal{F}$ by defining~$T_y F  \coloneqq  \{ T_y \a \,:\, \a \in F\}$ for~$F\in \mathcal{F}$. 

\smallskip

We consider a probability measure~$\P$  on~$(\Omega,\mathcal{F})$ satisfying the three assumptions~\ref{a.stationarity},~\ref{a.ellipticity} and~\ref{a.CFS} stated below. The first is that~$\P$ is statistically homogeneous.

\begin{enumerate}[label=(\textrm{P\arabic*})]
\setcounter{enumi}{0}
\item \emph{Stationarity with respect to~$\Zd$--translations:}
\label{a.stationarity}
\begin{equation*}
\P \circ T_z = \P, \quad \forall z \in \Zd\,.
\end{equation*}
\end{enumerate}

We turn next to the ellipticity assumption. Conceptually, this assumption requires that the field behaves elliptically only in 
a suitable coarse-grained sense, a much less rigid condition than uniform ellipticity.
We formulate this assumption in terms of the \emph{coarse-grained matrices}, which are objects that play a central role in this paper. 
They are denoted, for each bounded Lipschitz domain~$U\subseteq \Rd$, by~$\s(U;\a)$,~$\s_*(U;\a)$ and~$\k(U;\a)$. These are random elements of~$\R^{2d\times 2d}_{\mathrm{sym}}$ which depend only on the restriction~$\a\vert_U$ of the field~$\a(\cdot)$ to~$U$.
As the notation suggests, we think of the matrices~$\s(U;\a)$ and~$\s_*(U;\a)$ as a coarse-graining of the symmetric part~$\s(\cdot)$ of the field~$\a(\cdot)$, and~$\k(U;\a)$ as a coarse-graining of the anti-symmetric part~$\k(\cdot)$. 
We postpone their definitions to Section~\ref{ss.bfA.def}. 

\smallskip

Motivated by~\eqref{e.ellipticity.nonsymm}, 
we introduce, for each exponent~$s\in [0,1)$ and scale parameter~$m\in\Z$, the \emph{coarse-grained ellipticity constants}~$\lambda_{s,\infty} (\cu_m;\a)$ and~$\Lambda_{s,\infty} (\cu_m;\a)$ as follows:
\begin{equation}
\label{e.coarse.grained.ellipticity.infty}
\left\{
\begin{aligned}
& {\Lambda}_{s,\infty}(\cu_m\,;\a)
 \coloneqq  
\sup_{k\in(-\infty,m]\cap \Z}
3^{-2s(m-k)} 
\max_{z\in 3^k\Zd \cap \cu_m} 
\bigl| (\s+ \k^t\s_*^{-1} \k ) (z+\cu_k; \a) \bigr|
\,, \\  &
{\lambda}_{s,\infty}(\cu_m;\a) 
 \coloneqq 
\biggl( 
\sup_{k\in(-\infty,m]\cap \Z} 
3^{-2s(m-k)} 
\max_{z\in 3^k\Zd \cap \cu_m} 
\bigl| \s_{*}^{-1}(z+\cu_k; \a) \bigr|
\biggr)^{\!-1}
\,.
\end{aligned}
\right.
\end{equation}
Here, and in the rest of the paper,~$|\mathbf{B}|$ denotes the spectral norm of a square matrix~$\mathbf{B}$, that is, the square root of the largest eigenvalue of~$\mathbf{B}^t\mathbf{B}$. If~$s=0$, then~$\Lambda_{s,\infty}(\cu_m)$ and~$\lambda_{s,\infty}(\cu_m)$ coincide precisely with the constants of uniform ellipticity. For~$s>0$, the definitions in~\eqref{e.coarse.grained.ellipticity.infty} become scale dependent as larger cubes are given greater weight and there is a discount for smaller cubes. If the quantities in~\eqref{e.coarse.grained.ellipticity.infty} are positive and finite, then we say that~$\a$ is \emph{coarse-grained elliptic} in~$\cu_m$. 

\smallskip

Our ellipticity assumption says that the field is coarse-grained elliptic, and the coarse-grained ellipticity constants are bounded by a deterministic constant on all scales larger than a sufficiently large (random) scale. The precise statement is: 

\begin{enumerate}[label=(\textrm{P\arabic*})]
\setcounter{enumi}{1}
\item \emph{Coarse-grained ellipticity on large scales.} \label{a.ellipticity}
There exist constants~$0< \lambda_0\leq \Lambda_0<\infty$, an exponent~$\gamma\in [0,1)$, an increasing function~$\Psi_\S:\R_+ \to [1,\infty)$, a constant~$K_{\Psi_\S}\in (1,\infty)$ satisfying the growth condition
\begin{equation}
\label{e.Psi.S.growth}
t \Psi_\S(t) \leq \Psi_\S(K_{\Psi_\S}  t), \quad \forall t \in [1,\infty)\,,
\end{equation}
and a nonnegative random variable~$\S$ satisfying the bound
\begin{equation}
\label{e.S.integrability}
\P \bigl[ \S > t   \bigr]
\leq 
\frac{1}{\Psi_\S(t)}
\,, 
\quad \forall t\in (0,\infty) \,,
\end{equation}
such that, for every~$m\in\Z$, 
\begin{equation}
\label{e.ellipticity}
3^m \geq \S 
\quad\implies \quad
\Lambda_{\nf\gamma2,\infty}(\cu_m;\a) \leq \Lambda_0 
\qand
\lambda_{\nf\gamma2,\infty}(\cu_m;\a) \geq \lambda_0 
\,.
\end{equation}
\end{enumerate}
 
The third and final assumption we need is a quantitative ergodicity condition. The one we use here is formulated in terms of \emph{concentration for sums}, a general mixing condition we previously introduced in~\cite{AK.Book}. 
It is a linear concentration inequality that is flexible enough to contain all of the different quantitative ergodic assumptions used in stochastic homogenization literature but still strong enough to yield optimal quantitative homogenization estimates for the most important examples. In particular, it contains finite range of dependence as well as the assumption of a logarithmic Sobolev inequality as special cases.

\smallskip

To state this condition, we require some terminology. 
We first introduce a measure of the ``sensitivity'' of a random variable with respect to perturbations of the coefficients in a given subset~$U\subseteq \Rd$. 
Given an~$\mathcal{F}$--measurable random variable~$X$ on~$\Omega$, we define another random variable~$|\mathrm{D}_UX|$by setting, for each~$\bfA\in\Omega$, 
\begin{align} 
\label{e.malliavin.mult}
\lefteqn{
\left| \mathrm{D}_U X\right| (\bfA) 
} \ \ & 
\notag \\ & 
 \coloneqq 
\limsup_{t\to 0} 
\frac1{2t} 
\sup 
\Bigl\{ X(\bfA_1) - X(\bfA_2) 
:
\bfA_i \in\Omega, \, 
\bigl| \bfA^{\!-\nicefrac12}\bfA_i \bfA^{\!-\nicefrac12}  - \Itwod \bigr| \leq t \indc_{U} ,\, \forall i\in\{1,2\}
\Bigr\}.
\end{align}
In contrast to the usual notion of ``Malliavin derivative'' which measures sensitivity with respect to uniformly elliptic fields, we measure our perturbations to~$\a(\cdot)$ \emph{multiplicatively} rather than \emph{additively}. Of course, in the uniformly elliptic case, this is a distinction without a difference, but our choice is the more natural one from the point of view of degenerate equations. 

\begin{enumerate}[label=(\textrm{P\arabic*})]
\setcounter{enumi}{2}
\item \emph{Concentrations for sum ($\CFS$).} \label{a.CFS}
There exist~$\beta \in \left[0,1\right)$ and~$\nu\in (0,\frac d2]$ and an increasing function~$\Psi:\R_+ \to [1,\infty)$ and a constant~$K_\Psi\in [3,\infty)$ satisfying the growth condition
\begin{equation}
\label{e.Psi.growth}
t \Psi(t) \leq \Psi(K_\Psi  t), \quad \forall t \in [1,\infty)\,,
\end{equation}
such that,
for every~$m,n\in\N$ with $\beta m < n<m$ and  family $\{ X_z \,:\, z\in 3^n\Zd\cap \cu_m\}$ of random variables satisfying, for every~$z\in3^n \Zd \cap \cu_m$,
\begin{equation*} 
\left\{
\begin{aligned}
& \E\left[ X_z \right]=0 \,, 
\\ & 
\left| X_z \right| \leq 1\,,
\\ &   
\left| \mathrm{D}_{z+\cu_n}  X_z\right| \leq 1 \,,
\\
& X_z \ \ \mbox{is $\mathcal{F}(z+\cu_n)$--measurable} \,, \\
\end{aligned}
\right.
\end{equation*}
we have the estimate 
\begin{equation} 
\label{e.CFS}
\P \Biggl[ 
\biggl| \,
\avsum_{z\in 3^n\Zd\cap \cu_m}  X_z
\biggr| 
\geq 
t3^{-\nu(m-n)}
\Biggr]
\leq
\frac{1}{\Psi(t)}
\,, \quad \forall t\in [1,\infty) \,.
\end{equation}
 \end{enumerate} 
The mixing condition~\ref{a.CFS} is discussed further in~\cite[Chapter 3]{AK.Book} as well as in Section~\ref{ss.examples}, below, where we also give some explicit examples satisfying it. We remark that, in the case of finite range of dependence or LSI, the assumption~\ref{a.CFS} is satisfied with~$(\beta,\nu)=(0,\frac d2)$ and~$\Psi(t) = \exp( c t^2)$ for some constant~$c(d)>0$: see~\cite[Chapter 3]{AK.Book}. 

\smallskip

The following is the main result of the paper. It gives an explicit estimate for the length scale at which we first see homogenization. Some of the notation appearing in the statement is defined below. For instance, the space~$H^1_{\a}$ is defined below in Section~\ref{ss.Sob}, see~\eqref{e.H1s} and~\eqref{e.H.one.a}. The underlines in the norms~$\| \cdot \|_{\underline{L}^2}$ and~$\| \cdot \|_{\underline{H}^{-1}}$ denote volume normalization; see~\eqref{e.volume.normalize} and~\eqref{e.volume.normalize.neg}. The fractional Sobolev space~$H^s(U)$ and its dual space~$H^{-s}(U)$ and their corresponding norms are defined in Section~\ref{ss.notation}: see~\eqref{e.fractional.Sob.def} and~\eqref{e.volume.normalize.neg}.   

\smallskip

A summary of the main ideas in the proof of the following theorem is given in Section~\ref{ss.ideas}.

\begin{theoremA}[Homogenization in high contrast]
\label{t.main}
Assume that~$\P$ is a probability measure on~$(\Omega,\mathcal{F})$ satisfying assumptions~\ref{a.stationarity},~\ref{a.ellipticity} and~\ref{a.CFS} above, with~$\nu > \gamma$. 
There exists~$C(\nu, \gamma, \beta, d)<\infty$, an exponent~$\kappa (\nu, \gamma, \beta, d) >0$, a nonnegative random variable~$\X$ and a matrix~$\ahom\in \R^{d \times d}_+$ such that, if we let~$\shom\coloneqq \frac12(\ahom+\ahom^t)$,~$\khom\coloneqq \ahom-\shom$ and~$\{ E_r \}_{r\geq 0}$ be as in Theorem~\ref{t.HC.intro} and define a length scale 
\begin{equation}
L \coloneqq  \exp\biggl( C\log (1+\Pi K_\Psi K_{\Psi_\S} )\log \Bigl(1+\frac{\Lambda_0}{\lambda_0}\Bigr) \biggr)\,,
\label{e.length.scale.L}
\end{equation}
then the following statements are valid:

\begin{itemize}

\item \underline{Estimate of the homogenization length scale}.
For every~$t\geq 1$ and for~$\mu  \coloneqq  (\nu-\gamma )(1-\beta)$,
\begin{equation}
\P \bigl[ \X \geq C L t \bigr] 
\leq 
\frac{1}{\Psi(t^{\mu})}
+
\frac{1}{\Psi_\S(CLt)}
\,.
\label{e.thm.B.X.integrability}
\end{equation}

\item \underline{Homogenization of the Dirichlet problem}. 
For every~$s\in (0,\nf12)$ and Lipschitz domain~$U \subseteq E_1$, there exists a constant~$C_0(U,s,d)<\infty$, depending only on~$s$,~$d$ and the Lipschitz character of~$\shom^{-\nf12}U$, such that, 
for every~$\ep >0$ with~$\ep^{-1} \geq \X$ and~$g \in W^{1,\infty}(U) \cap H^{1+s}(U)$, if we denote~$\a^\ep\coloneqq \a (\tfrac \cdot \ep)$, then the unique solutions~$h \in g+H^1_0(U)$ and~$u^\ep  \in g+H^1_{\a^\ep ,0}(U)$ of the Dirichlet problems
\begin{equation*}
\left\{
\begin{aligned}
& -\nabla \cdot \ahom \nabla h = 0 & \mbox{in} & \ U \,, 
\\ 
& h = g & \mbox{on} & \ \partial U 
\,,
\end{aligned}
\right.
\qquad \mbox{and} \qquad
\left\{
\begin{aligned}
& -\nabla \cdot \a^\ep \nabla u^\ep = 0 & \mbox{in} & \ U \,, 
\\ 
& u^\ep  = g & \mbox{on} & \ \partial U 
\,,
\end{aligned}
\right.
\end{equation*}
then we have the estimate
\begin{equation}
\| \shom^{\nf12} ( \nabla u - \nabla h ) \|_{H^{-s} (U)}
+
\| \shom^{-\nf12} ( \a^\ep \nabla u - \ahom \nabla h ) \|_{H^{-s} (U)}
\leq 
C_0 
\bigl( \ep \X \bigr)^{\kappa}
\| \shom^{\nf12} \nabla g \|_{{H}^s(U)}
\,.
\label{e.big.homogenization.estimate}
\end{equation}

\item \underline{First-order corrector estimates}. 
There exist~$\Zd$--stationary gradient fields~$\{ \nabla \phi_e \,:\, e \in \Rd \}$ satisfying the equation
\begin{equation*}
-\nabla \cdot \a (e+ \nabla \phi_e) = 0 \quad \mbox{in} \ \Rd
\end{equation*}
and, for every~$e\in \Rd$, 
\begin{equation}
\label{e.corrector.estimates}
\| \shom^{\nicefrac12} \nabla \phi_e \|_{\underline{H}^{-1}(E_r)} 
+
\| \shom^{-\nicefrac12} (\a (e+\nabla \phi_e) - \ahom e) \|_{\underline{H}^{-1}(E_r)} 
\leq 
C |\shom^{\nicefrac12}e | 
\Bigl(\frac{r} {\X} \Bigr)^{\!-\kappa}
\,, \quad \forall r \geq \X \,.
\end{equation}
Moreover, for every~$\theta\in (0,1)$, if we define 
\begin{equation}
\label{e.A.one.def}
\A_1(\Rd)   \coloneqq  \Bigl\{ v \in H^1_{\s,\mathrm{loc}}(\Rd) \,:\, -\nabla \cdot \a\nabla v = 0 \ \mbox{in} \ \Rd, \ \limsup_{r\to \infty} r^{-(1+\theta)} \| v \|_{\underline{L}^2(B_r)} = 0 \Bigr\} \,,
\end{equation}
then the linear space~$\A_1(\Rd)$ coincides with~$\{ x\mapsto e\cdot x + \phi_e(x) + c \,:\, e \in\Rd ,\ c \in\R \}$. 

\item \underline{Large-scale regularity}.
For every~$R \geq \X$ and solution~$u\in H^1_\a (E_R)$ of the equation
\begin{equation*}
 -\nabla \cdot \a \nabla u = 0 \quad \mbox{in} \ E_R \,, 
\end{equation*}
we have the estimate
\begin{equation}
\label{e.largescale.C01}
\sup_{ r \in [ \X , R ] } \| \s^{\nicefrac12} \nabla u \|_{\underline{L}^2(E_{r})} \leq C \| \s^{\nicefrac12} \nabla u \|_{\underline{L}^2(E_{R})}
\,.
\end{equation}
Moreover, there exists~$v \in \A_1(\Rd)$ such that, for every~$\theta\in (0,1)$ and~$r\in [\X,R]$, 
\begin{equation}
\label{e.largescale.C1theta}
\| \s^{\nicefrac12} \nabla ( u - v) \|_{\underline{L}^2(E_{r})} 
\leq C\Bigl( \frac r R \Bigr)^{\!\theta}  
\| \s^{\nicefrac12} \nabla u \|_{\underline{L}^2(E_{R})}\,,
\end{equation}
where~$C_{\eqref{e.largescale.C1theta}}$ depends on~$\theta$ in addition to~$d$. 
\end{itemize}
\end{theoremA}

\subsection{Overview of the proof of Theorem~\ref{t.main}}
\label{ss.ideas} 

In this subsection, we explain the main ideas comprising the proof of Theorem~\ref{t.main} and explain where they are formalized in the paper. We break the argument into five informal ``assertions.'' 

\begin{assertion}
\label{e.cg.elliptic}
The ellipticity condition~\ref{a.ellipticity} is sufficiently strong, implying basic~$L^2$ elliptic theory on large scales. That is, on scales larger than~$\S$, we get the same basic $L^2$-type estimates as in the uniformly elliptic case---with the ellipticity constants~$\lambda_0$ and~$\Lambda_0$ taking the place of the usual constants of uniform ellipticity appearing in~\eqref{e.UE.intro}.
\end{assertion}

This is the main purpose of Section~\ref{s.coarse.graining}. There, we introduce the coarse-grained matrices and explore their basic properties, including the basic \emph{coarse-graining inequalities} in~\eqref{e.gradtofluxcg},~\eqref{e.energymaps.nonsymm},~\eqref{e.energymaps.nonsymm.flux} and~\eqref{e.energymaps.nonsymm.all}. These basic estimates allow us to control the spatial averages (and thus the negative Sobolev norms) of the gradients and fluxes of solutions of the equation. Later in the section, we show how these estimates may be combined with the assumption~\ref{a.ellipticity} to obtain coarse-grained versions of the Poincar\'e and Caccioppoli inequalities---the two basic estimates needed in elliptic regularity theory. See Lemmas~\ref{l.crude.weaknorms},~\ref{l.Poincare.largescale},~\ref{l.testing.makes.sense} and Proposition~\ref{p.coarse.grained.Caccioppoli}.

\begin{assertion}
\label{assert.Theta.good.if.small}
The quantity~$\Theta - 1$ quantifies, in a deterministic way, the difference between the solutions of the equation~\eqref{e.pde} and those of the constant-coefficient equation
\begin{equation*}
-\nabla \cdot \a_0 \nabla u = 0 \,,
\end{equation*}
on scales larger than~$\S$, where~$\a_0 \coloneqq  \s_0 + \k_0$ and~$\s_0$ and~$\k_0$ are as defined in~\eqref{e.E.naught.components}. 
\end{assertion}

This is an extension of the previous assertion: we are saying that the parameter~$\Theta$ defined in~\eqref{e.Theta} is a sufficiently good measure of the ellipticity ratio in the sense that, if~$\Theta$ is close to unity, then~\ref{a.ellipticity} exhibits the behavior we expect from equations with (uniformly) small contrast. In particular, the solutions are close to those of a constant-coefficient equation. This assertion is formalized rigorously in Section~\ref{s.homogenization}, using purely deterministic arguments. See Propositions~\ref{p.big.black.box} for quantitative homogenization estimates (and note that Lemma~\ref{l.mathcal.E.to.Lambdas} provides the link between the random variables in that proposition and the coarse-grained ellipticity ratio).

\begin{assertion}
\label{assert.renormalize}
The ellipticity assumption~\ref{a.ellipticity} is \emph{renormalizable}. If we let~$\P_{n_0}$ be the pushforward of~$\P$ under the dilation map 
\begin{equation*}
\a \mapsto \bigl( x \mapsto \a(3^{n_0}x) \bigr)\,,
\end{equation*}
then~$\P_{n_0}$ satisfies~\ref{a.ellipticity} with~$\mathbf{E}_0$ replaced by the expectation~$\bfAhom(\cu_{n_0-l_0})$ of~$\bfA(\cu_{n_0-l_0})$, where~$l_0$ is a constant which is roughly~$C\lceil \log \Theta\rceil$. 
In other words, if we ``zoom out'' and view the equation from a larger scale, then we have the same assumptions as before, except that the mean of the coarse-grained coefficients (on a slightly smaller scale) takes the role of the ellipticity upper bounds. 
\end{assertion}

The rigorous version of Assertion~\ref{assert.renormalize} is stated and proved in Section~\ref{ss.renormalization}: see in particular Proposition~\ref{p.renormalization.P} and Lemma~\ref{l.renormalize.ellipticity}. It is a relatively simple consequence of the subadditivity of the coarse-grained matrices, combined with an application of assumption~\ref{a.CFS}. 

\smallskip

Assertion~\ref{assert.renormalize} gives rise to the renormalization (semi)group. 
It is natural then to define a scale-dependent notion of ellipticity ratio; we do this by defining~$\Theta_n$ to be the quantity defined analogously to~\eqref{e.Theta}, but with~$\bfAhom(\cu_n)$ in place of~$\bfE$: see~\eqref{e.Theta.n.def}. 

\smallskip

The subadditivity property of the coarse-grained matrices implies that~$n\mapsto \Theta_n$ is monotone decreasing, and qualitative homogenization implies that it does converge to~$1$ as~$n\to \infty$. Meanwhile, Assertion~\ref{assert.Theta.good.if.small} says that quantitative homogenization estimates will immediately follow once we give an upper bound on the scale~$n$ such that the quantity~$\Theta_n-1$ is small. 

\smallskip

This naturally leads to the problem of estimating the  scale~$n$ such that~$\Theta_n-1$ is no larger than~$\frac12 (\Theta-1)$. Such an estimate could then be iterated many times, with the help of Assertion~\ref{assert.renormalize} above, to obtain an estimate of the scale on which the renormalized ellipticity ratio is at most~$1+\delta$, for~$\delta>0$ as small as desired.

\begin{assertion}
\label{assert.reduce.Theta}
If~$\Theta \geq 2$, then for every~$\sigma\in (0,\frac12]$ and~$n\in\N$ satisfying~$n\geq C \log (1+\sigma^{-1}\Pi)$,
\begin{equation}
\label{e.assert4}
\Theta_{n} -1
\leq
\sigma \bigl( \Theta - 1 \bigr) \,.
\end{equation}
\end{assertion}

The precise version of Assertion~\ref{assert.reduce.Theta} is stated in Proposition~\ref{p.renormalize}, and the proof of this proposition is the analytic heart of the paper. Here, we see the full power of the renormalization and coarse-graining arguments. 
As one can see from the statement of Proposition~\ref{p.renormalize}, the precise version of~\eqref{e.assert4} is slightly more complicated and involves a second quantity~$\XiDet_n$ which is a variant of~$\Theta_n$. 

\smallskip

The proof is inspired by ideas that originate in~\cite{AS}. That paper, and subsequent works, obtain an inequality which (substantially simplified) states roughly that 
\begin{equation}
\label{e.naive}
\Theta_{m+10} - 1 \leq \bigl( 1 - C_0(d,\Lambda/\lambda)^{-1} \bigr)  ( \Theta_m -1 )\,.
\end{equation}
The constant~$C_0(d,\Lambda/\lambda)$ comes from various applications of elliptic estimates and the Poincar\'e inequality, so it has the form~$C_0(d,\Lambda/\lambda) = C(d) \cdot (\Lambda/\lambda)^{p}$ for some exponent~$p$.\footnote{Even a single application of an energy estimate such as Caccioppoli's inequality will produce a factor of~$(\Lambda/\lambda)^{\nicefrac12}$, so we would have~$p\geq \nicefrac12$ in~\eqref{e.boundverybad}.} It is not hard to see that this inequality must be iterated approximately~$C_0(d,\Lambda/\lambda)\log (\Lambda/\lambda)$ many times before the error~$\Theta_m-1$ is smaller than~$\nicefrac12$. Therefore, the upper bound for the length scale of homogenization that this argument gives is roughly
\begin{equation}
\label{e.boundverybad}
\X \lesssim 3^{C_0(d,\Lambda/\lambda)\log (\Lambda/\lambda)}
\simeq 
\exp\bigl( C C_0(d,\Lambda/\lambda)\log (\Lambda/\lambda) \bigr) 
\simeq 
\exp\bigl( C (\Lambda/\lambda)^p \log (\Lambda/\lambda) \bigr) 
\,.
\end{equation}
If there is any hope to improve this bound to a sub-exponential bound in the ellipticity ratio, it seems that we need to \emph{remove all dependence on the ellipticity constants} from our argument! This may seem quite hopeless since elliptic estimates come with dependence on~$\Lambda/\lambda$, and ellipticity is obviously an important assumption that we need to use.  

\smallskip

But an estimate without dependence on the ellipticity constants is precisely what Assertion~\ref{assert.reduce.Theta} gives---with the caveat that we must take a bigger step, say from~$m$ to~$m+C \log (1+\sigma^{-1}\Pi)$ rather than from~$m$ to~$m+10$ like in~\eqref{e.naive}. 
This is the only way that dependence on~$\Theta$ or~$\Pi$ is allowed to enter into the proof of Assertion~\ref{assert.reduce.Theta}: via the scale restriction (the lower bound on~$n$). Note that here~$n$ is the size of the step in the iteration since, by renormalization (Assertion~\ref{assert.renormalize}), we can assume~$m=0$ without loss of generality. 

\smallskip

To get rid of the ellipticity dependence, we rely on the coarse-grained version of elliptic estimates summarized in Assertion~\ref{e.cg.elliptic}. The idea is to look for a sequence of consecutive scales~$\{ n_1,\ldots,n_1+k\}$ such that~$\E [ \bfA(\cu_m)]$ does not change much for~$m\in[ n_1,n_1+k]$. Since this quantity is monotone decreasing in~$m$, it can only change in one direction, and therefore, a suitable sequence of consecutive scales can be found by a simple pigeonhole argument. (This pigeonhole argument is the one place where the scale restriction is needed in the proof.)
We then argue that, along this finite sequence of scales, the optimizing functions in certain variational formulas for the coarse-grained matrices \emph{must be flat}: that is, their gradients and fluxes must be close to constant functions. This implies, using a new \emph{coarse-grained} div-curl argument, that the expectations of the two coarse-grained matrices~$\s(\cu)$ and~$\s_*(\cu)$ are close to each other (with~$\cu$ being the cube on the largest scale in this range of scales). Since the difference~$\E [ \s(\cu_n) - \s_*(\cu_n)]$ actually upper bounds the quantity~$\Theta_n$, this yields the desired conclusion.

\smallskip

If the ellipticity~$\Theta$ is sufficiently close to one, then the statement of Assertion~\ref{assert.reduce.Theta} can be improved, and the convergence of the renormalized diffusivities to one can be sped up. The idea is essentially that, if~$\Theta -1$ is small, then we can reduce~$\Theta-1$ by a factor of two by zooming out only a fixed finite number of scales---provided we are working in a suitable geometry. An iteration then yields an algebraic rate of decay, summarized in the following informal statement.

\begin{assertion}
\label{assert.five}
There exist~$\sigma_0(d) , \alpha(d) \in (0,\nicefrac12 ]$ and~$C(d) < \infty$ such that, if~$\Theta - 1 \leq \sigma_0$ and
\begin{equation*}
\frac12 \Id \leq \s_0 \leq 2 \Id\,,
\end{equation*}
then we have the estimate
\begin{equation*}
\Theta_m - 1 
\leq C 3^{-m\alpha} (\Theta - 1)\,, 
\quad \forall m\in\N\,.
\end{equation*}
\end{assertion}

Assertion~\ref{assert.five} is proved in Section~\ref{ss.smallcontrast}: see Proposition~\ref{p.algebraic.exp}.

\smallskip

The five assertions above are assembled into a proof of Theorem~\ref{t.main} in Section~\ref{ss.proofs.main.results}. 
The statement of Theorem~\ref{t.HC.intro} is essentially a corollary of Theorem~\ref{t.main} and its proof also appears in Section~\ref{ss.proofs.main.results}. 

\smallskip

In Section~\ref{ss.pigeon.prime}, we demonstrate that the hypotheses of Theorem~\ref{t.main} can be further relaxed. We anticipate that such a generalization will be important for applications, and indeed, it is necessary for the arguments in~\cite{ABK.SD}.

\subsection{Examples satisfying the hypotheses of Theorem~\ref{t.main}} 
\label{ss.examples}

The primary motivating example for the main results in this paper is a finite-range dependent, uniformly elliptic coefficient field~$\a(x)$, which may have very large ellipticity constants and exhibit \emph{near-critical behavior}. 

\smallskip

However, while it is \emph{not} the main focus of the paper, the general ellipticity assumption~\ref{a.ellipticity} we have introduced above also enables us to treat degenerate and/or unbounded coefficient fields. Indeed, under this assumption, certain ``large contrast" examples can be reinterpreted within our framework as having small ellipticity contrast. To demonstrate these points, we present three fundamental examples of random fields that satisfy our assumptions~\ref{a.stationarity},~\ref{a.ellipticity}, and~\ref{a.CFS}. For each example, the results of Theorem~\ref{t.main} are new.

\smallskip

The first example is a scalar coefficient field with \emph{Poisson inclusions}.
We consider two Poisson point clouds~$\omega_1$ and~$\omega_2$ on~$\Rd$ with intensities~$\rho_1\geq 0$ and~$\rho_2\geq 0$, respectively. Let~$\lambda\in(0,1]$, $\Lambda\in[1,\infty)$ and define the scalar matrix-valued field
\begin{equation}
\label{e.a.inclusions}
\a \coloneqq  
\bigl(  1 + (\Lambda-1) \indc_{B_1} \ast \omega_1 + (\lambda-1) \indc_{B_1} \ast \omega_2 \bigr) \Id
\,.
\end{equation}
This field clearly satisfies~\ref{a.stationarity}. Since it has a finite range of dependence, it also satisfies~\ref{a.CFS} with~$(\beta,\nu)=(0,\frac d2)$ and~$\Psi(t) = \exp( c t^2)$ for some constant~$c(d)>0$. The interest is in the ellipticity assumption~\ref{a.ellipticity}. 
As mentioned above, regardless of the values of intensities~$\rho_1$ and~$\rho_2$, this coefficient field satisfies the uniform ellipticity assumption~\eqref{e.UE.intro} with constants~$\lambda$ and~$\Lambda$, and therefore~\ref{a.ellipticity} with~$\mathcal{S}=0$,~$\gamma=0$ and~$(\lambda_0,\Lambda_0)=(\lambda,\Lambda)$. 

\smallskip

However, in the case that~$\rho_1$ and~$\rho_2$ are small (perhaps~$10^{-2}$) and both~$\Lambda$ and~$\lambda^{-1}$ are very large, we can do better than using the uniform ellipticity condition to check~\ref{a.ellipticity}. In this case, the random inclusions are rare, and the connected components of their union will be far from percolating. Therefore, while the uniform ellipticity ratio~$\Lambda\lambda^{-1}$ is very large, Theorem~\ref{t.main} will give a pessimistic bound for the homogenization length scale. 
\smallskip

However, the coarse-grained ellipticity constants are not large. Therefore, to get a better quantitative estimate, we can use the flexibility of the condition~\ref{a.ellipticity}. We argue instead that, on a sufficiently large (random) scale (the typical size of which is a power of~$\Lambda\lambda^{-1}$), the coefficient field has a coarse-grained ellipticity contrast close to one (and in particular less than two). Precisely, we have the following statement, the proof of which appears in Appendix~\ref{ss.Poisson.inclusions}.

\begin{proposition}[Poisson inclusions]
\label{p.inclusions}
There exist constants~$c(d)\in (0,1]$ and~$C(d) \in [1,\infty)$ such that, if~$\max\{ \rho_1, \rho_2 \} \leq c$ and~$\gamma \in (0,1)$, then the random field~$\a(\cdot)$ defined in~\eqref{e.a.inclusions} above satisfies assumptions~\ref{a.stationarity},~\ref{a.ellipticity} and~\ref{a.CFS} with the following parameters:
\begin{equation*}
\Lambda_0 = \lambda_0^{-1}  = (1+C\left|\log \rho\right|^{-2}) \Itwod\,, \quad 
\Psi_\S(t) =  \exp\biggl( c \max\{ \Lambda,\lambda^{-1}\}^{-\frac{1}{d+2} - \frac{\gamma}{d}}  t^{\frac{\gamma}{d+2}} -1 \biggr),
\, \quad 
\Psi (t) = \exp( ct^2 )\,.
\end{equation*}
In particular,~$\Theta \leq \Pi \leq 1+C\left|\log \rho\right|^{-2}$,~$K_\Psi = C$ and~$K_{\Psi_\S} =  (  C\gamma^{-1} )^{\! \frac
{d+2}{\gamma}} \max\{\Lambda , \lambda^{-1}\}^{\frac{1}{\gamma} + 1+ \frac{2}{d}}$.
\end{proposition}

Proposition~\ref{p.inclusions} says that if the intensities of the Poisson processes are small enough, then this seemingly ``high contrast'' homogenization problem is actually a small contrast problem. Consequently, an application of Theorem~\ref{t.main} implies that, in this case, the length scale for homogenization is proportional to a power of~$\max\{ \Lambda, \lambda^{-1} \}$ (with stretched exponential moments). 

\smallskip 

For the second example, we consider the \emph{advection-diffusion equation}
\begin{equation}
\label{e.advdiff}
-\lambda \Delta u + \b(x) \cdot \nabla u = 0 \,.
\end{equation}
Here~$\b(x)$ is a divergence-free, random vector field which can be written as 
\begin{equation}
\label{e.stream}
\b = -\nabla \cdot \k\,, 
\end{equation}
for a stream matrix~$\k$ which is a Gaussian random field taking values in the set~$\R^{d\times d}_{\mathrm{skew}}$ of anti-symmetric matrices. Specifically, assume that each of the entries of~$\k$ is given by the convolution of a fractional Gaussian field with Hurst parameter~$-\sigma\in(-\nicefrac d2,0)$ and the standard mollifier (see Appendix~\ref{ss.FGF} for the definition and explicit construction of the fractional Gaussian fields). Here we do not make any restriction on the covariance structure of the different entries in~$\k$: in particular, we do not assume they are independent or uncorrelated. 

\smallskip

The identity~\eqref{e.stream} allows us to write the equation~\eqref{e.advdiff} as
\begin{equation*}
- \nabla \cdot \bigl( \lambda \Id + \k(x) \bigr) \nabla u = 0\,.
\end{equation*}
Since~$\k$ is a Gaussian random field, it does not belong to~$L^\infty(\Rd)$, and so the equation is not literally uniformly elliptic. 
However, we show that it satisfies~\ref{a.ellipticity} with an ellipticity ratio of order~$\sigma^{-3}\lambda^{-2}$. 

\begin{proposition}[Gaussian stream matrices]
\label{p.example.advdiff}
Consider the random field~$\a = \lambda \Id + \k$, where each of the entries of~$\k \in \R^{d\times d}_{\mathrm{skew}}$ is given by the convolution of a fractional Gaussian field with Hurst parameter~$-\sigma\in (-\nicefrac d2,0)$ and the standard mollifier. 
There exists~$C(d) <\infty$ such that the field~$\a(\cdot)$ satisfies
the assumptions~\ref{a.stationarity},~\ref{a.ellipticity} and~\ref{a.CFS} with the following parameters:
\begin{equation*}
\left\{
\begin{aligned}
&\gamma \in (0,\sigma\wedge 1)\,,
\\ &
\Lambda_0 
=2(\lambda + C\lambda^{-1} \sigma^{-3} ) \qand 
\lambda_0^{-1} = 
2\lambda^{-1}  \,,
\\ &
\Psi_\S(t) = (\sigma-\gamma) \exp \bigl( C^{-1}t^\gamma -C \gamma^{-1} | \log \gamma| \bigr) \,,
\\ &
\beta = 1-\nicefrac{2\sigma}d\,,
\\ &
\Psi(t) = \Gamma_2(c (\tfrac d2-\sigma) t)\,.
\end{aligned}
\right.
\end{equation*}
\end{proposition}
The proof of Proposition~\ref{p.example.advdiff} appears in Appendix~\ref{ss.example.advdiff}. 

\smallskip

If the Hurst parameter~$\sigma$ is equal to zero, the~$L^2$ oscillation of the stream matrix~$\k$ is no longer bounded as a function of the scale, and the ellipticity is infinite (even in the sense of~\ref{a.ellipticity}). In this case, the equation does not homogenize in the usual sense and instead exhibits \emph{superdiffusivity}. As we show in~\cite{ABK.SD}, the techniques introduced in this paper are nevertheless able to analyze it.

\smallskip 

We turn next to our third example: \emph{log-normal coefficient fields} which are of the form  
\begin{equation*}
\a(x) = \exp( h \mathbf{g} (x)) \,, 
\end{equation*}
where~$h>0$ and~$\mathbf{g}$ is a Gaussian field valued in the set~$\R^{d\times d}$ of (not necessarily symmetric) real~$d$-by-$d$ matrices. For concreteness, we assume that each of the entries of~$\mathbf{g}$ is given by the convolution of a fractional Gaussian field with Hurst parameter~$-\sigma \in (-\nicefrac d2,0)$ and the standard mollifier.

\begin{proposition}[Log-normal fields]
\label{p.example.log-normal}
Consider the random field~$\a = \exp( h \g)$, where each of the entries of~$\g \in \R^{d\times d}$ is given by the convolution of a fractional Gaussian field with Hurst parameter~$-\sigma\in (-\nicefrac d2,0)$ and the standard mollifier. 
There exists~$C(d) <\infty$ such that the field~$\a(\cdot)$ satisfies
the assumptions~\ref{a.stationarity},~\ref{a.ellipticity} and~\ref{a.CFS} with the following parameters:
\begin{equation*}
\left\{
\begin{aligned}
&\gamma \in (0,1)\,,
\\ &
\Lambda_0 = \lambda_0^{-1} = 
\exp(Ch^2 \sigma^{-2} ) 
\\ &
\Psi_\S(t) =\exp \bigl( C^{-1}h^{-2} \sigma^{2} \log^2 t -C h^2 \sigma^{-2} \gamma^{-2}\bigr) \,,
\\ &
\beta = 1-\nicefrac {2\sigma}{d}\,,
\\ & 
\Psi(t)=\Gamma_2 \bigl(c (\tfrac d2-\sigma) t\bigr)
\,.
\end{aligned}
\right.
\end{equation*}
\end{proposition}
The proof of Proposition~\ref{p.example.log-normal} appears in Appendix~\ref{ss.log-normal}.

\subsection{Notation}
\label{ss.notation} 

We denote~$r\wedge s \coloneqq  \min\{ r,s\}$ and~$r\vee s \coloneqq  \max\{ r,s\}$. The H\"older conjugate of an exponent~$p\in[1,\infty]$ is denoted by~$p'$, 
where $p' \coloneqq  p(p-1)^{-1}$ if~$p\neq 1,\infty$, 
$p'= \infty$ if $p=1$, and~$p'=1$ if~$p=\infty$. The Euclidean norm on~$\R^m$ is denoted by~$|\cdot|$. For every~$s \in (0,\infty)$, we denote~$\cs  \coloneqq  1 - 3^{-s}$.  

\smallskip

We let~$\R^{m\times n}$ denote the set of~$m$-by-$n$ matrices with real entries. We let~$B^t$ denote the transpose of a matrix~$B$. The~$n$-by-$n$ identity matrix is~$\mathrm{I}_n$. The symmetric and anti-symmetric~$n$-by-$n$ matrices are denoted respectively by~$\R^{n\times n}_{\mathrm{sym}}$ and~$\R^{n\times n}_{\mathrm{skew}}$. The Loewner ordering on~$\R^{n\times n}_{\mathrm{sym}}$ is denoted by~$\leq$; that is, if~$A,B\in \R^{n\times n}_{\mathrm{sym}}$ then we write~$A\leq B$ if~$B-A$ has nonnegative eigenvalues. Unless otherwise indicated, the norm we use for~$\R^{m\times n}$, denoted by~$|A|$, is the spectral norm, that is, the square root of the largest eigenvalue of~$A^tA$. 
For every square matrix~$A$, we have~$|A|=|A^t|$. The spectral norm is also submultiplicative: for all square matrices~$A$ and~$B$, we have~$|AB| \leq |A| |B|$. We also denote~$\R^{n\times n}_{+} \coloneqq \{ A \in \R^{n\times n}\,:\, \xi \cdot A \xi >0, \ \forall \xi \in \R^{n} \setminus \{ 0 \} \}$ and~$\R^{n\times n}_{\sym,+} = \R^{n\times n}_{\mathrm{sym}} \cap \R^{n\times n}_{+}$. 

\smallskip

The Lebesgue measure of a (measurable) subset~$U \subseteq\Rd$ is~$|U|$. If~$V\subseteq \Rd$ is of codimension~$1$ (for instance, the boundary~$\partial U$ of a Lipschitz domain~$U$), then~$|V|$ instead denotes the~$(d-1)$-dimensional Hausdorff measure of~$V$. Volume-normalized integrals and~$L^p$ norms are denoted, for~$p\in[1,\infty)$, by
\begin{equation}
\label{e.volume.normalize}
(f)_U  \coloneqq  
\fint_U f(x) \,dx  \coloneqq  \frac{1}{|U|} \int_U f(x)\,dx
\qquad \mbox{and} \qquad 
\| f \|_{\underline{L}^p(U)} \coloneqq  \Bigl( \fint_U |f(x)|^p \,dx \Bigr)^{\nicefrac1p}
\,.
\end{equation}
We denote by~$|A|$ the cardinality of a finite set~$A$. A slash through the sum symbol~$\sum$ denotes the sample mean: for every~$f:A \to \R^n$,
\begin{equation*}
\avsum_{a \in A} f(a)  \coloneqq  \frac{1}{|A|} \sum_{ a\in A} f(a)\,.
\end{equation*}
Indicator functions of events and subsets of~$\R^n$ are denoted by~$\indc$. 

\smallskip

We let~$C_0(\Rd)$ denote the space of continuous functions~$u:\Rd\to\R$ such that~$\lim_{|x|\to \infty} u(x)=0$,
~$C_c^k(\Rd)$ denotes the subspace of~$C^k(\Rd)$ with compact support in~$\Rd$, and, more generally,~$C_{c}^k(U)$ denotes the space of $k$-times continuously differentiable functions on~$U$ with compact support in~$U$. We use the notation~$C_{c}^\infty(U)\coloneqq \bigcap_{k\in\N} C_{c}^k(U)$. 
The standard H\"older spaces are denoted by~$C^{k,\alpha}(U)$ for every~$k\in\N_0 \coloneqq \{ 0\} \cup \N$,~$\alpha \in (0,1]$ and domain~$U\subseteq\Rd$ and Sobolev spaces are denoted by~$W^{s,p}(U)$ for~$s\in \R$ and~$p\in[1,\infty]$. The volume-normalized norms for the  classical Sobolev spaces are defined, for~$p\in [1,\infty)$, by 
\begin{equation*}
\| f \|_{\underline{W}^{1,p}(U)}
 \coloneqq  
\Bigl( |U|^{-\nf pd} \| \nabla f \|_{{L}^p(U)}^p+ \| f \|_{{L}^p(U)}^p \Bigr)^{\! \nf 1p} 
\,, \quad 
\| f \|_{\underline{W}^{1,\infty}(U)}
 \coloneqq  
|U|^{-\nf 1d}\| \nabla f \|_{{L}^{\infty}(U)}+ \| f \|_{{L}^\infty(U)} \,.
\end{equation*}
The volume-normalized seminorms for the fractional Sobolev spaces with $s\in (0,1)$ and~$p \in [1,\infty)$ are defined as
\begin{equation} 
\label{e.fractional.Sob.def}
[f]_{\underline{W}^{s,p}(U)} 
\coloneqq
\biggl( 
\fint_{U} \int_U \frac{|f(x) - f(y)|^{p}}{|x-y|^{d+sp}} \, dx \, dy
\biggr)^{\! \nf 1p}   
\qand
[f]_{W^{s,\infty}(U)} 
\coloneqq \sup_{x,y \, \in \, U} \frac{|f(x) - f(y)|}{|x-y|^s} \,,
\end{equation}
and then, for~$s \in (0,1)$ and~$p \in [1,\infty)$,  
\begin{equation*} 
\| f \|_{\underline{W}^{s,p}(U)}
 \coloneqq  
\Bigl( |U|^{-\nf {sp}d} \| f \|_{\underline{L}^{p}(U)}^p+ [f]_{\underline{W}^{s,p}(U)}^p \Bigr)^{\!\nf 1p} \,,
\quad
\| f \|_{\underline{W}^{s,\infty}(U)}
 \coloneqq  
 |U|^{-\nf sd} \| f \|_{L^{\infty}(U)}+ [f]_{W^{s,\infty}(U)}\,. 
\end{equation*}
In the case~$p=2$, we denote~$H^s(U) = W^{s,2}(U)$. 
We let~$W^{s,p}_0(U)$ denote the closure of~$C^\infty_c(U)$ in~$W^{s,p}(U)$ with respect to the norm~$\| \cdot \|_{W^{s,p}(U)}$. If~$X(U)$ is a function space defined for every domain~$U \subseteq\Rd$, then~$X_{\mathrm{loc}}(U)$ denotes the set of functions on an open set~$U$ which belong to~$X(V)$ for every Lipschitz domain~$V$ such that~$\overline{V} \subset U$. The dual seminorms of negative regularity are defined, for~$s \in (0,1]$ and~$p \in [1,\infty]$, by
\begin{equation}
\label{e.volume.normalize.neg}
\bigl[ f \bigr]_{\underline{W}^{-s,p'}(U)} \coloneqq  
\sup
\biggl\{ 
 \fint_U f g \,:\,
g\in C^\infty_c(U)\,,\ \| g \|_{\underline{W}^{s,p}(U)} \leq 1 \biggr\}\,,
\end{equation}
which is the volume-normalized norm for the dual space of~$W_0^{s,p}(U)$,
and
\begin{equation*}
\bigl[ f \bigr]_{\Wminusul{-s}{p'}(U)} \coloneqq  
\sup
\biggl\{ 
\fint_U f g  \,:\, g \in C^\infty(U)\,, \
\| g \|_{\underline{W}^{s,p}(U)} \leq 1  \biggr\}\,,
\end{equation*}
which is the volume-normalized norm for the dual space of~$W^{s,p}(U)$. Whenever possible, we represent the duality pairing between~$f \in X$ and~$g \in X'$ by the integral notation~$\int_U f g$, suppressing the abstract brackets~$\langle f,g\rangle$ in favor of this integral form. If~$p=p'=2$, then we write~$H^{-s}$ in place of~$W^{-s,p}$. 

\smallskip

We use the~$\O_\Psi(\cdot)$ notation defined in Section~\ref{ss.big.O} to keep track of the stochastic integrability of our random variables. Throughout, for~$\sigma \in (0,\infty)$ we denote~$\Gamma_\sigma(t) \coloneqq  \exp (t^\sigma)$. The bold symbol~$\gammafun$ denotes the gamma function~$\gammafun(s)  \coloneqq  \int_0^\infty t^{s-1} \exp(-t) \,dt$.

\section{A coarse-graining theory for elliptic operators}
\label{s.coarse.graining}

In this section, we introduce the coarse-grained diffusion matrices that form the basis of our approach in this paper. These quantities are not new here and go back to the works~\cite{AS,AM}. The novelty of this paper lies in the precise way they are used to renormalize the equation. (For historical context and a more complete presentation of some of the material covered in this section, we refer to~\cite[Chapters 4 \& 5]{AK.Book}.) 

\smallskip

There are various possible definitions for a ``coarse-grained diffusion matrix" and the ones we introduce are not the only plausible choices. Particularly in the general nonsymmetric case, our definitions may initially seem counterintuitive. However, these specific quantities are crucial for proving results such as Theorem~\ref{t.HC.intro}. They exhibit a complex algebraic structure and possess essential properties that facilitate coarse-graining. Attempting to substitute alternative notions of ``box diffusivity" into our arguments would result in failure. To paraphrase Steven Weinberg's Third Law of Progress~\cite{Weinberg}, \emph{you may use any quantities you like to study elliptic homogenization, but if you use the wrong ones, you'll be sorry}. 

\smallskip

Given a bounded domain~$U\subseteq\Rd$, we will define two symmetric matrices~$\s_*(U)$ and~$\s(U)$, which we think of as two competing coarse-grained versions of the symmetric part~$\s(\cdot)$ of the coefficient field, and another matrix~$\k(U)$ which may not be antisymmetric but we still consider to be the coarse-grained version of the anti-symmetric part. 
The two symmetric matrices satisfy the ordering~$\s_*(U) \leq \s(U)$, as we will show, and we think of the pair as giving us lower and upper bounds for the coarse-grained symmetric part---with their difference representing the error in the coarse-graining procedure. 

\smallskip

The coarse-grained matrices have a rich structure and many interesting properties. 
In this section, we will list the facts that are used in this paper while omitting some of the proofs of the more basic properties (each of which can be found in~\cite[Section 5]{AK.Book}). 
We also prove an important statement about renormalizations of the assumptions (see Section~\ref{ss.renormalization}) as well as provide some functional inequalities which indicate that the assumption of~\ref{a.ellipticity} is a good notion of a scale-dependent ellipticity condition (see Section~\ref{ss.coarse.grained.ineqs}). 

\subsection{Basic Sobolev space framework}
\label{ss.Sob}

Recall that~$\a\in\Omega$ means that~$\s$,~$\s^{-1}$ and~$\k^t \s^{-1} \k$ have entries belonging to~$L^1_{\mathrm{loc}}(\Rd;\R^{d\times d})$, where~$\s$ and~$\k$ are respectively the symmetric and anti-symmetric parts of~$\a$. 
It is important to keep in mind that this is the minimal \emph{qualitative} requirement for our coefficients fields. Of course, our main results require quantitative ellipticity assumptions, namely~\ref{a.ellipticity}. As we will show, this ensures that the solutions of the equation behave much better than what we can show under the qualitative assumption. However, to even define the coarse-grained coefficients appearing in the quantitative ellipticity assumption, we must introduce some basic notions from elliptic theory, which are somewhat nonstandard due to the general qualitative setting, which allows for unbounded and highly degenerate equations. For this reason, we give a thorough (if succinct) presentation. 
 
\smallskip

We also define, for each~$U\subseteq\Rd$, 
\begin{equation}
\Omega(U)\coloneqq \bigl\{ \a : U \to \R^{d\times d}_{+} \,:\,
\a = \s+\k
\,,\  
\s = \tfrac12(\a+\a^t)
\,,\  
\s,\, \s^{-1} ,\,  
\k^t \s^{-1} \k 
\in L^1_{\mathrm{loc}}(\overline{U};\R^{d\times d})
\bigr\}
\,.
\label{e.Omega.U.def}
\end{equation}
For each~$U \subseteq\Rd$ and~$\a\in \Omega(U)$, we define the function spaces~$H^1_\s(U)$ as the completion of~$C^\infty(U)$ with respect to the norm
\begin{equation}
\label{e.H1s}
\big\| u \bigr\|_{H^1_\s(U)} \coloneqq \Bigl( \| u \|_{L^1(U)}^2 + \int_U \nabla u \cdot \s \nabla u\Bigr)^{\nicefrac12} \,.
\end{equation}
Observe that, by H\"older's inequality, we have that 
\begin{equation}
\label{e.Omega.imp}
u\in H^1_\s(U)
\implies 
\nabla u , \,\a\nabla u \in L^1(U)
\,.
\end{equation}
Indeed,~$u\in H^1_\s(U)$ implies that~$\s^{\nicefrac12} \nabla u \in L^2(U)$ and the assumption of~$\a\in\Omega$ implies that~$\s^{\nicefrac12}$,~$\s^{-\nicefrac12}$ and~$\k \s^{-\nicefrac12}$ also belong to~$L^2(U)$. This, together with Cauchy-Schwarz, give the implication~\eqref{e.Omega.imp}.
According to~\cite[Theorem 1.11]{KO84}\footnote{The paper~\cite{KO84} considers the case of scalar~$\s$, but the proof generalizes to a general (matrix-valued) case.}, the space~$H^1_\s(U)$ is a complete Hilbert space for every~$U\subseteq\Rd$ and~$\a\in\Omega(U)$. 
It is clear that 
\begin{equation*}
C^\infty_c (U) \subseteq H^1_{\s}(U)\,.
\end{equation*}
We also define the subspace of~``trace zero'' functions by
\begin{equation*}
H^1_{\s,0}(U) \coloneqq  \mbox{closure of~$C^\infty_c (U)$ with respect to~$\| \cdot\|_{H^1_{\s}(U)}$}\,.
\end{equation*}
The linear subspace of~$H^1_{\s}(U)$ consisting of solutions of the equation~$\nabla\cdot \a\nabla u =0$ is denoted by
\begin{equation*}
\A(U;\a)  \coloneqq 
\bigl\{ 
u \in H^1_\s(U) \,:\,
\nabla \cdot \a\nabla u = 0 \ \mbox{in} \ U
\bigr\}
\,.
\end{equation*}
Here the equation~$\nabla \cdot \a\nabla u=0$ is to be understood in the sense of distributions; that is, 
\begin{equation*}
\fint_U \nabla \psi \cdot \a\nabla u 
= 0 \,,
\qquad \forall \psi \in C^\infty_c(U)\,.
\end{equation*}
We may write~$\A(U)$ instead of~$\A(U;\a)$ when the coefficient field~$\a$ is clear from the context.

\smallskip

Things are now set up correctly for the application of the Riesz representation theorem to the Dirichlet problem
\begin{equation}
\label{e.H1s.Dirichlet}
\left\{
\begin{aligned}
& -\nabla \cdot \s\nabla v = f
& \mbox{in} & \ U \,,
\\ 
& v = 0 & \mbox{on} & \ \partial U\,.
\end{aligned}
\right.
\end{equation}
We deduce that, for every bounded domain~$U\subseteq\Rd$ and element~$f$ of the dual space of~$H^1_{\s,0}(U)$, the boundary-value problem~\eqref{e.H1s.Dirichlet} has a unique solution~$v\in H^1_{\s,0}(U)$.
This means that~$v$ belongs to~$H^1_{\s,0}(U)$ and, for every~$u\in H^1_{\s,0}(U)$,
\begin{equation*}
\int_U \nabla u \cdot \s \nabla v = \langle u, f \rangle\,,
\end{equation*}
with the brackets~$\langle \cdot,\cdot\rangle$ denoting the pairing between~$H^1_{\s,0}(U)$ and its dual. 
We deduce that the dual space of~$H^1_{\s,0}(U)$, which we denote by~$H^{-1}_{\s}(U)$, can be characterized as 
\begin{equation*}
H^{-1}_{\s}(U)
 \coloneqq 
\bigl\{
\nabla \cdot \s^{\nicefrac12} \f \,:\, \f \in L^2(U)^d
\bigr\}
\,.
\end{equation*}
Indeed, the inclusion~$\supseteq$ is obvious, and the reverse inclusion~$\subseteq$ follows from the solvability of~\eqref{e.H1s.Dirichlet}. We define the dual norm~$\| \cdot \|_{H^{-1}_{\s}(U)}$ by
\begin{equation*}
\| f \|_{H^{-1}_{\s}(U)}  \coloneqq  \sup \bigl\{ \langle u,f \rangle \,:\, u\in H^1_{\s,0}(U), \ \| u \|_{H^1_{\s}(U)} \leq 1 \bigr\}
\end{equation*}
We often abuse notation by writing~$\int_U uf$ in place of~$\langle u,f \rangle$ when~$u\in H^1_{\s,0}(U)$ and~$f\in H^{-1}_\s(U)$. 

\smallskip

We next discuss the solvability of the (not necessarily self-adjoint) equation~$-\nabla \cdot \a\nabla u = 0$ in every bounded domain~$U\subseteq\Rd$. 
For this purpose, we introduce the norm
\begin{equation} 
\label{e.H.one.a}
\| u \|_{H^1_\a(U)}  \coloneqq  \bigl(  \| u \|_{H^1_\s(U)}^2 + \| \nabla \cdot \k \nabla u \|_{H^{-1}_\s(U)}^2 \bigr)^{\nicefrac12} \,,
\end{equation}
and we let~$H^1_{\a}(U)$ denote the closure of~$C^\infty(U)$ with respect to the norm in~\eqref{e.H.one.a}. 
We also let~$H^1_{\a,0}(U)$ denote the closure of~$C^\infty_c(U)$ with respect to~$\| \cdot \|_{H^1_\a(U)}$.

\smallskip

The Lions-Lax-Milgram lemma (see~\cite[Theorem 2.1, page 109]{Showalter}) implies, for every~$f\in H^{-1}_\s(U)$, 
the existence of a unique solution~$u\in H^1_{\a,0}(U)$ of the Dirichlet problem
\begin{equation}
\label{e.H1a.Dirichlet}
\left\{
\begin{aligned}
& -\nabla \cdot \a\nabla u = f
& \mbox{in} & \ U \,,
\\ 
& u = 0 & \mbox{on} & \ \partial U\,.
\end{aligned}
\right.
\end{equation}
Being a solution of~\eqref{e.H1a.Dirichlet} means that~$u \in H_{\a,0}^1(U)$ and~$u$ satisfies
\begin{equation*}
\int_{U} \nabla w \cdot  \s \nabla u  
= 
\langle  \nabla \cdot \k \nabla u  + f , w  \rangle   \,, \quad \forall w \in H^1_{\s,0}(U)
\,.
\end{equation*}
We interpret this simply as
\begin{equation*}
\int_{U} \nabla w \cdot  \a \nabla u  = \langle w,f  \rangle  \,, \quad \forall w \in H^1_{\s,0}(U)\,.
\end{equation*}
Similarly, the Lions-Lax-Milgram lemma implies the well-posedness of the Neumann problem. We introduce the space
\begin{equation*}
L^2_{\s,\mathrm{sol}}(U) \coloneqq  
\bigl\{ \g \, :\, \s^{-\nicefrac12 } \g \in L^2(U)^d\,, \nabla \cdot \g = 0 \bigr\}\,
\end{equation*}
and we let~$\hat{H}^{-1}_\s (U)$ be the closure of~$C^\infty_c(U)$ with respect to the norm 
\begin{equation*}
\| f \|_{\hat{H}^{-1}_\s (U)}  \coloneqq  \sup\biggl\{ \int_U f u \,:\, u\in H^1_{\s}(U), \ \| u \|_{H^1_{\s}(U)} \leq 1 \biggr\} \,.
\end{equation*} 
Since constant functions belong to~$H^1_{\s}(U)$, each element~$f\in \hat{H}^{-1}_\s (U)$ has a well-defined mean value on~$U$ which we denote by~$(f)_U$. For every~$f\in \hat{H}^{-1}_\s (U)$ with~$(f)_U=0$ and~$\g\in L^2_{\s,\mathrm{sol}}(U)$, there exists a unique~$u\in H^1_{\a}(U)$ satisfying~$(u)_U = 0$ and 
\begin{equation}
\label{e.H1a.Neumann}
\left\{
\begin{aligned}
& -\nabla \cdot \a\nabla u = f
& \mbox{in} & \ U \,,
\\ 
& \mathbf{n} \cdot ( \a\nabla u - \g) = 0 & \mbox{on} & \ \partial U\,,
\end{aligned}
\right.
\end{equation}
where~$\mathbf{n}$ is the outward-pointing unit normal vector on~$\partial U$.  The interpretation of~\eqref{e.H1a.Neumann} is that
\begin{equation*}
\int_{U} \nabla w \cdot ( \a\nabla u - \g) = \langle w,f \rangle \,, \quad \forall w \in H^1_{\s}(U)\,.
\end{equation*}

The soft analysis discussed above, which assumes only~$\a\in\Omega$, is very limited. We cannot perform basic energy estimates, nor test the equation with the solutions multiplied by a cutoff function, since, in general, the product~$\varphi u$ need not belong to~$H^1_\s(U)$ even if~$u\in H^1_\s(U)$ and~$\varphi\in C_c^\infty(U)$.

\smallskip

To address this issue and proceed further, 
we will need some basic Sobolev-type embeddings for our space~$H^1_\s$, and this requires a stronger assumption on the coefficient field~$\a(\cdot)$ beyond that~$\a \in \Omega$. 
As it turns out, certain bounds on the coarse-grained coefficients--- implied by assumption~\ref{a.ellipticity}---provide exactly what is needed. 
These Sobolev-type embeddings are presented below in Section~\ref{ss.besov} (see Lemma~\ref{l.crude.weaknorms} for the embeddings and Lemma~\ref{l.testing.makes.sense} for the justification of testing).
However, we must first introduce the coarse-grained matrices and explore their basic properties.

\subsection{The double-variable matrix field~\texorpdfstring{$\bfA(x)$}{A(x)}}

It is convenient to arrange the symmetric and anti-symmetric parts of the field~$\a(\cdot)$ as the block entries of an~$\R^{2d\times 2d}_{\mathrm{sym}}$-valued random field~$\bfA$, which is defined by
\begin{equation}
\label{e.bfA.def}
\bfA(x)  \coloneqq 
\begin{pmatrix} 
( \s + \k^t\s^{-1}\k )(x) 
& -(\k^t\s^{-1})(x) 
\\ - ( \s^{-1}\k )(x) 
& \s^{-1}(x) 
\end{pmatrix}
\,.
\end{equation}
The field~$\bfA(x)$ defined in~\eqref{e.bfA.def} arises naturally in the variational interpretation of~\eqref{e.pde} for coefficient fields~$\a(x)$ which may not be symmetric, and it consequently plays an essential role in coarse-graining. Notice that the assumption~\eqref{e.qual.ellipticity} is just the requirement that~$\bfA$ belong to~$L^1_{\mathrm{loc}}(\Rd;\R^{2d\times 2d}_{\mathrm{sym}})$. 
We may equivalently regard the set~$\Omega$ as being the collection of fields~$\bfA(\cdot)$ having the form of~\eqref{e.bfA.def}, with~$\s \in \R^{d\times d}_{\mathrm{sym}}$ and~$\k \in \R^{d\times d}_{\mathrm{skew}}$, and with entries belonging to~$L^1_{\mathrm{loc}}(\Rd)$. 
By abuse of notation, we will sometimes consider either~$\a(\cdot)$ or~$\bfA(\cdot)$ as the canonical element of~$\Omega$, whichever is more convenient. Throughout, the random fields~$\s(\cdot)$ and~$\k(\cdot)$ always refer to those defined in~\eqref{e.sk.def}.

\subsection{Coarse-grained matrices: definitions and basic properties}
\label{ss.bfA.def}

For every bounded Lipschitz domain~$U\subseteq \Rd$ and realization of the coefficients~$\a\in\Omega(U)$, we associate three matrices~$\s(U;\a)$,~$\s_*(U;\a)$ and~$\k(U;\a)$. The matrices~$\s(U;\a)$ and~$\s_*(U;\a)$ are symmetric, invertible and satisfy the ordering~$\s_*(U;\a) \leq \s(U;\a)$. Together, this pair represents the symmetric part of the coarse-grained field, and the difference between them can be thought of as a quantification of the ``coarse-graining error,'' as we will see below. The matrix~$\k(U;\a)$ represents the anti-symmetric part of the coarse-grained field. It is not necessarily anti-symmetric, in general, but its symmetric part~$\frac12(\k+\k^t)(U;\a)$ is bounded by the gap~$(\s-\s_*)(U;\a)$, as will be shown below, and is thus bounded by the uncertainty. 

\smallskip

There are several equivalent ways to define the matrices~$\s(U;\a)$,~$\s_*(U;\a)$ and~$\k(U;\a)$, and here we opt for a variational formulation. We first introduce the quantity~$J(U,p,q;\a)$, which is defined for and~$p,q\in\Rd$ by
\begin{equation}
\label{e.J.def}
J(U,p,q;\a) 
\coloneqq
\sup_{u \in  \A(U;\a) }
\fint_{U} \biggl( - \frac12 \nabla u \cdot \s \nabla u - p \cdot \a \nabla u + q \cdot \nabla u \biggr) 
\,.
\end{equation} 
Notice that~$J(U,p,q;\a)$ is well-defined by the discussion in the previous subsection. In particular, the integrand in~\eqref{e.J.def} belongs to~$L^1(U)$
for each~$\a\in\Omega(U)$ and~$u\in\A(U;\a)$.
We also define the analog of this quantity for the adjoint operator by
\begin{equation}
\label{e.Jstar.def}
J^*(U,p,q;\a) 
\coloneqq
J(U,p,q;\a^t) 
=
\sup_{u \in  \A^*(U;\a) }
\fint_{U} \biggl( - \frac12 \nabla u \cdot \s \nabla u - p \cdot \a^t \nabla u + q \cdot \nabla u \biggr) 
\end{equation} 
where
\begin{equation*}
\A^*(U;\a)  \coloneqq \A(U;\a^t) =
 \bigl\{ u \in H^1_{\a}(U) \, : \, -\nabla \cdot \a^t \nabla u = 0 \ \mbox{in} \ U \big\}
\end{equation*}
denotes the set of solutions to the adjoint equation in~$U$. 
The supremums in the variational problems on the right sides of~\eqref{e.J.def} and~\eqref{e.Jstar.def}  are achieved, and the maximizers belong to~$H^1_{\a}(U)$ and are unique up to additive constants. We denote them by~$v(\cdot,U,p,q;\a)$ and~$v^*(\cdot,U,p,q;\a)$, respectively. 

\smallskip

The mapping~$(p,q) \mapsto J(U,p,q;\a)$ is quadratic, and it is convenient to write it using matrices. We let~$\s(U;\a), \s_*(U;\a)\in \R^{d\times d}_{\mathrm{sym}}$ and~$\k(U;\a)\in \R^{d\times d}$ be defined in such a way that the following relation is satisfied:
\begin{equation}
\label{e.J.mat}
J(U,p,q;\a) =
\frac 12p \cdot \s(U;\a)p 
+ \frac 12 (q+\k(U;\a) p) \cdot \s_*^{-1}(U;\a) (q+\k(U;\a) p) 
- p \cdot q \,.
\end{equation}
It turns out that~$J^*$ can be written using the same matrices; it satisfies
\begin{equation}
\label{e.J.mat.star}
J^*(U,p,q;\a) =
\frac 12p \cdot \s(U;\a)p 
+ \frac 12 (q-\k(U;\a) p) \cdot \s_*^{-1}(U;\a) (q-\k(U;\a) p) 
- p \cdot q \,.
\end{equation}
That is,~$\s(U;\a^t) = \s(U;\a)$,~$\s_*(U;\a^t) = \s_*(U;\a)$ and~$\k(U;\a^t)=-\k(U;\a)$. 
This non-obvious fact cannot be deduced using algebra alone: it follows from the variational principles~\eqref{e.J.P0.Dirichlet} and~\eqref{e.bfJ.var} below: see~\cite[Lemma 5.2]{AK.Book} for a proof. 

\smallskip

To lighten the notation, we will often drop the dependence on~$\a$ from the quantities defined above, when the coefficient field~$\a$ is clear from context.

\smallskip

We collect the coarse-grained matrices into a single~$2d$-by-$2d$ matrix by defining
\begin{equation}
\label{e.bigA.def}
\bfA(U;\a)
 \coloneqq  
\begin{pmatrix} 
\s + \k^t\s_*^{-1}\k  
& -\k^t\s_*^{-1}
\\ - \s_*^{-1}\k 
& \s_*^{-1}
\end{pmatrix}(U;\a)
\,.
\end{equation}
It is sometimes helpful to refer to the top-left~$d$-by-$d$ block of~$\bfA(U;\a)$, so we define
\begin{equation*}
\b(U;\a) \coloneqq  ( \s + \k^t\s_*^{-1}\k )(U;\a) 
\,.
\end{equation*}
This matrix~$\bfA(U;\a)$ can be considered a coarse-graining of the field~$\bfA(x)$ defined in~\eqref{e.bfA.def}. Indeed, it has the following variational interpretation (see~\cite[Lemma 5.2]{AK.Book}), which gives an alternative way of defining the coarse-grained matrices: for every~$P\in\R^{2d}$, 
\begin{equation}
\label{e.J.P0.Dirichlet}
\frac12 P \cdot \bfA(U;\a) P 
=
\inf\biggl\{ 
\fint_{U} 
\frac12 (X + P) \cdot \bfA (X + P)
\, : \, 
X \in  \Lpoto(U) \times  \Lsolo(U)  
\biggr\}
\,.
\end{equation}
Here~$\Lpoto(U)$ is defined as the closure of the set~$\{  \nabla \phi \, : \, \phi \in C_{\mathrm{c}}^\infty(U)  \}$ of smooth, compactly supported gradients with respect to the norm~$\f \mapsto ( \int_{U} \f\cdot \s \f )^{\nicefrac12}$, and~$\Lsolo(U)$ is the closure of the set~$
\{  \f  \, : \, \f \in C_{\mathrm{c}}^\infty(U;\Rd)\,, \; \nabla \cdot \f = 0  \bigr\}$ of smooth, compactly supported divergence-free fields with respect to the norm~$\f \mapsto ( \int_{U} \f\cdot \s^{-1} \f )^{\nicefrac12}$.

\smallskip

The coarse-grained quantity~$\bfA(U;\a)$ has  the same information as~$J$ and~$J^*$, or equivalently the coarse-grained matrices~$\s(U;\a)$,~$\s_*(U;\a)$ and~$\k(U;\a)$. The interplay between the different ways of writing these quantities has many algebraic and analytic advantages. 

\smallskip

To lighten the notation, we will often drop the dependence on~$\a$ from the quantities defined above, when the identity of the coefficient field~$\a$ is clear from context.

\smallskip

By straightforward algebraic manipulations, we observe that the identities~\eqref{e.J.mat} and~\eqref{e.J.mat.star} are equivalent to
\begin{equation}
\label{e.Jaas.matform}
J(U,p,q) 
=
\frac 12 
\begin{pmatrix} 
-p \\ q
\end{pmatrix}
\cdot \bfA(U)
\begin{pmatrix} 
-p \\ q
\end{pmatrix}
-p\cdot q
\quad \mbox{and} \quad
J^*(U,p,q) 
=
\frac 12 
\begin{pmatrix} 
p \\ q
\end{pmatrix}
\cdot \bfA(U)
\begin{pmatrix} 
p \\ q
\end{pmatrix}
-p\cdot q\,.
\end{equation}
We ``double the variables'' by combining~$J$ and~$J^*$ into a single quantity by defining 
\begin{equation}
\label{e.bfJ}
\bfJ
\biggl(U, \begin{pmatrix} p  \\ q \end{pmatrix}, \begin{pmatrix} q^* \\ p^* \end{pmatrix} ;\a \biggr)
\coloneqq 
\frac12 J\bigl(U, p-p^* , q^*\!-q ;\a\bigr)
+ 
\frac12 J^*\bigl(U, p^*\!+p , q^*+q;\a \bigr)
\,.
\end{equation}
By~\eqref{e.Jaas.matform} and some straightforward algebraic manipulations, the definition~\eqref{e.bfJ} is equivalent to
\begin{align}
\label{e.Jsplitting}
\bfJ(U,P,Q )
=
\frac12 P \cdot \bfA(U) P 
+ 
\frac12 Q \cdot \bfA_*^{-1}(U) Q
- P \cdot Q, \qquad \forall P,Q  \in \R^{2d}
\,,
\end{align}
where the matrix~$\bfA_*^{-1}(U)$ is defined by swapping the rows and columns of~$\bfA(U)$:
\begin{equation}
\label{e.bigAstar.def}
\bfA_*^{-1}(U;\a)
 \coloneqq  
\mathbf{R}
\bfA(U;\a)
\mathbf{R}
=
\begin{pmatrix} 
\s_*^{-1}
& -  \s_*^{-1}\k 
\\ - \k^t\s_*^{-1}
& \s + \k^t\s_*^{-1}\k  
\end{pmatrix}(U;\a)
\,,
\end{equation}
where we introduce the matrix
\begin{equation}
\label{e.R.def}
\mathbf{R} \coloneqq  \begin{pmatrix} 
0 & \Id
\\ \Id
& 0
\end{pmatrix}
\,.
\end{equation}
The following formulas for the inverses of~$\bfA(U)$ and~$\bfA_*^{-1}(U)$ are obtained by a direct computation:
\begin{equation*}
\left\{
\begin{aligned}
& 
\bfA^{-1}(U)
= 
\begin{pmatrix} 
\s^{-1}
& \s^{-1}\k^t 
\\ \k \s^{-1}
& \s_* + \k \s^{-1}\k^t 
\end{pmatrix}(U) 
\,,
\\ & 
\bfA_*(U)
= 
\begin{pmatrix} 
\s_* + \k \s^{-1}\k^t 
& \k \s^{-1}
\\ \s^{-1}\k^t 
& \s^{-1}
\end{pmatrix}(U) 
\,.
\end{aligned}
\right.
\end{equation*}
The quantity~$\bfJ(U,P,Q;\a)$ also has the following variational formulation, which is easy to check  (or see~\cite[Lemma 5.2]{AK.Book}): for every~$P,Q \in \R^{2d}$, 
\begin{equation}
\label{e.bfJ.var}
\bfJ(U, P,Q;\a )
=
\max_{X  \in \S(U;\a)}
\fint_U
\biggl( -\frac12  X  \cdot \bfA X -  P \cdot \bfA X +Q\cdot X \biggr)
\,,
\end{equation}
where the space~$\S(U;\a)$ is defined by
\begin{equation*}
\S(U;\a)
 \coloneqq 
\biggl\{ 
\begin{pmatrix} \nabla v+\nabla v^*  \\ \a\nabla v - \a^t\nabla v^* \end{pmatrix}
\,:\, 
v\in \A(U;\a), \ v^*\in \A^*(U;\a)
\biggr\}
\,.
\end{equation*}
The coarse-grained objects defined above have a rich structure. The properties above and those we list below can be found in~\cite[Sections 5.1 and 5.2]{AK.Book}. 

\smallskip

We continue by discussing basic upper and lower bounds. The matrices~$\s(U)$ and~$\s_*(U)$ are ordered, and the coarse-grained matrices~$\s_*(U)$ and~$\b(U)$ are bounded from above and below by the averages of the field. We have that
\begin{equation}
\label{e.ssk.bounds}
\biggl( \fint_{U} \s^{-1} (x)\,dx \biggr)^{\!-1} \leq \s_*(U) \leq \s(U) 
\leq
\b (U) \leq 
\fint_U \bigl ( \s + \k^t \s_*^{-1} \k \bigr )(x)\, dx
\,.
\end{equation}
Each of the inequalities in~\eqref{e.ssk.bounds} is very easy to prove with the exception of~$\s_*(U)\leq \s(U)$, which is a consequence of the identities~\eqref{e.J.mat},~\eqref{e.J.mat.star} and a duality argument: see~\cite[Lemma 5.4]{AK.Book} and the discussion following it.
The other bounds~\eqref{e.ssk.bounds} can be written more compactly and also more generally in terms of the~$2d$ block matrices:
\begin{equation}
\label{e.bfA.bounds}
\biggl( \fint_{U} \bfA^{-1} (x)\,dx \biggr)^{\!-1} \leq \bfA_{*}(U) \leq \bfA(U) \leq
\fint_{U} \bfA  (x)\,dx
\,.
\end{equation}
The matrix~$\bfA(U)$ and its inverse also satisfy, for every~$\eta > 0$, 
\begin{equation}
\label{e.bfA.bounds.diag}
\left\{
\begin{aligned}
&
\bfA(U)
\leq 
\begin{pmatrix} 
\bigl( \s + (1 + \eta^{-1} )\k^t\s_*^{-1}\k \bigr)(U) 
& 0 
\\ 0 
& (1 + \eta ) \s_*^{-1}(U) 
\end{pmatrix}
\,, 
\\ & 
\bfA^{-1}(U)
\leq 
\begin{pmatrix} 
(1 + \eta ) \s^{-1}(U)
& 0 
\\ 0 
&  
\bigl( \s_* + (1 + \eta^{-1} )\k \s^{-1}\k^t \bigr)(U) 
\end{pmatrix}
\,.
\end{aligned}
\right.
\end{equation}
Subadditivity is another important property of the quantity~$J(U,p,q)$. We write this in terms of the~$2d$-by-$2d$ block matrices as follows. For every bounded Lipschitz domain~$U\subseteq \Rd$ and disjoint partition~$\{ U_i \}_{i=1,\ldots,N}$ of~$U$ (up to a zero Lebesgue measure set), we have 
\begin{equation}
\label{e.subadditivity}
\bfA (U)
\leq 
\sum_{i=1}^N \frac{|U_i|}{|U|} \bfA (U_i)
\quad \mbox{and} \quad 
\bfA_{*}^{-1} (U)
\leq 
\sum_{i=1}^N \frac{|U_i|}{|U|} \bfA_{*}^{-1} (U_i)
\,.
\end{equation}
The bounds in~\eqref{e.subadditivity} should be regarded as a generalization of~\eqref{e.bfA.bounds}, a coarse-grained version of the latter. 
Note that while each of~$\bfA(U)$,~$\bfA_*^{-1}(U)$,~$\b(U)$ and~$\s_*^{-1}(U)$ is subadditive, but neither~$\s(U)$ nor~$\k(U)$ is subadditive is any sense.

\smallskip

By~\cite[Lemma 5.2]{AK.Book}, the symmetric part of~$\k$ is controlled by the gap between~$\s(U)$ and~$\s_*(U)$:
\begin{equation*}
(\k+\k^t)(U) \leq (\s-\s_*)(U)
\quad \mbox{and} \quad
-(\k+\k^t)(U) \leq (\s-\s_*)(U)
\,.
\end{equation*}
This is also proved below in~\eqref{e.symm.k.quad.small}.
The difference of~$\s(U)$ and~$\s_*(U)$ can also be expressed by means of~$J$ and~$J^*$ via the identity
\begin{equation}
\label{e.gap.by.J}
e\cdot (\s-\s_*)(U) e
=
J(U,e,(\s_*- \k)(U) e) 
+ J^*(U,e,(\s_*+\k)(U) e)
\,.
\end{equation}
More generally, we have that by~\cite[Lemma 5.2]{AK.Book}, for every~$\tilde{\s} \in\R^{d\times d}_{\mathrm{sym}}$ and~$\tilde{\k}\in \R^{d\times d}$, we have
\begin{multline} 
\label{e.diagonalset.nosymm}
e \cdot \bigl (\s(U) -\s_*(U)\bigr )e
+
\bigl (\tilde{\s} - \s_*(U)\bigr)e \cdot \s_*^{-1}(U) \bigl (\tilde{\s} - \s_*(U)\bigr )e
+
\bigl ( \tilde{\k}-\k(U) \bigr ) e \cdot \s_*^{-1}(U)\bigl ( \tilde{\k}-\k(U) \bigr ) e
\\
= 
J(U,e,(\tilde{\s}-\tilde{\k}) e) + J^*(U,e,(\tilde{\s}+\tilde{\k}) e)
\,. 
\end{multline}
Equivalently, for any~$\tilde{\bfA} \in \R^{2d\times 2d}_{\sym,+}$ and~$P \in \R^{2d}$, 
\begin{equation}
\frac12 P \cdot (\bfA-\bfA_*)(U) P 
+
\frac12 (\bfA(U) - \tilde{\bfA})P \cdot 
\bfA_*^{-1}(U) (\bfA(U) - \tilde{\bfA}) P
=
\bfJ(U, P, \tilde{\bfA}P) 
\,.
\label{e.diagonalset.nosymm.bfJ}
\end{equation}

We next explore properties of the coarse-grained matrices that give us information about general solutions. 
The first variation of the optimization problem in~\eqref{e.J.def} is  
\begin{align}
\label{e.first.variation}
q \cdot \fint_U \nabla w 
- p \cdot \fint_U \a \nabla w 
=
\fint_U 
\nabla w 
\cdot \s \nabla v(\cdot,U,p,q)
\,, \quad \forall w \in \A(U)\,.
\end{align}
The second variation says that
\begin{align}
\label{e.quadratic.response}
& J(U,p,q) - \fint_U \Bigl  ( -\frac12 \nabla w \cdot \s\nabla w -p\cdot \a\nabla w+ q\cdot \nabla w   \Bigr  )
\notag \\ & \qquad \qquad\qquad 
=
\fint_U \frac12 \bigl ( \nabla v(\cdot,U,p,q) - \nabla w \bigr )\cdot \s\bigl ( \nabla v(\cdot,U,p,q) - \nabla w \bigr )
\,,\quad \forall w \in \A(U)\,.
\end{align}
By taking~$w=0$ in~\eqref{e.quadratic.response}, it follows that~$J$ can be expressed as the energy of its maximizer:
\begin{equation}
\label{e.J.energy}
J(U,p,q)
=
\fint_U 
\frac12 \nabla v(\cdot,U,p,q) 
\cdot \s \nabla v(\cdot,U,p,q)\,.
\end{equation}
We can read off the spatial averages of the gradient and flux of the maximizer~$v(\cdot,U,p,q)$ from the quantity~$J$ itself. We have 
\begin{equation}
\label{e.all.averages.entries}
\left\{
\begin{aligned}
& \fint_U 
\nabla v(\cdot,U,p,q)
=
- p + \s_{*}^{-1}(U) (q + \k(U) p)
\\ & 
\fint_U 
\a \nabla v(\cdot,U,p,q)
=
(\Id -\k^t  \s_{*}^{-1}\bigr )(U)  q - \b(U) p\,.
\end{aligned}
\right.
\end{equation}
These identities play a central role in the analysis in this paper. It will be convenient to write them more compactly, using matrix notation. Observe  that~\eqref{e.all.averages.entries} is equivalent to
\begin{equation}
\label{e.all.averages}
\fint_U
\begin{pmatrix} 
\nabla v 
\\ 
\a \nabla v
\end{pmatrix}
(\cdot,U,p,q)
=
\bigl( \mathbf{R} \bfA(U) + \Itwod \bigr) 
\begin{pmatrix} 
-p \\ q
\end{pmatrix}
\,.
\end{equation}
The quantity~$J$ allows us to relate the spatial averages of gradients and fluxes of arbitrary solutions: for every~$p,q\in\Rd$ and~$w \in \A(U)$, we have by~\eqref{e.first.variation},~\eqref{e.J.energy} and H\"older's inequality that
\begin{equation}
\label{e.fluxmaps}
\Bigl | \fint_{U} \bigl ( p \cdot \a \nabla w - q \cdot \nabla w \bigr ) \Bigr |
=
\Bigl | \fint_U \nabla w 
\cdot  \s \nabla v\bigl (\cdot, U, p,q \bigr )  \Bigr |
\leq
(2J \bigl (U, p,q \bigr ) )^{\nicefrac12}
\Bigl( \fint_U \nabla w \cdot \s \nabla w \Bigr)^{\!\nicefrac12}
\,.
\end{equation}
This inequality is useful when~$J(U,p,q)$ is small, which requires~$q$ and~$p$ to be related and the gap between~$\s(U)$ and~$\s_*(U)$ to be small. Indeed, letting~$\tilde \s \in \R^{d \times d}$ be a positive symmetric matrix and choosing~$p = \tilde \s^{-\nicefrac12} e$ and~$q=(\s_*-\k^t)(U)p$ and taking the supremum over~$|e|=1$ yields, in view of~\eqref{e.gap.by.J}, for every~$w\in \A(U)$, 
\begin{equation}
\label{e.gradtofluxcg}
\Bigl|\tilde \s^{-\nicefrac12} \fint_{U} \bigl( \a \nabla w - (\s_*-\k^t)(U) \nabla w \bigr) \Bigr|^2
\leq
2
\bigl| \tilde \s^{-\nicefrac12}(\s-\s_*)(U) \tilde \s^{-\nicefrac12} \bigr|
\fint_U \nabla w \cdot \s \nabla w 
\,.
\end{equation}
This motivates the definition
\begin{equation*}
\a_*(U;\a)  \coloneqq 
\s_*(U;\a) - \k^t(U;\a)\,.
\end{equation*}
We can then write the previous inequality as 
\begin{equation}
\label{e.gradtofluxcg.a}
\Bigl| \tilde \s^{-\nicefrac12}  \fint_{U} (\a_*(U)-\a) \nabla w \Bigr |^2
\leq
2
\bigl| \tilde \s^{-\nicefrac12}  (\s-\s_*)(U) \tilde \s^{-\nicefrac12}  \bigr|
\fint_U \nabla w \cdot \s \nabla w 
\,, \quad \forall w\in \A(U)
\,. 
\end{equation}
The coarse-grained matrix~$\s_*(U)$ gives a lower bound for the spatial average of the gradient of an arbitrary solution in terms of its energy: 
\begin{align}
\label{e.energymaps.nonsymm}
\frac12\Bigl( \fint_U \nabla u \Bigr) \cdot \s_*(U) \Bigl( \fint_U \nabla u \Bigr)
\leq
\fint_U \frac12 \nabla u \cdot \s\nabla u 
\,, \quad 
\forall u\in \mathcal{A}(U)\,. 
\end{align}
Similarly, the coarse-grained matrix~$\b(U)$ gives a lower bound for the spatial average of the flux of an arbitrary solution in terms of its energy: 
\begin{align}
\label{e.energymaps.nonsymm.flux}
\frac12\Bigl( \fint_U \a \nabla u \Bigr) \cdot \b^{-1} (U) \Bigl( \fint_U \a \nabla u \Bigr)
\leq
\fint_U \frac12 \nabla u \cdot \s\nabla u 
\,, \quad 
\forall u\in \mathcal{A}(U)\,.
\end{align}
In more generality, we have 
\begin{equation} \label{e.energymaps.nonsymm.all}
\frac12 (X)_U\cdot \bfA_*(U) (X)_U \leq  \frac12 \fint_{U} X \cdot \bfA X
\, 
\quad \forall X \in \mathcal{S}(U) 
\,.
\end{equation}
The proof of~\eqref{e.energymaps.nonsymm.all} is simple: we compute
\begin{align}
\frac12 (X)_U\cdot \bfA_*(U) (X)_U  & = \sup_{Q \in \R^{2d}} \bigl( Q \cdot (X)_U - \frac12 Q \cdot \bfA_*^{-1}(U) Q \bigr) \notag \\ & = 
\sup_{Q \in \R^{2d}} \bigl( Q \cdot (X)_U - \bfJ(U,0,Q) \bigr) \notag \\ &  =  \sup_{Q \in \R^{2d}} \inf_{Z \in \S(U)}\biggl( \fint_{U} \bigl( Q \cdot (X -Z)_U +  \frac12 Z \cdot \bfA Z \bigr) \biggr) 
\leq \fint_{U}  \frac12 X \cdot \bfA X   \,.\notag \end{align} 
We refer to  inequalities like~\eqref{e.gradtofluxcg},~\eqref{e.energymaps.nonsymm},~\eqref{e.energymaps.nonsymm.flux} and~\eqref{e.energymaps.nonsymm.all} as \emph{coarse-graining inequalities}. They give strong evidence that the coarse-grained matrices are aptly named, and they play a central role in the arguments in this paper.

\subsection{Coarse-grained elliptic and functional inequalities}
\label{ss.coarse.grained.ineqs}

In this subsection, we introduce a new notion of coarse-grained ellipticity and show that the space~$\A(U;\a)$ of solutions embeds into certain fractional Besov spaces, provided that certain bounds on the coarse-grained matrices are satisfied. 

Throughout, for every~$s\in (0,\infty)$, we define~$\cs \coloneqq 1-3^{-s}$, which is the normalizing constant that makes~$n\mapsto \cs 3^{-ns}$ a probability density on~$\N$.  

\begin{definition}[Coarse-grained ellipticity constants]
For every~$s\in (0,\infty)$,~$q\in [1,\infty)$,~$m\in\Z$ and coefficient field~$\a: \cu_m \to \R^{d\times d}_+$, we define the multiscale composite quantities 
\begin{equation}
\label{e.coarse.grained.ellipticity}
\left\{
\begin{aligned}
& {\Lambda}_{s,q}(\cu_m\,;\a)
 \coloneqq  
\biggl( 
\css{sq} \sum_{k=-\infty}^{m} 
3^{-sq(m-k)} 
\max_{z\in 3^k\Zd \cap \cu_m} 
\bigl| \b(z+\cu_k; \a) \bigr|^{\nf q2} 
\biggr)^{\!\nf{2}{q}}
\,, \\  &
{\lambda}_{s,q}(\cu_m;\a) 
 \coloneqq 
\biggl(\css{sq} \sum_{k=-\infty}^{m} 
3^{-sq(m-k)} 
\max_{z\in 3^k\Zd \cap \cu_m} 
\bigl| \s_{*}^{-1}(z+\cu_k; \a) \bigr|^{\nf q2}
\biggl)^{\!- \nf{2}{q}}
\,.
\end{aligned}
\right.
\end{equation}
For the exponent~$q=\infty$, these quantities were already defined above in~\eqref{e.coarse.grained.ellipticity.infty}. We extend these definitions to translations~$y+\cu_m$ of the cube~$\cu_m$ in the obvious way. 
\end{definition}

For every~$m\in\Z$,~$q\in [1,\infty]$ and~$t,s\in (0,\infty)$ with~$t<s$, we have 
\begin{equation}
\lambda_{t,q}(\cu_m; \a)
\leq
\lambda_{s,q} (\cu_m; \a) 
\leq |\s_*^{-1}(\cu_m;\a) |^{-1} 
\leq 
|\b(\cu_m;\a)| 
\leq 
\Lambda_{s,q}(\cu_m; \a)
\leq 
\Lambda_{t,q}(\cu_m; \a)
\,.
\label{e.ellipticities.monotone.ordered}
\end{equation}
Indeed, the monotonicity of~$\lambda_{s,q}^{-1}$ and~$\Lambda_{s,q}$ in the parameter~$s$, as well as the second and fourth inequalities in the display, follows from subadditivity~\eqref{e.subadditivity} and the general fact that, for any nondecreasing sequence~$\{ \alpha_n\}_{n\in\N}$, 
\begin{equation}
s\mapsto \cs \sum_{n=0}^{\infty} 
3^{-sn} \alpha_n
\quad \mbox{is nonincreasing}
\,.
\label{e.deep.fact}
\end{equation}
The third inequality of~\eqref{e.ellipticities.monotone.ordered} is immediate from~\eqref{e.ssk.bounds}. 
Similarly, by subadditivity, for every~$k \leq m$, 
\begin{equation}
\label{e.lambdas.downscales}
\left\{
\begin{aligned}
& 
\max_{z\in 3^k\Zd\cap \cu_m} 
|\b (z+\cu_k;\a)|
\leq
\max_{z\in 3^k\Zd\cap \cu_m} 
\Lambda_{s,q} (z+\cu_k;\a)
\leq 3^{2s(m-k)} \Lambda_{s,q} (\cu_m;\a)
\,,
\\ & 
\max_{z\in 3^k\Zd\cap \cu_m} 
|\s_*^{-1} (z+\cu_k;\a)|
\leq
\max_{z\in 3^k\Zd\cap \cu_m} 
\lambda_{s,q}^{-1} (z+\cu_k;\a)
\leq 3^{2s(m-k)} \lambda_{s,q}^{-1} (\cu_m;\a)
\,.
\end{aligned}
\right.
\end{equation}
By the Lebesgue differentiation theorem and the equivalance of~\eqref{e.UE.intro} and~\eqref{e.ellipticity.nonsymm},
\begin{equation*}
\text{$\a$ is uniformly elliptic} 
\quad \iff \quad 
\lim_{s\to 0} \lambda_{s,q}(\cu_m; \a) >0
\qand
\lim_{s\to 0} \Lambda_{s,q}(\cu_m; \a) <\infty\,,
\end{equation*}
in which case the two limits on the right coincide with the lower and upper constants of uniform ellipticity, respectively. Note that, if~$\lambda_{s,q}(\cu_m; \a)$ is positive for any~$s$, then~$\lim_{s\to \infty} \lambda_{s,q}(\cu_m; \a) = |\s_*^{-1}(\cu_m;\a) |^{-1}$ and if~$\Lambda_{s,q}(\cu_m; \a)$ is finite for any~$s$, then~$\lim_{s\to \infty} \Lambda_{s,q}(\cu_m; \a) = |\b(\cu_m;\a)|$. Thus the pair~$( \lambda_{s,q}(\cu_m; \a), \Lambda_{s,q}(\cu_m; \a))$ interpolates between the uniform ellipticity constants and the smallest and largest eigenvalues~$( |\s_*^{-1}(\cu_m;\a) |^{-1} , |\b(\cu_m;\a)|)$ of the coarse-grained matrices.

\smallskip

We next introduce our notation for Besov spaces. 
For each~$s \in (0,1)$, $p\in [1,\infty)$, $q\in [1,\infty)$ and~$n\in \N$, we define a (volume-normalized) Besov seminorm in the cube~$\cu_n$ by
\begin{equation}
\label{e.Bs.seminorm}
\left[ g \right]_{\underline{B}_{p,q}^{s}(\cu_{n})}
 \coloneqq 
\Biggl( 
\sum_{k=-\infty}^n
3^{- s q k} \biggl( 
\avsum_{z\in 3^{k-1}\Zd, \, z + \cu_k \subseteq \cu_n}
\bigl \| g - (g)_{z+\cu_k}\bigr \|_{\underline{L}^p(z+\cu_k)}^p
\biggr)^{\!\nicefrac qp}
\Biggr)^{\! \nicefrac1q}
\,.
\end{equation}
In the case~$q=\infty$, we define the Besov seminorms for every~$s\in[0,1]$ and $p \in [1,\infty)$ by
\begin{equation}
\label{e.Bs.seminorm.infty}
\left[ g \right]_{\underline{B}_{p,\infty}^{s}(\cu_{n})}
 \coloneqq 
\sup_{k \in (-\infty,n] \cap \Z}
3^{- s k}
\biggl(
\avsum_{z\in 3^{k-1}\Zd, \,  z + \cu_k \subseteq \cu_n}
\bigl \| g - (g)_{z+\cu_k}\bigr \|_{\underline{L}^p(z+\cu_k)}^p
\biggr)^{\!\nicefrac1p}\,.
\end{equation}
The corresponding (volume-normalized) Besov norms are defined by
\begin{equation*} 
\left\| g \right\|_{\underline{B}_{p,q}^{s}(\cu_{n})}
 \coloneqq  
3^{-sn}| (g)_{\cu_n}|  +
\left[ g \right]_{\underline{B}_{p,q}^{s}(\cu_{n})} 
\,.
\end{equation*}
The Banach space~$B_{p,q}^{s}(\cu_{n})$ is defined to be the closure of~$C^\infty(\overline{\cu}_n)$ with respect to~$\left\| \cdot \right\|_{\underline{B}_{p,q}^{s}(\cu_{n})}$.

\smallskip

Note that~$s\in \{ 0,1\}$ is allowed if~$q=\infty$, and we actually obtain the more familiar Sobolev spaces. Indeed,~$\left[ \cdot \right]_{\underline{B}_{p,\infty}^{1}(\cu_{n})}$ is equivalent to the (volume-normalized)~$W^{1,p}(\cu_n)$ seminorm and, similarly,~$\left[ \cdot \right]_{\underline{B}_{p,\infty}^{0}(\cu_{n})}$ is equivalent to the volume-normalized~$L^p(\cu_n)$ norm modulo constants:
\begin{equation*}  
\| g - (g)_{\cu_n} \|_{\underline{L}^p(\cu_n)} \leq \left[ g \right]_{\underline{B}_{p,\infty}^{0}(\cu_{n})} \leq  C(d) \| g - (g)_{\cu_n} \|_{\underline{L}^p(\cu_n)} \,.
\end{equation*}
We also work with weak Besov norms with negative regularity exponents which are defined as the dual spaces of~$B^{s}_{p,q}$. For every~$s \in (0,1]$,~$p\in [1,\infty]$ and~$q\in [1,\infty]$, we 
let $p'$ and~$q'$ denote the H\"older conjugate exponents of~$p$ and~$q$, respectively, and we define
\begin{equation}
\label{e.Bs.minus.seminorm}
\left[ f \right]_{\Bhatminusul{-s}{p}{q} (\cu_{n})}
 \coloneqq 
\sup \biggl\{ \fint_{\cu_n} f g \, : \, g \in B_{p',q'}^{s}(\cu_{n}) \,, \; 
\left\|  g \right\|_{\underline{B}_{p',q'}^{s}(\cu_{n})} \leq 1 \biggr\}
\,.
\end{equation}
The dual space of the subspace of~$B_{p',q'}^{s}(\cu_n)$ with zero boundary values is defined by 
\begin{equation}
\label{e.Bs.minus.seminorm.zero}
\left[ f \right]_{\underline{B}_{p,q}^{-s}(\cu_{n})}
 \coloneqq 
\sup \biggl\{ \fint_{\cu_n} f g \, : \, g \in C_{\mathrm{c}}^\infty(\cu_n) \,, \; 
\left[  g \right]_{\underline{B}_{p',q'}^{s}(\cu_{n})} \leq 1 \biggr\}
\,.
\end{equation}
Finally, we introduce another variant of these negative spaces by defining
\begin{equation}
\label{e.Bs.minus.seminorm.explicit}
\left[ f \right]_{\Besov{-s}{p}{q}(\cu_n)}
 \coloneqq 
\biggl(
\sum_{k=-\infty}^n
3^{s qk}
\biggl(
\avsum_{z\in 3^k\Zd \cap \cu_n}
\bigl| (f)_{z+\cu_k}\bigr |^p
\biggr)^{\!\nicefrac qp}
\biggr)^{\! \nicefrac1q}
\,.
\end{equation}
The latter definition will sometimes be useful when estimating the negative seminorms from above since we have that, for every~$f$, 
\begin{equation}
\label{e.weak.norms.ordering}
\left[ f \right]_{\underline{B}_{p,q}^{-s}(\cu_{n})}
\leq
\left[ f \right]_{\Bhatminusul{-s}{p}{q} (\cu_{n})}
\leq
3^{d+s}
\left[ f \right]_{\Besov{-s}{p}{q}(\cu_n)}
\,.
\end{equation}
The first inequality in~\eqref{e.weak.norms.ordering} is immediate from the definitions and the second inequality is proved in Lemma~\ref{l.duality.Bs} in the appendix.

\smallskip

In the next lemma, we use the coarse-graining inequalities~\eqref{e.energymaps.nonsymm} and~\eqref{e.energymaps.nonsymm.flux} to obtain embeddings of the solution space~$\A(U;\a)$ into Besov spaces. In fact, we show that bounds on the coarse-grained matrices imply bounds on the weak Besov norms of the gradient and flux of a solution. In other words, we obtain weak, spatially averaged information about an arbitrary solution in terms of the total energy of the solution and the coarse-grained matrices. Note that the estimates in~\eqref{e.grad.B21minuss} and~\eqref{e.flux.B21minuss} are obvious if we replace the coarse-grained matrices in each cube by the supremum of~$\s^{-1}$ and~$\b$, respectively, in that cube. What is new in our approach is that we are able to prove these estimates \emph{without} using pointwise information about the coefficient field.

\begin{lemma} 
\label{l.crude.weaknorms}
For every~$s \in (0,1]$,~$n \in \N$ and~$u \in \mathcal{A}(\cu_n;\a)$,
\begin{equation}
\label{e.grad.B21minuss}
\left[ \nabla u  \right]_{\Besov{-s}{2}{q}(\cu_n)}
\leq 
\css{sq}^{-\nf1q} 
\lambda_{s,q}^{-\nf 12}(\cu_n\,;\a) \| \s^{\nf 12} \nabla u \|_{\underline{L}^2(\cu_n)}
\end{equation}
and
\begin{equation} \label{e.flux.B21minuss}
\left[ \a \nabla u  \right]_{\Besov{-s}{2}{q}(\cu_n)}
\leq
\css{sq}^{-\nf1q} 
\Lambda_{s,q}^{\nf 12}(\cu_n\,;\a) \| \s^{\nf 12} \nabla u \|_{\underline{L}^2(\cu_n)}
\,.
\end{equation}
\end{lemma}

\begin{proof}
For the gradient we have by~\eqref{e.energymaps.nonsymm} and~\eqref{e.coarse.grained.ellipticity} that 
\begin{align*} 
3^{-sqn} [ \nabla u ]_{\Besov{-s}{2}{q}(\cu_n)}^q
& =  
\sum_{k=-\infty}^n  3^{sq (k-n)} 
\biggl( \avsum_{z \in 3^{k} \Z^d \cap \cu_n}  \bigl|(\nabla u)_{z+\cu_k} \bigr|^2
\biggr)^{\! \nf q2} 
\notag \\ &
\leq 
\sum_{k=-\infty}^n  3^{ sq (k-n)} 
\max_{z \in 3^{k} \Z^d \cap \cu_n} 
|\s_*^{-1}(z + \cu_k)|^{\nf q2} 
\biggl( 
\avsum_{z \in 3^{k} \Z^d \cap \cu_n}  \|\s^{\nf12} \nabla u \|_{\underline{L}^2(z + \cu_k)}^2
\biggr)^{\! \nf q2}
\notag \\ &
=
\css{sq}^{-1} 
\|\s^{\nf12} \nabla u \|_{\underline{L}^2(\cu_n)}^q
\css{sq} 
\sum_{k=-\infty}^n  3^{-sq (n-k)} \max_{z \in 3^{k} \Z^d \cap \cu_n} |\s_*^{-1}(z + \cu_k)|^{\nf q2} 
\notag \\ &
= 
\css{sq}^{-1} 
\lambda_{s,q}^{-\nf q2}(\cu_n) \| \s^{\nf 12} \nabla u \|_{\underline{L}^2(\cu_n)}^q
\,.
\end{align*}
This is~\eqref{e.grad.B21minuss}.
The bound~\eqref{e.flux.B21minuss} for the flux follows similarly, using~\eqref{e.energymaps.nonsymm.flux}. 
We have 
\begin{align*} 
3^{-sq m} [ \a \nabla u ]_{\Besov{-s}{2}{q}(\cu_m)}^q
& =  
\sum_{k=-\infty}^m  3^{sq (k-m)}
\biggl( \avsum_{z \in 3^{k} \Z^d \cap \cu_m}  \bigl|(\a \nabla u)_{z+\cu_k} \bigr|^2 
\biggr)^{\!\nf q2}
\notag \\ &
\leq 
\sum_{k=-\infty}^m  3^{sq (k-m)} \max_{z \in 3^{k} \Z^d \cap \cu_m} |\b(z + \cu_k)|^{\nf q2}  
\biggl( \avsum_{z \in 3^{k} \Z^d \cap \cu_m}  \|\s^{\nf12} \nabla u \|_{\underline{L}^2(z + \cu_k)}^2 \biggr)^{\!\nf q2} 
\notag \\ &
= 
\css{sq}^{-1} \Lambda_{s}^{\nf q2}  (\cu_m) \| \s^{\nf 12} \nabla u \|_{\underline{L}^2(\cu_m)}^q
\, . 
\end{align*}
This completes the proof. 
\end{proof}

The previous lemma implies a \emph{coarse-grained Poincar\'e inequality} for solutions, because the left side of~\eqref{e.grad.B21minuss} actually controls a \emph{positive} Besov norm of~$u-(u)_{\cu_n}$. 
Indeed, we will show that 
\begin{equation}
\lambda_{s,1}(\cu_n;\a) > 0 
\quad \implies \quad 
\mathcal{A}(\cu_n;\a) \hookrightarrow
B^{1-s}_{2,\infty}(\cu_n)
\,.
\label{e.embedding}
\end{equation}
This embedding is immediate from~Lemma~\ref{l.crude.weaknorms} and a purely functional analytic fact that is presented in Lemma~\ref{l.dualitylemma} in Appendix~\ref{s.Besov.appendix}.

\begin{lemma}[Coarse-grained Poincar\'e inequality]
\label{l.Poincare.largescale}
There exists a constant~$C(d) < \infty$ such that, for every~$s\in [0,1)$,~$n \in \N$ and~$u \in \A(\cu_n;\a)$, 
\begin{equation}
\label{e.Poink}
\left\| u- (u)_{\cu_n}  \right\|_{\underline{B}_{2,\infty}^{s}(\cu_n)}
\leq
C3^{(1-s)n}
\lambda_{1-s,1}^{-\nf 12}(\cu_n\,;\a) 
\|\s^{\nicefrac12} \nabla u \|_{\underline{L}^2(\cu_n)}
\,.
\end{equation}
In particular, under assumption~\ref{a.ellipticity}, we have that, for every~$s\in [0,1-\nicefrac\gamma2)$,~$n \in \N$ with~$3^n \geq \S$ and~$u \in \A(\cu_n;\a)$,
\begin{equation}
\label{e.Poinked}
\| u- (u)_{\cu_n} \|_{\underline{B}_{2,\infty}^{s}(\cu_n)}
\leq
\frac{C\lambda_0^{-\nicefrac12} 3^{(1-s)n}}{2-2s-\gamma} 
\|\s^{\nicefrac12} \nabla u \|_{\underline{L}^2(\cu_n)}
\,.
\end{equation}
\end{lemma}
\begin{proof}
We combine~\eqref{e.divcurl.est0} in Lemma~\ref{l.dualitylemma} with~\eqref{e.grad.B21minuss} to obtain~\eqref{e.Poink}. 
Thus, if~$3^n \geq \S$, we use~\ref{a.ellipticity} and 
\begin{equation*}
\lambda_{1-s,1}^{-\nf 12}(\cu_n\,;\a) 
\leq 
C (1-s-\nf\gamma 2)^{-1} 
\lambda_{\gamma,\infty}^{-\nf 12}(\cu_n\,;\a) 
\leq
C (1-s-\nf\gamma 2)^{-1} \lambda_0^{-\nf12} 
\end{equation*}
to obtain~\eqref{e.Poinked}.
\end{proof}

The next lemma says that if a solution~$u\in\A(\cu_n)$ has the property that its gradient and flux belong to compatible Besov spaces, then we are able to test the equation with~$\varphi u$ and integrate by parts. In view of the above lemmas, this allows us to justify basic energy estimates in this general non-uniformly elliptic framework---such as the Caccioppoli inequality---clearing the way for basic elliptic theory. 

\begin{lemma} \label{l.testing.makes.sense}
Let $n \in \N$, $s\in(0,1)$ and $\ep \in (0,1-s)$.  If~$u \in \mathcal{A}(\cu_n;\a)$ is such that
\begin{equation} \label{e.weak.norms.qualitative}
[ \nabla u]_{\Besov{s-1+\ep}{2}{1}(\cu_{m})} + \bigl[ \a \nabla u \bigr]_{\Besov{-s}{2}{1}(\cu_{m})} < \infty\,,
\end{equation}
then, for every~$\varphi \in C_c^\infty(\cu_n)$,
\begin{equation} \label{e.testing.makes.sense}
\fint_{\cu_n} \varphi \nabla u \cdot \s \nabla u 
=
-
\fint_{\cu_n}  (u -(u)_{\cu_n}) \nabla \varphi \cdot \a \nabla u 
\,.
\end{equation}
\end{lemma}

\begin{proof}
Without loss of generality, we may assume that~$(u)_{\cu_n} = 0$. 
By~\eqref{e.weak.norms.qualitative} and~\eqref{e.divcurl.est0}, we have that~$\left\| u \right\|_{\underline{B}_{2,\infty}^{s+\ep}(\cu_n)}<\infty$. 
For~$k \in \N$ with~$k\geq 10$, we set~$u_k  \coloneqq  (u \wedge k) \vee (-k)$ and~$v_k  \coloneqq  u - u_k$. Observe that~$v_k \to 0$ in~$L^2(\cu_n)$ as~$k\to \infty$. Let~$\varphi \in C_c^\infty(\cu_n)$. Then~$u_k \varphi \in H_{\s,0}^{1}(\cu_n)$, because~$u_k$ is bounded and~$u \in H_{\s}^{1}(\cu_n)$. Since~$u \in \mathcal{A}(\cu_n)$, we  obtain
\begin{align} \label{e.testing.makes.sense.pre}
0  = \int_{\cu_n} \a  \nabla u \cdot \nabla (u_k \varphi) 
& =  \int_{\cu_n} \varphi  \a \nabla u \cdot \nabla u_k 
+
\int_{\cu_n} u_k  \a  \nabla u \cdot \nabla \varphi  
\notag \\  &
=  \int_{\cu_n} \varphi \s \nabla u_k \cdot \nabla u_k  
+
\int_{\cu_n} u_k \a  \nabla u \cdot \nabla \varphi   
\,.
\end{align}
We will argue that~\eqref{e.weak.norms.qualitative} allows us to pass to the limit~$k\to \infty$ in~\eqref{e.testing.makes.sense.pre} and thereby obtain~\eqref{e.testing.makes.sense}. We first observe that the first term on the right side of~\eqref{e.testing.makes.sense.pre} converges to the energy of~$u$ by the monotone convergence theorem.  To show that we can pass to the limit in the second term on the right, we use duality to get, by~\eqref{e.duality.Bs}, 
\begin{equation*}  
\biggl| \int_{\cu_n} v_k \a  \nabla u \cdot \nabla \varphi    \biggr| 
\leq 
3^{d+s} s^{-1}
 \left[  \a  \nabla u \right]_{\Besov{-s}{2}{1}(\cu_{m})}
\left\| v_k \nabla \varphi \right\|_{\underline{B}_{2,\infty}^{s}(\cu_n)} 
\,.
\end{equation*}
The first factor on the right is finite by assumption~\eqref{e.weak.norms.qualitative}. 
We will show that the second factor on the right vanishes in the limit~$k\to \infty$. To that end, we apply~\eqref{e.divcurl.est1.pre}, which gives us
\begin{equation}
\label{e.vk.approx.splot}  
\left\| v_k \nabla \varphi \right\|_{\underline{B}_{2,\infty}^{s}(\cu_n)} 
\leq 
C ( 3^{n} \| \nabla^2 \varphi \|_{L^{\infty}(\cu_n)} +  \| \nabla \varphi \|_{L^\infty(\cu_n} \bigr) \left\| v_k \right\|_{\underline{B}_{2,\infty}^{s}(\cu_n)} 
\,.
\end{equation}
The average term in~$\left\| v_k \right\|_{\underline{B}_{2,\infty}^{s}(\cu_n)}$ on the right side of~\eqref{e.vk.approx.splot} converges to zero as~$k\to \infty$ since, as mentioned above,~$v_k \to 0$ in~$L^2(\cu_n)$. For the semi-norm, we first use that 
\begin{align*}
\| u_k - (u_k)_{z+\cu_j}  \|_{\underline{L}^2(z+\cu_j)}^2 
&
\leq
\fint_{z+\cu_j} \fint_{z+\cu_j} |  u_k(x) - u_k(y)|^2 \, dx \, dy
\notag \\ &
\leq 
\fint_{z+\cu_j} \fint_{z+\cu_j} |  u(x) - u(y)|^2 \, dx \, dy
\leq  
4 \| u - (u)_{z+\cu_j}  \|_{\underline{L}^2(z+\cu_j)}^2 
\,.
\end{align*}
Select~$\tilde \ep\in (0,1)$ so that~$s(1-\tilde \ep)^{-1}= s + \ep$, and then use H\"older's inequality to obtain that
\begin{align}   
[v_k]_{\underline{B}_{2,\infty}^{s}(\cu_n)}
& =  \sup_{j \in (-\infty,n] \cap \Z} 3^{sj} \biggl( \avsum_{z \in 3^j \Z^d \cap \cu_n} \| v_k - (v_k)_{z+\cu_j}  \|_{\underline{L}^2(z+\cu_j)}^2 \biggr)^{\! \nicefrac12}
\notag \\  &
\leq
\sup_{j \in (-\infty,n] \cap \Z} 3^{sj} \biggl( \avsum_{z \in 3^j \Z^d \cap \cu_n} \| v_k - (v_k)_{z+\cu_j} \|_{\underline{L}^2(z+\cu_j)}^{2(1-\tilde \ep)} \| v_k \|_{\underline{L}^2(z+\cu_j)}^{2\tilde \ep}  \biggr)^{\! \nicefrac12} 
\notag \\  &
\leq
 \| v_k \|_{\underline{L}^2(\cu_n)}^{\tilde \ep} 
\sup_{j \in (-\infty,n] \cap \Z} 
3^{sj} \biggl( \avsum_{z \in 3^j \Z^d \cap \cu_n} \| v_k - (v_k)_{z+\cu_j} \|_{\underline{L}^2(z+\cu_j)}^{2}  \biggr)^{\! \nicefrac{(1-\tilde \ep)}2} 
\notag \\  &
\leq
 \| v_k \|_{\underline{L}^2(\cu_n)}^{\tilde \ep} 
 \sup_{j \in (-\infty,n] \cap \Z}  3^{sj} \biggl( 4 \avsum_{z \in 3^j \Z^d \cap \cu_n} \| u - (u)_{z+\cu_j} \|_{\underline{L}^2(z+\cu_j)}^{2}  \biggr)^{\! \nicefrac{(1-\tilde \ep)}2} 
\notag \\  &
\leq
2  \| v_k \|_{\underline{L}^2(\cu_n)}^{\tilde \ep}  [u]_{\underline{B}_{2,\infty}^{s + \ep}(\cu_n)}^{1-\tilde \ep}
\,.
\notag
\end{align}
In view of the assumption~\eqref{e.weak.norms.qualitative} and~\eqref{e.divcurl.est0}, and the fact that~$v_k \to 0$ in~$L^2(\cu_n)$, we deduce that~$[v_k]_{\underline{B}_{2,\infty}^{s}(\cu_n)} \to 0$ as~$k \to \infty$ and thus the right side of~\eqref{e.vk.approx.splot} tends to zero as~$k\to \infty$. 
Thus, we may pass to the limit~$k\to \infty$  in~\eqref{e.testing.makes.sense.pre} to obtain~\eqref{e.testing.makes.sense}. The proof is complete.
\end{proof}

Given a uniformly elliptic coefficient field~$\a(\cdot)$ satisfying~\eqref{e.UE.intro}, the classical Caccioppoli inequality  states that every solution~$u\in \mathcal{A}(\cu_m;\a)$ satisfies, for a universal constant $C<\infty$, 
\begin{equation*}
\| \s^{\nf12}  \nabla u \|_{\underline{L}^2(\cu_{m-1})}^2
\leq 
C \Lambda 3^{-2m}
\| u \|_{\underline{L}^2(\cu_{m})}^2
\,.
\end{equation*}
We next prove a generalization of this classical inequality by relaxing the requirement of uniform ellipticity to coarse-grained ellipticity.

\begin{proposition}[Coarse-grained Caccioppoli inequality]
\label{p.coarse.grained.Caccioppoli}
There exists a constant~$C(d)<\infty$ such that, 
for every~$s,t\in (0,1)$ with~$s+t < 1$ and for every~$u\in\A(\cu_m; \a)$, we have the estimate
\begin{equation}
\label{e.coarse.grained.Caccioppoli}
\| \s^{\nf12}  \nabla u \|_{\underline{L}^2(\cu_{m-1})}^2
\leq 
\biggl( \frac{C}{1-s-t} \biggr)^{\! 2 + \frac{4s}{1-s-t}} 
\biggl( \frac{\Lambda_{s,1}(\cu_m; \a)}{\lambda_{t,1}(\cu_m; \a)} \biggr)^{\!\frac{s}{1-s-t}}
{\Lambda}_{s,1}(\cu_m; \a) 
3^{-2m}
\| u \|_{\underline{L}^2(\cu_{m})}^2
\,.
\end{equation}
\end{proposition}
\begin{proof}
By scaling we may suppose that~$m=0$. 
We may assume that~${\Lambda}_{s,1}(\cu_0)$ and~${\lambda_{t,1}^{-1}(\cu_0)}$ are finite, otherwise there is nothing to show. 
Select~$u\in H^1(\cu_0)$ satisfying~$-\nabla \cdot \a\nabla u = 0$ in~$\cu_0$. By subtracting a constant, we may suppose~$(u)_{\cu_0}=0$. 

\smallskip

Define~$\sigma\coloneqq 1-s-t > 0$. 
The main step in the argument is to prove the following inequality: for every~$\rho_1, \rho_2 \in [\nf 12,1)$ 
with~$\rho_1 < \rho_2$, 
\begin{align} 
\label{e.Cacc.pre.iteration}
\| \s^{\nf12}  \nabla u \|_{\underline{L}^2(\rho_1\cu_0)}^2
&
\leq
\frac12 \| \s^{\nf12}  \nabla u \|_{\underline{L}^2(\rho_2 \cu_0)}^2
\notag \\ & \quad
+ 
C (\rho_2-\rho_1)^{- \frac{2(s + \sigma)}{\sigma}} 
\biggl(  \frac{C}{s(1-s)} \biggl(\frac{\Lambda_{s,1}(\cu_{0})}{\lambda_{t,1}(\cu_{0}) } 
\biggr)^{\!\nf12}   \biggr)^{\! \nf{2s}{\sigma}} \Lambda_{s,1}  (\cu_{0})\| u \|_{L^2(\cu_0)}^2
 \,.
\end{align}
Assuming~\eqref{e.Cacc.pre.iteration}, a standard iteration argument (see~\cite[Lemma C.6]{AKMbook}) then yields, using also that~$s^{-s} \leq C$,
\begin{equation*} 
\| \s^{\nf12}  \nabla u \|_{\underline{L}^2(\cu_{-1})}^2
\leq
 \bigl(C \sigma^{-1} \bigr)^{\frac{2(s + \sigma)}{\sigma}} 
 (1-s)^{-\nf {2s}\sigma}  
\biggl(\frac{\Lambda_{s,1}(\cu_{0})}{\lambda_{t,1}(\cu_{0}) } 
\biggr)^{\! \nf{s}{\sigma}}
\Lambda_{s,1}  (\cu_{0})\| u \|_{L^2(\cu_0)}^2
 \,.
\end{equation*}
Plugging in~$\sigma = 1-s-t$, we get~\eqref{e.coarse.grained.Caccioppoli}.

\smallskip

Fix parameters~$\rho_1, \rho_2 \in [\nf 12,1)$ 
with~$\rho_1 < \rho_2$ and scale separation parameters~$h,k\in\N$ satisfying~$h\geq k+4$ and 
$3^{-4} (\rho_2 - \rho_1) \leq 3^{-k} \leq 3^{-3} (\rho_2 - \rho_1)$. Select a smooth cut-off function~$\varphi \in C^\infty_c(\cu_0)$ satisfying
\begin{equation*} 
\indc_{\rho_1 \cu_0}  
\leq
\varphi \leq 
\indc_{\frac12(\rho_1 + \rho_2) \cu_0} 
\qand
\| \nabla^j\varphi\|_{L^\infty(\cu_0)} \leq 3^{k j} \leq C (\rho_2 - \rho_1)^{-j}\,, \quad \forall j \in \{1,2\} 
\,.
\end{equation*} 
Denote the set of subcubes which overlap with the support of~$\nabla \varphi$ by
\begin{equation*}
S \coloneqq  \bigl\{ z \in 3^{-h} \Zd \cap \cu_0 \,:\, \supp(\nabla \varphi) \cap (z+\cu_{-h}) \neq \emptyset \bigr\}\,.
\end{equation*}
Since~$h \geq k+4$, we observe that~$z \in S$ implies~$z+\cu_{-h} \subseteq \rho_2 \cu_0$. 
By Lemmas~\ref{l.dualitylemma} and~\ref{l.crude.weaknorms}, for every~$r\in \{s,1-t\}$, 
\begin{align*} 
\bigl\| (u - (u)_{z+\cu_{-h}})  \nabla \varphi \bigr\|_{\underline{B}^{r}_{2,\infty}(z+\cu_{-h})} 
& 
\leq 
C 3^{k} [\nabla u ]_{\Besov{r-1}{2}{1}(z+\cu_{-h})} 
\notag \\ &
\leq
\frac{C}{1-s} 3^{k-(1-s)h} \lambda_{1-r,1}^{-\nf12} (z+\cu_{-h}) \| \s^{\nf 12} \nabla u \|_{\underline{L}^2(z+\cu_{-h})}  
 \,.
\end{align*}
Lemma~\ref{l.crude.weaknorms} also yields, for~$r \in \{s,1\}$, 
\begin{equation*} 
[ \a \nabla u ]_{\Besov{-r}{2}{1}(z+\cu_{-h})} 
\leq
2 r^{-1} 
3^{-r h}\Lambda^{\nf12 }_{r,1} (z + \cu_{-h})
\| \s^{\nf 12} \nabla u \|_{\underline{L}^2(z+\cu_{-h})}  
\,.
\end{equation*}
In particular, since~$1-s-t>0$, 
we may apply Lemma~\ref{l.testing.makes.sense} to obtain 
\begin{align}
\label{e.Caccioppoli.the.beginning}
\int_{\cu_0} 
\varphi \nabla u \cdot \a \nabla u 
=
\int_{\cu_0} 
- u \nabla \varphi \cdot \a\nabla u 
& 
= 
\avsum_{z\in 3^{-h}\Zd \cap \cu_0} \indc_{S}(z)
\fint_{z+\cu_{-h}} 
- u \nabla \varphi \cdot \a\nabla u 
\,.
\end{align}
Using~\eqref{e.weak.norms.ordering.A}, we estimate, for~$z \in S$, 
\begin{align*}
\fint_{z+\cu_{-h}} 
- u \nabla \varphi \cdot \a\nabla u 
& 
\leq 
C \bigl| (u)_{z+\cu_{-h}}\bigr| 
\| \nabla \varphi \|_{\underline{B}_{2,\infty}^{1}(z+\cu_{-h})}
[ \a \nabla u ]_{\Besov{-1}{2}{1}(z+\cu_{-h})}
\notag \\ & \qquad 
 + 
 C 
\bigl\| (u - (u)_{z+\cu_{-h}})  \nabla \varphi \bigr\|_{\underline{B}^{s}_{2,\infty}(z+\cu_{-h})}
[ \a \nabla u ]_{\Besov{-s}{2}{1}(z+\cu_{-h})}  \,.
\end{align*}
Since~$\| \nabla \varphi \|_{\underline{B}_{2,\infty}^{1}(z+\cu_{-h})} \leq C3^{k+h}$, we may combine the previous three displays to obtain
\begin{align*}
\fint_{z+\cu_{-h}} 
- u \nabla \varphi \cdot \a\nabla u 
&
\leq
C 3^{k} \Lambda^{\nf12 }_{1,1} (z + \cu_{-h}) 
\bigl| (u)_{z+\cu_{-h}}\bigr| 
\| \s^{\nf 12} \nabla u \|_{\underline{L}^2(z+\cu_{-h})}  
\notag \\ & \qquad 
+
\frac{C 3^{k -h} }{s(1-s)} \biggl( \frac{\Lambda_{s,1}(z+\cu_{-h})}{\lambda_{1-s,1}(z+\cu_{-h})} 
\biggr)^{\! \nf 12} \| \s^{\nf 12} \nabla u \|_{\underline{L}^2(z+\cu_{-h})}^2 
 \,.
\end{align*}
We now give up a bit of the regularity exponent in the ellipticity constants to take advantage of the scale separation: using that~$1-s > t$, we apply~\eqref{e.ellipticities.monotone.ordered} in the form~$\lambda_{1-s}^{-1} (z{+}\cu_{-h}) \leq \lambda_{1- s - \sigma}^{-1}(z{+}\cu_{-h})$ with~$\sigma  \coloneqq  1-s-t$, and then use~\eqref{e.lambdas.downscales} to obtain
\begin{equation*}
3^{-h} 
\max_{z\in 3^{-h} \Zd \cap \cu_0} 
\biggl( 
\frac{\Lambda_{s,1}(z+\cu_{-h})}{\lambda_{1-s,1}(z+\cu_{-h}) } 
\biggr)^{\!\nf12} 
\leq 
3^{-\sigma h} 
\biggl( 
\frac{\Lambda_{s,1}(\cu_{0})}{\lambda_{t,1}(\cu_{0}) } 
\biggr)^{\!\nf12}
\,. 
\end{equation*}
We use this to estimate the second term on the right side of the previous display. For the first term, we use~\eqref{e.ellipticities.monotone.ordered} and~\eqref{e.lambdas.downscales} once again to get
\begin{equation*}
\max_{z\in 3^{-h} \Zd \cap \cu_0} 
\Lambda_{1,1} (z+\cu_{-h}) 
\leq 
\max_{z\in 3^{-h} \Zd \cap \cu_0} 
\Lambda_{s,1} (z+\cu_{-h}) 
\leq
3^{sh} \Lambda_{s,1} (\cu_{0}) 
\,.
\end{equation*}
Combining the above three displays with~\eqref{e.Caccioppoli.the.beginning}, we obtain 
\begin{align}
\label{e.Caccioppoli.almost.done}
\| \s^{\nf12}  \nabla u \|_{\underline{L}^2(\rho_1\cu_0)}^2
&
\leq 
\frac{C3^{k-\sigma h}}{s(1-s)} 
\biggl( 
\frac{\Lambda_{s,1}(\cu_{0})}{\lambda_{t,1}(\cu_{0}) } 
\biggr)^{\!\nf12}
\| \s^{\nf12}  \nabla u \|_{\underline{L}^2(\rho_2\cu_0)}^2
\notag \\ & \qquad
+
C3^{k+sh}
\Lambda_{s,1}^{\nf12}  (\cu_{0})
\| u \|_{L^2(\cu_0)}
\| \s^{\nf12}  \nabla u \|_{\underline{L}^2(\rho_2 \cu_0)}
\,.
\end{align}
We now select
\begin{equation*} 
h  \coloneqq  \Biggl\lceil \sigma^{-1} \log_3 \biggl( \frac{4C_{\eqref{e.Caccioppoli.almost.done}}3^{k}}{s(1-s)} 
\biggl(\frac{\Lambda_{s,1}(\cu_{0})}{\lambda_{t,1}(\cu_{0}) } 
\biggr)^{\!\nf12} \biggr) \Biggl\rceil
\,. 
\end{equation*}
Plugging this into the previous display and using Young's inequality, we obtain~\eqref{e.Cacc.pre.iteration}. 
This completes the proof.  
\end{proof}

\subsection{Centering the anti-symmetric part of the coefficient field}
\label{ss.skew}
The set of solutions of the equation
\begin{equation*}
-\nabla \cdot \a\nabla u=0
\end{equation*}
does not change when we add a constant anti-symmetric matrix to the coefficient field~$\a(\cdot)$. 
We may even consider that the field~$\a(\cdot)$ and its anti-symmetric part are defined only modulo a constant anti-symmetric matrix. This is an important invariance that is reflected in the properties of the coarse-grained matrices.

\smallskip

For convenience, we extend the definition of the quantity~$J$ by defining, for each given constant anti-symmetric matrix~$\h_0\in\R^{d\times d}_{\mathrm{skew}}$, 
\begin{equation*}  
J_{\h_0}(U,p,q;\a) 
 \coloneqq 
\sup_{u \in  \A(U;\a) }
\fint_{U} \biggl( - \frac12 \nabla u \cdot \s \nabla u - p \cdot ( \a - \h_0 ) \nabla u + q \cdot \nabla u \biggr) 
\,.
\end{equation*} 
In other words, the quantity~$J_{\h_0}$ is the same as~$J$ if we replace the coefficient field~$\a$ with~$\a-\h_0$. 
The observation is then that the~$J$'s for different~$\h_0$'s are equivalent in the sense that
\begin{equation}
\label{e.your.face}
J_{\h_0}(U,p,q;\a) 
=
J(U,p,q;\a-\h_0) 
=
J(U,p ,q - \h_0p;\a)\,.
\end{equation}
Indeed, the solution space~$\A(U;\a)$ is unchanged by the subtraction of~$\h_0$ and therefore, 
if we let~$v_{\h_0}(\cdot,U,p,q;\a) \in \A(U;\a)$ denote the maximizer of~$J_{\h_0}(U,p,q;\a)$, it follows immediately from the definitions that
\begin{equation*}
\nabla v_{\h_0}(\cdot,U,p,q;\a)  = \nabla v(\cdot,U,p,q- \h_0p;\a)
\,, \qquad \forall p,q \in \Rd\,.
\end{equation*}
We obtain~\eqref{e.your.face} from this and a routine computation. 
The quantity~$J_{\h_0}(U,\cdot,\cdot)$ can be represented by a matrix~$\bfA_{\h_0}(U)$ which is computed in terms of~$\bfA(U)$ as follows:
\begin{align*}
\begin{pmatrix} 
-p \\ q
\end{pmatrix}
\cdot \bfA_{\h_0} (U)
\begin{pmatrix} 
-p \\ q
\end{pmatrix}
&
=
2J(U,p ,q - \h_0p) + 2 p\cdot q
\notag \\ & 
=
\begin{pmatrix} 
-p \\ q - \h_0p
\end{pmatrix}
\cdot \bfA (U)
\begin{pmatrix} 
-p \\ q- \h_0p
\end{pmatrix}
+ 2 p\cdot \h_0 p
\notag \\ & 
=
\begin{pmatrix} 
-p \\ q
\end{pmatrix}
\cdot
\begin{pmatrix} 
\Id
& 0
\\ \h_0
& \Id
\end{pmatrix}^{\!t}
\bfA(U)
\begin{pmatrix} 
\Id
& 0
\\ \h_0
& \Id
\end{pmatrix}
\begin{pmatrix} 
-p \\ q
\end{pmatrix}
\,.
\end{align*}
Here, we used that~$p\cdot \h_0 p = 0$, since $\h_0$ is anti-symmetric.
Therefore, we deduce that 
\begin{align}
\label{e.Ak0.formula}
\bfA_{\h_0} (U)
&
=
\mathbf{G}_{\h_0}^{t}
\bfA(U)
\mathbf{G}_{\h_0}
=
\begin{pmatrix} 
\s + (\k-\h_0)^t\s_*^{-1}(\k-\h_0)
& -(\k-\h_0)^t\s_*^{-1}
\\ - \s_*^{-1}(\k-\h_0)
& \s_*^{-1}
\end{pmatrix} (U)
\,,
\end{align}
where we denote 
\begin{equation}
\label{e.Gk.def}
\mathbf{G}_{\h_0} \coloneqq  \begin{pmatrix} 
\Id
& 0
\\ \h_0
& \Id
\end{pmatrix}\,.
\end{equation}
Comparing~\eqref{e.Ak0.formula} with~\eqref{e.bigA.def}, we see that the subtraction of a constant anti-symmetric matrix~$\h_0$ ``commutes'' with the coarse-graining operation in the sense that it leaves~$\s(U;\a)$ and~$\s_*(U;\a)$ unchanged and simply subtracts~$\h_0$ from~$\k(U;\a)$: for every~$\h_0\in \R^{d\times d}_{\skew}$, 
\begin{equation}
\left\{
\begin{aligned}
& \s(U;\a-\h_0) = \s(U;\a) 
\\ & 
\s_*(U;\a-\h_0) = \s_*(U;\a) 
\\ & 
\k(U;\a-\h_0) = \k(U;\a) - \h_0\,.
\end{aligned}
\right.
\label{e.what.h0.does}
\end{equation}
We also define
\begin{equation*}
\b_{\h_0}(U;\a) \coloneqq 
\b(U;\a - \h_0) =( \s + (\k-\h_0)^t\s_*^{-1}(\k-\h_0)) (U)
\end{equation*}
We will use this in our analysis to ``center'' the quantity~$J$. 

\smallskip

We next observe that the leaves the eigenvalues of ratios of pairs of coarse-grained matrices unchanged. Indeed, for every matrix~$\h_0 \in \R^{d\times d}$ (not necessarily skew-symmetric) and every pair of symmetric matrices~$\mathbf{D},\mathbf{E} \in \R^{2d\times 2d}_{\sym}$ with~$\mathbf{D}$ being positive definite, if we denote 
\begin{equation*}  
\mathbf{D}_{\h_0}  \coloneqq  \mathbf{G}_{\h_0}^t  \mathbf{D} \mathbf{G}_{\h_0} \qand
\mathbf{E}_{\h_0}  \coloneqq  \mathbf{G}_{\h_0}^t  \mathbf{E} \mathbf{G}_{\h_0} \,,
\end{equation*}
then
\begin{equation*}  
\mathbf{D}_{\h_0}^{-1}  \mathbf{E}_{\h_0} 
=
\mathbf{G}_{\h_0}^{-1}
\mathbf{D}^{-1}
\mathbf{E}
\mathbf{G}_{\h_0}
\,.
\end{equation*}
The matrix~$\mathbf{G}_{\h_0}$ is invertible with the inverse~$\mathbf{G}_{\h_0}^{-1} = \mathbf{G}_{-\h_0}$. Thus~$\mathbf{D}_{\h_0}^{-1}  \mathbf{E}_{\h_0}$ 
and~$\mathbf{D}^{-1}\mathbf{E}$ 
are similar. It follows that~$\mathbf{D}^{-\nicefrac 12}  \mathbf{E} \mathbf{D}^{-\nicefrac 12}$ and~$\mathbf{D}_{\h_0}^{-\nicefrac 12}  \mathbf{E}_{\h_0}  \mathbf{D}_{\h_0}^{-\nicefrac 12}$ have the same set of eigenvalues. 
In particular, 
\begin{equation}
\label{e.Ak.vs.A.diff}
\bigl | \mathbf{D}^{-\nicefrac 12}\mathbf{E}\mathbf{D}^{-\nicefrac 12}  \bigr | 
=
\bigl | \mathbf{D}_{\h_0}^{-\nicefrac 12}  \mathbf{E}_{\h_0}  \mathbf{D}_{\h_0}^{-\nicefrac 12} \bigr | 
\qand
\bigl | \mathbf{D}^{-\nicefrac 12}\mathbf{E}\mathbf{D}^{-\nicefrac 12} - \Itwod \bigr | 
=
\bigl | \mathbf{D}_{\h_0}^{-\nicefrac 12}  \mathbf{E}_{\h_0}  \mathbf{D}_{\h_0}^{-\nicefrac 12}- \Itwod \bigr | 
\,.
\end{equation}
We next show that, if~$\mathbf{E}_1, \mathbf{E}_2\in\R^{2d\times2d}$ satisfy, for some~$\theta \in [1,\infty)$, 
\begin{equation}  
\label{e.Eone.and.Etwo}
\mathbf{E}_1 \leq \theta \mathbf{E}_2 \quad \mbox{and} \quad 
\mathbf{E}_j  \coloneqq  
\G_{-\k_j}^t 
\begin{pmatrix} 
\s_j & 0 \\ 0 & \s_{*,j}^{-1}
\end{pmatrix}
\G_{-\k_j}\,,
\quad
\s_j\,, \s_{*,j}^{-1} \in \R_{\mathrm{sym}}^{d\times d}\,, \ 
\k_j \in \R^{d\times d}\,,
\  j \in \{ 1,2\}\,,
\end{equation}
then we have
\begin{equation}
\label{e.Eone.vs.Etwo.one}
\theta \s_2 \geq \s_1 + (\k_1 - \k_2)^t \s_{*,1}^{-1} (\k_1 - \k_2)
\quad \mbox{and} \quad 
\s_{*,2} \leq \theta \s_{*,1} \,.
\end{equation}
To see this, we observe that, for every~$\h \in \R^{d\times d}$, 
\begin{equation*}
0\leq 
\G_{\h}^{t} (\theta \mathbf{E}_2 - \mathbf{E}_1) \G_{\h} 
=
\begin{pmatrix} 
\theta \b_{2,\h} -\b_{1,\h} & (\k_1 - \h)^t \s_{*,1}^{-1} - \theta (\k_2 - \h)^t \s_{*,2}^{-1} \\
\s_{*,1}^{-1} (\k_1 - \h) - \theta \s_{*,2}^{-1} (\k_2 - \h)&  \theta\s_{*,2}^{-1} - \s_{*,1}^{-1}
\end{pmatrix}
\,,
\end{equation*}
and thus
\begin{equation*}
\b_{1,\h} \leq \theta \b_{2,\h} \quad \forall \h \in \R^{d\times d} \qand \s_{*,2} \leq \theta \s_{*,1} 
 \,.
\end{equation*}
Taking~$\h = \k_2$ yields~\eqref{e.Eone.vs.Etwo.one}.

\subsection{Rewriting the ellipticity assumption~\ref{a.ellipticity}}
\label{ss.dagger}

We next show that the ellipticity condition~\ref{a.ellipticity} can be rewritten in an equivalent way, which turns out to be more flexible and convenient. 
The new condition is:
\begin{enumerate}[label=(\textrm{P\arabic*\textdagger})]
\setcounter{enumi}{1}
\item \emph{Coarse-grained ellipticity on large scales.} \label{a.ellipticity.dagger}
There exist~$\bfE\in \R^{2d\times 2d}_{\sym}$, an exponent~$\gamma\in [0,1)$, an increasing function~$\Psi_\S:\R_+ \to [1,\infty)$, a constant~$K_{\Psi_\S}\in (1,\infty)$ satisfying the growth condition
\begin{equation}
\label{e.Psi.S.growth.prime}
t \Psi_\S(t) \leq \Psi_\S(K_{\Psi_\S}  t), \quad \forall t \in [1,\infty)\,,
\end{equation}
and a nonnegative random variable~$\S$ satisfying the bound
\begin{equation}
\label{e.S.integrability.prime}
\P \bigl[ \S > t   \bigr]
\leq 
\frac{1}{\Psi_\S(t)}
\,, 
\quad \forall t\in (0,\infty) \,,
\end{equation}
such that, for every~$m\in\Z$, 
\begin{equation}
\label{e.ellipticity.bfE}
3^m \geq \S 
\quad\implies \quad
\bfA(z+ \cu_k;\a) \leq 3^{\gamma(m-k)} \bfE \,, \quad \forall k \in (-\infty,m]\cap \Z\,, z \in 3^k\Zd\cap \cu_m
\,.
\end{equation}
\end{enumerate}

The difference between~\ref{a.ellipticity} and~\ref{a.ellipticity.dagger} is that we have replaced the condition~\eqref{e.ellipticity}, which we recall is 
\begin{equation}
\label{e.ellipticity.again}
3^m \geq \S 
\quad\implies \quad
\Lambda_{\nf\gamma2,\infty}(\cu_m;\a) \leq \Lambda_0 
\qand
\lambda_{\nf\gamma2,\infty}(\cu_m;\a) \geq \lambda_0 
\,,
\end{equation}
with the condition~\eqref{e.ellipticity.bfE}
where~$\bfE\in \R^{2d\times 2d}_{\sym}$ is a given matrix. Let us show that these are essentially equivalent. First, in view of~\eqref{e.bfA.bounds.diag} and~\eqref{e.lambdas.downscales}, for every~$k \in (-\infty,m]\cap \Z$ and~$z \in 3^k\Zd\cap \cu_m$, 
\begin{equation*}
\bfA(z+\cu_k)
\leq
\begin{pmatrix} 
\s + 2 \k^t\s_*^{-1}\k  
& 0 
\\ 0 
& 2 \s_*^{-1}
\end{pmatrix}(z+\cu_k)
\leq 
3^{\gamma(m-k)} 
\begin{pmatrix} 
2\Lambda_{\nf\gamma2,\infty}
& 0 
\\ 0 
& 2\lambda_{\nf\gamma2,\infty}^{-1} 
\end{pmatrix}(\cu_m) 
\,.
\end{equation*}
Thus~\eqref{e.ellipticity.again} implies~\eqref{e.ellipticity.bfE} for the particular choice of matrix 
\begin{equation*}
\bfE = \begin{pmatrix} 
2\Lambda_0 \Id 
& 0 
\\ 0 
& 2\lambda_0^{-1} \Id
\end{pmatrix}
\,.
\end{equation*}
On the other hand, given a matrix~$\bfE$, we write it in block form as
\begin{equation*}
\bfE = 
\begin{pmatrix} 
\mathbf{E}_{11}
& \mathbf{E}_{12}
\\ \mathbf{E}_{21}
& \mathbf{E}_{22}
\end{pmatrix} \,,
\end{equation*}
where~$\mathbf{E}_{ij} \in \R^{d\times d}$ for~$i,j\in \{1,2\}$, and we define
\begin{equation}
\label{e.E.naught.components}
\left\{
\begin{aligned}
& \s_{*,0}  \coloneqq  \mathbf{E}_{22}^{-1}\,,
\\ &
\k_0  \coloneqq  -\mathbf{E}_{22}^{-1} \mathbf{E}_{21}\,,
\\ &
\s_{0}  \coloneqq 
\mathbf{E}_{11} - 
\mathbf{E}_{12}\mathbf{E}_{22}^{-1} \mathbf{E}_{21}\,,
\\ &
\b_{0}  \coloneqq 
\mathbf{E}_{11}
\,.
\end{aligned}
\right.
\end{equation}
In other words, we have given names to the block entries of~$\bfE$ so that
\begin{equation}
\label{e.bfE.matform}
\bfE =
\begin{pmatrix} 
\s_0 + \k_0^t\s_{*,0}^{-1}\k_0
& -\k_0^t\s_{*,0}^{-1}
\\ - \s_{*,0}^{-1}\k_0 
& \s_{*,0}^{-1} 
\end{pmatrix}
=
\begin{pmatrix} 
\b_0 
& -\k_0^t\s_{*,0}^{-1}
\\ - \s_{*,0}^{-1}\k_0 
& \s_{*,0}^{-1} 
\end{pmatrix}
\,.
\end{equation} 
It follows that
\begin{equation*} 
\bfE  
\leq 
2 \begin{pmatrix} \s_0 + \k_0^t\s_{*,0}^{-1}\k_0 & 0 \\  0 & \s_{*,0}^{-1}\end{pmatrix} \,.
\end{equation*}
Therefore,~\eqref{e.ellipticity.bfE} implies~\eqref{e.ellipticity.again} for~$\Lambda_0  \coloneqq  2 |\s_0 + \k_0^t\s_{*,0}^{-1}\k_0|$ and~$\lambda_0 \coloneqq  2|\s_{*,0}^{-1}|^{-1}$. 

\smallskip

In most of the paper we take~\ref{a.ellipticity.dagger} as our ellipticity assumption, instead of~\ref{a.ellipticity}. In view of the discussion in Section~\ref{ss.skew}, above, we also define the \emph{intrinsic ellipticity ratio}~$\Theta$ by
\begin{equation}
\label{e.Theta}
\Theta  \coloneqq  \min_{\h \in \R^{d\times d}_{\mathrm{skew}}}
\bigl|\s_{*,0}^{-\nicefrac12} (\s_0 + (\k_0-\h)^t\s_{*,0}^{-1}(\k_0-\h)) \s_{*,0}^{-\nicefrac12}\bigr|\,. 
\end{equation}
We let~$\h_0 \in \R^{d\times d}_{\mathrm{skew}}$ to be a minimizer of the above quantity, so that
\begin{equation*}  
\Theta = \bigl|\s_{*,0}^{-\nicefrac12} (\s_0 + (\k_0-\h_0)^t\s_{*,0}^{-1}(\k_0-\h_0)) \s_{*,0}^{-\nicefrac12}\bigr|
\,.
\end{equation*}
We also define the ellipticity constants~$0<\lambda_0\leq \Lambda_0<\infty$ by
\begin{equation}
\label{e.lambda.Lambda.def}
\lambda_0  \coloneqq  
\bigl|\s_{*,0}^{-1}\bigr|^{-1} 
\quad \mbox{and} \quad 
\Lambda_0 \coloneqq  
\min_{\h\in \R^{d\times d}_{\mathrm{skew}}}
\bigl| \s_0 + (\k_0-\h)^t\s_{*,0}^{-1}(\k_0-\h) \bigr|
\end{equation}
and the \emph{aspect ratio}~$\Pi$ by
\begin{equation}
\label{e.Pi}
\Pi  \coloneqq  \frac{\Lambda_0}{\lambda_0}\,.
\end{equation}
In view of the discussion in Section~\ref{ss.skew}, for any fixed~$\h_0\in \R^{d\times d}$, the assumption~\eqref{e.ellipticity.bfE} is equivalent to
\begin{equation*}
3^m \geq \S 
\implies 
\bfA_{\h_0}(z+\cu_n;\a) 
\leq 
3^{\gamma(m-n)}
\mathbf{E}_{0,\h_0}\,,
\quad
\forall n \in \Z \cap (-\infty,m]\,,
\ z \in 3^n\Zd \cap \cu_m\,,
\end{equation*}
where~$\mathbf{E}_{0,\h_0}  \coloneqq  \mathbf{G}_{\h_0}^t  \mathbf{E}_0 \mathbf{G}_{\h_0}$.
Since the transformation~$\bfE\mapsto \mathbf{E}_{0,\h_0}$ leaves~$\s_0$ and~$\s_{*,0}$ unchanged, the ellipticity contrast~$\Theta$ is invariant under this transformation, while the new value of~$\Pi$ is bounded above by
\begin{equation*}
\bigl|\s_0 + (\k_0-\h_0)^t\s_{*,0}^{-1}(\k_0-\h_0) \bigr|\bigl| \s_{*,0}^{-1} \bigr |
\leq 
2\Pi + 2 \bigl |\h_0^t\s_{*,0}^{-1}\h_0\bigr | \bigl| \s_{*,0}^{-1} \bigr |\,.
\end{equation*}

We will discover (in Lemma~\ref{l.bfE.bounds} below) that, since~$\bfE$ dominates the coarse-grained matrices by~\ref{a.ellipticity}, we must have the ordering 
\begin{equation*}
\s_0 \geq \s_{*,0}\,.
\end{equation*}
It follows that
\begin{equation*}
1\leq \Theta\leq \Pi  =  \frac{\Lambda_0}{\lambda_0}\,.
\end{equation*}
As shown above, the constants~$\lambda_0^{-1}$ and~$\Lambda_0$ defined under assumption~\ref{a.ellipticity.dagger} are no larger than twice their values under assmption~\ref{a.ellipticity}, and for this reason we do not distinguish between them.

\smallskip 

Why do we have two competing notions of ellipticity ratio,~$\Theta$ and~$\Pi$? 
The classical ellipticity assumption~\eqref{e.UE.intro} simultaneously controls two different things, which we need to keep separate: (i) the ratio of the size of~$\a(x)$ to its smallest eigenvalue at each point~$x$; and (ii) the ratio of matrices~$\a(x)$ and~$\a(y)$ at two different points~$x$ and~$y$. It is important in our setting to distinguish these two because, obviously, homogenization should be concerned with~(ii) but not with~(i). In the terminology here, it is the aspect ratio~$\Pi$ which measures~(i),  while the ellipticity ratio~$\Theta$ measures~(ii).

\subsection{Two-sided bounds from one-sided bounds}

In the next two lemmas, we formalize an important observation, which is that if~$\Theta-1$ is small and~$\bfA(U) \leq \bfE$, then in fact the difference~$\bfE - \bfA(U)$ must also be small. We get a two-sided bound from a one-sided one for free if the ellipticity contrast is small. This is related to the idea that the difference (or ratio) of~$\s(U)$ and~$\s_*(U)$ should represent the ``uncertainty'' in the coarse-graining map. Since this is an essentially algebraic fact, we present a slightly more general statement that will prove useful. 

\begin{lemma}
\label{l.bfE.bounds}
Suppose that~$\mathbf{E}_1,\mathbf{E}_{*,1} \in \R^{2d\times 2d}_{\mathrm{sym}}$ are symmetric matrices having the form 
\begin{equation}
\label{e.E1.form}
\mathbf{E}_1 =
\begin{pmatrix} 
\s_1 + \k_1^t\s_{*,1}^{-1}\k_1 
& -\k_1^t\s_{*,1}^{-1}
\\ -  \s_{*,1}^{-1}\k_1 
& \s_{*,1}^{-1} 
\end{pmatrix} 
\qquad \mbox{and} \qquad 
\mathbf{E}_{*,1}=
\begin{pmatrix} 
\s_{*,1} + \k_1\s_1^{-1}\k_1^t 
& \k_1\s_{1}^{-1}
\\  \s_{1}^{-1}\k_1^t 
& \s_{1}^{-1} 
\end{pmatrix}
\,,
\end{equation}
satisfying the inequality 
\begin{equation}
\label{e.E1ord}
\mathbf{E}_{*,1} \leq \mathbf{E}_{1}\,,
\end{equation}
where~$\s_1, \s_{*,1}, \k_1 \in \R^{d \times d}$ with~$\s_1$ and~$\s_{*,1}$ being positive definite.
Then~$\s_{*,1} \leq \s_{1}$ and, by denoting
\begin{equation}
\label{e.Theta.tilde.def}
\tilde{\Theta}_1 \coloneqq  
\bigl|\s_{*,1}^{-\nicefrac12}\s_{1} \s_{*,1}^{-\nicefrac12}\bigr|\,,
\end{equation}
we have the inequalities
\begin{equation}
\label{e.symm.k.quad.small}
\bigl| \s_{*,1}^{-\nicefrac12} (\k_1 + \k_1^t) \s_{*,1}^{-\nicefrac12}  \bigr|^2  
\leq
(\tilde{\Theta}_1-1)
\min\bigl\{ 
2 
\,,
(\tilde{\Theta}_1-1) 
\bigr\} 
\end{equation}
and
\begin{equation}
\label{e.grok}
\tilde{\Theta}_1 - 1 
\leq
\bigl|
\mathbf{E}_{*,1}^{-\nicefrac12}\mathbf{E}_{1}\mathbf{E}_{*,1}^{-\nicefrac12}
-\Itwod
\bigr|
\leq 
6 \bigl(\tilde{\Theta}_1 - 1 \bigr)
\,.
\end{equation}
\end{lemma}
\begin{proof}
The inequality~$\s_1 \geq \s_{*,1}$ is immediate from~\eqref{e.E1ord}, since the bottom right matrices in the block forms in~\eqref{e.E1.form} must be ordered. Consequently, 
\begin{equation} \label{e.s.vs.sstar.one}
\bigl |\s_{*,1}^{-\nicefrac12} \s_1 \s_{*,1}^{-\nicefrac12} - \Id \bigr|   = \tilde{\Theta}_1 - 1
\quad \mbox{and} \quad 
\bigl|\s_{*,1}^{\nicefrac12} (\s_{*,1}^{-1} - \s_1^{-1})  \s_{*,1}^{\nicefrac12} \bigr| 
\leq 
\tilde{\Theta}_1 - 1
\,.
\end{equation}
Let~$\h_0 \in \R^{d\times d}$ and recall the definition of~$\mathbf{G}_{\h_0}$ from~\eqref{e.Gk.def}. Observe that
\begin{equation*}
\tilde{\mathbf{E}}_1= \tilde{\mathbf{E}}_{1,\h_0} 
 \coloneqq  
\mathbf{G}_{\h_0}^t\mathbf{E}_1\mathbf{G}_{\h_0}
=
\begin{pmatrix} 
\s_1 + (\k_{1}-\h_0)^t\s_{*,1}^{-1}(\k_{1}-\h_0)
& -(\k_{1}-\h_0)^t\s_{*,1}^{-1}
\\ -  \s_{*,1}^{-1}(\k_{1}-\h_0)
& \s_{*,1}^{-1} 
\end{pmatrix}
\end{equation*}
and
\begin{equation*}
\tilde{\mathbf{E}}_{*,1}
=
\tilde{\mathbf{E}}_{*,1,\h_0} 
 \coloneqq  
\mathbf{G}_{\h_0}^t\mathbf{E}_{*,1} \mathbf{G}_{\h_0}
= 
\begin{pmatrix} 
\s_{*,1} + (\k_{1}+\h_0^t)\s_1^{-1}(\k_{1}+\h_0^t)^t 
& (\k_{1}+\h_0^t)\s_{1}^{-1}
\\  \s_{1}^{-1}(\k_{1}+\h_0^t)^t
& \s_{1}^{-1} 
\end{pmatrix}\,.
\end{equation*}
Then~$\tilde{\mathbf{E}}_{1}$ and~$\tilde{\mathbf{E}}_{*,1}$ are positive and~\eqref{e.E1ord} is equivalent to~$\tilde{\mathbf{E}}_{*,1}\leq \tilde{\mathbf{E}}_{1}$.
Moreover, as discussed after~\eqref{e.Gk.def}, 
the matrices~$\tilde{\mathbf{E}}_1 \tilde{\mathbf{E}}_{*,1}^{-1}$ and~$\mathbf{E}_1\mathbf{E}_{*,1}^{-1}$ are similar. In particular, since both~$\tilde{\mathbf{E}}_{1}$ and~$\tilde{\mathbf{E}}_{*,1}$ are symmetric,
\begin{equation}
\label{e.rats.dont.change}
\bigl|
\tilde{\mathbf{E}}_{*,1}^{-\nicefrac12}\tilde{\mathbf{E}}_{1}\tilde{\mathbf{E}}_{*,1}^{-\nicefrac12}
-\Itwod
\bigr|
=
\bigl|
\mathbf{E}_{*,1}^{-\nicefrac12}\mathbf{E}_{1}\mathbf{E}_{*,1}^{-\nicefrac12}
-\Itwod
\bigr|\,.
\end{equation}

\smallskip

We next make a reduction to the case that~$\k_1$ is symmetric. Let~$\k_{1,s}$ and~$\k_{1,a}$ denote, respectively, the symmetric and anti-symmetric parts of~$\k_1$ and take~$\h_0= \k_{1,a}$ in the above definitions. 
We then find that
\begin{equation*}
\tilde{\mathbf{E}}_1
=
\begin{pmatrix} 
\s_1 + \k_{1,s}\s_{*,1}^{-1}\k_{1,s}
& -\k_{1,s}\s_{*,1}^{-1}
\\ -  \s_{*,1}^{-1}\k_{1,s}
& \s_{*,1}^{-1} 
\end{pmatrix}
\quad \mbox{and} \quad
\tilde{\mathbf{E}}_{*,1}
= 
\begin{pmatrix} 
\s_{*,1} + \k_{1,s}\s_1^{-1}\k_{1,s} 
& \k_{1,s}\s_{1}^{-1}
\\  \s_{1}^{-1}\k_{1,s} 
& \s_{1}^{-1} 
\end{pmatrix}\,.
\end{equation*}
These matrices satisfy the same assumptions as~$\mathbf{E}_{1}$ and~$\mathbf{E}_{*,1}$, and the symmetric part of~$\k_1$ is unchanged, but the anti-symmetric part of~$\k_1$ has been removed. In view of~\eqref{e.rats.dont.change}, we assume without loss of generality that~$\k_1$ is symmetric; otherwise we replace the pair~$(\mathbf{E}_{1},\mathbf{E}_{*,1})$ by~$(\tilde{\mathbf{E}}_{1},\tilde{\mathbf{E}}_{*,1})$.

\smallskip

We next take~$\h_0=\eta\k_1$ for~$\eta \in\R$ in the definition of~$\mathbf{G}_{\h_0}$ and eventually optimize over the parameter~$\eta$. The inequality~$\tilde{\mathbf{E}}_{*,1}\leq \tilde{\mathbf{E}}_{1}$ reads as
\begin{equation*}
\begin{pmatrix} 
\s_1 - \s_{*,1} +(1-\eta)^2\k_{1}\s_{*,1}^{-1}\k_{1}- (1+\eta)^2\k_{1}\s_1^{-1}\k_{1}
& -(1-\eta) \k_{1}\s_{*,1}^{-1}+(1+\eta) \k_{1}\s_{1}^{-1}
\\ -(1-\eta)\s_{*,1}^{-1} \k_{1}+(1+\eta) \s_{1}^{-1}\k_{1}
& \s_{*,1}^{-1}-\s_{1}^{-1} 
\end{pmatrix}\geq 
0\,.
\end{equation*}
The nonnegativity of the top left block says that
\begin{equation*}
(1+\eta)^2\k_{1}\s_{1}^{-1}\k_{1}
\leq 
\s_1 - \s_{*,1} +(1-\eta)^2\k_{1}\s_{*,1}^{-1}\k_{1}
\leq
\bigl(\tilde{\Theta}_1-1 \bigr)\s_{*,1} 
+\tilde{\Theta}_1 (1-\eta)^2\k_{1}\s_{1}^{-1}\k_{1} \,.
\end{equation*}
Rearranging, we obtain
\begin{equation*}
\bigl ( (1+\eta)^2 - \tilde{\Theta}_1 (1-\eta)^2 \bigr )\k_{1}\s_{1}^{-1}\k_{1}
\leq 
\bigl(\tilde{\Theta}_1-1\bigr)\s_{*,1}\,.
\end{equation*}
We now optimize in~$\eta$ by taking~$\eta \coloneqq (\tilde{\Theta}_1+1)(\tilde{\Theta}_1-1)^{-1}$ to get
\begin{equation}
\label{e.verygood.quads}
\s_{*,1}^{-\nicefrac12}\k_{1}\s_{1}^{-1}\k_{1}\s_{*,1}^{-\nicefrac12}
\leq 
\frac{\bigl(\tilde{\Theta}_1-1\bigr)^2}{4\tilde{ \Theta}_1}\Id\,.
\end{equation}
This and the bound~$\s_1\leq \tilde{\Theta}_1 \s_{*,1}$ yields
\begin{equation}
\label{e.good.bound.small.Theta}
\bigl(\s_{*,1}^{-\nicefrac12}\k_{1}\s_{*,1}^{-\nicefrac12} \bigr)^2 
=
\s_{*,1}^{-\nicefrac12}\k_{1}\s_{*,1}^{-1}\k_{1}\s_{*,1}^{-\nicefrac12}
\leq 
\tilde{\Theta}_1 
\s_{*,1}^{-\nicefrac12}\k_{1}\s_{1}^{-1}\k_{1}\s_{*,1}^{-\nicefrac12}
\leq 
\frac14\bigl(\tilde{\Theta}_1-1\bigr)^2 \Id
\,.
\end{equation}
This gives~\eqref{e.symm.k.quad.small} if~$\tilde{\Theta}_1\leq 3$. 

\smallskip

For the case~$\tilde{\Theta}_1>3$, we use the above factorization with~$\h_0=-\k_1^t = - \k_1$ and have
\begin{equation*}
\tilde{\mathbf{E}}_{1,-\k_1}
=
\begin{pmatrix} 
\s_1 + 4 \k_1  \s_{*,1}^{-1} \k_1 
& -2 \k_1 \s_{*,1}^{-1}
\\ -2 \s_{*,1}^{-1} \k_1
& \s_{*,1}^{-1} 
\end{pmatrix}
\quad \mbox{and} \quad
\tilde{\mathbf{E}}_{*,1,-\k_1}
= 
\begin{pmatrix} 
\s_{*,1} 
& 0
\\  0
& \s_{1}^{-1} 
\end{pmatrix}\,.
\end{equation*}
We again select a free parameter~$\eta \in (0,\infty)$ and set~$ p  \coloneqq  \eta^{-1} e$ and~$q  \coloneqq  \eta \s_1^{-\nf 12} \s_{*,1}^{\nf 12}e'$ for~$e,e' \in \Rd$ with~$|e|,|e'| \leq 1$. Then, using~$\tilde{\mathbf{E}}_{1,-\k_1} \geq \tilde{\mathbf{E}}_{*,1,-\k_1}$, we get 
\begin{align*} 
0 & 
\leq
\begin{pmatrix} 
 p \\ q
\end{pmatrix} 
\cdot
\bigl(  \tilde{\mathbf{E}}_{*,1,-\k_1}^{-\nicefrac12} \tilde{\mathbf{E}}_{1,-\k_1} \tilde{\mathbf{E}}_{*,1,-\k_1}^{-\nicefrac12}
-\Itwod \bigr) 
\begin{pmatrix} 
 p \\ q
\end{pmatrix} 
\notag \\ & 
= 
\begin{pmatrix} 
 p \\ q
\end{pmatrix} 
\cdot
\begin{pmatrix} 
\s_{*,1}^{-\nicefrac12}\s_1\s_{*,1}^{-\nicefrac12} -\Id
+ 4\s_{*,1}^{-\nicefrac12}\k_{1}\s_1^{-1}\k_{1}\s_{*,1}^{-\nicefrac12} 
& -2\s_{*,1}^{-\nicefrac12}\k_{1}\s_{*,1}^{-1}\s_{1}^{\nicefrac12}
\\ -2\s_{1}^{\nicefrac12}\s_{*,1}^{-1}\k_{1}\s_{*,1}^{-\nicefrac12}
& \s_{1}^{\nicefrac12}\s_{*,1}^{-1} \s_{1}^{\nicefrac12}-\Id
\end{pmatrix}
\begin{pmatrix} 
 p \\ q
\end{pmatrix} 
\notag \\ &
\leq 
\eta^{-2} 
\bigl| 
\s_{*,1}^{-\nicefrac12}\s_1\s_{*,1}^{-\nicefrac12} -\Id
+ 4\s_{*,1}^{-\nicefrac12}\k_{1}\s_1^{-1}\k_{1}\s_{*,1}^{-\nicefrac12} 
\bigr|
+\eta^2
\bigl| \Id - \s_{*,1}^{\nicefrac12}\s_{1}^{-1} \s_{*,1}^{\nicefrac12} \bigr| 
- 4  e \cdot \s_{*,1}^{-\nicefrac12}\k_{1}\s_{*,1}^{-\nf12} e'
\,.
\end{align*}
Moving the last term on the right side to the left side, then optimizing over all unit vectors~$e,e'$ and using the bounds~\eqref{e.verygood.quads} and~$|\Id - \s_{*,1}^{\nicefrac12}\s_{1}^{-1} \s_{*,1}^{\nicefrac12}| \leq 1$, we get
\begin{equation*}
4 \bigl| \s_{*,1}^{-\nicefrac12}\k_{1}\s_{*,1}^{-\nf12} \bigr| 
\leq 
2\eta^{-2} 
(\tilde{\Theta}_1-1) 
+\eta^2 \,. 
\end{equation*}
Optimizing in~$\eta$ leads to the choice~$\eta= 2^{\nf14} (\tilde{\Theta}_1-1)^{\nf14}$, which yields 
\begin{equation}
\label{e.good.bound.large.Theta}
\bigl| \s_{*,1}^{-\nicefrac12}\k_{1}\s_{*,1}^{-\nf12} \bigr|^2 
\leq 
\frac12 (\tilde{\Theta}_1-1) \,.
\end{equation}
The bounds~\eqref{e.good.bound.small.Theta} and~\eqref{e.good.bound.large.Theta} together give~\eqref{e.symm.k.quad.small}.

\smallskip

To obtain~\eqref{e.grok}, we use~$\tilde{\mathbf{E}}_{1,-\k_1} \geq \tilde{\mathbf{E}}_{*,1,-\k_1}$ together with bounds~\eqref{e.s.vs.sstar.one} and~\eqref{e.verygood.quads}, and  get 
\begin{align*}
\Bigl| \tilde{\mathbf{E}}_{*,1,-\k_1}^{-\nicefrac12} \tilde{\mathbf{E}}_{1,-\k_1} \tilde{\mathbf{E}}_{*,1,-\k_1}^{-\nicefrac12}
-\Itwod
\Bigr|
&
=
\Biggl| \begin{pmatrix} 
\s_{*,1}^{-\nicefrac12}\s_1\s_{*,1}^{-\nicefrac12} -\Id
+ 4\s_{*,1}^{-\nicefrac12}\k_{1}\s_1^{-1}\k_{1}\s_{*,1}^{-\nicefrac12} 
& -2\s_{*,1}^{-\nicefrac12}\k_{1}\s_{*,1}^{-1}\s_{1}^{\nicefrac12}
\\ -2\s_{1}^{\nicefrac12}\s_{*,1}^{-1}\k_{1}\s_{*,1}^{-\nicefrac12}
& \s_{1}^{\nicefrac12}\s_{*,1}^{-1} \s_{1}^{\nicefrac12}-\Id
\end{pmatrix}
\Biggr|
\notag \\ &
\leq 
2 |\s_{*,1}^{-\nicefrac12}\s_1\s_{*,1}^{-\nicefrac12} -\Id|
+
8| \s_{*,1}^{-\nicefrac12}\k_{1}\s_1^{-1}\k_{1}\s_{*,1}^{-\nicefrac12} |
+ 2|\s_{1}^{\nicefrac12}\s_{*,1}^{-1} \s_{1}^{\nicefrac12}-\Id|
\notag \\ &
\leq 
4 (\tilde{\Theta}_1 - 1)
+
\frac{2\bigl(\tilde{\Theta}_1-1\bigr)^2}{\tilde{ \Theta}_1}
\leq 
6 (\tilde{\Theta}_1 - 1)
\,.
\end{align*}
In view of~\eqref{e.rats.dont.change} this completes the proof of~\eqref{e.grok}. 
\end{proof}

\begin{lemma}
\label{l.bfE.bfA.bounds}
Let~$\mathbf{E}_1 \in \R^{2d\times 2d}_{\mathrm{sym}}$ be of the form~\eqref{e.E1.form},  let~$\tilde{\Theta}_1$ be defined by~\eqref{e.Theta.tilde.def} and~$U$ be a bounded Lipschitz domain such that
\begin{equation}
\label{e.bfA.vs.bfEord}
\bfA(U;\a) \leq \mathbf{E}_1\,.
\end{equation}
Then 
\begin{equation}
\label{e.AEA.squash}
\max\bigl\{ 
\big| \mathbf{E}_1^{-\nicefrac12} \bfA(U;\a)\mathbf{E}_1^{-\nicefrac12} - \Itwod\bigr|
\,,\,
\big| \mathbf{E}_1^{\nicefrac12} \bfA_*^{-1}(U;\a)\mathbf{E}_1^{\nicefrac12} - \Itwod\bigr|
\bigr\}
\leq 
6 \bigl(\tilde{\Theta}_1 - 1 \bigr)
\,. 
\end{equation}
\end{lemma}
\begin{proof}
Define~$\mathbf{E}_{*,1} \coloneqq  \mathbf{R} \mathbf{E}_1^{-1} \mathbf{R}$. 
The hypothesis~\eqref{e.bfA.vs.bfEord}, the identity~\eqref{e.bigAstar.def} and the ordering~\eqref{e.bfA.bounds} imply that 
\begin{equation*}
\mathbf{E}_{*,1} 
\leq
\bfA_*(U;\a) \leq \bfA(U;\a)
\leq \mathbf{E}_{1}\,.
\end{equation*}
The inequality~\eqref{e.AEA.squash} then follows from~\eqref{e.grok}.
\end{proof}

\subsection{Stochastic bounds for the  coarse-grained matrices}

We show first that the assumption of~\ref{a.ellipticity.dagger} implies control across a range of mesoscopic scales. The~$\O_{\Psi}$ notation is defined in Appendix~\ref{ss.big.O}.

\begin{lemma}[Improving ellipticity on large mesoscales]
\label{l.ellipticity.mesoscales}
Assume that~$\P$ satisfies~\ref{a.stationarity} and~\ref{a.ellipticity.dagger}.
For every~$h \in \N$, there exists a random scale~$\S_h$ satisfying 
\begin{equation}
\label{e.Sh.O}
\S_h \leq \O_{\Psi_{\S}} \bigl ( K_{\Psi_{\S}}^{4(d+1)} 3^{h+1} \bigr)
\end{equation}
such that, for every~$m\in\Z$,
\begin{equation}
\label{e.ellipticity.mesogrid}
3^m \geq \S_h 
\implies 
\bfA(z+\cu_n) 
\leq 
3^{\gamma(m-h-n)_+}
\bfE\,,
\quad
\forall n \in \Z \cap (-\infty,m]\,,
\ z \in 3^n\Zd \cap \cu_m \,.
\end{equation}
\end{lemma}
\begin{proof}
Fix~$h\in\N$. For every~$m\in\N$ with~$m \geq h$, we have 
\begin{align*}
\lefteqn{
\P\biggl[
\sup_{n\in \Z\cap  (-\infty,m]}
3^{-\gamma(m-h-n)_+}
\sup_{z \in 3^n\Zd \cap \cu_m}
\bigl| \bfE^{-\nicefrac12} \bfA(z+\cu_n) \bfE^{-\nicefrac12} \bigr|
> 1
\biggr ]
} \qquad &
\notag \\ & 
=
\P\biggl[
\sup_{n\in \Z\cap  (-\infty,m-h]}
3^{-\gamma(m-h-n)}
\sup_{z \in 3^n\Zd \cap \cu_m}
\bigl| \bfE^{-\nicefrac12} \bfA(z+\cu_n) \bfE^{-\nicefrac12} \bigr|
> 1
\biggr ]
\notag \\ & 
\leq 
\sum_{z' \in 3^{h}\Zd\cap\cu_m}
\P\biggl[
\sup_{n\in \Z\cap  (-\infty,m-h]}
3^{-\gamma(m-h-n)}
\sup_{z \in z'+3^n\Zd \cap \cu_{m-h}}
\bigl| \bfE^{-\nicefrac12} \bfA(z+\cu_n) \bfE^{-\nicefrac12} \bigr|
> 1
\biggr ]
\notag \\ & 
=
3^{d(m-h)}
\P\biggl[
\sup_{n\in \Z\cap  (-\infty,m-h]}
3^{-\gamma(m-h-n)}
\sup_{z \in 3^n\Zd \cap \cu_{m-h}}
\bigl| \bfE^{-\nicefrac12} \bfA(z+\cu_n) \bfE^{-\nicefrac12} \bigr|
> 1
\biggr ]
\notag \\ &
\leq 
3^{d(m-h)}\P\bigl[ \S > 3^{m-h} \bigr ]
\notag \\ &
\leq 
3^{-(m-h)}
\bigl( \Psi_{\S} \bigl(  K_{\Psi_{\S}}^{-4(d+1)}  3^{m-h} \bigr)  \bigr )^{-1}  \,.
\end{align*}
In the above display, subadditivity was used to get the first equality, a union bound gives the next line, then stationarity in the following line, and finally, assumption~\ref{a.ellipticity.dagger} and~\eqref{e.eat.the.poly} in the last line.  Define
\begin{equation*}
\S_h \coloneqq  \sup
\biggl \{
3^m \,:\,
m\in\Z\,, \ 
\sup_{n\in \Z\cap  (-\infty,m]}
3^{-\gamma(m-h-n)_+}
\sup_{z \in 3^n\Zd \cap \cu_m}
\bigl| \bfE^{-\nicefrac12} \bfA(z+\cu_n) \bfE^{-\nicefrac12} \bigr|
> 1
\biggr \}
\,.
\end{equation*}
By another union bound and using~\eqref{e.Psi.S.growth}, we obtain, for every~$m\in \N$ with~$m\geq h+1$, 
\begin{equation*}
\P \bigl[ \S_h \geq 3^m \bigr ]
\leq
\sum_{n=m}^\infty 
3^{-(n-h)}
\bigl( \Psi_{\S} \bigl(  K_{\Psi_{\S}}^{-4(d+1)}  3^{n-h} \bigr)  \bigr )^{-1} 
\leq 
\frac{1}{
\Psi_{\S} \bigl( K_{\Psi_{\S}}^{-4(d+1)}  3^{m-h} \bigr) }
\,.
\end{equation*}
This completes the proof of the lemma. 
\end{proof}

In view of~\eqref{e.bigAstar.def}, it is convenient to define~$\mathbf{E}_{*,0} 
 \coloneqq 
\bigl( \mathbf{R} \bfE \mathbf{R} \bigr)^{-1}$, where~$\bfE$ is the matrix from assumption~\ref{a.ellipticity.dagger}. 

\smallskip

We next show that~\ref{a.ellipticity.dagger} gives us control over all finite moments of the coarse-grained matrices.

\begin{lemma}[Upper bounds for coarse-grained matrices]
\label{l.bfA.upperbounds}
Assume that~$\P$ satisfies~\ref{a.stationarity} and~\ref{a.ellipticity.dagger}.
For every~$m\in\N$ and~$n\in\Z$ with~$n\leq m$, 
\begin{align}
\label{e.bfA.crude.moment}
& \sup_{z\in3^n \Zd \cap \cu_m}
\bigl| \bfE^{-\nicefrac12} \bfA(z+\cu_n) \bfE^{-\nicefrac12} \bigr| 
\leq 
3^{\gamma(m-n)} \bigl( 1 + \O_{\Psi_\S} \bigl( 3^{\gamma-m} \bigr) \bigr)\,,
\\ & 
\sup_{z\in3^n \Zd \cap \cu_m}
\bigl| \mathbf{E}_{*,0}^{\nicefrac12} \bfA_*^{-1} (z+\cu_n) \mathbf{E}_{*,0}^{\nicefrac12} \bigr| 
\leq 
3^{\gamma(m-n)} \bigl( 1 + \O_{\Psi_\S} \bigl( 3^{\gamma-m} \bigr) \bigr)\,.
\label{e.bfA.star.crude.moment}
\end{align}
In particular, for every~$n\in\N$, 
\begin{equation}
\label{e.bfA.crude.moment.E}
\bigl| \bfE^{-\nicefrac12} \bfA(\cu_n) \bfE^{-\nicefrac12} \bigr|
\leq
1 + \O_{\Psi_\S} ( 3^{\gamma-n} )
\qand 
\bigl| \mathbf{E}_{*,0}^{\nicefrac12}  \bfA_*^{-1} (\cu_n) \mathbf{E}_{*,0}^{\nicefrac12}  \bigr|
\leq
1 + \O_{\Psi_\S} ( 3^{\gamma-n} )
\,.
\end{equation}
\end{lemma}
\begin{proof}
The assumption~\ref{a.ellipticity.dagger} implies that, for every~$m\in\N$ and~$n\in\Z$ with~$n\leq m$, 
\begin{align*}
\sup_{z\in3^n \Zd \cap \cu_m}\bigl| \bfE^{-\nicefrac12} \bfA(z+\cu_n) \bfE^{-\nicefrac12} \bigr| 
&
\leq
3^{\gamma(m-n)}  + \Bigl( \frac{3\S}{3^n} \Bigr)^{\!\gamma} \indc_{\{ \S > 3^m \}}
\notag \\ & 
\leq
3^{\gamma(m-n)} + 3^{\gamma(m-n+1)} \Bigl( \frac{\S}{3^m} \Bigr)
\notag \\ & 
\leq 
3^{\gamma(m-n)} \bigl( 1 + \O_{\Psi_\S} \bigl( 3^{\gamma-m} \bigr) \bigr)\,.
\end{align*}
Notice that~\eqref{e.ellipticity.bfE} implies by the identity~\eqref{e.bigAstar.def} that, for every~$3^m\geq \mathcal{S}$,  
\begin{equation*}
\bfA_{*}^{-1} (z+\cu_n) \leq 3^{\gamma(m-n)} \mathbf{E}_{*,0}^{-1} \,,
\quad \forall n\in \Z \cap (-\infty,m], \ z \in 3^n \Zd \cap \cu_m\,.
\end{equation*}
Using this the bounds in~\eqref{e.bfA.star.crude.moment} for~$\bfA_*^{-1}$ are obtained similarly. 
\end{proof}

In view of~\eqref{e.bfA.crude.moment.E} and~\eqref{e.Psi.moments.bound}, we have the boundedness of all finite moments of~$|\bfA(\cu_n)|$. 
In fact, by~\eqref{e.bfA.crude.moment}, this can be extended to~$|\bfA(U)|$ for any bounded Lipschitz domain~$U\subseteq\Rd$ by partitioning the domain into triadic cubes and using subadditivity and the fact that~$\gamma <1$. This kind of argument can be found in the proof of Lemma~\ref{l.crude.moments} below, so we do not give it here. 

\smallskip

According to Lemma~\ref{l.bfA.upperbounds}, the assumption of~\ref{a.ellipticity.dagger} implies that~$\bfA(U)$ have all finite moments bounded. We may therefore define 
\begin{equation*}
\bfAhom(U)
 \coloneqq 
\E \bigl[ \bfA(U) \bigr]
\quad \text{and} \quad
\bfAhom_*^{-1}(U)
 \coloneqq 
\E \bigl[ \bfA_*^{-1} (U) \bigr]
\,.
\end{equation*}
We let~$\shom(U)$, $\shom_*(U)$,~$\khom(U)$ and~$\bhom(U)$ denote deterministic matrices which satisfy:
\begin{equation}
\label{e.meet.the.homs}
\left\{
\begin{aligned}
& \shom_*(U) = \E \bigl[ \s_*^{-1}(U) \bigr]^{-1} \,, 
\\ & 
\shom_*(U) \khom(U) 
=  \E \bigl[ \s_*^{-1}(U) \k(U) \bigr] \,, 
\\ & 
\bhom(U)  \coloneqq  \shom(U) + \khom^t(U) \shom_*^{-1}(U) \khom(U)
=
\E \bigl[ \s(U) + \k^t(U) \s_*^{-1}(U) \k(U) \bigr]\,.
\end{aligned}
\right.
\end{equation}
We see immediately that the first line of~\eqref{e.meet.the.homs} defines~$\shom_*(U)$, the second line defines~$\khom(U)$, and the third line defines~$\shom(U)$ and~$\bhom(U)$. 
In other words, these matrices are defined in such a way that 
\begin{equation*}
\bfAhom(U) 
= 
\begin{pmatrix} 
\shom + \khom^t\shom_*^{-1}\khom
& - \khom^t \shom_*^{-1}
\\ -\shom_*^{-1}\khom
& \shom_*^{-1}
\end{pmatrix}(U)
\,.
\end{equation*}
We also have
\begin{equation*}
\bfAhom_*^{-1} (U) 
= 
\begin{pmatrix} 
\shom_*^{-1}
& - \shom_*^{-1}\khom
\\ - \khom^t\shom_*^{-1}
& \shom + \khom^t\shom_*^{-1}\khom
\end{pmatrix}(U)
\,.
\end{equation*}
In most cases, taking the expectation of a natural expression involving the coarse-grained matrices   amounts to putting bars over each matrix. 

\smallskip

Observe that, by subadditivity, for every~$m,n \in \N$ with~$m\geq n$, 
\begin{equation}
\label{e.bfAhom.subadd}
\bfAhom(\cu_m) \leq \bfAhom(\cu_n) 
\end{equation}
and, in particular, 
\begin{equation} 
\label{e.monotone.s}
\bhom(\cu_m) \leq \bhom(\cu_n)\,,\quad
\shom(\cu_m) \leq \shom(\cu_n) 
\quad \mbox{and} \quad
\shom_{*}(\cu_m) \geq \shom_{*}(\cu_n) \,. 
\end{equation}
The monotonicity of~$\bhom(\cu_n)$ and~$\shom_*(\cu_n)$ is immediate from~\eqref{e.bfAhom.subadd}, while the monotonicity of~$\shom(\cu_n)$ is due to~\eqref{e.bfAhom.subadd} and~\eqref{e.Eone.vs.Etwo.one}. 

\smallskip

By Lemma~\ref{l.bfA.upperbounds} and~\eqref{e.Psi.moments.bound},
for every~$n\in\N$,  
\begin{equation*}
\bigl( 1 + 3^{3-n} K_{\Psi_\S}^2
\bigr)^{-1} 
\mathbf{E}_{*,0}
\leq 
\bfAhom_*(\cu_n)
\leq 
\bfAhom(\cu_n) 
\leq 
\bigl( 1 + 3^{3-n} K_{\Psi_\S}^2
\bigr)
\bfE
\,.
\end{equation*}
By~\ref{a.ellipticity.dagger} and~Lemma~\ref{l.bfE.bounds}, we have that 
\begin{equation*}
|\mathbf{E}_{*,0}^{-\nf12} \bfE \mathbf{E}_{*,0}^{-\nf12}| 
\leq 
1 + 6
(\Theta-1) 
\,,
\end{equation*}
and so we deduce that, for every~$n\in\N$, 
\begin{equation}
\label{e.bfAhom.by.E0}
\bigl( 1 + 3^{3-n} K_{\Psi_\S}^2
\bigr)^{-1} \bfAhom(\cu_n) 
\leq 
\bfE
\leq
(1 + 6 (\Theta-1))
\bigl( 1 + 3^{3-n} K_{\Psi_\S}^2
\bigr)
\bfAhom(\cu_n)\,.
\end{equation}

\smallskip

We next discuss sensitivity estimates for the random matrix~$\bfA(U)$. 
It is immediate from the variational characterization of~$\bfA(U)$ in~\eqref{e.J.P0.Dirichlet} that, with~$\mathrm{D}_U$ defined in~\eqref{e.malliavin.mult} above, 
\begin{equation}
\label{e.bfA.malliavin}
\bigl| \mathrm{D}_U \bigl( P \cdot \bfA(U) P) \bigr| 
\leq 
P \cdot \bfA(U) P\,, 
\qquad \forall P\in\R^{2d}\,.
\end{equation}
Indeed, if we fix a bounded Lipschitz domain~$U\subseteq \Rd$ and let~$\a_1,\a_2\in\Omega(U)$ with corresponding~$2d$-by-$2d$ matrices~$\bfA_1$ and~$\bfA_2$ satisfying~$\| 
\bfA_2^{-\nicefrac 12}
\bfA_1
\bfA_2^{-\nicefrac 12}
\|_{L^\infty(U)} \leq 2$. Note that this implies that~$\s_2^{-\nf12} \s_1 \s_2^{-\nf12} \geq \frac12\Id$ in~$U$ and hence~$L^2_{\a_1,\mathrm{pot},0}(U)=L^2_{\a_2,\mathrm{pot},0}(U)$ and~$L^2_{\a_1,\mathrm{sol},0}(U)=L^2_{\a_2,\mathrm{sol},0}(U)$. We compute 
\begin{align*}
P \cdot \bfA_1(U) P 
&
= 
\inf \biggl\{ 
\fint_{U}
(X + P) \cdot \bfA_1 (X + P)
\, : \, X \in  L^2_{\a_1,\mathrm{pot},0} (U) \times  L^2_{\a_1,\mathrm{sol},0}(U)  \biggr\}
\notag \\ &
\leq 
\| 
\bfA_2^{-\nicefrac 12}
\bfA_1
\bfA_2^{-\nicefrac 12}
\|_{L^\infty(U)}
\inf \biggl\{ 
\fint_{U}  
(X + P) \cdot \bfA_2 (X + P)
 :  X \in  L^2_{\a_2,\mathrm{pot},0}(U) {\times}  L^2_{\a_2,\mathrm{sol},0}(U)  \biggr\}
\notag \\ &
=
\| 
\bfA_2^{-\nicefrac 12}
\bfA_1
\bfA_2^{-\nicefrac 12}
\|_{L^\infty(U)}
\cdot
P \cdot \bfA_2(U) P 
\,.
\end{align*}
This implies~\eqref{e.bfA.malliavin}. 
It is immediate from the definitions that 
\begin{equation}
\label{e.bfA.local}
\bfA(U) \quad \mbox{is~$\mathcal{F}(U)$--measurable.}
\end{equation}
The sensitivity estimate~\eqref{e.bfA.malliavin} and the locality~\eqref{e.bfA.local} of~$\bfA(U)$ will allow us to apply our mixing assumption~\ref{a.CFS} to sums of coarse-grained matrices.

\smallskip

We next apply the CFS condition~\ref{a.CFS} to sums of the coarse-grained matrix~$\bfA(U)$. Since these random variables are not bounded, we need to apply a cutoff function and use the previous lemma to control the resulting error. 

\begin{lemma}[Concentration for sums of~$\bfA$'s]
\label{l.bfA.CFS}
Assume that~$\P$ satisfies~\ref{a.stationarity},~\ref{a.ellipticity.dagger} and~\ref{a.CFS}.
For every~$k,m,n\in\N$ with~$\beta m < n < k \leq m$ and~$z\in 3^k\Zd\cap \cu_m$,
\begin{equation}
\label{e.bfA.CFS.upper}
\avsum_{z'\in 3^n\Zd\cap (z+\cu_k)}\!\!\!\!
\bigl(  \bfA(z'{+}\cu_n) - \bfAhom(\cu_n) \bigr) 
\indc_{\{ \S \leq 3^m\} }
\\ \leq 
\O_{\Psi}\Bigl( 4 \cdot 3^{\gamma(m-n)}3^{-\nu(k-n) }\bfE  \Bigr) \,.
\end{equation}
\end{lemma} 
\begin{proof}
Denote~$T  \coloneqq  3^{\gamma(m-n)}$.  We take a smooth cutoff function~$\varphi:\R_+ \to [0,1]$ satisfying
\begin{equation*}
\indc_{[0,T]} \leq \varphi \leq \indc_{[0,2T]}\,, \quad | \varphi'| \leq 2T^{-1}. 
\end{equation*}
Denote, for each~$z\in 3^n\Zd\cap \cu_m$,  
\begin{equation*}
\tilde{\mathbf{A}}_z  \coloneqq 
\varphi\bigl( \bigl| 
\bfE^{-\nicefrac12} \bfA(z+\cu_n) \bfE^{-\nicefrac12}  \bigr|\bigr)
\bfE^{-\nicefrac12} \bfA(z+\cu_n) \bfE^{-\nicefrac12} 
\,.
\end{equation*}
It is clear that~$\bigl| \tilde{\mathbf{A}}_z\bigr| \leq 2T$. 
According to~\eqref{e.bfA.malliavin}, we have that 
\begin{equation}
\label{e.malliavin.after.cutoff}
\bigl| \mathrm{D}_{z+\cu_n}  \tilde{\mathbf{A}}_z\bigr|
\leq 
\bigl|\tilde{\mathbf{A}}_z\bigr|
\bigl(1 + 
\varphi'\bigl( \bigl| 
\bfE^{-\nicefrac12}
\bfA(z+\cu_n) 
\bfE^{-\nicefrac12}
\bigr|\bigr)\bigr)
\leq 4T. 
\end{equation}
By~\eqref{e.bfA.local}, it is clear that~$\tilde{\mathbf{A}}_z$ is~$\mathcal{F}(z+\cu_n)$--measurable. 
We may, therefore, apply~\ref{a.CFS} to obtain,
for every~$k\in\N\cap (n,m]$ and~$z\in 3^k\Zd \cap \cu_m$,
\begin{equation}
\label{e.tilde.dub}
\biggl| \avsum_{z'\in 3^n\Zd\cap (z+\cu_k)}
\bigl(  \tilde{\mathbf{A}}_{z'} - \E \bigl[ \tilde{\mathbf{A}}_{z'} \bigr] \bigr) \biggr|
\leq 
\O_{\Psi}\bigl( 4T 3^{-\nu(k-n)} \bigr)\,.
\end{equation}
We also have that 
\begin{equation*}
 \bfA(z+\cu_n) \indc_{\{ \S \leq 3^m\} }  =  \tilde{\mathbf{A}}_z \indc_{\{ \S \leq 3^m\} }
\quad \mbox{and} \quad 
\bfAhom(\cu_n) \geq  \E \bigl[ \tilde{\mathbf{A}}_z \bigr]
\,,
\end{equation*}
and therefore 
\begin{equation*}  
\avsum_{z'\in 3^n\Zd\cap (z+\cu_k)}
\bfE^{-\nicefrac12}\bigl(  \bfA(z'+\cu_n) - \bfAhom(\cu_n) \bigr) \bfE^{-\nicefrac12} 
\indc_{\{ \S \leq 3^m\} }
\leq 
\!\!\!
\avsum_{z'\in 3^n\Zd\cap (z+\cu_k)} \! \!
\bigl(  \tilde{\mathbf{A}}_{z'} - \E \bigl[ \tilde{\mathbf{A}}_{z'} \bigr] \bigr) 
\,.
\end{equation*}
Therefore,~\eqref{e.bfA.CFS.upper} follows by~\eqref{e.tilde.dub}. 
The proof is complete. 
\end{proof}

\subsection{Renormalization of the ellipticity assumption}
\label{ss.renormalization}

We next show that we can \emph{renormalize} the assumptions~\ref{a.stationarity},~\ref{a.ellipticity.dagger} and~\ref{a.CFS}. 
To formalize this, we introduce the mapping~$D_{n_0}:\Omega \to \Omega$ given by dilation by~$3^{n_0}$,
\begin{equation}
\label{e.dilation.def}
(D_{n_0} \a )(x) = \a (3^{n_0}x)
\end{equation}
and we define~$\P_{n_0}$ by
\begin{equation}
\label{e.Pn0}
\P_{n_0} \coloneqq  \mbox{\,the pushforward of~$\P$ under~$D_{n_0}$.}
\end{equation}
We show that~$\P_{n_0}$ satisfies the same assumptions as~$\P$---with the ellipticity matrix~$\bfE$ replaced by~$\bfAhom(\cu_{n_0-l_0})$ for a sufficiently large scale separation parameter~$l_0$---and some suitable modifications to the other parameters (we must also slightly enlarge~$\gamma$ and replace~$\S$ by a new minimal scale~$\S'$ which has integrability quantified by a new function~$\Psi_{\S'}$ given in terms of the~$\Psi_\S$ and~$\Psi$).

\smallskip

The main point is that the ellipticity ratio for~$\bfAhom(\cu_{n_0-l_0})$ may be much smaller than for~$\bfE$. 
It is natural therefore to define, for each~$n\in\N$, a new parameter~$\Theta_n$, which we call the \emph{renormalized ellipticity ratio~$\Theta_n \in [1,\infty)$ at scale}~$3^n$, which is the ellipticity ratio for~$\bfAhom(\cu_n)$. 
In view of~\eqref{e.Theta} and~\eqref{e.bigA.def}, we define it by 
\begin{equation}
\label{e.Theta.n.def}
\Theta_n  \coloneqq  
\min_{\h_0 \in \R^{d\times d}_{\mathrm{skew}}}
\bigl| ( \shom_{*}^{-\nicefrac12} \bhom_{\h_0}\,\shom_{*}^{-\nicefrac12}  ) (\cu_n) \bigr|
\,.
\end{equation}
Observe that, by Lemma~\ref{l.bfE.bounds}, 
\begin{equation*}
\bigl| ( \shom_*^{-\nf12} \shom\shom_*^{-\nf12})(\cu_n) - \Id \bigr|
\leq 
\Theta_n -1 
\leq 
3\bigl| ( \shom_*^{-\nf12} \shom\shom_*^{-\nf12})(\cu_n) - \Id \bigr| 
\,.
\end{equation*}
We also define 
\begin{equation}
\label{e.hat.Theta.m}
\XiDet_n 
 \coloneqq 
\det\bigl( (\shom \shom_*^{-1})(\cu_n) \bigr)^{\nf1d} 
=
\det\bigl( \bfAhom (\cu_n) \bigr)^{\nf1d} 
\,.
\end{equation}
Note that both~$n\mapsto \Theta_n$ and~$n\mapsto \XiDet_n$ are monotone decreasing, since the matrices~$\shom(\cu_n)$ and~$\shom_{*}^{-1}(\cu_n)$ are, by~\eqref{e.monotone.s}. 
Using that~$\shom_*(\cu_n) \leq \shom(\cu_n)$, we find that 
\begin{equation}
\label{e.Theta.n.by.Theta.n.hat} 
\XiDet_n -1 
\leq
\bigl| ( \shom_*^{-\nf12} \shom\shom_*^{-\nf12})(\cu_n) - \Id \bigr| 
\leq 
\XiDet_n^d-1  \,.
\end{equation}
The first inequality of~\eqref{e.Theta.n.by.Theta.n.hat} follows from the arithmetic-geometric mean inequality; the second inequality follows from the fact that~$\tr(A) \leq \det( (\Id+A)) - 1$ for any nonnegative matrix~$A$.

\smallskip

For convenience, we introduce an exponent~$\mu$ used throughout the rest of the paper, defined by 
\begin{equation}
\label{e.def.mu}
\mu \coloneqq  (\nu-\gamma)(1-\beta)\,.
\end{equation}

\begin{proposition}[Renormalization of the assumptions]
\label{p.renormalization.P}
Suppose~$\P$ satisfies~\ref{a.stationarity},~\ref{a.ellipticity.dagger} and~\ref{a.CFS} with~$\nu>\gamma$.
Let~$\rho \in (\gamma, \min\{ \nu,1\})$ and~$\delta>0$. 
Suppose that~$l_0\in\N$ satisfies
\begin{equation}
\label{e.l0.condition}
l_0 \geq 
\frac{1}{\rho-\gamma}
\Bigl(1  + \frac d{\mu}\Bigr) \bigl( 6 + 
\log_3 ( \delta^{-1} \Theta )
\bigr)  +\frac{6}{\mu} \bigl( 1 + \log K_{\Psi} \bigr) \,.
\end{equation}
For every~$n_0\in\N$ with~$n_0 \geq l_0 + 2 \log K_\Psi$, the pushforward~$\P_{n_0}$ of~$\P$ under the dilation map given in~\eqref{e.dilation.def} satisfies the assumptions~\ref{a.stationarity},~\ref{a.ellipticity.dagger} and~\ref{a.CFS}, where the parameters~$(\gamma,\Psi_\S,\bfE)$ in assumption~\ref{a.ellipticity.dagger} are replaced by~$(\rho,\Psi_{\S'},(1+\delta) \bfAhom(\cu_{n_0-l_0}))$ and~$\Psi_{\S'}$ is defined by
\begin{equation}
\label{e.new.Psi.S}
\Psi_{\S'}(t) \coloneqq \frac12 \min\bigl\{ \Psi_\S(3^{n_0}t) ,\Psi  ( t^{\mu })\bigr\}
\,.
\end{equation}
\end{proposition}
\begin{proof}
The conditions~\ref{a.stationarity} and~\ref{a.CFS} for~$\P_{n_0}$ are immediate from their validity for~$\P$, as the dilation causes no harm; the only condition which needs to be checked is therefore~\ref{a.ellipticity.dagger}, and this is the content of Lemma~\ref{l.renormalize.ellipticity}, below. Indeed, taking~$\mathcal{S}^\prime \coloneqq  3^{-n_0}\max\{ \S, \mathcal{R}_{n_0} \}$ with~$\mathcal{R}_{n_0}$ as in Lemma~\ref{l.renormalize.ellipticity} yields~\ref{a.ellipticity.dagger} with~$\Psi_{\mathcal{S}^\prime}$ given by~\eqref{e.new.Psi.S}. 
\end{proof}

The function~$\Psi_{\S'}$ satisfies~$t\Psi_{\S'}(t) \leq \Psi_{\S'}(K_{\Psi_{\S'}}t)$ for all~$t\geq 1$ with~$K_{\Psi_{\S'}}$ given by
\begin{equation*}
K_{\Psi_{\S'}}  \coloneqq  
\max\big\{ K_{\Psi_{\S}} , K_{\Psi}^{\lceil \nicefrac1\mu \rceil} \bigr\}\,.
\end{equation*}
This follows from the definition of~$\Psi_{\S'}$ in~\eqref{e.new.Psi.S}  and~\eqref{e.eat.poly.factors}. 
The new values of the ellipticity ratios~$\Theta$ and~$\Pi$ are at most~$(1+\delta)^2\Theta_{n_0-l_0}$ and~$256(1+\delta)^2\Pi$, respectively. This follows immediately from the definition~\eqref{e.Theta.n.def} of~$\Theta_n$, and, respectively,~\eqref{e.bfAhom.by.E0} and the bound~$n_0-l_0\geq \log K_\Psi$.

\smallskip

We turn to the main part of the proof of Proposition~\ref{p.renormalization.P}, formulated in the following lemma.

\begin{lemma}[Renormalization of ellipticity]
\label{l.renormalize.ellipticity}
Suppose~$\P$ satisfies~\ref{a.stationarity},~\ref{a.ellipticity.dagger} and~\ref{a.CFS} with~$\nu>\gamma$.
Let~$\rho \in (\gamma,1)$ and~$\delta>0$. 
Define
\begin{equation}
\label{e.ell.naught.def}
l_0  \coloneqq   
\biggl\lceil \frac{1}{\rho-\gamma}
\Bigl(1  + \frac d{\mu}\Bigr) \bigl( 6 + 
\log_3 ( \delta^{-1} \Theta )
\bigr)  +\frac{6}{\mu} \bigl( 1 + \log K_{\Psi} \bigr) 
\biggr\rceil
\,.
\end{equation}
For every~$n\in\N$ with~$n \geq l_0 + 2\log K_\Psi$, 
there exists a minimal scale~$\mathcal{R}_n$ satisfying
\begin{equation}
\label{e.new.scale}
\mathcal{R}_n^\mu = \O_{\Psi} ( 3^{n\mu})
\,,
\end{equation}
such that, for every~$m\in\N$ with~$m\geq n$, 
\begin{align}
\label{e.new.ellipticity}
\lefteqn{ 
3^m \geq \max\{ \S , \mathcal{R}_n \} 
} \quad &
\notag \\ &
\implies 
\bfA(z+\cu_k) 
\leq 
\bigl( 1 + 
\delta 3^{\rho(m-k)} \bigr)
\bfAhom(\cu_{n-l_0})
\,,
\quad
\forall k \in \Z \cap (-\infty,m]\,,
\ z \in 3^k\Zd \cap \cu_m\,.
\end{align}
\end{lemma}
\begin{proof}
Let~$h\in\N$ be a parameter to be selected below. Let~$m,n ,l_0 \in\N$ with~$m \geq n$ and~$n-l_0 \geq \log K_{\Psi_\S}$. 
Fix~$k,l \in \N$ with~$m-h < k \leq m$ and~$\max\{ n - l_0, \beta k \} < l < k$.
Using~\eqref{e.bfAhom.by.E0},~\eqref{e.bfA.CFS.upper} and subadditivity, we find that, for every~$z\in 3^k\Zd \cap \cu_m$,
\begin{align}
\label{e.renormalize.subadd}
\bfA(z+\cu_k) \indc_{\{ \S \leq 3^m \}}
&
\leq
\avsum_{y\in 3^l \Zd \cap (z+\cu_k)} 
\bfA(y+\cu_{l})
\indc_{\{ \S \leq 3^m \}}
\notag \\ & 
\leq
\bfAhom(\cu_{l}) 
+
\O_{\Psi} \bigl( 4 \cdot 3^{\gamma(m-l)} 3^{-\nu(k-l)} \bfE \bigr) 
\notag \\ & 
\leq
\bigl( 
1
+
\O_{\Psi}
\bigl(3^6\Theta( 1 + 3^{-l} K_{\Psi_\S}^2) \cdot 3^{\gamma(m-l)} 3^{-\nu(k-l)} \bigr) 
\bigr)
\bfAhom(\cu_{l}) 
\notag \\ & 
\leq
\bigl( 
1
+
\O_{\Psi}
\bigl( 3^7 \Theta 3^{\gamma(m-l)} 3^{-\nu(k-l)} \bigr) 
\bigr)
\bfAhom(\cu_{l}) 
\,.
\end{align}
In the last line of~\eqref{e.renormalize.subadd} we used that~$3^{-l} K_{\Psi_\S}^2 \leq 1$, which is valid by the assumption that~$l \geq n-l_0 \geq 2\log K_{\Psi_\S}\geq 2 \log K_{\Psi_\S}/\log 3$. 
By a union bound, we deduce that, for every~$k\in \N \cap ((m-h) \vee l ,m]$ and~$T\geq 1$, 
\begin{align*}
\lefteqn{
\P \Biggl[ \sup_{z\in 3^k\Zd \cap  \cu_m} 
\bfA(z+\cu_k)\indc_{\{ \S \leq 3^m \}}
\not\leq 
\bigl( 1 + 
3^7\Theta 3^{\gamma(m-l)} 3^{-\nu(k-l)}T \bigr)  
\bfAhom(\cu_{l}) 
\Biggr]
} \qquad &
\notag \\ & 
\leq 
\sum_{z\in 3^k\Zd \cap \cu_m} \!\!
\P \Bigl[ 
\bfA(z+\cu_k) \indc_{\{ \S \leq 3^m \}}
\not\leq
\bigl( 1 + 
3^7 \Theta 3^{\gamma(m-l)} 3^{-\nu(k-l)} T\bigr) 
\bfAhom(\cu_{l}) 
\Bigr]
\leq
\frac{3^{d(m-k)}}{\Psi (T )}
\,.
\end{align*}
We set~$l  \coloneqq  n-l_ 0+ \lceil \beta (k-n + l_0) \rceil$. Observe that this choice of~$l$ satisfies~$\max\{ n - l_0, \beta k \} < l < k$ announced above. We have
\begin{equation*}
m-l
\geq
m-k + (1-\beta)(k-n+l_0)-1\,.
\end{equation*}
We also define 
\begin{equation*}
T \coloneqq 
3^{-7} 3^{\gamma - \nu} \delta \Theta^{-1}3^{\mu(k-n+l_0)} 3^{(\rho-\gamma)(m-k)}\,,
\end{equation*}
and observe that, since~$ m < k + h$ and~$\rho > \gamma$ and~$\gamma-\nu \geq -\nu \geq -\frac d2$,
\begin{align*}
T
& 
= 
3^{-7} 3^{\gamma - \nu} \delta \Theta^{-1}
3^{\mu(l_0 -h)} 3^{\mu(k+h - n)} 3^{(\rho-\gamma)(m-k)}
\geq
3^{-7-d}\delta \Theta^{-1} 3^{\mu (l_0 - h)}
3^{\mu(m-n)}
\geq 1\,,
\end{align*}
provided that 
\begin{equation*}
l_0 \geq 
h +
\frac {7+d+ \log_3 (\delta^{-1} \Theta) }{\mu} 
\,.
\end{equation*}
Note that this also implies that~$l < m-h$. Since~$\nu > \gamma$, we have
\begin{equation*} 
\Theta 3^{-7} 3^{\gamma(m-l)} 3^{-\nu(k-l)} T 
\leq 
3^{\nu-\gamma} \cdot 3^{-7}  \Theta 3^{\gamma(m-k) + (\gamma- \nu)(1-\beta)(k-n + l_0)} T 
\leq 
\delta 3^{\rho(m-k)}
\,.
\end{equation*}
Substituting this choice of~$T$ into the inequalities above yields
\begin{align*}
\P \biggl[ \sup_{z\in 3^k\Zd \cap  \cu_m} 
\bfA(z+\cu_k)\indc_{\{ \S \leq 3^m \}}
\not\leq 
\bigl( 1 + 
\delta 3^{\rho(m-k)} \bigr)
\bfAhom(\cu_{l}) 
\biggr]
&
\leq
\frac{3^{d(m-k)}}{\Psi (3^{-7-d}  \delta \Theta^{-1} 3^{\mu (l_0 - h)}
3^{\mu(m-n)})}
\,.
\end{align*}
Summing over~$k \in \{ m-h+1,\ldots,m\}$ and using~\eqref{e.Psi.growth},~\eqref{e.eat.it} with~$p=1$, and a union bound, we obtain
\begin{align*}
&
\P \Biggl[ 
\exists k \in \N \cap [m-h+1, m ],
\,
\sup_{z\in 3^k\Zd \cap  \cu_m} 
\bfA(z+\cu_k)
\indc_{\{ \S \leq 3^m \}}
\not\leq  
(1+\delta3^{\rho(m-k)}) \bfAhom(\cu_{n - l_0}) 
\Biggr]
\\ & \qquad 
\leq 
\frac{3^{dh}}{\Psi (3^{-7-d}  \delta \Theta^{-1} 3^{\mu (l_0 - h)}
3^{\mu(m-n)})}
\leq 
\frac{1}{\Psi ( 3^{-7-d}  K_{\Psi}^{-6} 3^{-dh}  \delta \Theta^{-1} 3^{\mu (l_0 - h)}
3^{\mu(m-n)})}
\,.
\end{align*}
If we impose another restriction on~$l_0$, namely that 
\begin{equation*}
l_0 \geq \Bigl(1  + \frac d{\mu}\Bigr) h +\frac{6}{\mu}  \log K_{\Psi} + \frac 1{\mu} \log_3 (\delta^{-1}\Theta)+ \frac {d+8}{\mu} 
\,,
\end{equation*}
then~$K_{\Psi}^{-6} 3^{-7-d} 3^{-dh} \delta \Theta^{-1} 3^{\mu (l_0 - h)} \geq 1$ and we  therefore  arrive at
\begin{multline*}
\P \Biggl[ \exists k \in \N \cap [m-h+1, m ], 
\,
\sup_{z\in 3^k\Zd \cap  \cu_m} 
\bfA(z+\cu_k)
\indc_{\{ \S \leq 3^m \}}
\not\leq  
(1+\delta3^{\rho(m-k)}) \bfAhom(\cu_{n - l_0}) 
\Biggr]
\\
\leq
\frac{1}{\Psi (3^{\mu(m-n)})}
\,.
\end{multline*}
If we define the minimal scale~$\mathcal{R}_n$ by
\begin{multline*}
\mathcal{R}_n
 \coloneqq 
\sup \biggl\{ 3^m \,:\, 
m\geq n, \ 
\exists k \in \N \cap [m-h+1, m ],
\\
\sup_{z\in 3^k\Zd \cap  \cu_m} 
\bfA(z+\cu_k)
\indc_{\{ \S \leq 3^m \}}
\not\leq  
\bigl (1+\delta3^{\rho(m-k)}\bigr ) \bfAhom(\cu_{n-l_0}) 
\biggr\}
\,,
\end{multline*}
then the previous display implies that~\eqref{e.new.scale} holds. 

\smallskip

For the small scales, we proceed more crudely: by~\eqref{e.bfAhom.by.E0}, we have that
\begin{equation*}
3^m \geq \S \implies
\sup_{ k \in \N \cap [-\infty, m-h]} \,
\sup_{z\in 3^k\Zd \cap  \cu_m} 
\bfA(z+\cu_k)
\indc_{\{ \S \leq 3^m \}}
\leq 
3^{\gamma(m-h)}
\bfE
\leq
\delta 3^{\rho(m-h)}
\bfAhom(\cu_{n-l_0}) 
\,,
\end{equation*}
provided that we choose~$h>0$ large enough so that
\begin{equation*}
\bigl( 1 + 3^{3-l_0} K_{\Psi_\S}^2
\bigr) \bigl(1 + 6(\Theta -1) \bigr)  
\bigr)  
\leq 
\delta 3^{(\rho-\gamma)h}\,.
\end{equation*}
It suffices to take any~$h\in\N$ satisfying
\begin{equation*}
h \geq 
\frac{1}{\rho-\gamma}
\Bigl( 5 + 
\log_3 ( \delta^{-1} \Theta )
\Bigr)
\,.
\end{equation*}
This completes the proof of~\eqref{e.new.ellipticity} and of the lemma. 
\end{proof}

\subsection{Intrinsic geometry, adapted cubes, and consequences of subadditivity}
\label{ss.subadditivity}

At various points in the paper, it will be necessary to change from Euclidean geometry to one that is more intrinsic to the coefficient field. For instance, in the case that the ellipticity ratio~$\Theta$ is close to unity, it is natural to use the affine geometry dictated by the symmetric part~$\shom$ of the homogenized matrix (see the discussion below~\eqref{e.bigM.naught.def}). 

\smallskip

When we are in the high contrast regime~$\Theta \gg 1$, it is still necessary to change the geometry, although the choice of the best geometry is more subtle. It turns out that, as we will see in Section~\ref{s.firstpigeon}, in this case, it is natural to use the affine geometry dictated by~$\b_0\#\, \s_{*,0}$, where~$\#$ denotes the geometric mean of two positive matrices (see Appendix~\ref{s.geomean} for the definition). 

\smallskip

Given a symmetric positive matrix~$\m_0 \in \R_{\mathrm{sym}}^{d\times d}$, we ``use the affine geometry of~$\m_0$'' by working with parallelepipeds that are adapted to~$\m_0$ instead the Euclidean cubes~$\cu_m$. These will be denoted by~$\cus_m$. The choice of~$\m_0$ will change from section to section, but to avoid burdensome notation, we do not display the dependence of~$\cus_m$ on~$\m_0$. We call these parallelepipeds the \emph{adapted cubes}. 

\smallskip

Given~$\m_0$, the adapted cubes are defined as follows. Since we work with a~$\Zd$ stationarity assumption, it will be convenient for the lattice corresponding to the adapted cubes~$\cus_m$ to have integer lattice points, at least for sufficiently large~$m$. We therefore define another symmetric matrix~$\mathbf{q}_0$ by taking a large integer~$k_0 \in\N$ (to be selected below in terms of~$d$, but not on~$\m_0$) and 
\begin{equation}
\label{e.q.naught}
\bigl (\mathbf{q}_0\bigr )_{ij}  \coloneqq 
3^{-k_0} \Bigl \lceil 3^{k_0} \bigl|\m_{0}^{-1}\bigr|^{\nicefrac12}(\m_0^{\nicefrac12})_{ij} \Bigr \rceil
\,.
\end{equation}
In other words, we take the matrix~$\bigl|\m_{0}^{-1}\bigr|^{\nicefrac12} \m_0^{\nicefrac12}$ and slightly alter each entry so that it belongs to the lattice~$3^{-k_0}\Zd$. Note that~$\mathbf{q}_0$ is symmetric, and it satisfies
\begin{equation*}
\bigl | \mathbf{q}_0 - \bigl|\m_{0}^{-1}\bigr|^{\nicefrac12} \m_0^{\nicefrac12} \bigr | 
\leq C 3^{-k_0}\,,
\end{equation*}
where~$C$ depends only on~$d$. In particular, we have
\begin{equation*}
\bigl ( 1 - C 3^{-k_0} \bigr ) \bigl|\m_{0}^{-1}\bigr|^{\nicefrac12} \m_0^{\nicefrac12} 
\leq 
\mathbf{q}_0   
\leq 
\bigl ( 1 + C 3^{-k_0} \bigr ) \bigl|\m_{0}^{-1}\bigr|^{\nicefrac12} \m_0^{\nicefrac12} 
\,.
\end{equation*}
By making~$k_0$ sufficiently large, depending only on~$d$, we obtain 
\begin{equation}
\label{e.szero.vs.qzero}
\frac{99}{100}
\mathbf{q}_0   
\leq 
\bigl|\m_{0}^{-1}\bigr|^{\nicefrac12} \m_0^{\nicefrac12} 
\leq
\frac{101}{100}
\mathbf{q}_0   
\qand
\frac{100}{101} \Id 
\leq 
\mathbf{q}_0 \leq \frac{100}{99} \bigl( \bigl|\m_{0}^{-1}\bigr| \bigl|\m_{0} \bigr| \bigr)^{\nicefrac12} \Id
\,.
\end{equation}
We then define the adapted cube~$\cus_k$ by 
\begin{equation}
\label{e.cus.k.def}
\cus_k 
 \coloneqq  
\mathbf{q}_0 (\cu_k)
=
\Bigl\{ 
x \in \Rd \,:\, 
\mathbf{q}_0^{-1} x \in \cu_k
\Bigr\} 
\,.
\end{equation}
Note that
\begin{equation*}
\mbox{$\m_0$ is a scalar matrix} \quad \implies  \quad
\mathbf{q}_0 = \Id 
\quad\implies  \quad
\cus_k = \cu_k\,.
\end{equation*}
The eccentricity of~$\cus_k$ (the ratio of largest to smallest side) is at most~$\frac{101}{99}\bigl( \bigl|\m_{0}^{-1}\bigr| \bigl|\m_{0} \bigr| \bigr)^{\nicefrac12}$, and 
\begin{equation}
\label{e.cus.bound}
\tfrac{99}{100} \cu_k
\subseteq \,\cus_k 
\subseteq \tfrac{101}{100}\bigl( \bigl|\m_{0}^{-1}\bigr| \bigl|\m_{0} \bigr| \bigr)^{\nicefrac12}\cu_k
\,.
\end{equation}
We let~$\Lat$ denote the lattice
\begin{equation}
\label{e.adapted.lattice}
\Lat \coloneqq  \mathbf{q}_0 (\Zd) = \bigl \{ \mathbf{q}_0z \,:\, z \in \Zd \bigr \} \,.
\end{equation}
Note that~$\{ z + \cus_n \,:\, z \in 3^n\Lat \cap \cus_m \}$ is a partition (up to a set of measure zero) of~$\cus_m$. 
By the construction of~$\mathbf{q}_0$, it is clear that 
\begin{equation} \label{e.adapted.stationarity}
\Lat \subseteq 
3^{-k_0} \Zd\,.
\end{equation}
This implies that~$3^n\Lat \subseteq \Zd$ for all~$n \geq k_0$. Finally, we also denote
\begin{equation}
\label{e.bigM.naught.def}
\lambda_{\m_0} \coloneqq  \bigl|\m_{0}^{-1}\bigr|^{-1} \,, \quad 
\Lambda_{\m_0} \coloneqq  \bigl|\m_{0}\bigr| 
\,, \quad 
\Pi_{\m_0} 
 \coloneqq  
\frac{\Lambda_{\m_0}}{\lambda_{\m_0}}
\qand
\mathbf{M}_0  \coloneqq  \begin{pmatrix} 
\m_0 & 0 \\ 0 & \m_0^{-1}
\end{pmatrix}
\,.
\end{equation}

As mentioned above, we work with the adapted rectangles~$\cus_m$ to avoid artificial factors of the aspect ratio~$\Pi$ from creeping into our estimates. To see why this is necessary, consider a constant-coefficient equation like~$-\nabla \cdot \s_0\nabla u = 0$. This equation can be considered to have an ellipticity ratio of one because the coefficients are constant. We could perform a simple affine change of coordinates and transform this equation to the Laplace equation. However, suppose we do not perform this change of variables, and we start performing standard elliptic estimates (such as, for instance, the Caccioppoli inequality) in standard Euclidean balls or cubes. In that case, we will see powers of the ratio~$\Pi$ of the largest to the smallest eigenvalue of~$\s_0$ appear in our estimates. These factors would not appear if we were clever enough to have changed variables beforehand. 

\smallskip

To avoid these extra factors of~$\Pi$, we must either perform the affine change of variables and then work in Euclidean balls or cubes, or else work in the original coordinate system but use the~$\s_0$-adapted balls or cubes. 

\smallskip

Using subadditivity and a Whitney-type decomposition of the adapted cubes into (normal) triadic cubes, we can reduce the upper bounds on the coarse-grained matrices in adapted cubes to those of~\eqref{e.bfA.crude.moment} and~\eqref{e.bfA.crude.moment.E}. 

\begin{lemma}[Upper bounds for~$\bfA$ in adapted cubes]
\label{l.crude.moments} 
Assume that~$\P$ satisfies~\ref{a.stationarity} and~\ref{a.ellipticity.dagger}.
Let~$\delta\in(0,1]$. There exists~$C(d)<\infty$ such that, if we define~$h' \in \N$ to be the smallest integer satisfying
\begin{equation}
\label{e.h.delta.relation}
3^{h'} \geq  \frac{C\Pi_{\m_0} }{\delta(1-\gamma)} \,,
\end{equation}
then, with~$\S_{h+h'}$ being the random scale in the statement of Lemma~\ref{l.ellipticity.mesoscales}, we have, for every~$m\in\N$, 
\begin{equation}
\label{e.crude.moments.adapted}
3^m \geq \S_{h+h'}
\implies 
\bfA(y+\cus_n) 
\leq 
(1{+}\delta)3^{\gamma(m-n-h)_+} \bfE\,,
\  
\forall n \leq m\,, \
y \in \Rd\,,
\ y+ \cus_n \subseteq \cus_m\,.
\end{equation}
\end{lemma}
\begin{proof}
Fix~$h'\in2\N$ to satisfy~\eqref{e.h.delta.relation}, where we will make the constant~$C$ sufficiently large where needed, but depending only on~$d$. Let~$m\in \N$ be such that~$3^{m} \geq \S_{h+h'}$, where~$\S_{h+h'}$ is the minimal scale given by Lemma~\ref{l.ellipticity.mesoscales}. Let~$l  \coloneqq  \lceil \log_3 (C(d) \Pi_{\m_0} ^{\nicefrac12}) \rceil$ with sufficiently large~$C(d)$ such that~$\cus_m\subseteq \cu_{m+l}$. By taking a larger constant~$C(d)$ in~\eqref{e.h.delta.relation}, we may assume that $h' \geq 2l$. Suppose that~$y\in\Rd$ and~$n\in\Z$ satisfy~$y+\cus_n \subseteq \cu_{m+l}$. 
Then~$y+\cus_n$ can be written as the disjoint union, up to a null set, of a family~$\{ V_j(y) \,:\, -\infty < j \leq n \}$ of sets such that each~$V_j(y)$ is the disjoint union of cubes of the form~$z+\cu_j$ with~$z\in 3^j\Zd$, and 
\begin{equation} 
\label{e.Whitney.property}
|V_j(y)| \leq C\Pi^{\nicefrac12} 3^{j-n} |\cus_n|
\quad \mbox{and} \quad 
\sum_{j = -\infty}^n \frac{|V_j(y)|}{|\cus_n|} = 1
\,.
\end{equation}
We can obtain such a partition recursively as follows. Define first
\begin{equation*}
V_n(y) \coloneqq  \bigcup \bigl\{ z+\cu_n \,:\, z \in 3^n \Zd\,, \; z+\cu_n \subseteq y+\cus_n \bigr\}
\end{equation*}
and then, having defined~$V_n(y),\ldots,V_j(y)$, we define~$V_{j-1}(y)$ by
\begin{equation*}
V_{j-1}(y) \coloneqq  \bigcup 
\bigl\{
z+ \cu_{j-1} \,:\,
z\in 3^{j-1}\Zd,\, 
z+ \cu_{j-1} \subseteq 
(y+\cus_n) \setminus ( V_{n}(y)  \cup \cdots \cup V_{j}(y) ) 
\bigr \}\,.
\end{equation*}
By subadditivity, the assumption~$3^{m+l}\geq \S_{h+h'}$ and~\eqref{e.ellipticity.mesogrid}, we obtain that
\begin{align*}
\bfA(y+\cus_n) 
\leq 
\sum_{j=-\infty}^n
\frac{|V_j(y)|}{|\cus_n|}
\bfA(V_j(y))
&
\leq
\sum_{j=-\infty}^n
\frac{|V_j(y)|}{|\cus_n|}
3^{\gamma(m + l - h' - h - j)_+} \bfE
\notag \\ &
\leq 
 \biggl( 1 
+
 \sum_{j=-\infty}^{n + l - h'}
(C\Pi_{\m_0} ^{\nicefrac12}3^{j-n})
3^{\gamma (n+l-h'-j)}
\biggr) 3^{\gamma(m - h - n)_+}  \bfE 
\notag \\ &
\leq 
\Bigl( 1 + \frac{C\Pi_{\m_0} }{1-\gamma} 3^{-h'}
 \Bigr)
 3^{\gamma(m - h - n)_+} \bfE 
 \,.
\end{align*} 
The second last inequality is a consequence of~\eqref{e.Whitney.property} and the triangle inequality:~$$(m + l - h' - h - j)_+ \leq (m - h - n)_+ + (n+l-h'-j)_+\,,$$
and the last inequality follows from the definition of~$l$. If we now choose the constant~$C(d)$ in~\eqref{e.h.delta.relation} large enough,~\eqref{e.crude.moments.adapted} follows.  This completes the proof.
\end{proof}

We next formalize a version of Lemma~\ref{l.bfA.CFS} in the adapted cubes. 

\begin{lemma}[Concentration for adapted cubes]
\label{l.bfA.CFS.adapted}
Assume that~$\P$ satisfies~\ref{a.stationarity},~\ref{a.ellipticity.dagger} and~\ref{a.CFS}.
There exists a constant~$C(d)<\infty$ such that, 
for every~$m,n\in\N$ with~$\beta m < n < m$, 
\begin{multline}
\label{e.bfA.CFS.adapted}
\biggl| \avsum_{z\in 3^{n}\Lat \cap \cus_m }
\bfE^{-\nicefrac12}\bigl(  \bfA(z+\cus_n) - \bfAhom(z+\cus_n) \bigr)\bfE^{-\nicefrac12} 
\biggr|
\\
\leq 
\frac{CK_{\Psi_\S}^2 \Pi_{\m_0}}{1-\gamma} 3^{\gamma(m-n)-m}
+
\O_{\Psi_{\S}} \biggl ( \frac{C\Pi_{\m_0}}{1-\gamma} 3^{\gamma(m-n) - m}\biggr )
+
\O_{\Psi}\Bigl( C\Pi_{\m_0}^{\nicefrac12}3^{-(\nu-\gamma)(m-n)}  \Bigr)
\,.
\end{multline}
\end{lemma} 
\begin{proof}
Denote~$T \coloneqq 3^{\gamma(m-n)}$ and let~$\S_{h'}$ be as in Lemma~\ref{l.ellipticity.mesoscales} where~$h'\in\N$ is chosen as small as possible such that~\eqref{e.h.delta.relation} holds. The proof is now similar to the one of Lemma~\ref{l.bfA.CFS}.
There is a slightly annoying complication caused by the fact that the assumption~\ref{a.CFS} is formulated in terms of (regular) triadic cubes, whereas now we need to apply it to sums of the adapted cubes. It turns out that this difficulty can be handled quite crudely, while only dropping a factor of~$\Pi_{\m_0}^{\nicefrac12}$, by putting the adapted cubes into groups based on membership in slightly larger Euclidean cubes. 
Throughout, we fix~$m,n\in\N$ with~$\beta m < n \leq m$.

\smallskip

We let~$n_0$ be the smallest positive integer such that~$\cus_{0} \subseteq \cu_{n_0}$; in view of~\eqref{e.cus.bound}, we have that~$3^{n_0} \leq 3\Pi_{\m_0}^{\nicefrac12}$. 
For each~$x\in\Rd$, let~$[x]$ denote the nearest point of the lattice~$3^{n+n_0}\Zd$ to~$x$, with the lexicographical ordering used as a tiebreaker if this point is not unique. 
We have then that 
\begin{equation*}
x+ \cus_n \subseteq [x]+\cu_{n+n_0+1}, \quad \forall x\in\Rd\,.
\end{equation*}
Meanwhile, each~$z\in 3^{n+n_0}\Zd$ satisfies~$z = [ x ]$ for at most~$3^{n_0+1}$ many distinct elements~$x$ belonging to the lattice~$3^n\Lat$. 

\smallskip

Select a smooth cutoff function~$\varphi:\R_+ \to [0,1]$ satisfying
\begin{equation*}
\indc_{[0,T]} \leq \varphi \leq \indc_{[0,2T]}\,, \quad | \varphi'| \leq 2T^{-1}. 
\end{equation*}
We define, for each~$z\in 3^{n+n_0}\Zd\cap \cu_{m+n_0+1}$,
\begin{equation*}
\mathbf{X}_z  \coloneqq  
\sum_{x \in z+3^n\Lat \cap (z+\cus_m)}
\varphi\bigl( \bigl| 
\bfE^{-\nicefrac12} \bfA(x+\cus_n) \bfE^{-\nicefrac12}  \bigr|\bigr)
\bfE^{-\nicefrac12} \bfA(x+\cus_n) \bfE^{-\nicefrac12} 
\,.
\end{equation*}
Since there are at most~$3^{n_0+1}\leq 9\Pi^{\nicefrac12}$ many distinct elements in the sum, we have that 
\begin{equation*}
\bigl| \mathbf{X}_z \bigr|
\leq 18 \Pi_{\m_0}^{\nicefrac12} T\,.
\end{equation*}
It is clear that~$\mathbf{X}_z$ is~$\mathcal{F}(z+\cu_{n+n_0+1})$--measurable and, similar to~\eqref{e.malliavin.after.cutoff}, we have that 
\begin{equation*}
\bigl| \mathrm{D}_{z+\cu_n}  \mathbf{X}_z\bigr|
\leq 36\Pi_{\m_0}^{\nicefrac12} T \,. 
\end{equation*}
We want to sum~$\mathbf{X}_z$ over~$z\in 3^{n+n_0}\Zd\cap \cu_{m+n_0+1}$, but there is some overlap in the cubes~$z+\cu_{n+n_0+1}$. So we break the sum into~$3^d$ many different sums, each with~$z$'s corresponding to disjoint cubes, and apply~\ref{a.CFS} to each of these. The result is 
\begin{equation*}
\biggl| \avsum_{z\in 3^{n+n_0}\Zd\cap \cu_{m+n_0+1}}
\bigl( \mathbf{X}_z - \E \bigl[ \mathbf{X}_z \bigr] \bigr) \biggr|
\leq 
\O_{\Psi}\bigl( C_d\Pi_{\m_0}^{\nicefrac12} T  \cdot 3^{-\nu(m-n)} \bigr)\,.
\end{equation*}
Since~$\indc_{\{ \S_{h'} \leq 3^m \}} \varphi\bigl( \bigl| 
\bfE^{-\nicefrac12} \bfA(z+\cus_n) \bfE^{-\nicefrac12}  \bigr|\bigr) = \indc_{\{ \S_{h'} \leq 3^m \}}$ for every~$z \in 3^n\Lat \cap \! \cus_m$, we deduce that
\begin{align} 
\notag 
\lefteqn{
\indc_{\{ \S_{h'} \leq 3^m \}} 
\avsum_{z \in 3^n\Lat \cap \cus_m}
\bfE^{-\nicefrac12} \bigl(  \bfA(z+\cus_n) - \bfAhom(z+\cus_n) \bigr) \bfE^{-\nicefrac12}
} \qquad &
\notag \\ &
=
\indc_{\{ \S_{h'} \leq 3^m \}}  
\avsum_{z \in 3^n\Lat \cap \cus_m} 
\bigl( \mathbf{X}_z - \E \bigl[ \mathbf{X}_z \bigr] \bigr) 
+
\indc_{\{ \S_{h'} \leq 3^m \}}  \E \bigl[ \bfE^{-\nicefrac12}\bfA(\cus_n)\bfE^{-\nicefrac12} \indc_{\{ \S_{h'} > 3^m \}}   \bigr]
\notag
\,.
\end{align}
It follows that
\begin{multline}
\biggl| \indc_{\{ \S_{h'} \leq 3^m \}} \! \! \! \avsum_{z \in 3^n\Lat \cap \cus_m}
\bfE^{-\nicefrac12}\bigl(  \bfA(z+\cus_n) - \bfAhom(z+\cus_n) \bigr) \bfE^{-\nicefrac12}
\biggr|
\\
\leq
\bigl|\E \bigl[ \bfE^{-\nicefrac12}\bfA(\cus_n)\bfE^{-\nicefrac12} \indc_{\{ \S_{h'} > 3^m \}}   \bigr]\bigr|
+
\O_{\Psi}\bigl( C\Pi_{\m_0}^{\nicefrac12}T 3^{-\nu(m-n)}  \bigr)\,.
\notag
\end{multline}
In view of~\eqref{e.Psi.moments.bound}, we apply~\eqref{e.crude.moments.adapted} with~$\delta=1$ and~$h=0$ to get that
\begin{align*}
\bigl|\E \bigl[ \bfE^{-\nicefrac12}\bfA(\cus_n)\bfE^{-\nicefrac12} \indc_{\{ \S_{h'} > 3^m \}}   \bigr]\bigr|
&
\leq 
6\cdot 3^{-n\gamma} \E\bigl[ \S_{h'}^{\gamma}
\indc_{\{\S_{h'}> 3^m\}}\bigr] 
\leq
 \frac{CK_{\Psi_\S}^2\Pi_{\m_0}}{1-\gamma} 3^{\gamma(m-n)-m}
 \,.
\end{align*}
Similarly, 
\begin{align*}
\avsum_{z \in 3^n\Lat \cap \cus_m}
\bigl| \bfE^{-\nicefrac12}
\bfA(z+\cus_n) 
\bfE^{-\nicefrac12}
\bigr|
\indc_{\{\S_{h'}> 3^m\}}
\leq
6\cdot 3^{-n\gamma}\S_{h'}^{\gamma}
\indc_{\{\S_{h'}> 3^m\}}
\leq 
\O_{\Psi_{\S}} \Bigl ( \frac{C\Pi_{\m_0}}{1-\gamma} 3^{\gamma(m-n) -m} \Bigr )
\,.
\end{align*}
Combining the last three displays implies~\eqref{e.bfA.CFS.adapted} and completes the proof. 
\end{proof}

In the following lemma, we use subadditivity arguments similar to the proof of Lemma~\ref{l.crude.moments} to compare the means of the coarse-grained matrices in Euclidean triadic cubes to those in the adapted cubes. 

\begin{lemma}
\label{l.tilt.to.Euc}
Assume that~$\P$ satisfies~\ref{a.stationarity},~\ref{a.ellipticity.dagger} and~\ref{a.CFS}.
There exists~$C(d) <\infty$ such that, for every~$y \in \Rd$ and~$k,n,m\in\N$ with~$k<n<m$, 
\begin{equation}
\label{e.tilt.by.Euc}
\bfAhom(y+\cus_n) 
\leq 
\bfAhom(\cu_k) 
+
\frac{C \Pi_{\m_0}^{\nicefrac12}}{1-\gamma} \bigl( 1 + K_{\Psi_{\S}}^2 3^{-k}\bigr)^{\gamma} 3^{-(n-k)}  \bfE
\end{equation}
and, if~$n \geq k_0$, 
\begin{equation}
\label{e.Euc.by.tilt}
\bfAhom(\cu_m)
\leq 
\bfAhom(\cus_n )
+
\frac{C \Pi_{\m_0}^{\nicefrac12}}{1-\gamma} \bigl( 1 + K_{\Psi_{\S}}^2 3^{-n}\bigr)^{\gamma} 3^{-(m-n)}  \bfE
\,.
\end{equation}
\end{lemma}
\begin{proof}
Fix~$n,k\in\N$ with~$k\leq n$. Following the proof of Lemma~\ref{l.crude.moments} above, we define~$V_{j}(y)$ as in that proof. Then, by taking expectations,  
\begin{equation*} 
\bfAhom(z+\cus_n)
\leq
\sum_{j=k}^{n} \frac{|V_{j}(y)|}{|\cus_n|} 
\bfAhom(V_j(y)) 
+
\sum_{j=-\infty}^{k-1} \frac{|V_{j}(y)|}{|\cus_n|} 
\bfAhom(V_j(y)) 
\,.
\end{equation*}
For the first term, we use subadditivity once more and get, by~$\Zd$-stationarity,
\begin{equation*} 
\sum_{j=k}^{n} \frac{|V_{j}(y)|}{|\cus_n|} 
\bfAhom(V_j(y)) 
\leq
\bfAhom(\cu_k)
\sum_{j=k}^{n} \frac{|V_{j}(y)|}{|\cus_n|} 
\leq 
\bfAhom(\cu_k)
\,.
\end{equation*}
The second term can be estimated using~\eqref{e.ellipticity.bfE} and H\"older's inequality as
\begin{align*} 
\sum_{j=-\infty}^{k-1} \frac{|V_{j}(y)|}{|\cus_n|} 
\bfAhom(V_j(y)) 
&
\leq 
C \Pi_{\m_0}^{\nicefrac12} 
3^{-n} 
\sum_{j=-\infty}^{k-1} 
3^{(1-\gamma) j}
\bigl( 3^{\gamma k} + 3^\gamma \E\bigl[\S^{\gamma}\bigr] \bigr)\bfE
\notag \\ &
\leq  
\frac{C \Pi_{\m_0}^{\nicefrac12}}{1-\gamma} 3^{-(n-k)} \bigl( 1 + 3^{-k} \E[\S] \bigr)^\gamma
\bfE
\leq
\frac{C \Pi_{\m_0}^{\nicefrac12}}{1-\gamma} 3^{-(n-k)} \bigl( 1 + K_{\Psi_{\S}}^2 3^{-k}\bigr)^{\gamma}
\,,
\end{align*}
where we also applied~$\E[\S] \leq 5 K_{\Psi_{\S}}^2$ implied by~\eqref{e.Psi.S.growth},~\eqref{e.S.integrability} and~\eqref{e.Psi.moments.bound}.

\smallskip

To get an estimate in the opposite direction, we need to partition the cube~$\cu_m$ into cubes of the form~$y'+\cus_n$ with~$y' \in 3^n \Lat$, plus a small boundary layer. We write 
\begin{equation*}
W \coloneqq  
\bigcup
\bigl \{ z + \cus_n \,:\,  z \in 3^n \Lat\,, \; z + \cus_{n} \subseteq \cu_m \bigr \}
\end{equation*}
and, analogously to the definition of~$V_j(y)$, we set~$W_n  \coloneqq  W$ and then, recursively, for~$j \in \Z$ with~$j \leq n$, 
\begin{equation*}
W_{j-1} \coloneqq  \bigcup 
\bigl\{
z+ \cu_{j-1} \,:\,
z\in 3^{j-1}\Zd,\, 
z+ \cu_{j-1} \subseteq 
\cu_m \setminus ( W_{n}  \cup \cdots \cup W_{j} ) 
\bigr \}\,.
\end{equation*}
The rest of the proof is analogous to the proof of~\eqref{e.tilt.by.Euc} using the fact that~$3^n \Lat \subset \Zd$ and the following upper bound for the volume fraction of the boundary layer for~$j<n$:
\begin{equation*}
| W_j |
\leq 
C\Pi_{\m_0}^{\nicefrac12}  3^{j-m}|\cu_m|
\,. 
\end{equation*}
This completes the proof of~\eqref{e.Euc.by.tilt} and thus of the lemma.
\end{proof}

\subsection{Besov spaces in adapted geometry}
\label{ss.besov}

We record here some analogs of the above estimates in the~$\m_0$--adapted geometry since these will be needed in what follows.
We will not give the proofs since they can be obtained by repeating the arguments above or by applying the statements above after performing an affine change of coordinates. 

\begin{itemize}

\item We first extend the Besov norms with positive regularity, defined in~\eqref{e.Bs.seminorm.cus}, to~$\cus_n$ by defining, for every~$s \in (0,1)$, $p\in [1,\infty)$, $q\in [1,\infty)$ and~$n\in \N$,
\begin{equation}
\label{e.Bs.seminorm.cus}
\left[ g \right]_{\underline{B}_{p,q}^{s}(\cus_{n})}
 \coloneqq 
\Biggl(
\sum_{k=-\infty}^n
3^{- s qk}
\biggl(
\avsum_{z\in 3^{k-1}\Lat, \, z + \cus_k \subseteq \cus_n}
\bigl \| g - (g)_{z+\cus_k}\bigr \|_{\underline{L}^p(z+\cus_k)}^p
\biggr)^{\!\nicefrac qp}
\Biggr)^{\! \nicefrac1q}
\,.
\end{equation}
We also define~$\left[ g \right]_{\underline{B}_{p,\infty}^{s}(\cus_{n})}$ similarly, in analogy to~\eqref{e.Bs.seminorm.infty}.

\item
The negative Besov norms are defined following~\eqref{e.Bs.minus.seminorm},~\eqref{e.Bs.minus.seminorm.zero} and~\eqref{e.Bs.minus.seminorm.explicit}:
\begin{equation*}
\left[ f \right]_{\Bhatminusul{-s}{p}{q} (\cus_{n})}
 \coloneqq 
\sup \biggl\{ \fint_{\cus_n} f g \, : \, g \in B_{p',q'}^{s}(\cus_{n}) \,, \; 
\left\|  g \right\|_{\underline{B}_{p',q'}^{s}(\cus_{n})} \leq 1 \biggr\}
\,,
\end{equation*}
\begin{equation*}
\left[ f \right]_{\underline{B}_{p,q}^{-s}(\cus_{n})}
 \coloneqq 
\sup \biggl\{ \fint_{\cus_n} f g \, : \, g \in C_{\mathrm{c}}^\infty(\cu_n) \,, \; 
\left[  g \right]_{\underline{B}_{p',q'}^{s}(\cus_{n})} \leq 1 \biggr\}
\,,
\end{equation*}
and
\begin{equation}
\label{e.Bs.minus.seminorm.explicit.cus}
\left[ f \right]_{\Besov{-s}{p}{q}(\cus_{n})}
 \coloneqq 
\Biggl(
\sum_{k=-\infty}^n
3^{s qk}
\biggl(
\avsum_{z\in 3^k\Lat \cap \cus_n}
\bigl| (f)_{z+\cus_k}\bigr |^p
\biggr)^{\!\nicefrac qp}
\Biggr)^{\! \nicefrac1q}
\,.
\end{equation}
For every~$s \in (0,1]$, $p\in [1,\infty]$ and~$q\in [1,\infty]$, we have that 
\begin{equation*}
\left[ f \right]_{\underline{B}_{p,q}^{-s}(\cus_{n})}
\leq
\left[ f \right]_{\Bhatminusul{-s}{p}{q} (\cus_{n})}
\leq
3^{d+s}
\left[ f \right]_{\Besov{-s}{p}{q}(\cus_{n})}
\,.
\end{equation*}

\item 
The statement of Lemma~\ref{l.dualitylemma} is modified as follows: 
there exists a constant~$C(d)<\infty$ such that, for every~$n \in \N$,~$s \in [0,1)$ and~$u \in H^1_\s(\cus_n)$, we have that
\begin{equation}  
\label{e.divcurl.est.cus}
\left\| u- (u)_{\cus_n}  \right\|_{\underline{B}_{2,\infty}^{s}(\cus_n)}
\leq 
C \bigl|\m_{0}^{-1}\bigr|^{\nicefrac12} 
\bigl[ \mathbf{m}_0^{\nf 12} \nabla u  \bigr]_{\Besov{s-1}{2}{1}(\cus_{n})}
\,,
\end{equation}
and, if~$\phi \in C^\infty(\cus_n)$ satisfies~$3^{n}  \bigl \| \mathbf{q}_0 \nabla \phi \bigr \|_{L^\infty(\cus_n)} + 3^{2n} \bigl \| (\mathbf{q}_0\nabla)^2 \phi \bigr \|_{L^\infty(\cus_n)} \leq 1$, then
\begin{equation}  
\label{e.divcurl.est1.cus}
\bigl \| (u- (u)_{\cus_n}) \m_0^{\nicefrac12} \nabla \phi \bigr \|_{\underline{B}_{2,\infty}^{s}(\cus_n)}
\leq 
C 3^{-n} \bigl[ \m_0^{\nicefrac12} \nabla u  \bigr]_{\Besov{s-1}{2}{1}(\cus_{n})}
\,.
\end{equation}

\item The statement of Lemma~\ref{l.crude.weaknorms} can be modified as follows. 
For every~$s \in (0,1]$,~$n \in \N$ and~$u \in \mathcal{A}(\cus_n)$, we have that
\begin{equation}
\label{e.grad.B21minuss.cus}
\bigl[ \m_0^{\nicefrac12} \nabla u  \bigr]_{\Besov{-s}{2}{1}(\cus_{n})}
\\
\leq 
\|\s^{\nicefrac12} \nabla u \|_{\underline{L}^2(\cus_n)}\,
\sum_{k=-\infty}^{n} 
\! \! 
3^{s k} 
\max_{z \in 3^k \Lat \cap \cus_n}\bigl| \m_{0}^{\nf 12} \s_{*}^{-1}(z+\cus_k) \m_{0}^{\nf 12} \bigr|^{\nicefrac12} 
\end{equation}
and
\begin{equation} \label{e.flux.B21minuss.cus}
\!\! 
\bigl[ \m_0^{-\nicefrac12} 
\a\nabla u  \bigr]_{\Besov{-s}{2}{1}(\cus_{n})}
\\
\leq
\|\s^{\nicefrac12} \nabla u \|_{\underline{L}^2(\cus_n)} \,
\sum_{k=-\infty}^{n} \! \! 3^{s k} 
\max_{z \in 3^k \Lat \cap \cus_n}\bigl| \m_{0}^{-\nf12} \b(z+\cus_k) \m_{0}^{-\nf12} \bigr|^{\nicefrac12} \,.
\end{equation}
\end{itemize}

\subsection{Gradient and flux estimates in weak norms}

We conclude this subsection with a lemma that allows us to estimate weak norms of the gradient and flux of a solution in terms of the coarse-grained coefficients. While the statement may initially seem a bit ugly, the explicit form of the estimate will be useful.

\begin{lemma}
\label{l.weaknorms.moreproto}
Let~$\rho \in (0,2)$,~$s \in (\nicefrac\rho2,1]$ and~$h,n \in \N$ with~$h < n$. Also let~$\m \in \R^{d\times d}$ be positive and symmetric and~$\h\in \R^{d\times d}$ be antisymmetric, and denote
\begin{equation*}
\mathbf{M}
 \coloneqq  \begin{pmatrix} 
\m + \h^t \m^{-1} \h & - \h^t \m \\ - \m \h  & \m^{-1}
\end{pmatrix}
\,.
\end{equation*}  
Given a symmetric and positive matrix~$\mathbf{E} \in \R^{2d\times 2d}$, define the random variable
\begin{equation} 
\label{e.event.moreproto}
\mathcal{M}_{n,\rho}  \coloneqq  \sup_{k \in \Z \cap (-\infty,n]}  3^{-\rho(n-k)}
 \max_{z\in 3^{k}\Lat  \cap \cus_n}  \bigl| \bigl( \mathbf{E}^{-\nicefrac12} ( \bfA(z + \! \cus_k;\a) - \mathbf{E} ) \mathbf{E}^{-\nicefrac12}\bigr)_+ \bigr| 
\,.
\end{equation}
Then there exists a universal constant~$C<\infty$ such that, for every~$\delta \in (0,1]$ and~$p,q\in\Rd$, by writing~$v_n  \coloneqq  v(\cdot,\cus_n,p,q;\a)$, 
\begin{align} 
\label{e.weaknorms.moreproto}
\lefteqn{
3^{-sn}
\biggl[
\mathbf{M}^{\nicefrac12} 
\begin{pmatrix} 
\nabla v_n -  (\nabla v_n)_{\cus_n}  \\ 
\a \nabla v_n- (\a\nabla v_n)_{\cus_n}
\end{pmatrix} 
\biggr]_{\Besov{-s}{2}{1}(\cus_{n})}
} \    &
\notag \\ &
\leq 
C 
\bigl|  \mathbf{M}^{-\nicefrac12} \mathbf{E} \mathbf{M}^{-\nicefrac12} \bigr|^{\nicefrac12}
\Bigl | 
\mathbf{E}^{\nicefrac 12}  
\Bigl(
\begin{matrix} 
-p \\ q
\end{matrix} 
\Bigr)
\Bigr|
\sum_{k=n-h}^n \! \! \! 3^{s(k-n)} \biggl(
\avsum_{z\in 3^{k}\Lat  \cap \cus_n}  \!\!\!\!
 \bigl| \mathbf{E}^{-\nicefrac12}( \bfA(z+\cus_k) - \bfA(\cus_n) ) \mathbf{E}^{-\nicefrac12} \bigr|^2
 \biggr)^{\! \nicefrac12} 
\notag \\ &
\quad 
+
C
\bigl|  \mathbf{M}^{-\nicefrac12} \mathbf{E} \mathbf{M}^{-\nicefrac12} \bigr|^{\nicefrac12} 
\Bigl | 
\mathbf{E}^{\nicefrac 12}  
\Bigl(
\begin{matrix} 
-p \\ q
\end{matrix} 
\Bigr)
\Bigr|
 \!\sum_{k=n-h}^{n} \! \! \! 3^{(s-\nf \rho2)(k-n)} 
\biggl| \avsum_{z\in 3^{k}\Lat  \cap \cus_n} \!\!\!\! \mathbf{E}^{-\nicefrac12}\bigl(\bfA(z{+}\cus_k) - \bfA(\cus_n)  \bigr) \mathbf{E}^{-\nicefrac12}  \biggr|^{\nicefrac12}
\notag \\ &
\quad 
+
\frac{C \delta^{-\nicefrac12} }{2s -  \rho}\biggl( 
\indc_{\{ \mathcal{M}_{n,\rho} > \delta \}}  \mathcal{M}_{n,\rho}^{\nicefrac12} 
{+} 3^{-(s-\nicefrac{\rho}{2})h} \indc_{\{ \mathcal{M}_{n,\rho} \leq \delta \}}  \biggr)
\bigl|  \mathbf{M}^{-\nicefrac12} \mathbf{E} \mathbf{M}^{-\nicefrac12} \bigr|^{\nicefrac12}
\bigl\| \s^{\nicefrac12} \nabla v_n \bigr\|_{\underline{L}^2(\cus_n)}
\,.
\end{align}
\end{lemma}

\begin{proof}
As in the statement, we write~$v_n \coloneqq  v(\cdot,\cus_n,p,q)$ to shorten the notation. 
Fix~$\delta \in (0,\infty)$,~$\rho \in (0,2)$,~$s \in (\nicefrac\rho2,1]$ and~$n,h \in \N$ with~$h<n$. We suppress~$\rho,\delta$ from the notation with~$\mathcal{M}_{\rho,\delta}$.  Fix also~$p,q \in \Rd$. 
We denote, for any Lipschitz domain~$U$, 
\begin{equation*} 
X(\cdot,U,p,q)  \coloneqq   \begin{pmatrix} 
\nabla v(\cdot,U,p,q)   \\ \a \nabla v(\cdot,U,p,q) 
\end{pmatrix}
\,.
\end{equation*}
It is sometimes convenient for computations to note that 
\begin{equation*}
\bfA(x)
\begin{pmatrix} 
p \\ \a(x) p 
\end{pmatrix}
=
\begin{pmatrix} 
\a(x) p \\ p 
\end{pmatrix}
\,.
\end{equation*}
This implies in particular that 
\begin{equation}
\label{e.energy.X.to.v}
X(\cdot,U,p,q) 
\cdot \bfA
X(\cdot,U,p,q)
=
2 \nabla v(\cdot,U,p,q) \cdot \s \nabla v(\cdot,U,p,q) 
\,.
\end{equation}
For every~$k \in \Z$ and~$z \in \Rd$, we denote $X_{z,k} \coloneqq X(\cdot,z+\cus_k,p,q)$ and, for~$z=0$, we suppress~$z$ from the notation and write~$X_k = X_{0,k}$. We also denote~$v_{k,z} \coloneqq v(\cdot,z+\cus_k,p,q)$.

\smallskip
  
\emph{Step 1.} We first show that, for any Lipschitz domains~$U,V$, 
\begin{multline}
\label{e.sqbound}
\Bigl| 
\mathbf{M}^{\nicefrac12} 
\Bigl( 
\bigl( X(\cdot,U,p,q) \bigr)_{U} 
- \bigl( X(\cdot,V,p,q) \bigr)_{V} 
\Bigr)
\Bigr|^2
\\ 
\leq 
\bigl| \mathbf{E}^{\nicefrac12}  \mathbf{M}^{-1} \mathbf{E}^{\nicefrac12}  \bigr|
\bigl| \mathbf{E}^{-\nicefrac12}  \bigl( \bfA(U) -\bfA(V) \bigr) \mathbf{E}^{-\nicefrac12} 
\bigr|^2
\Bigl | 
\mathbf{E}^{\nicefrac 12}  
\Bigl(
\begin{matrix} 
-p \\ q
\end{matrix} 
\Bigr)
\Bigr|^2
\,.
\end{multline}
We have the following identity by~\eqref{e.all.averages}:
\begin{equation*}
\bigl( X(\cdot,U,p,q) \bigr)_{U} 
=
\bigl( \mathbf{R}  \bfA(U) + \Itwod \bigr)
\begin{pmatrix} 
-p \\ q
\end{pmatrix}
\quad \mbox{with} \quad \mathbf{R}  \coloneqq  \begin{pmatrix} 
0 & \Id \\ \Id & 0  
\end{pmatrix}
\,.
\end{equation*}
It follows that
\begin{align*} 
\lefteqn{
\Bigl| 
\mathbf{M}^{\nicefrac12} 
\Bigl( 
\bigl( X(\cdot,U,p,q) \bigr)_{U} 
- \bigl( X(\cdot,V,p,q) \bigr)_{V} 
\Bigr)
\Bigr|^2
} \qquad &
\notag \\ &
=
\biggl|
\begin{pmatrix} 
-p \\ q
\end{pmatrix}
\cdot
(\bfA(U) - \bfA(V)) \mathbf{R}   \mathbf{M} \mathbf{R}  (\bfA(U) - \bfA(V))
\begin{pmatrix} 
-p \\ q
\end{pmatrix}
\biggr|
\,,
\end{align*}
from which~\eqref{e.sqbound} follows since~$\mathbf{R}   \mathbf{M} \mathbf{R} = \mathbf{M}^{-1}$. 

\smallskip
  
\emph{Step 2.} We show that, for every~$X \in \mathcal{S}(\cus_n)$, 
\begin{align} 
\label{e.proto.X.bound}
\lefteqn{
3^{-sn}\bigl[ 
\mathbf{M}^{\nicefrac12}  X
\bigr]_{\Besov{-s}{2}{1}(\cus_{n})} 
} \qquad &
\notag \\ &
\leq 
C \bigl| \mathbf{E}^{\nicefrac12}  \mathbf{M}^{-1} \mathbf{E}^{\nicefrac12}  \bigr|^{\nicefrac12} 
\| \bfA^{\nicefrac12} X \|_{\underline{L}^2(\cus_n)}
 s \! 
\sum_{k=-\infty}^n \! \!  3^{s(k-n)}
 \! \!   \max_{z\in 3^{k}\Lat  \cap \cus_n} \! \! \!
 \bigl| \mathbf{E}^{-\nicefrac12} \bfA(z {+} \! \cus_k) \mathbf{E}^{-\nicefrac12} \bigr|^{\nicefrac12} 
\,.
\end{align}
To see this, we use~$\mathbf{R}^2 = \Itwod$,~$\bfA(U) = \mathbf{R}\bfA_*^{-1}(U) \mathbf{R}$,~$\mathbf{M}^{-1} = \mathbf{R} \mathbf{M} \mathbf{R}$, noting that this also implies that~$\mathbf{M}^{\nicefrac12} = \mathbf{R} \mathbf{M}^{-\nicefrac12} \mathbf{R}$, and the fact that~$\mathbf{R}$ is unitary and~$\bfA(z+\cus_k)$ is positive semidefinite, to deduce that 
\begin{align*}
\bigl| \mathbf{M}^{\nicefrac12} \bfA_*^{-1} (z+\cus_k)\mathbf{M}^{\nicefrac12} \bigr|
& 
= \bigl| \mathbf{R} \mathbf{M}^{-\nicefrac12} \mathbf{R}\bfA_*^{-1} (z+\cus_k)\mathbf{R} \mathbf{M}^{-\nicefrac12} \mathbf{R}\bigr|
\notag \\ & 
=\bigl| \mathbf{R} \mathbf{M}^{-\nicefrac12}\bfA (z+\cus_k)\mathbf{M}^{-\nicefrac12} \mathbf{R}\bigr|
= 
\bigl| \mathbf{M}^{-\nicefrac12}\bfA (z+\cus_k)\mathbf{M}^{-\nicefrac12} \bigr|
\,.
\end{align*}
Thus, we obtain by~\eqref{e.energymaps.nonsymm.all} that
\begin{align}  
\label{e.proto.X.bound.pre}
\avsum_{z\in 3^{k}\Lat  \cap \cus_n} \! \! \!
\bigl| \mathbf{M}^{\nicefrac12}
(X)_{z+\cus_k} 
\bigr|^2
&
\leq 
\! \! \!
\avsum_{z\in 3^{k}\Lat  \cap \cus_n} 
\bigl| \bfA_*^{-\nicefrac12} (z+\cus_k)\mathbf{M} \bfA_*^{-\nicefrac12} (z+\cus_k)\bigr|
\bigl| \bfA_*^{\nicefrac12} (z+\cus_k) ( X  )_{z+\cus_k} \bigr|^2
\notag \\ &
= 
\! \! \!
\avsum_{z\in 3^{k}\Lat  \cap \cus_n} 
\bigl| \mathbf{M}^{\nicefrac12} \bfA_*^{-1} (z+\cus_k)\mathbf{M}^{\nicefrac12} \bigr|
\bigl| \bfA_*^{\nicefrac12} (z+\cus_k) ( X  )_{z+\cus_k} \bigr|^2
\notag \\ &
\leq 
\bigl| \mathbf{E}^{\nicefrac12} \mathbf{M}^{-1}\mathbf{E}^{\nicefrac12}\bigr| \| \bfA^{\nicefrac12} X \|_{\underline{L}^2(\cus_n)}^2 
\max_{z\in 3^{k}\Lat  \cap \cus_n}  \! \! \! \bigl|   \mathbf{E}^{-\nicefrac12} \bfA(z+\cus_k) \mathbf{E}^{-\nicefrac12}\bigr|  
\,,
\end{align}
which gives us~\eqref{e.proto.X.bound}. 

\smallskip

\emph{Step 3.} We next show that 
\begin{align}
\label{e.weaknorms.proto.above.minscale}
\lefteqn{
3^{-sn} 
\Bigl[ 
\mathbf{M}^{\nicefrac12} 
\bigl(  X_n    -  ( X_n  )_{\cus_n} \bigr)
\Bigr]_{\Besov{-s}{2}{1}(\cus_{n})}
\indc_{\{ \mathcal{M} > \delta\}} 
} \qquad & 
\notag \\ &
\leq
 \frac{C}{2s-\rho}
\frac{1+\delta^{\nicefrac12}}{\delta^{\nicefrac12}}  
\bigl| \mathbf{E}^{\nicefrac12} \mathbf{M}^{-1}\mathbf{E}^{\nicefrac12}\bigr|^{\nicefrac12} 
\| \s^{\nicefrac12} \nabla v_n \|_{\underline{L}^2(\cus_n)}
\mathcal{M}^{\nicefrac12} 
\indc_{\{ \mathcal{M} > \delta\}}  
\,.
\end{align}
In view of~\eqref{e.event.moreproto},~\eqref{e.energy.X.to.v} and~\eqref{e.proto.X.bound}, we obtain 
\begin{align*} 
\lefteqn{
3^{-sn}\bigl[ 
\mathbf{M}^{\nicefrac12}  X_n 
\bigr]_{\Besov{-s}{2}{1}(\cus_{n})}
\indc_{\{ \mathcal{M} > \delta\}} 
} \qquad &  
\notag \\ & 
\leq
3^{d}\bigl| \mathbf{E}^{\nicefrac12} \mathbf{M}^{-1}\mathbf{E}^{\nicefrac12}\bigr|^{\nicefrac12} 
\| \bfA^{\nicefrac12} X_n \|_{\underline{L}^2(\cus_n)}
\sum_{k=-\infty}^n 3^{s(k-n)} \bigl(1 + \mathcal{M}^{\nicefrac12}  3^{\frac{\rho}{2}(n-k)}\bigr)
\indc_{\{ \mathcal{M} > \delta\}} 
\notag \\ &
\leq
 \frac{C}{2s-\rho}
\frac{1+\delta^{\nicefrac12}}{\delta^{\nicefrac12}}  
\bigl| \mathbf{E}^{\nicefrac12} \mathbf{M}^{-1}\mathbf{E}^{\nicefrac12}\bigr|^{\nicefrac12} 
\| \s^{\nicefrac12} \nabla v_n \|_{\underline{L}^2(\cus_n)}
\mathcal{M}^{\nicefrac12}
\indc_{\{ \mathcal{M} > \delta\}} 
 \,,
\end{align*}
which completes the proof of~\eqref{e.weaknorms.proto.above.minscale}. 

\smallskip
  
\emph{Step 4.} We conclude by proving~\eqref{e.weaknorms.moreproto} under the event~$\{ \mathcal{M} \leq \delta \} $. 
First, by~\eqref{e.proto.X.bound.pre}, 
\begin{align*} 
\notag
\lefteqn{
\sum_{k=-\infty}^{n-h} 3^{s(k-n)} \biggl(
\avsum_{z\in 3^{k}\Lat  \cap \cus_n} \! \! \!
\bigl| \mathbf{M}^{\nicefrac12} \bigl(
( X_n  )_{z+\cus_k}  - ( X_n  )_{\cus_n} \bigr)
\bigr|^2
\biggr)^{\!\nicefrac 12}
\indc_{\{ \mathcal{M} \leq \delta \}} 
}  \qquad &
\notag \\ &  
\leq
2 \bigl| \mathbf{E}^{\nicefrac12} \mathbf{M}^{-1}\mathbf{E}^{\nicefrac12}\bigr|^{\nicefrac12}
\| \s^{\nicefrac12} \nabla v_n \|_{\underline{L}^2(\cus_n)} \,
\sum_{k=-\infty}^{n-h} 3^{s(k-n)} \bigl(
1 + \delta^{\nicefrac12} 3^{\frac{\rho}{2}(n-k)} \bigr)
\notag \\ &  
\leq
C \biggl( s^{-1}3^{-hs} + \frac{ \delta^{\nicefrac12}3^{-(s-\nicefrac{\rho}{2}) h}}{2s-\rho} \biggr)
\bigl| \mathbf{E}^{\nicefrac12} \mathbf{M}^{-1}\mathbf{E}^{\nicefrac12}\bigr|^{\nicefrac12} 
\| \s^{\nicefrac12} \nabla v_n \|_{\underline{L}^2(\cus_n)}
\,.
\end{align*}
Second,  using the triangle inequality, 
for each~$k\in \Z$ with~$n - h \leq k\leq n$, we get 
\begin{align}
\label{e.dumbsplit}
\lefteqn{
\avsum_{z\in 3^{k}\Lat  \cap \cus_n}
\bigl| \mathbf{M}^{\nicefrac12} \bigl(   ( X_n  )_{z+\cus_k}  -( X_n  )_{\cus_n} \bigr)
\bigr|^2
} \qquad &
\notag \\ &
\leq 
2
\avsum_{z\in 3^{k}\Lat  \cap \cus_n}
\Bigl( 
\bigl| \mathbf{M}^{\nicefrac12}\bigl(  ( X_{z,k}  )_{z+\cus_k}  -  ( X_{n}  )_{\cus_n} 
\bigr)
\bigr|^2
+
\bigl|\mathbf{M}^{\nicefrac12} ( X_{n} {-} X_{z,k}  )_{z+\cus_k} \bigr|^2
\Bigr)
\,.
\end{align}
The contribution of the first term on the right side of~\eqref{e.dumbsplit} can be estimated using~\eqref{e.sqbound}:
\begin{align*}
\lefteqn{
\bigl| \mathbf{M}^{\nicefrac12}\bigl(  ( X_{z,k}  )_{z+\cus_k}  {-}  ( X_{n}  )_{\cus_n} 
\bigr)
\bigr|^2
} \qquad &
\notag \\ &
\leq
\bigl| \mathbf{E}^{\nicefrac12} \mathbf{M}^{-1}\mathbf{E}^{\nicefrac12}\bigr|
\bigl| \mathbf{E}^{-\nicefrac12}   \bigl( \bfA(z+\cus_{k}) - \bfA(\cus_n) \bigr)  \mathbf{E}^{-\nicefrac12} \bigr|^2 
\Bigl | 
\mathbf{E}^{\nicefrac 12}  
\begin{pmatrix} 
-p \\ q
\end{pmatrix} 
\Bigr|^2
\,.
\end{align*}
We estimate the second term on the right side of~\eqref{e.dumbsplit} using the same computation as in~\eqref{e.proto.X.bound.pre}, but now for~$X_{n} {-} X_{z,k}$ instead of~$X_n$:
\begin{align*}
\lefteqn{
\avsum_{z\in 3^{k}\Lat  \cap \cus_n} \!\!\!
\bigl|\mathbf{M}^{\nicefrac12} ( X_{n} {-} X_{z,k}  )_{z+\cus_k} \bigr|^2
} \qquad & 
\notag \\ &
\leq
\bigl| \mathbf{E}^{\nicefrac12} \mathbf{M}^{-1}\mathbf{E}^{\nicefrac12}\bigr|
\max_{z\in 3^{k}\Lat  \cap \cus_n} 
\bigl|\mathbf{E}^{-\nicefrac12}  \bfA(z+\cus_k) \mathbf{E}^{-\nicefrac12}  \bigr|
\avsum_{z\in 3^{k}\Lat  \cap \cus_n}
\| \bfA^{\nicefrac12} ( X_{n} {-} X_{z,k}  ) \|_{\underline{L}^2(z+\cus_k)}^2 
\,.
\end{align*}
Thus,
\begin{align*}
\lefteqn{
\avsum_{z\in 3^{k}\Lat  \cap \cus_n} \!\!\!
\bigl|\mathbf{M}^{\nicefrac12} ( X_{n} {-} X_{z,k}  )_{z+\cus_k} \bigr|^2
\indc_{\{ \mathcal{M} \leq \delta \}}
} \qquad &
\notag \\ &
\leq
\bigl(1+\delta 3^{\rho(n-k)} \bigr) 
\bigl| \mathbf{E}^{\nicefrac12} \mathbf{M}^{-1}\mathbf{E}^{\nicefrac12}\bigr|
\avsum_{z\in 3^{k}\Lat  \cap \cus_n}
\| \bfA^{\nicefrac12} ( X_{n} {-} X_{z,k}  ) \|_{\underline{L}^2(z+\cus_k)}^2 
\,.
\end{align*}
By quadratic response~\eqref{e.quadratic.response},~\eqref{e.J.energy}, the first variation~\eqref{e.first.variation} and~\eqref{e.Jaas.matform}, we have that
\begin{align*} 
\avsum_{z\in 3^{k}\Lat  \cap \cus_n} \!\!\!
\| \bfA^{\nicefrac12} ( X_{n} {-} X_{z,k}  ) \|_{\underline{L}^2(z+\cus_k)}^2 
&
=
2  \avsum_{z\in 3^{k}\Lat  \cap \cus_n} \bigl( \| \s^{\nicefrac12} \nabla v_{k,z}   \|_{\underline{L}^2(z+\cus_k)}^2 - \| \s^{\nicefrac12} \nabla v_{n}  \|_{\underline{L}^2(\cus_n)}^2 \bigr)
\notag \\ &
=
4  \avsum_{z\in 3^{k}\Lat  \cap \cus_n}  
\begin{pmatrix} 
-p \\ q
\end{pmatrix} 
\cdot \bigl(\bfA(z+\cus_k) - \bfA(\cus_n) \bigr)
\begin{pmatrix} 
-p \\ q
\end{pmatrix} 
\notag \\ &
\leq 
4 \Bigl | 
\mathbf{E}^{\nicefrac 12}  
\begin{pmatrix} 
-p \\ q
\end{pmatrix} 
\Bigr|^2
\biggl| \avsum_{z\in 3^{k}\Lat  \cap \cus_n} \mathbf{E}^{-\nicefrac12}  ( \bfA(z+\cus_k) - \bfA(\cus_n) ) \mathbf{E}^{-\nicefrac12}  \biggr|
 \,.
\end{align*}
Putting the above estimates together gives us~\eqref{e.weaknorms.moreproto} under the event~$\{ \mathcal{M} \leq \delta \} $. 
\end{proof}

\section{Renormalization in high contrast}
\label{s.firstpigeon}

The purpose of this section is to give an estimate of the length scale at which a general elliptic coefficient field~$\a(x)$, with (possibly very large) ellipticity ratio~$\Theta\in[1,\infty)$, has homogenized to within a specified finite error.
The precise statement is given in the following theorem.

\begin{theorem}[Homogenization in high contrast]
\label{t.HC}
Assume that~$\P$ is a probability measure on~$(\Omega,\mathcal{F})$ satisfying assumptions~\ref{a.stationarity},~\ref{a.ellipticity.dagger} and~\ref{a.CFS}, with~$\nu > \gamma$.
Let~$\alpha \coloneqq  (\min\{\nu,1\}-\gamma )(1-\beta)$. There exists a constant~$C(d)<\infty$ such that, for every~$\sigma\in (0,\frac12 \Theta]$ and~$m\in\N$ satisfying 
\begin{equation}
\label{e.highcontrast.tamed.m}
m \geq 
\frac{C}{\sigma^2} 
\biggl(  
\log K_{\Psi_\S} + \frac{1}{\alpha^2} \log \Bigl( \frac{\Pi K_{\Psi}}{\alpha \sigma} \Bigr)  
\biggr)
\log (1+\Theta)
\,,
\end{equation}
the renormalized ellipticity ratio~$\Theta_m$ defined in~\eqref{e.Theta.n.def} satisfies
\begin{equation}
\label{e.squoosh}
\Theta_m - 1
\leq 
\sigma 
\,.
\end{equation}
\end{theorem}

The reader is encouraged to ignore the details of the rather explicit form of~\eqref{e.highcontrast.tamed.m}, on first reading, and notice only that the theorem asserts that, for a constant~$C$ depending on the parameters~$(d,\gamma,\beta)$, but not on~$\Theta$ or~$\Pi$, 
\begin{equation*}
m \geq
C 
\log \bigl( 1+\max \{ \Pi, K_{\Psi} , K_{\Psi_\S} \} \bigr) \log(1+\Theta) 
\quad \implies \quad 
\Theta_m - 1 \leq 10^{-6}
\,.
\end{equation*}
If we wish, we can ignore also the parameters~$(K_{\Psi},K_{\Psi_\S})$ and use the trivial bound~$\Theta \leq \Pi$ to obtain that, for a constant~$C(d,\gamma,\beta,\nu,K_{\Psi},K_{\Psi_\S})<\infty$, 
\begin{equation*}
m\geq C \log^2 (1+\Pi)
\quad \implies \quad 
\Theta_m - 1 \leq 10^{-6}
\,.
\end{equation*}
This matches the length scale appearing in the statement of Theorem~\ref{t.HC.intro} in the introduction. 

\smallskip

Since the parameter~$\sigma$ in Theorem~\ref{t.HC} can be taken arbitrarily small, the theorem statement provides an explicit convergence rate for~$\Theta_m -1$. However, this rate is not very useful when~$\sigma$ is small. The main role of the theorem is, therefore, to reduce~$\Theta_m -1$ from a possibly very large number to a somewhat small number, say,~$10^{-6}$. In the next section, we use this estimate as a starting point for the derivation of a much better estimate on the rate of~$\Theta_m - 1$ to zero. 

\smallskip

Throughout the section, we assume that ~$\P$ is a probability measure on~$(\Omega,\mathcal{F})$ satisfying assumptions~\ref{a.stationarity},~\ref{a.ellipticity.dagger} and~\ref{a.CFS}, with~$\nu > \gamma$, and we let~$\alpha$ be the parameter defined in the statement of Theorem~\ref{t.HC}. 

\smallskip

The main step in the proof of Theorem~\ref{t.HC} lies in the following proposition, which formalizes one step of the renormalization procedure. It says that either~$\Theta_m$ or its variant~$\XiDet_n$, defined in~\eqref{e.hat.Theta.m}, contracts by a factor of~$\sigma$ after we zoom out on the order of~$\log \Pi$ many triadic scales.

\begin{proposition}[One renormalization step]
\label{p.renormalize}
There exists a constant~$C(d)<\infty$ such that, for every~$\sigma\in (0,\nicefrac12]$ and~$m\in\N$ satisfying 
\begin{equation}
\label{e.m.explivomit}
m\geq 
\frac{C}{\sigma^2} 
\biggl(  
\log K_{\Psi_\S} + \frac{1}{\alpha^2} \log \Bigl( \frac{\Pi K_{\Psi}}{\alpha \sigma} \Bigr)  
\biggr)
\,,
\end{equation}
we have
\begin{equation}
\label{e.improve}
\min \biggl\{ 
\frac{\Theta_m-1}{\Theta_0} \,,\, 
\frac{\XiDet_m}{\XiDet_0} 
\biggr\} 
\leq 
\sigma 
\,.
\end{equation}
\end{proposition}

We have written the dependence of the lower bound on~$m$ in~\eqref{e.m.explivomit}  explicitly in all parameters except for~$d$. If we wish, we can write it in a nicer-looking (but less informative) way as
\begin{equation*}
m \geq \frac{C\left|\log \sigma\right| }{\sigma^2} \log(1+\Pi)\,, 
\end{equation*}
but now the constant~$C$ depends on~$(d,K_{\Psi_\S},K_{\Psi},\gamma,\beta,\nu)$, but not on~$\sigma$ nor on the ellipticity ratios~$\Pi$ and~$\Theta$. 
Therefore, Proposition~\ref{p.renormalize} says, informally, that the renormalized ellipticity ratio is reduced by a constant factor if we zoom out on the order of~$\log(1+\Pi)$ many geometric scales. 

\smallskip

The proof of Proposition~\ref{p.renormalize} is the main focus of this section. Once its proof is complete, we will iterate the statement on the order of~$\log (1+\Theta)$ many times, renormalizing at each step with the help of Proposition~\ref{p.renormalization.P}, to obtain Theorem~\ref{t.HC}.

\subsection{A reduction: finding a good range of scales}

The first step in the proof of Proposition~\ref{p.renormalize} is to make a reduction to the following statement. 

\begin{proposition}
\label{p.renormalize.reduce}
Suppose that~$\delta,\sigma \in (0,\nicefrac12]$ 
and~$l\in\N$ satisfy 
\begin{equation} 
\label{e.l.def.implies}
\max\biggl\{ 
3^{- \frac14(1-\beta) l} \Pi
\,,\, 
K_{\Psi}^8  \Pi 3^{-\frac12(\nu -\gamma)(1-\beta) l} 
\,,\,
\frac{K_{\Psi_\S}^{16d} \Pi^4}{(1-\gamma)^4}  3^{ -l } 
\,,\,
\frac{3^{-(1 - \gamma) l }}{1-\gamma}  
\biggr\}
\leq 
\delta \sigma^2
\,.
\end{equation}
Suppose that~$m \in\N$ with~$m\geq 100l$ and that the matrix~$\bfE$ in~\ref{a.ellipticity.dagger}  satisfies
\begin{equation} \label{e.renormalize.A.vs.E0}
\bfAhom(\cu_0) \leq \bfE 
\qand 
\bigl| 
\bfE^{\nicefrac12}\bfAhom^{-1}(\cu_m) \bfE^{\nicefrac12}  - \Itwod \bigr|
\leq \delta \sigma^2
\,.
\end{equation}
Then there exists a constant~$\delta_0(d)>0$ such that~$\delta \leq \delta_0$ implies  
\begin{equation}
\label{e.improve.reduce}
\Theta_m - 1
\leq 
\sigma \Theta
\,.
\end{equation}
\end{proposition}

The proof of Proposition~\ref{p.renormalize.reduce} is given in Section~\ref{s.one.renorm}, below. 
We first demonstrate that it implies Proposition~\ref{p.renormalize}, which relies on the following lemma. The idea is very simple: since the deterministic matrix~$\bfAhom(\cu_n)$ is monotone in~$n$ and its determinant is bounded between~$1$ and~$\Theta^d$, we can find a sequence of consecutive~$n$ for which it does not change much. 

\begin{lemma}[Pigeonhole lemma]
\label{l.pigeon}
Let~$\delta_1 , \sigma \in (0,\nf12]$ and~$h,N \in \N$ such that~$N \geq h \lceil 2 \delta_1^{-1}\left|\log \sigma\right|\rceil$. 
For every~$m_1\in\N_0$, at least one of the following two alternatives is valid: 
\begin{itemize} 
\item there exists~$m \in \N \cap [m_1+h, m_1+N]$ such that~$\bfAhom(\cu_{m - h} ) \leq (1+\delta_1) \bfAhom(\cu_{m})$;
or 
\item $\XiDet_{m_1+N}\leq \sigma \XiDet_{m_1}$. 
\end{itemize}
\end{lemma}
\begin{proof}
For simplicity, assume~$m_1=0$. 
If the first alternative fails, then, by the monotonicity of~$j\mapsto \bfAhom(\cu_j)$, we obtain, for~$k  \coloneqq  \lfloor N h^{-1} \rfloor \geq \lceil 2 \delta_1^{-1}\left|\log \sigma\right|\rceil$, 
\begin{align*}
\XiDet_{N} \XiDet_{0}^{-1} 
\leq 
\XiDet_{kh} \XiDet_{0}^{-1}
=
\prod_{j=1}^k 
\det \bigl( 
\bfAhom(\cu_{jh})  \bfAhom^{-1} (\cu_{(j-1)h}) 
\bigr)^{\nf1d}
\leq 
(1+\delta_1)^{-k}
\leq \sigma 
\,,
\end{align*} 
proving the second alternative. 
\end{proof}

\begin{proof}[Reduction of Proposition~\ref{p.renormalize} to Proposition~\ref{p.renormalize.reduce}]
Fix~$\delta,\sigma\in(0,\nicefrac12]$ and let~$L(\delta,\sigma)$ be defined by
\begin{equation} 
\label{e.l.one}
L(\delta,\sigma)
 \coloneqq  \biggl\lceil 8d \log K_{\Psi_\S} + \frac{100d}{\alpha^2} \log \biggl( \frac{2^{10}\Pi K_{\Psi}}{\alpha \delta \sigma^2} \biggr)  
 \biggr\rceil\,.
\end{equation}
We apply Lemma~\ref{l.pigeon} with the following choices of parameters:
\begin{equation}
\label{e.params.choices}
\delta_1 \coloneqq 
\frac12 \delta\sigma^2 \,,
\quad
n  \coloneqq  100L(\delta,\sigma)\,, 
\quad 
h  \coloneqq  2n\,,
\quad 
N  \coloneqq  2n \bigl\lceil 2 \delta^{-1} \left|\log \sigma \right| \bigr\rceil
\quad\mbox{and}\quad
m_1  \coloneqq  N
\,.
\end{equation}	
The lemma asserts that one of the following alternatives is valid: 
\begin{enumerate}

\item[(i)] 
there exists~$m\in\N$ with~$m \in [m_1+2n, 2m_1]$ such that 
\begin{equation*}
\bigl| 
\bfAhom^{-\nicefrac12} (\cu_m) \bfAhom(\cu_{m-2n})\bfAhom^{-\nicefrac12} (\cu_m) - \Itwod \bigr|
\leq \frac12  \delta\sigma^2
\,;
\end{equation*}
\item[(ii)] $\XiDet_{2m_1}\leq \sigma \XiDet_{m_1}$. 

\end{enumerate}
If we are in case~(ii), there is nothing left to show, since, in view of the definitions in~\eqref{e.l.one} and~\eqref{e.params.choices}, this implies~\eqref{e.improve} by the monotonicity of~$j\mapsto \bfAhom(\cu_j)$. 

\smallskip

We may therefore assume we are in case~(i). 
We intend to apply Proposition~\ref{p.renormalization.P} with parameters~$n_0=m-n$,~$l_0 =n$,~$\rho =\frac12 (\min\{\nu,1\}+\gamma)$ and~$\frac14\delta\sigma^2$ instead of~$\delta$. In order to apply the proposition, we need to check that the condition~\eqref{e.l0.condition} is valid. This is however immediate from the choice~$n = 100L(\delta,\sigma)$ and the definition of~$L(\delta,\sigma)$ above.
The application of the proposition yields that the probability measure~$\P_{m-n}$, defined in~\eqref{e.Pn0} as the pushforward of~$\P$ under the dilation map~$\a \mapsto \a(3^{m-n}\cdot)$, satisfies assumptions~\ref{a.stationarity},~\ref{a.ellipticity.dagger} and~\ref{a.CFS} with the new parameters
\begin{equation}
\label{e.newparams}
\left\{
\begin{aligned}
& \mathbf{E}_{\mathrm{new}}  \coloneqq  (1+\tfrac14 \delta \sigma^2) \bfAhom(\cu_{m-2n}) 
\,, \\ & 
\gamma_{\mathrm{new}}  \coloneqq  \frac12 (\min\{1,\nu\}+\gamma) 
\, \\ &
K_{\Psi,\mathrm{new}}  \coloneqq  K_{\Psi}
\, \\ &
K_{\Psi_\S,\mathrm{new}}  \coloneqq  \max\big\{ K_{\Psi_{\S}} , K_{\Psi}^{\lceil \nicefrac1\mu \rceil} \bigr\}
\,, \\ & 
\Theta_{\mathrm{new}}  \coloneqq  
(1+\tfrac14 \delta \sigma^2)^2
\Theta_{m-n}\leq 
( 1 + \delta \sigma^2)
\Theta
\,, \\ & 
\Pi_{\mathrm{new}}  \coloneqq  2^{10}\Pi\,.
\end{aligned}
\right.
\end{equation}
We will now apply Proposition~\ref{p.renormalize.reduce} with~$(\P_{m-n},n)$ in place of~$(\P,m)$. This requires that~$n\geq 100 l$ for some~$l$ satisfying~\eqref{e.l.def.implies} with the new parameters in~\eqref{e.newparams}. 
For this we take~$l= L(\delta,\sigma)$, we note that~$n=100l$ and that~\eqref{e.l.def.implies} with the new parameters is valid by the definition of~$L(\delta,\sigma)$ in~\eqref{e.l.one}. 
To verify the final hypothesis~\eqref{e.renormalize.A.vs.E0}, we observe that~$\bfAhom(\cu_{m-h}) \leq  \mathbf{E}_{\mathrm{new}}$ and
\begin{align*} 
\lefteqn{ 
\bigl| 
\bfAhom^{-\nicefrac12} (\cu_m)  \mathbf{E}_{\mathrm{new}} \bfAhom^{-\nicefrac12} (\cu_m)- \Itwod \bigr|
} \qquad & 
\notag \\ & \leq 
(1+\tfrac14 \delta \sigma^2) 
\bigl| 
\bfAhom^{-\nicefrac12} (\cu_m) \bfAhom(\cu_{m-2n}) \bfAhom^{-\nicefrac12} (\cu_m)- \Itwod \bigr| 
+ 
\frac14 \delta \sigma^2
\leq 
\delta \sigma^2
\,.
\end{align*}
We note that~$\bfAhom(\cu_m)$ for~$\P$ is the same as~$\bfAhom(\cu_n)$ for~$\P_{m-n}$. 
The application of Proposition~\ref{p.renormalize.reduce} therefore yields that~$\Theta_m$ (which is~$\Theta_{n}$ for~$\P_{m-n}$) satisfies
\begin{equation*}
\Theta_m - 1
\leq 
\sigma \Theta_{\mathrm{new}} \leq 4 \sigma \Theta
\,.
\end{equation*}
This yields~\eqref{e.improve} with~$4\sigma$ in place of~$\sigma$; this can be fixed by enlarging the constant~$C$ in~\eqref{e.m.explivomit} and replacing~$\sigma$ with~$\frac1{4}\sigma$. 
By~$m\leq 2m_1$, the definition of~$m_1$ in~\eqref{e.params.choices} and the fact that~$k\mapsto \Theta_k$ is monotone nonincreasing, the proof is now complete.
\end{proof}

\subsection{One renormalization step} 
\label{s.one.renorm}

We present the proof of  Proposition~\ref{p.renormalize.reduce} in this subsection.
Throughout, we work with the following fixed parameters:

\begin{itemize}

\item $\delta_0 = \delta_0(d)\in (0,\nicefrac12]$ is a small constant that will be selected at the end of the proof;

\item $\sigma \in (0,\nicefrac12]$ and~$\delta \in (0,\delta_0]$ are given constants in the statement of Proposition~\ref{p.renormalize.reduce}.

\item We fix an integer~$l\in\N$ representing a mesoscopic scale; we take it to be the smallest positive integer satisfying the condition~\eqref{e.l.def.implies}, which we repeat here for convenience: 
\begin{equation} 
\label{e.l.def}
\max\biggl\{ 
\frac{K_{\Psi_\S}^{3\gamma} \Pi}{1-\gamma} 
3^{- \frac14(1-\beta) l} 
\,,\, 
K_{\Psi}^8  \Pi 3^{-\frac12(\nu -\gamma)(1-\beta) l} 
\,,\,
\frac{K_{\Psi_\S}^{16d} \Pi^4}{(1-\gamma)^4}  3^{ -l } 
\,,\,
\frac{3^{-(1 - \gamma) l }}{1-\gamma}  
\biggr\}
\leq 
\delta \sigma^2
\,.
\end{equation} 
We require also that~$l$ is large enough that~$3^l \Lat \subseteq \Zd$; in view of~\eqref{e.adapted.stationarity}, this can be ensured by taking~$\delta_0$ sufficiently small (depending only on~$d$).
We emphasize that the parameter~$l$ above depends on~$\delta$ and~$\sigma$, in addition to the other parameters~$(K_\Psi,K_{\Psi_\S},\gamma,\nu,\beta,\Pi)$, but this dependence is quite explicit. 

\item We assume that~$m\in \N$ with~$m\geq 100 l$ is such that~\eqref{e.renormalize.A.vs.E0} holds, that is, 
\begin{equation} \label{e.renormalize.A.vs.E0.again}
\bfAhom(\cu_0) \leq \bfE 
\qand 
\bigl| 
\bfAhom^{-\nicefrac12} (\cu_m) \bfE\bfAhom^{-\nicefrac12} (\cu_m)  - \Itwod \bigr|
\leq \delta \sigma^2
\,.
\end{equation}
\end{itemize}
To simplify the presentation,  throughout this subsection, we work with the following notational conventions and assumptions:

\begin{itemize}

\item 
To keep the notation short, for every~$j\in\N$ we define~$\bfAhom_j \coloneqq  \bfAhom(\cu_j)$,~$\shom_j \coloneqq \shom(\cu_j)$,~$\shom_{*,j} \coloneqq \shom_*(\cu_j)$ and~$\khom_j \coloneqq  \khom(\cu_j)$. Similarly we define~$\bhom_{j}  \coloneqq  
\shom_j + \khom_j^t\shom_{*,j}^{-1} \khom_j$ and, given a constant matrix~$\h$, we set~$\bhom_{\h,j}  \coloneqq  
\shom_j + (\khom_j - \h)^t\shom_{*,j}^{-1} (\khom_j - \h)$.

\item The coefficient field is ``centered'' so that the anti-symmetric part of a certain annealed coarse-grained matrix vanishes. By subtracting the antisymmetric matrix~$\frac12 (\khom_m - \khom_m^t)$ from the coefficient field, and recentering both~$\bfE$ and~$\bfA(U)$ accordingly (as in Section~\ref{ss.skew}), we may assume that
\begin{equation} 
\label{e.centering}
\khom_m = \khom_m^t  
\,.
\end{equation}
This centering does not alter~$\Theta$ or~$\Pi$. 

\item We choose~$\m_0$ to be the (metric) geometric mean of the matrices~$\b_0$ and~$\s_{*,0}$, denoted by
\begin{equation}
\label{e.m.naught.def}
\m_0  \coloneqq  \b_0\#\, \s_{*,0} 
\,.
\end{equation}
This notion of the geometric mean of two positive matrices is given in Appendix~\ref{s.geomean}. 
With this choice of~$\m_0$ in mind, throughout the rest of this section, we work with the adapted cubes~$\cus_m$ and the lattice~$\mathbb{L}_0$ defined in Section~\ref{ss.subadditivity}, above.

\end{itemize}
We stress that~\eqref{e.centering} is assumed to be in force for the rest of the proof of Proposition~\ref{p.renormalize.reduce}.

\smallskip

We next write down some inequalities involving the coarse-grained matrices that are needed throughout this subsection. 
We first claim that, in view of the above centering condition in~\eqref{e.centering},
\begin{equation} 
\label{e.mnaught.vs.bnaught}
\bigl| \m_0^{-\nicefrac12} \b_{0} \m_0^{-\nicefrac12}\bigr|  
= 
\bigl|\m_0^{\nicefrac12} \s_{*,0}^{-1} \m_0^{\nicefrac12} \bigr| 
=
\bigl| \s_{*,0}^{-\nicefrac12} \b_{0} \s_{*,0}^{-\nicefrac12}  \bigr|^{\nicefrac12}
\leq  
4 \Theta^{\nicefrac12} \,.
\end{equation}
By the definition of~$\m_0$ in~\eqref{e.m.naught.def} and property~\eqref{e.ricatti} of the geometric mean that
\begin{equation*} 
\m_0 \s_{*,0}^{-1} \m_0 = \s_0 + \k_0^t\s_{*,0}^{-1}\k_0 = \b_{0}\,,
\end{equation*}
from which we get the first identity in~\eqref{e.mnaught.vs.bnaught}.
To see the second identity, we compute as follows:
\begin{align}
\label{e.rat.bounds.before}
\bigl| \m_0^{-\nicefrac 12}  \b_{0} \m_0^{-\nicefrac 12} \bigr|
=
\bigl| 
\s_{*,0}^{-\nicefrac12} \m_0 \s_{*,0}^{-\nicefrac12} 
\bigr|
&
=
\bigl| \s_{*,0}^{-\nicefrac12}  \m_0 \s_{*,0}^{-1} \m_0\s_{*,0}^{-\nicefrac12} \bigr|^{\nicefrac12} 
=
\bigl|  \s_{*,0}^{-\nicefrac12}\b_{0}  \s_{*,0}^{-\nicefrac12}\bigr|^{\nicefrac12} 
\,.
\end{align}
To prove the last inequality of~\eqref{e.mnaught.vs.bnaught}, we first apply~\eqref{e.renormalize.A.vs.E0.again} to obtain
\begin{equation}
\label{e.Enaught.vs.Am}
\bfAhom_m  
\leq
\bfAhom_0
\leq
\bfE 
\leq 
\bigl| \bfAhom_m^{-\nicefrac12} \bfE \bfAhom_m^{-\nicefrac12}  \bigr| \bfAhom_m 
\leq 
2 \bfAhom_m 
\,.
\end{equation}
It follows by~\eqref{e.Eone.vs.Etwo.one} that 
\begin{equation} 
\label{e.szero.vs.sm}
\s_{*,0} \leq \shom_{*,m} \leq  2 \s_{*,0}
\qand
\s_0 \leq 2 \shom_m \leq 2 \bigl( \shom_m + (\khom_m - \k_0)^t \shom_{*,m}^{-1}
(\khom_m - \k_0) \bigr) 
\leq  4 \s_0
\,.
\end{equation}
Using the triangle inequality,~\eqref{e.szero.vs.sm},~\eqref{e.symm.k.quad.small} and~$\Theta_m \leq \Theta$, we get
\begin{align*} 
\bigl| \s_{*,0}^{-\nicefrac12} \b_0 \s_{*,0}^{-\nicefrac12} \bigr| 
& 
\leq 
\bigl| \s_{*,0}^{-\nicefrac12} \s_0 \s_{*,0}^{-\nicefrac12} \bigr| +
2 \bigl| \s_{*,0}^{-\nicefrac12} (\k_0 - \khom_m)^t\s_{*,0}^{-1} (\k_0 - \khom_m) \s_{*,0}^{-\nicefrac12} \bigr|
+
2 \bigl| \s_{*,0}^{-\nicefrac12} \khom_m^t \s_{*,0}^{-1} \khom_m \s_{*,0}^{-\nicefrac12} \bigr|
\notag \\ &
\leq
\Theta 
+
4\bigl| \s_{*,0}^{-\nicefrac12} (\k_0 - \khom_m)^t\s_{*,m}^{-1} (\k_0 - \khom_m) \s_{*,0}^{-\nicefrac12} \bigr|
+
8 \bigl| \s_{*,m}^{-\nf 12} \khom_m \s_{*,m}^{-\nicefrac12} \bigr|^2
\notag \\ &
\leq
\Theta 
+
8\bigl| \s_{*,0}^{-\nicefrac12} \s_0 \s_{*,0}^{-\nicefrac12} \bigr|
+
2(\Theta_m - 1)
\notag \\ &
\leq
16 \Theta
\,,
\end{align*}
which completes the proof of~\eqref{e.mnaught.vs.bnaught}. 
We next compute, using~\eqref{e.mnaught.vs.bnaught} and~\eqref{e.rat.bounds.before}, 
\begin{equation*}
| \m_0 | 
\leq 
| \s_{*,0}| 
| 
\s_{*,0}^{-\nicefrac12} \m_0 \s_{*,0}^{-\nicefrac12} 
|
\leq 
4 \Theta^{\nf12} | \s_{*,0}| 
\qquad \mbox{and} \qquad 
| \m_0^{-1} |
\leq 
| \b_0^{-1}  | 
| \m_0^{-\nicefrac 12}  \b_{0} \m_0^{-\nicefrac 12} | 
\leq 4\Theta^{\nf12} | \s_{*,0}^{-1} |\,.
\end{equation*}
As a result we obtain an upper bound on the quantity~$\Pi_{\m_0}$ defined in~\eqref{e.bigM.naught.def},
\begin{equation}
\label{e.Pi.Pi.old}
\Pi_{\m_0}
=
| \m_0 | |\m_0^{-1}| 
\leq
16\Theta \Pi 
\leq 
16 \Pi^{2} \,.
\end{equation}

\smallskip

We continue by transferring the assumed bounds~\eqref{e.renormalize.A.vs.E0.again} from Euclidean cubes to the adapted cubes and, using the mixing condition, obtain an estimate on the variance of~$\bfA(\cus_n)$ across roughly the same range of scales. 
We note that~$x \mapsto \bfA(x +\cus_n)$ is~$\Z^d$-stationary for every~$n \in \N$ with~$n \geq l$, since~$3^l \Lat \subseteq \Zd$. 

\begin{lemma}
\label{l.pigeonhole} 
There exists a constant~$C(d) < \infty$ such that
\begin{equation}
\label{e.first.pigeon.adapted}
\max_{n \in \N \cap [l,m-l]}
\bigl| 
\bfAhom_m^{-\nicefrac12} \bfAhom(\cus_n) \bfAhom_m^{-\nicefrac12}   - \Itwod \bigr|
\leq 
C \delta \sigma^2 
\end{equation}
and
\begin{equation} 
\label{e.pigeon.var}
\max_{n \in \N \cap [l,m-l]}
\E\Bigl  [
\bigl  |  \bfAhom_m^{-\nicefrac12} \bfA(\cus_n) \bfAhom_m^{-\nicefrac12}  - \Itwod \bigr  |^{2} \Bigr  ]
\leq 
C \delta \sigma^2
\,.
\end{equation}
\end{lemma}
\begin{proof}
Define~$h  \coloneqq  \lceil \frac14(1-\beta) l \rceil$. By~\eqref{e.l.def}, we have that each of the following conditions is valid:
\begin{equation*} 
\frac{K_{\Psi_\S}^8 \Pi^2}{(1-\gamma)^2}  3^{ -l } \leq \delta \sigma^2\,, \quad   K_{\Psi}^8  \Pi^2 3^{-\frac12(\nu -\gamma)(1-\beta) l} 
\leq \delta \sigma^2 
\qand
3^{- \frac14(1-\beta) l} \Pi \leq  \delta \sigma^2
\,.
\end{equation*}

\emph{Step 1.}
The proof of~\eqref{e.first.pigeon.adapted}. We will show that
\begin{equation}
\label{e.first.pigeon.adapted.pre}
\max_{n \in \N \cap [h,m-h]}
\bigl| 
\bfAhom_m^{-\nicefrac12} \bfAhom(\cus_n) \bfAhom_m^{-\nicefrac12} - \Itwod \bigr|
\leq 
C \delta \sigma^2 \,.
\end{equation}
This inequality is a consequence of~\eqref{e.renormalize.A.vs.E0.again} and Lemma~\ref{l.tilt.to.Euc}.
The latter implies that there exists a constant~$C (d)<\infty$ such that, for every~$n,m\in\N$ with $h\leq n \leq m-h$,
\begin{align*}
\Itwod
=
\bfAhom_m^{-\nicefrac 12}\bfAhom_m\bfAhom_m^{-\nicefrac 12}
&
\leq
 \bfAhom_m^{-\nicefrac 12}\bfAhom (\cus_n)\bfAhom_m^{-\nicefrac 12} 
+
C 3^{n-m}  \bfAhom_m^{-\nicefrac 12}\bfE  \bfAhom_m^{-\nicefrac 12}
\notag \\ &
\leq 
\bfAhom_m^{-\nicefrac 12}\bfAhom_{0} \bfAhom_m^{-\nicefrac 12}
+
\frac{C \Pi K_{\Psi_{\S}}^{3\gamma}}{1-\gamma}  \bigl( 3^{n-m} + 3^{-n} \bigr)  \bfAhom_m^{-\nicefrac 12}\bfE  \bfAhom_m^{-\nicefrac 12}
\notag \\ &
\leq 
\biggl( 1 + \frac{C \Pi K_{\Psi_{\S}}^{3\gamma}}{1-\gamma}  3^{-h}  \biggr)  \bigl(1 + \delta \sigma^2\bigr) \Itwod
\,.
\end{align*}
This string of inequalities implies~\eqref{e.first.pigeon.adapted} provided that~$3^h \geq ((1-\gamma)\delta \sigma^2)^{-1} \Pi K_{\Psi_{\S}}^{3\gamma}$, which is guaranteed by the choice~$h = \lceil \frac14(1-\beta) l \rceil$ and~\eqref{e.l.def}.

\smallskip

\emph{Step 2.}
We next show that,\footnote{The statement in Step~2 is valid without the assumptions of Proposition~\ref{p.renormalize.reduce}, as the proof does not use them. It also does not use the recentering assumption~\eqref{e.centering}. We will need these observations in Section~\ref{ss.pigeon.prime} when we need to reuse the arguments here.} for every~$m,n,k \in \N$ with~$k\leq n$, we have that 
\begin{align} 
\label{e.variance.HC}
\lefteqn{
\E\Bigl [
\bigl | 
\bfAhom^{-\nicefrac12} (\cu_m)  \bfA(\cus_n) \bfAhom^{-\nicefrac12} (\cu_m)  - \Itwod  
\bigr |^{2}
\Bigr ]
} \quad &
\notag \\ & 
\leq 
40d \max_{j \in \{k,n \}} \Bigl( \bigl  |\bfAhom^{-\nicefrac12} (\cu_m) \bfAhom(\cus_j) \bfAhom^{-\nicefrac12} (\cu_m) - \Itwod \bigr  | 
+ 
\bigl | \bfAhom^{-\nicefrac12} (\cu_m) \bfAhom(\cus_j) \bfAhom^{-\nicefrac12} (\cu_m)  - \Itwod \bigr  |^2 \Bigr)
\notag \\ &  \quad 
+ 
27 \E\Biggl[
\biggl |
\avsum_{z\in3^{k}\Lat \cap \cus_{n}} \! \! \!
\bfAhom^{-\nicefrac12} (\cu_m)
\bigl ( 
\bfA(z+\cus_k) - \bfAhom(z+\cus_k)
\bigr )
\bfAhom^{-\nicefrac12} (\cu_m)
\biggr |^2\Biggr]
\,.
\end{align}
By the triangle inequality, we get that
\begin{align}
\label{e.variance.basic.split}
\lefteqn{
\bigl | 
\bfAhom_m^{-\nicefrac12}  \bfA(\cus_n) \bfAhom_m^{-\nicefrac12}    - \Itwod  
\bigr |^2
} \qquad &
\notag \\ &
\leq
3\, \biggl | 
\avsum_{z\in3^{k}\Lat \cap \cus_{n}} 
\!\!\!
\bfAhom_m^{-\nicefrac12} \bigl (\bfA(z+\cus_k) - \bfAhom(\cus_k) \bigr ) \bfAhom_m^{-\nicefrac12} 
\biggr |^{2}
+
3 \, \Bigl | 
\bfAhom_m^{-\nicefrac12}  \bfAhom(\cus_k) \bfAhom_m^{-\nicefrac12}  - \Itwod
\Bigr |^{2}
\notag \\ &
\qquad
+
3\,\biggl | 
\bfAhom_m^{-\nicefrac12}   \!
\Bigl ( \bfA(\cus_n)   - \! \! \! \avsum_{z\in3^{k}\Lat \cap \cus_{n}} 
\!\!\!
\bfA(z+\cus_k) \Bigr ) \bfAhom_m^{-\nicefrac12} 
\biggr |^{2}
\,.
\end{align}
The first term and the second term appear on the right side of~\eqref{e.variance.HC}. The last term on the right side of~\eqref{e.variance.basic.split} is the square of the additivity defect, written in terms of~$\bfA$. 
By subadditivity, we also have that   
\begin{equation}
\label{e.monot.tilq}
\bfAhom_m^{-\nicefrac12}  
\biggl   ( \bfA(\cus_n)   - \avsum_{z\in3^{k}\Lat \cap \cus_{n}} \bfA(z+\cus_k) \biggr   )
\bfAhom_m^{-\nicefrac12} 
\leq
0
\,.
\end{equation}
We thus obtain
\begin{align*}
\lefteqn{
\E \Biggl [ \biggl| \bfAhom_m^{-\nicefrac12} 
\biggl   ( \bfA(\cus_n)   - \!\!\!\avsum_{z\in3^{k}\Lat \cap \cus_{n}} \bfA(z+\cus_k) \biggr   ) \bfAhom_m^{-\nicefrac12} 
\biggr|
\Biggr]
} \qquad &
\notag \\ &
\leq
2 d \biggl| \E \biggl [ \bfAhom_m^{-\nicefrac12} 
\biggl   ( \bfA(\cus_n)   - \avsum_{z\in3^{k}\Lat \cap \cus_{n}} \bfA(z+\cus_k) \biggr   ) \bfAhom_m^{-\nicefrac12} 
\biggr] \biggr|
\\ & 
\leq 2d
\bigl| \bfAhom_m^{-\nicefrac12} 
\bigl( 
 \bfAhom(\cus_k) 
-
\bfAhom(\cus_n)   
\bigr)
\bfAhom_m^{-\nicefrac12} 
\bigr|
\\ & 
\leq 
4 d \max_{j \in \{k,n \}} \bigl  |\bfAhom_m^{-\nicefrac12} \bfAhom(\cus_j) \bfAhom_m^{-\nicefrac12}   - \Itwod \bigr  | 
\,.
\end{align*}
We use~\eqref{e.monot.tilq} to  (crudely) estimate one factor from above: we have
\begin{align*}
\lefteqn{
\biggl | 
\bfAhom_m^{-\nicefrac12}  \!
\biggl ( \bfA(\cus_n)   - \avsum_{z\in3^{k}\Lat \cap \cus_{n}} \bfA(z+\cus_k) \biggr ) \bfAhom_m^{-\nicefrac12} 
\biggr | 
} \qquad & 
\notag \\ &
\leq 
\biggl |   
\avsum_{z\in3^{k}\Lat \cap \cus_{n}} \bfAhom_m^{-\nicefrac12} \bfA(z+\cus_k) \bfAhom_m^{-\nicefrac12} 
\biggr | 
\notag \\ &
\leq 
1+ \bigl  |\bfAhom_m^{-\nicefrac12}  \bfAhom(\cus_k) \bfAhom_m^{-\nicefrac12}   - \Itwod \bigr  | 
+
\biggl |
\avsum_{z\in3^{k}\Lat \cap \cus_{n}} 
\bfAhom_m^{-\nicefrac12} 
\bigl ( 
\bfA(z+\cus_k) - \bfAhom(z+\cus_k)
\bigr ) \bfAhom_m^{-\nicefrac12} 
\biggr |
\,.
\end{align*}
For any nonnegative random variable~$X$, we have $X^2 
\leq 2 X + 4 (X-1)_+^2$. Therefore, by the previous two displays, 
\begin{align} 
\lefteqn{
\E\Biggl [
\biggl  | 
\bfAhom_m^{-\nicefrac12} 
\biggl   ( \bfA(\cus_n)   - \! \! \! \! \avsum_{z\in3^{k}\Lat \cap \cus_{n}} \! \! \!  \bfA(z+\cus_{k}) \biggr   ) \bfAhom_m^{-\nicefrac12} 
\biggr  |^{2}
\Biggr ]
} \qquad &
\notag \\ & 
\leq 
8 d \max_{j \in \{k,n \}} \bigl  |\bfAhom_m^{-\nicefrac12} \bfAhom(\cus_j) \bfAhom_m^{-\nicefrac12}  - \Itwod \bigr  | 
+ 8 \bigl  |\bfAhom_m^{-\nicefrac12}  \bfAhom(\cus_k)\bfAhom_m^{-\nicefrac12}   - \Itwod \bigr  |^2
\notag \\ & \qquad
+ 
8 \E\Biggl[
\biggl |
\avsum_{z\in3^{k}\Lat \cap \cus_{n}} 
\bfAhom_m^{-\nicefrac12} 
\bigl ( 
\bfA(z+\cus_k) - \bfAhom(z+\cus_k)
\bigr )
\bfAhom_m^{-\nicefrac12} 
\biggr |^2\Biggr]
\notag
\,.
\end{align}
Combining this with~\eqref{e.variance.basic.split} yields~\eqref{e.variance.HC}.

\smallskip

\emph{Step 3.}
The proof of~\eqref{e.pigeon.var}. 
Let~$ n\in\N$ with~$l \leq n \leq m-l$  and let~$k  \coloneqq  n - h$. We then have that~$k - \beta n = (1-\beta) n - h \geq (1-\beta) l - h > 0$ and $k \geq h$.
Appealing to Lemma~\ref{l.bfA.CFS.adapted}, we obtain that,
\begin{multline*}
\biggl| \avsum_{z\in 3^{k}\Lat \cap \cus_n }\!\!\!\!\!
\bfE^{-\nicefrac12}  \bigl(  \bfA(z+\cus_k) - \bfAhom(z+\cus_k) \bigr) \bfE^{-\nicefrac12}
\biggr|
\\
\leq 
\frac{CK_{\Psi_\S}^2 \Pi}{1-\gamma} 3^{\gamma h - l}
+
\O_{\Psi_{\S}} \biggl ( \frac{C\Pi}{1-\gamma} 3^{\gamma h - l}\biggr )
+
\O_{\Psi}\Bigl( C\Pi^{\nicefrac12}3^{-(\nu-\gamma) h}  \Bigr)
\,.
\end{multline*}
Since~$h  \coloneqq  \lceil \frac14(1-\beta) l \rceil$, it follows by~\eqref{e.l.def} that
\begin{align}
\label{e.var.term.pre}
\lefteqn{
\E\Biggl[ 
\biggl |
\avsum_{z\in3^{k}\Lat \cap \cus_{n}} 
\! \! \! \! \bfAhom_m^{-\nicefrac12}    \bigl (\bfA(z+\cus_k) - \bfAhom(\cus_k) \bigr ) \bfAhom_m^{-\nicefrac12} 
\biggr |^2
\Biggr]
} \qquad &
\notag \\ & 
\leq
\frac{CK_{\Psi_\S}^8 \Pi^2 3^{ -l }}{(1-\gamma)^2}  + C K_{\Psi}^8  \Pi 3^{-\frac12(\nu -\gamma)(1-\beta) l} 
\leq 
C \delta \sigma^2
\,.
\end{align}
Since $k\geq h$, the expectation of the third term on the right side of~\eqref{e.variance.HC} is controlled by~\eqref{e.first.pigeon.adapted.pre}:
\begin{equation}
\label{e.var.term3}
\max_{j \in \{k,n \}} \bigl  |\bfAhom_m^{-\nicefrac12} \bfAhom(\cus_j) \bfAhom_m^{-\nicefrac12}   - \Itwod \bigr  | 
\leq
C \delta^2 \sigma^8
\leq 
C \delta \sigma^2
\,.
\end{equation}
Therefore, putting together~\eqref{e.var.term.pre},~\eqref{e.var.term3} and~\eqref{e.variance.HC} completes the proof of the lemma. 
\end{proof}

Our next goal is to relate~$\Theta_m-1$ to the variational quantities~$J$ and~$J^*$ and their maximizers. It is convenient to introduce the following variant of~$J$:
\begin{equation}
\tilde{J}(U,p,q)
\label{e.Jminusmeans}
 \coloneqq 
J(U,p,q)
-
\frac12 \E\biggl [ \fint_U \nabla v (\cdot,U,p,q)  \biggr ]
\cdot
\E\biggl [ \fint_U \a \nabla v (\cdot,U,p,q)  \biggr ]
\,.
\end{equation}
This ``centers'' the quantity~$J$ by removing the part of the energy due to the ``bias'' in the spatial averages of the gradient and flux of its maximizer~$v (\cdot,U,p,q)$.\footnote{This bias does not usually appear in the theory because a choice of the parameters~$p,q$ is typically made so that one of the factors vanishes. In the high contrast setting, this would create additional error terms that are too large. Here, we must be more careful in our choice of these parameters~$p,q$ so as to balance various error terms.}
We let~$\tilde{J}^*(U,p,q)$ denote the analogous quantity defined like in~\eqref{e.Jminusmeans} but for the adjoint coefficient field~$\a^t$.

\smallskip

We also introduce the following matrices: first, the anti-symmetric part of~$\khom(U)$ is denoted by
\begin{equation}
\label{e.hhomU}
\hhom(U)  \coloneqq 
\frac12(\khom-\khom^t)(U)\,,
\end{equation}
and the geometric mean of~$\bhom_{\hhom( U) }(U)$ and~$\shom_{*}(U)$ is given by 
\begin{equation}
\label{e.thomU}
\thom(U)  \coloneqq 
\bigl( \shom+ ( \khom - \hhom   )^t 
\shom_{*}^{-1} 
( \khom-\hhom  ) \bigr)  (U)
\#
\shom_{*}(U)
=
\bigl (\bhom_{\hhom( U) } \# \shom_{*}\bigr )(U)
\,.
\end{equation}
Note that our centering assumption~\eqref{e.centering} implies that $\hhom(\cu_m) = 0$, so that  
\begin{equation}
\label{e.thom.m.def}
\thom_m  \coloneqq   \bhom_m \#\, \shom_{*,m} = 
\bigl( \shom(\cu_m)+  \khom(\cu_m)   
\shom_{*}^{-1}(\cu_m) \khom(\cu_m)
\bigr)
\#\,
\shom_{*}(\cu_m)\,.
\end{equation}

In the next lemma we control the ratio of~$\thom(U)$ and~$\shom_*(U)$ by an expression involving the expectations of~$\tilde{J}(U,p,q)$ and~$\tilde{J}^*(U,p,q')$ and, consequently, give an upper bound for~$\Theta_m -1$ in terms of the latter for~$U=\cu_m$.

\begin{lemma}
\label{l.Jminusmeans.to.Theta}
For every bounded Lipschitz domain~$U\subseteq\Rd$, 
\begin{align}
\label{e.our.strategy}
\bigl |\bigl (\shom_{*}^{-\nicefrac12}\bhom_{\hhom(U)}  \shom_{*}^{-\nicefrac12} \bigr )(U) - \Id \bigr |
&
\leq  
2d \sup_{|e|=1}\biggl(
\E\Bigl    [ \tilde{J}(U,\thom^{-\nicefrac 12} (U) e, \thom^{\nicefrac 12} (U) e-\hhom (U) \thom^{-\nicefrac 12} (U) e) \Bigr  ]
\notag \\ & \qquad\qquad \quad
+ 
\E\Bigl  [\tilde{J}^*(U,\thom^{-\nicefrac 12} (U) e, \thom^{\nicefrac 12} (U) e+\hhom (U) \thom^{-\nicefrac 12} (U) e)\Bigr  ]
\biggr)
\,.
\end{align}
In particular, since~$\hhom(\cu_m) = 0$, 
\begin{equation}
\label{e.Thetam.by.Jildes}
\Theta_m - 1
\leq 
2 d 
\sup_{|e|=1}
\E\bigl  [ 
 \tilde{J}(\cu_m,\thom^{-\nicefrac 12}_m e , \thom^{\nicefrac 12}_m e) 
+ \tilde{J}^*(\cu_m,\thom^{-\nicefrac 12}_m e, \thom^{\nicefrac 12}_m e)\bigr  ]
\,.
\end{equation}
\end{lemma}
\begin{proof}
The lemma is a consequence of the following claim:
For every~$p,q,h\in \Rd$,
\begin{align}
\label{e.Jminusmeans.to.Theta}
\lefteqn{
\E\bigl [ \tilde{J}(U,p,q-h) + \tilde{J}^*(U,p,q+h) \bigr ]
}
\qquad &
\notag \\ &
=
q \cdot \bigl( \shom(U)\shom_*^{-1}(U) -\Id \bigr) p
+
\shom_*^{-1}(U)q
\cdot
(\khom(U)+\khom^t(U))
\shom_*^{-1}(U)(\khom(U)p - h)
\,.
\end{align}
Indeed, we obtain~\eqref{e.our.strategy} from~\eqref{e.Jminusmeans.to.Theta} by using the definitions of~$\thom(U)$ and~$\hhom(U)$ in~\eqref{e.thomU} and~\eqref{e.hhomU}, respectively, and the following computation (all of the matrices are evaluated at~$U$):
\begin{align*}
\lefteqn{
\sum_{j=1}^d
\E\Bigl  [ \tilde{J}(U,\thom^{-\nicefrac 12} e_j , \thom^{\nicefrac 12} e_j-\hhom  \thom^{-\nicefrac 12} e_j) + \tilde{J}^*(U,\thom^{-\nicefrac 12} e_j , \thom^{\nicefrac 12} e_j + \hhom  \thom^{-\nicefrac 12} e_j)\Bigr  ]
}
\qquad &
\notag \\ &
=
\sum_{j=1}^d
\Bigl (
e_j \cdot \thom^{\nicefrac 12}\bigl( \shom \shom_{*}^{-1} -\Id \bigr)\thom^{-\nicefrac 12} e_j
+
\frac12 e_j
\cdot
\thom^{\nicefrac 12} \shom_{*}^{-1} (\khom  +\khom ^t )
\shom_{*}^{-1}(\khom  +\khom ^t )\thom^{-\nicefrac 12} e_j
\Bigl )
\notag \\ &
=
\tr\bigl( \thom^{\nicefrac 12}\bigl( \shom \shom_{*}^{-1} -\Id \bigr)\thom^{-\nicefrac 12} \bigr)
+\frac12 \tr\bigl( \thom^{\nicefrac 12} \shom_{*}^{-1} (\khom  +\khom ^t )
\shom_{*}^{-1}(\khom  +\khom ^t )\thom^{-\nicefrac 12} \bigr)
\notag \\ &
= 
\tr\bigl(  \shom_{*}^{-\nicefrac12} \shom \shom_{*}^{-\nicefrac12} - \Id \bigr)
+ \frac12 \tr\bigl( \shom_{*}^{-\nicefrac12} (\khom  -\hhom )^t\,
\shom_{*}^{-1}(\khom  -\hhom )  \shom_{*}^{-\nicefrac12}  \bigr )
\notag \\ &
\geq 
\frac12
\tr\bigl(  \shom_{*}^{-\nicefrac12} \bigl( \shom +  (\khom  -\hhom )^t\,
\shom_{*}^{-1}(\khom  -\hhom ) \bigr ) \shom_{*}^{-\nicefrac12} -\Id \bigr )
\notag \\ &
\geq 
\frac12
\bigl|  \shom_{*}^{-\nicefrac12} \bigl( \shom +  (\khom  -\hhom )^t\,
\shom_{*}^{-1}(\khom  -\hhom ) \bigr ) \shom_{*}^{-\nicefrac12} -\Id \bigr |
\,,
\end{align*}
where the cyclic property of traces was applied to obtain the last equality and the non-negativity of the matrices inside of the traces to obtain the last inequalities.

\smallskip

We turn to the proof of~\eqref{e.Jminusmeans.to.Theta}. Recall from~\eqref{e.J.mat} and~\eqref{e.meet.the.homs} that we have the identity 
\begin{equation}
\label{e.J.mat.E}
\E \bigl[ J(U,p,q) \bigr] 
= 
\frac 12p \cdot \shom(U) p 
+ \frac 12 (q+\khom(U) p) \cdot \shom_*^{-1}(U) (q+\khom(U) p) 
- p \cdot q
\,.
\end{equation}
According to~\eqref{e.all.averages.entries}, we have that 
\begin{equation}
\label{e.all.averages.entries.again}
\left\{
\begin{aligned}
& 
\E \biggl[ \fint_U 
\nabla v(\cdot,U,p,q)\biggr] 
=
- p + \shom_{*}^{-1}(U) (q + \khom(U) p)
\\ & 
\E \biggl[ 
\fint_U 
\a \nabla v(\cdot,U,p,q) \biggr]
=
(\Id -\khom^t  \shom_{*}^{-1}\bigr )(U)  q - \bhom(U) p\,.
\end{aligned}
\right.
\end{equation}
Therefore, 
\begin{align*}
\lefteqn{ 
\E \biggl[ \fint_U 
\nabla v(\cdot,U,p,q)\biggr] \cdot
\E \biggl[ 
\fint_U 
\a \nabla v(\cdot,U,p,q) \biggr]
} \qquad & 
\notag \\ & 
=
\bigl( - p + \shom_{*}^{-1}(U) (q + \khom(U) p) \bigr) 
\cdot 
\bigl( (\Id -\khom^t  \shom_{*}^{-1}\bigr )(U)  q - \bhom(U) p \bigr)
\notag \\ & 
=
p\cdot \shom  p 
+ (q+\khom p) \cdot \shom_{*}^{-1} (q+\khom p) 
-
(q+\khom p) \cdot 
\shom_*^{-1} \khom \shom_{*}^{-1}
(q+\khom p)
- \shom_*^{-1}\shom p \cdot (q+\khom p)
-p\cdot q\,,
\end{align*}
where in the last line and in the rest of the argument, 
we suppress the dependence on~$U$ from the notation since it plays no important role.
Combining~\eqref{e.J.mat.E} and the previous display, we obtain  the following expression for~$\E [ \tilde{J}(U,p,q)]$ in terms of the coarse-grained matrices:
\begin{align*}
\E \bigl[ \tilde{J}(U,p,q) \bigr] 
&
=
\E \bigl[ J(U,p,q) \bigr] 
-
\frac12 \E\biggl [ \fint_U \nabla v (\cdot,U,p,q)  \biggr ]
\cdot
\E\biggl [ \fint_U \a \nabla v (\cdot,U,p,q)  \biggr ]
\notag \\ & 
= 
\frac12 (q+\khom p) \cdot 
\shom_*^{-1} \khom \shom_{*}^{-1}
(q+\khom p)
+\frac12 \shom_*^{-1}\shom p \cdot (q+\khom p)
-\frac12 p\cdot q
\,.
\end{align*}
By a similar computation, we obtain the following formula for~$\E [ \tilde{J}^*(U,p,q)]$:
\begin{align*}
\E \bigl[ \tilde{J}^*(U,p,q) \bigr] 
= 
\frac12 (q-\khom p) \cdot 
\shom_*^{-1} \khom \shom_{*}^{-1}
(q-\khom p)
+\frac12 \shom_*^{-1}\shom p \cdot (q-\khom p)
-\frac12 p\cdot q
\,.
\end{align*}
Combining the above displays gives us
\begin{equation*} 
\E\bigl [ \tilde{J}(U,p,q{-}h) + \tilde{J}^*(U,p,q{+}h) \bigr ] 
=
p \cdot (\shom \shom_*^{-1} - \Id)   q + (\khom p - h )\cdot \shom_*^{-1} (\khom + \khom^t) \shom_*^{-1} q 
\,.
\end{equation*}
This completes the proof of~\eqref{e.Jminusmeans.to.Theta} and thus of~\eqref{e.our.strategy}.

\smallskip

In view of the fact that~$\hhom(\cu_m)=0$, the combination of~\eqref{e.Theta.n.def} and~\eqref{e.our.strategy} implies~\eqref{e.Thetam.by.Jildes}.
\end{proof}

Motivated by Lemma~\ref{l.Jminusmeans.to.Theta},
our goal is now to get an upper bound estimate on
\begin{equation}
\label{e.whatwewanttobound}
\E\bigl  [ \tilde{J}(\cu_m,p,q) + \tilde{J}^*(\cu_m,p,q)\bigr  ] 
\,,
\end{equation}
with the choices
\begin{equation}
\label{e.pqh.choice}
p \coloneqq  \thom^{-\nicefrac 12}_m e
\quad \mbox{and} \quad
q \coloneqq  \thom^{\nicefrac 12}_m e
 \,,
\end{equation}
where~$\thom_m$ is defined by~\eqref{e.thom.m.def} and~$e\in\Rd$ with $|e|=1$ is the unit vector attaining the supremum on the right side~\eqref{e.our.strategy} with~$U=\cu_m$: that is, 
\begin{align}
\label{e.e.choice}
\E\bigl  [ \tilde{J}(\cu_m,p,q) + \tilde{J}^*(\cu_m,p,q)\bigr  ]
=
\sup_{|e'|=1}
\E\Bigl  [ 
 \tilde{J}(\cu_m,\thom^{-\nicefrac 12}_m e' , \thom^{\nicefrac 12}_m e') 
+ \tilde{J}^*(\cu_m,\thom^{-\nicefrac 12}_m e' , \thom^{\nicefrac 12}_m e')\Bigr  ]
\,.
\end{align}

To get a bound on the right side of~\eqref{e.whatwewanttobound}, we 
first switch to the adapted cubes at the cost of giving up a few scales: using~\eqref{e.Euc.by.tilt} and our choice of~$(p,q)$ in~\eqref{e.pqh.choice}, we have that, for~$n=m-l$, 
\begin{align}
\label{e.dragon.egg}
\lefteqn{
\E\bigl  [ \tilde{J}(\cu_m,p,q) + \tilde{J}^*(\cu_m,p,q)\bigr  ]
} \quad  & 
\notag \\ &
=
\E\bigl  [ {J}(\cu_m,p,q) + {J}^*(\cu_m,p,q)\bigr  ]
- \frac12 P\cdot Q -  \frac12 P^*\cdot Q^*
\notag \\ & 
\leq
\E\bigl  [ {J}(\cus_n,p,q) + {J}^*(\cus_n,p,q)\bigr  ]
- \frac12 P\cdot Q -  \frac12 P^*\cdot Q^*
+
\frac{C\Pi^2 K_{\Psi_{\S}}^{3\gamma}}{1-\gamma} 3^{-(m-n)} 
\,,
\end{align}
where we let~$P,Q,P^*,Q^*\in\Rd$ denote the vectors
\begin{equation}
\label{e.PQ.choice}
\begin{pmatrix}
P \\ Q 
\end{pmatrix}
 \coloneqq  \E\biggl[ \fint_{\cu_m} 
\begin{pmatrix}
\nabla v(\cdot,\cu_m,p,q)  \\ \a \nabla v(\cdot,\cu_m,p,q)
\end{pmatrix}
\biggr]
\qand
\begin{pmatrix}
P^* \\ Q^* 
\end{pmatrix}
 \coloneqq  \E\biggl[ \fint_{\cu_m} 
\begin{pmatrix}
\nabla v^*(\cdot,\cu_m,p,q)  \\ \a^t \nabla v^*(\cdot,\cu_m,p,q)
\end{pmatrix}
\biggr]
\,.
\end{equation}
Therefore, what we want to bound is the quantity
\begin{equation}
\label{e.thequantity}
\E\bigl  [ J(\cus_n,p , q ) \bigr]
-\frac12 P \cdot Q\,,
\end{equation}
where~$p$ and~$q$ are as in~\eqref{e.pqh.choice},~$e$ is chosen so that~\eqref{e.e.choice} holds,~$P$ and~$Q$ are defined in~\eqref{e.PQ.choice}, and~$n\in \N$ is some suitably chosen parameter with~$n<m$. The analogous bound for~$J^*$ is a consequence of the one for~$J$ if we apply it with the random field~$\a^t$ in place of~$\a$.

\smallskip

The idea now is to write the quantity in~\eqref{e.thequantity} 
as the integral of the product of the centered gradient and centered flux of their maximizers:
\begin{align}
\label{e.Jtilde.to.energy}
\E\bigl  [  J(\cus_n,p , q)  \bigr]
-\frac12 P \cdot Q
&
=
\frac12 
\begin{pmatrix} 
Q \\ 
P 
\end{pmatrix} 
\cdot
\E \biggl[
\begin{pmatrix} 
\nabla v(\cdot,\cus_n,p,q) - P \\ 
\a \nabla v(\cdot,\cus_n,p,q) - Q 
\end{pmatrix} 
\biggr]
\notag \\ & \qquad
+
\frac12
\E\biggl[ 
\fint_{\cus_n}\!
\bigl( \nabla v(\cdot,\cus_n,p,q) - P \bigr) \cdot
\bigl(\a\nabla v(\cdot,\cus_n,p,q) - Q\bigr)
\biggr]
\,.
\end{align}
Note the formula~\eqref{e.Jtilde.to.energy} is valid for any~$P$ and~$Q$ and does not use the particular choice in~\eqref{e.PQ.choice}. 

\smallskip

The way one would normally proceed (in the uniformly elliptic case) with estimating the right side of~\eqref{e.Jtilde.to.energy} is to use the additivity defect to estimate the expectation on the first line since the spatial averages of gradients and fluxes can be expressed in terms of the coarse-grained matrices themselves. The second line is then typically estimated using the Caccioppoli inequality, which reduces the energy term to an~$L^2$ type oscillation of the maximizer itself, which can then be reduced once again to the spatial averages of the gradient and thus the additivity defect.

\smallskip

The problem with this strategy in our context is that the Caccioppoli inequality produces estimates with factors of the (pointwise) ellipticity constants of the microscopic matrix~$\a(x)$. Since we do not assume that~$\a(\cdot)$ is uniformly elliptic, this strategy is unavailable; even if we did make a uniform ellipticity assumption, this estimate would produce factors of the ellipticity ratio that would result in a much worse estimate than the one we will prove. 

\smallskip

We instead proceed in a more \emph{coarse-grained} fashion by directly relating the energy of the maximizer to the weak Sobolev norms of the gradient and flux---while paying only the price required by the \emph{coarse-grained} ellipticity ratio.

\smallskip

This is the content of the next lemma, in which we control the left side of~\eqref{e.Jtilde.to.energy} by using a similar identity to~\eqref{e.Jtilde.to.energy}, but with a cutoff function smuggled in, and then a quantitative ``div-curl'' type argument  to control the energy term.

\begin{lemma}
\label{l.Jtilde.energy.bound}
There exists a constant~$C(d)<\infty$ such that, for~$n = m - l$, we have the estimate
\begin{align}
\label{e.Jtilde.energy.bound}
\Bigl| \E\bigl  [ J(\cus_n,p , q ) \bigr]
- \frac12 P \cdot Q 
\Bigr|
\leq
C 3^{-n} 
\E \Biggl[
\biggl[ 
\mathbf{M}_0^{\nicefrac12} 
\begin{pmatrix} 
\nabla v(\cdot,\cus_n,p,q) - P \\ 
\a \nabla v(\cdot,\cus_n,p,q) - Q 
\end{pmatrix} 
\biggr]_{\Besov{-\nf12}{2}{1}(\cus_{n})}^2
\Biggr]
+
C \delta^{\nicefrac12} \sigma 
\Theta
\,.
\end{align}
\end{lemma}
\begin{proof}
We first prove a preliminary statement which is valid for general~$P,Q \in \Rd$ (not only for the specific choices made in~\eqref{e.PQ.choice}), and also for a general~$\m_0$ and general~$\mathbf{M}_0$ with the defining property that~$3^{k_0} \Lat \subseteq \Zd$ and~$\mathbf{q}_0 = \mbox{const} \cdot \mathbf{m}_0^{\nicefrac12}$ (not only for the specific choices made in~\eqref{e.m.naught.def} and~\eqref{e.q.naught}). We then conclude in Step 6 below using the particular choices~\eqref{e.PQ.choice},~\eqref{e.m.naught.def} and~\eqref{e.q.naught}.

\smallskip

To simply notation, we drop~$p,q$ and denote, for every~$z \in \Z^d$ and~$k,n \in \N$ with~$n > k+ 3 > k \geq k_0$, 
\begin{equation*}  
v_{k,z}  \coloneqq  v(\cdot,z +\cus_k,p,q)
\,, \quad
J(z+\cus_k)
=
J(z+\cus_k,p,q)
\,, \quad 
\bar{\tau}_{n,k}  \coloneqq  
\E\big[ J(\cus_k) - J(\cus_n)\bigr]
\,.
\end{equation*}
Denote also~$v_n  \coloneqq  v(\cdot,\cus_n,p,q)$.  We fix a nonnegative smooth test function~$\varphi \in C_{c}^{\infty}(\cus_n)$ such that 
\begin{equation*}  
(\varphi)_{\cus_n}  = 1 \,, 
\quad 
0 \leq  \varphi \leq 2
\qand
\bigl \| \mathbf{q}_0^{j} \nabla^j \varphi \bigr \|_{L^\infty(\cus_n)} \leq 3^{-j(n-2)}\,, \quad j \in \{1,2\}\,.
\end{equation*}
Notice that, by~\eqref{e.first.variation},~\eqref{e.quadratic.response} and~\eqref{e.J.energy}, we have
\begin{equation} 
\label{e.subadd.gives.this}
\frac12 \avsum_{z \in 3^k \Lat \cap \cus_n} \E\Bigl[ \bigl\| \s^{\nicefrac12} \nabla (v_n - v_{k,z})  \bigr\|_{\underline{L}^2(z+\cus_k)}^2 \Bigr] = \bar{\tau}_{n,k} \,.
\end{equation}
Next, observe that
\begin{align}
\label{e.Jtilde.decomp}
\E\bigl  [ J(\cus_n)\bigr]
-\frac12 P \cdot Q
&
=
\E\biggl[ 
\fint_{\cus_n} 
\frac12 \varphi \bigl( \nabla v_{n} - P \bigr) \cdot \bigl(\a \nabla v_{n} - Q \bigr) 
\biggr]   
\notag \\ 
& \qquad 
+ 
\E\biggl[ J(\cus_n ) -  
\fint_{\cus_n} \frac12 \varphi \nabla v_{n}  \cdot \s \nabla v_{n}  
\biggr]   
\notag \\ 
& \qquad 
+ 
\frac12 Q \cdot \E\bigl [\bigl( (\varphi-1) \nabla v_n)_{\cus_n} \bigr ]
+ 
\frac12 P \cdot \E\bigl [\bigl( (\varphi-1) \a \nabla v_n)_{\cus_n} \bigr ]
\notag \\ & \qquad 
+
\frac12 Q \cdot \bigl( \E\bigl[(\nabla v_n)_{\cus_n}\bigr] - P )  
+ \frac12 P \cdot \bigl( \E\bigl[(\a \nabla v_n)_{\cus_n}\bigr] - Q ) 
\,.
\end{align}
We proceed by estimating each of the four lines on the right side of~\eqref{e.Jtilde.decomp} separately.

\smallskip

\emph{Step 1.}
We show that there exists a constant~$C(d)<\infty$ such that
\begin{equation}  
\label{e.divcurl.applied}
\biggl| 
\fint_{\cus_n}
\varphi ( \nabla v_{n} - P ) \cdot
(\a \nabla v_{n} - Q)
\biggr|
\leq 
C3^{-n} \bigl[ \m_0^{\nicefrac12} (\nabla v_n - P)  \bigr]_{\Besov{-\nf12}{2}{1}(\cus_{n})} \bigl [  \mathbf{m}_{0}^{-\nicefrac12} (\a \nabla v_{n} - Q) \bigr]_{\underline{B}_{2,1}^{-\nicefrac12}(\cus_n)} 
\,.
\end{equation}
By Lemma~\ref{l.testing.makes.sense} we may test the equation for~$v_n$ with~$(v_n - \ell_P) \varphi$ with~$\ell_P(x)  \coloneqq  (v_n)_{\cus_n} +  P \cdot x$ and integrate by parts. Using~\eqref{e.szero.vs.qzero}, we get  
\begin{align}   
\biggl| 
\fint_{\cus_n}
\varphi ( \nabla v_{n} - P ) \cdot
(\a \nabla v_{n} - Q)
\biggr|
 &
=
\biggl| 
\fint_{\cus_n}
 ( v_{n} - \ell_P ) \mathbf{m}_{0}^{\nicefrac12}  \nabla \varphi \cdot
\mathbf{m}_{0}^{-\nicefrac12}  (\a \nabla v_{n} - Q)
\biggr|
\notag \\ &
\leq 
C
\bigl \| (v_{n} - \ell_{P})  \mathbf{m}_{0}^{\nicefrac12} \nabla \varphi \bigr\|_{\underline{B}_{2,\infty}^{\nicefrac12}(\cus_n)} 
\bigl [  \mathbf{m}_{0}^{-\nicefrac12} (\a \nabla v_{n} - Q) \bigr]_{\underline{B}_{2,1}^{-\nicefrac12}(\cus_n)} 
\,.
\notag
\end{align}
By applying~\eqref{e.divcurl.est1.cus}, we find that 
\begin{equation*}  
\bigl \| (v_{n} - \ell_{P})  \mathbf{m}_{0}^{\nicefrac12} \nabla \varphi  \bigr\|_{\underline{B}_{2,\infty}^{\nicefrac12}(\cus_n)} 
\leq 
C 3^{-n} \bigl[ \m_0^{\nicefrac12} (\nabla v_{n} - P)  \bigr]_{\Besov{-\nf12}{2}{1}(\cus_{n})}
\,,
\end{equation*}
and~\eqref{e.divcurl.applied} follows. 

\smallskip

\emph{Step 2.} 
We show that, for every~$k\in\N$ with~$k_0 \leq k\leq n$, 
\begin{align}
\label{e.cutoffenergy.vs.J}
\biggl| 
\E\biggl[ 
\fint_{\cus_n} \frac12  \varphi \nabla v_{n}  \cdot \s \nabla v_{n}  
\biggr]   
-
\E\bigl[ J(\cus_n
) \bigr] 
\biggr| 
\leq
2 \bar{\tau}_{n,k} 
+ 
4 \bar{\tau}_{n,k}^{\nicefrac 12}\E[J(\cus_n  )]^{\nicefrac12} 
+ 
3^{k + 2 -n}  \E\bigl[ J(\cus_n ) \bigr] 
\,.
\end{align}
We split the energy term, writing it as
\begin{align} \label{e.cutoffenergy.vs.J.prepre}
\frac12 \E\biggl[  
\fint_{\cus_n} \varphi \nabla v_{n}  \cdot \s \nabla v_{n}  
\biggr]  
& = 
 \frac12   \avsum_{z\in 3^{k}\Lat  \cap \cus_n} 
(\varphi)_{z+ \cus_k}  \E\biggl[   \fint_{z+ \cus_k}  \nabla v_{n}  \cdot \s \nabla v_{n}  \biggr]  
\notag \\ 
& \qquad 
+   \frac12   \avsum_{z\in 3^{k}\Lat  \cap \cus_n} \E\biggl[  \fint_{z+ \cus_k} (\varphi -  (\varphi)_{z+ \cus_k} ) 
\nabla v_{n}  \cdot \s \nabla v_{n} \biggr]  
\,.
\end{align}
The second term is very small: we have that 
\begin{align} \label{e.cutoffenergy.vs.J.pre1}
\lefteqn{
\Biggl | 
\avsum_{z\in 3^{k}\Lat  \cap \cus_n} 
\E\biggl[  
\fint_{z+ \cus_k} (\varphi -  (\varphi)_{z+ \cus_k} )  \nabla v_{n}  \cdot \s \nabla v_{n}  
\biggr]  
\Biggr |
} \qquad  & 
\notag \\ & 
\leq 
\max_{z\in 3^{k}\Lat  \cap \cus_n}\|  \varphi -  (\varphi)_{z+ \cus_k}   \|_{L^\infty(z+ \cus_k)}
\E\biggl[  \fint_{\cus_n}  \nabla v_{n}  \cdot \s \nabla v_{n} \biggr]  
\leq
2\cdot 3^{- (n-k) + 2}  \E\bigl[ J(\cus_n)\bigr] 
\,.
\end{align}
Next, we have by~\eqref{e.J.energy} that
\begin{align*} 
\frac12 \fint_{z+ \cus_k}  \nabla v_{n}  \cdot \s \nabla v_{n}  
& = 
\frac12 \fint_{z+ \cus_k}  \nabla v_{k,z}  \cdot \s \nabla v_{k,z} + 
\frac12 \fint_{z+ \cus_k}  \nabla (v_n - v_{k,z})  \cdot \s \nabla (v_n + v_{k,z} )
\notag \\ & 
=  J(z+\cus_k) + \frac12
\fint_{z+ \cus_k}  \nabla (v_n - v_{k,z})  \cdot \s \nabla (v_n + v_{k,z} )
\,.
\end{align*}
Taking the expectation of the previous display and using~$\Zd$--stationarity and~$(\varphi)_{\cus_m}=1$, we get
\begin{align*} 
\lefteqn{
\frac12 \avsum_{z\in 3^{k}\Lat  \cap \cus_n}
(\varphi)_{z+ \cus_k} \E\biggl[ \fint_{z+ \cus_k}  \nabla v_{n}  \cdot \s \nabla v_{n} \biggr]
} \qquad & 
\notag \\ &
= \E\bigl[J(\cus_k) \bigr] 
+
\frac12  \avsum_{z\in 3^{k}\Lat  \cap \cus_n}
(\varphi)_{z+ \cus_k}    
\E\biggl[  \fint_{z+ \cus_k}  \nabla (v_n - v_{k,z})  \cdot \s \nabla (v_n + v_{k,z} )\biggr]
\,.
\end{align*}
Using H\"older's inequality,~$\Zd$--stationarity,~\eqref{e.J.energy},~\eqref{e.subadd.gives.this} and~$\| \varphi \|_{L^\infty(\cus_n)} \leq 2$, we have
\begin{align*} 
\lefteqn{
\frac12  \avsum_{z\in 3^{k}\Lat  \cap \cus_n}
(\varphi)_{z+ \cus_k}    
\E\biggl[  \fint_{z+ \cus_k}  \nabla (v_n - v_{k,z})  \cdot \s \nabla (v_n + v_{k,z} )\biggr]
} \qquad &
\notag \\ &
\leq 
2 \bigl(\E\bigl[J(\cus_k) \bigr]  + \E\bigl[J(\cus_n) \bigr] \bigr)^{\nicefrac12} 
\bar{\tau}_{n,k}^{\nicefrac12} 
\leq 
2^{\nicefrac32} \E\bigl[J(\cus_n) \bigr]^{\nicefrac12} \bar{\tau}_{n,k}^{\nicefrac12}  + 2  \bar{\tau}_{n,k}
\,.
\end{align*}
By combining the previous three displays with~\eqref{e.cutoffenergy.vs.J.prepre} and~\eqref{e.cutoffenergy.vs.J.pre1}, we obtain~\eqref{e.cutoffenergy.vs.J}.

\smallskip

\emph{Step 3.}
We next show that there exists~$C(d)<\infty$ such that,
for every~$k\in\N$ with~$k_0\leq k\leq n$,
\begin{multline}
\label{e.Pvarphi.incontrol}
\bigl| \E\bigl[  ( (\varphi-1) \nabla v_{n}  )_{\cus_n}\bigr]  \cdot Q \bigr|
\leq
(2 \bar{\tau}_{n,k})^{\nicefrac12} \bigl| \shom_{*}^{-\nicefrac12}(\cus_k) Q \bigr|  
\\
+
C  3^{-(n{-}k)} \E\bigl[ J(\cus_n)\bigr]^{\nicefrac12} 
\biggl( \sum_{j=-\infty}^k 
3^{\frac32(j-k)}   \!\!\!\!
 \avsum_{z \in 3^j\Lat \cap \cus_k} \! \!  \! \bigl|\shom_{*}^{-\nicefrac12}(z +\cus_j) Q \bigr|^2 \biggr)^{\! \nicefrac12} 
\,.
\end{multline}
We first rewrite
\begin{align} \label{e.split.Pvarphi}
\E\bigl[  \bigl( (1-\varphi) \nabla v_{n} \bigr)_{\cus_n}\bigr] 
&
=
\avsum_{z\in 3^{k}\Lat  \cap \cus_n}
\bigl( 1 - (\varphi)_{z+\cus_k}\bigr) \E\bigl[ ( \nabla v_{n} - \nabla v_{k,z})_{z + \cus_k} \bigr] 
\notag \\ & \qquad
+ \avsum_{z\in 3^{k}\Lat  \cap \cus_n}
\bigl( 1 - (\varphi)_{z+\cus_k}\bigr) \E\bigl[ (\nabla v_{k,z})_{z + \cus_k} \bigr]  
\notag \\  & \qquad 
+
\avsum_{z\in 3^{k}\Lat  \cap \cus_n} \, 
\E\biggl[ \fint_{z + \cus_k} \bigl( (\varphi)_{z+\cus_k}  - \varphi \bigr) \nabla v_{n} \biggr]
\,.
\end{align}
By H\"older's inequality,~\eqref{e.energymaps.nonsymm} and stationarity, we get
\begin{align} 
\label{e.energymaps.nonsymm.mean}
\Bigl| \E\bigl[ ( \nabla v_{n} - \nabla v_{k,z})_{z + \cus_k} \bigr] \cdot Q \Bigr|
&\leq
\E\Bigl[ \bigl| \s_{*}^{\nicefrac12} (z+\cus_k)( \nabla v_{n} - \nabla v_{k,z})_{z + \cus_k} \bigr|^2 \Bigr]^{\nicefrac12} \E\Bigl[ \bigl| \s_{*}^{-\nicefrac12}(z+\cus_k) Q\bigr|^2\Bigr]^{\nicefrac12} 
\notag \\ &
\leq  \E\Bigl[ \bigl\| \s^{\nf12} ( \nabla v_{n} - \nabla v_{k,z} ) \bigr\|_{\underline{L}^2(z+\cus_k)}^2\Bigr]^{\nicefrac12} \bigl| \shom_{*}^{-\nicefrac12} (\cus_k)Q\bigr|
\,.
\end{align}
The contribution of the first term on the right in~\eqref{e.split.Pvarphi} can then be estimated using the above display, H\"older's inequality and~$0 \leq \varphi \leq 2$:
\begin{equation*}  
\avsum_{z\in 3^{k}\Lat  \cap \cus_n}\Bigl| 
\bigl( 1 - (\varphi)_{z+\cus_k}\bigr) 
\E\bigl[ ( \nabla v_{n} - \nabla v_{k,z})_{z + \cus_k} \bigr] \cdot Q \Bigr|
\leq 
(2\bar{\tau}_{n,k})^{\nicefrac12} \bigl| \shom_{*}^{-\nicefrac12}(\cus_k) Q \bigr| 
\,.
\end{equation*}
The second term on the right side of~\eqref{e.split.Pvarphi}, by~$(\varphi)_{\cus_n} = 1$ and~$\Z^d$-stationarity, is zero:
\begin{equation*}  
\avsum_{z\in 3^{k}\Lat  \cap \cus_n}
\bigl( 1 - (\varphi)_{z+\cus_k}\bigr) \E\bigl[ (\nabla v_{k,z})_{z + \cus_k} \bigr]   
=
\E\bigl[ (\nabla v_{0,k})_{\cus_k} \bigr] 
\avsum_{z\in 3^{k}\Lat  \cap \cus_n}
\bigl( 1 - (\varphi)_{z+\cus_k}\bigr)  
=
0
\,.
\end{equation*}
The last term on the right side of~\eqref{e.split.Pvarphi} is small, and we start estimating it, by using~\eqref{e.energymaps.nonsymm}:
\begin{align*}  
\biggl| 
\fint_{z + \cus_k} \bigl( (\varphi)_{z+\cus_k}  {-} \varphi \bigr) \nabla v_{n}  \cdot  Q 
\biggr|
&
\leq 
\bigl[ \varphi \bigr]_{\underline{B}_{1,\infty}^{1}(z + \cus_k)} 
\bigl[ \nabla v_{n}  \cdot  Q \bigr]_{\underline{B}_{1,1}^{-1}(z + \cus_k)} 
\notag \\ &
\leq
C 3^{-n}
\sum_{j =-\infty}^k 3^{j}
\avsum_{z' \in 3^j\Lat \cap (z + \cus_k)}
\bigl| \bigl( \nabla v_{n} \bigr)_{z'+\cus_j} \cdot  Q \bigr |
\notag \\ &
\leq
C 3^{-n}
\sum_{j =-\infty}^k 3^{j}
\! \! \! \avsum_{z' \in z+ 3^j\Lat \cap \cus_k} \! \! \!
 \| \s^{\nicefrac12} \nabla v_n \|_{\underline{L}^2(z' + \cus_j )} \bigl| \s_*^{-\nicefrac12}(z' +\cus_j)  Q \bigr| 
\,.
\end{align*}
By H\"older's inequality, we then obtain 
\begin{align*}
\lefteqn{
\avsum_{z\in 3^{k}\Lat  \cap \cus_n} \E\biggl[  \biggl| 
\fint_{z + \cus_k} \bigl( (\varphi)_{z+\cus_k}  - \varphi \bigr) \nabla v_{n} \cdot  Q 
\biggr|
\biggr]
} \qquad &
\notag \\ & 
\leq
C 3^{-(n-k)}
\E\bigl[ J(\cus_n)\bigr]^{\nicefrac 12} 
\biggl(\sum_{j=-\infty}^k 
3^{\frac32(j-k)} 
\avsum_{z \in 3^j \Lat \cap \cus_n} \! \!  Q \cdot \shom_*^{-1}(z' +\cus_j) Q \biggr)^{\! \nicefrac12} 
\,.
\end{align*}
Combining the above displays yields~\eqref{e.Pvarphi.incontrol}. 

\smallskip

\emph{Step 4.}
Similarly to the previous step, we argue next that there exists~$C(d)<\infty$ such that, for every~$k\in\N$ with~$k_0 \leq k\leq n$,
\begin{multline}
\label{e.Qvarphi.incontrol}
\bigl| 
\E\bigl [\bigl( (\varphi-1) \a \nabla v_n)_{\cus_n} \bigr ] \cdot P \bigr|
\leq
(2 \bar{\tau}_{n,k})^{\nicefrac12} \bigl| \bhom^{\nicefrac12}(\cus_k) P \bigr| 
\\ + 
C  3^{-(n{-}k)} \E\bigl[ J(\cus_n)\bigr]^{\nicefrac12} 
\biggl( \sum_{j=-\infty}^k 
3^{\frac32(j-k)}   \!\!\!\!
 \avsum_{z \in 3^j\Lat \cap \cus_k} \! \!  \! \bigl|\bhom^{\nicefrac12}(z +\cus_j) P \bigr|^2 \biggr)^{\! \nicefrac12} 
\,.
\end{multline}
The proof is almost a verbatim repetition of Step 3 above using identities for fluxes in place of gradients, in particular, an appropriate version of~\eqref{e.energymaps.nonsymm.flux} similar to~\eqref{e.energymaps.nonsymm.mean}.  We omit the details. 

\smallskip

\emph{Step 5.} 
We next show that\footnote{The statement asserted in Step~5 is valid without the assumptions of Proposition~\ref{p.renormalize.reduce} and without the centering assumption~\eqref{e.centering}, and for general~$p,q, P,Q\in\Rd$. Indeed, neither the assumptions of the proposition, the centering assumption, nor the choices of these parameters are used in the argument of Lemma~\ref{l.Jtilde.energy.bound} before Step~6. This more general statement will be used later, both in Section~\ref{ss.algebraic} and Section~\ref{ss.pigeon.prime}.} there exists a constant~$C(d)<\infty$ such that,
for every~$k,n\in\N$ with~$k_0 \leq k \leq n - 4$ and~$\ep \in (0,1]$,
we have
\begin{align}  
\label{e.divcurl.conclusion.pre}
\Bigl| \E\bigl  [ J(\cus_n)\bigr] -\frac12 P \cdot Q \Bigr|
& 
\leq 
50\ep^{-1} \bar{\tau}_{n,k} 
+
4 \ep
\bigl(  \bigl| \shom_{*}^{-\nicefrac12}(\cus_k) Q \bigr|  +  \bigl| \bhom^{\nicefrac12}(\cus_k) P \bigr| \bigr)^2
\notag \\ &  \quad 
+
C 3^{-(n-k)}  
\sum_{j=-\infty}^k 
3^{\frac32(j-k)}   \!\!\!\!
 \avsum_{z \in 3^j\Lat \cap \cus_k} \! \!  \! \bigl( \bigl|\shom_{*}^{-\nicefrac12}(z +\cus_j) Q \bigr| + \bigl|\bhom^{\nicefrac12}(z +\cus_j) P \bigr|  \bigr)^2
\notag \\ & \quad 
+
\frac12 \Bigl| Q \cdot \bigl( \E\bigl[(\nabla v_n)_{\cus_n}\bigr] - P )  
+ P \cdot \bigl( \E\bigl[(\a \nabla v_n)_{\cus_n}\bigr] - Q ) \Bigr|
\notag \\ & \quad 
+
C 3^{-n} 
\E \Biggl[
\biggl[ 
\mathbf{M}_0^{\nicefrac12} 
\begin{pmatrix} 
\nabla v_n - P \\ 
\a \nabla v_n - Q 
\end{pmatrix} 
\biggr]_{\Besov{-\nf12}{2}{1}(\cus_{n})}^2
\Biggr] 
\,.
\end{align}
Above~$k_0(d)$ is as in~\eqref{e.adapted.stationarity} guaranteeing that~$\bfAhom(z' +\cus_j) = \bfAhom(\cus_j)$ for every $j\in \N$ with~$j \geq k_0$. Fix~$\ep \in (0,1]$.
To prove~\eqref{e.divcurl.conclusion.pre}, we  combine~\eqref{e.Jtilde.decomp},~\eqref{e.divcurl.applied},~\eqref{e.cutoffenergy.vs.J},~\eqref{e.Pvarphi.incontrol} and~\eqref{e.Qvarphi.incontrol} and obtain
\begin{align}  
\lefteqn{
\Bigl| \E\bigl  [ J(\cus_n)\bigr] -\frac12 P \cdot Q \Bigr|
} \quad &
\notag \\ & 
\leq 
\bar{\tau}_{n,k}^{\nicefrac 12} \Bigl( 2 \bar{\tau}_{n,k}^{\nicefrac 12} + 4 \E[J(\cus_n  )]^{\nicefrac12} +  2^{\nicefrac12} \bigl(  \bigl| \shom_{*}^{-\nicefrac12}(\cus_k) Q \bigr|  +  \bigl| \bhom^{\nicefrac12}(\cus_k) P \bigr| \bigr) \Bigr)
+ 
3^{k + 2 -n}  \E\bigl[ J(\cus_n ) \bigr] 
\notag \\ & 
\quad 
+
 C 3^{-(n-k)} \E\bigl[ J(\cus_n)\bigr]^{\nicefrac12}
\biggl( \sum_{j=-\infty}^k 
3^{\frac32(j-k)}   \!\!\!\!
 \avsum_{z \in 3^j\Lat \cap \cus_k} \! \!  \! \bigl( \bigl|\shom_{*}^{-\nicefrac12}(z +\cus_j) Q \bigr| + \bigl|\bhom^{\nicefrac12}(z +\cus_j) P \bigr|  \bigr)^2\biggr)^{\! \nicefrac12} 
 \!\!\!
\notag \\ &  \quad 
+
\frac12 \Bigl| Q \cdot \bigl( \E\bigl[(\nabla v_n)_{\cus_n}\bigr] - P )  
+ P \cdot \bigl( \E\bigl[(\a \nabla v_n)_{\cus_n}\bigr] - Q ) \Bigr|
\notag \\ & 
\quad 
+
C 3^{-n}\E \Bigl[
\bigl[ \m_0^{\nicefrac12} (\nabla v_n - P)  \bigr]_{\Besov{-\nf12}{2}{1}(\cus_{n})} \bigl [  \mathbf{m}_{0}^{-\nicefrac12} (\a \nabla v_{n} - Q) \bigr]_{\underline{B}_{2,1}^{-\nicefrac12}(\cus_n)} 
\Bigr]
\,.
\notag
\end{align}
By the triangle inequality and Young's inequality, using also~$\bhom(\cus_k) \geq \shom_{*}(\cus_k)$, we get
\begin{equation*}  
\E\bigl[ J(\cus_n ) \bigr]
\leq 
\Bigl| \E[J(\cus_n )]  -\frac12 P\cdot Q \Bigr|
+ 
\frac14 \bigl(  \bigl| \shom_{*}^{-\nicefrac12}(\cus_k) Q \bigr|  +  \bigl| \bhom^{\nicefrac12}(\cus_k) P \bigr| \bigr)^2 \,.
\end{equation*}
Thus, since~$k \leq n-4$, 
\begin{equation*} 
3^{k + 2 -n}  \E\bigl[ J(\cus_n ) \bigr] 
\leq \frac19 \Bigl| \E[J(\cus_n )]  -\frac12 P\cdot Q \Bigr| + 3^{k + 1 -n}  \bigl(  \bigl| \shom_{*}^{-\nicefrac12}(\cus_k) Q \bigr|  +  \bigl| \bhom^{\nicefrac12}(\cus_k) P \bigr| \bigr)^2
\,.
\end{equation*}
It also follows, again by Young's inequality, that
\begin{align*}  
4 \bar{\tau}_{n,k}^{\nicefrac 12}\E[J(\cus_n  )]^{\nicefrac12}  
& 
\leq
\frac \ep4 \Bigl| \E[J(\cus_n )]  -\frac12 P\cdot Q \Bigr|
+
16\ep^{-1} \bar{\tau}_{n,k}
+
\frac \ep{16} \bigl(  \bigl| \shom_{*}^{-\nicefrac12}(\cus_k) Q \bigr|  +  \bigl| \bhom^{\nicefrac12}(\cus_k) P \bigr| \bigr)^2\,.
\end{align*}
Similarly, 
\begin{equation*}
(2 \bar{\tau}_{n,k})^{\nicefrac 12} \bigl(  \bigl| \shom_{*}^{-\nicefrac12}(\cus_k) Q \bigr|  +  \bigl| \bhom^{\nicefrac12}(\cus_k) P \bigr| \bigr) 
\leq \ep^{-1} \bar{\tau}_{n,k} +   
\frac\ep2 \bigl(  \bigl| \shom_{*}^{-\nicefrac12}(\cus_k) Q \bigr|  +  \bigl| \bhom^{\nicefrac12}(\cus_k) P \bigr| \bigr)^2
\,.
\end{equation*}
Furthermore, by H\"older's and Young's inequalities,
\begin{align*} 
\lefteqn{
 C 3^{-(n-k)} \E\bigl[ J(\cus_n)\bigr]^{\nicefrac12}
\biggl( \sum_{j=-\infty}^k 
3^{\frac32(j-k)}   \!\!\!\!
 \avsum_{z \in 3^j\Lat \cap \cus_k} \! \!  \! \bigl( \bigl|\shom_{*}^{-\nicefrac12}(z +\cus_j) Q \bigr| + \bigl|\bhom^{\nicefrac12}(z +\cus_j) P \bigr|  \bigr)^2\biggr)^{\! \nicefrac12} 
 \!\!\!
} \qquad &
\notag \\ &
\leq
\frac1{4} \Bigl| \E[J(\cus_n )]  -\frac12 P\cdot Q \Bigr|  
\notag \\ & \qquad
+ C 3^{-(n-k)}  
\sum_{j=-\infty}^k 
3^{\frac32(j-k)}   \!\!\!\!
 \avsum_{z \in 3^j\Lat \cap \cus_k} \! \!  \! \bigl( \bigl|\shom_{*}^{-\nicefrac12}(z +\cus_j) Q \bigr| + \bigl|\bhom^{\nicefrac12}(z +\cus_j) P \bigr|  \bigr)^2
\,.
\end{align*}
Consequently, we obtain~\eqref{e.divcurl.conclusion.pre} by combining the estimates in the previous six displays and reabsorbing the term~$(\nicefrac19 + \nicefrac \ep4 + \nicefrac14) | \E[J(\cus_n )]  -\frac12 P\cdot Q |$  from the right. 

\smallskip

\emph{Step 6.} We complete the proof of~\eqref{e.Jtilde.energy.bound}. We fix~$n  \coloneqq  m - l$,~$k  \coloneqq  n-l$, and estimate the quantities appearing in~\eqref{e.divcurl.conclusion.pre}. 
First, we show that 
\begin{equation} 
\label{e.pq.bounds}
\biggl|
\bfAhom_{m}^{\nicefrac 12} 
\begin{pmatrix} 
-p \\ q
\end{pmatrix}
\biggr|^2
\leq
2\bigl ( |\bhom_m^{\nicefrac12} p|^2 + |\shom_{*,m}^{-\nicefrac12} q|^2 \bigr)
\leq 
4
\bigl|  \shom_{*,m}^{-\nicefrac12}\bhom_{m} \shom_{*,m}^{-\nicefrac12} \bigr|^{\nicefrac12} 
\leq 
C \Theta^{\nf12}
\,.
\end{equation}
The first inequality of~\eqref{e.pq.bounds} is a consequence of~\eqref{e.bfA.bounds.diag}. Next, since~$\tt_m  \coloneqq  \bhom_{m}\# \shom_{*,m}^{-1}$, by the properties of the metric geometric mean, namely the equation~\eqref{e.ricatti}, we have 
\begin{equation*}
\tt_m^{-\nicefrac12} 
\bhom_{m} \tt_m^{-\nicefrac12} 
=
\tt_m^{\nicefrac12} \shom_{*,m}^{-1} \tt_m^{\nicefrac12} 
\,.
\end{equation*}
Similarly to~\eqref{e.rat.bounds.before}, we get
\begin{equation*}
\bigl| \tt_m^{-\nicefrac 12}  \bhom_{m} \tt_m^{-\nicefrac 12} \bigr|
=
\bigl| 
\tt_m^{\nicefrac12} \shom_{*,m}^{-1} \tt_m^{\nicefrac12} 
\bigr|
=
\bigl|  \shom_{*,m}^{-\nicefrac12}\bhom_{m}  \shom_{*,m}^{-\nicefrac12}\bigr|^{\nicefrac12} 
\,.
\end{equation*}
From the above display, we obtain
\begin{equation}
\label{e.dingus.1}
\bigl| \bhom_{m}^{\nicefrac12} p\bigr|
=
\bigl| \bhom_{m}^{\nicefrac12} \tt_m^{-\nicefrac12} e\bigr|
\leq
\bigl| \tt_m^{-\nicefrac12} \bhom_{m} \tt_m^{-\nicefrac12} \bigr|^{\nicefrac12} 
=
\bigl|  \shom_{*,m}^{-\nicefrac12}\bhom_{m}  \shom_{*,m}^{-\nicefrac12}\bigr|^{\nicefrac14}
\end{equation}
and, similarly, 
\begin{equation}
\label{e.dingus.2}
\bigl| \shom_{m,*}^{-\nicefrac12} q \bigr| 
=
\bigl| \shom_{m,*}^{-\nicefrac12} \tt_m^{\nicefrac12} e\bigr|
\leq 
\bigl| 
\tt_m^{\nicefrac12} \shom_{*,m}^{-1} \tt_m^{\nicefrac12} 
\bigr|^{\nicefrac12} 
=
\bigl|  \shom_{*,m}^{-\nicefrac12}\bhom_{m}  \shom_{*,m}^{-\nicefrac12}\bigr|^{\nicefrac14}
\,.
\end{equation}
The middle inequality of~\eqref{e.pq.bounds} is a consequence of~\eqref{e.dingus.1} and~\eqref{e.dingus.2}. 
The last inequality of~\eqref{e.pq.bounds} follows by~\eqref{e.symm.k.quad.small} and the symmetry of~$\khom_m$.

\smallskip

Next we use~\eqref{e.first.pigeon.adapted} and~\eqref{e.bfA.bounds.diag} to control the additivity defect: we have
\begin{align} 
\bar{\tau}_{n,k} 
= 
\E\big[J(\cus_k) - J(\cus_n)\bigr]
& 
=
\frac12 \begin{pmatrix} - p \\ q  \end{pmatrix}
\cdot 
\bigl(  \bfAhom(\cus_k) - \bfAhom(\cus_n) \bigr)
\begin{pmatrix} - p \\ q   \end{pmatrix}
\notag \\ & 
\leq
\max_{j\in\{k,n\}}\Bigl| \bfAhom_{m}^{-\nicefrac 12} \bfAhom(\cus_j)  \bfAhom_{m}^{-\nicefrac 12}\! - \Itwod
\Bigr| 
\biggl |
\bfAhom_{m}^{\nicefrac 12} 
\begin{pmatrix} 
-p \\ q
\end{pmatrix}
\biggr |^2 
\leq 
C \delta \sigma^2 
\Theta^{\nf12}
\,.
\label{e.what.is.bartau}
\end{align}
By~\eqref{e.l.def}, we see that, since~$k=m-2l > 90 l$,
\begin{equation} 
\label{e.l.cond.two}
\frac{K_{\Psi_\S}^{90} \Pi^{90} }{(1-\gamma)^{90}} 3^{-k}  \leq 1 
 \,.
\end{equation}
Using~\eqref{e.l.cond.two} and~\eqref{e.tilt.by.Euc} we obtain 
\begin{align*} 
\lefteqn{
\sum_{j=-\infty}^k 
3^{\frac32(j-k)}    \! \! \! \avsum_{z \in 3^j\Lat \cap \cus_n} \! \! \!
\bigl( \bigl|\shom_{*}^{-\nicefrac12}(z +\cus_j) Q \bigr| + \bigl|\bhom^{\nicefrac12}(z +\cus_j) P \bigr|  \bigr)^2
} \qquad &
\notag \\ &
\leq 
\bigl( \bigl|\shom_{*,0}^{-\nicefrac12} Q \bigr| +   \bigl|\bhom_0^{\nicefrac12} P \bigr|  \bigr)^2 \sum_{j=-\infty}^k 
3^{\frac32(j-k)} 
\! \! \! \! \! \avsum_{z \in 3^j\Lat \cap \cus_n} \! \! \!  
\bigl( \bigl|  \shom_{*,0}^{\nicefrac12}   \shom_{*}^{-1}(z+\cus_j) \shom_{*,0}^{\nicefrac12}  \bigr| + 
\bigl|  \bhom_{0}^{-\nicefrac12}   \bhom(z+\cus_j) \bhom_{0}^{-\nicefrac12}    \bigr|   \bigr)
\,.
\end{align*}
On the one hand, for~$j \geq k_0$, we deduce by a similar computation as in~\eqref{e.szero.vs.sm} and stationarity that
\begin{equation*} 
\sum_{j=k_0}^k 
3^{\frac32(j-k)} 
\! \! \! \! \! \avsum_{z \in 3^j\Lat \cap \cus_n} \! \! \!  
\bigl( \bigl|  \shom_{*,0}^{\nicefrac12}   \shom_{*}^{-1}(z+\cus_j) \shom_{*,0}^{\nicefrac12}  \bigr| + 
\bigl|  \bhom_{0}^{-\nicefrac12}   \bhom(z+\cus_j) \bhom_{0}^{-\nicefrac12}    \bigr|   \bigr)
\leq
C 
\,.
\end{equation*} 
On the other hand, by Lemma~\ref{l.crude.moments} and~\eqref{e.l.cond.two} we get, for every~$j \in \Z$ with~$j\leq n$ and~$z \in 3^j\Lat \cap \cus_n$,
\begin{equation*} 
\bfAhom(z+\cus_j) \leq \biggl( 1+ \frac{C\Pi 3^{-n}}{1-\gamma}  \biggr)^{\gamma} 3^{\gamma(n-j)} \bfE \leq C 3^{\gamma(n-j)} \bfE
 \,.
\end{equation*}
Thus, as in~\eqref{e.szero.vs.sm}, we obtain 
\begin{equation*} 
\sum_{j=-\infty}^{k_0} 
3^{\frac32(j-k)} 
\! \! \! \! \! \avsum_{z \in 3^j\Lat \cap \cus_n} \! \! \!  
\bigl( \bigl|  \shom_{*,0}^{\nicefrac12}   \shom_{*}^{-1}(z+\cus_j) \shom_{*,0}^{\nicefrac12}  \bigr| + 
\bigl|  \bhom_{0}^{-\nicefrac12}   \bhom(z+\cus_j) \bhom_{0}^{-\nicefrac12}    \bigr|   \bigr)
\leq
C 3^{-\frac32 k + \gamma n}
 \,.
\end{equation*}
Since~$-\frac32 k + \gamma n = -(\frac32 - \gamma)n + \gamma l \leq -\frac12 n + 2 l \leq - 10 l$, we then deduce that
\begin{equation} 
\label{e.Ahomj.vs.Enaught.sum}
\sum_{j=-\infty}^k 
3^{\frac32(j-k)}    \! \! \! \avsum_{z \in 3^j\Lat \cap \cus_n} \! \! \!
\bigl( \bigl|\shom_{*}^{-\nicefrac12}(z +\cus_j) Q \bigr| + \bigl|\bhom^{\nicefrac12}(z +\cus_j) P \bigr|  \bigr)^2
\leq 
C \bigl( \bigl|\shom_{*,0}^{-\nicefrac12} Q \bigr| +   \bigl|\bhom_0^{\nicefrac12} P \bigr|  \bigr)^2  
\,.
\end{equation}
We next estimate~$\bigl| \b_{0}^{\nicefrac12} P \bigr| +  \bigl| \s_{*,0}^{-\nicefrac12} Q \bigr| $. We recall the definition of~$\mathbf{R}$ in~\eqref{e.R.def} and also define
\begin{equation*}  
\widetilde{\mathbf{E}}_0  \coloneqq   \begin{pmatrix} 
\b_0 & 0 \\ 0 & \s_{*,0}^{-1} 
\end{pmatrix}
\,.
\end{equation*}
In view of~\eqref{e.pqh.choice},~\eqref{e.all.averages},~\eqref{e.PQ.choice} and~\eqref{e.pq.bounds}, we have by the triangle inequality that
\begin{align*}  
\biggl| 
\widetilde{\mathbf{E}}_0^{\nicefrac12}  \begin{pmatrix} 
P \\ Q
\end{pmatrix}
\biggr|^2
& =
\begin{pmatrix} 
-p \\ q
\end{pmatrix}
\cdot \bigl( 
(\bfAhom_{m}  \mathbf{R} -\Itwod )\widetilde{\mathbf{E}}_0  
( \mathbf{R} \bfAhom_{m}  -\Itwod )
\bigr)
\begin{pmatrix} 
-p \\ q
\end{pmatrix}
\notag \\ &
\leq
2 \bigl( 
\bigl| \bfAhom_{m}^{\nicefrac 12} \mathbf{R}  \widetilde{\mathbf{E}}_0 \mathbf{R}  \bfAhom_{m}^{\nicefrac 12} \bigr|
+
\bigl| \bfAhom_{m}^{-\nicefrac 12} \widetilde{\mathbf{E}}_0 \bfAhom_{m}^{-\nicefrac 12} \bigr|  \bigr) 
\biggl| \bfAhom_{m}^{\nicefrac 12} 
\begin{pmatrix} 
-p \\ q
\end{pmatrix}
\biggr|^2
\notag \\ &
\leq 
C\Theta^{\nf12}
\bigl( 
\bigl| \bfAhom_{m}^{\nicefrac 12} \mathbf{R}  \widetilde{\mathbf{E}}_0 \mathbf{R}  \bfAhom_{m}^{\nicefrac 12} \bigr|
+
\bigl| \bfAhom_{m}^{-\nicefrac 12} \widetilde{\mathbf{E}}_0 \bfAhom_{m}^{-\nicefrac 12} \bigr|  \bigr)   
\,.
\end{align*}
Since~$\bfAhom_m  \leq \bfE$, we obtain by~\eqref{e.bfA.bounds.diag},~\eqref{e.mnaught.vs.bnaught} and~\eqref{e.szero.vs.sm} that
\begin{align*}  
\bigl| \bfAhom_{m}^{\nicefrac 12} \mathbf{R}  \widetilde{\mathbf{E}}_0 \mathbf{R}  \bfAhom_{m}^{\nicefrac 12} \bigr|
=
\Biggl| \begin{pmatrix} 
\s_{*,0}^{-\nicefrac12}  & 0 \\ 0 & \b_0^{\nicefrac12} 
\end{pmatrix} \bfAhom_{m} \begin{pmatrix} 
\s_{*,0}^{-\nicefrac12}  & 0 \\ 0 & \b_0^{\nicefrac12} 
\end{pmatrix}  \Biggr|
&
\leq 
2 \bigl( \bigl| \s_{*,0}^{-\nicefrac12} \bhom_m \s_{*,0}^{-\nicefrac12} \bigr|   + \bigl| \shom_{*,m}^{-\nicefrac12} \b_0 \shom_{*,m}^{-\nicefrac12} \bigr|   \bigr)
\notag \\ &
\leq 
C \Theta \bigl( \bigl|  \b_{0}^{-\nicefrac12} \bhom_m \b_{0}^{-\nicefrac12}  \bigr| + 
\bigl|  \shom_{*,m}^{-\nicefrac12}  \s_{*,0}  \shom_{*,m}^{-\nicefrac12} \bigr| \bigr)
\notag \\ & 
\leq
C\Theta
\end{align*}
and
\begin{align*}  
\bigl| \bfAhom_{m}^{-\nicefrac 12} \widetilde{\mathbf{E}}_0 \bfAhom_{m}^{-\nicefrac 12} \bigr|
& 
=
\bigl| \widetilde{\mathbf{E}}_0^{\nicefrac 12}  \bfAhom_{m}^{-1}
\widetilde{\mathbf{E}}_0^{\nicefrac 12} \bigr|
\notag \\ &
\leq
2 \bigl( \bigl|  \b_{0}^{\nicefrac12}  \shom_m^{-1} \b_{0}^{\nicefrac12}  \bigr| + \bigl|  \s_{*,0}^{-\nicefrac12}  (\shom_{*,m} + \khom_m \shom_m^{-1} \khom_m^t) \s_{*,0}^{-\nicefrac12}  \bigr| \bigr)
\notag \\ &
\leq
2 \bigl( \bigl|  \b_{0}^{\nicefrac12}  \shom_{*,0}^{-1} \b_{0}^{\nicefrac12}  \bigr| + \bigl|  \s_{*,0}^{-\nicefrac12}  \b_m \s_{*,0}^{-\nicefrac12}  \bigr| \bigr)
\leq 
C \bigl( \Theta + \bigl|  \s_{*,0}^{-\nicefrac12}  \b_0 \s_{*,0}^{-\nicefrac12}  \bigr| \bigr)
\leq
C \Theta
\,.
\end{align*}
It follows that
\begin{equation}
\label{e.PQ.bound.pre}
| \b_0^{\nicefrac12}P |^2 + | \s_{*,0}^{-\nicefrac12}Q |^2  
=  \biggl| \widetilde{\mathbf{E}}_0^{\nicefrac12}  \begin{pmatrix} 
P \\ Q
\end{pmatrix} \biggr|^2
\leq C \Theta \biggl| \bfAhom_{m}^{\nicefrac 12} 
\begin{pmatrix} 
-p \\ q
\end{pmatrix} \biggr|^2
\leq C\Theta^{\nf32}
\,.
\end{equation}
Next, using~\eqref{e.PQ.choice} and~\eqref{e.all.averages}, we have, by~\eqref{e.renormalize.A.vs.E0},~\eqref{e.first.pigeon.adapted},~\eqref{e.pq.bounds} and~\eqref{e.PQ.bound.pre},
\begin{align}
\label{e.amazinglyeasy.apply.really}
\Biggl| 
\begin{pmatrix} 
Q \\ P
\end{pmatrix}  
\cdot
\E \bigg[ \fint_{\cus_n}
\begin{pmatrix} 
\nabla v_{n}  \\  \a \nabla v_{n} 
\end{pmatrix}  -
\begin{pmatrix} 
P \\ Q 
\end{pmatrix} 
\biggr]
\Biggr| 
&
=
\Biggl| 
\begin{pmatrix} 
P \\ Q
\end{pmatrix}  
\cdot \bigl( \bfAhom(\cus_n) - \bfAhom_m \bigr)
\cdot
\begin{pmatrix} 
-p \\ q 
\end{pmatrix} 
\biggr]
\Biggr|  
\notag \\ &
\leq
\bigl| \bfE^{-\nicefrac 12} \bigl( \bfAhom(\cus_n) - \bfAhom_m \bigr) \bfE^{-\nicefrac 12} \bigr|
\biggl |
\bfE^{\nicefrac 12} 
\begin{pmatrix} 
-p \\ q
\end{pmatrix}
\biggr |
\biggl |
\bfE^{\nicefrac 12} 
\begin{pmatrix} 
P \\ Q
\end{pmatrix}
\biggr |
\notag \\ &
\leq
C \delta \sigma^2 
\biggl |
\bfAhom_{m}^{\nicefrac 12} 
\begin{pmatrix} 
-p \\ q
\end{pmatrix}
\biggr |
\bigl( \bigl| \b_{0}^{\nicefrac12} P \bigr| +  \bigl| \s_{*,0}^{-\nicefrac12} Q \bigr|   \bigr)
\notag \\ &
\leq
C \delta \sigma^2 \Theta
\,.
\end{align}
Inserting now~\eqref{e.Ahomj.vs.Enaught.sum},~\eqref{e.amazinglyeasy.apply.really} and~\eqref{e.what.is.bartau} into~\eqref{e.divcurl.conclusion.pre} yields
\begin{align*} 
\biggl| \E\bigl  [ J(\cus_n)\bigr] -\frac12 P \cdot Q \biggr| 
& 
\leq
C3^{-n} 
\E \Biggl[
\biggl[ 
\mathbf{M}_0^{\nicefrac12} 
\begin{pmatrix} 
\nabla v_n - P \\ 
\a \nabla v_n - Q 
\end{pmatrix} 
\biggr]_{\Besov{-\nf12}{2}{1}(\cus_{n})}^2
\Biggr]
\notag \\ &  \qquad 
+
C\Theta^{\nf12} \bigl( 
\delta \sigma^2 (\ep^{-1} + \Theta^{\nicefrac12}) 
+
(\ep+3^{-l})  \Theta
\bigr)
\,.
\end{align*}
By~\eqref{e.l.def}, $3^{-l}  \Theta \leq \delta \sigma^2$. Thus, by taking~$\ep  \coloneqq  \delta^{\nicefrac12} \sigma \Theta^{-\nicefrac12}$, we obtain
\begin{equation*} 
\biggl| \E\bigl  [ J(\cus_n)\bigr] -\frac12 P \cdot Q \biggr| 
\leq
C3^{-n} 
\E \Biggl[
\biggl[ 
\mathbf{M}_0^{\nicefrac12} 
\begin{pmatrix} 
\nabla v_n - P \\ 
\a \nabla v_n - Q 
\end{pmatrix} 
\biggr]_{\Besov{-\nf12}{2}{1}(\cus_{n})}^2
\Biggr]
+
C \delta^{\nicefrac12} \sigma\Theta
\,.
\end{equation*}
The proof is complete.  
\end{proof}

The previous lemma reduces our task to estimating the right side of~\eqref{e.Jtilde.energy.bound}.
This is the content of the following lemma, which is based on an application of Lemma~\ref{l.weaknorms.moreproto}.  

\begin{lemma} 
\label{l.weaknorms}
There exists a constant~$C(d)<\infty$ such that, for~$n = m-l$, we have that
\begin{equation*}  
3^{-n} 
\E \Biggl[
\biggl[ 
\mathbf{M}_0^{\nicefrac12} 
\begin{pmatrix} 
\nabla v(\cdot,\cus_n,p,q) - P \\ 
\a \nabla v(\cdot,\cus_n,p,q) - Q 
\end{pmatrix} 
\biggr]_{\Besov{-\nf12}{2}{1}(\cus_{n})}^2
\Biggr]
\leq
C  \delta \sigma^2 
\Theta
\,.
\end{equation*}
\end{lemma}

\begin{proof}
Denote~$v_n  \coloneqq  v(\cdot,\cus_n,p,q)$ for short. 
To prepare for the application of Lemma~\ref{l.weaknorms.moreproto}, we observe that, since~$n = m-l > 90 l$, we have by~\eqref{e.l.def} that
\begin{equation} 
\label{e.l.cond.three}
\frac{ \Pi^4 K_{\Psi_\S}^{4d + 15} }{(1-\gamma)^5} 3^{-n} \leq \delta\sigma^2   
\qand
\frac{3^{-(1 - \gamma) l }}{1-\gamma}  \leq \delta \sigma^2
 \,.
\end{equation}
Taking also~$\mathbf{M}=\mathbf{M}_0$ and~$\mathbf{E}= \bfE$, we observe that by~\eqref{e.mnaught.vs.bnaught} we have
\begin{equation}
\label{e.MEbound}
\bigl| \mathbf{M}^{-\nicefrac12} \mathbf{E} \mathbf{M}^{-\nicefrac12} \bigr|
=
\bigl| \mathbf{M}_0^{-\nicefrac12} \bfE \mathbf{M}_0^{-\nicefrac12} \bigr|
\leq
2 \bigl( \bigl| \m_0^{-\nicefrac12} \b_0 \m_0^{-\nicefrac12}   \bigr|
+ \bigl| \m_0^{\nicefrac12} \s_{*,0}^{-1} \m_0^{\nicefrac12}   \bigr| \bigr)
\leq 
C\Theta^{\nicefrac12}  \,.
\end{equation}
Using~\eqref{e.Enaught.vs.Am}, we see that $\bfE \leq 2\bfAhom_m$ and, by~\eqref{e.pq.bounds} and~\eqref{e.pqh.choice}, we have the estimate
\begin{equation} 
\label{e.Enaught.vs.p.q.verified.first}
\max\Bigl\{ \Bigl | 
\bfE^{\nicefrac 12}  
\Bigl(
\begin{matrix} 
-p \\ q
\end{matrix} 
\Bigr)
\Bigr|\,,  |p\cdot q| \Bigr\} 
\leq 
\max\Bigl\{ 2 \Bigl | 
 \bfAhom_m^{\nicefrac 12}  
\Bigl(
\begin{matrix} 
-p \\ q
\end{matrix} 
\Bigr)
\Bigr|\,,  1  \Bigr\} 
\leq C \Theta^{\nf14}
\,.
\end{equation}
Moreover, by the definitions of~$\mathcal{M}_{n,\gamma}$ in~\eqref{e.event.moreproto},~$p$ and~$q$, we get 
\begin{align} 
\label{e.grad.vn.bound.in.proof}
\| \s^{\nicefrac12} \nabla v_n  \|_{\underline{L}^2(\cus_n)}^2
= 
\begin{pmatrix} 
-p \\ q
\end{pmatrix}  
\cdot \bfA(\cus_n)  \begin{pmatrix} 
-p \\ q
\end{pmatrix}  
- 2 p\cdot q
&
\leq 
(1+\mathcal{M}_{n,\gamma}) 
\Bigl | 
\bfE^{\nicefrac 12}  
\Bigl(
\begin{matrix} 
-p \\ q
\end{matrix} 
\Bigr)
\Bigr|^2
\notag \\ & 
\leq
C (1+\mathcal{M}_{n,\gamma})  
\Theta^{\nf12} 
\,.
\end{align}
We now apply Lemma~\ref{l.weaknorms.moreproto}: we square the inequality~\eqref{e.weaknorms.moreproto} with~$s=\nf12$,~$h=l$ and~$\rho =\gamma$,  take the expectation of the result, and substitute for some factors on the right side with~\eqref{e.MEbound},~\eqref{e.Enaught.vs.p.q.verified.first} and~\eqref{e.grad.vn.bound.in.proof}, and get rid of some squares on the right side by using H\"older's inequality in the form
\begin{align*} 
\biggl( \sum_{k=n-l}^n 3^{\frac12 (k-n)} X_k^{\nicefrac12}   \biggr)^2 
&
=
\biggl( \sum_{k=n-l}^n 3^{\frac12 \ep(k-n)}3^{\frac12(1-\ep) (k-n)} X_k^{\nicefrac12}   \biggr)^2 
\notag \\ &
\leq
\biggl( \sum_{k=n-l}^n 3^{\ep(k-n)} \biggr) \biggl( \sum_{k=n-l}^n 3^{(1-\ep) (k-n)} X_k   \biggr) \leq \frac{C}{\ep} \sum_{k=n-l}^n 3^{(1-\ep) (k-n)} X_k 
\,.
\end{align*}
We obtain, for every~$\ep \in (0,1)$, 
\begin{align} 
\label{e.weaknorms.proto.applied}
\lefteqn{
3^{-n}\E\Biggl[ \biggl[
\mathbf{M}_0^{\nicefrac12} 
\begin{pmatrix} 
\nabla v_n -  (\nabla v_n)_{\cus_n}   \\ 
\a \nabla v_n- (\a\nabla v_n)_{\cus_n}
\end{pmatrix} 
\biggr]_{\Besov{-\nf12}{2}{1}(\cus_{n})}^2
\Biggr]
} \quad  &
\notag \\ &
\leq 
\frac{C \Theta^{\nicefrac12}}{\ep} 
\Bigl | 
\bfE^{\nicefrac 12}  
\Bigl(
\begin{matrix} 
-p \\ q
\end{matrix} 
\Bigr)
\Bigr|^2
\sum_{k=n-l}^n  3^{(1-\ep)(k-n)}  
\E\Bigl[  \bigl|  \bfE^{-\nicefrac12} ( \bfA(\cus_k) - \bfA(\cus_n) )  \bfE^{-\nicefrac12}\bigr|^2 
\Bigl]
\notag \\ &
\qquad 
+
\frac{C \Theta^{\nicefrac12}}{\ep} \Bigl | 
\bfE^{\nicefrac 12}  
\Bigl(
\begin{matrix} 
-p \\ q
\end{matrix} 
\Bigr)
\Bigr|^2 
\sum_{k=n-l}^{n} 3^{(1-\ep)(k-n)}
\E\Biggl[ \biggl| 
\avsum_{z\in 3^{k}\Lat  \cap \cus_n}  \bfE^{-\nicefrac12}\bigl(\bfA(z{+}\cus_k) - \bfA(\cus_n)  \bigr)   \bfE^{-\nicefrac12}
\biggr| \Biggl]
\notag \\ &
\qquad 
+
\frac{C\Theta^{\nicefrac12}}{1-\gamma} \Bigl | 
\bfE^{\nicefrac 12}  
\Bigl(
\begin{matrix} 
-p \\ q
\end{matrix} 
\Bigr)
\Bigr|^2
\Bigl( 3^{-(1-\gamma)l} 
+
\E\bigl[\mathcal{M}_{n,\gamma}^2 \indc_{\{\mathcal{M}_{n,\gamma}  > 1\}}  
 \bigr]   \Bigr)
\,.
\end{align}
We will take~$\ep = \nicefrac12$ in the above display. 
Using~\eqref{e.Enaught.vs.Am},~\eqref{e.first.pigeon.adapted},~\eqref{e.pigeon.var} and the triangle inequality, the first term on the right is bounded by~$C \delta \sigma^2 \Theta$. Similarly,  by~\eqref{e.Enaught.vs.Am} and~\eqref{e.first.pigeon.adapted}, 
\begin{align} 
\label{e.e.weaknorms.proto.applied.second}
\lefteqn{
\sum_{k=n-l}^{n} 3^{(1-\ep)(k-n)}
\E\Biggl[ \biggl|  \avsum_{z\in 3^{k}\Lat  \cap \cus_n}   \bfE^{-\nicefrac12}\bigl(\bfA(z{+}\cus_k) - \bfA(\cus_n)  \bigr)  \bfE^{-\nicefrac12}  \biggr|\Biggl]
} \qquad &
\notag \\ &
\leq 
\sum_{k=n-l}^{n} 3^{(1-\ep)(k-n)}
 \tr\bigl(  \bfE^{-\nicefrac12} \bigl(\bfAhom(\cus_k) - \bfAhom(\cus_n)  \bigr)  \bfE^{-\nicefrac12}  \bigr)
\leq C \delta \sigma^2
\,.
\end{align}
The second inequality is valid since, by subadditivity, $\avsum_{z\in 3^{k}\Lat  \cap \cus_n} \bfA(z{+}\cus_k) \geq \bfA(\cus_n)$. Furthermore, letting~$\S_{h'}$ be as in Lemma~\ref{l.crude.moments} (with~$h'$ being the smallest integer such that~\eqref{e.h.delta.relation} is valid with~$\delta=1$ and~$h=0$), and recalling the definition in~\eqref{e.event.moreproto}, we have the implication 
\begin{equation*} 
3^{n}  \geq \S_{h'} 
\implies
\mathcal{M}_{n,\gamma} \leq \sup_{k \in \Z \cap (-\infty,n]} 
3^{-\gamma (n-k)}   \max_{z\in 3^{k}\Lat  \cap \cus_n}  \bigl| \bfE^{-\nicefrac12} \bfA(y+\cus_k) \bfE^{-\nicefrac12} \bigr|
\leq 1\,.
\end{equation*}
Moreover, by~\eqref{e.crude.moments.adapted}, we get
\begin{equation*} 
\mathcal{M}_{n,\gamma} 
\leq 3^{\gamma} 3^{-\gamma n} \S_{h'}^{\gamma} 
\,.
\end{equation*}
Therefore, by the above two displays,~\eqref{e.h.delta.relation},~\eqref{e.Sh.O},~\eqref{e.Psi.moments.bound} and~\eqref{e.l.cond.three},
\begin{align} 
\label{e.this.is.so.nice.pre}
\frac{1}{1-\gamma}
\E[\mathcal{M}_{n,\gamma}^2 \indc_{\{\mathcal{M}_{n,\gamma}  > 1\}} ] 
&
\leq
\frac{3^{2\gamma} 3^{-2\gamma n} }{1-\gamma}
\E\Biggl[
\S_{h'}^{2\gamma} \biggl( \frac{\S_{h'}}{3^n} \biggr)^{2(1-\gamma)} \Biggr] 
\leq
\frac{C K_{\Psi_\S}^{8d+30} \Pi^4 3^{-2n}}{(1-\gamma)^3} 
\leq 
C\delta \sigma^2
\end{align}
and
\begin{equation} 
\label{e.this.is.so.nice.pre.second}
\frac{1}{1-\gamma} 3^{-(1-\gamma)l}
\leq
C \delta \sigma^2
\,.
\end{equation}
By combining the above displays, we obtain that
\begin{equation*}  
3^{-n} 
\E \Biggl[
\biggl[ 
\mathbf{M}_0^{\nicefrac12} 
\begin{pmatrix} 
\nabla v(\cdot,\cus_n,p,q) -  (\nabla v_n)_{\cus_n} \\ 
\a \nabla v(\cdot,\cus_n,p,q) -  (\a\nabla v_n)_{\cus_n}
\end{pmatrix} 
\biggr]_{\Besov{-\nf12}{2}{1}(\cus_{n})}^2
\Biggr]
\leq
C  \delta \sigma^2 \Theta \, .
\end{equation*}
The last thing to check is that
\begin{equation*}  
\E\Biggl[ \biggl|
\mathbf{M}_0^{\nicefrac12} 
\begin{pmatrix} 
P -  (\nabla v_n)_{\cus_n}   \\ 
Q - (\a\nabla v_n)_{\cus_n}
\end{pmatrix} 
\biggr|^2 
\Biggr]
\leq 
C  \delta \sigma^2 \Theta \, .
\end{equation*}
To see this, we use~\eqref{e.PQ.choice} and the same computation as for~\eqref{e.sqbound} to estimate
\begin{equation*} 
\biggl|
\mathbf{M}_0^{\nicefrac12} 
\begin{pmatrix} 
P -  (\nabla v_n)_{\cus_n}   \\ 
Q - (\a\nabla v_n)_{\cus_n}
\end{pmatrix} 
\biggr|^2 
\leq
\bigl| \bfE^{\nicefrac12}  \mathbf{M}_0^{-1} \bfE^{\nicefrac12} \bigr|
\bigl| \bfE^{-\nicefrac12}
\bigl( \bfAhom_m - \bfA(\cus_n) \bigr) \bfE^{-\nicefrac12} \bigr|^2
\biggl| \bfE^{\nicefrac12}\begin{pmatrix} 
-p  \\ 
q
\end{pmatrix} 
\biggr|^2
\,.
\end{equation*}
Thus, we obtain by~\eqref{e.MEbound},~\eqref{e.Enaught.vs.p.q.verified.first},~\eqref{e.Enaught.vs.Am} and~\eqref{e.pigeon.var} that
\begin{align*}
\E\Biggl[ \biggl|
\mathbf{M}_0^{\nicefrac12} 
\begin{pmatrix} 
P -  (\nabla v_n)_{\cus_n}   \\ 
Q - (\a\nabla v_n)_{\cus_n}
\end{pmatrix} 
\biggr|^2 
\Biggr]
& 
\leq
C \Theta 
 \bigl| \bfE^{-\nicefrac12} \bfAhom_m  \bfE^{-\nicefrac12}\bigr|
\E\Bigl[
\bigl| \bfAhom_m^{-\nicefrac12}  \bfAhom(\cus_n) \bfAhom_m^{-\nicefrac12}  - \Itwod \bigr|^2 \Bigr]
\leq
C \delta \sigma^2 \Theta 
\,.
\end{align*}
The proof is complete.
\end{proof}

We turn to the proof of Proposition~\ref{p.renormalize.reduce}.

\begin{proof}[{Proof of Proposition~\ref{p.renormalize.reduce}}]
We just need to assemble the estimates we have proved above and choose the parameters appropriately. As above, we set~$n  \coloneqq  m - l$.

\smallskip

Combining~\eqref{e.Thetam.by.Jildes},~\eqref{e.dragon.egg} and Lemmas~\ref{l.Jtilde.energy.bound} and~\ref{l.weaknorms},  using also the analogous estimates for~$J^*$, we obtain, for a constant~$C(d)<\infty$,
\begin{align}
\label{e.Jtilde.energy.bound.2}
\Theta_m -1
&
\leq
2d\Bigl ( \E\bigl  [ {J}(\cus_n,p,q) + {J}^*(\cus_n,p,q)\bigr  ]
- \frac12  P\cdot Q -\frac12 P^*\cdot Q^*
\Bigr )
+
C(1-\gamma)^{-1} \Pi^2 K_{\Psi_{\S}}^{3\gamma}3^{-l} 
\notag \\ &
\leq C \delta^{\nf12} \sigma \Theta 
+
C(1-\gamma)^{-1} \Pi^2 K_{\Psi_{\S}}^{3\gamma}3^{-l} 
\,.
\end{align}
Since~$l=m-n$, we have by~\eqref{e.l.def} that, if~$\delta_0(d)$ is sufficiently small, then the second term on the right side is bounded by~$\frac12\sigma$. Assuming also that~$\delta_0$ satisfies~$\delta_0 \leq
\bigl( 2C_{\eqref{e.Jtilde.energy.bound.2} }
\bigr)^{\!-2}$ so that the first term on the right is bounded by~$\frac12\sigma \Theta$, 
we obtain~$\Theta_m-1 \leq \frac12 \sigma \Theta + \frac12\sigma \leq \sigma \Theta$, which is~\eqref{e.improve.reduce}. 
\end{proof}

\subsection{The iteration from high contrast to small contrast}

We give the proof of Theorem~\ref{t.HC}, which is based on an iteration of Proposition~\ref{p.renormalize}, renormalizing between each iteration step by appealing to  Proposition~\ref{p.renormalization.P}. 

\begin{proof}[{Proof of Theorem~\ref{t.HC}}]
We introduce the parameters
\begin{equation}
\label{e.newparams.encore.pre}
\left\{
\begin{aligned}
& \gamma_{*}  \coloneqq  \frac12 (\min\{ \nu,1 \} +\gamma) 
\, \\ &
K_{\Psi_\S}^{*}  \coloneqq  \max\big\{ K_{\Psi_{\S}} , K_{\Psi}^{\lceil \nicefrac1\mu \rceil} \bigr\}
\,, \\ & 
\Theta^{*}  \coloneqq  4\Theta
\,, \\ & 
\Pi^{*}  \coloneqq  2^{10}\Pi\,.
\end{aligned}
\right.
\end{equation}
Also denote~$\rho \coloneqq  \frac12(\min\{ \nu,1\}+ \gamma_*)$ and~$\alpha^* \coloneqq  (\min\{ \nu,1\}-\gamma^{*})(1-\beta)$. 
Motivated by~\eqref{e.l0.condition}, we let~$l_0 \in\N$ be defined by
\begin{equation}
\label{e.special.l0}
l_0 \coloneqq 
\biggl\lceil
\frac{1}{\rho-\gamma^{*}}
\Bigl(1  + \frac d{\alpha^{*}}\Bigr) \bigl( 6 + 
\log (\Theta^{*})
\bigr)  +\frac{6}{\alpha^{*}} \bigl( 1 + \log K_{\Psi} \bigr)  
\biggr\rceil
\,,
\end{equation}
Motivated by~\eqref{e.m.explivomit}, we define
\begin{equation*}
l_1 \coloneqq  
\biggl\lceil 
\frac{C}{\sigma^2} 
\biggl(  
\log K_{\Psi_\S}^* + \frac{1}{(\alpha^*)^2} \log \Bigl( \frac{\Pi ^*K_{\Psi}}{\alpha^* \sigma} \Bigr)  
\biggr)
\biggr\rceil\,,
\end{equation*}
where~$C(d)<\infty$ is the constant in Proposition~\ref{p.renormalize}. 
In terms of the original parameters, we have that 
\begin{equation}
\label{e.l0.vomitbound}
l_0 \leq 
\biggl\lceil
\frac{4}{\alpha}
\Bigl(1  + \frac {2d}{\alpha}\Bigr) \bigl( 9 + 
\log \Theta
\bigr)  +\frac{12}{\alpha} \bigl( 1 + \log K_{\Psi} \bigr)  
\biggr\rceil
\leq 
\frac{C}{\alpha^2} \log \Bigl( \frac{\Pi K_{\Psi}}{\alpha\sigma } \Bigr)
\,,
\end{equation}
where we recall~$\alpha= (\min\{\nu,1\}-\gamma )(1-\beta)$; and, by inflating~$C$ by an additional (universal) factor, 
\begin{equation}
\label{e.l1.vomitbound}
l_1 \leq 
\frac{C}{\sigma^2} 
\biggl(  
\log K_{\Psi_\S} + \frac{1}{\alpha^2} \log \Bigl( \frac{\Pi K_{\Psi}}{\alpha\sigma } \Bigr)  
\biggr)
\,.
\end{equation}
We also denote~$m_0  \coloneqq  l_0  + \lceil \log K_\Psi\rceil$.

\smallskip

For each~$n\in\N$ with~$n \geq m_0$, we may apply Proposition~\ref{p.renormalization.P} to obtain that the pushforward probability measure~$\P_{n}$, defined in~\eqref{e.Pn0}, satisfies the assumptions~\ref{a.stationarity},~\ref{a.ellipticity.dagger} and~\ref{a.CFS} with~$\delta\leq 1$ and the new parameters
\begin{equation*}
\left\{
\begin{aligned}
& \mathbf{E}_{\mathrm{new}}  \coloneqq  2 \bfAhom(\cu_{n-2l_0}) 
\,, \\ & 
\gamma_{\mathrm{new}}  \coloneqq  \gamma^*
\, \\ &
K_{\Psi,\mathrm{new}}  \coloneqq  K_{\Psi}
\, \\ &
K_{\Psi_\S,\mathrm{new}}  \coloneqq  K_{\Psi_{\S}}^{*}
\,, \\ & 
\Theta_{\mathrm{new}}  \coloneqq  4 \Theta_{n-2l_0} \leq 4\Theta = \Theta^*
\,, \\ & 
\Pi_{\mathrm{new}} \leq \Pi^* \,.
\end{aligned}
\right.
\end{equation*}
For such~$n$, applying Proposition~\ref{p.renormalize} with~$\P_{n}$ in place of~$\P$, with~$\delta=1$, and with~$\sigma=\sigma_0=\nicefrac18$ as above, 
we obtain that, for every~$n\geq m_0$, 
\begin{equation*}
\min \biggl\{ 
\frac{\Theta_{n+2l_0+l_1} -1}{\Theta_n} \,, 
\frac{\XiDet_{n+2l_0+l_1}}{\XiDet_n} 
\biggr\} 
\leq 
\sigma 
\,.
\end{equation*}
This implies that 
\begin{equation*}
\bigl( \Theta_{n+2l_0+l_1} - (1+2\sigma) \bigr) \XiDet_{n+2l_0+l_1}
\leq 
\sigma \bigl( \Theta_n - (1+2\sigma) \bigr) \XiDet_n\,.
\end{equation*}
Since~$\sigma \leq \nf12$ and~$\Theta_{m_0} \XiDet_{m_0} \leq \Theta^2$, iterating this inequality gives 
\begin{align*}
\bigl( 
\Theta_{m_0 + k ( 2l_0+l_1)} 
- (1+2\sigma) 
\bigr) 
\XiDet_{m_0 + k ( 2l_0+l_1)}
\leq 
\sigma^k 
\bigl( 
\Theta_{m_0 } 
- (1+2\sigma) 
\bigr) 
\XiDet_{m_0 }
\leq \sigma^k \Theta^2 
\,.
\end{align*}
Inserting~$k= \lceil 2\log_2 \Theta + 1\rceil$, we therefore obtain 
\begin{equation*}
m \geq m_0 + 2(\log_2 \Theta+1) (2l_0+l_1) 
\quad \implies \quad 
\Theta_m \XiDet_m \leq 1+ 3\sigma\,.
\end{equation*}
Since~$\XiDet_m \geq 1$, we also obtain~$\Theta_m\leq 1+ 2\sigma$ for all such~$m$.
In view of~\eqref{e.l0.vomitbound},~\eqref{e.l1.vomitbound} and the definition of~$m_0$, we find that
\begin{equation*}
m_0 + 2(\log_2 \Theta+1) (2l_0+l_1)  
\leq 
\frac{C}{\sigma^2} 
\biggl(  
\log K_{\Psi_\S} + \frac{1}{\alpha^2} \log \Bigl( \frac{\Pi K_{\Psi}}{\alpha \sigma} \Bigr)  
\biggr)
\log (1+\Theta)
\,,
\end{equation*}
which, after relabeling~$2\sigma$ to~$\sigma$, yields~\eqref{e.squoosh} and completes the proof of the theorem.
\end{proof}

\section{Renormalization in small contrast}
\label{ss.algebraic}

The main point of Theorem~\ref{t.HC} is to estimate, in the case that~$\Theta$ is very large, the  scale parameter~$m$ at which~$\Theta_m$ is not very large (say bounded by two). It does not however give a useful quantitative estimate on the rate of convergence of~$\Theta_m-1$ to zero, as the right side of~\eqref{e.highcontrast.tamed.m} blows up very fast as~$\sigma \to 0$. The purpose of this section is to improve this by showing that~$\Theta_m-1$ decays algebraically in the scale~$3^m$, beginning at the length scale prescribed by Theorem~\ref{t.HC}.

\begin{theorem}
\label{t.main.algebraic}
Assume that~$\P$ satisfies assumptions~\ref{a.stationarity},~\ref{a.ellipticity.dagger} and~\ref{a.CFS} with~$\nu > \gamma$.
There exist constants~$C(d)<\infty$ and~$c(d) \in (0,\nicefrac12]$ such that, if we define
\begin{equation}
\label{e.alpha.kappa.def}
\left\{
\begin{aligned}
&
\alpha \coloneqq  \bigl( \min\{ \nu,1\} - \gamma\bigr)(1-\beta)\,,
\\ & 
\kappa  \coloneqq  \min\{ c , \nicefrac \alpha3\}
\\ & 
m_*  \coloneqq  
C
\Bigl(  
\log K_{\Psi_\S} + \frac{1}{\alpha^2} \log \Bigl( \frac{\Pi K_{\Psi}}{\alpha} \Bigr)  
\Bigr)\log (1+\Theta)
\,,
\end{aligned}
\right. 
\end{equation}
then, for every~$m\in\N$ with~$m\geq m_*$,
\begin{equation}
\label{e.algebraic.Thetam.main}
\Theta_m-1
\leq 
3^{-\kappa (m - m_*) } 
\,.
\end{equation} 
\end{theorem}

The idea of the proof of Theorem~\ref{t.main.algebraic} is simple. If we renormalize at scale~$3^{m_*+l_0}$, with~$m_*\in\N$ larger than the right side of~\eqref{e.highcontrast.tamed.m} with~$\sigma = 10^{-3}$ and~$l_0$ given in~\eqref{e.ell.naught.def}, then the combination of Theorem~\ref{t.HC} and Lemma~\ref{l.renormalize.ellipticity} says that the renormalized ellipticity ratio (in the sense of~\ref{a.ellipticity.dagger}) is bounded by~$1.001$. In other words, we are in the small ellipticity contrast regime. We should therefore try to adapt the usual iteration arguments, developed in the uniformly elliptic setting (see for instance~\cite{AKMbook,AK.Book}), to show that the coarse-grained matrices converge to the homogenized matrix at an algebraic rate. Since the contrast is now small, we cannot pick up any further dependence on ellipticity. 
We just need to demonstrate that these iteration arguments are flexible enough that they still work under  assumption~\ref{a.ellipticity.dagger}, rather than requiring uniform ellipticity.

\smallskip

This is the purpose of the following proposition, which is independent of the main results of Section~\ref{s.firstpigeon} and asserts that, if~$\Theta_0 -1$ is sufficiently small, then the convergence of~$\Theta_m-1$ is algebraic in the length scale~$3^m$ and has power-like dependence in the ellipticity ratios and other given parameters. This result is of interest beyond its application to the proof of Theorem~\ref{t.main.algebraic} because many problems which exhibit large ellipticity contrast---in the sense of uniform ellipticity---are actually of small ellipticity contrast in the sense of \emph{coarse-grained} ellipticity.

\begin{proposition}
[Homogenization in small contrast]
\label{p.algebraic.exp}
Assume~$\P$ satisfies assumptions~\ref{a.stationarity},~\ref{a.ellipticity.dagger} and~\ref{a.CFS} with~$\nu > \gamma$.
There exist constants~$C(d)<\infty$ and~$c(d) \in (0,\nicefrac12]$ such that, if we define
\begin{align}
\left\{
\begin{aligned}
& 
\label{e.kappa.def}
\kappa 
 \coloneqq 
\min\biggl\{ c  \,,\, \frac{1- \gamma}{2} \,,\,\frac{\nu-\gamma}{1+\nu-\gamma}  \,,\, \frac{(1-\beta)(\nu-\gamma)}{\beta+(1-\beta)(\nu-\gamma)} \biggr\}
\quad \mbox{and} 
\\& 
m_0  \coloneqq  \biggl\lceil \frac{C}{\kappa}\log_3 \Bigl( \max\bigl\{ \Pi, K_{\Psi_\S} , K_{\Psi}, (1-\gamma)^{-1}\bigr\} \Bigr) \biggr\rceil 
\,,
\end{aligned}
\right.
\end{align}
then we have
\begin{equation*}
\Theta_0 - 1\leq c
\ \  \implies \ \ 
\Theta_m-1
\leq 
3^{-\kappa  (m - m_0) } 
\,,
\quad \forall m \in \N \cap [m_0,\infty) \,.
\end{equation*}
\end{proposition}

The proof Proposition~\ref{p.algebraic.exp} is the main result of this section. 
We show first that Theorem~\ref{t.main.algebraic} can be reduced to its statement, by using Theorem~\ref{t.HC} and the renormalization lemma (Proposition~\ref{p.renormalization.P}) to remove the restriction~$\Theta_0 -1 \leq c$.

\begin{proof}[{Proof of Theorem~\ref{t.main.algebraic}}]
Using Theorem~\ref{t.HC}, we find a scale~$n_0$ with 
\begin{equation*}
n_0 \leq 
C
\biggl(  
\log K_{\Psi_\S} + \frac{1}{\alpha^2} \log \Bigl( \frac{\Pi K_{\Psi}}{\alpha} \Bigr)  
\biggr)
\log (1+\Theta)
\end{equation*}
such that, if we let~$\sigma_0(d)$ be the constant given in the statement of~Proposition~\ref{p.algebraic.exp}, then 
\begin{equation*}
\Theta_{n_0} \leq 1 + \sigma_0\,.
\end{equation*}
Let~$l_0$ be defined as in~\eqref{e.newparams.encore.pre}--\eqref{e.special.l0} and let~$m_0  \coloneqq  n_0 + 2l_0 + \lceil \log K_\Psi \rceil$. 
As in the argument at the beginning of the proof of Theorem~\ref{t.HC}, we may apply  Proposition~\ref{p.renormalization.P} to obtain that the pushforward probability measure~$\P_{m_0-2l_0}$, defined in~\eqref{e.Pn0}, satisfies the assumptions~\ref{a.stationarity},~\ref{a.ellipticity.dagger} and~\ref{a.CFS} with the new parameters
\begin{equation*}
\left\{
\begin{aligned}
& \mathbf{E}_{\mathrm{new}}  \coloneqq  2 \bfAhom(\cu_{m_0-2l_0}) \leq 2 \bfAhom(\cu_{n_0}) 
\,, \\ & 
\gamma_{\mathrm{new}}  \coloneqq  \frac12 (\min\{1,\nu\}+\gamma)
\, \\ &
K_{\Psi,\mathrm{new}}  \coloneqq  K_{\Psi}
\, \\ &
K_{\Psi_\S,\mathrm{new}}  \coloneqq  K_{\Psi_{\S}}^{*}
\,, \\ & 
\Theta_{\mathrm{new}}  \coloneqq  4\Theta_{m_0-2l_0} \leq 4\Theta_{n_0} \leq 4+4\sigma_0 \leq 8
\,, \\ & 
\Pi_{\mathrm{new}} \leq \Pi^* \,.
\end{aligned}
\right.
\end{equation*}
Observe that~$\Theta_0$ for~$\P_{m_0-2l_0}$ is the same as~$\Theta_{m_0-2l_0}$ for~$\P$, and we have~$\Theta_{m_0-2l_0} \leq \Theta_{n_0} \leq 1+\sigma_0$. We may therefore apply Proposition~\ref{p.algebraic.exp} with~$\P_{m_0-2l_0}$ in place of~$\P$, to obtain that, for every~$n\in\N$,
\begin{equation*}
\Theta_{n_0 + \lceil \log K_\Psi\rceil + n}
=
\Theta_{m_0-2l_0 +n } 
\leq
\frac{C (\Pi^*)^2\max\{ K_{\Psi_\S}^* , K_{\Psi}\}^8  }{ (1-\gamma^*)^2} 
3^{-\kappa  n}
\leq
\frac{C \Pi^2}{\alpha^2} \max\big\{ K_{\Psi_\S} , K_{\Psi}^{\lceil \nicefrac1\alpha \rceil} \big\}^8  
3^{-\kappa  n}\,,
\end{equation*}
where~$\kappa$ is as defined in~\eqref{e.kappa.def} and~$\alpha = ( \min\{ \nu,1\} - \gamma)(1-\beta)$. 
Reindexing this, we obtain
\begin{equation*}
\Theta_{m} 
\leq 
\frac{C \Pi^2}{\alpha^2} \max\bigl\{ K_{\Psi_\S} , K_{\Psi}^{\lceil \nicefrac1\alpha \rceil}\bigr\}^8  
3^{-\kappa  \left(m-n_0-\lceil \log K_\Psi \rceil \right)} \,,
\quad \forall m \geq n_0 + \lceil \log K_\Psi\rceil\,,
\end{equation*}
We next observe that 
\begin{equation*}
\frac{C \Pi^2\max\bigl\{ K_{\Psi_\S} , K_{\Psi}^{\lceil \nicefrac1\alpha \rceil}\bigr\}^8  }{ \alpha^2} 
3^{\kappa (n_0 + \lceil \log K_\Psi \rceil )} 
\leq
\exp\biggl( C
\biggl(  
\log K_{\Psi_\S} + \frac{1}{\alpha^2} \log \Bigl( \frac{\Pi K_{\Psi}}{\alpha} \Bigr)  
\biggr)
\log (1+\Theta)
\biggr)\,.
\end{equation*}
Substituting this into the previous display gives~\eqref{e.algebraic.Thetam.main} for every~$m \geq n_0 + \lceil \log K_\Psi\rceil$. For~$m\in\N$ with~$m \leq n_0 + \lceil \log K_\Psi\rceil$, the inequality is obtained trivially from~$\Theta_m \leq \Pi$, after possibly enlarging the constant~$C$. This completes the proof of the corollary. 
\end{proof}

In the next subsection, we identify the homogenized matrix~$\ahom$ and state a quenched version of Theorem~\ref{t.main.algebraic}, the proof of which we defer to Section~\ref{ss.main.algebraic.quenched.proof}. The proof of Proposition~\ref{p.algebraic.exp}, adapted from~\cite[Section 5]{AK.Book}, appears in Section~\ref{ss.smallcontrast}. 

\smallskip

Throughout the rest of this section, we assume~$\P$ is a probability measure on~$(\Omega,\mathcal{F})$ satisfying~\ref{a.stationarity},~\ref{a.ellipticity.dagger} and~\ref{a.CFS}, with~$\nu > \gamma$.

\subsection{Identification of the homogenized matrix}
Theorem~\ref{t.HC} implies the qualitative limit
\begin{equation*}
\lim_{m\to \infty} (\Theta_m -1) = 0\,,
\end{equation*}
and this allows us to identify the homogenized matrix. To see this, we recall that, by~\eqref{e.monotone.s}, each of the maps~$n \mapsto \shom(\cu_n)$ and~$n \mapsto \shom_{*}^{-1}(\cu_n)$ is nonincreasing and bounded. Thus there exist~$\shom$ and~$\shom_*^{-1}$ such that~$\shom_* \leq \shom$ and 
\begin{equation*}
\shom = \lim_{n\to \infty} \shom(\cu_n) \quad \mbox{and} \quad 
\shom_{*}^{-1} = \lim_{n\to \infty} \shom_*^{-1}(\cu_n)
\,.
\end{equation*}
The definition of~$\Theta_n $ implies that, for every~$n\in\N$,
\begin{equation*}
\shom_*  \leq \shom \leq \shom(\cu_n) \leq  \Theta_n  \shom_*(\cu_n)  \leq \Theta_n  \shom_*  \,.
\end{equation*}
Therefore, the qualitative limit~$\Theta_n \to 1$ implies that~$\shom=\shom_*$.
To obtain a limit for~$\khom(\cu_n)$, we recall two facts: first, by~\eqref{e.Eone.vs.Etwo.one} we have that, for every~$m,n\in\N$ with~$m\geq n$, 
\begin{equation*}  
\shom(\cu_n) + (\khom(\cu_n) - \khom(\cu_m))^t \shom_{*}^{-1}(\cu_n) (\khom(\cu_n) - \khom(\cu_m))
\leq 
\shom(\cu_m )
\leq  
\Theta_m\shom_{*}(\cu_m )
\,.
\end{equation*}
Together with the qualitative limit~$\Theta_m \to 1$ and~$\shom=\shom_*$, this implies that~$\khom(\cu_n)$ has a limit, which we denote by~$\khom$. 
Second, we recall that by~\eqref{e.symm.k.quad.small}, we have 
\begin{equation*}  
\bigl| \s_{*}^{-\nicefrac12}(\cu_m)  (\khom(\cu_m) + \khom(\cu_m)^t)\s_{*}^{-\nicefrac12}(\cu_m)   \bigr| 
\leq
\Theta_m - 1 
\,.
\end{equation*}
Sending~$m\to \infty$ yields that~$\khom$ is anti-symmetric. 
We define
\begin{equation*}
\ahom  \coloneqq  \shom + \khom\,.
\end{equation*}
We let~$\bfAhom$ denote the corresponding~$2d$-by-$2d$ limiting matrix
\begin{equation*}
\label{e.Ahom.final}
\bfAhom  \coloneqq 
\begin{pmatrix} 
\shom + \khom^t \shom^{-1}\khom & -\khom^t \shom^{-1} \\ -\shom^{-1}\khom & \shom^{-1}
\end{pmatrix}\,.
\end{equation*}
It follows that~$\lim_{m\to \infty} \bfAhom(\cu_m) = \bfAhom$. 
Moreover, due to the ordering~$\bfAhom \leq \bfAhom(\cu_m)$, we can apply Lemma~\ref{l.bfE.bounds} to obtain
\begin{equation}
\label{e.Thetam.controls.Ahom}
0
\leq 
\bfAhom(\cu_m) - \bfAhom
\leq 
6(\Theta_m -1) \bfAhom\,.
\end{equation}

\smallskip

By combining Theorem~\ref{t.main.algebraic} with the renormalization lemmas in Section~\ref{ss.renormalization}, we obtain \emph{quenched} estimates for the difference between the random matrix~$\bfA(\cu_m)$ and~$\bfAhom$.

\begin{corollary} 
\label{c.main.algebraic.quenched}
Assume the assumptions of Theorem~\ref{t.main.algebraic}, and let~$\kappa$ and~$m_*$ be as in the  statement. Also let~$\mu$ be defined as in~\eqref{e.def.mu}. 
For each~$\delta>0$ and~$\rho \in (\gamma,1)$ and~$\mu'\in (0,\mu)$, there exists a random variable~$\mathcal{R} = \mathcal{R}_{\delta,\rho}$ satisfying
\begin{equation}
\label{e.cor.sec.four.R}
\mathcal{R}^{\mu} 
= 
\O_\Psi \Bigl( 
3^{m_* \mu}
\exp \Bigl( 
\frac{C}{\rho-\gamma}
\log \bigl( \max\{ K_{\Psi},\Theta,\delta^{-1} \} \bigr) 
\Bigr) 
\Bigr)\,,
\end{equation}
such that, if we define the exponent 
\begin{equation}
\theta \coloneqq  \frac14  \min \Bigl\{ \kappa, \frac12 \frac{(\rho-\gamma)\mu }{d+\mu} \Bigr\} 
\,,
\label{e.theta.def}
\end{equation}
then, for every~$m\in \N$, 
\begin{align}
\label{e.final.P2}
\lefteqn{
3^m\geq \max\bigl\{ \S , \mathcal{R} \bigr\}  
} \ \ 
\notag \\& 
\! \implies \! 
\bfA(z+\cu_k) 
\leq 
\biggl( 1 + \delta \biggl( \frac{3^m}{\max\{ \mathcal{S},\mathcal{R}\} } \biggr)^{\!-\theta}  3^{\rho(m-k)} \biggr)
\bfAhom
\,,
\quad  \forall k \leq m\,, \ z \in 3^k \Z^d \cap \cu_m\,.
\end{align}
\end{corollary}
The proof of Corollary~\ref{c.main.algebraic.quenched} is given in Section~\ref{ss.main.algebraic.quenched.proof} below.

\subsection{Iteration from small contrast to homogenization}
\label{ss.smallcontrast}

The proof of Proposition~\ref{p.algebraic.exp} is based on an iteration argument which uses some of the same ingredients as the one in the proof of Theorem~\ref{t.HC}, but takes advantage of the smallness of~$\Theta_0 -1$ to accelerate the convergence of~$\Theta_m-1$ to zero. 
As it is crucial that our estimates which do not have dependence on~$\Theta$ or~$\Pi$, we must work with the adapted cubes~$\cus_{n}$ instead of the Euclidean cubes~$\cu_n$. 
In this section, these adapted cubes are defined as in Section~\ref{ss.subadditivity} with respect to the matrix
\begin{equation*}
\m_0  \coloneqq  \s_0\,.
\end{equation*}
Since we are working in the low contrast regime in this section, all reasonable definitions of~$\m_0$ will be equivalent.

\smallskip

It will be convenient to define an analog of~$\Theta_m$ in terms of the adapted cubes. In fact we will require two variations of~$\Theta_m$, defined by
\begin{equation*} 
\hat{\Theta}_m \coloneqq  
\frac1d
\tr\bigl( ( \shom_{*}^{-\nicefrac12} \shom\shom_{*}^{-\nicefrac12})(\cus_{m}) \bigr)
\qand
\tilde{\Theta}_m \coloneqq  \bigl| ( \shom_{*}^{-\nicefrac12} \shom\shom_{*}^{-\nicefrac12})(\cus_{m}) \bigr|
\,.
\end{equation*}
By subadditivity, both of the mappings~$m \mapsto \hat{\Theta}_m$ and~$m \mapsto \tilde{\Theta}_m$ are monotone decreasing. We also have that
\begin{equation}
\label{e.ordering.hat.tilde.Theta}
\hat{\Theta}_m \leq \tilde{\Theta}_m \leq d \hat{\Theta}_m 
\,.
\end{equation}

\smallskip

We begin the proof of Proposition~\ref{p.algebraic.exp} with an alternative to Lemma~\ref{l.pigeonhole}. Although the estimate here is quenched, rather than an estimate for the variance as in~\eqref{e.variance.HC}, it is nevertheless very close to~\eqref{e.variance.HC}. 
One significant improvement should be highlighted. 
In~\eqref{e.variance.HC}, there are three error terms on the right side: the first term is analogous to~$\hat{\Theta}_k - \hat{\Theta}_m$, up to constant factors, which appears on the right side of~\eqref{e.algebraic.add.error} below. The difference here is that the error is \emph{quadratic}; that is, the term is squared in~\eqref{e.algebraic.add.error} compared to~\eqref{e.variance.HC}. (Note that the left side of~\eqref{e.variance.HC} has a square, but the left side of~\eqref{e.algebraic.add.error} does not.)
This may seem like a subtle and technical point at first glance, but it is the main reason for the accelerated convergence in small contrast. 
The reader can also consult the arguments of~\cite[Sections 4.2 \& 5.1]{AK.Book} for more discussion of this point.

\begin{lemma}[Fluctuation estimate]
\label{l.variance}
Suppose~$k,n,m\in\N$ satisfy~$k \leq n \leq m$ and
\begin{equation}
\label{e.smallness.ass}
\tilde{\Theta}_m   - 1 
\leq  (80d)^{-1}
\,.
\end{equation}
Then, we have the estimates
\begin{equation}
\label{e.algebraic.add.error}
\bigl| \bfAhom^{-\nicefrac12}(\cus_m)   \bfAhom(\cus_k) \bfAhom^{-\nicefrac12}(\cus_m)   - \Itwod \bigr|
\leq 
4d \bigl(\hat{\Theta}_k - \hat{\Theta}_m \bigr)
\end{equation}
and
\begin{align}
\label{e.variance.proto} 
\lefteqn{
\bigl |
\bfAhom^{-\nicefrac12}(\cus_m) 
\bfA(\cus_n)  \bfAhom^{-\nicefrac12}(\cus_m) 
-\Itwod
\bigr | 
} \qquad &
\notag \\ &
\leq 
16 d (\hat{\Theta}_k-1) 
+  
4 \, \biggl |  \avsum_{z \in 3^k \Lat \cap \cus_n} \! \!
\bfAhom^{-\nicefrac12}(\cus_m)  \bigl ( \bfA(z+\cus_k) - \bfAhom(\cus_k)\bigr ) \bfAhom^{-\nicefrac12}(\cus_m)  
\biggr |
\,.
\end{align}
\end{lemma}
\begin{proof}
Fix~$n,m,k\in\N$ with~$k \leq n  \leq m$. We slightly deviate from the notation in the previous section and set, for every~$j\in \N$, 
\begin{equation*}  
\bfAhom_{j}  \coloneqq  \bfAhom(\cus_j)\,, \ \ 
\bfAhom_{*,j}  \coloneqq  \bfAhom_*(\cus_j)\,, \ \  
\shom_j  \coloneqq  \shom(\cus_j)\,, \quad
\shom_{*,j}  \coloneqq  \shom_{*}(\cus_j)\,, \ \ 
\khom_{j}  \coloneqq  \khom(\cus_j) \ \ 
\bhom_j  \coloneqq  \bhom(\cus_j)\,.
\end{equation*}
We assume, without loss of generality, that $\khom_k$ is symmetric. Otherwise, we recenter by subtracting the anti-symmetric part of~$\khom_k$ from the coefficient field~$\a(\cdot)$, using the observations from Section~\ref{ss.skew}. 
Note that~$\hat{\Theta}_j$ as well as the matrix ratios in~\eqref{e.algebraic.add.error} and~\eqref{e.variance.proto} are invariant with respect to recenterings of the anti-symmetric part of the coefficient field, as explained in Section~\ref{ss.skew}, and therefore so are the assumptions and conclusions of the lemma. 

\smallskip

\emph{Step 1}. The proof of~\eqref{e.algebraic.add.error}. 
We argue as in the proof of Lemma~\ref{l.bfE.bounds}. Denote, for~$\eta\in [0,1]$, 
\begin{equation*}
\mathbf{G}_\eta
 \coloneqq  
\mathbf{G}_{\khom_{k} + (1-\eta)(\khom_m-\khom_k)}
 \coloneqq \begin{pmatrix} 
\Id
& 0
\\ \khom_{k} + (1-\eta)(\khom_m-\khom_k)
& \Id
\end{pmatrix}
\end{equation*}
and compute the top left block of the matrix~$\mathbf{G}_\eta^t
\bigl( \bfAhom_k - \bfAhom_m \bigr)\mathbf{G}_\eta$, which is 
\begin{align*}
&
\mbox{the top left block of} \ \mathbf{G}_\eta^t
\bigl( \bfAhom_k - \bfAhom_m \bigr)\mathbf{G}_\eta
\notag \\ & \qquad
=
(\shom_k-\shom_m) +
(1-\eta)^2
(\khom_k-\khom_m)^t\shom_{*,k}^{-1} (\khom_k-\khom_m)
- \eta^2 (\khom_k-\khom_m)^t\shom_{*,m}^{-1} (\khom_k-\khom_m)
\,.
\end{align*}
The nonnegativity of~$\mathbf{G}_\eta^t
\bigl( \bfAhom_k - \bfAhom_m \bigr)\mathbf{G}_\eta$ implies that the matrix above is also nonnegative, and thus
\begin{align*}
 \eta^2 (\khom_k-\khom_m)^t\shom_{*,m}^{-1} (\khom_k-\khom_m)
&
\leq 
(\shom_k-\shom_m) +
(1-\eta)^2
(\khom_k-\khom_m)^t\shom_{*,k}^{-1} (\khom_k-\khom_m)
\notag \\ &
\leq 
(\shom_k-\shom_m) +
\bigl |\shom_{*,m}^{\nicefrac12}\shom_{*,k}^{-1}\shom_{*,m}^{\nicefrac12}\bigr |
(1-\eta)^2
(\khom_k-\khom_m)^t\shom_{*,m}^{-1} (\khom_k-\khom_m)\,.
\end{align*}
After rearranging this, we get
\begin{align*}
\bigl (\eta^2 
- 
\bigl |\shom_{*,m}^{\nicefrac12}\shom_{*,k}^{-1}\shom_{*,m}^{\nicefrac12}\bigr |
(1-\eta)^2
\bigr )
(\khom_k-\khom_m)^t\shom_{*,m}^{-1} (\khom_k-\khom_m)
\leq
\shom_k-\shom_m
\,.
\end{align*}
Optimizing in~$\eta$ leads to the choice~$\eta= \bigl |\shom_{*,m}^{\nicefrac12}\shom_{*,k}^{-1}\shom_{*,m}^{\nicefrac12}\bigr |/\bigl |\shom_{*,m}^{\nicefrac12}\shom_{*,k}^{-1}\shom_{*,m}^{\nicefrac12}-\Id\bigr |$ and we obtain
\begin{align*}
(\khom_k-\khom_m)^t\shom_{*,m}^{-1} (\khom_k-\khom_m)
&
\leq
\bigl |\shom_{*,m}^{\nicefrac12}\shom_{*,k}^{-1}\shom_{*,m}^{\nicefrac12}-\Id\bigr |
(\shom_k-\shom_m)
\notag \\ &
\leq
\bigl |\shom_{*,m}^{\nicefrac12}\shom_{*,k}^{-1}\shom_{*,m}^{\nicefrac12}-\Id\bigr |
\bigl |\shom_{m}^{-\nicefrac12}\shom_{k}\shom_{m}^{-\nicefrac12}-\Id\bigr |
\shom_m
\,.
\end{align*}
Observe furthermore that 
\begin{align*}
d\bigl ( \hat{\Theta}_k - \hat{\Theta}_m \bigr )
=
\tr\bigl( \shom_k \shom_{*,k}^{-1} - \shom_m \shom_{*,m}^{-1} \bigr) 
&
=
\tr\Bigl( (\shom_k -\shom_m)\shom_{*,k}^{-1}
+
\shom_m( \shom_{*,k}^{-1}-\shom_{*,m}^{-1}) \Bigr) 
\notag \\ & 
\geq 
\tr\Bigl( (\shom_k -\shom_m)\shom_{m}^{-1}
+
\shom_{*,m}( \shom_{*,k}^{-1}-\shom_{*,m}^{-1}) \Bigr) 
\notag \\ & 
=
\tr\Bigl(
\shom_k \shom_m^{-1} 
+
\shom_{*,k}^{-1} 
\shom_{*,m} 
\Bigr) 
- 2d
\notag \\ & 
\geq 
\bigl (\bigl |\shom_{m}^{-\nicefrac12}\shom_{k}\shom_{m}^{-\nicefrac12}-\Id\bigr | + \bigl |\shom_{*,m}^{\nicefrac12}\shom_{*,k}^{-1}\shom_{*,m}^{\nicefrac12}-\Id\bigr |
 \bigr )
\,.
\end{align*}
Consequently, using the above two displays together with Young's inequality, we get
\begin{equation} 
\label{e.good.quads.really}
\bigl( \bigl |\shom_{m}^{-\nicefrac12}\shom_{k}\shom_{m}^{-\nicefrac12}-\Id\bigr | 
+
\bigl |\shom_{*,m}^{\nicefrac12}\shom_{*,k}^{-1}\shom_{*,m}^{\nicefrac12}-\Id\bigr |\bigr) 
\vee
\bigl |\shom_{*,m}^{-\nicefrac12} (\khom_k-\khom_m)  \shom_{m}^{-\nicefrac12} \bigr | 
\leq
d\bigl ( \hat{\Theta}_k - \hat{\Theta}_m \bigr )
 \,.
\end{equation}
Next, we observe that in the case~$\eta=0$ we have, by~\eqref{e.Ak0.formula}, 
\begin{equation*}
\mathbf{G}_0^t\bfAhom_m \mathbf{G}_0 
=
\begin{pmatrix} 
\shom_m 
& 0
\\ 0
& \shom_{*,m}^{-1} 
\end{pmatrix}
\ \  \mbox{and} \ \
\mathbf{G}_0^t\bfAhom_k\mathbf{G}_0 
=
\begin{pmatrix} 
\shom_k + (\khom_{k}-\khom_m)^t\shom_{*,k}^{-1}(\khom_{k}-\khom_m)
& -(\khom_{k}-\khom_m)^t\shom_{*,k}^{-1}
\\ -  \shom_{*,k}^{-1}(\khom_{k}-\khom_m)
& \shom_{*,k}^{-1} 
\end{pmatrix}
\,.
\end{equation*}
Using~\eqref{e.Ak.vs.A.diff},~$\bfAhom_k \geq \bfAhom_m$,~\eqref{e.ordering.hat.tilde.Theta},~\eqref{e.smallness.ass} and~\eqref{e.good.quads.really}, we obtain 
\begin{align*}
\bigl| \bfAhom_{m}^{-\nicefrac12}   \bfAhom_{k} \bfAhom_{m}^{-\nicefrac12}   - \Itwod \bigr|
&
=
\bigl| (\mathbf{G}_0^t\bfAhom_m\mathbf{G}_0)^{-\nicefrac12} (\mathbf{G}_0^t\bfAhom_k\mathbf{G}_0) 
(\mathbf{G}_0^t\bfAhom_m\mathbf{G}_0)^{-\nicefrac12}  - \Itwod \bigr|
\notag \\ &
\leq
2\bigl( \bigl |\shom_{m}^{-\nicefrac12}\shom_{k}\shom_{m}^{-\nicefrac12}-\Id\bigr | 
+ 
\bigl |\shom_{*,m}^{-\nicefrac12} (\khom_k-\khom_m)  \shom_{m}^{-\nicefrac12} \bigr |^2 \bigr)
+ 
2\bigl |\shom_{*,m}^{\nicefrac12}\shom_{*,k}^{-1}\shom_{*,m}^{\nicefrac12}-\Id\bigr |
\notag \\ &
\leq 
4d\bigl ( \hat{\Theta}_k - \hat{\Theta}_m \bigr ) \,.
\end{align*}
This completes the proof of~\eqref{e.algebraic.add.error}.

\smallskip

\emph{Step 2.} 
We show~\eqref{e.variance.proto}.  By the triangle inequality, we have that
\begin{align}
\label{e.var.splitting}
\bigl | \,
\bfAhom_m^{-\nicefrac12} 
\bfA(\cus_n)  
\bfAhom_m^{-\nicefrac12} 
-\Itwod
\bigr |
&
\leq 
\Bigl|\bfAhom_m^{-\nicefrac12}  \Bigl (
\bfA(\cus_n)   - \avsum_{z} \bfA(z+\cus_k)
\Bigr )\bfAhom_m^{-\nicefrac12} 
\Bigr|
\notag \\ & \qquad
+
\Bigl| 
\bfAhom_m^{-\nicefrac12}
\avsum_{z} \bigl (\bfA(z+\cus_k) - \bfAhom_k \bigr)
\bfAhom_m^{-\nicefrac12} 
\Bigr|
\notag \\ & \qquad
+ \bigl | 
\bfAhom_m^{-\nicefrac12} 
\bfAhom_k
\bfAhom_m^{-\nicefrac12} 
-\Itwod
\bigr |
\,.
\end{align}
To lighten the notation, we will drop the index set in the sums over~$z$ in this step, which is in every instance over~$z \in 3^k \Lat \cap \! \! \cus_n$. 
The last term can be estimated using~\eqref{e.algebraic.add.error}
and the second last term on the right is the second term on the right in~\eqref{e.variance.proto}. 
Therefore, we focus on estimating the first term on the right side of~\eqref{e.var.splitting}.

\smallskip

Fix any matrix~$\bfB\in \R^{2d\times 2d}_{\mathrm{sym}}$ and consider the following string of inequalities: 
\begin{align}
\label{e.big.burrito}
\bfA(\cus_n)
&
\leq
\avsum_{z}
\bfA(z+\cus_k)
\notag \\ & =
\avsum_{z}
\bfA_*(z+\cus_k)
+
\avsum_{z}
(\bfA-\bfA_*) (z+\cus_k)
\notag \\ & 
\leq 
\biggl ( \avsum_{z} \bfA_*^{-1}(z+\cus_k)\biggr)^{\!-1}
+
\avsum_{z}
\bigl ( (\bfA_* - \bfB) \bfA_*^{-1} (\bfA_* - \bfB)  \bigr ) (z+\cus_k)
+
\avsum_{z}
(\bfA-\bfA_*) (z+\cus_k)
\notag \\ &
\leq 
\bfA_*(\cus_n) 
+
\avsum_{z}
\bigl ( (\bfA_* - \bfB) \bfA_*^{-1} (\bfA_* - \bfB)  \bigr )(z+\cus_k)
+
\avsum_{z}
(\bfA-\bfA_*) (z+\cus_k)
\notag \\ &
=
\bfA_*(\cus_n) 
+
\avsum_{z}
\bigl ( (\bfA(z+\cus_k) - \bfB) +\bfB (\bfA_*^{-1}(z+\cus_k) - \bfB^{-1})\bfB  \bigr )
\notag \\ &
\leq 
\bfA(\cus_n) 
+
\avsum_{z}
\bigl ( (\bfA(z+\cus_k) - \bfB) +\bfB (\bfA_*^{-1}(z+\cus_k) - \bfB^{-1})\bfB  \bigr )
\,.
\end{align}
The first and fourth lines in the display above are valid by the subadditivity of~$\bfA$ and~$\bfA_*^{-1}$, respectively. The second and fifth lines are just rearrangements. The sixth line is valid by the ordering~$\bfA_*\leq \bfA$. The third line is the key step which says roughly that the sample mean and harmonic mean are 
separated by, at most, the sample variance. 
The inequality we used here can be derived as follows.
Denoting the sample mean and the harmonic mean by
\begin{equation*} 
 \mathbf{M}
  \coloneqq  \avsum_{z} \bfA_{*}(z+\cus_k)
 \qand
 \mathbf{H} \coloneqq  \biggl(
\avsum_{z}  \bfA_{*}^{-1}(z+\cus_k) \biggr)^{\!-1} \,,
\end{equation*}
respectively, then we have the following identity, which can be checked by a direct computation: for every symmetric nonnegative matrix~$\tilde \bfB \in \R_{\mathrm{sym}}^{2d \times 2d}$, we have
\begin{equation*}  
\mathbf{M} 
= 
\mathbf{H} 
+ 
\avsum_{z} \bigl(  \bfA_*(z+\cus_k)  - \widetilde{\bfB}  \bigr)   \bfA_*^{-1}(z+\cus_k)  \bigl(  \bfA_*(z+\cus_k) - \widetilde{\bfB} \bigr) 
-
( \mathbf{H} - \widetilde{\bfB})\mathbf{H}^{-1} ( \mathbf{H} - \widetilde{\bfB}) 
\,.
\end{equation*}
Discarding the last nonnegative term yields the inequality in the third line of~\eqref{e.big.burrito}, above. 

\smallskip

Next, by comparing the first and last lines of~\eqref{e.big.burrito} and inserting~$\bfB=\bfAhom(\cus_k)$, we obtain that 
\begin{equation}
\label{e.burrito.wrap}
0\leq 
\avsum_{z} \bfA(z+\cus_k) - \bfA(\cus_n) 
\leq 
\avsum_{z}
\Bigl ( \bigl(\bfA - \bfAhom_k\bigr) +\bfAhom\bigl (\bfA_{*}^{-1} - \bfAhom^{-1}\bigr )\bfAhom \Bigr )(z+\cus_k)
\,.
\end{equation}
We next multiply the matrix inequality from the left and right by~$\bfAhom_m^{-\nicefrac12}$, 
and decompose the resulting second matrix on the right, by using
\begin{equation*} 
\bfA_*^{-1}(z+\cus_k) = \mathbf{R}  \bfA(z+\cus_k)\mathbf{R}\,, \quad \bfAhom_{*}^{-1}(\cus_k) = \mathbf{R}  \bfAhom(\cus_k) \mathbf{R}
\qand 
\bfAhom_{m}^{-\nicefrac12} = \mathbf{R}  \bfAhom_{*,m}^{\nicefrac12} \mathbf{R}\,,
\end{equation*}
as
\begin{align*} 
\lefteqn{
\bfAhom_m^{-\nf12} \Bigl(
\bigl ( \bfAhom\bigl (\bfA_{*}^{-1} - \bfAhom^{-1}\bigr )\bfAhom \bigr )(z+\cus_k)
\Bigr) \bfAhom_m^{-\nf12}
} \qquad &
\notag \\ &
=
 \mathbf{R}  \bfAhom_{*,m}^{\nf12} \bfAhom_{*,k}^{-1} \bfAhom_{m}^{\nf12} \Bigl( \bfAhom_{m}^{-\nf12}  (\bfA(z+\cus_k) - \bfAhom_k) \bfAhom_{m}^{-\nf12} 
 \Bigr)
\bfAhom_{m}^{\nf12} \bfAhom_{*,k}^{-1}  \bfAhom_{*,m}^{\nf12} \mathbf{R}
\notag \\ & \qquad 
+
\mathbf{R}  \bfAhom_{*,m}^{\nf12} \bfAhom_{*,k}^{-\nf 12}  \bigl(
\bfAhom_{*,k}^{-\nf 12}  \bfAhom_k \bfAhom_{*,k}^{-\nf 12} - \Itwod
\bigr)
\bfAhom_{*,k}^{-\nf 12}  \bfAhom_{*,m}^{\nf12} \mathbf{R}
\,.
\end{align*}
Therefore, using Lemma~\ref{l.bfE.bounds}, we obtain 
\begin{align*} 
\lefteqn{
\biggl| \avsum_{z} \bfAhom_m^{-\nicefrac12}\Bigl (\bfAhom\bigl (\bfA_*^{-1} - \bfAhom^{-1}\bigr )\bfAhom \Bigr )(z+\cus_k) \bfAhom_m^{-\nicefrac12}
\biggr|
} \qquad &
\notag \\ &
\leq
\bigl| \bfAhom_{*,m}^{\nicefrac12}  \bfAhom_{*,k}^{-1}\bfAhom_m^{\nicefrac12} \bigr|^2 \biggl| \avsum_{z}   \bfAhom_m^{-\nicefrac12} \bigl( \bfA(z+\cus_k) - \bfAhom(\cus_k) \bigr)\bfAhom_m^{-\nicefrac12} \biggr| 
+
6(\tilde{\Theta}_k-1)\bigl| \bfAhom_{*,m}^{\nicefrac12}  \bfAhom_{*,k}^{-1}\bfAhom_{*,m}^{\nicefrac12} \bigr|
\,.
\end{align*}
By~\eqref{e.grok},~\eqref{e.algebraic.add.error} and~\eqref{e.smallness.ass} we get 
\begin{align*}
\bigl| \bfAhom_{*,m}^{\nicefrac12}  \bfAhom_{*,k}^{-1}\bfAhom_m^{\nicefrac12} \bigr|^2
&  
= \bigl| 
\mathbf{R} \bfAhom_{*,m}^{-\nicefrac12}\bfAhom_m^{\nicefrac12} \bigl( \bfAhom_m^{-\nicefrac12} \bfAhom_k \bfAhom_m^{-\nicefrac12} \bigr)^2 \bfAhom_m^{\nicefrac12} \bfAhom_{*,m}^{-\nicefrac12} \mathbf{R} 
\bigr|
\notag \\ & 
\leq
\bigl| \bfAhom_{*,m}^{-\nf12}  \bfAhom_{m} \bfAhom_{*,m}^{-\nf12}  \bigr| \bigl| \bfAhom_{m}^{-\nicefrac12} \bfAhom_k \bfAhom_{m}^{-\nicefrac12} \bigr|^2 
\notag \\ &
\leq 
\bigl(1 + 6(\tilde\Theta_m-1) \bigr) \bigl( 1+ \bigl| \bfAhom_{m}^{-\nicefrac12} \bfAhom_k \bfAhom_{m}^{-\nicefrac12} - \Itwod \bigr| \bigr)^2
\leq 2
\end{align*}
and
\begin{align} 
\bigl| \bfAhom_{*,m}^{\nicefrac12}  \bfAhom_{*}^{-1}(\cus_k)\bfAhom_{*,m}^{\nicefrac12} \bigr| &
= 
\bigl| \mathbf{R} \bfAhom_{m}^{-\nicefrac12}  \bfAhom_k  \bfAhom_{m}^{-\nicefrac12} \mathbf{R}  \bigr|
\leq
\bigl( 1+ \bigl| \bfAhom_{m}^{-\nicefrac12} \bfAhom_k \bfAhom_{m}^{-\nicefrac12} - \Itwod \bigr| \bigr)
\leq 2
\,.
\notag
\end{align}
Therefore,~\eqref{e.ordering.hat.tilde.Theta},~\eqref{e.burrito.wrap} and the above three displays give us
\begin{equation*} 
\biggl| \avsum_{z} \bfAhom_m^{-\nicefrac12} \bigl( \bfA(z+\cus_k) - \bfA(\cus_n) \bigr)\bfAhom_m^{-\nicefrac12} \biggr| 
\leq 
3 
\biggl| \avsum_{z} \bfAhom_m^{-\nicefrac12} \bigl( \bfA(z+\cus_k) - \bfAhom(\cus_k) \bigr)\bfAhom_m^{-\nicefrac12} \biggr|
+ 12d(\hat{\Theta}_k-1)
 \,.
\end{equation*}
Combining this with~\eqref{e.burrito.wrap},~\eqref{e.algebraic.add.error} and~\eqref{e.var.splitting} yields~\eqref{e.variance.proto},  completing the proof.  
\end{proof}

\begin{lemma}
\label{l.Jagainsmallconstrast}
Assume that~$\Theta -1 \leq \sigma \leq  \nicefrac{1}{10}$. 
Then there exist~$\sigma_0(d)>0$ and~$C(d) <\infty$ such that, if~$\sigma \leq \sigma_0$ then, for every~$n \in\N$ satisfying 
\begin{equation}
\label{e.nn0cond}
n \geq 2n_0, \quad n_0  \coloneqq  \biggl\lceil \frac{50}{1-\gamma} \log \frac{CK_{\Psi_\S} \Pi}{1-\gamma}  \biggr\rceil
\,,
\end{equation}
we have the estimate
\begin{equation}
\label{e.Jagainsmallconstrast}
\hat{\Theta}_{n} - 1
\leq
C \!\!
\sum_{k=n_0}^n  3^{\frac12(k-n)} 
\Bigl(   
\E\Bigl[  \bigl| \bfE^{-\nicefrac12}  ( \bfA(\cus_k) - \bfAhom(\cus_k) ) \bfE^{-\nicefrac12}  \bigr|^2  \Bigl]
+
\bigl(\hat{\Theta}_{k} - \hat{\Theta}_n \bigr)
\Bigr)
+
C3^{-\frac12(1-\gamma) n} 
\,.
\end{equation}
\end{lemma}
\begin{proof}
We denote~$n_1\in\N$ by
\begin{equation*} 
n_1  \coloneqq  
\Bigl\lceil (1-\gamma)^{-1} +
4 \log K_{\Psi_{\S}} + 2 \log \Pi + k_0 + A \Bigr\rceil \,,
\end{equation*}
where~$k_0(d)$ is assumed to be so large that~$3^{k_0} \Lat \subseteq \Zd$ (see Section~\ref{ss.subadditivity}), and~$A(d)>0$ is sufficiently large and~$\sigma_0(d)$ sufficiently small so that, by~\eqref{e.bfAhom.by.E0}, for given~$c(d) \in (0,1)$, 
\begin{equation*}
\bfAhom (\cu_{n_1}) 
\leq (1+c(d)) \bfE 
\end{equation*}
and, consequently, by Lemma~\ref{l.tilt.to.Euc}, 
\begin{equation*}
\bfAhom(\cus_{2n_1}) 
\leq 
 (1+c(d)) \bfE \,.
\end{equation*}
Moreover, by Lemmas~\ref{l.bfE.bounds} and~\ref{l.bfE.bfA.bounds} and the assumption~$\sigma \leq \sigma_0(d)$, we deduce that, for every~$m \geq 2n_1$, 
\begin{equation}
\label{e.Enaught.vs.AHom.alg}
\bigl | \bfE^{-\nicefrac12}  \bfAhom(\cus_{m}) \bfE^{-\nicefrac12}  - \Itwod \bigr | 
\leq 
c(d)\,.
\end{equation}
Note that~$2n_1 \leq n_0$ with~$n_0$ defined in~\eqref{e.nn0cond}, provided the~$C$ in~\eqref{e.nn0cond} is sufficiently large. 

\smallskip

Fix~$n\in\N$ with~$n\geq 2n_0$. 
We recenter the coefficient field by subtracting the constant anti-symmetric matrix~$\frac12 (\khom - \khom^t)(\cus_n)$ from~$\a(\cdot)$ and relabel the new matrix field as~$\a(\cdot)$. This allows us to assume, without loss of generality, that~$\khom(\cus_n)$ is symmetric.\footnote{Note that this recentering may, as we have seen above in~\eqref{e.Pi.Pi.old}, change the value of~$\Pi$ by at most a factor of~$100d$ since~$\Theta \leq 2$. Because this factor of~$100d$ can be absorbed into the constant~$C$ in the statement of the proposition, we will ignore this issue in the argument.} 
We, therefore, have that, for any~$e\in\Rd$, 
\begin{align*}
\lefteqn{
\E \bigl[ J(\cus_n,\shom_{*}^{-\nicefrac12}(\cus_n)e,\shom_{*}^{\nicefrac12}(\cus_n)e)\bigr] 
} \qquad & 
\notag \\ & 
=
\frac12 e \cdot 
(\shom_*^{-\nicefrac12} \shom \shom_*^{-\nicefrac12} - \Id)(\cus_n) e 
+
\frac12 e \cdot \bigl( 2\shom_*^{-\nicefrac12}\khom \shom_*^{-\nicefrac12} + \shom_*^{-\nicefrac12}\khom \shom_*^{-1}\khom \shom_*^{-\nicefrac12} \bigr)(\cus_n) e
\notag \\ & 
\geq
\frac12 e \cdot 
(\shom_*^{-\nicefrac12} \shom \shom_*^{-\nicefrac12} - \Id )(\cus_n) e 
+
e \cdot \shom_*^{-\nicefrac12}\khom \shom_*^{-\nicefrac12} e
\end{align*}
and, likewise, 
\begin{equation*}
\E \bigl[ J^*(\cus_n,\shom_{*}^{-\nicefrac12}(\cus_n)e,\shom_{*}^{\nicefrac12}(\cus_n)e)\bigr] 
\geq
\frac12 e \cdot 
(\shom_*^{-\nicefrac12} \shom \shom_*^{-\nicefrac12} - \Id )(\cus_n) e 
-
e \cdot \shom_*^{-\nicefrac12}\khom \shom_*^{-\nicefrac12} e \,.
\end{equation*}
Consequently, we deduce that 
\begin{equation}
\label{e.Thetam.byJ.agh}
\hat{\Theta}_n -1 
\leq 
C 
\max_{|e|=1}
\Bigl( 
\E \bigl[ J(\cus_n,\shom_{*}^{-\nicefrac12}(\cus_n)e,\shom_{*}^{\nicefrac12}(\cus_n)e)\bigr] 
+
\E \bigl[ J^*(\cus_n,\shom_{*}^{-\nicefrac12}(\cus_n)e,\shom_{*}^{\nicefrac12}(\cus_n)e)\bigr]  \Bigr)
\,.
\end{equation}
We fix a unit direction~$e \in \R^d$ with~$|e|=1$ and set, for the remainder of the argument,  
\begin{equation*} 
p  \coloneqq  \shom_{*}^{-\nicefrac12}(\cus_n)e 
\qquad \mbox{and} \qquad 
q  \coloneqq  \shom_{*}^{\nicefrac12}(\cus_n)e
\,.
\end{equation*}
Having fixed~$(p,q)$, we will denote~$v_{m}  \coloneqq  v(\cdot,\cus_m,p,q)$ for short, for every~$m\in \N$.

\smallskip

Applying~\eqref{e.divcurl.conclusion.pre} with~$(P,Q)=(0,0)$ and~$\ep=1$,~\eqref{e.Jaas.matform} and~\eqref{e.algebraic.add.error}, we obtain
\begin{align}  
\label{e.divcurl.conclusion.pre.applied}
\E\bigl  [ J(\cus_n,p,q)\bigr]  
&
\leq 
C 3^{-n}\E \Biggl[\biggl[ 
\mathbf{M}_0^{\nicefrac12} 
\begin{pmatrix} 
\nabla v_n  \\ 
\a \nabla v_n  
\end{pmatrix} 
\biggr]_{\Besov{-\nf12}{2}{1}(\cus_{n})}^2
\Biggr]
+  100 \bigl( 
\E \bigl  [ J(\cus_{n-4} ,p,q)\bigr] - \E \bigl  [ J(\cus_{n} ,p,q)\bigr] \bigr)
\notag \\ & 
\leq
C 3^{-n}\E \Biggl[\biggl[ 
\mathbf{M}_0^{\nicefrac12} 
\begin{pmatrix} 
\nabla v_n  \\ 
\a \nabla v_n 
\end{pmatrix} 
\biggr]_{\Besov{-\nf12}{2}{1}(\cus_{n})}^2
\Biggr]
+
C (\hat{\Theta}_{n-4} - \hat{\Theta}_n ) 
\,.
\end{align}
To estimate the first term on the right side, we need to apply Lemma~\ref{l.weaknorms.moreproto}. 
To prepare, we make some preliminary estimates. 
First, by~\eqref{e.bfA.bounds.diag} and~\eqref{e.symm.k.quad.small} (applied for $\mathbf{E}_1 = \bfAhom(\cus_n)$), we have 
\begin{align}
\label{e.Am.pq.bound}
\biggl |
\bfAhom^{\nicefrac 12} (\cus_n)
\begin{pmatrix} 
-p \\ q
\end{pmatrix}
\biggr |^2 
&
\leq 
2 \bigl | (\shom_{*}^{-\nicefrac12} \bhom \shom_{*}^{-\nicefrac12}) (\cus_n)\bigr| 
\notag \\ &
\leq 
2 +  \bigl | \shom_{*}^{-\nicefrac12} \shom \shom_{*}^{-\nicefrac12} - \Id \bigr|(\cus_n)  + 2 \bigl | \shom_{*}^{-\nicefrac12} \khom \shom_{*}^{-\nicefrac12} \bigr|^2(\cus_n)
\leq 
3
\,.
\end{align}
We note next that, by~$\Theta -1 \leq \nicefrac1{10}$, 
\begin{equation*} 
\bigl| \mathbf{M}_0^{-\nicefrac12} \bfE\mathbf{M}_0^{-\nicefrac12}  \bigr| \leq 2
\end{equation*}
and, by~\eqref{e.Am.pq.bound} and~\eqref{e.Enaught.vs.AHom.alg}, 
\begin{equation*} 
\Bigl | 
\bfE^{\nicefrac 12}  
\Bigl(
\begin{matrix} 
-p \\ q
\end{matrix} 
\Bigr)
\Bigr|\leq 4\,.
\end{equation*}
We now square and take the expectation of~\eqref{e.weaknorms.moreproto}. Some of the appearing terms have already been analyzed in the proof of Lemma~\ref{l.weaknorms}, such as~\eqref{e.this.is.so.nice.pre} and~\eqref{e.this.is.so.nice.pre.second}, and they will suffice for our purposes and we will not repeat them here. 
Note that, while we make particular choices of the parameters~$p$ and~$q$ in that section, the estimates that we rely on here are valid in general (as we noted there). 
We also use~\eqref{e.Enaught.vs.AHom.alg}, which gives us~$\bigl| \bfE^{-\nf12}  ( \bfAhom(\cus_n) - \bfAhom(\cus_k) ) \bfE^{-\nf12} \bigr|\leq \nf1{40}$,
and thus, by the triangle inequality, 
\begin{align*} 
\E\Bigl[  \bigl| \bfE^{-\nicefrac12}  ( \bfA(\cus_k) - \bfA(\cus_n) ) \bfE^{-\nicefrac12}  \bigr|^2  \Bigl]
&
\leq 
3\E\Bigl[  \bigl|\bfE^{-\nicefrac12} ( \bfA(\cus_k) - \bfAhom(\cus_k) ) \bfE^{-\nicefrac12}  \bigr|^2
 \Bigl]
 \notag \\ & \qquad 
+ 
3\E\Bigl[  \bigl| \bfE^{-\nicefrac12}   ( \bfA(\cus_n) - \bfAhom(\cus_n) ) \bfE^{-\nicefrac12}  \bigr|^2
 \Bigl]
 \notag \\ & \qquad 
+
\frac3{40}\bigl| \bfE^{-\nicefrac12} ( \bfAhom(\cus_k) - \bfAhom(\cus_n) ) \bfE^{-\nicefrac12}  \bigr|
\,.
\end{align*}
Using this display, we argue as in~\eqref{e.weaknorms.proto.applied} to obtain~$C(d)<\infty$ such that, for every~$\ep \in (0,1)$, 
\begin{align} 
\label{e.weaknorms.proto.applied.again}
\lefteqn{
3^{-n} \E\Biggl[ \biggl[
\mathbf{M}_0^{\nicefrac12} 
\begin{pmatrix} 
\nabla v_n -   (\nabla v_n)_{\cus_n}  \\ 
\a \nabla v_n-  (\a\nabla v_n)_{\cus_n}  
\end{pmatrix} 
\biggr]_{\Besov{-\nf12}{2}{1}(\cus_{n})}^2
\Biggr]
} \quad  &
\notag \\ &
\leq 
\frac{C}{\ep} 
\sum_{k=n_0}^n  3^{(1-\ep)(k-n)}  
\Bigl( 
\E\Bigl[  \bigl| \bfE^{-\nicefrac12}  ( \bfA(\cus_k) - \bfAhom(\cus_k) ) \bfE^{-\nicefrac12}  \bigr|^2  \Bigl]
+ 
\bigl| \bfE^{-\nicefrac12}  ( \bfAhom(\cus_n) - \bfAhom(\cus_k) ) \bfE^{-\nicefrac12}  \bigr|
\Bigr)
\notag \\ &
\qquad 
+
\frac{C K_{\Psi_\S}^{6} \Pi^3}{(1-\gamma)^4} 3^{-3n}
+
\frac{C K_{\Psi_\S} \Pi }{(1-\gamma)^2} 3^{-(1-\gamma)(n-n_0)}
\,.
\end{align}
We take~$\ep = \nicefrac12$ in~\eqref{e.weaknorms.proto.applied.again}.
The last line of~\eqref{e.weaknorms.proto.applied.again} can be brutally estimated, using~\eqref{e.nn0cond}, by
\begin{equation*} 
\frac{K_{\Psi_\S}^{6} \Pi^3}{(1-\gamma)^4} 3^{-2n}
+
\frac{K_{\Psi_\S} \Pi }{(1-\gamma)^2} 3^{-(1-\gamma)(n-n_0)}
\leq 
3^{-\frac12(1-\gamma) n}
\,.
\end{equation*}
The second term on the right side of the first line can be estimated using~\eqref{e.algebraic.add.error} by
\begin{equation*}  
\sum_{k=n_0}^{n} 3^{(1-\ep)(k-n)}
\bigl|  \bfE^{-\nicefrac12}   ( \bfAhom(\cus_n) - \bfAhom(\cus_k) ) \bfE^{-\nicefrac12}  \bigr|
\leq 
C 
\sum_{k=n_0}^n 3^{\frac12(k-n)} \bigl(\hat{\Theta}_{k} - \hat{\Theta}_n \bigr)
\,.
\end{equation*}
Finally, we can remove the constant from the left side of~\eqref{e.weaknorms.proto.applied.again} using the choice of~$p$ and~$q$, which yield, in view of~\eqref{e.all.averages.entries.again},
\begin{equation*}
\begin{pmatrix} 
\shom_*^{\nicefrac12 } (\cus_n) \E [ (\nabla v_n)_{\cus_n}]  \\ 
\shom_*^{-\nicefrac12 } (\cus_n) \E[ (\a\nabla v_n)_{\cus_n} ] 
\end{pmatrix} 
=
\begin{pmatrix} 
(\shom_*^{-\nicefrac12 } \khom \shom_*^{-\nicefrac12} )(\cus_n)e  \\
- (\shom_*^{-\nicefrac12 }\khom\shom_*^{-1} \khom \shom_*^{-\nicefrac12 } + \shom_*^{-\nicefrac12 } \khom \shom_*^{-\nicefrac12} )(\cus_n)e
\end{pmatrix} 
\,.
\end{equation*}
It follows from the previous display,~\eqref{e.all.averages},~\eqref{e.symm.k.quad.small} and~\eqref{e.Am.pq.bound}  that 
\begin{align*}
\E\biggl[ \biggl|
\mathbf{M}_0^{\nicefrac12} \!
\begin{pmatrix} 
(\nabla v_n)_{\cus_n}  \\ 
(\a\nabla v_n)_{\cus_n}  
\end{pmatrix} 
\biggr|^2
\biggr]
&
\leq
2\E\biggl[ \biggl|
\mathbf{M}_0^{\nicefrac12} \!
\begin{pmatrix} 
(\nabla v_n)_{\cus_n}- \E [(\nabla v_n)_{\cus_n}] \\ 
(\a\nabla v_n)_{\cus_n} - \E [ (\a\nabla v_n)_{\cus_n} ]
\end{pmatrix} 
\biggr|^2
\biggr]
+
2\biggl|
\mathbf{M}_0^{\nicefrac12} \!
\begin{pmatrix} 
 \E [ (\nabla v_n)_{\cus_n} ]  \\ 
\E [  (\a\nabla v_n)_{\cus_n} ]
\end{pmatrix} 
\biggr|^2
\\ &
\leq 
C\E\Bigl[  \bigl|  \bfE^{-\nicefrac12} ( \bfA(\cus_n) - \bfAhom(\cus_n) )  \bfE^{-\nicefrac12}  \bigr|^2  \Bigl]
+
C\bigl( \hat{\Theta}_n -1 \bigr)^2
\,.
\end{align*}
Combining the above, we obtain 
\begin{align*}
\E\bigl  [ J(\cus_n,p,q)\bigr]  
&
\leq
C \!\!
\sum_{k=n_0}^n  3^{\frac12(k-n)} 
\Bigl(   
\E\Bigl[  \bigl|  \bfE^{-\nicefrac12}   ( \bfA(\cus_k) - \bfAhom(\cus_k) )  \bfE^{-\nicefrac12}  \bigr|^2  \Bigl]
+
\bigl(\hat{\Theta}_{k} - \hat{\Theta}_n \bigr)
\Bigr)
\\ & \qquad 
+
C\bigl( \hat{\Theta}_n -1 \bigr)^2
+
C3^{-\frac12(1-\gamma) n} \,.
\end{align*}
By symmetry, we obtain the same bound for~$J^*$ in place of~$J$. By~\eqref{e.Thetam.byJ.agh}, we obtain~\eqref{e.Jagainsmallconstrast} with an additional term of~$C\bigl( \hat{\Theta}_n -1 \bigr)^2$ on the right side. This term can, however, now be absorbed by the left side if we require~$\sigma_0$ to be sufficiently small. The proof is complete. 
\end{proof}

We are now ready to prove the proposition. 

\begin{proof}[Proof of Proposition~\ref{p.algebraic.exp}]
We combine the previous two lemmas and then iterate the result. 
Using Lemma~\ref{l.variance} and~\eqref{e.Enaught.vs.AHom.alg} for sufficiently small~$c(d)$, we have that, for every~$l \geq 2n_0$, 
\begin{align*}
\lefteqn{
\E \Bigl[ \bigl | 
\bfE^{-\nicefrac12} 
\bigl( \bfA(\cus_k)  
- \bfAhom(\cus_k) \bigr)  \bfE^{-\nicefrac12} 
\bigr | ^2 \Bigr]
} \qquad &
\notag \\ &
\leq 
C (\hat{\Theta}_l-1)^2
+  
C
\E \Biggl[ \biggl | \,
\bfE^{-\nicefrac12} 
\! \! \! \! \avsum_{z \in 3^l \Lat \cap \cus_k} \! \!
\bigl ( \bfA(z+\cus_l) - \bfAhom(\cus_l)\bigr ) \bfE^{-\nicefrac12} 
\biggr |^2 \Biggr]
\,.
\end{align*}
By~\eqref{e.bfA.CFS.adapted}, we have that, for every~$k,l \in\N$ with~$\beta k < l \leq k$,
\begin{equation*}
\E \Biggl[
\biggl |
\bfE^{-\nicefrac12} \! \! \! \! \! \! \avsum_{z \in 3^l \Lat \cap \cus_k} \! \! \! \!  \! 
 \bigl ( \bfA(z+\cus_l) - \bfAhom(\cus_l) \bigr ) \bfE^{-\nicefrac12} 
\biggr |^2
\Biggr ]
\leq 
 \frac{C\Pi^2 K_{\Psi_\S}^8}{(1-\gamma)^2} 3^{2(\gamma(k-l) - k)}
+
C \Pi K_{\Psi}^8 3^{-2 (\nu-\gamma)(k-l)} 
\,.
\end{equation*}
For each~$k \in \N$, we choose~$l = l_k$ defined by
\begin{equation*} 
l_k  \coloneqq  \biggl\lceil\frac{\nu - \gamma}{\kappa + \nu -\gamma} k \biggr\rceil + 1
\,.
\end{equation*}
By the definition of~$\kappa$ in~\eqref{e.kappa.def}, we observe that  
\begin{align*}
\kappa l_k \geq (\nu-\gamma)(k-l_k) \geq \kappa(l_k -2)  
\,,\quad 
l_k > \max\Bigl\{ \beta, \frac12\Bigr\}  k
\qquad \mbox{and} \qquad
2(k - \gamma (k-l_k)) 
\geq 
\kappa k - 4 
\,.
\end{align*}
We deduce that 
\begin{equation*}
\E \Bigl[ \bigl | 
\bfE^{-\nicefrac12} 
\bigl( \bfA(\cus_k)  
- \bfAhom(\cus_k) \bigr) \bfE^{-\nicefrac12} 
\bigr | ^2 \Bigr]
\leq 
C (\hat{\Theta}_{l_k}-1)^2
+
CH 3^{-\kappa k}\,,
\end{equation*}
where we set
\begin{equation*}
H \coloneqq  \frac{\Pi^2 \max\{ K_{\Psi_\S}, K_\Psi\} ^8}{(1-\gamma)^2} \,.
\end{equation*}
Inserting this estimate into the result of Lemma~\ref{l.Jagainsmallconstrast} yields, for every~$n\geq 2n_0$, 
\begin{align*}
\hat{\Theta}_{n} - 1
&
\leq
C \!\!
\sum_{k=2n_0}^n  3^{\frac12(k-n)} 
\Bigl(   
(\hat{\Theta}_{l_k}-1)^2
+
H 3^{-\kappa k}
+
\bigl(\hat{\Theta}_{k} - \hat{\Theta}_n \bigr)
\Bigr)
+
C3^{-\frac12(1-\gamma) n} 
\notag \\ & 
\leq 
C \!\!
\sum_{k=2n_0}^n  3^{\frac12(k-n)} 
\Bigl(   
(\hat{\Theta}_{l_k}-1)^2
+
\bigl(\hat{\Theta}_{k} - \hat{\Theta}_n \bigr)
\Bigr)
+
CH 3^{-\kappa n} 
\,.
\end{align*}
Iterating this inequality yields (see for instance~\cite[Lemma 4.8]{AK.Book}), for every~$n\geq 2n_0$, 
\begin{equation*}
\hat{\Theta}_{n} - 1
\leq 
C H 3^{-\kappa (n-2n_0)} 
\,.
\end{equation*}
To switch back to Euclidean cubes, we use~\eqref{l.tilt.to.Euc} to obtain, for every~$m,n\in\N$ with~$m \geq 2 ( n - n_0)$ and~$n \geq 2n_0$, 
\begin{equation*}
\Theta_m - 1 
\leq 
\hat{\Theta}_{n} - 1
+
C H 3^{-(m-n)} 
\leq 
C H 3^{-\kappa (n-2n_0)} 
\leq 
C H 3^{-\frac 12 \kappa (m-2n_0)}
\,. 
\end{equation*}
After taking~$m_*$ as in~\eqref{e.alpha.kappa.def} with~$C_{\eqref{e.alpha.kappa.def}}$ suitably large and relabeling~$\frac12\kappa$ as~$\kappa$, we obtain~\eqref{e.algebraic.Thetam.main}. 
\end{proof}

\subsection{Quenched convergence of the coarse-grained matrices}
\label{ss.main.algebraic.quenched.proof}

In this subsection we present some quenched estimates on the coarse-grained matrices which are consequences of Theorem~\ref{t.main.algebraic}, beginning with the proof of Corollary~\ref{c.main.algebraic.quenched}. 

\begin{proof}[{Proof of Corollary~\ref{c.main.algebraic.quenched}}]
Fix~$\rho \in (\gamma,1)$ and~$\delta>0$ and let~$m_*$ and~$\kappa$ be as in~\eqref{e.alpha.kappa.def}. 
Also define~$\gamma' \coloneqq \frac12(\gamma+\rho)$ and 
\begin{equation}
\label{e.cor.beta.constr}
\eta \coloneqq \frac14 \frac{(\rho-\gamma)\mu }{d+\mu}
\,.
\end{equation}
Denote
\begin{equation*}
\delta_n  \coloneqq \delta 3^{-\eta(n-m_*)} 
\end{equation*}
and, for each~$n\geq m_*$, let~$l_n \in\N$ be defined as in~\eqref{e.ell.naught.def} but with~$\delta_n$ in place of~$\delta$ and~$\gamma'$ in place of~$\rho$, that is, 
\begin{equation*}
l_n  \coloneqq  
\biggl\lceil \frac{1}{\gamma'-\gamma}
\Bigl(1  + \frac d{\mu}\Bigr) \bigl( 6 + \eta(n-m_*)+
\log_3 ( \delta^{-1} \Theta )
\bigr)  +\frac{6}{\mu} \bigl( 1 + \log K_{\Psi} \bigr) 
\biggr\rceil
\,.
\end{equation*}
Observe that 
\begin{equation*}
l_n \leq 
\underbrace{
\biggl\lceil \frac{1}{\gamma'-\gamma}
\Bigl(1  + \frac d{\mu}\Bigr) \bigl( 6 +
\log_3 ( \delta^{-1} \Theta )
\bigr)  +\frac{6}{\mu} \bigl( 1 + \log K_{\Psi} \bigr) 
+1
\biggr\rceil
}_{=:l_0}
+
\underbrace{\frac{2(d+\mu)\eta}{(\rho-\gamma)\mu} 
}_{=: \tau}
(n-m_*) \,. 
\end{equation*}
In view of the constraint~\eqref{e.cor.beta.constr}, we have~$\tau\leq \nf12$. By the definition of~$l_n$, we find that, for every~$N$,  
\begin{equation}
\label{e.tau.imps}
n 
\geq 
m_*
+
( 1 - \tau)^{-1} 
\bigl( l_0 + N \bigr) 
\implies
n - l_n - m_* \geq N \,.
\end{equation}
For every~$n \geq n_* \coloneqq  m_*
+( 1 - \tau )^{-1} 
\bigl( l_0 + 2\log ( \delta^{-1} K_\Psi) \bigr)$, we let~$\mathcal{R}_n$ be as in Lemma~\ref{l.renormalize.ellipticity}, again with~$\delta_n$ in place of~$\delta$. We obtain from the lemma that~\eqref{e.new.ellipticity} holds, which states that, for every~$m\geq n \geq n_*$, 
\begin{align*}
\lefteqn{ 
3^m \geq \max\{ \S , \mathcal{R}_n \} 
} \ \  &
\notag \\ &
\implies 
\bfA(z+\cu_k) 
\leq 
\bigl( 1 + 
\delta 3^{-\eta(n-m_*)} 3^{\gamma'(m-k)} \bigr)
\bfAhom(\cu_{n-l_n})
\,,
\quad
\forall k \in \Z \cap (-\infty,m]\,,
\, z \in 3^k\Zd \cap \cu_m\,.
\end{align*}
By Theorem~\ref{t.main.algebraic} and~\eqref{e.Thetam.controls.Ahom}, we have that~$n\geq n_*$ implies 
\begin{equation*}
\bfAhom(\cu_{n-l_n})
\leq 
\bigl( 1 + 6\cdot 3^{-\kappa(n-l_n-m_*)} \bigr) \bfAhom
\leq 
\bigl( 1 + \delta 3^{-\kappa(1-\tau) (n-n_*)} \bigr) \bfAhom
\,.
\end{equation*}
Combining the previous two displays and defining~$\theta \coloneqq \min\{ \eta, \kappa(1-\tau) \}\geq \min\{ \eta, \tfrac12\kappa \}$ yields
\begin{align*}
\lefteqn{ 
3^m \geq \max\{ \S , \mathcal{R}_n \} 
} \ \  &
\notag \\ &
\implies 
\bfA(z+\cu_k) 
\leq 
\bigl( 1 + 
4\delta 3^{-\theta(n-n_*)} 3^{\gamma'(m-k)} \bigr)
\bfAhom
\,,
\quad
\forall k \in \Z \cap (-\infty,m]\,,
\ z \in 3^k\Zd \cap \cu_m\,,
\notag \\ &
\implies 
\sum_{k=-\infty}^m
3^{- \rho(m-k)}
\max_{z \in 3^k \Zd \cap \cu_m} \Bigl| \bigl( \bfAhom^{-\nicefrac12}( \bfA(z+\cu_k) - \bfAhom) \bfAhom^{-\nicefrac12} \bigr)_+ \Bigr| 
\leq 
C\delta(\rho-\gamma)^{-1} 3^{-\theta(n-n_*)}
\,.
\end{align*}
This motivates the definition~$N(m) \coloneqq 
\min \bigl\{ k \in\N \,:\, \mathcal{R}_{m-k} \leq 3^m \bigr\}$.
Using~\eqref{e.new.scale}, we have
\begin{equation*}
\P \bigl[ N(m) \geq k \bigr] 
\leq 
\P \bigl[  \mathcal{R}_{m-k} \geq 3^m  \bigr] 
=
\P \bigl[  \mathcal{R}_{m-k}^\mu \geq 3^{\mu m}  \bigr] 
\leq 
\frac{1}
{\Psi( 3^{k\mu})}
\,.
\end{equation*}
Thus~$3^{\mu N(m)} \leq \O_{\Psi}( 1 )$.
We next define the quantity 
\begin{equation*}
F_n \coloneqq 
\sum_{m=n}^\infty 3^{\frac12\theta(m-n)} 
\sum_{k=-\infty}^m
3^{- \rho(m-k)}
\max_{z \in 3^k \Zd \cap \cu_m} \Bigl| \bigl( \bfAhom^{-\nicefrac12}( \bfA(z+\cu_k) - \bfAhom) \bfAhom^{-\nicefrac12} \bigr)_+ \Bigr| 
\indc_{\{ \S \leq 3^m \}}
\,.
\end{equation*}
We observe that 
\begin{align*}
F_n 
&
\leq 
\frac{C\delta}{\rho-\gamma} \sum_{m=n}^\infty 3^{\frac12\theta(m-n)} 
3^{-\theta(m - N(m) -n_*)}
\leq 
\frac{C\delta}{\rho-\gamma}
3^{-\theta(n-n_*)}
\sum_{m=n}^\infty 3^{-\frac12\theta (m-n)} 
3^{\theta N(m) }
\end{align*}
and therefore 
\begin{equation*}
F_n^{\nf \mu\theta} 
\leq 
\O_{\Psi} \left( (C\delta(\rho-\gamma)^{-1}\theta^{-1} )^{\nf \mu\theta} 
3^{-\mu(n-n_*)} \right)
\,.
\end{equation*}
Define the minimal scale~$\mathcal{R} 
\coloneqq 
\sup 
\bigl\{ 
3^n \,:\, 
F_n \geq \delta 
\bigr\}$  and observe that 
\begin{align*}
\indc_{\{ \mathcal{R} \geq 3^{n} \}} 
\leq
\sum_{k=n}^\infty 
\indc_{\{  F_k \geq \delta \}} 
&
\leq 
\O_{\Psi} \biggl( 
\sum_{k=n}^\infty 
(C(\rho-\gamma)^{-1}\theta^{-1} )^{\nf \mu\theta} 
3^{-\mu(k-n_*)} 
\biggr)
\notag \\ & 
\leq 
\O_{\Psi} \bigl( 
\mu^{-1} 
(C(\rho-\gamma)^{-1}\theta^{-1} )^{\nf \mu\theta} 
3^{-\mu(n-n_*)} 
\bigr)
\,.
\end{align*}
This implies~\eqref{e.cor.sec.four.R}.
Finally, we observe that~$3^m\geq \max\{ \S,\mathcal{R}\}$ implies that 
\begin{equation*}
\sum_{k=-\infty}^m
3^{- \rho(m-k)}
\max_{z \in 3^k \Zd \cap \cu_m} \Bigl| \bigl( \bfAhom^{-\nicefrac12}( \bfA(z+\cu_k) - \bfAhom) \bfAhom^{-\nicefrac12} \bigr)_+ \Bigr| 
\indc_{\{ \S \leq 3^m \}}
\leq 
\delta 3^{-\frac12\theta(m-n)}\,.
\end{equation*}
This completes the proof. 
\end{proof}

\section{Quantitative homogenization}
\label{s.homogenization}

Up to this point, we have focused on developing properties of the coarse-grained coefficients~$\bfA(U)$ and proving, under the assumptions~\ref{a.stationarity},~\ref{a.ellipticity.dagger} and~\ref{a.CFS}, quantitative estimates for their convergence to the homogenized matrix~$\bfAhom$ as the domain~$U$ grows. 
In this section the perspective changes: we show how the coarse-grained coefficients can be used to control general solutions.
In particular, we present a broad ``black box'' principle (Proposition~\ref{p.big.black.box}) which asserts that, provided~$\bfA(\cu)$ is quantitatively close to~$\bfAhom$ for an appropriate collection of subcubes~$\cu \subseteq U$, one obtains quantitative estimates on the difference between solutions of the Dirichlet problems for~$-\nabla \cdot \a\nabla$ and~$-\nabla \cdot \ahom\nabla$ in~$U$.
Using this black box in combination with the results obtained in previous sections (in particular, Corollary~\ref{c.main.algebraic.quenched}), we will obtain the second bullet in Theorem~\ref{t.main}. 
The proof of Proposition~\ref{p.big.black.box} is the content of Section~\ref{ss.blackbox}.
By combining this result with Corollary~\ref{c.main.algebraic.quenched}, we will establish the second bullet of Theorem~\ref{t.main}.

\smallskip

Parallel black box principles will be developed in Sections~\ref{ss.LSreg} and~\ref{ss.correctors}.
In these, the closeness of the coarse-grained matrices~$\bfA(\cu_n)$ to~$\bfAhom$ over a range of scales~$n$ leads, respectively, to large-scale regularity estimates and first-order corrector estimates.
Together with Corollary~\ref{c.main.algebraic.quenched}, these results establish the final two bullets of Theorem~\ref{t.main}.

\smallskip

We emphasize that the black box statements and their proofs are \emph{purely deterministic}---they do not rely on the probabilistic assumptions~\ref{a.stationarity},~\ref{a.ellipticity.dagger} and~\ref{a.CFS}, nor on the main results of the earlier sections.
The proofs of Theorem~\ref{t.main} and Theorem~\ref{t.HC.intro} are completed in Section~\ref{ss.proofs.main.results}.

\subsection{Deterministic black box for quantitative homogenization} 
\label{ss.blackbox}

We begin by introducing the following multiscale composite quantity, which is closely related to the coarse-grained ellipticity constants and quantifies the homogenization error. 
Recall that~$\css{s} = 1 - 3^{-s}$ and that~$\Omega(\cu_m)$ is defined in~\eqref{e.Omega.U.def}.

\begin{definition}[Homogenization error]
\label{d.mathcal.E}
For every~$s \in (0,1]$,~$q\in [1,\infty)$,~$m,n \in \Z$ with~$n\leq m$, coefficient field~$\a \in \Omega( \cu_m)$ and matrix~$\a_1 \in \R^{d\times d}_+$, we let~$\bfA_1$ denote the matrix defined in terms of~$\a_1$ by the usual formula~\eqref{e.bfA.def}, that is, 
\begin{equation}
\label{e.form.of.A.naught}
\bfA_1  \coloneqq 
\begin{pmatrix} 
\s_1 + \k_1^t\s_1^{-1}\k_1  
& -\k_1^t\s_1^{-1}
\\ - \s_1^{-1}\k_1 
& \s_1^{-1}
\end{pmatrix}
\quad \mbox{where} \quad 
\s_1 \coloneqq  \frac12(\a_1+\a_1^t)
\,, 
\quad
\k_1 \coloneqq  \frac12(\a_1-\a_1^t)
\,,
\end{equation}
and we define the multiscale composite quantity
\begin{equation}
\mathcal{E}_{s,q} (\cu_m ; \a, \a_1 ) 
\coloneqq
\biggl( 
\css{sq} 
\sum_{l=-\infty}^m 
3^{-sq(m-l)} \!\!
\max_{z \in 3^l\Zd \cap \cu_m}
\max_{|e| =1} \bigl( \bfJ\bigl(z {+} \cu_l,\bfA_0^{-\nf 12} e, \bfA_0^{\nf 12} e \,; \a\bigr) \bigr)^{\nf q2}
\biggr)^{\!\nf1q} 
\,.
\label{e.mathcal.E.def}
\end{equation}
We extend this definition to~$q=\infty$ and to translations~$y+\cu_m$ of~$\cu_m$ in the obvious way. 
\end{definition}

\begin{remark}
The random variable~$\mathcal{E}_{s,q}(\cu_m;\a,\a_1)$ quantifies the closeness of~$\bfA(\cu)$ to~$\bfA_1$ for all triadic subcubes~$\cu$ of~$\cu_m$. 
Indeed, by~\eqref{e.Jsplitting},~\eqref{e.bigAstar.def} and~$\bfA_1^{\nf 12} = \mathbf{R} \bfA_1^{-\nf 12} \mathbf{R}$, 
\begin{equation}
\bfJ\bigl( U ,\bfA_1^{-\nf 12} e, \bfA_1^{\nf 12} e \,; \a\bigr)
\leq 
\bigl| \bigl( \bfA_1^{-\nf12} \bigl( \bfA(U;\a) -  \bfA_1 \bigr)  \bfA_1^{-\nf12} \bigr)_+
\bigr| 
\label{e.J.t.bfA.Ahom}
\end{equation}
and a converse of this inequality (which we do not use in this section) can be obtained from~\eqref{e.diagonalset.nosymm.bfJ}.
Under the assumptions of Theorem~\ref{t.main.algebraic}, the estimate stated in Corollary~\ref{c.main.algebraic.quenched} gives us good control over the random variable defined in~\eqref{e.mathcal.E.def} when~$\bfA_1 = \bfAhom$. 
Select~$\rho \in (\gamma,1)$ and~$\delta>0$. Applying Corollary~\ref{c.main.algebraic.quenched} yields, in view of~\eqref{e.J.t.bfA.Ahom}, for every~$m\in \N$, and~$s \in ( \nf\rho2, \nf12)$, 
\begin{equation}
\label{e.impl.to.mathcal.E}
3^m\geq \max\bigl\{ \S , \mathcal{R}_{\delta,\rho}  \}
\quad \implies \quad
\mathcal{E}_{s,2} ( \cu_m;\a,\ahom )^2 
\leq 
\delta 
(\rho - 2s)^{-1} 
\biggl( \frac{3^m}{\max\{ \mathcal{S},\mathcal{R}\} } \biggr)^{\!-\theta}  
\,.
\end{equation}
\end{remark}

\begin{proposition}[Homogenization black box]
\label{p.big.black.box} 
Assume that~$s\in (0,\nf12)$,~$U \subseteq \cu_0$ is a Lipschitz domain,~$\a \in \Omega(\cu_0)$ and~$\a_1 \in \R^{d\times d}_{+}$ is a matrix satisfying, for some constants~$0<\lambda < \Lambda < \infty$,
\begin{equation}
\lambda|\xi|^2 \leq \xi \cdot \a_1 \xi
\quad \mbox{and} \quad 
\Lambda^{-1}|\xi|^2 \leq \xi \cdot \a_1^{-1} \xi
\,.
\label{e.a0.ue}
\end{equation}
Then there exists a constant~$C(U,s,\lambda,\Lambda,d)<\infty$ such that the following hold.
\begin{itemize}

\item \emph{Homogenization error estimate.}
For every~$h \in H^{1-s}(U)$ and~$u\in H^1_{\a}(U)$ satisfying 
\begin{equation*}
\left\{
\begin{aligned}
& -\nabla \cdot \a\nabla u = 0 & \mbox{in} & \ U \,, 
\\ 
& -\nabla \cdot \a_1\nabla h = 0 & \mbox{in} & \ U \,, 
\\ 
& 
u - h \in H^{1-s}_0(U)
\,,
\end{aligned}
\right.
\end{equation*}
we have the estimate
\begin{equation}
\|  \nabla u - \nabla h \|_{\Hminushat{-s} (U)}
+
\| \a \nabla u - \a_1\nabla h \|_{\Hminushat{-s} (U)}
\leq 
C 
\mathcal{E}_ {s,2} (\cu_0;\a,\a_1)
\|\s^{\nf12} \nabla u \|_{{L}^2(U)}
\,.
\label{e.bbb.homogenization}
\end{equation}

\item \emph{Attainability of~$H^{\nf12+s}(\partial U)$ boundary data.}
For every~$h \in W^{1,\infty}(U)\cap H^{1+s}(U)$, the unique solution~$u\in g+H^{1}_{\a,0}(U)$ of the Dirichlet problem 
\begin{equation}
\left\{
\begin{aligned} 
& -\nabla \cdot \a\nabla u = 0 & \mbox{in} & \ U \,, 
\\ 
& u = g & \mbox{on} & \ \partial U 
\,,
\end{aligned}
\right.
\label{e.bbb.Dirichlet}
\end{equation}
satisfies the estimate
\begin{equation}
\| \s^{\nf12} \nabla u \|_{\underline{L}^2(U)} 
\leq 
C 
\bigl( 
1+
\mathcal{E}_ {s,2} (\cu_0;\a,\a_1)
\bigr)
\| \nabla g \|_{H^s(U)}
\,.
\label{e.attainability.estimate}
\end{equation}

\end{itemize}
\end{proposition}

The reason we allow the constant~$C$ in Proposition~\ref{p.big.black.box}  to have implicit dependence on the ellipticity constants~$\lambda$ and~$\Lambda$ is because the lemma will be applied only in the case~$\lambda ,\Lambda \in [\nf12,2]$. Indeed, we will apply Proposition~\ref{p.big.black.box} after an affine change of variables into the ``adapted'' geometry of the homogenized matrix. 
The dependence of the constant~$C$ on the domain~$U$ enters only via the Lipschitz character of~$\partial U$. Likewise, the dependence of~$C$ on~$s$ enters only via a lower bound for~$s(1-2s)$ which can be made explicit by a detailed examination of the arguments, but in the proof of Theorem~\ref{t.main} we will anyway apply it with the specific choice~$s= \frac12(1+\gamma)$. These remarks also  hold for every statement in this subsection. 

\smallskip

The two statements of Proposition~\ref{p.big.black.box} can be used in combination with a density argument, to obtain, for every~$h \in H^{1+s}(U)$ satisfying~$\nabla \cdot \a_1 \nabla h = 0$ in~$U$, the existence of a solution~$u\in H^1_{\a}(U) \cap (h+H^{1-s}_0(U))$ of~\eqref{e.bbb.Dirichlet} which satisfies the estimate 
\begin{equation}
\|  \nabla u - \nabla h \|_{\Hminushat{-s} (U)}
+
\| \a \nabla u - \a_1\nabla h \|_{\Hminushat{-s} (U)}
\leq 
C 
\bigl( 
1+
\mathcal{E}_ {s,2} (\cu_0;\a,\a_1)
\bigr) 
\mathcal{E}_ {s,2} (\cu_0;\a,\a_1)
\| \nabla h \|_{H^s(U)}
\,.
\label{e.combine.bbb.statements}
\end{equation}
If, in addition, the boundary data belongs to~$W^{1,\infty} (\partial U)$, then the solution of~\eqref{e.bbb.Dirichlet} is unique.
In particular, the second bullet of Proposition~\ref{p.big.black.box} says that any element of~$H^{\nf12+s}(\partial U)$ can be realized as the trace of a solution~$u\in \A(U;\a)$, provided that the quantity~$\mathcal{E}_{s,2}(\cu_0;\a,\a_1)$ is finite (which, as will see below in Lemma~\ref{l.mathcal.E.to.Lambdas}, is equivalent to the finiteness of both~$\lambda_{s,2}^{-1}(\cu_0;\a)$ and~$\Lambda_{s,2}(\cu_0;\a)$). Therefore Proposition~\ref{p.big.black.box} says that, if~$\mathcal{E}_ {s,2} (\cu_0;\a,\a_1)$ is small, then (i) any finite-energy element of~$\mathcal{A}(U;\a)$ can be approximated by an~$\a_1$-harmonic function; and, conversely, (ii) any~$\a_1$-harmonic function in~$U$ with boundary data in~$H^{\nf12+s}(\partial U)$ can be approximated by an element of~$\mathcal{A}(U;\a)$ with the same boundary data. 

\smallskip

We begin the proof of Proposition~\ref{p.big.black.box} with an abstract duality argument which estimates the left side of~\eqref{e.bbb.homogenization} by a weak norm of the vector field~$(\a-\a_1)\nabla u$. 

\begin{lemma}
\label{l.homogenization.by.duality}
Let~$s \in (0,\nf12)$ and~$U$ be a bounded Lipschitz domain,~$\a \in \Omega(U)$ and~$\a_1 \in \R^{d\times d}_{+}$ be a matrix satisfying~\eqref{e.a0.ue} for some~$0<\lambda < \Lambda$.
Then there exists~$C(U,s,\lambda,\Lambda,d)<\infty$ such that, for every~$u,v\in H^{1-s}(U)$ satisfying \footnote{The first line of~\eqref{e.homogenization.duality.eq} is to be understood in the sense of distributions, that is, $\int_U \nabla \phi \cdot (\a \nabla u - \a_1 \nabla v ) = 0$ for every~$\phi \in C^\infty_c(U)$.}
\begin{equation}
\left\{
\begin{aligned}
& \nabla \cdot (\a \nabla u - \a_1 \nabla v ) = 0
\quad \mbox{in} \ U\,, 
\\
& 
u-v \in H_0^{1-s}(U)\,, 
\\ & 
\a \nabla u \in H^{-s}(U;\Rd)\,,
\end{aligned}
\right.
\label{e.homogenization.duality.eq}
\end{equation}
we have the estimate
\begin{equation}
\|  \nabla u - \nabla v \|_{\Hminushat{-s} (U)}
+
\| \a \nabla u - \a_1\nabla v \|_{\Hminushat{-s}  (U)}
\leq
C
\| (\a-\a_1) \nabla u  \|_{\Hminushat{-s}  (U)}
\,.
\label{e.duality.error.estimate}
\end{equation}
\end{lemma} 
\begin{proof}
We use a duality argument, starting from the identity
\begin{equation}
\|  \nabla u - \nabla v \|_{\Hminushat{-s} (U)}
=
\sup \biggl\{ 
\int_{U} \h \cdot (\nabla u  -\nabla v) 
\,:\, \h\in C^\infty(U ;\Rd)\,, \ 
\| \h \|_{H^{s} (U)} \leq 1 
\biggr\} 
\,.
\label{e.duality.identity}
\end{equation}
Fix~$\h \in C^\infty(\cu_m ;\Rd)$ satisfying~$\| \h \|_{{H}^{s} (U)}^{2} \leq 1$
and let~$w\in H^1(U )$ be the solution of the Dirichlet problem
\begin{equation*}
\left\{
\begin{aligned}
& -\nabla \cdot \a_1^t\nabla w = \nabla \cdot \h 
&  \mbox{in} & \ U \,, \\ 
& w = 0 &  \mbox{on} & \ \partial  U \,,
\end{aligned}
\right.
\end{equation*}
Since~$s < \nf12$, we may apply~\cite[Theorem 0.5(ii)]{JeKe} to obtain\footnote{If the domain~$U$ has a smooth boundary, then the estimate~\eqref{e.CZ.Lp.app.nablaw} is classical, has a much simpler proof, and there is no dependence of~$C$ on~$s$ (see for instance~\cite[Theorem 0.3]{JeKe}).} the existence of~$C(s,U,\lambda,\Lambda,d)<\infty$ such that 
\begin{equation}
\label{e.CZ.Lp.app.nablaw}
\| \nabla w \|_{{H}^{s}(U)} 
\leq 
C \| \h \|_{{H}^{s}(U)} \leq C 
\,.
\end{equation}
Therefore the equation for~$w$ can be tested with~$H^{1-s}_0(U)$ and thus in particular with~$u-v$. 
Note that~$w$ can be approximated in~$H^{1+s}(U)$ by elements of~$C^\infty_c(U)$, and thus by passing to limits and using also that~$\a\nabla u - \a_1 \nabla v \in H^{-s}(U)$, we can also test the equation in the first line of~\eqref{e.homogenization.duality.eq} with~$w$. 
We therefore obtain 
\begin{equation}
\label{e.duality.testing}
\int_{U}
\h \cdot (\nabla u  -\nabla v) 
=
- 
\int_{U}
\nabla w \cdot 
\a_1( \nabla u  - \nabla v) 
=
\int_{U} (\a - \a_1)\nabla u \cdot \nabla w
\,.
\end{equation}
We deduce that 
\begin{equation*}
\biggl| \int_{U} 
\h \cdot (\nabla u  -\nabla v)  \biggr| 
\leq 
\bigl[ (\a - \a_1)\nabla u  \bigr]_{H^{-s}(U)} 
\| \nabla w \|_{H^{s}(U)} 
\leq
C \bigl[ (\a - \a_1)\nabla u  \bigr]_{H^{-s}(U)} 
\,.
\end{equation*}
In view of~\eqref{e.duality.identity}, the proof of the bound for~$\|  \nabla u - \nabla v \|_{\Hminushat{-s} (U)}
$ is complete. To estimate the difference of the fluxes, we use the triangle inequality and the estimate for the gradients to obtain
\begin{align*}
\| \a \nabla u - \a_1\nabla v \|_{\Hminushat{-s} (U)}
&
\leq 
\| (\a-\a_1) \nabla u  \|_{\Hminushat{-s} (U)}
+
\| \a_1( \nabla u - \nabla v) \|_{\Hminushat{-s} (U)}
\notag \\ & 
\leq 
\| (\a-\a_1) \nabla u  \|_{\Hminushat{-s} (U)}
+
C \|  \nabla u - \nabla v \|_{\Hminushat{-s} (U)}
\notag \\ & 
\leq 
C \| (\a-\a_1) \nabla u  \|_{\Hminushat{-s} (U)}
\,.
\end{align*}
This completes the proof of~\eqref{e.duality.error.estimate}.
\end{proof}

We next show that weak norms of the vector field~$(\a-\a_1)\nabla u$ can be controlled by~$\mathcal{E}_ {s,2} (\cu_m;\a,\a_1)$ and the energy of the solution~$u$. We also give an extension of Lemma~\ref{l.crude.weaknorms} to domains which are not necessarily triadic cubes, obtaining in particular the embedding
\begin{equation}
U \subseteq \cu_0
\ \mbox{ and } \
\lambda_{s,2}(\cu_0;\a) > 0
\quad \implies \quad
\A(U;\a) \hookrightarrow H^{1-s}(U)
\,.
\label{e.general.embedding}
\end{equation} 

\begin{lemma}
\label{l.coarse.graining.operator}
Let~$s\in (0,\nf 12)$,~$U \subseteq \cu_0$ be a Lipschitz domain and~$\a \in \Omega(\cu_0)$.
There exists a constant~$C(U,s,d)$ such that, for every~$\a_1\in \R^{d\times d}_{+}$, if we denote~$\s_1 \coloneqq \frac12(\a_1+\a_1^t)$ and~$\k_1 \coloneqq \frac12(\a_1-\a_1^t)$, then, for every~$u\in\A(U; \a)$, 
\begin{equation}
\label{e.CG.average.flux}
\| \s_1^{\nf 12}\nabla u \|_{\Hminusuls{-s}(U)}
+
\| \s_1^{-\nf 12}(\a-\k_1)\nabla u \|_{\Hminusuls{-s}(U)}
\leq
C
\bigl(1+\mathcal{E}_ {s,2} (\cu_0;\a,\a_1)\bigr)
\|\s^{\nf12} \nabla u \|_{\underline{L}^2(U)}
\end{equation}
and
\begin{equation}
\label{e.CG.elliptic.operator}
\| \s_1^{-\nf 12} ( \a-\a_1) \nabla u \|_{\Hminusuls{-s}(U)}
\leq
C
\mathcal{E}_ {s,2} (\cu_0;\a,\a_1)
\|\s^{\nf12} \nabla u \|_{\underline{L}^2(U)}
\,.
\end{equation}
\end{lemma}
\begin{proof}
We may assume without loss of generality that~$\a_1$ is symmetric; otherwise we recenter the fields by subtracting the constant anti-symmetric matrix~$\k_1 \coloneqq \frac12(\a_1-\a_1^t)$ from both~$\a$ and~$\a_1$. Notice that~$\A(U; \a) = \A(U; \a-\k_1)$. Moreover, by~\eqref{e.your.face} and~\eqref{e.what.h0.does}, for every Lipschitz domain~$V$,
\begin{equation} 
\mathcal{E}_ {s,2} (\cu_0;\a,\a_1) 
=
\mathcal{E}_ {s,2} (\cu_0;\a-\k_1,\s_1)
\qand
\s_*(V;\a-\h_0) = \s_*(V;\a) 
\,.
 \label{e.center.mathcalE.and.sstar}
\end{equation}
For each~$n \in \Z$, we define a collection of boundary layer cubes of~$U$ of size~$3^{n}$ by
\begin{equation*} 
\mathcal{W}_n(U)  \coloneqq  \Bigl\{ z + \cu_{n} \, : \, z \in 3^{n} \Zd \cap U \,, z + \cu_{n+1} \subset U \,,  (z + \cu_{n+2}) \cap \partial U \neq \emptyset \Bigr\}\,.
\end{equation*}
By applying~\eqref{e.weak.norms.ordering.A} and H\"older's inequality, we find that, for every~$\g \in H^s(U\,; \Rd)$, 
\begin{align*} 
\biggl| \fint_{U} \g \cdot \s_1^{\nf 12} \nabla u  \biggr| & 
\leq 
\sum_{n \in \Z} \sum_{\cu \in \mathcal{W}_n(U)} 
\frac{|\cu_n|}{|U|} 
\biggl|   \fint_{\cu} \g \cdot  \s_1^{\nf 12} \nabla u  \biggr| 
\notag \\ &
\leq
C \sum_{n \in \Z} \sum_{\cu \in \mathcal{W}_n(U)}  \frac{|\cu_n|}{|U|}   [  \s_1^{\nf 12} \nabla u  ]_{\Besov{-s}{2}{2}(\cu)} \| \g \|_{\underline{B}^{s}_{2,2}(\cu)}
\notag \\ &
\leq
C \biggl( \sum_{n \in \Z}
\sum_{\cu \in \mathcal{W}_n(U)}  \frac{|\cu_n|}{|U|}   [ \s_1^{\nf 12} \nabla u  ]_{\Besov{-s}{2}{2}(\cu)}^2  \biggr)^{\! \nf12}
\biggl( \sum_{n \in \Z}
\sum_{\cu \in \mathcal{W}_n(U)}  \frac{|\cu_n|}{|U|}  \| \g \|_{\underline{B}^{s}_{2,2}(\cu)}^2  \biggr)^{\! \nf12}
\,.
\end{align*}
To estimate the first factor on the right side, we apply~\eqref{e.energymaps.nonsymm} and~\eqref{e.center.mathcalE.and.sstar} to get
\begin{align*} 
\lefteqn{
\sum_{n \in \Z}
\sum_{\cu \in \mathcal{W}_n(U)}  \frac{|\cu_n|}{|U|}   [ \s_1^{\nf 12} \nabla u  ]_{\Besov{-s}{2}{2}(\cu)}^2
} \quad  &
\notag \\ &
=
\sum_{n \in \Z}
\sum_{\cu \in \mathcal{W}_n(U)}  \frac{|\cu_n|}{|U|}  \sum_{k = -\infty}^n 
3^{2s k} 
\avsum_{z \in 3^k \Zd \cap \cu}  \bigl| \s_1^{\nf 12} ( \nabla u )_{z+\cu_k} \bigr|^2 
\notag \\ &
\leq
C
\sum_{k=-\infty}^0  
3^{2s k} 
\max_{z \in 3^k \Zd \cap U, z+\cu_k \subset U} \bigl| \s_1^{\nf 12} \s_*^{-1}(z+\cu_k) \s_1^{\nf12} \bigr|
\frac1{|U|}
\sum_{n \in \Z} 
\sum_{\cu \in \mathcal{W}_n(U)}  
\| \s^{\nf12} \nabla u \|_{L^2(\cu)}^2
\notag \\ &
\leq
C s^{-1} \bigl( 1+ \mathcal{E}_{s,2}(\cu_0, \a,\a_1  ) \bigr)^2 \| \s^{\nf12} \nabla u \|_{\underline{L}^2(U)}^2
\,.
\end{align*}
For the second factor, we use Lemma~\ref{l.Wsp.vs.Bspp} and Proposition~\ref{p.fractional.hardy} to obtain, for every~$\g \in H^s(U\,;\Rd)$,
\begin{align*} 
\lefteqn{
\sum_{n \in \Z} \sum_{\cu \in \mathcal{W}_n(U)}  \frac{|\cu_n|}{|U|} \| \g \|_{\underline{B}^{s}_{2,2}(\cu)}^{2} 
} \qquad &
\notag \\ &
\leq
\frac{C}{|U|} \sum_{n \in \Z} \sum_{\cu \in \mathcal{W}_n(U)}  3^{-2sn} \| \g \|_{L^2(\cu)}^{2} 
+
 \frac{C}{|U|}  \sum_{n \in \Z} \sum_{\cu \in \mathcal{W}_n(U)} 
\int_{\cu} \int_{\cu} \frac{|\g(x) - \g(y)|^2}{|x-y|^{d+2s}} \, dx \, dy
\notag \\ &
\leq
\frac{C}{1-2s} |(\g)_U|^2 
+
C \fint_{U}  \frac{|\g(x) - (\g)_U|^2}{\dist(x,\partial U)^{2s}} \, dx
+
C
\fint_{U} \int_{U} \frac{|\g(x) - \g(y)|^2}{|x-y|^{d+2s}} \, dx \, dy
\notag \\ &
\leq
\frac{C}{1-2s}  |(\g)_U|^2 
+ 
C(U,s,d) [\g]_{\underline{H}^s(U)}
\,.
\end{align*}
Combining the previous three displays yields the result for the gradient in~\eqref{e.CG.average.flux}.   To show~\eqref{e.CG.elliptic.operator}, as above, we have
\begin{align*} 
\lefteqn{
\biggl| \fint_{U} \g \cdot \s_1^{-\nf 12}( \a-\a_1) \nabla u  \biggr|
} \qquad  &
\notag \\ & 
\leq
C \biggl( \sum_{n \in \Z}
\sum_{\cu \in \mathcal{W}_n(U)}  \frac{|\cu_n|}{|U|}   \bigl[ \s_1^{-\nf 12}( \a-\a_1) \nabla u  \bigr]_{\Besov{-s}{2}{2}(\cu)}^2  \biggr)^{\! \nf12}
\biggl( \sum_{n \in \Z}
\sum_{\cu \in \mathcal{W}_n(U)}  \frac{|\cu_n|}{|U|}  \| \g \|_{\underline{B}^{s}_{2,2}(\cu)}^2  \biggr)^{\! \nf12}
\,.
\end{align*}
To estimate the first factor on the right side, we apply~\eqref{e.fluxmaps} to get
\begin{align*}
\lefteqn{
\sum_{n \in \Z}
\sum_{\cu \in \mathcal{W}_n(U)}  \frac{|\cu_n|}{|U|}   \bigl[ \s_{0}^{-\nf 12} ( \a-\a_1) \nabla u  \bigr]_{\Besov{-s}{2}{2}(\cu)}^2
} \quad  &
\notag \\ &
=
\sum_{n \in \Z}
\sum_{\cu \in \mathcal{W}_n(U)}  \frac{|\cu_n|}{|U|}  \sum_{k = -\infty}^n 
3^{2s k} 
\avsum_{z \in 3^k \Zd \cap \cu}  \bigl| \bigl( \s_{0}^{-\nf 12} ( \a-\a_1) \nabla u \bigr)_{z+\cu_k} \bigr|^2 
\notag \\ &
\leq
C 
\sum_{k=-\infty}^0 
3^{2s k} 
\max_{z \in 3^k \Zd \cap U, z+\cu_k \subset U}
\sup_{|e| \leq 1}
J(z+\cu_k, \s_{0}^{-\nf 12}e, \s_{0}^{\nf 12} e)
\frac1{|U|}
\sum_{n \in \Z} 
\sum_{\cu \in \mathcal{W}_n(U)}  
\| \s^{\nf12} \nabla u \|_{L^2(\cu)}^2
\notag \\ &
\leq
C s^{-1}   \mathcal{E}_{s,2}(\cu_0, \a,\a_1  )^2 \| \s^{\nf12} \nabla u \|_{\underline{L}^2(U)}^2
\,.
\end{align*}
Combining the previous three displays yields~\eqref{e.CG.elliptic.operator}.  The proof is complete. 
\end{proof}

The combination of Lemmas~\ref{l.homogenization.by.duality}  and~\ref{l.coarse.graining.operator} immediately yield the first statement of Proposition~\ref{p.big.black.box}. The second statement of the proposition on the energy estimate for~$H^{1+s}$ boundary data is the focus of the rest of the section. 

\smallskip

To control a general solution on a Lipschitz domain~$U$ using coarse-grained matrices, it is natural to discretize the problem by introducing a triangulation of~$U$ into simplices.
The gradients and fluxes of solutions of the homogenized equation can then be approximated by piecewise constant fields subordinate to the triangulation.
This enables a subadditivity argument: the energy of any solution~$u \in \A(U;\a)$ can be bounded by the (properly weighted) sum of the coarse-grained matrices over the simplexes, and hence by~$1+\mathcal{E}_{s,2}$, multiplied by the energy of the homogenized solution.
To formalize this argument, we draw\ on standard ideas from finite element theory.

\smallskip

We begin with a geometric lemma, which provides a triangulation of an arbitrary bounded domain~$U\subseteq\mathbb R^d$ into~$d$--dimensional simplexes. 
The partition is designed with affine approximation in mind, and in particular it is \emph{conforming}: 
whenever two neighboring simplexes meet, their intersection is a common face of both. 
It also satisfies a Whitney--type property: the diameter of each simplex is comparable to its distance from the boundary of~$U$, and so the simplexes become arbitrarily small near the boundary.

\begin{lemma}[Whitney--graded conforming simplicial mesh]
\label{l.whitney.graded.mesh}
Let~$U \subseteq \mathbb{R}^d$ be a bounded open set. 
There exists a constant~$C(d) > 0$, depending only on~$d$, 
and a countable family~$\mathcal{P}$ of closed~$d$--simplexes satisfying the following:

\begin{itemize}
\item \emph{Partitioning of $U$:} 
\begin{equation}
U = \bigcup_{\triangle \in \mathcal{P}} \triangle
\qquad \mbox{and} \qquad 
\intr(\triangle) \cap \intr (\triangle') =\emptyset\,, 
\quad \forall \triangle,\triangle' \in \mathcal{P}, \ \triangle\neq \triangle'
\,.
\label{e.partitioning}
\end{equation}

\item \emph{Conformity of the mesh:} for every~$\triangle,\triangle' \in \mathcal{P}$, the intersection 
$\triangle\cap \triangle'$ is either empty or a common subsimplex of both. 

\item \emph{Whitney--type grading:} 
for every $\triangle \in \mathcal{P}$,
\begin{equation}
C^{-1} \dist(\triangle,\partial U) 
\leq \diam(\triangle) 
\leq 
C \dist(\triangle,\partial U)
\,.
\label{e.Whitney.grading}
\end{equation}

\item \emph{Shape regularity:} 
for every~$\triangle \in \mathcal{P}$,
\begin{equation}
\diam(\triangle) 
\leq 
C  \rho(\triangle)
\,,
\label{e.shape.regularity}
\end{equation}
where $\rho(\triangle) \coloneqq \sup\{ r >0 \,:\, \exists x \in \triangle, \ B_r(x) \subseteq \triangle \}$ denotes the inradius of $\triangle$.

\item \emph{Balanced adjacency:} 
for every~$\triangle,\triangle'\in\mathcal{P}$ such that~$\triangle \cap \triangle' \neq \emptyset$, 
\begin{equation}
\frac13  \leq \frac{\diam(\triangle)}{\diam(\triangle')} \leq 3
\,.
\label{e.balanced.adjacency}
\end{equation}
\end{itemize}
\end{lemma}
\begin{proof}
Here, unlike in most of the rest of the paper, we dyadic cubes rather than triadic ones. We therefore denote~$Q_n \coloneqq [ -2^{n-1} , 2^{n-1} ]^d$ for each~$n\in\Z$. Note that~$Q_n$ is closed. 

\smallskip

\emph{Step 1.} Whitney decomposition into dyadic cubes.
There exists a countable family~$\mathcal{Q}$ of closed dyadic cubes of the form~$z+ Q_n \subseteq U$ with~$n\in\Z$, such that
\begin{equation*}
\intr(Q) \cap \intr(Q') = \emptyset\,, 
\quad 
\forall
Q,Q'\in \mathcal{Q}\,, \ Q \neq Q' \,,
\end{equation*}
\begin{equation*}
Q \cap Q' \neq \emptyset 
\quad \implies \quad 
\frac12 \diam (Q) \leq \diam(Q') \leq 2 \diam(Q)\,,
\end{equation*}
and 
\begin{equation}
C^{-1} 
\diam(Q)
\leq 
\dist(Q,\partial U)
\leq C \diam(Q)
\,.
\end{equation}
This is the classical Whitney covering. 

\smallskip

\emph{Step 2.} 
We obtain a simplicial partition~$\mathcal{S}$ by subdividing each cube in~$\mathcal{Q}$ into~$d!$ simplexes, according to the usual ``coordinate-sorting'' rule. 
Explicitly, for $z+Q_n\in\mathcal Q$ we set
\begin{equation*}
z+ \triangle_n^\pi 
\coloneqq  z + 2^n \Bigl\{ (x_1,\ldots,x_d) \in\Rd \,:\, -\tfrac12 \leq x_{\pi(1)} \leq x_{\pi(2)} \leq  \cdots \leq x_{\pi(d)} \leq\tfrac12  \Bigr\},
\end{equation*}
where~$\pi \in S_d$ and~$S_d$ denotes the set of permutations of~$\{ 1,\ldots , d\}$. We call these the \emph{dyadic simplexes}. Thus
\begin{equation*}
\mathcal{S} \coloneqq 
\bigl\{ 
z+ \triangle_n^\pi 
\,:\,
z+Q_n \in \mathcal{Q}, \ \pi \in S_d \bigr\}.
\end{equation*}
Each simplex in~$\mathcal{S}$ is shape-regular with constants depending only on~$d$. 
Moreover, if two simplexes of~$\mathcal{S}$ intersect, then either their diameters are equal, or else one has twice the size of the other. 
In the equal-size case, their intersection is a common subsimplex and their induced triangulations agree. 
The only possible nonconformity therefore arises when a simplex of diameter~$2^n$ meets neighbors of diameter~$2^{n-1}$: in this situation the triangulations do not coincide on their common boundary.

\smallskip

To remedy this, we perform local refinements of coarse simplexes of diameter~$2^n$ so that, on every face shared with finer neighbors of diameter~$2^{n-1}$, the induced triangulations match exactly. 
These refinements are carried out in such a way that the induced triangulation on any subsimplex which already agrees with its neighbors is left unchanged. 

\smallskip

\emph{Step 3.} We introduce the local geometric operation that we use to refine each coarse simplex. 
Let~$\triangle=\operatorname{conv}\{v_0,\dots,v_d\}$ be a~$d$--simplex and~$G$ be a subsimplex of~$\triangle$ of dimension~$k \in \{1,\ldots,d-1\}$, that is, $G=\operatorname{conv}\{v_{i_0},\dots,v_{i_k}\}$ for some subset~$\{ i_0,\ldots,i_k \} \subseteq 
\{ 0,\ldots, d\}$ of the indices. Enumerate the indices of~$\triangle$ which do not belong to~$G$ by
\begin{equation*}
\{ j_1,\ldots,j_{d-k}\}
= 
\{ 0,\ldots, d\} \setminus \{ i_0,\ldots,i_k \}
\,.
\end{equation*}
Suppose $x$ lies in the relative interior of $G$. 
The \emph{stellar subdivision of $\triangle$ at $(G,x)$} is obtained by dividing~$\triangle$ into the following set of~$k+1$ $d$--simplexes: 
\begin{equation*}
\operatorname{St}_{G,x}(\triangle)
\coloneqq
\bigl\{ 
\operatorname{conv}
\bigl( 
v_{j_1},\dots,v_{j_{d-k}}, 
x, v_{i_0},\dots, v_{i_{m-1}},v_{i_{m+1}},\ldots,  v_{i_k}
\bigr) 
\,:\, \quad m \in \{ 0,\ldots,k\}
\bigr\}
\,.
\end{equation*}
In other words, replace one of the vertices of~$G$ by~$x$, and repeat for each vertex of~$G$. Denote the set of~$k+1$ simplexes obtained in this way by~$\operatorname{St}_{G,x}(\triangle)$. 

\smallskip

\emph{Step 4.} 
We describe the algorithm for refining~$\mathcal{S}$ to ensure conformity. 

\smallskip

We call a $(d{-}1)$--dimensional subsimplex $F$ a \emph{coarse-fine face} if $F$ is a face of some $\triangle\in\mathcal S$ and, on the other side, abuts simplexes in $\mathcal S$ of half the diameter of~$\triangle$. 

\smallskip

On such an~$F$, we denote by~$\mathcal{V}(F)$ be the set of vertices of the fine triangulation of $F$ that do not belong to the coarse triangulation of $F$ (The ``hanging nodes'' as seen from the coarse side). 
Order the points of $\mathcal{V}(F) = \{ x_1,\ldots,x_N\}$ in order of increasing skeleton dimension: 
first those lying in the relative interior of edges of the coarse triangulation of $F$, then those lying in the interior of two dimensional subsimplexes, and so forth, ending with (the unique) point in the interior of $F$ itself. 
Within subsimplexes of the same dimension, any fixed ordering may be used.

For each point~$x_j\in \mathcal{V}(F)$, we let~$H\subseteq F$ be the unique subsimplex of the current coarse triangulation of~$F$ whose relative interior contains~$x_j$.
We apply the stellar subdivision at~$(H,x_j)$ simultaneously to all coarse~$d$--simplexes in the current partition that contain~$H$. Note that the finer neighboring simplexes already contain~$x_j$ as a vertex and are left unchanged.
This procedure modifies the current partition.  
We perform this procedure~$N$ times, first for~$x_1$, then for~$x_2$, and so on, until~$x_N$. The resulting induced partition of~$\triangle$ has the property that the induced triangulation of the face~$F$ coincides exactly with the original fine side triangulation.

\smallskip

We repeat this operation for all of the (countably many) coarse-fine interfaces in~$\mathcal{S}$. This eliminates all hanging nodes and results in a conforming partition which we denote by~$\mathcal{P}$. The refinements to~$\mathcal{S}$ are purely local and, because our construction operates from increasing order of skeletal dimension, the resulting partition obtained is independent of the order in which the coarse-fine faces are processed. 
The refinements act only along coarse-fine faces, leave faces without hanging nodes unchanged and enforce exact agreement with the fine-side triangulations on each processed face.

If~$\triangle\in \mathcal{S}$ is a coarse simplex, 
then at every stage of the refinement, each descendant~$K$ of~$\triangle$ contains a~$(d{-}1)$--simplex~$S$ of half the size of~$\triangle$, which is a common face of a~$d$-simplex~$\triangle'\in \mathcal{S}$ and~$\triangle''\in \mathcal{P}$ with~$\triangle'' \subseteq \triangle'$. Then
\begin{equation*}
\diam(K) 
\geq 
\diam(S) = 2^{n-1}\sqrt{d-1}
\geq 
\frac12 \cdot 2^{n-1} \sqrt{d} 
=
\frac12 \diam(\triangle')
\geq \frac12 \diam(\triangle'')
\end{equation*}
and 
\begin{equation*}
\diam(K) \leq \diam(\triangle) = 2^n \sqrt{d}
\leq 
3 \cdot 
2^{n-1} \sqrt{d-1} 
= 3 \diam(S) \leq
3 \diam(\triangle'') 
\,.
\end{equation*}
This confirms the balanced adjacency~\eqref{e.balanced.adjacency} condition. The Whitney grading condition~\eqref{e.Whitney.grading} is also inherited from~$\mathcal{S}$. The partitioning property~\eqref{e.partitioning} is likewise clear. 
Shape regularity follows from the fact that each element of~$\mathcal{P}$ is either one of the dyadic simplexes, or else contains a~$d$-simplex and is contained in another dyadic~$d$-simplex twice as large. 
\end{proof}

In the next lemma, we approximate an arbitrary gradient field in~$L^2_{\pot}(U)$ by a gradient field which is piecewise constant on the elements of~$\mathcal{P}$. 

\begin{lemma}[Potential projector]
\label{l.Whitney.potential.projector}
Let $s\in(0,1]$, $U\subseteq\mathbb R^d$ be a bounded Lipschitz domain, and $\mathcal P$ be the simplicial mesh of Lemma~\ref{l.whitney.graded.mesh}.
Let~$\omega(\triangle)$ denote the union of all~$\triangle'\in\mathcal{P}$ with~$\triangle\cap \triangle'\neq \emptyset$.
There exists a constant~$C(d)<\infty$ and a linear projection
\begin{equation}
S: L^2_{\pot}(U) \to \ L^2_{\pot}(U)
\end{equation}
such that, for every $\mathbf h\in L^2_{\pot}(U)$,
\begin{equation}
S\mathbf h \in \mathbf h + L^2_{\pot,0}(U)\,,
\label{e.S.projector.properties}
\end{equation}
and, for every~$\triangle\in \mathcal{P}$, 
\begin{equation}
\| S \h - (\h)_{\triangle} \|_{\underline{L}^2(\triangle)} 
\leq 
C 
\| \h - (\h)_{\triangle} \|_{\underline{L}^2(\omega(\triangle))}
\qand
(S\mathbf h)\big|_{\triangle}\ \text{is constant.}
\,.
\label{e.h.prime.size.estimate} 
\end{equation}
\end{lemma}
\begin{proof}
The argument is based on the classical Scott--Zhang interpolation scheme~\cite{ScottZhang} from finite element theory, 
which constructs local quasi-interpolants by averaging over neighboring elements.

\smallskip

\emph{Step 1.}
Let $\mathcal V$ denote the set of all vertices of the mesh $\mathcal P$.
For each vertex $a\in\mathcal V$, let $\varphi_a$ be the unique continuous function on $U$ which is affine on each simplex $\triangle\in\mathcal P$ and satisfies
\begin{equation*}
\varphi_a(a)=1, 
\qquad 
\varphi_a(b)=0 \quad \forall b\in\mathcal V \setminus \{ a \} \,.
\end{equation*}
The family $\{\varphi_a\}_{a\in\mathcal V}$ forms a partition of unity on $U$, and every continuous function that is affine on each simplex can be written uniquely as a linear combination of the $\varphi_a$.

\smallskip
 
For each vertex~$a \in \mathcal{V}$, fix once and for all an \emph{anchor element}~$\triangle_a\in\mathcal P$ that contains $a$. We may choose the simplex of maximal diameter among those containing~$a$, with lexicographical ordering used to break ties.
Fix~$a\in \mathcal{V}$. Denote the barycentric coordinates on~$\triangle_a$ by~$\{\lambda_b \}_{b\in\mathcal V\cap 
\triangle_a}$, that is, we let~$\lambda_b$ be the restriction of~$\varphi_b$ to~$\triangle_a$.  
Let $M$ be the symmetric~$(d{+}1)$-by-$(d{+}1)$ \emph{local mass matrix} with entries 
\begin{equation*}
M_{bc} = \fint_{\triangle_a}\lambda_b\lambda_c \,dx\,, \quad 
b,c \in \mathcal{V} \cap \triangle_a\,.
\end{equation*}
By the shape regularity property~\eqref{e.shape.regularity}, the matrix~$M$ is uniformly well-conditioned, that is, we have the estimate
\begin{equation}
|M| + |M^{-1}| \leq C
\,.
\label{e.potential.projector.M.cond}
\end{equation}
Define the \emph{local dual} $\psi_a \in \mathrm{span}\{\lambda_b \,:\,b\in\mathcal V\cap \triangle_a \}$ by the relations
\begin{equation*}
\fint_{\triangle_a} 
\psi_a \lambda_b \,dx
=
\indc_{\{ a = b \}} 
\,,\qquad 
b\in\mathcal V\cap \triangle_a
\,.
\end{equation*}
In other words,~$\psi_a$ is the $L^2(\triangle_a)$-dual of the local nodal basis, obtained by inverting $M$.

\smallskip

For $v\in H^1(U)$, define the node functionals and the global quasi-interpolants by
\begin{equation*}
N_a(v) 
\coloneqq
\fint_{\triangle_a} 
v\psi_a\,dx
\qquad \mbox{and} \qquad
I[v](x)
\coloneqq
\sum_{a \in \mathcal{V} }
N_a(v) \varphi_a(x)
\,.
\end{equation*}
It is clear that~$N_a(v)$ depends only on $v$ on~$\triangle_a$, and that~$I[\varphi_a]$ is supported on the vertex star of~$a$ (the union of simplexes containing~$a$).
By construction, if the restriction of~$v\in H^1(U)$ to each element of~$\triangle$ is affine, then~$Iv = v$. By mapping to a reference simplex and using~\eqref{e.potential.projector.M.cond}, we find that 
\begin{equation}
|N_a(v)| 
\leq 
C \| v\|_{\underline{L}^2(\triangle_a)} 
\qquad \mbox{and} \qquad 
\|\nabla I[v] \|_{\underline{L}^2(\triangle)}
\leq 
C
\| v - (v)_{\omega(\triangle)} \|_{\underline{H}^1(\omega(\triangle))}
\,.
\label{e.quasi.interpolant.bounds}
\end{equation}
We now define the operator~$S$ by
\begin{equation}
S\mathbf h \ \coloneqq\ \nabla I[u]
\,.
\end{equation}
It is clear that~$\h \mapsto S\h$ is linear,~$S\h \in L^2_{\pot}(U)$, and~$S\h$ is constant on each simplex of $\mathcal P$.
Moreover, by~\eqref{e.quasi.interpolant.bounds} and the Whitney grading of the mesh~\eqref{e.Whitney.grading}, we have, for every~$p\in\Rd$ and~$\triangle\in \mathcal{P}$, 
\begin{equation*}
\| S \h - p \|_{\underline{L}^2(\triangle)} 
\leq 
C 
\| \h - p \|_{\underline{L}^2(\omega(\triangle))}
\,.
\end{equation*}
and thus 
\begin{equation*}
\| S \h \|_{\underline{L}^2(U)} 
\leq 
C 
\| \h \|_{\underline{L}^2(U)}
\,.
\end{equation*}
It remains to show that~$S\h - \h$ belongs to~$H^1_0(U)$. In view of the previous display, it suffices to check this in the case that~$\h = \nabla u$ with~$u\in W^{2,\infty}(U)$.
In this case, 
let~$\hat{u}$ be the piecewise affine interpolant with respect to~$\mathcal{P}$, defined by~$\hat{u} = u$ at all vertices of~$\mathcal{P}$. We then have~$I[\hat{u}] = \hat{u}$ and it follows by~\eqref{e.quasi.interpolant.bounds} that, for~$U_\ep := \{ x \in U \, : \, \dist(x ,\partial U) > \ep \}$,
\begin{equation*}
\| I[u] - u \|_{L^{\infty}(U \setminus U_\ep )}
=
\| I[u-\hat{u}] - (u - \hat{u}) \|_{L^{\infty}(U \setminus U_\ep )}
\leq 
C \ep^2 
\| \nabla^2 u \|_{L^\infty(U)}   
\,.
\end{equation*}
This implies that~$I[u] - u$ vanishes on~$\partial U$ in the~$H^1$-trace sense, hence~$I[u]-u\in H^1_0(U)$, 
and thus~$S[\nabla u] - \nabla u = \nabla(I[u]-u)$ belongs to~$L^2_{\pot,0}(U)$. 
\end{proof}

In the next lemma, we give the divergence-free analogue of the previous lemma, by approximating, for any~$s\in (0,\nf12)$, an arbitrary divergence-free field in~$L^2_{\sol}(U) \cap H^s(U)$ by another divergence-free field in~$L^2_{\sol}(U) \cap H^s(U)$ which is piecewise constant on the elements of~$\mathcal{P}$. We extra bit of regularity represented by~$s$ is needed because the construction is based on looking at the flux of the field through the faces of each simplex, and for this we need the availability of certain trace theorems. 

\begin{lemma}[Solenoidal projector]
\label{l.Whitney.solenoidal.projectior}
Let~$s\in (0,\nf12)$,~$U\subseteq\mathbb R^d$ be a bounded Lipschitz domain and~$\mathcal P$ be the simplicial mesh of Lemma~\ref{l.whitney.graded.mesh}.
There exists a constant~$C(s,d)<\infty$ and a linear projection 
\begin{equation}
T: 
L^2_{\mathrm{sol}}(U) \cap H^s(U;\Rd)
\to
L^2_{\mathrm{sol}}(U) \cap H^s(U;\Rd)
\end{equation}
satisfying, for every~$\g\in L^2_{\mathrm{sol}}(U)\cap H^s(U;\Rd)$, 
\begin{equation}
T\g - \g \in L^2_{\mathrm{sol},0}(U)
\label{e.projector.boundary.data}
\end{equation}
and, for every~$\triangle \in \mathcal{P}$, 
\begin{equation}
\| T\g - \g \|_{\underline{H}^s(\triangle)}
\leq 
C \| \g - (\g)_{\triangle} 
\|_{\underline{H}^s(\triangle)}
\qquad \mbox{and} \qquad 
T\g \vert_{\triangle} \quad \mbox{is constant.}
\label{e.g.prime.size.estimate}
\end{equation}
\end{lemma}
\begin{proof}
The argument is based on Raviart--Thomas interpolation from the theory of finite element methods (see for instance~\cite{Brezzi}). 
We define~$T\g$ on each simplex~$\triangle \in \mathcal{P}$ by looking at the flux of~$\g$ through each face of the simplex. We then check that we can glue these together to obtain a global divergence-free field with the correct boundary data. 

\smallskip

We will use the following classical trace theorem: if $V\subseteq \Rd$ is a bounded Lipschitz domain with outward normal~$\mathbf{n}$ and~$p\in (1,\infty)$, then any field~$\h$ belonging to the space
\begin{equation*}
W^p(\mathrm{div};V)
\coloneqq
\bigl\{ \f \in L^p(V;\Rd)\,:\, \nabla \cdot \f \in L^p(V) \bigr\} \,, \quad 
\| \f \|_{W^p(\mathrm{div};V)} 
\coloneqq 
\bigl( 
\| \f \|_{L^p(V)}^p + \| \nabla \cdot \f \|_{L^p(V)}^p
\bigr)^{\nf1p} \,,
\end{equation*}
has a well-defined normal trace $\mathbf{h}\cdot \mathbf{n}\in W^{-\nf1p,p}(\partial V)$, and there exists~$C(p,V,d)<\infty$ such that
\begin{equation}
\| \mathbf{h}\cdot \mathbf{n} \|_{W^{-\nf1p,p}(\partial V)} 
\leq 
C 
\| \mathbf{h} \|_{W^p(\mathrm{div};V)} 
\,.
\label{e.normal.trace.thm}
\end{equation}
In particular, this holds for divergence-free fields $\mathbf{h}\in L^p_{\mathrm{sol}}(V)$. Observe that, by the fractional Sobolev embedding, any~$\g \in L^2_{\sol}(V) \cap H^s(V;\Rd)$ belongs to~$L^p_{\sol}(V)$ for~$p=\frac{2d}{d-2s} > 2$.

\smallskip

Let~$\g \in L^2_{\sol}(U) \cap H^s(U;\Rd)$. Fix~$\triangle\in\mathcal P$ and denote the~$(d{+}1)$ faces of~$\triangle$ by~$F_0,\ldots,F_d$ and the corresponding outward unit normals by~$\mathbf{n}_0,\ldots,\mathbf{n}_d$. 
\begin{equation}
b_i \coloneqq \int_{F_i} \g \cdot \mathbf n_i \, dS.
\end{equation}
This is well-defined, since the indicator functions of the faces belong to~$H^{s'}(\partial \triangle)$ for every~$s'\in (0,\nf12)$, and thus in particular~$W^{\nf1p,p'}(\partial \triangle)$ for~$p=\frac{2d}{d-2s}$ and we have the estimate
\begin{equation}
\diam(\triangle)^{\nf1p}  
\| \indc_{F_i} \|_{\underline{W}^{\nf1p,p'}(\partial \triangle)}
\leq C 
\,.
\label{e.indicator.of.faces}
\end{equation}
Let~$M$ be the~$d\times d$ matrix with rows equal to~$\mathbf{n}_1,\ldots,\mathbf{n}_d$. Note that we have omitted~$\mathbf{n}_0$ from the definition of~$M$. 
By the shape regularity of the simplicial partition~$\mathcal{P}$, the matrix~$M$ is invertible and we have 
\begin{equation}
\max \bigl\{ |M| , |M^{-1}| \bigr\} 
\leq 
C\,.
\label{e.M.mat.bound}
\end{equation}
Set~$\mathbf{d}\coloneqq ( b_{1}|F_{1}|^{-1} , \ldots, b_{d}|F_{d}|^{-1}) \in\Rd$ and observe that, by~\eqref{e.normal.trace.thm}, the fractional Sobolev inequality (as explained above) and~\eqref{e.indicator.of.faces}, we have, for~$p=\frac{2d}{d-2s}$, 
\begin{equation}
| \mathbf{d} | 
\leq 
C \| \g \cdot \mathbf{n} \|_{\underline{W}^{-\nf1p,p}(\partial \triangle)} 
\sum_{i=1}^d
\| \indc_{F_i} \|_{\underline{W}^{\nf1p,p'}(\partial \triangle)} 
\leq 
C \diam(\triangle)^s
\| \g \|_{\underline{H}^s(\triangle)}
\,.
\label{e.bound.dee}
\end{equation}
Define $(T\mathbf g)|_{\triangle}\coloneqq M^{-1} \mathbf{d}$. This construction is evidently linear in~$\g$, commutes with the addition of constants and satisfies, for each face $F_i$ with~$i\in \{1,\ldots,d\}$, 
\begin{equation*}
\int_{F_i}(T\mathbf g)\cdot \mathbf n_i
= b_i
= \int_{F_i}\mathbf g\cdot \mathbf n_i
\,.
\end{equation*}
To check that this relation is also satisfied for the index~$i=0$, we observe that the divergence-free condition gives, by Green's formula, 
\begin{equation*}
\sum_{i=0}^d 
\int_{F_i}\mathbf g\cdot \mathbf n_i
=
0
\qquad\mbox{and} \qquad
\sum_{i=0}^d \int_{F_i}(T\mathbf g)\vert_\triangle \cdot \mathbf n_i
=
0
\end{equation*}
and therefore
\begin{equation*}
\int_{F_0}(T\mathbf g)\cdot \mathbf{n}_0 = 
- \sum_{i=1}^d 
\int_{F_i}(T\mathbf g)\cdot \mathbf n_i
=  
- \sum_{i=1}^d
\int_{F_i}\mathbf g\cdot \mathbf n_i 
=
\int_{F_0}\mathbf g\cdot \mathbf{n}_0 \,. 
\end{equation*}
By~\eqref{e.M.mat.bound} and~\eqref{e.bound.dee}, we have
\begin{equation}
|T\g\vert_{\triangle} | 
\leq 
|M^{-1}| |\mathbf{d}| 
\leq 
C \diam(\triangle)^s
\| \g \|_{\underline{H}^s(\triangle)}
\,.
\label{e.c.estimate}
\end{equation}
If two simplexes share a face $F$ with unit normal $\mathbf n$ (with opposite orientations on the two sides), then the previous identities give
\begin{equation}
|F| ((T\mathbf g)|_{\triangle^+} - (T\mathbf g)|_{\triangle^-})
\cdot 
\mathbf n = 0
\,,
\end{equation}
so by the conformity of the partition~$\mathcal{P}$ we have that the normal component of~$T\mathbf g$ is continuous across every interior face. This implies that~$T\mathbf g\in L^2_{\mathrm{sol}}(U)$. Using that~$\g\mapsto T\g\vert_{\triangle}$ commutes with addition constant vectors, we obtain that
\begin{equation}
\| T\g - \g \|_{\underline{L}^2(\triangle)}
=
\| T( \g - (\g)_{\triangle})  - ( \g - (\g)_{\triangle})  \|_{\underline{L}^2(\triangle)}
\leq 
C \diam(\triangle)^s
[ \g ]_{\underline{H}^s(\triangle)}
\,.
\label{e.c.estimate.Hs.two}
\end{equation}
This implies the estimate in~\eqref{e.g.prime.size.estimate} since~$T\g$ is constant on each~$\triangle \in \mathcal{P}$. 
Moreover, we have the estimate
\begin{equation}
\| T \g \|_{\underline{H}^s(U)}^2 
\leq 
\frac{C}{s \wedge (1-2s)} \| \g \|_{\underline{H}^s(U)}^2 
\,.
\label{e.T.bounded.in.Hs}
\end{equation}
We will delay the demonstration of~\eqref{e.T.bounded.in.Hs} this until the end of the proof. 
It is clear that~$T$ is the identity operator on the closed subspace 
\begin{equation*}
G\coloneqq 
\bigl\{ \g \in L^2_{\mathrm{sol}}(U)\cap H^s(U;\Rd)
\,:\, \forall \triangle\in\mathcal{P}\,, \ \g\vert_{\triangle} \ \mbox{is constant} \, \bigr\} 
\,.
\end{equation*}
We have therefore proved that~$T$ a projection operator from~$L^2_{\mathrm{sol}}(U)\cap H^s(U;\Rd)$ onto~$G$. 

\smallskip

We next check the validity of the boundary condition~\eqref{e.projector.boundary.data}. Since~$U$ is a Lipschitz domain, we have that~$L^2_{\sol}(U) \cap W^{1,\infty}(U;\Rd)$ is dense in~$L^2_{\sol}(U) \cap H^s(U;\Rd)$. We therefore only need to prove~\eqref{e.projector.boundary.data} in the case that~$\g \in L^2_{\sol}(U) \cap W^{1,\infty}(U;\Rd)$. 
For such~$\g$, we obtain from~\eqref{e.c.estimate.Hs.two} and the fact that~$(T\g)$ is constant on each~$\triangle\in\mathcal{P}$ that 
\begin{align*}
\| T\g - \g \|_{L^\infty(\triangle)} 
&
\leq 
\| \g - (\g)_{\triangle} \|_{L^\infty(\triangle)} 
+
\| T\g - (\g)_{\triangle} \|_{L^\infty(\triangle)} 
\notag \\ & 
=
\| \g - (\g)_{\triangle} \|_{L^\infty(\triangle)} 
+
\| T\g - (\g)_{\triangle} \|_{\underline{L}^2(\triangle)} 
\notag \\ & 
\leq 
C\| \g - (\g)_{\triangle} \|_{L^\infty(\triangle)} 
\leq 
C\| \nabla \g \|_{L^\infty(U)} 
\diam(\triangle)\,.
\end{align*}
We deduce from the Whitney grading that, if we denote~$U_\ep\coloneqq \{ x\in U\,:\, \dist(x,\partial U) > \ep\}$, then 
\begin{equation*}
\| T\g - \g \|_{L^\infty(U \setminus U_\ep)} 
\leq 
C\ep \| \nabla \g \|_{L^\infty(U)} \,.
\end{equation*}
Select~$\eta_\ep \in C^{\infty}_c(U)$ with~$\indc_{U_\ep} \leq \eta_\ep \leq \indc_{U}$ and~$\|\nabla \eta_\ep\|_{L^\infty(U)} \leq C\ep^{-1}$. 
Denote 
\begin{equation*}
\f_\ep \coloneqq 
\eta_\ep (T\g - \g) 
\end{equation*}
and observe that~$h_\ep\coloneqq \nabla \cdot \f_\ep = \nabla \eta_\ep (T\g-\g)\in L^2(U)$ and~$(h_\ep)_U$. Let~$w_\ep$ be the solution of the Neumann problem 
\begin{equation*}
\left\{
\begin{aligned}
& -\Delta w_\ep = h_\ep & \mbox{in} & \ U \,, \\
& \mathbf{n} \cdot \nabla w_\ep = 0  & \mbox{on} & \ \partial U  \,. 
\end{aligned}
\right.
\end{equation*}
We have that 
\begin{equation*}
\| \nabla w_\ep \|_{{L}^2(U )}
\leq 
C \| h_\ep \|_{{L}^2(U)} 
\leq 
C
\| \nabla \eta_{\ep} \|_{L^2(U)} 
\| T\g - \g\|_{L^\infty(U\setminus U_\ep )} 
\leq 
C \ep^{\nf12} 
\| \nabla \g \|_{L^\infty(U)} 
\,.
\end{equation*}
Set~$\mathbf{k}_\ep \coloneqq \f_\ep + \nabla w_\ep \indc_{U\setminus U_\ep}$ and observe that~$\k_\ep$ belongs to~$L^2_{\sol,0}(U)$, and
\begin{equation*}
\| T\g-\g - \mathbf{k}_\ep \|_{\underline{L}^2(U)} 
\leq 
\| (T\g-\g )(1-\eta_\ep) \|_{\underline{L}^2(U)} 
+
\| \nabla w_\ep \indc_{U\setminus U_\ep}  \|_{\underline{L}^2(U )} 
\leq 
C \ep^{\nf12} 
\| \nabla \g \|_{L^\infty(U)} 
\,.
\end{equation*}

\smallskip

It remains to prove the claimed estimate~\eqref{e.T.bounded.in.Hs}.
Using the shape regularity~\eqref{e.shape.regularity} and the balanced adjacency~\eqref{e.balanced.adjacency} properties of~$\mathcal{P}$, we have that
\begin{align*} 
[ T \g ]_{H^s(U)}^2
&
= 
\sum_{\triangle, \triangle' \in \mathcal{P}} 
\int_{\triangle}\int_{\triangle'}
\frac{\bigl|| (T\g)\vert_{\triangle}\indc_{\triangle}(x) -(T\g)\vert_{\triangle'}\indc_{\triangle'}(y) \bigr||^2 }{|x-y|^{d+2s}} \, dx \, dy
\notag \\ &
\leq
\sum_{\triangle, \triangle' \in \mathcal{P}}
\bigl| (T\g)\vert_{\triangle}-(T\g)\vert_{\triangle'} \bigr|^2
\biggl( 
\frac{C}{1-2s} \frac{| \triangle| \indc_{\triangle \sim \triangle'}}{\diam(\triangle)^{2s}} 
+
\frac{| \triangle| |\triangle'|}{\dist(\triangle,\triangle')^{d+2s}} \indc_{\triangle \not\sim \triangle'} 
\biggr)
\notag \\ &
\leq
C \sum_{\triangle \in \mathcal{P}}
\frac{| \triangle| \bigl| (T\g)\vert_{\triangle}-(\g)_{\triangle} \bigr|^2}{\diam(\triangle)^{2s}}
\biggl( 
\frac{1}{1-2s} 
+
\sum_{\triangle' \in \mathcal{P},\triangle \not\sim \triangle'}
\frac{|\triangle'|\diam(\triangle)^{2s}}{\dist(\triangle,\triangle')^{d+2s}}
\biggr)
\notag \\ & \qquad  
+
C \sum_{\triangle, \triangle' \in \mathcal{P}}
| \triangle| \bigl| (\g)_{\triangle}-(\g)_{\triangle'} \bigr|^2
\biggl( 
\frac{1}{1-2s} \frac{\indc_{\triangle \sim \triangle'}}{\diam(\triangle)^{2s}} 
+
\frac{|\triangle'|\indc_{\triangle \not\sim \triangle'} }{\dist(\triangle,\triangle')^{d+2s}} 
\biggr)
 \,.
\end{align*}
Here, we denote by~$\triangle' \sim \triangle$ if~$\dist(\triangle , \triangle') = 0$. By balanced adjacency~\eqref{e.balanced.adjacency}, we have that 
\begin{equation*}
\sup_{x \in \triangle, y \in \triangle' }(\diam(\triangle) + \diam(\triangle') + |x-y|) \leq C \dist(\triangle,\triangle')
\qquad 
\mbox{if} \ 
\triangle' \not\sim \triangle
\,,
\end{equation*}
and hence
\begin{align*} 
\sum_{\triangle' \in \mathcal{P},\triangle \not\sim \triangle'}
\frac{|\triangle'|\diam(\triangle)^{2s}}{\dist(\triangle,\triangle')^{d+2s}}
&
\leq
C 
\sup_{x \in U}
\sup_{t>0} 
\int_{U} \frac{t^{2s}}{(t + |x-y|)^{d+2s}} \, dy
\leq
C 
\sup_{t>0} 
\int_{t}^\infty \frac{t^{2s}}{r^{1+2s}} \,dr 
\leq
C s^{-1}
\,.
\end{align*}
Using the estimatre~\eqref{e.c.estimate.Hs.two}, we therefore obtain
\begin{equation*} 
\sum_{\triangle \in \mathcal{P}}
\frac{| \triangle| \bigl| (T\g)\vert_{\triangle}-(\g)_{\triangle} \bigr|^2}{\diam(\triangle)^{2s}}
\biggl( 
\frac{1}{1-2s} 
+
\sum_{\triangle' \in \mathcal{P},\triangle \not\sim \triangle'}
\frac{|\triangle'|\diam(\triangle)^{2s}}{\dist(\triangle,\triangle')^{d+2s}}
\biggr)
\leq 
Cs^{-1}(1-2s)^{-1}
[ \g ]_{\underline{H}^s(\triangle)}
 \,.
\end{equation*}
By the fractional Hardy inequality (Proposition~\ref{p.fractional.hardy}), we deduce that
\begin{equation*} 
\sum_{\triangle, \triangle' \in \mathcal{P}}
 \bigl| (\g)_{\triangle}-(\g)_{\triangle'} \bigr|^2
\biggl( 
\frac{| \triangle| \indc_{\triangle \sim \triangle'}}{\diam(\triangle)^{2s}} 
+
\frac{| \triangle| |\triangle'|}{\dist(\triangle,\triangle')^{d+2s}} \indc_{\triangle \not\sim \triangle'} 
\biggr) 
\leq
Cs^{-1}(1-2s)^{-1}[\g]_{H^{s}(U)}^2
\,.
\end{equation*}
The mean~$(T\g)_U$ of~$T\g$ can be bounded using~\eqref{e.c.estimate.Hs.two}. The estimate~\eqref{e.T.bounded.in.Hs} is now a consequence of the above displays. The proof is complete.  
\end{proof}

In the next lemma, we show that~$\mathcal{E}_{s,q}$ is closely related to the ratio of the coarse-grained ellipticity constants and in particular we can bound the ellipticity constants~$\lambda_{s,q}^{-1}(\cu_0;\a)$ and~$\Lambda_{s,q}(\cu_0;\a)$ from above by~$C(1+\mathcal{E}_{s,q}(\cu_0,\a,\s_1)^2)$. 

\begin{lemma}
\label{l.mathcal.E.to.Lambdas}
For every~$q \in [1,\infty]$,~$s\in (0,\infty)$, coefficient field~$\a \in \Omega(\cu_0)$ and positive symmetric matrix~$\s_1 \in \R^{d\times d}_{\sym,+}$,
\begin{equation}
\label{e.bound.Es.by.Lambdas}
\frac12 \mathcal{E}_{s,q} 
(\cu_0;\a,\s_1)^2 
+
\indc_{\{ q=2\}}
\leq 
\max\bigl\{ 
|\s_1^{-1}| \Lambda_{s,q}(\cu_0;\a) 
\,,
|\s_1| \lambda_{s,q}^{-1} (\cu_0;\a) 
\bigr\}
\end{equation}
and
\begin{equation}
\label{e.bound.Lambdas.by.Es}
\max\bigl\{ 
|\s_1|^{-1}
\Lambda_{s,q}(\cu_0;\a) 
\,,
|\s_1^{-1}|^{-1} \lambda_{s,q}^{-1} (\cu_0;\a) 
\bigr\}
\leq 
1+
2 \mathcal{E}_{s,q} (\cu_0;\a,\s_1)^2
+
2\mathcal{E}_{s,q}(\cu_0;\a,\s_1)
\,.
\end{equation}
\end{lemma}
\begin{proof}
The statement in the case~$q = \infty$ follows by sending~$q \to \infty$ in~\eqref{e.bound.Lambdas.by.Es}, so we assume~$q \in [1,\infty)$. 
By~\eqref{e.Jsplitting}, for every~$e \in \R^{2d}$ with~$|e|=1$,~$\bfA_1 \in \R^{2d \times 2d}_{\sym,+}$ and Lipschitz domain~$U \subseteq \cu_0$, 
\begin{equation*}
\bfJ(U,\bfA_1^{-\nf12}e,\bfA_1^{\nf12} e ; \a)
=
\frac12 e \cdot \bfA_1^{-\nf12} \bfA(U;\a) \bfA_1^{-\nf12} e 
+
\frac12 e \cdot \bfA_1^{\nf12} \bfA_*^{-1} (U;\a) \bfA_1^{\nf12} e 
- 1
\,.
\end{equation*}
In the case that~$\bfA_1$ has the form of~\eqref{e.form.of.A.naught} with~$\k_1=0$, we obtain 
\begin{align}
\frac12
\bigl| 
\s_1^{-\nf12} \b(U;\a) \s_1^{-\nf12}
+
\s_1^{\nf12} \s_*^{-1} (U;\a)\s_1^{\nf12}
-2 \Id
\bigr|
&
\leq
\max_{|e|=1} \bfJ(U,\bfA_1^{-\nf12}e,\bfA_1^{\nf12} e ; \a)
\notag \\ & 
\leq 
\bigl| 
\s_1^{-\nf12} \b(U;\a) \s_1^{-\nf12} 
+
\s_1^{\nf12} \s_*^{-1} (U;\a)\s_1^{\nf12}
-2 \Id
\bigr|
\,.
\label{e.J.by.f}
\end{align}
Using also that~$\s_1^{-\nf12} \b(U;\a)\s_1^{-\nf12}  
+
\s_1^{\nf12}  \s_*^{-1} (U;\a)\s_1^{\nf12} 
-2 \Id \geq 0$, which implies 
\begin{align*}
\bigl| 
\s_1^{-\nf12} \b(U;\a) \s_1^{-\nf12} 
+
\s_1^{\nf12}  \s_*^{-1} (U;\a)\s_1^{\nf12} 
-2 \Id
\bigr|
&
=
\bigl| 
\s_1^{-\nf12}  \b(U;\a) \s_1^{-\nf12} 
+
\s_1^{\nf12}  \s_*^{-1} (U;\a)\s_1^{\nf12} 
\bigr|
-2
\notag \\ & 
\leq 
|\s_1^{-1} | | \b(U;\a) | +  |\s_1| |\s_*^{-1} (U;\a)| - 2 
\,, 
\end{align*}
we obtain by the triangle inequality and the definition~\eqref{e.mathcal.E.def} that
\begin{align*}
\mathcal{E}_{s,q} 
(\cu_0;\a,\s_1) ^2 
&
\leq
2 \max\bigl\{ 
|\s_1^{-1} | \Lambda_{s,q}(\cu_0;\a) 
\,,
|\s_1 | \lambda_{s,q}^{-1} (\cu_0;\a) 
\bigr\}
\,,
\end{align*}
which implies the first inequality of~\eqref{e.bound.Lambdas.by.Es}. 
In the case~$q = 2$, we can replace the left side in the above inequality by~$\mathcal{E}_{s,2} 
(\cu_0;\a,\s_1) ^2  + 2$. This implies~\eqref{e.bound.Es.by.Lambdas}. 

\smallskip

Turning to the second inequality, we use that 
\begin{align}
\label{e.xminusonetimesxminusonesquared.sstar}
\lefteqn{ 
\bigl( 
\s_1^{\nf12}  \s_*^{-1} (U;\a) \s_1^{\nf12}
- \Id 
\bigr)
\s_1^{-\nf12} \s_*(U;\a) \s_1^{-\nf12}
\bigl( 
\s_1^{\nf12}  \s_*^{-1} (U;\a) \s_1^{\nf12}
- \Id 
\bigr)
} \ \ & 
\notag \\ & 
\leq 
\s_1^{-\nf12} ( \b(U;\a) - \s_* (U;\a))\s_1^{-\nf12}
+
\bigl( 
\s_1^{\nf12}  \s_*^{-1} (U;\a) \s_1^{\nf12}
{-} \Id 
\bigr)
\s_1^{-\nf12} \s_*(U;\a) \s_1^{-\nf12}
\bigl( 
\s_1^{\nf12}  \s_*^{-1} (U;\a) \s_1^{\nf12}
{-} \Id 
\bigr)
\notag \\ & 
=
\s_1^{-\nf12} \b(U;\a) \s_1^{-\nf12}
+
\s_1^{\nf12} \s_*^{-1} (U;\a)\s_1^{\nf12}
-2 \Id
\,.
\end{align}
Similarly, 
\begin{multline}
\label{e.xminusonetimesxminusonesquared.b}
\bigl( 
\s_1^{-\nf12} \b(U;\a) \s_1^{-\nf12}
{-} \Id 
\bigr)
\s_1^{\nf12} \b^{-1}(U;\a) \s_1^{\nf12}
\bigl( 
\s_1^{-\nf12} \b(U;\a)\s_1^{-\nf12} 
{-} \Id 
\bigr)
\\
\leq 
\s_1^{-\nf12} \b(U;\a) \s_1^{-\nf12}
+
\s_1^{\nf12} \s_*^{-1} (U;\a) \s_1^{\nf12}
- 2 \Id
\,.
\end{multline}
An eigenvector of~$\s_1^{\nf12} \s_*^{-1} (U;\a)\s_1^{\nf12}$ with eigenvalue~$\mu$ is also an eigenvector of the left side of~\eqref{e.xminusonetimesxminusonesquared.sstar} with eigenvalue~$f(\mu)$, where we set~$f(x)  \coloneqq  x^{-1}(x-1)^2$. It follows that  
\begin{equation*}
f\bigl(  |\s_1^{\nf12}\s_*^{-1} (U;\a)\s_1^{\nf12} | \bigr) 
\leq
\bigl| 
\s_1^{-\nf12} \b(U;\a) \s_1^{-\nf12}
+
\s_1^{\nf12} \s_*^{-1} (U;\a)\s_1^{\nf12}
-2 \Id
\bigr| \,.
\end{equation*}
For every~$\delta,\ep>0$, we have that~$f(x) \leq \delta$ implies~$x-1 \leq \delta + \delta^{\nf12} \leq (1+\ep^{-1} )\delta + \frac14\ep$. In view of the first inequality in~\eqref{e.J.by.f}, we obtain 
\begin{equation*}
|\s_1^{\nf12}\s_*^{-1} (U;\a)\s_1^{\nf12} - \Id |   
\leq 
2 (1+\ep^{-1} ) \max_{|e|=1} \bfJ(U,\bfA_1^{-\nf12}e,\bfA_1^{\nf12} e ; \a)
+
\frac14 \ep\,.
\end{equation*}
This implies that 
\begin{equation*}
|\s_1^{-1}|^{-1} \lambda_{s,q}^{-1} (\cu_0;\a) - 1 
\leq
2 (1+\ep^{-1} ) \mathcal{E}_{s,q} 
(\cu_0;\a,\s_1)^2
+
\frac14 \ep
\,.
\end{equation*}
Optimizing in~$\ep$ yields the bound in~\eqref{e.bound.Lambdas.by.Es} for~$ \lambda_{s,q}^{-1} (\cu_0;\a)$.
The bound for~$\Lambda_{s,q}(\cu_0;\a)$ is obtained similarly, using~\eqref{e.xminusonetimesxminusonesquared.b} in place of~\eqref{e.xminusonetimesxminusonesquared.sstar}. This completes the proof of the lemma.  
\end{proof}

We are now ready to complete the proof of Proposition~\ref{p.big.black.box}. 

\begin{proof}[{Proof of Proposition~\ref{p.big.black.box}}]
The first statement of the proposition is an immediate consequence of Lemmas~\ref{l.homogenization.by.duality}  and~\ref{l.coarse.graining.operator}, as already announced. 

\smallskip

The proof of the second statement is based on Lemmas~\ref{l.Whitney.potential.projector},~\ref{l.Whitney.solenoidal.projectior} and a variational formulation of the Dirichlet problem for general non-self adjoint operators. Let~$g \in W^{1,\infty}(U) \cap H^{1+s}(U) \subseteq H^1_{\a}(U)$. 
As discussed in Section~\ref{ss.Sob}, the Lions-Lax-Milgram lemma gives the existence and uniqueness of the solution~$u\in g+H^1_{\a,0}(U)$ of the Dirichlet problem 
\begin{equation}
\left\{ 
\begin{aligned}
& -\nabla \cdot \a\nabla u = 0 & \mbox{in} & \ U \,, 
\\ 
& u = g & \mbox{on} & \ \partial U 
\,.
\end{aligned}
\right.
\label{e.bbb.Dirichlet.applied}
\end{equation}
Moreover, a variational principle states that the solution~$u$ of~\eqref{e.bbb.Dirichlet.applied} is characterized by the fact that~$(u,\a\nabla u)$ is the unique minimizer among pairs in~$(g+H^1_{\a,0}(U))\times L^2_{\s,\mathrm{sol}}(U)$ of the convex integral functional 
\begin{align*}
(w,\g) \mapsto 
\int_U 
\frac12 
\begin{pmatrix} 
\nabla w \\ \g
\end{pmatrix}
\cdot
\bfA
\begin{pmatrix} 
\nabla w \\ \g
\end{pmatrix}
\,dx
\,.
\end{align*}
A proof can be found for example in~\cite[Section 10.1]{AKMbook}, which was written in the uniformly elliptic setting but with an argument that is valid in our generality. 

\smallskip

We next extend the definitions of~$J(U,p,q\;\a)$ in~\eqref{e.J.def} and of~$\bfJ(U,P,Q)$ in~\eqref{e.bfJ} by defining, for every~$\mathbf{p},\mathbf{p}^* \in L^2_{\s,\pot}$ and~$\mathbf{q},\mathbf{q}^* \in L^2_{\s,\sol}$, 
\begin{equation}
J(U,\mathbf{p},\mathbf{q}; \a) 
\coloneqq
\sup_{u \in  \A(U;\a)}
\fint_{U} \biggl( - \frac12 \nabla u \cdot \s \nabla u - \mathbf{p} \cdot \a \nabla u + \mathbf{q} \cdot \nabla u \biggr) \,,
\label{e.J.general.def}
\end{equation}
as well as~$J^*(U,\mathbf{p},\mathbf{q}; \a) \coloneqq J(U,\mathbf{p},\mathbf{q}; \a^t)$ and
\begin{equation}
\label{e.bfJ.general}
\bfJ
\biggl(U, \begin{pmatrix} \mathbf{p}  \\ \mathbf{q} \end{pmatrix}, \begin{pmatrix} \mathbf{q}^* \\ \mathbf{p}^* \end{pmatrix} \biggr)
\coloneqq 
\frac12 
J(U,\mathbf{p}- \mathbf{p}^*,\mathbf{q}^*-\mathbf{q}; \a) 
+
\frac12 
J^*(U,\mathbf{p} + \mathbf{p}^*,\mathbf{q}+\mathbf{q}^*) 
\,.
\end{equation}
We let~$v(\cdot,U,\mathbf{p},\mathbf{q})$ denote the unique maximizer of~\eqref{e.J.general.def}, and~$v^*(\cdot,U,\mathbf{p},\mathbf{q})$ analogously for~$J^*$. 

\smallskip

Let~$\s_1 \coloneqq \frac12(\a_1 + \a_1^t)$ and~$h$ be the~$\s_1$-harmonic function in~$U$ with Dirichlet boundary data~$g$. 
Since~$s < \nf12$, by~\cite[Theorem 0.5(ii)]{JeKe} we have
\begin{equation}
\| \nabla h\|_{H^s(U)}
\leq 
C\| \nabla g\|_{H^s(U)}
\,.
\label{e.JeKe.again}
\end{equation}
Note that~$\a_1\nabla {h}\in L^2_{\sol}(U)$. We may test the minimization principle with
\begin{equation}
\left\{
\begin{aligned}
& 
w \coloneqq 
\frac12 ( v (\cdot,U,-\nabla h,\a_1 \nabla{h})
+ 
v^*(\cdot,U,-\nabla h, -\a_1 \nabla h) \bigr)\,,
\\ & 
\g \coloneqq 
\frac12 ( \a \nabla v (\cdot,U,-\nabla h,\a_1 \nabla {h})
- 
\a^t \nabla v^*(\cdot,U,-\nabla h, -\a_1 \nabla h) \bigr)\,. 
\end{aligned}
\right.
\label{e.nabla.w.and.g}
\end{equation}
A calculation gives
\begin{align*}
\fint_U 
\nabla u \cdot \a\nabla u 
=
\fint_U 
\frac12 
\begin{pmatrix} 
\nabla u \\ \a\nabla u 
\end{pmatrix}
\cdot
\bfA
\begin{pmatrix} 
\nabla u \\ \a\nabla u 
\end{pmatrix}
\,dx
\leq
\fint_U 
\frac12 
\begin{pmatrix} 
\nabla w \\ \g
\end{pmatrix}
\cdot
\bfA
\begin{pmatrix} 
\nabla w \\ \g
\end{pmatrix}
\,dx
=
\mathbf{J} 
\left(U, \begin{pmatrix} 
\nabla {h}\\ \a_1\nabla h
\end{pmatrix}
, 0 \,;\a
\right)
\,.
\end{align*}
The last equality in the previous display follows from the (purely algebraic) fact that~$\bfJ$ defined in~\eqref{e.bfJ.general} satisfies the variational identity~\eqref{e.bfJ.var} and its maximizer is precisely~$(\nabla w,\g) \in \mathcal{S}(U)$ defined in~\eqref{e.nabla.w.and.g}. 
Using~\eqref{e.bfJ.general} and that~$J$ is quadratic, we deduce that 
\begin{align}
\| \s^{\nf12} \nabla u \|_{\underline{L}^2(U)}^2
&
\leq 
J(U,\nabla h, 0 )
+
J(U,0, \a_1\nabla{h})
+ 
J^*(U,\nabla{h}, 0 ) 
+
J^*(U,0,\a_1 \nabla{h}) 
\,.
\label{e.energy.by.Js}
\end{align}
To estimate the right side of~\eqref{e.energy.by.Js}, we apply Lemmas~\ref{l.Whitney.potential.projector} and~\ref{l.Whitney.solenoidal.projectior} with~$\h = \nabla {h}$ and~$\g=\a_1\nabla {h}$, respectively. 
For the first term, we use~\eqref{e.S.projector.properties}, subadditivity,~\eqref{e.h.prime.size.estimate},~\eqref{e.J.mat},~\eqref{e.h.prime.size.estimate} and Proposition~\ref{p.fractional.hardy} to obtain 
\begin{align*}
J(U,\nabla {h}, 0) 
=
J(U,S \nabla {h}, 0) 
& \leq 
\sum_{\triangle \in \mathcal{P}} 
\frac{|\triangle|}{|U|}
J(U, (S\nabla {h})(\triangle),0)
\notag \\ & 
\leq 
\sum_{\triangle \in \mathcal{P}} 
\frac{|\triangle|}{|U|}
| \b(\triangle) | 
|(S\nabla {h})(\triangle)|^2 
\notag \\ & 
\leq 
C
\Lambda_s(\cu_0;\a) 
\sum_{\triangle \in \mathcal{P}} 
\frac{|\triangle|}{|U|}
\diam(\triangle)^{-2s} 
|(S\nabla {h})(\triangle)|^2 
\notag \\ & 
\leq 
C \Lambda_s(\cu_0;\a) 
\biggl( 
\int_U \dist(x,\partial U)^{-2s} | \nabla {h}(x)|^2\,dx
+
[ \nabla {h} ]_{H^s(U)}^2
\biggr) 
\notag \\ & 
\leq 
C \Lambda_s(\cu_0;\a) 
\| \nabla {h} \|_{H^s(U)}^2
\,.
\end{align*}
For the second term, we use~\eqref{e.projector.boundary.data},  subadditivity,~\eqref{e.g.prime.size.estimate},~\eqref{e.J.mat},~\eqref{e.g.prime.size.estimate} and Proposition~\ref{p.fractional.hardy} to get
\begin{align*}
J(U,0,\a_1\nabla {h}) 
=
J(U,0,T(\a_1\nabla {h})) 
& 
\leq 
\sum_{\triangle \in \mathcal{P}} 
\frac{|\triangle|}{|U|}
J(U, (T(\a_1\nabla {h}))(\triangle),0)
\notag \\ & 
\leq 
\sum_{\triangle \in \mathcal{P}} 
\frac{|\triangle|}{|U|}
| \s_*^{-1} (\triangle) | 
|(T(\a_1\nabla {h}))(\triangle)|^2 
\notag \\ & 
\leq 
C
\lambda_s^{-1} (\cu_0;\a) 
\sum_{\triangle \in \mathcal{P}} 
\frac{|\triangle|}{|U|}
\diam(\triangle)^{-2s} 
|(T(\a_1\nabla {h}))(\triangle)|^2 
\notag \\ & 
\leq 
C \lambda_s^{-1} (\cu_0;\a) 
|\a_1|^2 
\| \nabla {h} \|_{H^s(U)}^2\,.
\end{align*}
The estimates for~$J^*(U,\nabla {h}, 0 )$ and~$J^*(U,0,\a_1 \nabla{h})$ are obtained similarly.
By Lemma~\ref{l.mathcal.E.to.Lambdas}, 
\begin{equation*}
\bigl( 
\Lambda_s (\cu_0;\a) 
+
\lambda_s^{-1} (\cu_0;\a) |\a_1|^2 
\bigr)^{\nf12}
\leq 
C|\a_1|^{\nf12} 
\bigl( 
1+
\mathcal{E}_ {s,2} (\cu_0;\a,\a_1)
\bigr) \,. 
\end{equation*}
Combining the above displays and allowing constants to depend on~$|\a_1|$ yields~\eqref{e.attainability.estimate}.
\end{proof}

\subsection{Large-scale regularity}
\label{ss.LSreg}

In this subsection, we formulate a black box principle, following the argument of~\cite{AS}, which asserts that on scales where the homogenization error is small, solutions exhibit improved regularity.
In particular, on such scales, the energy density of the solution behaves like a uniformly bounded function, consistent with Lipschitz regularity of the solution. Since such an estimate is expected to fail at smaller scales in our setting, this is typically called a \emph{large-scale} Lipschitz estimate. 
We also prove a local version of a large-scale~$C^{1,\theta}$ estimate, following arguments of~\cite{GNO.reg,AKMbook}, by showing that solutions can be approximated to first order by finite-volume Dirichlet correctors.

\smallskip

Define, for every~$m,n \in \Z$ with~$n\leq m$,~$s\in (0,\nf12)$ and~$t \in [1,\infty]$, the ``good" event
\begin{equation} 
\label{e.G.max}
G_{n,m}^{s,t}(\delta, \a_1) 
 \coloneqq  \biggl\{ 
\a \in \Omega(\cu_m) 
\,:\,
\biggl( \sum_{k=n}^m  \bigl(\mathcal{E}_{s,2}(\cu_k;\a,\a_1)  \bigr)^t\biggr)^{\! \nicefrac1t}
\leq 
\delta
\biggr\}
\,.
\end{equation}
We also set
\begin{equation*}
G_{n,\infty}^{s,t}(\delta, \a_1) \coloneqq \bigcap_{m\geq n} G_{n,m}^{s,t}(\delta, \a_1)
\,.
\end{equation*}
For~$t = \infty$, the sum over~$k$ in~\eqref{e.G.max} is replaced by the maximum over~$k \in \Z \cap [n,m]$. 
For every~$m\in\N$ and~$\a\in\Omega(\cu_m)$, we define the \emph{finite-volume Dirichlet corrector}~$w(\cdot,\cu_m,e) \in \mathcal{A}(\cu_m;\a)$ as the solution of 
\begin{equation}
\left\{
\begin{aligned}
& -\nabla \cdot \a \nabla w(\cdot,\cu_m,e)
= 0 
&  \mbox{in} &  \ \cu_m\,, 
\\
& 
w(\cdot,\cu_m,e) = \ell_e 
&  \mbox{on} &  \ \partial \cu_m\,.
\end{aligned}
\right.
\label{e.Dirichlet.corrector}
\end{equation}
The boundary conditions means that~$w(\cdot,\cu_m,e) - \ell_e \ \in H_{\a,0}^1(\cu_m)$.
The well-posedness of~\eqref{e.Dirichlet.corrector} is a consequence of the Lions-Lax-Milgram, as discussed in Section~\ref{ss.Sob}.
We note that, by the embedding~\eqref{e.general.embedding} and the quantitative estimates of Lemma~\ref{l.coarse.graining.operator}, we have~$w(\cdot,\cu_m,e) - \ell_e \ \in H_{0}^{1-s}(\cu_m)$ provided that~$\mathcal{E}_{s,2}(\cu_m;\a,\a_1)<\infty$ for some~$\a_1$. 
 
\begin{proposition}
\label{p.big.black.regularity.box} 
Let~$s \in (0,\nf 12)$,~$m\in\N$,~$\a \in \Omega(\cu_m)$ and~$\a_1 \in \R^{d\times d}_{+}$ be a matrix satisfying, for some constants~$0<\lambda < \Lambda < \infty$,
\begin{equation}
\lambda|\xi|^2 \leq \xi \cdot \a_1 \xi
\quad \mbox{and} \quad 
\Lambda^{-1}|\xi|^2 \leq \xi \cdot \a_1^{-1} \xi
\,.
\label{e.a0.ue.Lipschitz}
\end{equation}
Then the following statements hold. 
\begin{itemize}
\item \emph{Large-scale H\"older and Lipschitz estimates}. There exist constants~$C<\infty$ and~$c\in(0,1)$, both depending only on~$(s,\lambda,\Lambda,d)$, such that, for every~$n \in \Z$ with~$n<m$ and~$u \in \mathcal{A}(\cu_m;\a)$, 
\begin{equation} 
\label{e.Es.small}
\delta \in (0,c] \implies 
\|\s^{\nicefrac12} \nabla u\|_{\underline{L}^2(\cu_{n})} \indc_{G_{n,m}^{s,\infty}(\delta,\a_1)}
\leq C\exp\bigl(C \delta (m-n) \bigr)\|\s^{\nicefrac12} \nabla u\|_{\underline{L}^2(\cu_{m})}
\end{equation}
and
\begin{equation} 
\label{e.sum.Es.small}
\|\s^{\nicefrac12} \nabla u\|_{\underline{L}^2(\cu_{n})}
\indc_{G_{n,m}^{s,1}(c,\a_1)}
\leq
C \|\s^{\nicefrac12} \nabla u\|_{\underline{L}^2(\cu_{m})}\,.
\end{equation}
\item
\emph{Large-scale~$C^{1,\eta}$-estimate}. For each~$\eta \in [\nf 12,1)$, there exists constants~$C<\infty$ and~$c\in(0,1)$, both depending only on~$(\eta,s,\lambda,\Lambda,d)$, such that,  for every~$n \in \Z$ with~$n<m$ and~$u \in \mathcal{A}(\cu_m;\a)$, there exists~$e \in \Rd$ such that
\begin{equation} 
\label{e.C.one.eta}
\bigl\|\s^{\nicefrac12} \nabla (u - w(\cdot,\cu_m,e) ) \bigr\|_{\underline{L}^2(\cu_{k})} \indc_{G_{n,m}^{s,\infty}(c,\a_1)}
\leq C3^{-\eta(m-k)} \|\s^{\nicefrac12} \nabla u\|_{\underline{L}^2(\cu_{m})}
\,.
\end{equation}
\end{itemize}
\end{proposition}

Before giving the proof of Proposition~\ref{p.big.black.regularity.box}, we collect some needed consequences of results in the previous section. 
By Lemmas~\ref{l.Poincare.largescale},~\ref{p.coarse.grained.Caccioppoli} and~\ref{l.mathcal.E.to.Lambdas}, there exists~$C(s,\lambda,\Lambda,d)<\infty$ such that, for every~$u \in \mathcal{A}(\cu_k;\a)$, 
\begin{equation} 
\label{e.CG.Cacc.Poincare}
\| \s^{\nf 12} \nabla u \|_{\underline{L}^2(\cu_{k-1})}
\indc_{G_{k,k}^{s,\infty}(1,\a_1)}
\leq C  3^{-k} \| u - (u)_{\cu_k} \|_{\underline{L}^2(\cu_{k})}
\indc_{G_{k,k}^{s,\infty}(1,\a_1)}
\leq
C \| \s^{\nf 12} \nabla u \|_{\underline{L}^2(\cu_{k})}
 \,.
\end{equation}
It is clear from the definition of the finite-volume Dirichlet correctors that~$e \mapsto w(\cdot,\cu_m,e)$ is linear.
Proposition~\ref{p.big.black.box} yields 
\begin{equation} 
\label{e.Dir.corr.energy}
\| \s^{\nf12} \nabla w(\cdot,\cu_m,e)  \|_{\underline{L}^2(\cu_m)} 
\leq 
C 
\bigl( 
1+
\mathcal{E}_ {s,2} (\cu_m;\a,\a_1)
\bigr) |e|
\end{equation}
and
\begin{multline} 
\label{e.Dir.corr.weak}
3^{-sm} \|  \nabla w(\cdot,\cu_m,e) - e \|_{\Hminushat{-s} (\cu_m)}
+
3^{-sm}  \| \a \nabla w(\cdot,\cu_m,e) - \a_1e \|_{\Hminushat{-s} (\cu_m)}
\\
\leq 
C 
\mathcal{E}_ {s,2} (\cu_m;\a,\a_1)
\bigl( 
1+
\mathcal{E}_ {s,2} (\cu_m;\a,\a_1)
\bigr) 
|e|
 \,.
\end{multline} 
Notice that the latter display implies, by the fractional Poincar\'e inequality that, for~$\delta \in (0,1]$, 
\begin{equation} 
\label{e.Dir.corr.Ltwo}
3^{-m}\| w(\cdot,\cu_m,e) - \ell_e \|_{\underline{L}^2(\cu_m)}
\leq C\mathcal{E}_ {s,2} (\cu_m;\a,\a_1)
\bigl( 
1+
\mathcal{E}_ {s,2} (\cu_m;\a,\a_1)
\bigr) |e|
  \,.
\end{equation}
Indeed, we have, for~$f \in C_c^\infty(U)$, that
\begin{equation*} 
\| f \|_{\underline{L}^2(\cu_m)}
\leq 
C\| \nabla f \|_{\underline{H}^{-1}(\cu_m)}
\leq
C 3^{(1-s)m}  \| \nabla f \|_{\underline{H}^{-s}(\cu_m)}
\,,
\end{equation*}   
and hence~\eqref{e.Dir.corr.Ltwo} follows by density of~$C_c^\infty(\cu_m)$ in~$H_0^{1-s}(\cu_m)$. 

We now present the proof of the first bullet in the statement of Proposition~\ref{p.big.black.regularity.box}. 

\begin{proof}[Proof of~\eqref{e.Es.small} and~\eqref{e.sum.Es.small}]
Fix~$n \in \Z$ with~$n<m$ and~$u \in \A(\cu_m)$. Denote,  
for every~$k \in \Z$ with~$k\leq m$, 
\begin{equation*}
E_k  \coloneqq   \inf_{\ell\; \mbox{\scriptsize{affine}}} 3^{-k}\|u - \ell \|_{\underline{L}^2(\cu_{k})}
\ \qand \
D_k  \coloneqq  \|\s^{\nicefrac12} \nabla u\|_{\underline{L}^2(\cu_{k})} \,.
\end{equation*}
Let~$\ell^{(k)}$ denote the affine function realizing the minimum in~$E_k$. Throughout the proof we assume that~$\a\in G_{n,m}^{\infty}(\delta,\a_1)$, which implies in particular that~$\mathcal{E}_{s,2}(\cu_k;\a,\a_1) \leq \delta$ for every~$k \in \Z \cap [n,m]$. 

\smallskip

For each~$k \in \Z \cap [n,m]$, using  Lemma~\ref{l.coarse.graining.operator}, we have that there exists~$h_k \in H^{1-s}(\cu_k)$ solving
\begin{equation}
\label{e.find.hk.for.u.in.cu.k}
\left\{
\begin{aligned}
& -\nabla \cdot \a_1\nabla h_k = 0 & \mbox{in} & \ \cu_k \,, 
\\ 
& 
u - h \in H^{1-s}_0(\cu_k)
\,,
\end{aligned}
\right.
\end{equation}
and, by Proposition~\ref{p.big.black.box}, 
\begin{equation}
\|  \nabla u - \nabla h_k \|_{\Hminushat{-s} (\cu_k)}
+
\| \a \nabla u - \a_0\nabla h_k \|_{\Hminushat{-s} (\cu_k)}
\leq 
C 
\mathcal{E}_{s,2}(\cu_k)
\|\s^{\nf12} \nabla u \|_{{L}^2(\cu_k)}
\,.
\label{e.bbb.homogenization.applied}
\end{equation}
In particular, 
\begin{equation*} 
3^{-k} \| u - h_k  \|_{\underline{L}^{2}(\cu_{k-1})} 
\leq
C \mathcal{E}_{s,2}(\cu_k)
\|\s^{\nicefrac12} \nabla u\|_{\underline{L}^2(\cu_{k})}
\,.
\end{equation*}
Here we suppress~$(\a,\a_1)$ from the notation for~$\mathcal{E}_{s,2}(\cu_k;\a,\a_1)$.

\smallskip

\emph{Step 1.} 
We show that there exists~$N(s,\lambda,\Lambda,d) \in \N$ and~$C(s,\lambda,\Lambda,d)<\infty$ such that, for every~$k \in \Z$ with~$n\leq k \leq m$, 
\begin{equation} 
\label{e.Lip.sum.this}
E_{k-N} 
\leq 
\frac18 E_{k}
+
C\mathcal{E}_{s,2}(\cu_k) D_k
\,.
\end{equation}
By the regularity of~$\a_1$-harmonic functions, there exists~$C(s,\lambda,\Lambda,d)<\infty$ such that
\begin{align} 
\label{e.regularity.applied}
E_{k-N}
& \leq 
 \inf_{\ell \mbox{ \scriptsize{affine}} }
3^{N-k}\| h_k  -\ell \|_{\underline{L}^{2}(\cu_{k-N})}  
+
 3^{N-k} \| u - h_k  \|_{\underline{L}^{2}(\cu_{k-N})} 
\notag \\ &
\leq 
C 3^{-N-k} \inf_{\ell \mbox{ \scriptsize{affine}} }
\| h_k  -\ell \|_{\underline{L}^{2}(\cu_{k-1})} 
+
3^{(\nicefrac d2 +1)N - k} \| u - h_k  \|_{\underline{L}^{2}(\cu_{k-1})} 
\notag \\ &
\leq
C 3^{-N}E_{k} + C  \Bigl(3^{-k-N} + 3^{(\nicefrac d2 +1)N-k} \Bigr)  \| u - h_k  \|_{\underline{L}^{2}(\cu_{k-1})} 
\notag \\ &
\leq
C  3^{-N}E_{k}  + C  3^{(\nicefrac d2 +1)N} \mathcal{E}_{s,2}(\cu_k)  D_k
\,.
\end{align}
We choose~$N \in \N$ so large that~$C_{\eqref{e.regularity.applied}} 3^{-N} \leq \nicefrac18$. Thus, the previous display implies~\eqref{e.Lip.sum.this}. 

\smallskip

\emph{Step 2.}
We next show that there exist constants~$C(s,\lambda,\Lambda,d)<\infty$ such that, for every~$h \in \Z$ with~$n \leq h\leq m$, 
\begin{equation} 
\label{e.Lip.iter.this}
D_{h}  \leq C D_m + C  \sum_{j = h+1}^{m} \mathcal{E}_{s,2}(\cu_j)  D_j    \,.
\end{equation}
To see this, we first sum over~$j$ in~\eqref{e.Lip.sum.this} to get
\begin{equation*}
\sum_{j = -1}^{\lfloor N^{-1} (m-h) \rfloor-1} E_{h + jN} 
\leq 
\frac12 \sum_{j = 0}^{\lfloor N^{-1} (m-h) \rfloor}  E_{h + jN}
+
C  \sum_{j = h}^{m} \mathcal{E}_{s,2}(\cu_j)  D_j   
\,.
\end{equation*}
Reorganizing and reabsorbing then leads to
\begin{equation*} 
\sum_{j=h}^{m} E_{j} \leq C E_{m} + C  \sum_{j = h+1}^{m} \mathcal{E}_{s,2}(\cu_j)  D_j     \,.
\end{equation*}
We then obtain, for every~$j \in \Z \cap [n,m] $, 
\begin{align} 
\label{e.affine.diff}
| \nabla \ell^{(j)} - \nabla \ell^{(m)} | 
& 
\leq 
C\sum_{j = n}^{m-1} 3^{-j} \| \ell^{(j+1)} - \ell^{(j)}\|_{\underline{L}^2(\cu_{j})}  
\leq
C\sum_{j = n}^{m} E_{j} 
\leq
C E_{m} + C  \sum_{j = h+1}^{m} \mathcal{E}_{s,2}(\cu_j)  D_j 
\,.
\end{align}
Moreover, by the Poincar\'e inequality in~\eqref{e.CG.Cacc.Poincare}, we get 
\begin{equation} 
\label{e.affine.end.point}
|\nabla \ell^{(m)}| 
\leq 
C 3^{-m}\|u - (u)_{\cu_m} \|_{\underline{L}^2(\cu_{m})}
+ CE_m
\leq 
C 3^{-m}\|u - (u)_{\cu_m} \|_{\underline{L}^2(\cu_{m})}
\leq
C D_m 
 \,.
\end{equation}
Consequently, by the coarse-grained Caccioppoli estimate in~\eqref{e.CG.Cacc.Poincare},
\begin{align} 
\notag
D_{h} 
& 
\leq 
C |\nabla \ell^{(h+1)}| 
+
C E_{h+1} 
\leq 
C D_{m} + C  \sum_{j = h+1}^{m} \mathcal{E}_{s,2}(\cu_j)  D_j 
\,,
\end{align}
which is~\eqref{e.Lip.iter.this}.  

\smallskip

\emph{Step 3.} We show~\eqref{e.Es.small}. 
Assume inductively that there exist constants~$H,K \in [1,\infty)$, to be determined by means of~$(s,\lambda,\Lambda,d)$, such that, for some~$h,k \in \Z$ with~$h  < k \leq m$, 
\begin{equation} 
\label{e.Lip.induction.one}
\sup_{k \in \N \cap [h+1,m] } 3^{- K \delta (m-k)} D_k \leq H D_m
\,.
\end{equation}
The induction assumption together with~\eqref{e.Lip.iter.this} and~$\delta \leq c  \coloneqq  K^{-1}$ implies that
\begin{equation*} 
D_h \leq C D_m + C \delta H \sum_{k=h+1}^m 3^{K \delta (m-k)} D_m
 \leq 
H 3^{K \delta (m-h)} D_m \Bigl( C H^{-1} 3^{-K \delta (m-h)}  +  C K^{-1} \Bigr)
\,,
\end{equation*}
and thus~\eqref{e.Lip.induction.one} follows by choosing~$H,K$ large enough by means of~$C$ in the above display. 

\smallskip

\emph{Step 4.}
We finally show~\eqref{e.sum.Es.small}. Assume that
\begin{equation} 
\label{e.Lip.induction.ass.two}
\sum_{k=n}^m  \mathcal{E}_{s,2}(\cu_k) \leq \delta_0 
\end{equation}
for some~$\delta_0(s,\lambda,\Lambda,d)$ to be fixed. Assume inductively that, for a given large constant~$H(s,\lambda,\Lambda,d) \in [1,\infty)$, we have, for some~$h \in \Z$ with~$n \leq h  <m$, that
\begin{equation} 
\label{e.reg.induction}
\sup_{j \in \N \cap [h+1,m]} D_j  \leq H  D_m
 \,.
\end{equation}
By~\eqref{e.Lip.iter.this},~\eqref{e.Lip.induction.ass.two} and~\eqref{e.reg.induction} we then deduce that
\begin{equation*}
D_h
\leq 
C D_m
+
C \delta_0 H D_m
\,.
\end{equation*}
Choosing thus~$\delta_0 = (2C)^{-1}$ and~$H = 2C$, we obtain the induction step and complete the proof.
\end{proof}

The proof of the second bullet in Proposition~\ref{p.big.black.regularity.box} relies on the following estimate of the  \emph{excess} of a finite-volume corrector on the good event. 

\begin{lemma}[Flatness at every scale]
\label{l.corr.flatness} 
Under the assumptions of Proposition~\ref{p.big.black.regularity.box}, there exist constants~$C(s,\lambda,\Lambda,d)<\infty$ and~$c(s,\lambda,\Lambda,d) \in (0,(2C)^{-1}]$ such that if~$\delta \in (0,c]$, then, for every~$n,m \in \Z$,  and~$\a \in G_{n,m}^{s,\infty}(\delta,\a_1)$, there exists a linear map~$e \mapsto Q_{n}[e;m]$ such that, for every~$e \in \Rd$ and~$\ell_e(x) \coloneqq  x \cdot e$, we have
\begin{equation*} 
3^{-n} 
\bigl\|w(\cdot,\cu_m,Q_{n}[e;m]) 
-(w(\cdot,\cu_n,Q_{n}[e;m]))_{\cu_n}
 - \ell_e \bigr\|_{\underline{L}^2(\cu_n)}
\indc_{G_{n,m}^{s,\infty}(\delta,\a_1)} 
\leq 
C \delta |e | 
\,.
\end{equation*}
Moreover,  we have under the event~$G_{n,m}^{s,\infty}(\delta,\a_1)$ that, for every~$e\in\Rd$,
\begin{equation} 
\label{e.corr.grad.flat}
C^{-1}| e |
\leq 
\bigl\| \s^{\nicefrac12} \nabla w(\cdot,\cu_m,Q_{n}[e;m] ) \bigr\|_{\underline{L}^2(\cu_n)} 
\leq C  | e |\,.
\end{equation} 
Finally, for every~$k\in \Z$ with~$n\leq k \leq m$ and~$e\in\Rd$, under the event~$G_{n,m}^{s,\infty}(\delta,\a_1)$, 
\begin{equation} 
\label{e.Qm.bounds}
(1+ C\delta)^{-(k-n)} \leq \frac{| Q_{k}[e;m] |}{|  Q_{n}[e;m] |}
\leq 
(1- C\delta)^{k-n}
\qand
\bigl| Q_{m}[\cdot \, ;m]  -\Id \bigr| \leq C\delta \leq \frac12 \,,
\end{equation}
and
\begin{equation} 
\label{e.Qm.bounds.sharp}
\bigl| Q_{n}[\cdot \, ;m]  -\Id \bigr| 
\indc_{G_{n,m}^{s,1}(\delta,\a_1)}
\leq C\delta \leq \frac 12 \,.
\end{equation}
\end{lemma}

\begin{proof}
Fix~$n,m \in \Z$ with~$n < m$. 
Throughout the proof, we assume that~$\a \in G_{n,m}^{s,\infty}(\delta,\a_1)$.
For~$k \in \Z$ with~$k \leq m$ and~$e \in \Rd$, define~$\ell^{(k)}[e;m]$ to be the affine function minimizing the following quantity 
\begin{equation*} 
E_{k}[e;m]  \coloneqq  \inf_{\ell\; \mbox{\scriptsize{affine}}} 3^{-k}\| w(\cdot,\cu_m,e)  - \ell \|_{\underline{L}^2(\cu_{k})}
=  
3^{-k}  \|w(\cdot,\cu_m,e)  - \ell^{(k)}[e;m] \|_{\underline{L}^2(\cu_{k})} \,.
\end{equation*}
Denote, for short,~$w = w(\cdot,\cu_m,e)$.  Denote also~$P_{n}[e;m]  \coloneqq  \nabla \ell^{(n)}[e;m]$. The mapping~$e \mapsto P_{n}[e;m]$ is linear by the linearity of~$e \mapsto \nabla w(\cdot,\cu_m,e)$. We claim that, for~$k \in \Z \cap [n,m]$, 
\begin{equation} 
\label{e.corr.induction}
E_{k}[e;m] \leq K \delta \bigl|  P_{k}[e;m] \bigr|  
\end{equation} 
for a large constant~$K(s,\lambda,\Lambda,d)$ to be fixed. Letting~$N(s,\lambda,\Lambda,d) \in \N$ be as in~\eqref{e.Lip.sum.this}, we have, for every~$k \in \Z \cap [n \vee (m-N) , m]$, 
\begin{equation} 
\label{e.induction.initial}
\bigl|e - P_{k}[e;m] \bigr|
\leq
C 3^{-k}\bigl\| \ell_{k}[e;m] - \ell_e \bigr\|_{\underline{L}^2(\cu_k)}
\leq 
C 3^{(\nicefrac d2 +1)N}   3^{-m} \bigl\| w - \ell_e \bigr\|_{\underline{L}^2(\cu_m)} 
\leq 
C \delta  | e |     \,.
\end{equation}
Thus, if~$C_{\eqref{e.induction.initial}} c \leq \nicefrac12$ and~$K \geq 2C_{\eqref{e.induction.initial}}$, we obtain~\eqref{e.corr.induction} for~$k \in \Z \cap [n \vee (m-N) , m]$. 

\smallskip

Next, if~$n < m-N$, we then assume  inductively that there exists~$h \in \Z \cap [n,m-N]$ such that~\eqref{e.corr.induction} is valid for every~$k \in \Z$ with~$h +1 \leq k \leq m$.  Since~$m$ and~$e$ are fixed, we drop them from the notation for both~$\ell$ and~$E$. By~\eqref{e.Lip.sum.this}, we have
\begin{equation} 
\label{e.corr.basic.iter.one}
E_{h}
\leq
\frac18 E_{h+N}
+
C \mathcal{E}_{s,2}(\cu_k)   \| \s^{\nicefrac12 }\nabla w \|_{\underline{L}^2(\cu_{h+N})} 
\leq
\frac{K}{8}\delta \bigl| \nabla \ell^{(h+N)} \bigr| 
+ 
C \delta \| \s^{\nicefrac12 }\nabla w \|_{\underline{L}^2(\cu_{h+N})} 
\,.
\end{equation}
In view of our claim~\eqref{e.corr.induction}, we define, for~$k \in\Z$, the composite quantity
\begin{equation*} 
F_k  \coloneqq  E_{k} - K \delta | \nabla \ell^{(k)}|\,,
\end{equation*}
and our goal is to show that~$F_h \leq 0$. The induction assumption guarantees that~$F_k\leq 0$ for~$k\in \Z \cap (h,m]$. Notice that we get by the triangle inequality, for every~$k \in \Z \cap [n,m]$, 
\begin{align} 
\label{e.affines.vs.F.pre.one}
|\nabla \ell^{(k+1)}-\nabla \ell^{(k)}|
& \leq
C E_{k+1}
\leq
C K \delta |\nabla \ell^{(k+1)}|
+C F_{k+1}\,,
\end{align}
and thus, if~$C_{\eqref{e.affines.vs.F.pre.one}} K c \leq \nicefrac12$, we have, for every~$k \in \Z \cap [n,m]$, 
\begin{equation} 
| \nabla \ell^{(k+1)}| \leq 
(1+\ep) | \nabla \ell^{(k)}|  +C F_{k+1}  \quad \mbox{with} \quad \ep  \coloneqq   \frac{C_{\eqref{e.affines.vs.F.pre.one}}K\delta}{1-C_{\eqref{e.affines.vs.F.pre.one}}K\delta}
\label{e.affines.vs.F}
\,.
\end{equation}
By the induction assumption and iteration, we then also get, for every~$k \in \Z$ with~$k> h$, 
\begin{equation}
\label{e.affines.k.vs.k.plus.one} 
|\nabla \ell^{(k)}| 
\leq
(1+\ep)^{h-k}  |\nabla \ell^{(h)}| 
\qand
|e| \leq  
C (1+\ep)^{m-h}  |\nabla \ell^{(h)}| 
\,.
\end{equation}
By~\eqref{e.CG.Cacc.Poincare}, we thus obtain 
\begin{align} 
\label{e.grad.corr.upper}
 \| \s^{\nicefrac12 }\nabla w \|_{\underline{L}^2(\cu_{h})} 
 \leq
 C
 \| \s^{\nicefrac12 }\nabla w \|_{\underline{L}^2(\cu_{h+N})} 
&
\leq
C  3^{-n} \| w - (w)_{\cu_{h+N+1}} \|_{\underline{L}^2(\cu_{h+N+1})} 
\notag \\ &
\leq
C E_{h+N+1} + C |\nabla \ell^{(h+N+1)}| 
\notag \\ &
\leq
C F_{h+N+1} + K \delta | \nabla \ell^{(h+N+1)}| + C |\nabla \ell^{(h+N+1)}|
\notag \\ &
\leq
C (1+\ep)^{N}  |\nabla \ell^{(h)}| 
\,.
\end{align}
Inserting the above two displays into~\eqref{e.corr.basic.iter.one} yields
\begin{equation} 
\label{e.Fh.is.what.it.is}
F_h \leq - K\delta |\nabla \ell^{(h)}| + (1+\ep)^N \biggl( \frac{1}{8} + \frac{C}{K} \biggr) K \delta |\nabla \ell^{(h)}| 
\end{equation}
with~$\ep$ as in~\eqref{e.affines.vs.F}. Taking~$K \geq 8C_{\eqref{e.Fh.is.what.it.is}}$ and then~$c$ so small that~$(1+\ep)^N \leq 2$, we deduce that~$F_h \leq 0$, proving the induction step and establishing~\eqref{e.corr.induction} for every~$k \in \Z \cap [n,m]$.

\smallskip

Next, similarly to~\eqref{e.affines.vs.F.pre.one} and using~\eqref{e.affines.k.vs.k.plus.one}, we obtain
\begin{equation*} 
|\nabla \ell^{(k)} - \nabla \ell^{(k+1)}| \leq CE_{k+1} \leq
C \delta  |\nabla \ell^{(k)}|   \,,
\end{equation*}
which gives us by~\eqref{e.induction.initial} that, for every~$k \in \Z \cap [n,m-1]$, 
\begin{equation} 
\label{e.what.affines.do}
1 - C\delta 
\leq
\frac{| \nabla \ell^{(k+1)}|}{| \nabla \ell^{(k)}|} 
\leq 
1 + C\delta 
\implies
(1 - C\delta)^{n+1} 
\leq
\frac{| P_{n}[e;m] |}{|e|} 
\leq
(1 + C\delta)^{n+1} 
\,.
\end{equation}
Thus~$e \mapsto P_n[e;m]$ has full rank and it is invertible, and we define~$Q_n[\cdot \, ;m] = P_n^{-1}[\cdot\,;m]$. The estimate~\eqref{e.Qm.bounds} then follows from the above display by increasing the constant~$C$.  

\smallskip

Next, the upper bound in~\eqref{e.corr.grad.flat} follows by~\eqref{e.grad.corr.upper}, which gives
\begin{equation*} 
 \| \s^{\nicefrac12 }\nabla w(\cdot,\cu_m,e) \|_{\underline{L}^2(\cu_{n})} 
\leq 
C  \bigl| P_{n}[e;m] \bigr|
 \,.
\end{equation*}
Plugging in~$Q_n[e \, ;m]$ instead of~$e$ gives the upper bound since~$P_{n}[Q_n[e \, ;m] ;m] = e$. To see the lower bound, we have by~\eqref{e.CG.Cacc.Poincare} and~\eqref{e.corr.induction} that
\begin{equation} 
\label{e.q.naught.Pn}
|  P_{n}[e;m] | 
\leq
C 3^{-n} \| w - (w)_{\cu_n}   \|_{\underline{L}^2(\cu_n)} + C E_n
\leq
C \| \s^{\nicefrac12} \nabla w(\cdot,\cu_m,e)   \|_{\underline{L}^2(\cu_n)} + C\delta | P_{n}[e;m] | 
\,,
\end{equation}
and thus the lower bound in~\eqref{e.corr.grad.flat} follows provided that~$C_{\eqref{e.q.naught.Pn}}c \leq \nicefrac12$.

\smallskip

Finally, by~\eqref{e.affine.diff} and~\eqref{e.affine.end.point}, the Lipschitz estimate~\eqref{e.Es.small} and~\eqref{e.corr.grad.flat}, we have, under the event~$G_{n,m}^{s,1}(\delta)$ with small enough~$\delta_0(s,\lambda,\Lambda,d) \in (0,1)$ and~$\delta \in (0,\delta_0]$, that 
\begin{equation*} 
\bigl| \nabla \ell_{n}[e;m]- \nabla \ell_{m}[e;m] \bigr|
\leq 
C \delta |e|  
 \,.
\end{equation*}
By~\eqref{e.induction.initial} we get that~$| e- \nabla \ell_{m}[e;m] |\leq C \delta |e|$, and thus  
\begin{equation*} 
\bigl|  P_{n}[ e ; m] - e \bigr|
\leq C\delta |e|
\implies 
\bigl| P_{n}[ \cdot \, ; m] - \Id \bigr| \leq C\delta \leq \frac12
\,.
\end{equation*}
Since~$Q_{n} = P_{n}^{-1}$, we obtain~\eqref{e.Qm.bounds.sharp}.  The proof is complete. 
\end{proof}

We next apply the previous lemma to prove the large-scale~$C^{1,\eta}$-estimate~\eqref{e.C.one.eta}, thereby completing the proof of Proposition~\ref{p.big.black.regularity.box}. 

\begin{proof}[Proof of~\eqref{e.C.one.eta}]
Throughout the proof, we fix~$s \in (0,\nf 12)$,~$\eta \in [\nf 12,1)$ and~$m,n \in \Z$ with~$n<m$. Let~$u \in \mathcal{A}(\cu_m;\a)$. Throughout the proof, we assume the good event~$G_{n,m}^{s,\infty}(\delta)$ as in~\eqref{e.G.max} for~$\delta \in (0,\theta_0]$ with~$\theta_0$ to be fixed.  Let~$e_k \in \Rd$ be such that
\begin{equation*} 
E_k \coloneqq 
\inf_{e\in\Rd} \bigl\| \s^{\nicefrac12} \nabla (u - w(\cdot,\cu_m,e)) \bigr\|_{\underline{L}^2(\cu_k)}
=
\bigl\| \s^{\nicefrac12} \nabla (u - w(\cdot,\cu_m,e_k)) \bigr\|_{\underline{L}^2(\cu_k)}  
 \,.
\end{equation*}

\smallskip

\emph{Step 1.} We first show that, for every~$\alpha \in [\nicefrac12,1)$, there exists constants~$C$ and~$\theta_1 \in (0,1)$, both depending only on~$(\eta,s,\lambda,\Lambda,d)$, such that if~$\delta \in (0,\theta_1]$, then
\begin{equation} 
\label{e.C.one.eta.pre}
E_{n} \leq C 3^{-\alpha(m-n)} E_m
\,.
\end{equation}
Set~$u_k  \coloneqq  u - w(\cdot,\cu_m,e_k)$. Letting~$h_k$ be as in~\eqref{e.find.hk.for.u.in.cu.k} (with~$u_k$ instead of~$u$), we have
\begin{equation*} 
3^{-k} \| u_k - h_k  \|_{\underline{L}^{2}(\cu_{k-1})} 
\leq 
C \delta  \|\s^{\nicefrac12} \nabla u_k \|_{\underline{L}^{2}(\cu_{k})} 
 \,.
\end{equation*} 
Notice that, by the triangle inequality,~\eqref{e.CG.Cacc.Poincare} and the above display, we also get  
\begin{equation*} 
3^{-k} \| h_k - (h_k)_{\cu_{k-1}} \|_{\underline{L}^{2}(\cu_{k-1})}
\leq 
C   \|\s^{\nicefrac12} \nabla u_k \|_{\underline{L}^{2}(\cu_{k})} 
 \,.
\end{equation*}
Now, by the regularity of~$\a_1$-harmonic functions, there is an affine function~$\ell^{(k)}$ such that
\begin{equation*} 
\| h_k -\ell^{(k)}  \|_{\underline{L}^{2}(\cu_{k-h})} \leq C3^{-2h}  \| h_k - (h_k)_{\cu_{k-1}} \|_{\underline{L}^{2}(\cu_{k-1})}
\leq
C 3^{k-2h} \|\shom^{\nicefrac12} \nabla u_k \|_{\underline{L}^{2}(\cu_{k})}
 \,.
\end{equation*}
The affine function above satisfies
\begin{equation*} 
| \nabla \ell^{(k)} |  
\leq 
\| \nabla h_k \|_{L^\infty(\cu_{k-2})} 
\leq 
C 3^{-k}\| h_k  - (h_k )_{\cu_{k-1}} \|_{L^\infty(\cu_{k-1})}
\leq 
C \|\s^{\nicefrac12} \nabla u_k \|_{\underline{L}^{2}(\cu_{k})} 
\,.
\end{equation*}
Lemma~\ref{l.corr.flatness} provides us~$\tilde{w}_k  \coloneqq  w(\cdot,\cu_m,\tilde{e}_k)$, which is a good approximant of~$\ell^{(k)}$ so that
\begin{equation*} 
3^{-k} \| \tilde{w}_k -  \ell^{(k)} \|_{\underline{L}^{2}(\cu_{k})} \leq C \delta | \nabla \ell^{(k)} | 
 \,.
\end{equation*}
Thus, by the triangle inequality, 
\begin{equation*} 
\| u_k -\tilde w_k  \|_{\underline{L}^{2}(\cu_{k+1-h})} 
\leq 
C3^k \bigl( 3^{-2h} + \delta 3^{\frac d2 h} \bigr) \|\shom^{\nicefrac12} \nabla u_k \|_{\underline{L}^{2}(\cu_{k})} 
 \,.
\end{equation*}
Since~$u_k -\tilde w_k  \in \mathcal{A}(\cu_{k-h+1};\a)$,~\eqref{e.CG.Cacc.Poincare} yields
\begin{equation} 
\label{e.Ek.excess.decay}
E_{k-h} \leq \| \nabla (u_k -\tilde w_k)  \|_{\underline{L}^{2}(\cu_{k-h})} 
\leq
C 
\bigl( 3^{-h} + \delta 3^{(\frac d2 +1) h} \bigr) E_k 
 \,.
\end{equation}
Choosing~$h_0(\eta,s,\lambda,\Lambda,d)$ to be the smallest integer such~$3^{-(1-\eta)h_0}C_{\eqref{e.Ek.excess.decay}} \leq \nicefrac 14$ and then requiring that~$\theta_1$ is so small that~$\theta_1  C_{\eqref{e.Ek.excess.decay}} 3^{(\frac d2 +2) h_0} \leq \nicefrac14$, we deduce that 
\begin{equation*} 
E_{k-h_0} \leq \frac12 \cdot 3^{-\eta h_0} E_{k} \,.
\end{equation*}
An iteration argument then proves~\eqref{e.C.one.eta.pre} using~$E_j \leq 3^{(\frac d2 +1)h_0} E_k$ for every~$k \in [n,m] \cap \Z$ and~$j \in [k-h_0,k] \cap \Z$. 

\smallskip

\emph{Step 2.} We prove~\eqref{e.C.one.eta}. assume that~$\theta_0 \leq \theta_1$ from Step 1  so that~\eqref{e.C.one.eta.pre} is valid for~$\delta \in (0,\theta_0]$. By~\eqref{e.Qm.bounds} and~\eqref{e.corr.grad.flat}, we deduce that there is~$\theta_2(\eta,s,\lambda,\Lambda,d) \in (0,1)$ such that~$\delta \leq \theta_2$ implies, for every~$k,j$ with~$n \leq j \leq k \leq m$ and~$e\in\Rd$,
\begin{equation*} 
\| \s^{\nicefrac12} \nabla w(\cdot,\cu_m,e)  \|_{\underline{L}^2(\cu_k)}   
\leq
C 3^{\frac14(k-j)}\| \s^{\nicefrac12} \nabla w(\cdot,\cu_m,e) \|_{\underline{L}^2(\cu_{j})} 
 \,.
\end{equation*} 
We fix~$\theta_0 \coloneqq \theta_1 \wedge \theta_2$, and assume that~$\delta \in (0,\theta_0]$. By~\eqref{e.C.one.eta.pre} and the triangle inequality, we get
\begin{equation*} 
\| \s^{\nicefrac12} \nabla w(\cdot,\cu_m,e_j-e_{j+1}) \|_{\underline{L}^2(\cu_{j})}
\leq 
E_j + 3^{\frac d2+1} E_{j+1} 
\leq
C3^{-\eta(m-j)} \|\s^{\nicefrac12} \nabla u\|_{\underline{L}^2(\cu_{m})}
 \,.
\end{equation*}
It follows by the previous two displays that
\begin{align*} 
\| \s^{\nicefrac12} \nabla w(\cdot,\cu_m,e_j-e_{j+1}) \|_{\underline{L}^2(\cu_k)}
&
\leq 
C 3^{\frac14(k-j)} 
\| \s^{\nicefrac12}  \nabla w(\cdot,\cu_m,e_j-e_{j+1})  \|_{\underline{L}^2(\cu_j)}
\notag \\ &
\leq 
C
3^{\frac14(k-j)} 3^{-\eta(m-j)} \| \s^{\nicefrac12} \nabla u \|_{\underline{L}^2(\cu_m)}  
\notag \\ &
= 
C 3^{-(\eta-\frac14) (k-j)} 3^{-\alpha(m-k)}
\| \s^{\nicefrac12} \nabla u \|_{\underline{L}^2(\cu_m)}  
\,.
\end{align*}
Since~$\eta \geq \nicefrac12$, 
by summing over~$j\in \{n,\ldots,k-1\}$ and using the triangle inequality, we obtain
\begin{equation*} 
\| \s^{\nicefrac12} \nabla w(\cdot,\cu_m,e_n-e_k)  \|_{\underline{L}^2(\cu_k)}
\leq 
C 3^{-\eta(m-k)} \| \s^{\nicefrac12} \nabla u \|_{\underline{L}^2(\cu_m)}   
 \,.
\end{equation*}
Therefore,~\eqref{e.C.one.eta} follows by the triangle inequality, the above display and~\eqref{e.C.one.eta.pre} by taking~$e  \coloneqq  e_n$. The proof of Proposition~\ref{p.big.black.regularity.box} is now complete.
\end{proof}

\subsection{Estimates for correctors}
\label{ss.correctors}

In this subsection, we present a deterministic black box principle which gives quantitative estimates on the first-order correctors on the good events~$G_{n,\infty}^{s,1}(c,\a_1)$ for sufficiently small~$c>0$. The idea is to telescope the estimates for the finite-volume Dirichlet correctors~$w(\cdot,\cu_m,e)$ across all scales, passing to the limit~$m\to \infty$. 

\begin{proposition}
\label{p.big.black.corrector.box}
Let~$s\in(0,\nf 12)$,~$\eta \in [\nf 12,1)$,~$\a_1 \in \R^{d\times d}_{+}$ be a matrix satisfying, for some constants~$0<\lambda < \Lambda < \infty$,
\begin{equation}
\lambda|\xi|^2 \leq \xi \cdot \a_1 \xi
\quad \mbox{and} \quad 
\Lambda^{-1}|\xi|^2 \leq \xi \cdot \a_1^{-1} \xi
\,.
\label{e.a0.ue.corr}
\end{equation}
There exist~$C(\eta,s,\lambda,\Lambda,d) <\infty$ and~$c(\eta,s,\lambda,\Lambda,d)\in (0,1)$ such that, for every~$\a \in \bigcup_{n\in\Z} G_{n,\infty}^{s,1}(c,\a_1)$ and~$e \in \Rd$, there exists~$\phi_e$ such that~$\ell_e + \phi_e \in \mathcal{A}(\Rd;\a)$ and the following statements hold. 
\begin{itemize}
\item 
For every~$n \in \Z$ and~$e \in \Rd$ with~$|e|=1$,
\begin{equation} 
\label{e.corr.weaknorms}
3^{-ns} \bigl( \|  \nabla \phi_e \|_{\Hminushat{-s} (\cu_n)}
+
\| \a (e{+}\nabla \phi_e) - \a_1e \|_{\Hminushat{-s} (\cu_n)} \bigr)
\indc_{G_{n,\infty}^{s,1}(c,\a_1)}
\\
\leq
C \!  \sum_{j=n}^{\infty} \mathcal{E}_ {s,2} (\cu_{j};\a,\a_1)
\end{equation}
and
\begin{equation} 
\label{e.corr.energy}
\Bigl| \| \s^{\nicefrac12} (e{+}\nabla \phi_e) \|_{\underline{L}^2(\cu_n)}^2 -e\cdot \s_1 e \Bigr| 
\indc_{G_{n,\infty}^{s,1}(c,\a_1)}
 \leq 
C\sum_{j=n}^{\infty} \mathcal{E}_ {s,2} (\cu_{j};\a,\a_1)
 \,.
\end{equation}
\item
For every~$m,n \in \Z$ with~$m\geq n$ and~$u \in \A(\cu_m)$, there exists~$e \in \Rd$ such that
\begin{equation} 
\| \s^{\nicefrac12} (\nabla u - (e + \nabla \phi_e))  \|_{\underline{L}^2(\cu_n)} 
\indc_{G_{n,\infty}^{s,1}(c,\a_1)}
\leq
C 3^{- \eta (m-k)} 
\| \s^{\nicefrac12} \nabla u \|_{\underline{L}^2(\cu_m)} 
\,.
\label{e.corr.C.one.eta}
\end{equation}
\item We have the Liouville property
\begin{equation} 
\Bigl\{\nabla u\, : \,  u \in \A(\Rd;\a)\,, 
\lim_{m \to \infty} 3^{-\eta m} \|\s^{\nicefrac12} \nabla u \|_{\underline{L}^2(\cu_m)} = 0 \Bigr\} =
\bigl\{ e+ \nabla \phi_e \, : \, e \in\Rd \bigr\} 
\,.
\label{e.corr.finite.dim}
\end{equation}  
\end{itemize}
\end{proposition}

\begin{proof}
Fix~$s \in (0,\nf 12)$. As before, we suppress~$(\a,\a_1)$ from the notation and denote~$\mathcal{E}_ {s,2} (\cu_m) \coloneqq \mathcal{E}_ {s,2} (\cu_m;\a, \a_1)$ and~$G_{n,\infty}^{s,1}(\delta)\coloneqq G_{n,\infty}^{s,1}(\delta,\a_1)$.   By~\eqref{e.Dir.corr.weak}, the finite-volume Dirichlet correctors satisfy
\begin{equation} 
\label{e.corr.weaknorms.applied}
\|  \nabla w(\cdot,\cu_m,e) - e \|_{\Hminushat{-s} (\cu_m)}
+
\| \a \nabla w(\cdot,\cu_m,e) - \a_1e \|_{\Hminushat{-s} (\cu_m)} 
\\
\leq
C3^{sm}  \mathcal{E}_ {s,2} (\cu_m) \bigl( 1 + \mathcal{E}_ {s,2} (\cu_m)\bigr)
 | e|  
\end{equation}
and, by~\eqref{e.Dir.corr.Ltwo},
\begin{equation} 
\label{e.corr.slope.applied}
3^{-m}
\bigl\| w(\cdot,\cu_m,e) - \ell_e \bigr\|_{\underline{L}^2(\cu_m)} 
\leq
C \mathcal{E}_ {s,2} (\cu_m) \bigl( 1 + \mathcal{E}_ {s,2} (\cu_m)\bigr)
 | e|  \,.
\end{equation}
Since~$w(\cdot,\cu_m,e) - \ell_e \in H_{\a,0}^{1}(\cu_m) \cap  H_0^{1-s}(\cu_m)$ and~$\a \nabla w(\cdot,\cu_m,e)  \in H^{-s}(\cu_m)$, and~$s \in (0,\nf 12)$, we obtain 
\begin{align*}
\fint_{\cu_m} \nabla w(\cdot,\cu_m,e) \cdot \a \nabla w(\cdot,\cu_m,e) 
&=
e \cdot \fint_{\cu_m} \a \nabla w(\cdot,\cu_m,e) 
= e\cdot \s_1 e +  e \cdot \fint_{\cu_m} (\a \nabla w(\cdot,\cu_m,e) - \a_1 e)
\,.
\end{align*}
Thus, for the energy we have by~\eqref{e.corr.weaknorms.applied} that
\begin{equation} 
\label{e.corr.energy.applied}
\bigl| \| \s^{\nicefrac12} \nabla w(\cdot,\cu_m,e) \|_{\underline{L}^2(\cu_m)}^2 - e\cdot \s_1 e  \bigr| 
 \leq 
 C\mathcal{E}_ {s,2} (\cu_m) \bigl( 1 + \mathcal{E}_ {s,2} (\cu_m)\bigr)
  |e|^2
 \,.
\end{equation}
Let~$w_m  \coloneqq  w(\cdot,\cu_{m+1},e)- w(\cdot,\cu_m,e)$. The Caccioppoli estimate~\eqref{e.CG.Cacc.Poincare}, together with~\eqref{e.corr.slope.applied} and the triangle inequality, implies that
\begin{equation}
\label{e.largescale.Caccioppoli.first.applied.corr}
\| \s^{\nicefrac12} \nabla w_m \|_{\underline{L}^2(\cu_{m-1})} 
\indc_{G_{m,m+1}^{s,\infty}(1)}
\leq
C 3^{-m} \| w_m \|_{\underline{L}^2(\cu_m)}\indc_{G_{m,m+1}^{s,\infty}(1)}
\leq 
C \mathcal{E}_ {s,2} (\cu_{m+1}) |e| \,.
\end{equation}
Thus, by the Lipschitz estimate~\eqref{e.sum.Es.small}, 
\begin{align*} 
\| \s^{\nicefrac12} \nabla w_m \|_{\underline{L}^2(\cu_n)}
\indc_{G_{n,m+1}^{s,1}(c)}
 &
\leq
C  \| \s^{\nicefrac12} \nabla w_m \|_{\underline{L}^2(\cu_{m-1})}
\indc_{G_{m,m+1}^{s,\infty}(1)}
\leq
C \mathcal{E}_ {s,2} (\cu_{m+1}) 
\,.
\end{align*}
In particular, by the triangle inequality, if~$\delta \in(0,c]$, for every~$k,m \in Z$ with~$m \geq k \geq n$, 
\begin{equation*} 
\| \s^{\nicefrac12} \nabla (w(\cdot,\cu_m,e) - w(\cdot,\cu_k,e)) \|_{\underline{L}^2(\cu_n)}\indc_{G_{n,\infty}^{s,1}(\delta)}
\leq C |e|  \sum_{j=k}^{m} \mathcal{E}_ {s,2} (\cu_{j}) \indc_{G_{n,\infty}^{s,1}(\delta)}
\leq C \delta |e|
\,.
\end{equation*}
Therefore,~$\{w(\cdot,\cu_m,e)\}_{m}$ is a Cauchy sequence in~$H_\s^1(\cu_n)$ if~$\a \in G_{n,\infty}^{s,1}(\delta)$. By a diagonal argument, we find~$\psi_e \in \mathcal{A}(\Rd ; \a)$ so that
\begin{equation*} 
\psi_e  \coloneqq  \lim_{m \to \infty} w(\cdot,\cu_m,e) 
\qand
\phi_e  \coloneqq  \psi_e - \ell_e
\,.
\end{equation*}
The telescope summation and the triangle inequality give us
\begin{equation*} 
\| \s^{\nicefrac12} \nabla(\psi_e - w(\cdot,\cu_n,e)) \|_{\underline{L}^2(\cu_n)}\indc_{G_{n,\infty}^{s,1}(\delta)}
\leq
C |e| \sum_{j=n}^{\infty} \mathcal{E}_ {s,2} (\cu_{j}) \indc_{G_{n,\infty}^{s,1}(\delta)}
\leq C \delta |e|
\,.
\end{equation*}
It is straightforward to check that, if~$\a \in G_{n,\infty}^{s,1}(\delta)$,~$\psi_e \in \A(\Rd;\a)$ and~$e \mapsto \psi_e$ is linear by the linearity of~$e \mapsto w(\cdot,\cu_n,e)$. Moreover, by~\eqref{e.corr.weaknorms.applied}, Lemma~\ref{l.crude.weaknorms},~\eqref{l.Wsp.vs.Bspp} and the triangle inequality,
\begin{equation} 
\label{e.corr.weaknorms.pre}
3^{-ns} \Bigl( 
\|  \nabla \phi_e \|_{\Hminushat{-s} (\cu_n)}
+
\| \a (\nabla \phi_e + e) - \a_1e \|_{\Hminushat{-s} (\cu_n)} \Bigr)\indc_{G_{n,\infty}^{s,1}(\delta)}
\\
\leq
C |e| \sum_{j=n}^{\infty} \mathcal{E}_ {s,2} (\cu_{j}) \,.
\end{equation}
and, by~\eqref{e.corr.energy.applied}, the energy satisfies
\begin{equation*} 
\bigl| \| \s^{\nicefrac12} (e+\nabla \phi_e) \|_{\underline{L}^2(\cu_m)}^2 - | \s_1^{\nf 12}e|^2 \bigr| \indc_{G_{n,\infty}^{s,1}(\delta)}
 \leq 
C |e|^2 \sum_{j=n}^{\infty} \mathcal{E}_ {s,2} (\cu_{j})
 \,.
\end{equation*}
The above two displays give us both~\eqref{e.corr.weaknorms} and~\eqref{e.corr.energy}. 
Furthermore, by Proposition~\ref{p.big.black.regularity.box} and linearity of~$e \mapsto \psi_e$ we have that there exists a linear map~$e \mapsto P_m[e] \in \Rd$ such that, for every~$n,m \in \N$ with~$m \geq n$ and~$\delta \in (0,c]$,
\begin{align*} 
\lefteqn{
\bigl\| \s^{\nicefrac12} \nabla \bigl(\psi_e - w(\cdot,\cu_m,P_m[e]) \bigr) \bigr\|_{\underline{L}^2(\cu_n)}
\indc_{G_{n,\infty}^{s,1}(\delta)}
} \qquad &
\notag \\ &
\leq
C 3^{- \eta (m-n)}
\bigl\| \s^{\nicefrac12} \nabla(\psi_e - w(\cdot,\cu_m,e)) \ \bigr\|_{\underline{L}^2(\cu_m)}
\indc_{G_{n,\infty}^{s,1}(\delta)}
\leq
C 3^{- \eta (m-n)} |e|
\sum_{j=m}^{\infty} \mathcal{E}_ {s,2} (\cu_{j}) \indc_{G_{n,\infty}^{s,1}(\delta)}
\,.
\end{align*}
By~\eqref{e.CG.Cacc.Poincare}, we deduce that
\begin{align*}
| e - P_m[e] | 
&
\leq 
C \bigl\| \s^{\nicefrac12} \nabla \bigl(w(\cdot,\cu_m,e)  - w(\cdot,\cu_m,P_m[e]) \bigr) \bigr\|_{\underline{L}^2(\cu_n)}
\indc_{G_{n,\infty}^{s,1}(\delta)}
\notag \\ &
\leq
C |e|
\sum_{j=m}^{\infty} \mathcal{E}_ {s,2} (\cu_{j}) \indc_{G_{n,\infty}^{s,1}(\delta)}
\leq C \delta |e|
\,.
\end{align*}
By taking~$\delta$ small enough, we find that~$e \mapsto P_m[e]$ is injective. Therefore we get that for every~$e \in \Rd$ there exists~$e' \in \Rd$ such that
\begin{equation*} 
\| \s^{\nicefrac12} \nabla(\psi_{e'} - w(\cdot,\cu_m,e)) \|_{\underline{L}^2(\cu_n)} 
\indc_{G_{n,\infty}^{s,1}(\delta)}
\leq
C 3^{- \eta (m-n)}
\| \s^{\nicefrac12} \nabla \psi_{e'}  \|_{\underline{L}^2(\cu_n)} 
\sum_{j=n}^{\infty} \mathcal{E}_ {s,2} (\cu_{j}) \indc_{G_{n,\infty}^{s,1}(\delta)}
\,.
\end{equation*}  
Using this,~\eqref{e.corr.energy.applied} and Proposition~\ref{p.big.black.regularity.box} we then deduce that, for every~$m,n \in \Z$ with~$m\geq n$ and for every~$u \in \A(\cu_m)$, there exists~$e \in \Rd$ such that
\begin{equation*} 
\| \s^{\nicefrac12} \nabla(u - \psi_{e} ) \|_{\underline{L}^2(\cu_n)} \indc_{G_{n,\infty}^{s,1}(\delta)}
\leq
C 3^{- \eta (m-k)} 
\| \s^{\nicefrac12} \nabla u \|_{\underline{L}^2(\cu_m)} 
\,.
\end{equation*}
This is~\eqref{e.corr.C.one.eta}. 
Finally, by taking~$\frac12(1+\eta)$ instead of~$\eta$, this implies that if~$\a \in G_{n,\infty}^{s,1}(\delta)$, we obtain the last bullet point~\eqref{e.corr.finite.dim} of Proposition~\ref{p.big.black.corrector.box}. The proof is complete.
\end{proof}

\subsection{Proofs of the main results} 
\label{ss.proofs.main.results}

We present the proofs of Theorems~\ref{t.HC.intro} and~\ref{t.main}, beginning with the latter. The idea is to use Corollary~\ref{c.main.algebraic.quenched} to enforce the hypothesis of the black box propositions (Propositions~\ref{p.big.black.box},~\ref{p.big.black.regularity.box} and~\ref{p.big.black.corrector.box}) after an affine change of variables to the intrinsic geometry of the homogenized matrix. 

\begin{proof}[{Proof of Theorem~\ref{t.main}}]
We assume, without loss of generality, that the anti-symmetric part~$\khom$ of the homogenized matrix~$\ahom$ vanishes, so that~$\ahom = \shom$. Otherwise, we recenter by subtracting~$\khom$ from~$\a(\cdot)$ and apply the observations of Section~\ref{ss.skew}.
We let~$\bfAhom$ be the matrix defined in~\eqref{e.Ahom.final} and we let
\begin{equation*}
\overline{\Pi} \coloneqq | \shom | |\shom^{-1} |^{-1}
\end{equation*}
denote the aspect ratio of the homogenized matrix. Note that~$\overline{\Pi} \leq \Pi$. 

\smallskip

We work in the intrinsic geometry of the homogenized matrix~$\shom$. As in~\eqref{e.cus.k.def}, we let~$\cus_{n} \coloneqq \mathbf{q}_0 (\cu_n)$, where~$\mathbf{q}_0$ is the positive matrix defined in~\eqref{e.q.naught} with~$\m_0 \coloneqq \shom$. 
Recall by~\eqref{e.cus.bound} that 
\begin{equation*}
\tfrac{99}{100} \cu_k
\subseteq \,\cus_k 
\subseteq \tfrac{101}{100}
\overline{\Pi}^{\nf12}
\cu_k
\subseteq 
\cu_{k+n_*} \quad \mbox{where} \ n_* \coloneqq 1 + \bigl\lceil \log_3  \overline{\Pi}^{\nf12} \bigr \rceil\,.
\end{equation*}
We will apply  Propositions~\ref{p.big.black.box},~\ref{p.big.black.regularity.box} and~\ref{p.big.black.corrector.box} after making the affine change of variables~$x \mapsto \mathbf{q}_0x$. Observe that if~$u$ is a solution of 
\begin{equation*}
-\nabla \cdot \a \nabla u = f \quad \mbox{in} \ U\,, 
\end{equation*}
and we define, for~$\mu\in(0,\infty)$,  
\begin{equation*}
\tilde{u} (x) \coloneqq u(\mathbf{q}_0 x)
\,, \quad 
\tilde{\a}(x) \coloneqq
\mu \mathbf{q}_0^{-1}\a(\mathbf{q}_0 x)\mathbf{q}_0^{-1} \,,
\quad 
\tilde{f}(x) \coloneqq \mu f (\mathbf{q}_0 x)\,,
\end{equation*}
then~$\tilde{u}$ is a solution of 
\begin{equation*}
-\nabla \cdot \tilde{\a} \nabla \tilde{u} = \tilde{f} \quad \mbox{in} \ \mathbf{q}_0^{-1}U
\,.
\end{equation*}
We introduce, for~$s\in (0,\infty)$ and~$q\in [1,\infty)$, the random variable 
\begin{equation}
\mathcal{E}_{s,q}(\cus_m;\a,\shom)
\coloneqq
\biggl( 
\css{sq} 
\sum_{l=-\infty}^m 
3^{-sq(m-l)} \!\!
\max_{z \in 3^l\Lat \cap \cus_m}
\max_{|e| =1} \bigl( \bfJ\bigl(z {+} \cus_l,\bfAhom^{-\nf 12} e, \bfAhom^{\nf 12} e \,; \a\bigr) \bigr)^{\nf q2}
\biggr)^{\!\nf1q} 
\,.
\label{e.mathcal.E.cus}
\end{equation}
where~$\Lat \coloneqq \mathbf{q}_0 (\Zd)$, as in~\eqref{e.adapted.lattice}.
We may compare~\eqref{e.mathcal.E.cus} to the definition of~$\mathcal{E}_{s,q} (\cu_m;\a, \shom)$  in~\eqref{e.mathcal.E.def}. 
In fact, by the above change of variables, we find that, for every~$\mu\in (0,\infty)$, 
\begin{equation}
\mathcal{E}_{s,q}(\cus_m;\a,\shom)
=
\mathcal{E}_{s,q}\bigl(\cu_m ; \mu \mathbf{q}_0^{-1}\a\mathbf{q}_0^{-1} ,\mu \mathbf{q}_0^{-1}\shom\mathbf{q}_0^{-1} \bigr)
\,.
\label{e.change.of.var.mathcal.E}
\end{equation}
We will choose~$\mu\coloneqq |\shom^{-1} |$. In this case, the estimate~\eqref{e.szero.vs.qzero} yields 
\begin{equation*}
\frac{99}{100} \Id 
\leq 
\mu \mathbf{q}_0^{-1}\shom\mathbf{q}_0^{-1}  
\leq 
\frac{101}{100} \Id \,,
\end{equation*}
so we will apply each of the propositions with ellipticity constants~$\lambda$ and~$\Lambda$ very close to one. We also define an event analogous to the one in~\eqref{e.G.max} by 
\begin{equation} 
\label{e.G.max.encore}
\tilde{G}_{n,m}^{s,t}(\delta, \a_1) 
\coloneqq
\biggl\{ \a \in L^1(\cus_m;\R^{d\times d}_{+}) \,:\,
\biggl( \sum_{k=n}^m  \bigl(\mathcal{E}_{s,2}(\cus_k;\a,\a_1)  \bigr)^t\biggr)^{\! \nicefrac1t}
\leq 
\delta
\biggr\}
\,.
\end{equation}
In view of~\eqref{e.change.of.var.mathcal.E}, the change of variables implies that
\begin{equation}
\a \in \tilde{G}_{n,m}^{s,t}(\delta,\ahom) 
\quad \iff \quad 
\mu \mathbf{q}_0^{-1}\a\mathbf{q}_0^{-1}
\in G_{n,m}^{s,t}(\delta, \mu \mathbf{q}_0^{-1}\shom\mathbf{q}_0^{-1} )
\,. 
\label{e.change.of.var.events.G}
\end{equation}

Before applying the propositions, we need to estimate~$\mathcal{E}_{s,2} (\cus_m;\a, \shom)$. 
We will show that, for every~$s \in (\nf\gamma2 , \nf12)$, if we select~$\rho \coloneqq \frac12(1+s) \in (s,1)$ and~$\delta>0$ and let~$\mathcal{R}_{\delta,\rho}$ and~$\theta$ be as in the statement of Corollary~\ref{c.main.algebraic.quenched}, then, for every~$m\in\N$, we have 
\begin{equation}
\label{e.control.mathcal.E.cus}
3^{m} \geq \max\bigl\{ \S , \mathcal{R}_{\delta,\rho}  \}
\implies  
\mathcal{E}_{s,2} ( \cus_{m};\a,\shom )^2 
\leq
C \delta (1-2s)^{-2}   \overline{\Pi}^{\frac12(1+2s)}
\biggl( \frac{3^{m}}{\max\{ \mathcal{S},\mathcal{R}_{\delta,\rho}  \} } \biggr)^{\!-\theta}  
\,.
\end{equation}
To prove~\eqref{e.control.mathcal.E.cus}, we use a subadditivity argument to compare it to~$\mathcal{E}_{s,2} (\cu_{m+n_*};\a, \shom)$ and then appeal to~\eqref{e.impl.to.mathcal.E}. 
As in the proof of Lemma~\ref{l.crude.moments}, we may partition every cube of the form~$y+\cus_k$ into~$\{ V_j(k,y) \,:\, j \in \Z \cap (-\infty,k] \}$ where each set~$V_j(k,y)$ is a union of cubes of the form~$z+ \cu_j$ with~$z\in 3^j\Zd$ and has Lebesgue measure satisfying~$|V_j(k,y)|\leq C \overline{\Pi}^{\nf12}3^{j-k} |\cus_k|$. 
We deduce that
\begin{align*}
\max_{z\in 3^k \Lat \cap \cus_m}
\bfJ\bigl(z {+} \cus_k,\bfAhom^{-\nf 12} e, \bfAhom^{\nf 12} e \,; \a \bigr)
&
\leq 
\sum_{j=-\infty}^k
\frac{|V_j(k,y)|}{|\cus_k|} 
\max_{z\in 3^j\Zd\cap \cu_{m+n_*}}
\bfJ\bigl(z {+} \cu_j,\bfAhom^{-\nf 12} e, \bfAhom^{\nf 12} e \,; \a \bigr) 
\notag\\ & 
\leq 
C \overline{\Pi}^{\nf12}
\sum_{j=-\infty}^k
3^{j-k}
\max_{z\in 3^j\Zd\cap \cu_{m+n_*}}
\bfJ\bigl(z {+} \cu_j,\bfAhom^{-\nf 12} e, \bfAhom^{\nf 12} e \,; \a \bigr) 
\end{align*}
and therefore, using Fubini for sums, 
\begin{align*}
\mathcal{E}_{s,2}(\cus_m;\a,\shom)^2
&
\leq
C s \overline{\Pi}^{\nf12}
\sum_{k=-\infty}^m 
3^{-2s(m-k)} 
\sum_{j=-\infty}^k
3^{j-k} 
\max_{z\in 3^j\Zd\cap \cu_{m+n_*}}
\bfJ\bigl(z {+} \cu_j,\bfAhom^{-\nf 12} e, \bfAhom^{\nf 12} e \,; \a \bigr) 
\notag \\ & 
\leq 
C s (1-2s)^{-1}  \overline{\Pi}^{\nf12}
\sum_{j=-\infty}^m
3^{-2s(m-j)} 
\max_{z\in 3^j\Zd\cap \cu_{m+n_*}}
\bfJ\bigl(z {+} \cu_j,\bfAhom^{-\nf 12} e, \bfAhom^{\nf 12} e \,; \a \bigr)
\notag \\ & 
\leq  
C (1-2s)^{-1}  \overline{\Pi}^{\nf12}
3^{2s n_*} 
\mathcal{E}_{s,2}(\cu_{m+n_*} ; \a,\shom)^2
\notag \\ & 
\leq 
C (1-2s)^{-1}  \overline{\Pi}^{\frac12(1+2s)}
\mathcal{E}_{s,2}(\cu_{m+n_*} ; \a,\shom)^2
\,.
\end{align*}
Applying~\eqref{e.impl.to.mathcal.E},  yields, for every~$m\in \N$, and~$s \in ( \nf\rho2, \nf12)$, 
\begin{align*}
3^{m} \geq \max\bigl\{ \S , \mathcal{R}_{\delta,\rho}  \}
&
\quad \implies \quad 
3^{m+n_*} \geq \max\bigl\{ \S , \mathcal{R}_{\delta,\rho}  \}
\notag \\ & 
\quad \implies \quad 
\mathcal{E}_{s,2} ( \cu_{m+n_*};\a,\shom )^2 
\leq 
\delta 
(\rho - 2s)^{-1} 
\biggl( \frac{3^{m+n_*}}{\max\{ \mathcal{S},\mathcal{R}_{\delta,\rho} \} } \biggr)^{\!-\theta}  
\notag \\ & 
\quad \implies \quad  
\mathcal{E}_{s,2} ( \cus_{m};\a,\shom )^2 
\leq
C \delta (1-2s)^{-2}   \overline{\Pi}^{\frac12(1+2s)}
\biggl( \frac{3^{m}}{\max\{ \mathcal{S},\mathcal{R}_{\delta,\rho}  \} } \biggr)^{\!-\theta}  
\,.
\end{align*}
This completes the proof of~\eqref{e.control.mathcal.E.cus}. 

\smallskip

If we define
\begin{equation*}
\X_s \coloneqq
\bigl(C_{\eqref{e.control.mathcal.E.cus}} 
(1-2s)^{-2}   \overline{\Pi}^{\frac12(1+2s)}
\bigr)^{\nf1\theta} 
\max\{ \mathcal{S},\mathcal{R}_{\delta,\rho} \}
\,, 
\end{equation*}
then we have 
\begin{equation*}
3^m \geq \X_s \quad \implies \quad
\mathcal{E}_{s,2} ( \cus_{m};\a,\shom )^2 
\leq
\delta 
\Bigl( \frac{3^{m}}{\X_s } \Bigr)^{\!-\theta}
\,.
\end{equation*}
By~\eqref{e.control.mathcal.E.cus},\begin{align}
3^n\geq \X_s
& \quad \implies \quad
\max_{k\in \Z\cap [n,\infty) }
\mathcal{E}_{s,2}(\cus_k;\a,\shom) 
\leq 
\sum_{k=n}^\infty 
\mathcal{E}_{s,2}(\cus_k;\a,\shom) 
\leq 
C \delta \Bigl( \frac{3^{n}}{\X_s } \Bigr)^{\!-\theta}
\notag \\
& \quad \implies \quad
\a \in \tilde{G}_{n,\infty}^{s,t}(C\delta (\X_s3^{-n})^\theta ,\ahom) \,, 
\quad \forall t \in [1,\infty]
\,.
\label{e.control.events.G}
\end{align}
By~\eqref{e.alpha.kappa.def},~\eqref{e.cor.sec.four.R} and~\eqref{e.theta.def}, and the selection~$\rho = \frac12(1+s)$, the random variable~$\X_s$ satisfies~\eqref{e.thm.B.X.integrability}. There is additional dependence of the constants~$C$ and~$\theta$ on~$s$ compared to what is stated in the theorem, but this can be removed by simply taking~$s\coloneqq \frac12(1+\gamma)$. This confirms the first bullet of the theorem statement. 

\smallskip

We are now in a position to apply the deterministic black boxes developed in previous subsections. First, in view of~\eqref{e.control.mathcal.E.cus} and the change of variables described above, an application of Proposition~\ref{p.big.black.box} immediately yields the second bullet in the statement of the theorem, namely the estimate~\eqref{e.big.homogenization.estimate}. Here we use both of the statements of the proposition in combination, as in~\eqref{e.combine.bbb.statements}. 

\smallskip

Second, in view of~\eqref{e.control.events.G}, if we take~$\delta$ sufficiently small, then an application of Proposition~\ref{p.big.black.corrector.box} gives the third bullet in the theorem statement, in particular the estimate~\eqref{e.corrector.estimates} and the Liouville result characterizing the space~$\mathcal{A}_1(\Rd)$. In addition,~\eqref{e.corr.C.one.eta} implies~\eqref{e.largescale.C1theta}.

\smallskip

Finally, applying Proposition~\ref{p.big.black.regularity.box} gives us~\eqref{e.largescale.C01}, thereby establishing the last bullet of the theorem. This completes the proof of the theorem. 
\end{proof}

\begin{proof}[{Proof of Theorem~\ref{t.HC.intro}}]
The statement is nearly a special case of Theorem~\ref{t.main}. 
The assumptions of uniform ellipticity and finite range dependence give us the hypotheses of Theorem~\ref{t.main} with parameters
\begin{equation*}
\gamma = 0
\,, \quad 
\mathcal{S} = 0 
\,, \quad 
\Psi(t)  = \exp(ct^2) 
\,, \quad 
\beta = 0 
\,, \quad 
\nu = \tfrac d2 
\,, \quad 
\Psi_\S \equiv + \infty
\,, \quad 
(\lambda_0,\Lambda_0) = ( \lambda,\Lambda)\,.
\end{equation*}
The length scale~$L$ in~\eqref{e.length.scale.L} then satisfies
\begin{equation*}
L \leq
\exp\biggl( C\log^2 \Bigl(1+\frac{\Lambda}{\lambda} \Bigr) \biggr)\,,
\end{equation*}
and we observe that~\eqref{e.thm.B.X.integrability} implies~\eqref{e.thm.A.estimate.for.X}. 
The assertions in the second and third bullets of Theorem~\ref{t.main} then imply the claims of Theorem~\ref{t.HC.intro} in the special case that~$f = 0$ and~$g\in H^{2}$ in~\eqref{e.BVPs}. Indeed, we obtain stronger quantitative estimates than those stated in Theorem~\ref{t.HC.intro}. 
The case of nonzero~$f$ and~$g$ only in~$H^1$ in~\eqref{e.BVPs} follow the estimates on the first-order correctors given in the fourth bullet of Theorem~\ref{t.main} and a standard two-scale expansion argument which can be found for example in~\cite{AKMbook} or~\cite{AK.Book}. This completes the proof. 
\end{proof}

\section{Equations with weaker mixing assumptions}
\label{ss.pigeon.prime}

Our arguments do not require a general mixing condition to hold; we just need a linear concentration for sums of~$\bfA(U_i)$ indexed over certain finite, disjoint~$\{ U_i\}$ (as in~\ref{a.CFS}).  This turns out to be essential for the application considered in~\cite{ABK.SD}, because in that paper, the correlations of the field decay very slowly, and the correlations of the coarse-grained matrices are actually much better. The more general assumption is stated in~\ref{a.CFS.weaker}, below. 

\smallskip

We also allow the ellipticity condition to be slightly weaker by permitting the upper bounds on ellipticity to be given in terms of finite moments (rather than deterministic above a minimal scale).
The assumption is: 

\begin{enumerate}[label=(\textrm{P\arabic*'})]
\setcounter{enumi}{1}
\item \emph{Ellipticity of coarse-grained coefficients.} 
\label{a.ellipticity.weaker}
There exists~$\gamma\in [0,1)$,~$H \in [1,\infty)$,~$D \in [0,\infty)$,~$m_{2} \in \N$, a nondecreasing function~$\Psi_{\S}:\R_+ \to \R_+$ with constants~$K_{\Psi_{\S}} \in [1,\infty)$ and~$ p_{\Psi_\S} \in (2,\infty)$ satisfying, for every~$p \in [2,p_{\Psi_\S}]$, 
\begin{equation*}
s^p \leq K_{\Psi_{\S}}^{3\lceil p\rceil^2} \frac{\Psi_{\S}(ts)}{\Psi_{\S}(t)} \,, \quad \forall t,s\in [1,\infty)\,,
\end{equation*}
such that, for every~$m\in\N$ with~$m \geq m_2$, we have
\begin{equation*}
\sup_{k \in \Z \cap (-\infty,m]} 3^{\gamma(k-m)}  \max_{z\in 3^{k} \Zd  \cap \cu_m} \bigl|  \bigl( \bfAhom^{-\nicefrac12}\!(\cu_m)  \bfA(z + \cu_k) \bfAhom^{-\nicefrac12}\!(\cu_m) - \Itwod \bigr)_+ \bigr| 
\leq
\O_{\Psi_{\S}}(H m^D)
\,.
\end{equation*}
\end{enumerate} 

\begin{enumerate}[label=(\textrm{P\arabic*'})]
\setcounter{enumi}{2}
\item \emph{Concentrations for sums ($\CFS$).} 
\label{a.CFS.weaker}
There exist~$\beta \in \left[0,1\right)$,~$L_1,L_2 \in [1,\infty)$,~$m_{3} \in \N$, a nonincreasing, positive sequence~$\N \ni n \mapsto \omega_n$ with~$\lim_{n\to \infty} \omega_n=0$, an increasing function~$\Psi:\R_+ \to\R_+$ and constants~$K_{\Psi} \in [1,\infty)$ and~$p_{\Psi} \in (d,\infty)$ satisfying the growth condition, for every~$p \in (1,p_{\Psi}]$,
\begin{equation}
\label{e.Psi.growth.prime}
s^p \leq 
K_{\Psi}^{3\lceil p\rceil^2} \frac{\Psi(ts)}{\Psi(t)} 
\,, \quad \forall t,s\in [1,\infty) \,,
\end{equation}
such that, for every~$m,n\in\N$ with~$n \geq m_3$ and~$\beta m < n < m - L_1 \log (L_2 n)$,
\begin{equation} 
\label{e.CFS.weaker}
\biggl| 
 \avsum_{z \in  3^n \Zd \cap \cu_m}  \! \! \!
\bfAhom^{-\nicefrac12}(\cu_n)
\bfA(z+\cu_n)\bfAhom^{-\nicefrac12}(\cu_n)- \Itwod \biggr|
\leq
\O_{\Psi}( \omega_n )   \,.
\end{equation}
\end{enumerate} 

Finally, in this section, we also make an isotropy assumption for simplicity (to avoid working with the adapted cubes~$\cus_n$).

\begin{enumerate}[label=(\textrm{P\arabic*})]
\setcounter{enumi}{3}
\item \emph{Dihedral symmetry:} the law~$\P$ is invariant under negation, reflections and permutations across the coordinate planes. That is, for every matrix~$R$ with exactly one~$\pm 1$ in each row and column and $0$s elsewhere, the law of the conjugated coefficient~$R^t \a(R \cdot) R$ is the same as that of~$\a$.
\label{a.iso}
\end{enumerate} 

Due to the isotropy assumption~\ref{a.iso}, for every~$n\in\N$, the coarse-grained matrices~$\shom(\cu_n)$ and~$\shom_*(\cu_n)$ are both scalar multiples of~$\Id$. Due to the negation symmetry assumption,~$\a$ and~$\a^t$ have the same law and therefore~$\khom(\cu_n) = 0$. We deduce that~$\bfAhom(\cu_n)$ has scalar block matrices. In particular, if~$\bfE =\bfAhom(\cu_n)$, then~$\m_0 = \Id$ and therefore also~$\mathbf{q}_0 = \Id$. Thus, the adapted cubes play no role in this case.

\smallskip

The analog of the estimate~\eqref{e.algebraic.Thetam.main} in~Theorem~\ref{t.main.algebraic} under these assumptions reads as follows. 

\begin{theorem}
\label{t.weaker.P3}
Assume that~$\P$ satisfies~\ref{a.stationarity},~\ref{a.ellipticity.weaker},~\ref{a.CFS.weaker} and~\ref{a.iso}. There exist constants~$C(d)<\infty$ and~$c(d),\alpha(d) \in (0,1)$ such that, by defining
\begin{equation*} 
\Upsilon_1 
 \coloneqq  
\frac{C K_{\Psi}^{4d^2}}{\min\{d+1,p_\Psi\} - d}
\,,\quad 
\Upsilon_2
 \coloneqq  
\frac{C}{\min\{3,p_{\Psi_\S}\} - 2} 
\exp \biggl(C \Bigl( L_1 \log L_2 + \frac{D+\log (H +  K_{\Psi_\S})}{1-\gamma} \Bigr)\biggr)  
\,,
\end{equation*}
then, for every~$m,m_0 \in \N$ satisfying the conditions~$m \geq \max\{m_2,m_3 \}$, 
\begin{equation*} 
\omega_{m}^2
\leq 
\Upsilon_1^{-1}
\qand
m_0
\geq  
\frac{CD}{(1-\beta)(1-\gamma)^2} \log \bigl( (\Upsilon_1 + \Upsilon_2) (m_2 \vee m_3 + m_0) \Theta_0  \bigr) \log( 3\Theta_0) \,,
\end{equation*}
we have, for~$\kappa  \coloneqq  \min\{ \alpha, 1-\gamma\}$ and for every~$n \in \N$ with~$n \geq m + 4m_0$, 
\begin{equation*} 
\Theta_n -1 \leq 
\Upsilon_1 \omega_{m}^2 + C 3^{-\kappa(n-m-4m_0)}
\,.
\end{equation*}
\end{theorem}

We have done most of the work for this theorem already in Sections~\ref{s.firstpigeon} and~\ref{ss.algebraic}. For this reason, we will be somewhat brief with the details.  

\smallskip

We first demonstrate how  assumptions~\ref{a.stationarity},~\ref{a.ellipticity.weaker},~\ref{a.CFS.weaker} and~\ref{a.iso} yield information about  weak norms. The main technical tool is  Lemma~\ref{l.weaknorms.moreproto} and estimates proved in Section~\ref{s.firstpigeon}, such as~\eqref{e.weaknorms.proto.applied} and~\eqref{e.e.weaknorms.proto.applied.second}.

\begin{lemma} 
\label{l.weaknorms.prime}
There exists a constant~$C(d)<\infty$ such that for every~$p,q\in\Rd$ and for every~$n,h \in \N$ with~$h \geq \max\{ (\beta \vee \nf 12) n , m_2 , m_3\} $, 
and for~$\bfE  \coloneqq \bfAhom(\cu_h)$ and~$\m_0,\mathbf{M}_0$ defined by~\eqref{e.m.naught.def} and~\eqref{e.bigM.naught.def}, we have that
\begin{align}  
\label{e.Hminus.s.prime}
\lefteqn{
3^{-n} 
\E\Biggl[ \biggl[
\mathbf{M}_0^{\nicefrac12} 
\begin{pmatrix} 
\nabla v_n -(\nabla v_n)_{\cu_n}   \\ 
\a \nabla v_n-  (\a\nabla v_n)_{\cu_n} 
\end{pmatrix} 
\biggr]_{\Besov{-\nf12}{2}{1}(\cu_{n})}^2
\Biggr]
} \qquad &
\notag \\ &
\leq
C \Theta_h^{\nicefrac12} 
\Bigl | 
\bfAhom^{\nicefrac 12}(\cu_h)  
\Bigl(
\begin{matrix} 
-p \\ q
\end{matrix} 
\Bigr)
\Bigr|^2
\sum_{k=h}^n  3^{\frac12(k-n)}  
\E\Bigl[  \bigl| \bfAhom^{-\nicefrac12}(\cu_h) \bigl( \bfA(\cu_k) - \bfA(\cu_n) \bigr) \bfAhom^{-\nicefrac12}(\cu_h) \bigr|^2 
\Bigl]
\notag \\ & \qquad 
+
C \Theta_h^{\nicefrac12} \Bigl | 
\bfAhom^{\nicefrac 12}(\cu_h)  
\Bigl(
\begin{matrix} 
-p \\ q
\end{matrix} 
\Bigr)
\Bigr|^2
\sum_{k=h}^n  3^{\frac12(k-n)}  
\bigl| \bfAhom^{-\nicefrac12}(\cu_h)  \bigl(\bfAhom(\cu_k) - \bfAhom(\cu_n)  \bigr)  
\bfAhom^{-\nicefrac12}(\cu_h)
\bigr|
\notag \\ & \qquad 
+
C\Theta_h^{\nicefrac12} \Bigl | 
\bfAhom^{\nicefrac 12}(\cu_h)  
\Bigl(
\begin{matrix} 
-p \\ q
\end{matrix} 
\Bigr)
\Bigr|^2
\biggl( 
\frac{K_{\Psi}^{4d^2} \omega_h^2 }{\min\{d+1,p_\Psi\} - d}    + \frac{ 3^{- \frac12 (1-\gamma)(n- h - K \log h - h_c)} }{\min\{3,p_{\Psi_\S}\} - 2}  
\biggr)
\,,
\end{align}
where
\begin{equation} 
\label{e.L.and.hc}
K  \coloneqq   L_1 + \frac{16D}{1-\gamma}
\qand
h_c  \coloneqq  2 K \log (2L_2) + \frac{100}{1-\gamma} (D+\log H +  K_{\Psi_\S} )
\,.
\end{equation}
\end{lemma}
\begin{proof} 
Fix~$p,q\in\Rd$ and~$n,h \in \N$ satisfying~$h \geq \max\{ (\beta \vee \nf 12) n , m_2 , m_3\} $. Set
\begin{equation*}  
\rho'  \coloneqq  \max\Bigl\{ \gamma\,,  \frac{d}{ \min\{d+1, p_\Psi\} } \,, \frac{2}{\min\{3, p_{\Psi_\S}\}}  \Bigr\}
\qand
\rho  \coloneqq  \frac{1+\rho'}{2}
\,.
\end{equation*}
Let~$\mathbf{E} = \bfE  \coloneqq  \bfAhom(\cu_h)$ and~$\mathbf{M} \coloneqq  \mathbf{M}_0$. 
We have that~$| \mathbf{M}^{-\nicefrac12}\mathbf{E} \mathbf{M}^{-\nicefrac12}| \leq C \Theta_h^{\nicefrac12}$. Lemma~\ref{l.weaknorms.moreproto} is applicable\footnote{The geometry of the adapted cubes does not play any particular role in the proof of Lemma~\ref{l.weaknorms.moreproto}.} with~$\cus_{ k} = \cu_k$ for every~$k\in \Z$, and we will estimate different terms appearing on the right in~\eqref{e.weaknorms.moreproto} and estimate them as in~\eqref{e.weaknorms.proto.applied} (with~$l = n-h$  and~$\ep = \nf 12$ there). Notice that the first two terms on the right in~\eqref{e.Hminus.s.prime} are exactly as in~\eqref{e.weaknorms.proto.applied} with~$\bfE = \bfAhom(\cu_h)$.  We show that 
\begin{equation} 
\label{e.mathcalM.m.rho.bound}
\E\bigl[\mathcal{M}_{n,\rho}^2 \bigr] 
\leq
C
\biggl( 
\frac{K_{\Psi}^{4d^2} \omega_h^2 }{\min\{d+1,p_\Psi\} - d}    + \frac{ K_{\Psi_\S}^{36}H^2 n^{2D} 3^{- \frac12 (1-\gamma)(n-h - L_1 \log (L_2 h))} }{\min\{3,p_{\Psi_\S}\} - 2}  
\biggr)
\,,
\end{equation}
from which~\eqref{e.Hminus.s.prime} follows by~\eqref{e.weaknorms.proto.applied} and by~$n^{2D} \leq 3^{\frac14(1-\gamma) K \log h}$ using the lower bound~$h \geq \frac12 n$. To show~\eqref{e.mathcalM.m.rho.bound}, in view of the definition~\eqref{e.event.moreproto}, we claim that
\begin{multline} 
\label{e.this.is.so.nice.again}
\E \Biggl[ \biggl( \max_{k \in \Z \cap (-\infty,n]}  3^{-\rho(n-k)}
 \max_{z\in 3^{k}\Lat  \cap \cu_n}  \Bigl| \bigl( \bfAhom^{-\nicefrac12}(\cu_h) ( \bfA(z + \! \cu_k) -\bfAhom(\cu_h) ) \bfAhom^{-\nicefrac12}(\cu_h)  \bigr)_+\Bigr| \biggr)^{\! 2} \Biggr]
\\  
\leq 
C
\biggl( 
\frac{K_{\Psi}^{4d^2} \omega_h^2 }{\min\{d+1,p_\Psi\} - d}    + \frac{ K_{\Psi_\S}^{36}H^2 n^{2D} 3^{- \frac12 (1-\gamma)(n-h - L_1 \log (L_2 h))} }{\min\{3,p_{\Psi_\S}\} - 2}  
\biggr)
 \,.
\end{multline}
On the one hand, by subadditivity and~\ref{a.CFS.weaker} we have that, for every~$k\in\N$ with~$\beta k \leq h \leq k$ and~$k \geq h + L_1 \log (L_2 h)$, 
\begin{align*} 
\notag
\lefteqn{
\Bigl| \bigl( \bfAhom^{-\nicefrac12}(\cu_h) ( \bfA(z + \! \cu_k) -\bfAhom(\cu_h) ) \bfAhom^{-\nicefrac12}(\cu_h)  \bigr)_+\Bigr|
} \qquad &
\notag \\ &
\leq 
\biggl| 
\avsum_{z'\in z+3^{h} \Lat \cap \cu_{k} }
\bfAhom^{-\nicefrac12}(\cu_h)
\bigl( \bfA(z' + \! \cu_h) - \bfAhom(\cu_h) \bigr)   
\bfAhom^{-\nicefrac12}(\cu_h) 
\biggr|
\leq \O_{\Psi}(\omega_h ) 
\,.
\end{align*}
By a union bound, it follows that, for every~$t \geq 1$ and~$k\in\N$ with~$\beta k \leq h \leq k$ and~$k \geq h + L_1 \log (L_2 h)$, 
\begin{multline} 
\notag
\P\biggl[ 
3^{- \rho(n-k)}  \max_{z \in 3^k \Lat \cap \cu_n}
\Bigl| \bigl( \bfAhom^{-\nicefrac12}(\cu_h) ( \bfA(z + \! \cu_k) -\bfAhom(\cu_h) ) \bfAhom^{-\nicefrac12}(\cu_h)  \bigr)_+\Bigr|
> 
\omega_h  t 
\biggr]
\\
\leq
\frac{3^{d(n-k)} }{\Psi(3^{\rho(n-k)}t)}
\leq
K^{4 d^2} 3^{- (\rho \eta - d)(n-k)} t^{-\eta}
\notag
\,.
\end{multline}
Thus, by another union bound and~\eqref{e.Psi.growth.prime}, with~$h' \coloneqq h + L_1 \log (L_2 h)$, 
\begin{multline*} 
\notag
\P\biggl[ 
\max_{k \in \N \cap [h',n]}  3^{- \rho(n-k)}  \max_{z \in 3^k \Lat \cap \cu_n}
\Bigl| \bigl( \bfAhom^{-\nicefrac12}(\cu_h) ( \bfA(z + \! \cu_k) -\bfAhom(\cu_h) ) \bfAhom^{-\nicefrac12}(\cu_h)  \bigr)_+\Bigr|
> 
\omega_h  t 
\biggr]
\\
\leq
\frac{2 K_{\Psi}^{4 d^2 }}{\eta \rho-d} t^{-\eta}
  \,.
\end{multline*}
On the other hand, since~$\bfAhom(\cu_h) \geq \bfAhom(\cu_n)$,  we have by~\ref{a.ellipticity.weaker} that
\begin{align*} 
\lefteqn{
\P \biggl[
\sup_{k \in \Z \cap (-\infty,h')} \! \! \!\!\!\! \!\!\!\!
3^{-\rho(n-k)} \!\! \!\! \!\!
\max_{z \in 3^k \Zd \cap \cu_n}
\bigl| \bigl( \bfAhom^{-\nicefrac12}(\cu_h)\bfA(z + \! \cu_k)\bfAhom^{-\nicefrac12}(\cu_h) - \Itwod \bigr)_{+} \bigr|
> H n^D 3^{- (\rho-\gamma)(n-h')}  t
 \biggr]
} \qquad &
\notag \\ &
\leq
\P \biggl[
\sup_{k \in \Z \cap (-\infty,n]} 3^{-\gamma(n-k)}  \max_{z \in 3^k \Zd \cap \cu_n} \bigl|\bigl( \bfAhom^{-\nicefrac12}(\cu_n)\bfA(z + \! \cu_k)\bfAhom^{-\nicefrac12}(\cu_n) - \Itwod \bigr)_{+} \bigr|
> H n^D t
\biggr]
\notag \\ &
\leq
\frac{1}{\Psi_\S \bigl( t \bigr)}
\leq
C K_{\Psi_\S}^{36} t^{-\theta}\,.
\end{align*}
Combining the previous two displays yields~\eqref{e.this.is.so.nice.again} in view of the definition of~$\rho$. This concludes the proof.
\end{proof}

\begin{proof}[Proof of Theorem~\ref{t.weaker.P3}]

\emph{Step 1.}
Since~$\khom(\cu_m) = 0$, we have,  for every~$m\in \Z$, 
\begin{equation*}
\Theta_m
= 
\bigl| (\shom_{*}^{-\nf12} \shom \shom_{*}^{-\nf12}) (\cu_m)\bigr|\,,
\end{equation*}
and observe that, since both~$\shom(\cu_m)$ and~$\shom_*(\cu_m)$ are scalar multiples of~$\Id$,~$\XiDet_m = \Theta_m$. Recall that~$\N \ni m \mapsto \Theta_m$ is non-increasing. 

\smallskip

We fix parameters~$\sigma \in (0,1)$ ans~$\delta \in (0,(80d)^{-2}]$ to be selected below to depend only on~$d$. We also let~$K$ and~$h_c$ be as in~\eqref{e.L.and.hc}. We take~$m_1 \in \N$ so large that both
\begin{equation} 
\label{e.omegan.small}
m_1 \geq \max\{3,m_2 , m_3\} \qand 
\frac{\omega_{m_1}^2  K_{\Psi}^{4d^2}}{\min\{d+1,p_\Psi\} - d}
\leq  
\delta^2 \sigma 
\end{equation}
are valid. We define constants~$\Upsilon_1,\Upsilon_2,\Upsilon$ by
\begin{equation*} 
\Upsilon_1   \coloneqq   \frac{K_{\Psi}^{4d^2}}{\min\{d+1,p_\Psi\} - d}\,,
\quad
\Upsilon_2  \coloneqq  
\frac{3^{h_c}}{\min\{3,p_{\Psi_\S}\} - 2} 
\qand
\Upsilon  \coloneqq  \Upsilon_1 \vee \Upsilon_2
\,.
\end{equation*}
Define~$m_0(\delta,\sigma)$ to be the smallest integer satisfying 
\begin{equation*} 
m_0(\sigma,\delta)
\geq 
\frac{400\bigl\lceil 2 \delta^{-1} \left|\log \sigma \right| \bigr\rceil}{1-\beta}
\biggl( \frac{K}{\delta \sigma^2(1-\gamma)}
 \log \Bigl( \frac{ (2m_0(\sigma,\delta)+m_1)^{D+1} \Upsilon \Theta_0}{\sigma^2 \delta} \Bigr)  + 8 h_c \biggr) \log (3\Theta_0) 
\,.
\end{equation*}
Define also
\begin{equation} 
\label{e.L.def.prime}
L = L(\sigma,\delta)  \coloneqq  \frac{100K}{\delta \sigma^2(1-\gamma)}  \log \Bigl( \frac{(2m_0(\sigma,\delta)+m_1)^{D+1}  \Upsilon \Theta_0}{\delta \sigma^2} \Bigr)   
+ 8 h_c
\end{equation}  
and
\begin{equation} 
\label{e.M.def.prime}
N = N(\sigma,\delta) \coloneqq  2L \bigl\lceil 2 \delta^{-1} \left|\log \sigma \right| \bigr\rceil
\,.
\end{equation}  
Observe that
\begin{equation} 
\label{e.prime.m.naught.lower.bound}
m_0(\sigma,\delta) \geq 
\frac{2\log(3\Theta_0)}{1-\beta} N(\sigma,\delta)
\geq 
2\log(3\Theta_0) N(\sigma,\delta)
\end{equation}
and, by the definition of~$h_c$, 
\begin{equation} 
\label{e.prime.L.large.enough}
2K \log(L_2 (m_1+2m_0))
\leq
2K \log(2L_2) +  2K \log(m_1)+ 2K\log(m_0)
\leq
\frac{1}{4} L\,.
\end{equation}

\smallskip

\emph{Step 2.}
We claim that there exists~$\delta_0(d) \in (0,1)$ such that if~$\delta \leq \delta_0$, then
\begin{equation} 
\label{e.prime.step.one}
\Theta_{m_1+2 m_0} \leq 1 + \sigma
\,.
\end{equation}
To show this, we follow the outline of Section~\ref{s.firstpigeon}.

\smallskip

Notice first that, by~\eqref{e.prime.L.large.enough}, for every~$n \in \N$ with~$m_1 \leq n \leq m_1 + 2m_0$, 
\begin{equation} 
\label{e.log.h.small}
L_1 \log (L_2 n)
+
K \log n + h_c  \leq \frac14 L  
\end{equation}
and, by~\eqref{e.omegan.small} and~\eqref{e.L.def.prime}, we have
\begin{equation*} 
\Upsilon_1 \omega_{m_1}^2  
+ 
\Upsilon  3^{- \frac18(1-\gamma)L } 
\leq 
2 \delta \sigma^2 
 \,.
\end{equation*}
Notice also that if~$n \in \N$ is such that~$n \geq m_1 + m_0$, then $2 N(\sigma,\delta) \leq (1-\beta)n$.

\smallskip

Using Lemma~\ref{l.pigeon} with~$h  \coloneqq  2L(\sigma,\delta)$ and~$N=N(\sigma,\delta)$ as in~\eqref{e.M.def.prime}, we deduce that, since~$\XiDet_n = \Theta_n$ for every~$n \in \N$, for every~$m \in \N$ with~$m_1 + m_0 \leq m \leq m_1 + 2m_0 - N$,  at least one of the following two alternatives is valid: 
\begin{itemize} 
\item there exists~$k \in \N \cap [m+2L, m+N]$ such that~$\bfAhom(\cu_{k - 2L} ) \leq (1+\sigma^2\delta) \bfAhom(\cu_{k})$;
or 
\item $\Theta_{m+N}\leq \sigma^2\delta \Theta_{m}$. 
\end{itemize} 
Let us now argue under the assumption that the first alternative is valid. We show that it yields that there exists~$\delta_0(d) \in (0,1)$ such that if~$\delta \leq \delta_0$, then~$\Theta_{m+N} 
\leq 
1+ \frac14 \sigma \Theta_m$. Notice then that both alternatives yield this estimate, provided that $ \delta \leq \nf {1} {4} $. 

\smallskip

The first alternative gives us that, for every~$n \in \N$ with~$k-2L \leq n \leq k$, 
\begin{equation}
\label{e.first.pigeon.launch.prime}
\bigl| \bfAhom^{-\nicefrac12}(\cu_k) \bfAhom(\cu_n) \bfAhom^{-\nicefrac12}(\cu_k)  - \Itwod \bigr|
\leq 
\delta\sigma^2
\,.
\end{equation}
Since~$j \geq \beta n$ for every~$j,n\in\N$ with~$ m_1 + m_0 \leq j \leq n \leq j+L$,~\eqref{e.CFS.weaker},~\eqref{e.log.h.small} and~\eqref{e.omegan.small} yield that, for every~$j,n \in \N$ with~$j \geq k - 2L$ and~$j + \frac14 L \leq n \leq k$, 
\begin{equation} 
\label{e.CFS.weaker.applied}
\E\Biggl[
\biggl| 
 \avsum_{z \in  3^j \Zd \cap \cu_n}  \! \! \!
\bfAhom^{-\nicefrac12}(\cu_j)  \bigl(\bfA(z+\cu_j) - \bfAhom(\cu_j) \bigr) \bfAhom^{-\nicefrac12}(\cu_j)  \biggr|^2 \Biggr]
\leq
K_{\Psi}^{27} \omega_h^2 \leq \delta \sigma^2  \,.
\end{equation}
Using~\eqref{e.first.pigeon.launch.prime} and~\eqref{e.CFS.weaker.applied} in~\eqref{e.variance.HC} (applied with $m=k$ and~$k=j$), we deduce that, for every~$n \in \N$ with~$ k - L \leq n \leq k$, 
\begin{equation} 
\label{e.variance.HC.prime}
\E\Bigl [
\bigl | 
\bfAhom^{-\nicefrac12}(\cu_k)  \bfA(\cu_n)\bfAhom^{-\nicefrac12}(\cu_k)    - \Itwod  
\bigr |^{2}
\Bigr ] 
\leq C \delta \sigma^2 
\,.
\end{equation}
Set next~$\bfE  \coloneqq  \bfAhom(\cu_{k-L})$ and~$\thom_k  \coloneqq   (\bhom \#\, \shom_{*})(\cu_k)$,  and take, for~$e \in \Rd$ with~$|e|=1$,  
\begin{equation*} 
\begin{pmatrix}
p \\ q
\end{pmatrix}
 \coloneqq 
\begin{pmatrix}
\thom_k^{-\nicefrac12} e \\ \thom_k^{\nicefrac12} e
\end{pmatrix}
\qand
\begin{pmatrix}
P \\ Q
\end{pmatrix}
\coloneqq
\E\biggl[ \fint_{\cu_k} 
\begin{pmatrix}
\nabla v(\cdot,\cu_k,p,q) \\ \a \nabla v(\cdot,\cu_k,p,q)
\end{pmatrix}
 \biggr]
 \,.
\end{equation*}
We insert these choices into~\eqref{e.divcurl.conclusion.pre} and obtain, for every~$\ep \in (0,1]$, 
\begin{align*}  
\biggl| \E\bigl  [ J(\cu_k)\bigr] -\frac12 P \cdot Q \biggr|
& 
\leq 
50\ep^{-1} \bar{\tau}_{k,k-L} 
+
4 \ep
\bigl(  \bigl| \shom_{*}^{-\nicefrac12}(\cu_k) Q \bigr|  +  \bigl| \bhom^{\nicefrac12}(\cu_k) P \bigr| \bigr)^2
\notag \\ &  \quad 
+
C 
\sum_{j=-\infty}^{k-L} 
3^{j-k}  
 \avsum_{z \in 3^j\Lat \cap \cus_{k-4}} 
 \bigl( \bigl|\shom_{*}^{-\nicefrac12}(z +\cu_j) Q \bigr| + \bigl|\bhom^{\nicefrac12}(z +\cu_j) P \bigr|  \bigr)^2
\notag \\ & \qquad 
+
C 3^{-k} 
\E \Biggl[
\biggl[ 
\mathbf{M}_0^{\nicefrac12} 
\begin{pmatrix} 
\nabla v_k - P \\ 
\a \nabla v_k - Q 
\end{pmatrix} 
\biggr]_{\Besov{-\nf12}{2}{1}(\cu_{k})}^2
\Biggr] 
\,.
\end{align*}
By selecting~$\ep  \coloneqq  \delta^{\nicefrac12} \sigma \Theta_{m}^{-\nicefrac12}$ and following the arguments leading to~\eqref{e.PQ.bound.pre} and~\eqref{e.pq.bounds}, we obtain
\begin{equation} 
\label{e.pq.bounds.prime}
\biggl|
\bfAhom^{\nicefrac 12}(\cu_k)
\begin{pmatrix} 
-p \\ q
\end{pmatrix}
\biggr|^2 
\leq C \Theta_m^{\nicefrac12}
\qand
| \b_0^{\nicefrac12}P | + | \s_{*,0}^{-\nicefrac12}Q |  \leq C \Theta_m^{\nicefrac34} 
\end{equation}
and, consequently, as in~\eqref{e.what.is.bartau}, by~\eqref{e.first.pigeon.launch.prime} and the previous display, we get 
\begin{equation*}
\bar{\tau}_{k,k-L}  \leq C \delta \sigma^2 \Theta_{m}^{\nicefrac12}   \,.
\end{equation*}
Furthermore, by~\eqref{e.first.pigeon.launch.prime}, we have
\begin{equation*} 
\sum_{j=k-2L}^{k-L} 
3^{j-k+L} 
 \avsum_{z \in 3^j\Lat \cap \cus_n} 
 \bigl( \bigl|  \shom_{*,0}^{\nicefrac12}   \shom_{*}^{-1}(z+ \cu_j) \shom_{*,0}^{\nicefrac12}  \bigr| + 
\bigl|  \bhom_{0}^{-\nicefrac12}   \bhom(z+\cu_j) \bhom_{0}^{-\nicefrac12}    \bigr|   \bigr)\leq C  
\end{equation*}
and, by~\eqref{e.this.is.so.nice.again},~\eqref{e.omegan.small} and the definition of~$L$,
\begin{equation*} 
\sum_{j=-\infty}^{k-2L} \!\!\!
3^{j-k+L} 
\! \! \! \! \! \avsum_{z \in 3^j\Zd \cap \cu_n} \! \! \!  
\bigl( \bigl|  \shom_{*,0}^{\nicefrac12}   \shom_{*}^{-1}(z+\! \cu_j) \shom_{*,0}^{\nicefrac12}  \bigr| + 
\bigl|  \bhom_{0}^{-\nicefrac12}   \bhom(z+\cu_j) \bhom_{0}^{-\nicefrac12}    \bigr|   \bigr)
\leq 
C
\,.
\end{equation*}
Thus,
\begin{equation*} 
\sum_{j=-\infty}^{k-L} \!\!\!
3^{j-k+L} 
\! \! \! \! \! \avsum_{z \in 3^j\Zd \cap \cus_n} \! \! \!  
\bigl( \bigl|  \shom_{*,0}^{\nicefrac12}   \shom_{*}^{-1}(z+\! \cu_j) \shom_{*,0}^{\nicefrac12}  \bigr| + 
\bigl|  \bhom_{0}^{-\nicefrac12}   \bhom(z+\cu_j) \bhom_{0}^{-\nicefrac12}    \bigr|   \bigr)
\leq
C
\,.
\end{equation*}
Combining the above displays then yields 
\begin{align*}  
\biggl| \E\bigl  [ J(\cu_n)\bigr] -\frac12 P \cdot Q \biggr|
& 
\leq 
C \bigl(\sigma \delta^{\nicefrac12} + 3^{-L} \Theta_m^{\nf 12} \bigr) \Theta_m 
+
C 3^{-n} 
\E \Biggl[
\biggl[ 
\mathbf{M}_0^{\nicefrac12} 
\begin{pmatrix} 
\nabla v_n - P \\ 
\a \nabla v_n - Q 
\end{pmatrix} 
\biggr]_{\Besov{-\nf12}{2}{1}(\cu_{n})}^2
\Biggr] 
\,.
\end{align*}
To estimate the last term on the right, we use Lemma~\ref{l.weaknorms.prime} (applied with~$n =k$  and~$h = k-L$),~\eqref{e.pq.bounds.prime},~\eqref{e.first.pigeon.launch.prime},~\eqref{e.variance.HC.prime} and to deduce that
\begin{equation*}
3^{-k} 
\E\Biggl[ \biggl[
\mathbf{M}_0^{\nicefrac12} 
\begin{pmatrix} 
\nabla v_k -(\nabla v_k)_{\cu_k}   \\ 
\a \nabla v_k- (\a\nabla v_k)_{\cu_k}
\end{pmatrix} 
\biggr]_{\Besov{-\nf12}{2}{1}(\cu_{k})}^2
\Biggr]
\leq  
C \delta \sigma^2\Theta_m
\,.
\end{equation*}
By similar estimates for~$J^*$, Lemma~\ref{l.Jminusmeans.to.Theta} and~\eqref{e.L.def.prime} yield 
\begin{equation*}
\Theta_{m}-1
\leq \Theta_k - 1 \leq  
C \bigl(\sigma \delta^{\nicefrac12}  + 3^{-L} \Theta_m^{\nicefrac12} \bigr) \Theta_m 
\leq 
C \sigma \delta^{\nicefrac12} \Theta_m
\,.
\end{equation*}
Using the monotonicity of~$m\mapsto \Theta_m$, we then deduce that there exists~$\delta_0(d) \in (0,1)$ such that, if~$\delta \leq \delta_0$, then 
\begin{equation*}
\Theta_{m+N} 
\leq 
1+ \frac14 \sigma \Theta_m
\,,
\end{equation*}
as claimed. By iterating this~$n$ times with~$n$ satisfying~$nN \leq m_0 < (n+1)N$, we obtain by the monotonicity of~$m\mapsto \Theta_m$ and~\eqref{e.prime.m.naught.lower.bound} that
\begin{equation*}
\Theta_{m_1+2m_0}   
\leq
\Theta_{m_1+m_0 + nN}   
\leq 
\frac{1}{1-\nf \sigma 4}
+ \Bigl( \frac{\sigma}{4}\Bigr)^n \Theta_0
\leq 
1 + \frac12 \sigma + \frac12 \sigma \exp(1-n) \Theta_0 \leq 1+\sigma
\,,
\end{equation*}
which gives us~\eqref{e.prime.step.one}.

\smallskip

\emph{Step 3.}
We next claim that there exist constants~$\alpha(d),\sigma(d) \in (0,(80d^2)^{-1})$ such that, for~$\kappa  \coloneqq  \min\{ \alpha ,1-\gamma \}$ and for every~$n \in \N$ with~$n \geq m_1 + 4m_0(\delta_0,\sigma)$, 
\begin{equation} 
\label{e.induction.prime}
\Theta_n - 1 \leq \min\Bigl\{ \sigma ,  \sigma^{-1} \bigl( \Upsilon_1 \omega_{m_1}^2 +  3^{-\kappa(n-m_1-4m_0)} \bigr) \Bigr\}
\,.
\end{equation}
We assume inductively that there exists $m\in \N$ with~$m> m_1 + 4 m_0$  such that~\eqref{e.induction.prime} is valid for every~$n \in \N$ with~$n \in [m_1 + 4m_0 , m-1]$. The initial step $n = m_1 + 4m_0$ is valid by Step 1. Our goal is to show that~\eqref{e.induction.prime} is true also for~$n=m$, which then proves the induction step. 

\smallskip

To prove~\eqref{e.induction.prime} with~$n=m$, we first go back to the proof of Lemma~\ref{l.variance}. Noting that~$\hat{\Theta}_m \coloneqq  \Theta_m$ (since the adapted cubes are Euclidean in our setting and~$\shom(\cu_m)$ and~$\shom_*(\cu_m)$ are scalar multiples of identity matrix) and repeating Step 1 in the proof of Lemma~\ref{l.variance}, we find that there exists~$C(d)<\infty$ such that, for every~$k,n \in \N$ with~$m_1 + 2m_0 \leq k \leq n$, 
\begin{equation} 
\label{e.algebraic.add.error.again}
\bigl| \bfAhom^{-\nicefrac12}(\cu_n)  \bfAhom(\cu_k) \bfAhom^{-\nicefrac12}(\cu_n)    - \Itwod \bigr| \leq 4d(\hat{\Theta}_k - \hat{\Theta}_n) \leq \frac1{40}
\,.
\end{equation}
We then revisit Step 2 in the proof of Lemma~\ref{l.variance}. As a consequence of~\eqref{e.variance.proto} we deduce that, for every~$k,m,n \in \N$ with~$2m_0 \leq k \leq n \leq m$ we have that 
\begin{align*} 
\notag
\lefteqn{
\E\Bigl  [ \bigl | \,
\bfAhom^{-\nicefrac12}(\cu_m)  
\bfA(\cu_n)  
\bfAhom^{-\nicefrac12}(\cu_m)  
-\Itwod
\bigr |^{2} \Bigr  ]
} \qquad &
\notag \\ &
\leq
C (\hat{\Theta}_k - 1)^2 
+ 
C
\E \Biggl[
\biggl |
\avsum_{z \in 3^k \Zd \cap \cu_n}
\bfAhom^{-\nicefrac12}(\cu_m)   \bigl ( \bfA(z+\cu_k) - \bfAhom_k\bigr ) \bfAhom^{-\nicefrac12}(\cu_m)  
\biggr |^2
\Biggr ] 
\,.
\end{align*}
If~$\max\{m_1 + 2 m_0, \lceil \beta m \rceil \} < k \leq m - L_1 \log (L_2 k)$, then, by~\ref{a.CFS.weaker} and~\eqref{e.algebraic.add.error.again},
\begin{equation*}
\E \Biggl[
\biggl |
\avsum_{z \in 3^k \Zd \cap \cu_n}
\bfAhom^{-\nicefrac12}(\cu_m)  
 \bigl ( \bfA(z+\cu_k) - \bfAhom_k\bigr ) 
 \bfAhom^{-\nicefrac12}(\cu_m)  
\biggr |^2
\Biggr ] 
\leq C K_\Psi^{27} \omega_k^2 \leq C \Upsilon_1 \omega_{m_1}^2
\,.
\end{equation*}
We thus obtain, for every~$n \in \N$ with~$m_1 + 3m_0 \leq n \leq m$,  that
\begin{align}
\label{e.variance.again.prime}
 \E\Bigl  [ \bigl | \,
\bfAhom^{-\nicefrac12}(\cu_m)  
\bfA(\cu_n)  
\bfAhom^{-\nicefrac12}(\cu_m)  
-\Itwod
\bigr |^{2} \Bigr ]
& 
\leq 
C (\hat{\Theta}_{n - \lceil L_1 \log (L_2n)\rceil} - 1)^2 + C  K_\Psi^{27}  \omega_{n - \lceil L_1 \log (L_2n)\rceil}^2
\notag \\ & 
\leq
C \Bigl(
\Upsilon_1 \omega_{m_1}^2 +  3^{-\kappa(n-m_1-4m_0)} \Bigr)\,.
\end{align}
We then follow the proof of Proposition~\ref{p.algebraic.exp}. Let~$\bfE  \coloneqq  \bfA(\cu_{m})$ and set
\begin{equation*}
\begin{pmatrix}
p \\ q
\end{pmatrix}
=
\begin{pmatrix}
\shom_{*}^{-\nicefrac12}(\cu_m)e \\ \shom_{*}^{\nicefrac12}(\cu_m)e
\end{pmatrix}
\qand
\begin{pmatrix}
P \\ Q
\end{pmatrix}
\coloneqq
\E\biggl[ \fint_{\cu_m} 
\begin{pmatrix}
\nabla v(\cdot,\cu_m,p,q) \\ \a \nabla v(\cdot,\cu_m,p,q)
\end{pmatrix}
 \biggr]
\,.
\end{equation*}
By repeating the argument for~\eqref{e.Thetam.byJ.agh} and~\eqref{e.divcurl.conclusion.pre.applied}, we deduce that there exists a constant~$C(d)<\infty$ such that
\begin{align*}
\hat{\Theta}_n -1
\leq
C 3^{-n}
\E \Biggl[\biggl[ 
\mathbf{M}_0^{\nicefrac12} 
\begin{pmatrix} 
\nabla v_n - P \\ 
\a \nabla v_n - Q 
\end{pmatrix} 
\biggr]_{\Besov{-\nf12}{2}{1}(\cu_{n})}^2
\Biggr]
+
C (\hat{\Theta}_{n-4} - \hat{\Theta}_n ) 
\end{align*}
for small enough~$\sigma(d)$. 
Lemma~\ref{l.weaknorms.prime}, together with~\eqref{e.algebraic.add.error.again} and~\eqref{e.variance.again.prime}, yields that
\begin{align} 
\lefteqn{
3^{-n} 
\E\Biggl[ \biggl[
\mathbf{M}_0^{\nicefrac12} 
\begin{pmatrix} 
\nabla v_n -  (\nabla v_n)_{\cu_n}   \\ 
\a \nabla v_n- (\a\nabla v_n)_{\cu_n}
\end{pmatrix} 
\biggr]_{\Besov{-\nf12}{2}{1}(\cu_{n})}^2
\Biggr]
} \qquad &
\notag \\ &
\leq
C  \! \! \! \!
\sum_{k=m_1+ 3m_0}^n  \! \! \! \! 3^{\frac12(k-n)}  
\E\Bigl[  \bigl| \bfAhom^{-\nicefrac12}(\cu_m)  \bigl( \bfA(\cu_k) - \bfAhom(\cu_k) \bigr)\bfAhom^{-\nicefrac12}(\cu_m)   \bigr|^2 
\Bigl]
\notag \\ & \qquad 
+
C \! \! \! \!
\sum_{k=m_1+ 3m_0}^n \! \! \! \! 3^{\frac12(k-n)}  
\bigl| \bfAhom^{-\nicefrac12}(\cu_m)   \bigl(\bfAhom(\cu_k) - \bfA(\cu_n)  \bigr)  \bfAhom^{-\nicefrac12}(\cu_m)  
\bigr|
\notag \\ & \qquad 
+ C \Bigl(
\Upsilon_1 \omega_{m_1}^2 +  3^{-\kappa(n-m_1-4m_0)} \Bigr)
\notag \\ &
\leq 
C  \! \! \! \! \sum_{k=m_1 + 3m_0}^n \! \! \! \! 3^{\frac12(k-n)} \bigl( \hat{\Theta}_{k} - \hat{\Theta}_{n}  \bigr)
+
C \Bigl(
\Upsilon_1 \omega_{m_1}^2 +  3^{-\kappa(n-m_1-4m_0)} \Bigr)
\,.
\notag
\end{align}
Combining this with~\eqref{e.variance.again.prime} leads to
\begin{equation*}
\hat{\Theta}_n -1 
\leq 
C  \sum_{k=m_1 + 4m_0}^n  3^{\frac12(k-n)} \bigl( \hat{\Theta}_{k} - \hat{\Theta}_{m}  \bigr)
+ C \Bigl(
\Upsilon_1 \omega_{m_1}^2 +  3^{-\kappa(n-m_1-4m_0)} \Bigr)
\,.
\end{equation*}
We have that
\begin{equation*} 
\sum_{k=m_1 + 4m_0}^n 3^{- 2\kappa_0 (n - k)} 3^{-(1-\gamma)(k -m_1- 4m_0)}  
\leq 
\frac{2}{\kappa_0} 3^{- \min\{1-\gamma ,\kappa_0 \}(n -m_1- 4m_0) } 
\,.
\end{equation*}
As in the proof of Proposition~\ref{p.algebraic.exp}, we obtain by iteration that there exist constants~$\alpha(d) \in (0,1)$ and~$C(d)<\infty$ such that, with~$\kappa  \coloneqq  \min\{\alpha, \frac12 (1-\gamma)\}$,  
\begin{align} 
\notag
\Theta_m - 1 
& \leq 
 \frac{C}{\kappa_0 \alpha}  \Bigl(
\Upsilon_1 \omega_{m_1}^2 +  3^{-\kappa(n-m_1-4m_0)} \Bigr)
\,.
\notag
\end{align}
Therefore, by taking~$\sigma  \coloneqq  C^{-1} \kappa_0 \alpha$, which then depends only on~$d$, we obtain~\eqref{e.induction.prime} for~$n=m$, proving the induction step.  

\smallskip

\emph{Step 4.} The conclusion. The result follows by~\eqref{e.induction.prime} after relabelling~$\Upsilon$ and replacing~$m_0$ by~$Cm_0$ for a large constant~$C$, if necessary, using the monotonicity of~$m \mapsto \Theta_m$. 
\end{proof}

\appendix

\section{Besov spaces and functional inequalities}
\label{s.Besov.appendix}

Recall the definition of the Besov seminorms given in~\eqref{e.Bs.seminorm}, which is 
\begin{equation}
\label{e.Bs.seminorm.A}
\left[ g \right]_{\underline{B}_{p,q}^{s}(\cu_{n})}
 \coloneqq 
\Biggl( 
\sum_{k=-\infty}^n
3^{- s q k} \biggl( 
\avsum_{z\in 3^{k-1}\Zd, \, z + \cu_k \subseteq \cu_n}
\bigl \| g - (g)_{z+\cu_k}\bigr \|_{\underline{L}^p(z+\cu_k)}^p
\biggr)^{\!\nicefrac qp}
\Biggr)^{\! \nicefrac1q}\,.
\end{equation}
The case~$q=\infty$ is defined in~\eqref{e.Bs.seminorm.infty} and the negative seminorms are given in~\eqref{e.Bs.minus.seminorm} and not repeated here. 
In~\eqref{e.Bs.minus.seminorm.explicit} we also defined the quantity
\begin{equation}
\label{e.Bs.minus.seminorm.explicit.A}
\left[ f \right]_{\Besov{-s}{p}{q}(\cu_n)}
 \coloneqq 
\Biggl(
\sum_{k=-\infty}^n
3^{s qk}
\biggl(
\avsum_{z\in 3^k\Zd \cap \cu_n}
\bigl| (f)_{z+\cu_k}\bigr |^p
\biggr)^{\!\nicefrac qp}
\Biggr)^{\! \nicefrac1q}
\,.
\end{equation}
We write~$f \in \Besovnoul{-s}{p}{q}(\cu_{n})$ if this quantity is finite. 
Throughout, the H\"older conjugate exponent of~$p \in [1,\infty]$ is denoted by~$p' \coloneqq  p/(p-1)$. 

\smallskip

The next lemma shows that the negative seminorm~$\left[ \cdot \right]_{\Bhatminusul{-s}{p}{q} (\cu_{n})}$ is upper bounded by~$\left[ \cdot \right]_{\Besov{-s}{p}{q}(\cu_n)}$.

\begin{lemma} 
\label{l.duality.Bs}
Let~$p,q\in [1,\infty]$ and~$s\in (0,1)$, with~$s=1$ allowed if~$q=1$. For every~$m\in \Z$ and~$f \in \Besovnoul{-s}{p}{q}(\cu_{m})$, 
\begin{equation}
\label{e.weak.norms.ordering.A}
\left[ f \right]_{\Bhatminusul{-s}{p}{q} (\cu_{m})}
\leq
3^{d+s} 
 \left[ f \right]_{\Besov{-s}{p}{q}(\cu_{m})}
\,.
\end{equation}
\end{lemma}
\begin{proof}
For each~$k \in (-\infty,m] \cap \Z$, we introduce the following approximation of~$g$:
\begin{equation*} 
[g]_k  \coloneqq  \sum_{z \in 3^k \Zd \cap \cu_m} (g)_{z+\cu_k} \indc_{z+\cu_k}\,.
\end{equation*}
By the Lebesgue differentiation theorem, we have
\begin{equation*} 
g = 
(g)_{\cu_m} 
+ 
\sum_{k=-\infty}^m 
\bigl( [g]_{k-1} - [g]_k \bigr)
\,.
\end{equation*}
We may therefore write
\begin{align*}
&
\fint_{\cu_m}
f   g
=
(f)_{\cu_m}  (g)_{\cu_m} 
+
\sum_{k=-\infty}^m \,
\avsum_{z\in 3^k\Zd\cap \cu_m} \,
\avsum_{y\in z+ 3^{k-1}\Zd\cap \cu_k}
\big(
(g)_{y+\cu_{k-1}} - (g)_{z+\cu_k}
\big)\cdot (f)_{y+\cu_{k-1}}
\,.
\end{align*}
By Jensen's inequality, for~$y \in z + 3^{k-1} \Zd \cap \cu_k$, we have
\begin{equation*} 
\big|
(g)_{y+\cu_{k-1}} - (g)_{z+\cu_k}
\big|
\leq 
\| g - (g)_{z+\cu_k} \|_{\underline{L}^1(y+\cu_{k-1})}
\leq 3^d \| g - (g)_{z+\cu_k} \|_{\underline{L}^1(z+\cu_{k})}
 \,.
\end{equation*}
Denote the H\"older conjugate exponent of~$p \in [1,\infty]$ by~$p' \coloneqq  p/(p-1)$, and similarly for~$q$. By H\"older's inequality,  for every~$p,q\in (1,\infty)$,
\begin{align*}  
\lefteqn{
\sum_{k=-\infty}^m  
\avsum_{z\in 3^k\Zd\cap \cu_m}
 \avsum_{y\in z+ 3^{k-1}\Zd\cap cu_k} \bigl|
\bigl( (g)_{y+\cu_{k-1}} - (g)_{z+\cu_k}
\bigr) \cdot  (f)_{y+\cu_{k-1}} \bigr| 
} 
\qquad & 
\\ & 
\leq 
3^d\sum_{k=-\infty}^m  
\avsum_{z\in 3^k\Zd\cap \cu_m} 
\| g - (g)_{z+\cu_k} \|_{\underline{L}^1(z+\cu_k)}  
\avsum_{y\in 3^{k-1}\Zd\cap(z+\cu_k)} \bigl| (f)_{y+\cu_{k-1}} \bigr|
\\ & 
\leq 
3^{d+s} \sum_{k=-\infty}^m 3^{-sk} 
 \biggl( \avsum_{z\in 3^k\Zd\cap \cu_m} \| g - (g)_{z+\cu_k} \|_{\underline{L}^{1}(z+\cu_k)}^{p'}  \biggr)^{\! \nicefrac {1}{p'}} 
 3^{s(k-1)} 
\biggl( \avsum_{z\in 3^{k-1}\Zd\cap \cu_m} \bigl| (f)_{z+\cu_{k-1}} \bigr|^p  \biggr)^{\! \nicefrac {1}{p}} 
\\ & 
\leq 
3^{d+s}
\left[ g  \right]_{\underline{B}_{p',q'}^{s}(\cu_{m})} 
\biggl( 
\sum_{k=-\infty}^{m-1} 
 3^{sq k} 
\biggl( \avsum_{z\in 3^{k}\Zd\cap \cu_m} \bigl| (f)_{z+\cu_{k}} \bigr|^p  \biggr)^{\! \nicefrac {q}{p}} 
 \biggr)^{\! \nf 1q}
\,.
\end{align*}
We may add the term~$|(g)_{\cu_m}| |(f)_{\cu_m}|$ to this inequality and absorb it in the sum on the right side by incrementing the sum to include~$k=m$. This yields, in the case~$p,q \in (1,\infty)$, that
\begin{equation}
\label{e.duality.Bs}
\fint_{\cu_m} f g \leq
3^{d+s} 
\| g  \|_{\underline{B}_{p',q'}^{s}(\cu_m)}
 \left[ f \right]_{\Besov{-s}{p}{q}(\cu_{m})}
\,.
\end{equation}
The case that one of the exponents~$p$ and~$q$ belongs to~$\{1,\infty\}$ is obtained similarly. The inequality~\eqref{e.duality.Bs} implies~\eqref{e.weak.norms.ordering.A} by the definition of the seminorms with negative regularity in~\eqref{e.Bs.minus.seminorm}. 
\end{proof}

The next lemma provides a Poincar\'e inequality. 
\begin{lemma} 
\label{l.multiscale.SP.C}
There exists a constant~$C(d)<\infty$ such that, for every~$m\in\Z$ and~$u \in L^2(\cu_m)$ with~$\nabla u\in\Besovnoul{-1}{2}{1}(\cu_{m})$, we have
\begin{equation} 
\label{e.multiscale.SP.C}
\left\| u - (u)_{\cu_m} \right\|_{\underline{L}^2(\cu_m)}
\leq
C 
\| \nabla u  \|_{\Besov{-1}{2}{1}(\cu_{m})}
 \,.
\end{equation}
\end{lemma}

\begin{proof}
Assume that~$(u)_{\cu_m} = 0$. 
Let~$w \in W^{2,2}(\cu_m)$ solve~$-\Delta w = u$ with zero Neumann data.
By the classical Calder\'on-Zygmund estimate, we have that 
\begin{equation*} 
\| \nabla w\|_{\underline{H}^{1}(\cu_m)} 
\leq
C\| u\|_{\underline{L}^2(\cu_m)}
 \,.
\end{equation*}
By the Poincar\'e inequality,
\begin{equation*} 
[ \nabla w]_{\underline{B}_{2,\infty}^1(\cu_m) }
=
\sup_{k \in (-\infty,m]} 
3^{-2k} \avsum_{z\in 3^{k-1}\Zd, \,  z + \cu_k \subseteq \cu_m}
\bigl \| \nabla w - (\nabla w)_{z+\cu_k}\bigr \|_{\underline{L}^2(z+\cu_k)}^2 
\leq
C \| u \|_{\underline{L}^2(\cu_m)}^2
 \,.
\end{equation*} 
Using this and testing the equation of~$w$ with~$u$, we deduce by~$(u)_{\cu_m} =0$ and~\eqref{e.duality.Bs} that
\begin{align*} 
\| u \|_{\underline{L}^2(\cu_m)}^2
& = 
\fint_{\cu_m}  \nabla w \cdot \nabla u
\leq
C \| \nabla w \|_{\underline{B}_{2,\infty}^{1}(\cu_m)} 
\| \nabla u  \|_{\Besov{-1}{2}{1}(\cu_{n})}
\leq
C \| u  \|_{\underline{L}^{2}(\cu_m)}
\| \nabla u  \|_{\Besov{-1}{2}{1}(\cu_{m})}
\,.
\end{align*}
This concludes the proof after an application of Young's inequality and reabsorption. 
\end{proof}

We next generalize the previous lemma to more general Besov spaces. 

\begin{lemma}
\label{l.dualitylemma}
There exists~$C(d)<\infty$ such that, for every~$m \in \Z$,~$s \in [0,1)$ and~$u \in B^{s}_{2,\infty}(\cu_{m})$,
\begin{equation}  
\label{e.divcurl.est0}
\left\| u- (u)_{\cu_m}  \right\|_{\underline{B}_{2,\infty}^{s}(\cu_m)}
\leq 
C
\left[ \nabla u  \right]_{\Besov{s-1}{2}{1}(\cu_{m})}
\,.
\end{equation}
Moreover, if~$\varphi \in W^{2,\infty}(\cu_m)$, then
\begin{equation}
\label{e.divcurl.est1}
\left\| (u- (u)_{\cu_m}) \nabla \varphi \right\|_{\underline{B}_{2,\infty}^{s}(\cu_m)}
\leq 
C
3^{m}
 \| \nabla \varphi\|_{\underline{W}^{1,\infty}(\cu_m)}
\left[ \nabla u  \right]_{\Besov{s-1}{2}{1}(\cu_{m})}
\,.
\end{equation}
\end{lemma}
\begin{proof}
Fix~$s\in (0,1)$ and~$m\in\Z$.
Without loss of generality, assume that~$(u)_{\cu_m} = 0$. 
Throughout the proof we denote the lattice appearing in~\eqref{e.Bs.seminorm} by~$\mathcal{Z}_k  \coloneqq  \{ z\in 3^{k-1}\Zd \, : \,  z + \cu_k \subseteq \cu_m\}$ for each~$k \in \Z$ with~$k \leq m$. Fix also~$\phi \in C_c^\infty(\cu_m)$ satisfying~$3^{m}  \| \nabla \phi \|_{L^\infty(\cu_m)} + 3^{2m} \| \nabla^2 \phi \|_{L^\infty(\cu_m)} \leq 1$. 

\smallskip

\emph{Step 1.} The proof of~\eqref{e.divcurl.est0}. We apply~\eqref{e.multiscale.SP.C} to get
\begin{align*}   
\| u  - (u )_{z+\cu_k} \|_{\underline{L}^2(z+\cu_k)}^2
&
\leq
 C \sum_{j=-\infty}^{k} 3^{j} \biggl(  \avsum_{z' \in z + 3^j \Z^d \cap  \cu_k}  \bigl | (\nabla u)_{z' + \cu_j} \bigr |^2 \biggr)^{\! \nf 12}
\,.
\notag
\end{align*}
Therefore, we obtain by H\"older's inequality that
\begin{align*}
\lefteqn{
3^{-2sk}\avsum_{z\in \mathcal{Z}_k} \! \! 
\| u  - (u )_{z+\cu_k} \|_{\underline{L}^2(z+\cu_k)}^2
} \qquad &
\notag \\  &
\leq
C 3^{-2sk}
\avsum_{z\in \mathcal{Z}_k}  \Biggl( \sum_{j=-\infty}^{k} 3^{j} 
\biggl( \avsum_{z' \in z+ 3^j \Z^d \cap  \cu_n}  \bigl | (\nabla u)_{z' + \cu_j} \bigr |^2 \biggr)^{\! \nicefrac12}  
\Biggr)^{\! 2}
\notag \\  &
=
C 3^{-2sk}
\sum_{j,j'=-\infty}^{k} 3^{j' + j} 
\avsum_{z\in \mathcal{Z}_k}  
\biggl( \avsum_{z' \in z+ 3^j \Z^d \cap  \cu_n}  \bigl | (\nabla u)_{z' + \cu_j} \bigr |^2 \biggr)^{\! \nicefrac12}  
\biggl( \avsum_{z' \in z+3^{j'} \Z^d \cap  \cu_n}  \bigl | (\nabla u)_{z' + \cu_{j'}} \bigr |^2 \biggr)^{\! \nicefrac12}  
\notag \\  &
\leq
C 3^{-2sk}
\sum_{j,j'=-\infty}^{k} 3^{j' + j} 
\biggl( \avsum_{z \in\mathcal{Z}_j}  \bigl | (\nabla u)_{z + \cu_j} \bigr |^2 \biggr)^{\! \nicefrac12}  
\biggl( \avsum_{z \in\mathcal{Z}_{j'}}  \bigl | (\nabla u)_{z + \cu_{j'}} \bigr |^2 \biggr)^{\! \nicefrac12}  
\notag \\  &
\leq
C
\Biggl(\sum_{j=-\infty}^{m} 3^{(1-s)j} 
\biggl( \avsum_{z \in 3^j \Zd \cap \cu_m}  \bigl | (\nabla u)_{z + \cu_j} \bigr |^2 \biggr)^{\! \nicefrac12}  
\Biggr)^2
\,.
\end{align*}
Thus~\eqref{e.divcurl.est0} follows by taking supremum over~$k \in \Z \cap (-\infty,m]$. 

\smallskip

\emph{Step 2.} The proof of~\eqref{e.divcurl.est1}. By the definition~\eqref{e.Bs.seminorm.A},
\begin{equation*}  
\left[ u \nabla \varphi \right]_{\underline{B}_{2,\infty}^{s}(\cu_m)}
=
\sup_{k \in (-\infty,m] \cap \Z}
3^{- s k} \biggl( 
\avsum_{z\in \mathcal{Z}_k}
\bigl \| u \nabla \varphi - (u \nabla \varphi)_{z+\cu_k}\bigr \|_{\underline{L}^2(z+\cu_k)}^2
\biggr)^{\! \nicefrac12}
\,.
\end{equation*}
By the triangle inequality and the Poincar\'e inequality, we get
\begin{align*}  
\lefteqn{
\avsum_{z\in \mathcal{Z}_k}
\bigl \| v \nabla \varphi - (v \nabla \varphi)_{z+\cu_k}\bigr \|_{\underline{L}^2(z+\cu_k)}^2
} \qquad &
\notag \\ & 
\leq 
4
\avsum_{z\in \mathcal{Z}_k} \Bigl( 
\bigl \| \nabla \varphi  - (\nabla \varphi)_{z+\cu_k} \bigr \|_{\underline{L}^2(z+\cu_k)}^2 | (v)_{z+\cu_k} |^2 
+
\bigl \| (v - (v)_{z+\cu_k} ) \nabla \varphi \bigr \|_{\underline{L}^2(z+\cu_k)}^2
\Bigr)
\notag \\ & 
\leq
C3^{k} \| \nabla^2 \varphi \|_{L^\infty(\cu_m)}  \| v \|_{\underline{L}^2(\cu_m)}
+
C \| \nabla \varphi \|_{L^\infty(\cu_m)} 
\biggl( 
\avsum_{z\in \mathcal{Z}_k}
\bigl \| v - (v)_{z+\cu_k}\bigr \|_{\underline{L}^2(z+\cu_k)}^2
\biggr)^{\! \nicefrac12}
\,.
\end{align*}
Since $| (u \nabla \varphi)_{\cu_m}| \leq \| \nabla \varphi \|_{L^{\infty}(\cu_m)} \| u \|_{\underline{L}^2(\cu_m)}$, we obtain that
\begin{equation}
\label{e.divcurl.est1.pre}
\left\| u \nabla \varphi \right\|_{\underline{B}_{2,\infty}^{s}(\cu_m)} 
\leq
C \bigl( \| \nabla \varphi \|_{L^\infty(\cu_m)} + 3^{m} \| \nabla^2 \varphi \|_{L^{\infty}(\cu_m)} 
\bigr) \left\| u \right\|_{\underline{B}_{2,\infty}^{s}(\cu_m)} 
\,.
\end{equation}
An application of~\eqref{e.divcurl.est0} then concludes the proof.  
\end{proof}

The next lemma shows that the seminorm of~$\underline{B}^s_{p,p}(\cu_m)$ is equivalent with the one of~$\underline{W}^{s,p}(\cu_m)$.

\begin{lemma} 
\label{l.Wsp.vs.Bspp}
There exists a constant~$C(d)<\infty$ such that, 
for every~$s\in (0,1)$,~$p\in [1,\infty)$,~$m \in \Z$ and~$u \in W^{s,p}(\cu_m)$, 
\begin{equation} 
\label{e.Wsp.vs.Bspp}
C^{-1} [  u ]_{\underline{W}^{s,p}(\cu_m)}
\leq
[  u ]_{\underline{B}^s_{p,p}(\cu_m)} 
\leq
C [  u ]_{\underline{W}^{s,p}(\cu_m)}
\,.
\end{equation}
Moreover, for every~$s\in (0,1)$,~$p\in [1,\infty)$,~$m \in \Z$ and~$f \in  \Besovnoul{-s}{p}{p}(\cu_{m})$,
\begin{equation} 
\label{e.Wminus.sp.vs.Bminus.sp}
[  f ]_{\Wminusul{-s}{p}(\cu_m)} \leq 
C \left[ f \right]_{\Besov{-s}{p}{p}(\cu_{m})} 
 \,.
\end{equation}
\end{lemma}
\begin{proof}
In what follows, we use the notation~$A \approx_{d} B$ to denote the existence of a constant~$C(d)$ such that~$A \approx_{d} B$ implies~$C^{-1} A \leq B \leq C$. 

\smallskip

Let~$\{\psi\}_j$ be a partition of unity so that~$\sum_{j\in\Z} \psi_j = 1$,~$\psi_j$ is supported in~$\cu_{j+2} \setminus \cu_j$,~$\sum_{j=n-1}^{n+1} \psi_j = 1$ in the support of~$\psi_n$ and~$\|\nabla \psi_j \|_{L^\infty(\Rd)} \leq C 3^{-j}$. 
We compute, for every~$h\in\N$ with~$h\geq 1$, 
\begin{align*}
\lefteqn{[  u ]_{\underline{W}^{s,p}(\cu_m)}
=
\biggl(
\fint_{\cu_m} \int_{\cu_m} 
\frac{|u(x) - u(y)|^p}{|x-y|^{d+sp}} 
\, dx \, dy
\biggr)^{\! \nf1p}
} \qquad & 
\notag \\ &
\approx_{d}
\biggl( 
\sum_{n=-\infty}^{m+1} 
3^{-n(d+sp)}
\fint_{\cu_m} \int_{\cu_m} |u(x) - u(y)|^p  \psi_{n}(x-y) \, dx \, dy 
\biggr)^{\! \nf1p}
\notag \\ &
\approx_{d}
\biggl( 
\sum_{n=-\infty}^{m+1} 
3^{-n(d+sp)} \avsum_{z\in 3^{n-h}\Zd \cap \cu_m}  
\fint_{(z+\cu_{n-h+1}) \cap \cu_m} \int_{\cu_m} |u(x) - u(y)|^p  \psi_{n}(x-y)   \, dy \, dx
\biggr)^{\! \nf1p}
\,.
\end{align*}
For~$y\in \cu_m$ and~$x \in z + \cu_{n-h+4}$ with~$h \geq 10$, we have
\begin{align*}
\bigl| \psi_{n}(x-y) 
-
\psi_{n}(z-y) \bigr|
= 
\biggl| 
(x-z) \cdot \int_{0}^1 \nabla \psi_n(t x + (1-t)z - y) \, dt  
\biggr|
\leq
C 3^{-h}\sum_{j=n-2}^{n+2}\psi_j(x-y) \,.
\end{align*} 
Taking~$h$ sufficiently large, depending only on~$d$, we deduce that 
\begin{align*} 
[  u ]_{\underline{W}^{s,p}(\cu_m)}^p
\geq 
c^p
\sum_{n=-\infty}^{m+1} 
3^{-n(d+sp)}
\avsum_{z\in 3^{n-h}\Zd \cap \cu_m}  
\fint_{(z+\cu_{n-h+4}) \cap \cu_m} \int_{\cu_m} |u(x) - u(y)|^p  \psi_{n}(z-y)   \, dy \, dx
\,.
\end{align*}
Next, since~$\int_{\cu_m} \psi_{n}(z-y) \, dy \geq c | \cu_n|$ for every~$z \in \cu_m$ and~$n \leq m-1$, Jensen's inequality yields
 \begin{align*}
\int_{\cu_m} |u(x) - u(y)|^p \psi_{n}(z-y) \, dy
\geq 
c^p  | \cu_n| \biggl| u(x) - \biggl( \int_{\cu_m} \psi_{n}(z-y)\,dy \biggr)^{-1} \int_{\cu_m}  u(y) \psi_{n}(z-y)\,dy  \biggr|^p\,.
\end{align*}
Therefore, 
\begin{multline*} 
\fint_{(z+\cu_{n-h+4}) \cap \cu_m} \int_{\cu_m} |u(x) - u(y)|^p  \psi_{n}(z-y)   \, dy \, dx
\\
\geq 
c^p | \cu_n| 
\inf_{a \in \R}
\|u - a \|_{\underline{L}^p((z+\cu_{n-h+4}) \cap \cu_m)}^p
\geq 
c^p|\cu_n| 
\|u - (u)_{(z+\cu_{n-h+4})\cap \cu_m } \|_{\underline{L}^p((z+\cu_{n-h+4}) \cap \cu_m)}^p
\,.
\end{multline*}
We therefore obtain
\begin{align*}
[  u ]_{\underline{W}^{s,p}(\cu_m)}^p
\geq 
c^p 
\sum_{n=-\infty}^{m-1} 
3^{-nsp}
\avsum_{z\in 3^{n-h}\Zd \cap \cu_m}  
\|u - (u)_{(z+\cu_{n-h+4})\cap \cu_m } \|_{\underline{L}^p((z+\cu_{n-h+4}) \cap \cu_m)}^p
\end{align*}
Using the triangle inequality and the fact that the cubes are overlapping, it is easy to check that the right side of this display is equivalent to the expression in the middle of~\eqref{e.Wsp.vs.Bspp}. This completes the proof of the lower bound in~\eqref{e.Wsp.vs.Bspp}. 

\smallskip

For the upper bound in~\eqref{e.Wsp.vs.Bspp}, we compute
\begin{align*} 
\lefteqn{
\sum_{n=-\infty}^{m+1} 
3^{-n(d+sp)}
\fint_{\cu_m} \int_{\cu_m} |u(x) - u(y)|^p  \psi_{n}(x-y) \, dx \, dy 
} \qquad & 
\notag  \\ &  
\leq 
C 
\sum_{n=-\infty}^{m} 
3^{-nsp} \avsum_{z \in 3^{n} \Zd \cap \cu_m } 
\fint_{(z+ \cu_{n+3}) \cap \cu_m} \fint_{(z+ \cu_{n+3}) \cap \cu_m} |u(x) - u(y)|^p  \, dx \, dy
\notag  \\ &  
\leq 
C 
\sum_{n=-\infty}^{m} 
3^{-nsp} \avsum_{z \in 3^{n} \Zd \cap \cu_m } \inf_{a \in \R}\| u - a\|_{\underline{L}^p( (z+ \cu_{n+1}) \cap \cu_m)}^p
 \,.
\end{align*}
Finally,~\eqref{e.Wminus.sp.vs.Bminus.sp} is a consequence of~\eqref{e.weak.norms.ordering.A} and~\eqref{e.Wsp.vs.Bspp}. The proof is now complete. 
\end{proof}

In Section~\ref{ss.blackbox}, we use the following fractional weighted Poincar\'e-type inequality. 

\begin{proposition}[{Fractional Hardy-Poincar\'e inequality~\cite{DLV}}]
\label{p.fractional.hardy}
For every~$s\in (0,\nf12)$ and bounded Lipschitz domain~$U \subseteq \Rd$, there exists~$C(U,d)<\infty$, depending only on~$d$ and the Lipschitz character of~$U$, such that, for every~$u\in H^s(U)$,
\begin{equation}
\int_{U} 
\dist(x,\partial U)^{-2s} 
\bigl|u(x) - (u)_U \bigr|^2 
\leq 
Cs^{-1}(1-2s)^{-1} [ u ]_{H^s(U)}^2
\,.
\label{e.fractional.hardy}
\end{equation} 
\end{proposition}
\begin{proof}
See~\cite[Remark 3.3]{DLV}.
\end{proof}

\section{Geometric means for positive matrices}
\label{s.geomean}

Recall that if~$A,B>0$ are positive real numbers, then the minimum of the map 
\begin{equation*}
x \mapsto 
\frac 12x^{-\nicefrac12} A x^{-\nicefrac12} + \frac 12x^{\nicefrac12}B^{-1} x^{\nicefrac12} 
\end{equation*}
is attained uniquely at~$x = A\# B$, where $A\#B  \coloneqq  (AB)^{\nicefrac12}$ is the geometric mean of~$A$ and~$B$, and the minimum is equal to~$A\#B^{-1}$.
It turns out that this fact can be generalized to positive definite matrices. 

\smallskip

There are two different notions of geometric mean for positive definite matrices. The first one, introduced by Ando~\cite{Ando}, is called the \emph{metric geometric mean}. It is defined for any pair of positive definite matrices~$A$ and~$B$ by 
\begin{equation*}
A\#B = A^{\nicefrac12} \bigl(A^{-\nicefrac12} B  A^{-\nicefrac12}  \bigr)^{\nicefrac12} \! A^{\nicefrac12} \,.
\end{equation*}
The matrix~$A\#B$ is the unique positive definite matrix solution~$X$ of the equation
\begin{equation}
\label{e.ricatti}
X A^{-1} X = B. 
\end{equation}
We see from this characterization that the metric geometric mean is symmetric in~$A$ and~$B$, that is,~$A\#B = B\#A$.

\smallskip

A second notion of geometric mean for positive definite matrices  was introduced later by Fiedler and Pt\'{a}k~\cite{FP}, is the \emph{spectral geometric mean}~$A\,\natural B$ of two matrices~$A$ and~$B$ is defined by
\begin{equation*}
A\,\natural B
 \coloneqq 
(A^{-1}\#B)^{\nicefrac12} A (A^{-1}\#B)^{\nicefrac12}
\,.
\end{equation*}
It is also characterized by the following identity that relates it to the metric geometric mean:
\begin{equation*}
A^{-1} \# (A\,\natural B ) = B \# (A\,\natural B )^{-1}\,.
\end{equation*}
It gets its name from the fact that~$(A\,\natural B)^2$ is positively similar to~$AB$. In fact, there exists a positive definite matrix~$C$ such that
\begin{equation*}
A\,\natural B = C A C = C^{-1} B  C^{-1}
\,,
\end{equation*}
and this property characterizes~$A\,\natural B$. It also follows from this characterization that the spectral geometric means are also symmetric in~$A$ and~$B$, that is,~$A\natural B = B\natural A$. Since~$(A\,\natural B)^2$ is similar to~$AB$, the eigenvalues of~$A\,\natural B$ are the square roots of those of~$AB$ or~$BA$ or~$A^{\nicefrac12}B A^{\nicefrac12}$ or~$B^{\nicefrac12}A B^{\nicefrac12}$. This property gives it its name. 

The largest eigenvalue of~$A\,\natural B$ is larger than the largest eigenvalue of~$A\# B$, while this relation is reversed for the smallest eigenvalue. In particular,~$|A\# B| \leq  |A\,\natural B|$. It turns out that~$A\# B = A\,\natural B$ if and only if~$A$ and~$B$ commute. If they do not commute, there is no relation between the two in the Loewner partial order. All of the facts asserted above can be found in~\cite{FP}.

\smallskip

It turns out that the map
\begin{equation*}
X \mapsto 
\frac 12X^{-\nicefrac12} A X^{-\nicefrac12} + \frac 12X^{\nicefrac12}B^{-1} X^{\nicefrac12} 
\end{equation*}
is minimized by~$X= A\# B$ and the minimum is equal to~$A \natural B^{-1}$.

\section{Orlicz quasi-norms and concentration inequalities}
\label{ss.big.O}

\subsection{The~\texorpdfstring{$\mathcal{O}_{\Psi}$}O-Psi notation for weak Orlicz quasi-norms}

If~$\Psi:\R_+\to [1,\infty)$ is an increasing function satisfying the mild growth condition
\begin{equation*}
\lim_{t\to \infty} \frac1t \Psi(t) = +\infty\,,
\end{equation*}
and~$X$ is a random variable, then we use the notation~``$X \leq \O_{\Psi}(A)$'' for~$A\geq 0$ as a shorthand for the statement
\begin{equation*}
\P \bigl[ X > tA \bigr] 
\leq 
\frac{1}{\Psi(t)}\,, \quad \forall t \in [1,\infty)\,.
\end{equation*}
For example, we could write~\eqref{e.S.integrability} as~$\S \leq \O_{\Psi_\S}(1)$ and, similarly,~\eqref{e.CFS} can be written as
\begin{equation*}
\biggl| \,
\avsum_{z\in 3^n\Zd\cap \cu_m}  X_z
\biggr| 
\leq 
\O_{\Psi} \bigl( 3^{-\frac d2(m-n)} \bigr)
\,.
\end{equation*}
Thus, the notation~``$X \leq \O_{\Psi}(A)$'' is an alternative way of writing that a weak Orlicz quasi-norm of~$X$ is bounded by~$A$. We note, however, that we do \emph{not} require~$\Psi$ to be convex here. 
We also write~$X = \O_{\Psi}(A)$ to mean that~$|X| \leq \O_{\Psi}(A)$. 

\smallskip

If~$\Psi_1$ and~$\Psi_2$ are two such functions, then we write~$X \leq \O_{\Psi_1}(A_i)$ to mean that~$X$ can be written as the sum~$X = Y_1+Y_2$ of two random variables~$Y_1$ and~$Y_2$ satisfying~$Y_i = \O_{\Psi_i}(A_i)$ for~$i\in\{1,2\}$.

\smallskip

For this~$\O_{\Psi}$ notation to be useful, we need to have other properties such as a (generalized) triangle inequality for infinite sums and sufficient growth of~$\Psi$ that~$X=\O_{\Psi}(A)$ implies bounds on finite moments of~$X$. 
In the next lemma, we prove such properties under the growth condition~\eqref{e.Psi.growth} used in our assumptions. 

\begin{lemma}
\label{l.Psi.growth}
Suppose that~$\Psi: \R_+ \to [1,\infty)$ is an increasing function and~$K_\Psi\in [2,\infty)$ such that
\begin{equation*}
t \Psi(t) \leq \Psi(K_\Psi  t), \quad \forall t \in [1,\infty)\,.
\end{equation*}
Then we have the following:
\begin{itemize}
\item For every~$p\in [1,\infty)$, 
\begin{equation}
\label{e.eat.the.poly}
\frac{t^p}{\Psi(t)} 
\leq 
\frac1{\Psi\bigl(K_\Psi^{-4\lceil p\rceil}t \bigr)}\,, \quad \forall \, t \in \bigl[ K_\Psi^{4\lceil p\rceil} ,\infty \bigr)\,,
\end{equation}

\item $\Psi$ satisfies the following minimal growth bound:
\begin{equation}
\label{e.explog2.growth}
\Psi(t) \geq 
\exp \biggl( \frac{\log^2 t}{9 \log K_\Psi} \biggr)\,, \quad \forall t \in [K_{\Psi}^2,\infty)\,.
\end{equation}

\item For every~$p\in[2,\infty)$, 
\begin{equation}
\label{e.Psi.doubling}
s^p \leq 
K_{\Psi}^{3p^2} \frac{\Psi(ts)}{\Psi(t)} 
\,, \quad \forall t,s\in [1,\infty)\,.
\end{equation}

\item For any random variable~$X$ and~$a\in (0,\infty)$ and~$p\in[1,\infty)$, 
\begin{equation}
\label{e.Psi.moments.bound}
X = \O_\Psi(a) 
\ \implies \
\E [ X^p ] \leq 
a^p \Bigl( 1+2p K_\Psi^{\lceil \frac12 p(p+1) \rceil} \bigl(1+\log K_\Psi \bigr) \Bigr)\,.
\end{equation}

\item For any sequence~$\{ X_k \}_{k\in\N}$ of random variables,
\begin{equation}
\label{e.Psi.triangle}
X_k \leq \O_\Psi (a_k) 
\ \implies \  
\sum_{k\in\N} X_k \leq \O_\Psi \biggl( 4 K_{\Psi}^{7} \!\sum_{k\in\N} a_k \biggr)\,.
\end{equation}
\end{itemize}

\end{lemma}
\begin{proof}

\emph{Step 1.}
We begin with the observation that~\eqref{e.Psi.growth} allows us to absorb powers of~$t$ in front of~$\Psi$ by a dilation of~$\Psi$. By induction, it implies that, for every power~$k \in \N$, 
\begin{equation}
\label{e.eat.poly.factors.pre}
t^k \Psi(t) 
\leq 
K_\Psi^{-\frac12 k(k-1)} \Psi(K_\Psi^{k} t)\,, \quad \forall t \in [1,\infty)\,.
\end{equation}
In particular, for every~$p\in [1,\infty)$,
\begin{equation}
\label{e.eat.poly.factors}
t^p \Psi(t) 
\leq 
\Psi\bigl(K_\Psi^{\lceil p\rceil} t\bigr)
\,, \quad \forall t\in [1,\infty)\,.
\end{equation}
This implies that, for every~$p\in [1,\infty)$,
\begin{equation*}
\frac{t^p}{\Psi(t)}
\leq 
\bigl( K_\Psi^{2\lceil 2p\rceil} t^{-1} \bigr)^{p} 
\frac{1}{\Psi\bigl(K_\Psi^{\lceil 2p\rceil} t\bigr)}\,,
\quad \forall t \geq K_\Psi^{\lceil 2p\rceil} \,.
\end{equation*}
If we restrict to~$t \geq K_\Psi^{4\lceil p\rceil}$, then we can ignore the first factor. This yields~\eqref{e.eat.the.poly}. 

\smallskip 

\emph{Step 2.} The proof of~\eqref{e.explog2.growth}.
We observe that~\eqref{e.eat.poly.factors.pre} implies, for every~$p\in \N\cap [1,\infty)$ and~$t \in [ K_\Psi^{p}, K_\Psi^{p+1} ]$,
\begin{equation}
\label{e.polylift}
\Psi(t) \geq \frac{\Psi(t)}{\Psi(1)} 
\geq 
K_\Psi^{\frac12 p(p-1)} 
\geq 
t^{\frac{p(p-1)}{2(p+1)}}
\geq 
\exp \biggl( \frac{p(p-1)}{2(p+1)} \log t \biggr)
\geq 
\exp \biggl( \frac{p(p-1)}{2(p+1)^2} \frac{\log^2 t}{\log K_\Psi} \biggr)
\,.
\end{equation}
Specializing to~$p\geq2$ yields~\eqref{e.explog2.growth}.

\smallskip

\emph{Step 3.}
The proof of~\eqref{e.Psi.doubling}.
As in~\eqref{e.polylift}, we use~\eqref{e.eat.poly.factors.pre} to find that, for every~$p\in\N$, we have
\begin{equation*}
\frac{\Psi(ts)}{\Psi(t)} 
\geq s^{\frac{p(p-1)}{2(p+1)}}\,,
\quad \forall t\in [1,\infty)\,,  \ 
s\in [ K_\Psi^p,K_{\Psi}^{p+1}]
\,.
\end{equation*}
The upper bound restriction on~$s$ can clearly be removed, and so we obtain
\begin{equation*}
\frac{\Psi(ts)}{\Psi(t)} 
\geq s^{\frac{p(p-1)}{2(p+1)}}\,,
\quad \forall t\in [1,\infty)\,,  \ 
s\in [ K_\Psi^p,\infty)\,,
\end{equation*}
and this implies
\begin{equation*}
\frac{\Psi(ts)}{\Psi(t)} 
\geq 
K_\Psi^{-\frac{p^2(p-1)}{2(p+1)}} s^{\frac{p(p-1)}{2(p+1)}}\,,
\quad \forall t\in [1,\infty)\,,  \ 
s\in [ 1,\infty)\,,
\end{equation*}
Restricting to~$p\geq 5$, we get 
\begin{equation*}
\frac{\Psi(ts)}{\Psi(t)} 
\geq 
K_\Psi^{-\nicefrac{p^2}{3}} s^{\nicefrac{p}{3}}\,,
\quad \forall t\in [1,\infty)\,,  \ 
s\in [ 1,\infty)\,,
\end{equation*}
This implies~\eqref{e.Psi.doubling}.

\smallskip 

\emph{Step 4.} We prove~\eqref{e.Psi.moments.bound}.
For any random variable~$X$ satisfying~$X= \O_{\Psi}(a)$, we have
\begin{equation*}
a^{-n} \E\bigl[|X|^n\bigr]  
= n \int_0^\infty t^{n-1}\P\bigl[|X|> a t\bigr] \, dt  
\leq 1 + n \int_1^\infty \frac{t^{n-1}}{\Psi(t)} \, dt \,.
\end{equation*}
Using~\eqref{e.eat.poly.factors.pre}, we find that, for every $j\in N$ and~$t\in [K_\Psi^{j},\infty)$,
\begin{equation*}
\Psi(t)
\geq K_\Psi^{-\frac{1}{2}j(j+1)} t^{j} \Psi(K_\Psi^{-j} t) \geq K_\Psi^{-\frac{1}{2}j(j+1)} t^{j}
\,.
\end{equation*}
Using this, we obtain
\begin{align} 
\notag
\int_1^\infty \frac{t^{n-1}}{\Psi(t)} \, dt   
& \leq 
\sum_{j=0}^n  \int_{K_\Psi^{j}}^{K_\Psi^{j+1}} \frac{t^{n-1}}{\Psi(t)} \, dt  
+ 
\int_{K_\Psi^{n+1}}^\infty \frac{t^{n-1}}{\Psi(t)} \, dt
\notag \\ &
\leq 
\sum_{j=0}^n K_\Psi^{\frac{1}{2}j(j+1)}  \int_{K_\Psi^{j}}^{K_\Psi^{j+1}} t^{n-j-1} \, dt   
 + n K_\Psi^{\frac12 (n+2)(n+1)} \int_{K_\Psi^{n+1}}^\infty t^{-2} \, dt 
\notag \\ &
\leq 
\sum_{j=0}^{n-1} \frac{1}{n-j} K_\Psi^{\frac{1}{2}j(j+1)+(n-j)(j+1)}
+
\bigl(1+\log K_\Psi \bigr) K_\Psi^{\frac12 n(n+1)}
\notag \\ &
\leq 
2 \bigl(1+\log K_\Psi \bigr) K_\Psi^{\frac12 n(n+1)}
\,.
\notag
\end{align}
This yields, for every~$n\in\N$, 
\begin{equation*}
\E\bigl[|X|^n\bigr]  
\leq 
a^n \Bigl( 1+2n K_\Psi^{\frac12 n(n+1)} \bigl(1+\log K_\Psi \bigr) \Bigr)\,,
\end{equation*}
completing the proof of~\eqref{e.Psi.moments.bound}. 

\smallskip

\emph{Step 5.} We show that~\eqref{e.Psi.doubling} implies the generalized triangle inequality~\eqref{e.Psi.triangle}.
We assume only that~$\Psi$ satisfies, for some~$p,C_0\in (1,\infty)$,
\begin{equation}
\label{e.weird.poly.doubling}
s^p \leq C_0 \frac{\Psi(ts)}{\Psi(t)} 
\,, \quad \forall t,s\in [1,\infty)\,.
\end{equation}
We argue that this condition~\eqref{e.weird.poly.doubling} implies the following generalized triangle inequality: there exists a constant~$C(p,C_0)<\infty$ (given explicitly below in~\eqref{e.CpC0}), such that, if~$\{ X_k \}_{k\in\N}$ is any sequence of random variables, then 
\begin{equation*}
X_k \leq \O_\Psi (a_k) 
\ \implies \  
\sum_{k\in\N} X_k \leq \O_\Psi \biggl( C\!\sum_{k\in\N} a_k \biggr)\,.
\end{equation*}
To see this, we set~$X \coloneqq  \sum_{k} X_k$, $a \coloneqq  \sum_{k} a_k$, fix~$t>0$ and compute
\begin{align*}
\P \bigl[ X > t \bigr]
\leq 
\P \biggl[ \sum_{k\in\N} X_k \indc_{\{ X_k   > ta_k/2a\} } 
> \frac12 t
\biggr]
\leq
\frac2t \sum_{k\in\N} \E \bigl[ X_k \indc_{\{ X_k   > ta_k/2a\} } \bigr]
\,.
\end{align*}
We then observe that 
\begin{align*}
\E \bigl[ X_k \indc_{\{ X_k   > ta_k/2a\} } \bigr]
&
=
\int_0^\infty \P \bigl[ X_k \indc_{\{ X_k   > ta_k/2a\} } > s  \bigr]\,ds
\\ & 
\leq
\frac{ta_k}{2a}
\P \biggl[ X_k > \frac{ta_k}{2a} \biggr]
+\int_{ta_k/2s}^\infty
\P \bigl[ X_k  > s  \bigr]\,ds
\\ &
\leq
\frac{ta_k}{2a}
\frac{1}{\Psi(t/2a)}
+
\int_{ta_k/2a}^\infty
\frac{1}{\Psi (s/a_k)}
\,ds
\\ &
\leq
\frac{1}{\Psi(t/2a)}
\biggl( \frac{ta_k}{2a}
+
\int_{ta_k/2a}^\infty
\frac{\Psi(t/2a)}{\Psi (s/a_k)}
\,ds \biggr)
\\ &
\leq
\frac{1}{\Psi(t/2a)}
\biggl( \frac{ta_k}{2a}
+
C_0 \int_{ta_k/2a}^\infty
\biggl( \frac{2as}{a_kt} \biggr)^{\!-p}\,ds \biggr)
=
\frac{1}{\Psi(t/2a)} \biggl(\frac{ta_k}{2a} \biggr) \biggl( 1 + \frac{C_0}{p-1} \biggr)\,.
\end{align*}
Here, we used~\eqref{e.weird.poly.doubling} in the last line. Inserting this into the previous display after summing it over~$k\in\N$ gives
\begin{equation*}
\P \bigl[ X > t \bigr]
\leq 
\frac{1+C_0(p-1)^{-1}}{\Psi(t/2a)}
\,.
\end{equation*}
Using~\eqref{e.weird.poly.doubling} again, we can bound the right side by~$\Psi(t/Ca)^{-1}$ with
\begin{equation}
\label{e.CpC0}
C(p,C_0) \coloneqq 2(C_0+C_0^2(p-1)^{-1})^{\nicefrac1p}
\,,
\end{equation}
which then implies~$X = \O_\Psi(C a)$, as claimed. 
If we have~\eqref{e.Psi.doubling}, then we may take~$p=2$ and~$C_0 = K_\Psi^{7}$ and we find that~$C(p,C_0) \leq 4 K_{\Psi}^{7}$, as claimed in~\eqref{e.Psi.triangle}.  
This completes the proof of the lemma. 
\end{proof}

We may also use~\eqref{e.weird.poly.doubling} in the equivalent form
\begin{equation}
\label{e.eat.it}
\frac{s}{\Psi(t)} 
\leq 
\frac1{\Psi\bigl(t(sC_0)^{-\nicefrac1p} \bigr)}\,,
\quad
\forall s \in\bigl[ C_0^{-1},\infty \bigr) \,, \ t \in \bigl[ (sC_0)^{\nicefrac1p},\infty\bigr)\,.
\end{equation}

\subsection{Examples of functions satisfying the growth condition}

We introduce, for every~$\sigma \in (0,\infty)$, the function
\begin{equation*}
\Gamma_\sigma (t)  \coloneqq  \exp \bigl(  t^\sigma \bigr)\,.
\end{equation*}
If~$\sigma\geq 1$, then the function~$\Gamma_\sigma$ is nonnegative, increasing on~$\R_+$, convex and satisfies~$t \Gamma_\sigma(t) \leq \Gamma_\sigma(2t)$ for every~$t\geq 1$ and~$\Gamma_\sigma$ satisfies the generalized triangle inequality~\eqref{e.Psi.triangle} with constant~$C=1$. 
In the case~$\sigma\in (0,1)$,~$\Gamma_\sigma$ is convex on the interval~$((\frac{1-\sigma}{\sigma})^{\nicefrac 1\sigma},\infty)$, the growth condition is satisfied for
\begin{equation*}
K_{\Gamma_\sigma} = \Bigl( \frac{\sigma+1}{\sigma} \Bigr)^{\nicefrac 1\sigma}\,.
\end{equation*}
and the generalized triangle inequality is valid with constant~$C = 2\sigma^{-\nf1\sigma}$.
For every~$\sigma_1, \sigma_2 \in (0, \infty)$ and random variables~$X_1$ and$X_2$, 
\begin{equation*}
X_1 \leq \O_{\Gamma_{\sigma_1}}(A_1) \qand X_2 \leq \O_{\Gamma_{\sigma_2}}(A_2) \implies X_1 X_2 \leq \O_{\Gamma_{\frac{\sigma_1 \sigma_2}{\sigma_1 + \sigma_2}}}(A_1 A_2) \, . 
\end{equation*}
For any~$\sigma,p,K\in(0,\infty)$ and random variable~$X$,
\begin{equation*}
X \leq \Gamma_{\sigma}(K)
\iff
X^{p} \leq \Gamma_{\sigma/p} (K^p)\,.
\end{equation*}
A normal random variable~$X$ with zero mean and variance~$\gamma^2>0$ satisfies
\begin{equation*}
X = \O_{\Gamma_2}(\gamma)\,.
\end{equation*}
If~$\sigma,A>0$ and~$X_1,\ldots,X_N$ is a sequence of random variables satisfying~$X_i = \O_{\Gamma_\sigma}(A)$ with~$N \geq 2$, then 
\begin{equation*}
\max_{1\leq i \leq N} X_i = \O_{\Gamma_\sigma} \bigl( (3 \log N)^{\nicefrac1\sigma} A\bigr)\,.
\end{equation*}
To see this, we observe that, for every~$t \geq 1$, we use a union bound to estimate
\begin{align*}
\P \bigl[ \max_{1\leq i \leq N} X_i > A (3 \log N)^{\nicefrac{1}{\sigma}} t  \bigr]
&\leq \sum_{i=1}^N \P \bigl[ X_i > A (3 \log N)^{\nicefrac{1}{\sigma}} t\bigr] 
\\&
\leq N \exp\Bigl( -3 t^{\sigma} \log N \Bigr)
\leq \exp\Bigl(-t^{\sigma} (3 \log N - \log N)  \Bigr)
\leq  \exp(-t^\sigma) \,.
\end{align*}
The indicator function~$\indc_E$ of an event~$E$ with~$0<\P[E] <1$ satisfies, for every~$\sigma\in (0,\infty)$, 
\begin{equation}
\label{e.indc.O.sigma}
\indc_{E} \leq \O_{\Gamma_\sigma} \bigl( \bigl| \log \P[E] \bigr|^{-\nicefrac 1\sigma} \bigr) 
\,.
\end{equation}
This is immediate from the definitions. 

\smallskip

We have seen in~\eqref{e.explog2.growth} that any admissible~$\Psi$ grows at least like~$t\mapsto c\exp( c\log^2 t)$.
Conversely, this function satisfies the growth condition, and it is an important example since it characterizes the integrability of log-normal random variables. For instance, 
\begin{equation*}
\Psi_1(t) = \exp \Bigl( \log^2 (1+t) \Bigr)
\end{equation*}
is nonnegative, increasing on~$[0,\infty)$ and satisfies~\eqref{e.Psi.growth} with~$K_{\Psi_1} = 10$. 
More generally, for each~$\sigma \in (0,\infty)$, we define
\begin{equation}
\label{e.Psi.log-normal}
\Psi_{\sigma} (t)  \coloneqq  \exp \biggl( \frac1{\sigma^2} \log^2 \bigl( 1 + \sigma t \bigr) \biggr)
\,, \qquad t\in[0,\infty)\,.
\end{equation}
This class of functions captures the stochastic integrability of log-normal random variables in the sense that, for every random variable~$X$,  
\begin{equation*}
X \leq \O_{\Gamma_2} (\sigma)
\iff
\exp(X)-1 \leq \O_{\Psi_{\sigma}}(\sigma)\,.
\end{equation*}
If~$\sigma\leq 1$, then 
\begin{equation*}
\frac1{\sigma^2} \log^2 \bigl( 1 + \sigma t \bigr)
\leq 
\log^2 \bigl( 1 + t \bigr)\,, \quad \forall t\in [0,\infty)
\end{equation*}
by the concavity of the logarithm, and so~$Y\leq \O_{\Psi_\sigma}(a)$ implies that~$Y \leq \O_{\Psi_1}(a)$. For this reason, we will generally use~$\Psi_{\sigma}$ only for~$\sigma\geq1$. For all such~$\sigma$, the function~$\Psi_\sigma$ is admissible since 
\begin{equation}
\mbox{for every~$\sigma\in[1,\infty)$, \quad $\Psi_\sigma$ satisfies~\eqref{e.Psi.growth} with constant $K=K_{\Psi_\sigma}  \coloneqq  2 \exp(2\sigma^2)$}
\,.
\label{e.K.Psi.sigma}
\end{equation}
To prove~\eqref{e.K.Psi.sigma}, observe that, for every~$\sigma,K\geq 1$, 
\begin{align*}
t \Psi_\sigma (t) \leq \Psi_\sigma (Kt)\,, \quad \forall t\in [1,\infty)
& 
\iff \sigma^2 \log t \leq \log^2 (1+K\sigma t) - \log^2 (1+\sigma t) \,,
\quad \forall t\in [1,\infty)
\notag \\ & \;
\Longleftarrow \ \,
\sigma^2 \log t \leq 
\bigl( \log (1+K \sigma t) \bigr) \log \biggl( \frac{1+K\sigma t}{1+\sigma t} \biggr)
\,,
\quad \forall t\in [1,\infty)\,,
\end{align*}
where in the last line we used that~$A^2-B^2 = (A+B)(A-B) \geq A (A-B)$ for all~$A\geq B > 0$.
Moreover, the last statement on the right is valid with the choice of~$K=2\exp(\sigma^2)$ since it makes the second logarithm factor on the right side larger than~$\sigma ^2$, while the first is clearly larger than~$\log t$. 
This completes the proof of~\eqref{e.K.Psi.sigma}.

\subsection{Concentration inequalities with respect to Orlicz quasi-norms}

We present a simple concentration inequality for random variables with at least exponential integrability.

\begin{lemma}[Concentration for exponential random variables]
\label{l.exp.concentration}
Let~$\sigma \in [1,2]$ and~$m\in\N$. Suppose that~$X_1,\ldots,X_m$ is a sequence of independent random variables satisfying 
\begin{equation*}
X_k = \O_{\Gamma_\sigma} (1) 
\quad \mbox{and} \quad 
\E[ X_k ] = 0 \, , \quad \forall k\in \{1,\ldots,m\}\,.
\end{equation*}
Then, for every~$t\geq1$, 
\begin{equation}
\label{e.Sm.exp.concentration}
\P \Biggl[ \sum_{k=1}^m X_k
> 
t \Biggr] 
\leq 
\left\{
\begin{aligned}
&
\max\biggl\{ 
\exp\Bigl( - \frac{t^2}{40m} \Bigr)  \, , \, \exp \Bigl( -\frac12 t \Bigr)  \biggr\}& \mbox{if} & \ \sigma = 1\,,
\\ & 
\max\biggl\{ 
\exp \Bigl( - \frac{t^2}{40m} \Bigr)
\,,\,
\frac{128}{(\sigma-1)^3} 
\exp \Bigl( - \frac1{2\sigma} t^\sigma \Bigr)
\biggr\} 
 & \mbox{if} & \ \sigma \in (1,2]\,.
\end{aligned}
\right.
\end{equation}
In particular, 
\begin{equation}
\label{e.exp.concentration}
\sum_{k=1}^m X_k = 
\left\{
\begin{aligned}
& \O_{\Gamma_1}\bigl(40m^{\nicefrac12} \bigr) & \mbox{if} & \ \sigma = 1\,,
\\ & 
\O_{\Gamma_\sigma} \bigl( \max\bigl\{  40m^{\nicefrac12} , 20 \left|\log(\sigma{-}1)\right|^{\nicefrac1\sigma} \bigr\}  \bigr) 
& \mbox{if} & \ \sigma \in (1,2] \,.
\end{aligned}
\right.
\end{equation}
\end{lemma}
\begin{proof}
Denote~$S_m \coloneqq X_1+\cdots+X_m$. 
We start from the Chernoff bound: for every~$a>0$,  
\begin{equation}
\label{e.Chernoff}
\P \bigl[ S_m \geq a \bigr] 
\leq 
\inf_{\lambda\in(0,\infty)} 
\exp(-\lambda a) \bigl[ \exp(\lambda S_m) \bigr]
=
\inf_{\lambda\in(0,\infty)} 
\exp(-\lambda a) \prod_{k=1}^N \E \bigl[ \exp(\lambda X_k) \bigr]\,.
\end{equation}
To estimate~$\E [ \exp(\lambda X_k) ]$, we use the elementary inequality 
\begin{equation}
\label{e.exp.Taylor.2}
\bigl| \exp(\lambda x) - \bigl( 1+\lambda x \bigr) \bigr| 
\leq \frac 12 \lambda^2 |x|^2 
\exp\bigl( \lambda \max\{ x,0 \} \bigr)
\,,\quad \forall x\in\R\,,
\end{equation}
and the centering assumption that~$\E[ X_k ] = 0$ to get
\begin{align}
\label{e.explambdaX.exp}
\bigl| \E \bigl[ \exp( \lambda X_k) \bigr] - 1 \bigr|
& 
\leq 
\frac12 \lambda^2 \E \Bigl[ |X_k|^2 \exp\bigl( \lambda \max\{ X_k, 0\} \bigr) \Bigr]
\notag \\ & 
\leq 
\frac12 \lambda^2 
\Bigl( \E \bigl[ |X_k|^2 \bigr]
+
\exp(\lambda) \Bigr)
+
\frac12 \lambda^2 
\int_1^\infty 
(2t+\lambda t^2) \exp( \lambda t) \P [ X_k > t ] \,dt
\notag \\ & 
\leq 
\frac12 \lambda^2 
\bigl( 3 +
\exp(\lambda) \bigr)
+
\frac12 \lambda^2 
\int_1^\infty 
(2t+\lambda t^2) \exp( \lambda t) \P [ X_k > t ] \,dt
\,,
\end{align}
where in the last line, we estimated the second moment by
\begin{equation*}
\E \bigl[ |X_k|^2 \bigr]
\leq
1 + 2 \int_1^\infty t \P [ X_k > t] \,dt 
\leq 
1+ 2 \int_1^\infty t\exp(-t^\sigma)\,dt
\leq 5 \,.
\end{equation*}
We split the estimate of the integral on the right side of~\eqref{e.explambdaX.exp} into two cases:~$\sigma=1$ and~$\sigma\in (1,2]$. 

\smallskip

In the case~$\sigma = 1$, we impose the additional restriction~$\lambda \leq \nicefrac12$, and then apply the assumption~$X_k = \O_{\Gamma_\sigma}(1)$ to get
\begin{align}
\label{e.lambdaleqhalf}
\int_1^\infty 
(2t+\lambda t^2) \exp( \lambda t) \P [ X_k > t ] \,dt
&
\leq 
\int_1^\infty 
(2t+\lambda t^2) \exp( \lambda t - t) \,dt
\notag \\ & 
\leq 
\int_1^\infty 
(2t+\lambda t^2) \exp\Bigl( -\frac12 t\Bigr) \,dt
=
8(1+2\lambda) \leq 16\,.
\end{align}
Combining the above displays, we deduce that, in the case~$\sigma=1$, for every~$\lambda\in (0,\nicefrac12]$, 
\begin{equation*}
\E \bigl[ \exp( \lambda X_k) \bigr] 
\leq 
1 + \frac12 \lambda^2 \bigl( 5 + \exp(\nicefrac12) + 16 \bigr)
\leq 1 + 12\lambda^2\,.
\end{equation*}
Returning then to~\eqref{e.Chernoff} and using the bound~$1+x \leq \exp(x)$, we get 
\begin{equation*}
\P \bigl[ S_m \geq a \bigr] 
\leq
\exp \Bigl( -\lambda a + 12 \lambda^2 m\Bigr)\,.
\end{equation*}
Taking~$\lambda \coloneqq  \min\bigl\{\nicefrac12 \,, \frac{a}{20m} \bigr\}$ yields, for every~$a>0$, 
\begin{equation*}
\P \bigl[ S_m \geq a \bigr] \leq 
\exp \Bigl( -\min \Bigl\{ \frac15 a , \frac{a^2}{40m} \Bigr\} \Bigr)
\,.
\end{equation*}
This is~\eqref{e.Sm.exp.concentration} in the case~$\sigma=1$. 
In particular, for every~$a\geq 1$, 
\begin{equation*}
\P \bigl[ S_m \geq m^{\nicefrac12 } a \bigr] 
\leq 
\exp \Bigl( -\min \Bigl\{ \frac15 m^{\nicefrac12} a , \frac{a^2}{40} \Bigr\} \Bigr)
\leq 
\exp \Bigl( - \frac{a}{40} \Bigr)\,,
\end{equation*}
which implies that~$S_m = \O_{\Gamma_1}(40m^{\nicefrac12})$. 

\smallskip

We next consider the case~$\sigma \in (1,2]$.
If~$a\leq 10m$, we may impose the restriction~$\lambda\leq \nicefrac12$ and then we may follow the computation leading to~\eqref{e.lambdaleqhalf} in the case~$\sigma=1$ to get 
\begin{equation*}
\int_1^\infty 
(2t+\lambda t^2) \exp( \lambda t) \P [ X_k > t ] \,dt
\leq 
\int_1^\infty 
(2t+\lambda t^2) \exp( \lambda t - t^\sigma) \,dt
\leq 16
\end{equation*}
and then select~$\lambda  = \frac{a}{20m}$ to obtain
\begin{equation*}
\P \bigl[ S_m \geq a \bigr] \leq 
\exp \Bigl( - \frac{a^2}{40m} \Bigr)\,, \quad \forall a\in (0,10m]\,.
\end{equation*}
In the case~$a > 10m$, we may need to select~$\lambda > \nicefrac12$, and therefore we must estimate the integral differently. 
Using the assumption that~$X_k = \O_{\Gamma_\sigma}(1)$, we have that, for every~$\lambda>0$,
\begin{align*}
\int_1^\infty 
(2t+\lambda t^2) \exp( \lambda t) \P [ X_k > t ] \,dt
&
\leq 
\int_1^\infty 
(2t+\lambda t^2) \exp( \lambda t - t^\sigma) \,dt
\notag \\ & 
\leq 
\exp\Bigl( \frac{\sigma{-}1}{\sigma} \lambda^{\frac{\sigma}{\sigma-1}} \Bigr)
\int_1^\infty 
(2t+\lambda t^2) \exp \Bigl( - \frac{\sigma{-}1}{\sigma}t^\sigma \Bigr) \,dt\,,
\end{align*}
where in the last line we used Cauchy's inequality in the form
\begin{equation*}
\lambda t \leq \frac{1}{\sigma} t^{\sigma} 
+
\frac{\sigma{-}1}{\sigma} \lambda^{\frac{\sigma}{\sigma-1}}
\,.
\end{equation*}
Continuing the computation, we find by a change of variables that 
\begin{align*}
\int_1^\infty 
(2t+\lambda t^2) \exp \Bigl( - \frac{\sigma{-}1}{\sigma}t^\sigma \Bigr) \,dt
& 
\leq
\frac{1}{\sigma} 
\int_0^\infty 
\biggl(
2 \Bigl( \frac\sigma{\sigma{-}1} \Bigr)^{\!\frac2\sigma} s^{\frac2\sigma-1} 
+
\lambda \Bigl( \frac\sigma{\sigma{-}1} \Bigr)^{\!\frac3\sigma} s^{\frac3\sigma-1} 
\biggr) 
\exp(-s)\,ds
\notag\\ &
=
\frac{2}{\sigma} 
\Bigl( \frac{\sigma}{\sigma{-}1}\Bigr)^{\!\frac 2\sigma}\gammafun\Bigl(\frac 2\sigma\Bigr) 
+
\frac{\lambda}{\sigma} 
\Bigl( \frac{\sigma}{\sigma{-}1}\Bigr)^{\!\frac 3\sigma}\gammafun\Bigl(\frac 3\sigma\Bigr) 
\leq \frac{2+\lambda}{\sigma} \Bigl( \frac{\sigma}{\sigma{-}1}\Bigr)^{\!\frac 3\sigma}\,.
\end{align*}
Inserting these bounds into~\eqref{e.explambdaX.exp} yields
\begin{equation*}
\E \bigl[ \exp( \lambda X_k) \bigr] 
\leq 
1 + \lambda^2 \biggl( 2 + 
\frac{8+4\lambda}{(\sigma-1)^3}
\exp\Bigl( \frac{\sigma{-}1}{\sigma} \lambda^{\frac{\sigma}{\sigma-1}} \Bigr)
\biggr)
\,.
\end{equation*}
Assuming~$\lambda\ge\nicefrac12$, we get
\begin{equation*}
\E \bigl[ \exp( \lambda X_k) \bigr]
\leq \Bigl( 12+ \frac{20}{(\sigma-1)^3} \Bigr) \lambda^3 \exp \Bigl( \frac{\sigma{-}1}{\sigma} \lambda^{\frac{\sigma}{\sigma-1}} \Bigr)
\leq \frac{32}{(\sigma-1)^3} \lambda^3 \exp \Bigl( \frac{\sigma{-}1}{\sigma} \lambda^{\frac{\sigma}{\sigma-1}} \Bigr)
\,.
\end{equation*}
Inserting this into~\eqref{e.Chernoff} and selecting~$\lambda = a^{\sigma - 1}$---which we note satisfies the constraint~$\lambda\geq \nicefrac12$ since we have assumed~$a > 10m$---we obtain
\begin{equation*}
\P \bigl[ S_m \geq a \bigr] 
\leq 
\frac{32a^{3(\sigma-1)}}{(\sigma-1)^3}\exp \Bigl( - \frac1\sigma a^\sigma \Bigr)
\leq
\frac{128}{(\sigma-1)^3} 
\exp \Bigl( - \frac1{2\sigma} a^\sigma \Bigr)
\,, \quad \forall a\in [10m,\infty)\,.
\end{equation*}
This completes the proof of~\eqref{e.Sm.exp.concentration} in the case~$\sigma\in(1,2]$, which also implies~\eqref{e.exp.concentration}. 
\end{proof}

Next, we present a variant of the concentration argument above for random variables that may have weaker, sub-exponential integrability quantified by general Orlicz quasinorms. The proof is similar to the arguments of the recent paper~\cite{BMP}, using a truncation method.

\begin{proposition}[Concentration for sums of independent and~$\O_{\Psi}$-bounded random variables]
\label{p.concentration.Psi}
\emph{}\\
Let~$\Psi:\R_+ \to [1,\infty)$ be an increasing function satisfying
\begin{equation}
\label{e.Psi.growth.concentration}
C_\Psi  \coloneqq  \int_1^\infty 
\frac{t}{\Psi(t)}\,dt
<\infty 
\end{equation}
Let~$m\in\N$ and~$X_1,\ldots,X_m$ be a sequence of independent random variables satisfying 
\begin{equation*}
X_k = \O_{\Psi} (1) 
\quad \mbox{and} \quad 
\E[ X_k ] = 0 \, , \quad \forall k\in \{1,\ldots,m\}\,.
\end{equation*}
Suppose that~$M\in [0,\infty)$,~$\lambda \in (0,1]$ and~$L\in [1,\infty)$ satisfy the relation
\begin{equation}
\label{e.lambda.L.constraint}
\lambda t \leq \log \Psi(t) - 4 \log t + \log M\,,  
\quad \forall t \in [1,L]\,.
\end{equation}	
Then, for every~$t>0$, 
\begin{equation}
\label{e.Psi.concentration}
\P \Biggl[ \sum_{k=1}^m X_k
> 
t \Biggr] 
\leq 
\frac{m}{\Psi(L)} + \exp\bigl(  - \lambda t  + \lambda^2 m \bigl( 2+M+C_\Psi \bigr)  \bigr)\,.
\end{equation}
\end{proposition}
\begin{proof}
Denote~$S_m \coloneqq X_1+\cdots+X_m$. 
Fix~$a>0$ and~$L \in [1,\infty)$ and define the truncations
\begin{equation*}
Y_k  \coloneqq  \min \{ X_k , L \} 
\quad \mbox{and} \quad
T_k  \coloneqq  \sum_{k=1}^m Y_k\,.
\end{equation*}
Either~$S_m = T_m$, or else~$\max_{1\leq k \leq m} X_k > L$. We deduce, therefore, that 
\begin{equation}
\label{e.truncation.split}
\P \bigl[ S_m > a \bigr]
\leq 
\P \bigl[ T_m > a \bigr] 
+
\P \Bigl[ \max_{1\leq k \leq m} X_k > L \Bigr]\,.
\end{equation}
By a union bound and the assumption~$X_k \leq \O_{\Psi}(1)$, the second term on the right side of~\eqref{e.truncation.split} is bounded by  
\begin{equation}
\label{e.maxterm.shipping}
\P \Bigl[ \max_{1\leq k \leq m} X_k > L \Bigr]
\leq 
\sum_{k=1}^m
\P \bigl[ X_k > L \bigr]
\leq 
\frac{m}{\Psi(L)}\,.
\end{equation}
For the first term, we use the Markov inequality and independence to obtain, for every~$\lambda>0$, 
\begin{equation}
\label{e.Markov.indy.prod}
\P \bigl[ T_m > a \bigr]
\leq 
\exp( - \lambda  a ) 
\E \bigl[ \exp( \lambda T_m) \bigr] 
\leq 
\exp( - \lambda a )
\prod_{k=1}^m 
\E \bigl[ \exp( \lambda Y_k) \bigr] 
\,.
\end{equation}
In order to estimate~$\E \bigl[ \exp( \lambda Y_k) \bigr]$, we use~\eqref{e.exp.Taylor.2}, which implies
\begin{equation*}
\bigl| \E \bigl[ \exp( \lambda Y_k) \bigr] - \bigl( 1 + \lambda \E [ Y_k ] \bigr) \bigr|
\leq 
\frac12 \lambda^2 \E \Bigl[ |Y_k|^2 \exp\bigl( \lambda \max\{ Y_k, 0\} \bigr) \Bigr]\,.
\end{equation*}
Due to the truncation, the random variable~$Y_k$ is not centered. However,~$\E [ Y_k ] \leq \E [ X_k ] \leq 0$, which suffices for our purposes. We deduce, therefore, that 
\begin{equation}
\label{e.exp.Taylor}
\E \bigl[ \exp( \lambda Y_k) \bigr] \leq 1+\frac12 \lambda^2 \E \Bigl[ Y_k^2 \exp\bigl( \lambda \max\{ Y_k, 0\} \bigr) \Bigr]\,.
\end{equation}
To estimate the right side of~\eqref{e.exp.Taylor}, we use the hypothesis that~$X_k = \O_\Psi (1)$ to compute
\begin{align*}
\E \Bigl[ Y_k^2 \exp\bigl( \lambda \max\{ Y_k, 0\} \bigr) \Bigr]
&
=
\exp(\lambda) 
+
\int_1^L
(2t + \lambda t^2) 
\exp(\lambda t) \P \bigl[ X_k > t \bigr] \,dt 
\notag \\ & 
\leq 
\E \bigl[ X_k^2\bigr] 
+
\int_0^1 
(2t + \lambda t^2)
\exp(\lambda t) 
\,dt
+
\int_1^L
\frac{(2t + \lambda t^2) 
\exp(\lambda t) }
{\Psi(t)} \,dt 
\notag \\ & 
=
\E \bigl[ X_k^2 \bigr] 
+
\exp(\lambda) 
+
\int_1^L
(2t + \lambda t^2) 
\exp\bigl (\lambda t - \log \Psi(t) \bigr) \,dt 
\,.
\end{align*} 
We continue now under the assumption that~$\lambda\in(0,1]$ and~$L\in[1,\infty)$ satisfy~\eqref{e.lambda.L.constraint}. We have 
\begin{align*}
\int_1^L \!
(2t + \lambda t^2) 
\exp\bigl (\lambda t - \log \Psi(t) \bigr) \,dt 
\leq 
M\int_1^L \! (2t + \lambda t^2)  t^{-4} \,dt 
&
\leq 
M\int_1^\infty \! (2t + \lambda t^2)  t^{-4} \,dt 
=
\Bigl( \frac23+\frac\lambda 3\Bigr)M .
\end{align*}
For the second moments of~$X_k$, we use the condition~\eqref{e.Psi.growth.concentration}, which implies that 
\begin{equation*}
\E \bigl[ X_k^2 \bigr]\leq 1+2C_\Psi\,.
\end{equation*}
We, therefore, obtain that 
\begin{equation*}
\E \bigl[ \exp( \lambda Y_k) \bigr] 
\leq 
1 
+
\frac12 \lambda^2 \Bigl( \E \bigl[ X_k^2 \bigr] + \exp(\lambda) + M \Bigr) 
\leq 
1 
+
\frac 12 \lambda^2 \bigl( 4+M+ 2C_\Psi  \bigr)
\leq
\exp \Bigl( \frac 12 \lambda^2 \bigl( 4+M+ 2C_\Psi  \bigr) \Bigr)
.
\end{equation*}
Inserting this result  into~\eqref{e.Markov.indy.prod}, we obtain, for every~$\lambda \in (0,1]$ and~$L\geq 1$ satisfying~\eqref{e.lambda.L.constraint},
\begin{equation}
\label{e.Tm.shipping}
\P \bigl[ T_m > a \bigr] \leq \exp\Bigl(  - \lambda a  + \frac{1}{2}\lambda^2 m \bigl( 4+M+2C_\Psi \bigr)  \Bigr) \,.
\end{equation}
Inserting~\eqref{e.maxterm.shipping} and~\eqref{e.Tm.shipping} into~\eqref{e.truncation.split} yields~\eqref{e.Psi.concentration}, completing the proof. 
\end{proof}

We next exhibit consequences of Proposition~\ref{p.concentration.Psi} for some particular heavy-tailed distributions, including stretched exponentials (Weibull distributions) and those with log-normal-type tails. 

\begin{corollary}[{Concentration for stretched exponential tails}]
\label{c.concentration.Gamma.sigma}
Let~$\sigma \in (0,1)$ and define
\begin{equation*}
\Gamma_{\sigma} (t)  \coloneqq  \exp \bigl( t^\sigma \bigr)
\,, \quad t\in[1,\infty)\,.
\end{equation*}
Let~$m\in\N$ and~$X_1,\ldots,X_m$ be a sequence of independent random variables satisfying 
\begin{equation*}
X_k = \O_{\Gamma_\sigma} (1) 
\quad \mbox{and} \quad 
\E[ X_k ] = 0 \, , \quad \forall k\in \{1,\ldots,m\}\,.
\end{equation*}
Then, for every~$t \geq 1$, 
\begin{equation}
\label{e.Gammasigma.concentration}
\P \Biggl[ \sum_{k=1}^m X_k >  t \Biggr] 
\leq 
(m+1) \exp\bigl( -4 t^\sigma \bigr)
+
\exp \Bigl( - \frac {t^2} { 4\bigl( (8^{\nicefrac 2\sigma}\gammafun(\nicefrac 2\sigma))^4 +1 \bigr) m} \Bigr)\,.
\end{equation}
In particular, 
\begin{equation*}
\sum_{k=1}^m X_k \leq  \O_{\Gamma_\sigma}\bigl( 2\bigl( (8^{\nicefrac 2\sigma}\gammafun(\nicefrac 2\sigma) )^2 +1 \bigr) m^{\nicefrac12} \bigr)\,.
\end{equation*}
\end{corollary}

\begin{corollary}[{Concentration for  log-normal tails}]
\label{c.concentration.Psisigma}
As in~\eqref{e.Psi.log-normal}, we define, for~$\sigma \in [1,\infty)$, 
\begin{equation*}
\Psi_{\sigma} (t)  \coloneqq  \exp \biggl( \frac1{\sigma^2} \log^2 \bigl( 1 + \sigma t \bigr) \biggr)
\,, \quad t\in[1,\infty)\,.
\end{equation*}
Let~$m\in\N$ and~$X_1,\ldots,X_m$ be a sequence of independent random variables satisfying 
\begin{equation*}
X_k = \O_{\Psi_\sigma} (1) 
\quad \mbox{and} \quad 
\E[ X_k ] = 0 \, , \quad \forall k\in \{1,\ldots,m\}\,.
\end{equation*}
Then, for every~$t \geq 4\sigma$, 
\begin{equation}
\label{e.Psisigma.concentration}
\P \Biggl[ \sum_{k=1}^m X_k >  t \Biggr] \leq  
(m+1) \exp \Bigl( - \frac1{\sigma^2} \log^2 \frac{\sigma t}{1600} \Bigr)
+
\exp \biggl( - \frac {t^2} {(2 \exp(32\sigma^2) + 20) m} \biggr)
\,.
\end{equation}
In particular, 
\begin{equation}
\label{e.Sm.OPsisigma}
\sum_{k=1}^m X_k \leq \O_{\Psi_\sigma} \bigl( 32 \exp(16\sigma^2)  m^{\nicefrac12} \bigr)\,.
\end{equation}
\end{corollary}

\begin{proof}[Proof of Corollary~\ref{c.concentration.Gamma.sigma}]
The constant~$C_{\Gamma_\sigma}$ defined in~\eqref{e.Psi.growth.concentration} is given by
\begin{equation*}
C_{\Gamma_\sigma}
=
\int_1^\infty 
\frac{t}{\Gamma_\sigma(t)}\,dt
=
\int_1^\infty 
t \exp (-t^\sigma) \,dt
=
\frac1\sigma \int_1^\infty
s^{\frac{2}{\sigma}-1} 
 \exp (-s) \,ds
\leq
\frac1{\sigma} \gammafun\Bigl( \frac2\sigma \Bigr)\,,
\end{equation*}
where~$\gammafun$ denotes the gamma function. We also have\footnote{The proof of~\eqref{e.L.cond.Psi.for.Psi.le2} can be found in the latex file for this paper (available on arXiv), commented out below this sentence.}
\begin{equation}
\sigma \in (0,1)\,, \quad 
M= M_\sigma \coloneqq  \bigl( 8^{\nicefrac2\sigma} \gammafun(\nicefrac 2\sigma) \bigr)^{4}
\,,\quad 
\lambda\in (0,1]\,, 
\quad 
1 \leq L \leq 
\bigl( 2\lambda \bigr)^{\!-\frac{1}{1-\sigma}}
\implies 
\mbox{\eqref{e.lambda.L.constraint}} \,.
\label{e.L.cond.Psi.for.Psi.le2}
\end{equation}
By applying Proposition~\ref{p.concentration.Psi}, 
we therefore obtain, for every~$\lambda\in(0,1]$ and~$t\geq 1$, 
\begin{equation}
\label{e.SECCapp}
\P \Biggl[ \sum_{k=1}^m X_k > t \Biggr] 
\leq 
m \exp \Bigl( - \bigl( 2\lambda \bigr)^{-\frac{\sigma}{1-\sigma}} \Bigr) + \exp\bigl(  - \lambda t  + A_\sigma \lambda^2 m  \bigr)\,.
\end{equation}
where we set~$A_\sigma\coloneqq
2\bigl( (\gammafun(\nicefrac 2\sigma) 8^{\nicefrac 2\sigma})^4+1 \bigr) \geq  \bigl( 2+M_\sigma+C_{\Gamma_\sigma} \bigr)$. Given~$t\in [1,\infty)$, we apply the above inequality with~$\lambda$ selected by
\begin{equation*}
\lambda  \coloneqq  
\min\biggl\{ 
\frac {t}{2A_\sigma m} , \, \frac12 t^{\sigma -1}
\biggr\}\,.
\end{equation*}
Note that this ensures~$\lambda\leq 1$ and, with this choice, the first and second terms on the right side of~\eqref{e.SECCapp} are estimated, respectively, by
\begin{equation*}
\exp\bigl(  - \lambda t  + A_\sigma \lambda^2 m  \bigr)
\leq 
\exp\Bigl(- \frac12 \lambda t \Bigr)
=
\max\biggl\{ 
\exp\biggl(  -\frac {t^2}{2A_\sigma m}\biggr) \,,\,
\exp\biggl( - \frac14t^\sigma \biggr) 
\biggr\} 
\end{equation*}
and
\begin{equation*}
m \exp \Bigl( - \bigl( 2\lambda \bigr)^{\! -\frac{\sigma}{1-\sigma}} \Bigr)
\leq
m \exp ( -t^\sigma)\,.
\end{equation*}
Combining the above displays yields~\eqref{e.Gammasigma.concentration}. 
\end{proof}

\begin{proof}[Proof of Corollary~\ref{c.concentration.Psisigma}]
We check that the constant in~\eqref{e.Psi.growth.concentration} satisfies
\begin{equation*}
\label{e.CPsi.for.Psi2}
C_{\Psi_\sigma}
\leq 
 \exp (4\sigma^2) 
\end{equation*}
and that the condition~\eqref{e.lambda.L.constraint} is satisfied with~$M=M_\sigma \coloneqq \exp(32\sigma^2)$ and every~$\lambda\in (0,1]$ and~$L$ satisfying\footnote{For completeness we have included full demonstrations of these assertions in a commented-out part of the latex file for this paper, available from the arXiv.}   
\begin{equation*}
\label{e.L.cond.Psi.for.Psi2}
1 \leq L \leq \frac{1}{32\sigma^2\lambda} \log^2 \Bigl( 1 + \frac{1}{\sigma \lambda} \Bigr)
\,.
\end{equation*}
Applying Proposition~\ref{p.concentration.Psi} therefore yields, for every~$t>0$ and $\lambda\in (0,1]$,
\begin{equation*}
\P \Biggl[ \sum_{k=1}^m X_k >  t \Biggr] 
\leq 
\frac{m}{\Psi_\sigma \Bigl(\tfrac{1}{32\sigma^2\lambda} \log^2 \bigl( 1 + \tfrac{1}{\sigma \lambda} \bigr)\Bigr)} + \exp\bigl(  - \lambda t  + A_\sigma \lambda^2 m   \bigr)\,,
\end{equation*}
where we have set
\begin{equation*}
A_\sigma  \coloneqq  
2 \exp(32\sigma^2) + 5
\geq 
( 2+M_\sigma +C_{\Psi_\sigma})
\,.
\end{equation*}
We apply the above inequality to each~$t\geq 4\sigma$ with~$\lambda$ chosen in terms of~$t$ by
\begin{equation*}
\lambda  \coloneqq  \min\biggl\{ 
\frac {t}{2A_\sigma m} , \, \frac{2}{\sigma^2 t}\log^2 (1+\sigma t) 
\biggr\}\,.
\end{equation*}
This choice of~$\lambda$ implies that
\begin{equation*}
\exp \bigl( -\lambda t + A_\sigma \lambda^2 m \bigr) 
\leq 
\exp \Bigl( -\frac12 \lambda t \Bigr) 
=
\max\biggl\{ \exp \Bigl( - 
\frac {t^2} {4A_\sigma m} \Bigr), \, \exp \Bigl( - \frac{1}{\sigma^2}\log^2 (1+\sigma t) \Bigr)
\biggr\}\,
\end{equation*}
as well as
\begin{equation*}
\frac{1}{32\sigma^2\lambda} \log^2 \Bigl( 1 + (\sigma \lambda)^{-1} \Bigr)
\geq 
\frac{t}{64} \frac{\log^2 \bigl( 1 + (\sigma \lambda)^{-1} \bigr)}{\log^2( 1 + \sigma t ) }
\geq
\frac{t} {64} 
\biggl( 
\frac{\log \bigl( 1 + \frac12\sigma t \log^{-2}(1+\sigma t) \bigr)}{\log \bigl( 1 + \sigma t \bigr)} 
\biggr)^{\!2}
\geq \frac {t}{1600}\,.
\end{equation*}
In the last inequality of the previous display, we used that 
\begin{equation} 
\label{e.famous.onefifth}
\log \Bigl(1 + \frac{1}{2}s\log^{-2}(1+s) \Bigr) \geq \frac1{5}
\log \bigl( 1 + s \bigr) 
\,, \quad \forall s\in [1,\infty)\,.
\end{equation}
The validity of the inequality~\eqref{e.famous.onefifth} for some (universal)~$c >0$ in place of the~$\nicefrac 15$ on the right side is clear due to the fact that the ratio of the left side and~$\log(1+s)$ is positive for all~$s\geq 1$ and tends to one as~$s\to \infty$. That the inequality is valid as stated with this constant equal to~$\nicefrac15$ can be confirmed by either Mathematica or Wolfram Alpha. We used the command
\begin{center}
\begin{BVerbatim}
Reduce[Log[1 + s/(2 Log[1 + s]^2)] > 1/5 Log[1 + s], s, Reals]
\end{BVerbatim}
\end{center}
in Mathematica to validate~\eqref{e.famous.onefifth}, which instantly reported its validity for all~$s>0$.
Combining the above displays, we obtain~\eqref{e.Psisigma.concentration}.

\smallskip

To obtain~\eqref{e.Sm.OPsisigma}, we use~\eqref{e.Psisigma.concentration} with~$4A_\sigma^{\nicefrac12} m^{\nicefrac12} t$ in place of~$t$, to obtain, for every~$t \geq 1$ (note that~$A_\sigma^{\nicefrac12}\geq \sigma$),
\begin{align*} 
\P \Biggl[ \sum_{k=1}^m X_k > 4 A_\sigma^{\nicefrac12} m^{\nicefrac12} t \Biggr]  \leq  \frac{m+1}{\Psi_\sigma\Bigl( \tfrac1{400} A_\sigma^{\nicefrac12} m^{\nicefrac12} t\Bigr)} + \exp \bigl( - 4t^2 \bigr) \,. 
\end{align*} 
By an easy exercise, it can be checked\footnote{The proof of these inequalities are routine, but for completeness, we have also commented them out in the latex file after this sentence.} that, for every~$\sigma,t,m\geq 1$, 
\begin{equation*} \frac{m+1}{\Psi_\sigma\Bigl( \tfrac1{400} A_\sigma^{\nicefrac12} m^{\nicefrac12} t\Bigr)}   \leq \frac1{2 \Psi_\sigma(t) } \quad \mbox{and} \quad \exp( - 4t^2) \leq \frac1{2 \Psi_\sigma(t) }  \,. \end{equation*} 
We also observe that~$4 A_\sigma^{\nicefrac12} \leq 8 \exp (16\sigma^2)$, for every~$\sigma\geq 1$. This completes the proof of~\eqref{e.Sm.OPsisigma}.
\end{proof}

\section{Examples of random fields satisfying the assumptions}

In this appendix we give some examples of random coefficient fields which are \emph{coarse-grained elliptic} in the sense of assumption~\ref{a.ellipticity}, but not uniformly elliptic. 

\subsection{Poisson inclusions}
\label{ss.Poisson.inclusions}

In this subsection, we prove Proposition~\ref{p.inclusions}. 
We consider two Poisson point clouds~$\omega_1$ and~$\omega_2$ on~$\Rd$ with intensities~$\rho_1\geq 0$ and~$\rho_2\geq 0$, respectively. Let~$\lambda\in(0,1]$, $\Lambda\in[1,\infty)$ and define the scalar matrix-valued field
\begin{equation}
\label{e.a.inclusions.app}
\a \coloneqq  
\bigl(  1 + (\Lambda-1) \indc_{B_{\nicefrac13}} \ast \omega_1 + (\lambda-1) \indc_{B_{\nicefrac13}} \ast \omega_2 \bigr) \Id
\,.
\end{equation}
Notice that our inclusions are balls of radius~$\nicefrac13$ rather than unit radius like in~\eqref{e.a.inclusions}. We have introduced this extra dilation for notational convenience. 
We also denote~$\rho  \coloneqq  \rho_1+\rho_2$ and~$\omega = \omega_1+\omega_2$. 

\smallskip

We adapt some arguments and notation from classical percolation theory.
We view~$\Zd$ as an undirected graph with vertices~$x,y\in\Zd$ connected by an edge if and only if~$\max_{i\in\{1,\ldots,d} |x_i-y_i| = 1$. When we speak of \emph{connected} subsets of~$\Zd$, we mean those that are connected with respect to this graph structure. 
A \emph{lattice animal} is a finite, connected subset of~$\Zd$. As is well-known, a crude combinatorial counting argument gives that, for each fixed~$z_0\in\Zd$, the number of distinct lattice animals which contains~$z_0$ and has exactly~$\ell$ elements is at most~$\exp( C \ell)$ for some~$C(d)<\infty$. For each lattice animal~$A \subseteq \Zd$ with~$|A|= \ell$, the probability that every unit cell~$z+\cu_0$ with~$z\in A$ has nonempty intersection with~$\omega$ is estimated by
\begin{equation*}
\P 
\Bigl[ \forall z\in A\,, \ (z + \cu_0) \cap \omega \neq \emptyset \Bigr] 
\leq 
\rho^{\ell} = \exp \bigl( - \ell \left|\log \rho \right| \bigr)\,.
\end{equation*}
Let~$E_\ell(z)$ denote the event that there exists \emph{any} lattice animal which has length at least~$\ell$, has nontrivial intersection with~$z+\ell \cu_0$, and overlaps with points of~$\omega$ in each of its cells:
\begin{equation*}
E_\ell (z)  \coloneqq  
\Biggl\{ 
\begin{aligned}
&
\mbox{there exists a lattice animal $A\subseteq \Rd$ with~$|A|\geq \ell$ and} 
\\ & 
\mbox{$A \cap (z+\ell \cu_0) \neq \emptyset$ such that, for every~$z'\in A$, $(z' + \cu_0) \cap \omega \neq \emptyset$}
\end{aligned}
\Biggr \} \,.
\end{equation*}
A union bound yields that the probability of~$E_\ell(0)$ is at most 
\begin{equation*}
\P [ E_\ell(0) ] 
\leq 
\sum_{k=\ell}^\infty
\ell^d 
\exp (C k - k \left|\log \rho\right|) 
\,.
\end{equation*}
If we restrict the intensity~$\rho$ of the Poisson cloud~$\omega$ by requiring~$\rho \leq c$ for sufficiently small~$c(d)>0$, then we obtain
\begin{equation}
\P [ E_\ell ] \leq
\exp \Bigl( d\log \ell - \frac13 \ell \left|\log \rho\right| \Bigr) 
\leq 
\exp\bigl( -c \ell \left|\log \rho\right| \bigr)
\,.
\label{e.counting.animals}
\end{equation}
By~\eqref{e.indc.O.sigma}, this implies that, for every~$\sigma \in (0,\infty)$,  
\begin{equation*}
\indc_{E_\ell} 
\leq \O_{\Gamma_{1/\sigma}}\bigl( C^{\sigma} (\ell \left|\log \rho\right|)^{-\sigma } \bigr)
\end{equation*}
We next discretize~$\omega$ by setting
\begin{equation*}
\widehat{\omega} \coloneqq  \bigl\{ z \in\Zd \,:\, (z+\cu_0) \cap \omega \neq \emptyset \bigr\}\,.
\end{equation*}
For each~$m\in\N$, we define~$\N$--valued,~$\Zd$--stationary random fields~$N$ and~$N_m$ on~$\Zd$ by
\begin{equation*}
\left\{
\begin{aligned}
& 
N (z)  \coloneqq  \mbox{the number of elements in the connected component of~$\widehat{\omega}$ containing~$z$}\,,
\\ & 
N_m  (z)  \coloneqq  \mbox{the number of elements in the connected component of~$\widehat{\omega} \cap (z+\cu_m)$ containing~$z$}\,.
\end{aligned}
\right.
\end{equation*}
Note that~$N(z) =N_m  (z)=0$ if~$z\not \in \widehat{\omega}$.
It is clear that~$N_m$ is nondecreasing as a function of~$m$. 
In view of~\eqref{e.counting.animals}, for every~$m\in\N$,
\begin{equation}
\label{e.Nm.SIB}
N_m (0) \leq N(0) \leq \O_{\Gamma_1} \bigl( C\left|\log \rho\right|^{-1} \bigr)\,.
\end{equation}
It is also clear that~$N_{m+1}(z)$ differs from~$N_m(z)$ for~$z\in \cu_{m}$ only if~$N_{m}(z) > 3^m$, and thus only on the event~$E_{3^m}$. 
More generally, for every~$n,m\in\N$ with~$n<m$,  
\begin{align}
\label{e.Nm.AFRD}
\P \bigl[ \exists z\in \cu_m\,, N_m(z) \neq N_n(z) \bigr] 
\leq 
\P \bigl[ \exists z\in \cu_m\,, E_{3^n} (z) > 3^n \bigr]
&
\leq
3^{md} 
\exp\bigl( -c 3^n \left|\log \rho\right| \bigr)
\notag \\ & 
=
\exp\bigl( Cm - c 3^n \left|\log \rho\right| \bigr)
\,.
\end{align}
This implies that~$N_m$ has an approximate finite range of dependence property, which allows us to use concentration inequalities.

\begin{lemma}
\label{l.Poisson.small.rho}
Let~$\a(\cdot)$ be the field defined in~\eqref{e.a.inclusions.app}, with parameters~$\rho$,~$\lambda$ and~$\Lambda$ as introduced there. There exist constants~$c(d)>0$ and~$C(d)<\infty$ such that, if~$\rho\leq c$ and~$\gamma \in (0,1)$, then there exists a minimal scale~$\S$ satisfying~$\S \leq \O_{\Psi_{\S}}(1) $ and
\begin{equation*} 
3^m \geq \S \implies 
\bfA(z+\cu_n) \leq 3^{\gamma(m-n)} \bfE \quad \forall n \in \Z \cap (-\infty,m]\,, \quad
\forall z\in 3^n\Zd \cap  \cu_m
\end{equation*}
with
\begin{equation*} 
\bfE  \coloneqq   (1+ C \left|\log \rho \right|^{-2} ) \Itwod
\qand
\Psi_{\S}(t)  \coloneqq  \exp\Bigl( c  \bigl(\Lambda \vee \lambda^{-1} \bigr)^{-\frac{1}{d+2} - \frac{\gamma}{d}}   t^{\frac{\gamma}{d+2}} -1 \Bigr)\,.
\end{equation*}
Moreover,~$\Psi_\S$ satisfies
\begin{equation*} 
t \Psi_\S(t) \leq \Psi_\S( K_{\Psi_\S}t) \quad \forall t \in [1,\infty) \quad \mbox{with}\quad 
 K_{\Psi_\S}  \coloneqq  \bigl(  C\gamma^{-1}  \bigr)^{\! \frac
{d+2}{\gamma}} \bigl(\Lambda \vee \lambda^{-1} \bigr)^{\frac{1}{\gamma} + 1+ \frac{2}{d}}
\,.
\end{equation*}
\end{lemma}
\begin{proof}
We construct, for each~$m\in\N$, a partition of the cube~$\cu_m$. We let~$\mathcal{C}_m(\omega)$ denote the collection of connected components of~$\widehat{\omega} \cap \cu_m$. For each~$A \in \mathcal{C}_m(\omega)$, we associate the continuum set
\begin{equation*}
\widetilde{A} \coloneqq  \bigcup_{z\in A} (z+\cu_0) 
\end{equation*}
and slightly enlarge this set by defining 
\begin{equation*}
\breve{A}  \coloneqq  \big\{ x \in \cu_m \,:\, \dist(x,\widetilde{A}) < \nicefrac 13 \} \,.
\end{equation*}
Observe that~$\breve{A}_1 \cap \breve{A}_2 = \emptyset$ if~$A_1$ and~$A_2$ are distinct connected components of~$\widehat{\omega}$. If~$\widetilde{A}$ does not touch the boundary~$\partial \cu_m$, we write~$A \in \mathcal{C}_m^\circ(\omega)$. We also let~$\mathcal{C}_m^b(\omega) \coloneqq  \mathcal{C}_m(\omega)\setminus  \mathcal{C}_m^\circ(\omega)$ be those connected components for which~$\widetilde{A}$ does touch~$\partial \cu_m$. 

\smallskip

We let~$\mathcal{P}$ be the collection of all subsets of~$\cu_m$ of the form: (i)~$\widetilde{A}$, for~$A\in \mathcal{C}_m(\omega)$; (ii)~$\breve{A} \setminus \widetilde{A}$, for~$A\in \mathcal{C}_m(\omega)$; and (iii)~$\cu_m \setminus ( \cup\{ \breve{A}\,:\, A \in \mathcal{C}_m(\omega)\})$. It is clear that~$\mathcal{P}$ is a partition of~$\cu_m$, up to a zero Lebesgue measure set.  

\smallskip

Given~$p\in\Rd$, we define a gradient field~$\nabla \phi_p$ on~$\cu_m$ with the following properties:
\begin{itemize}
\item $\nabla \phi_p$ vanishes on~$\widetilde{A}$, for every~$A \in \mathcal{C}_m^\circ(\omega)$.

\item $\nabla \phi_p = p$ in $\cu_m \setminus \cup\{ \breve{A}\,:\, A \in \mathcal{C}_m^\circ(\omega)\}$.

\item For each~$A \in \mathcal{C}_m^\circ(\omega)$, we have~$| \nabla \phi_p | \leq 20 |p| |A|$ in~$\breve{A} \setminus \widetilde{A}$.
\end{itemize}
In particular, since~$\nabla \phi_p = p$ in a neighborhood of the boundary~$\partial \cu_m$, we have that
\begin{equation*}
p\cdot \a(\cu_m) p 
\leq 
\min_{u \in \ell_p+ H^1_0(\cu_m)} 
\fint_{\cu_m} 
\nabla u \cdot \a\nabla u
\leq 
\fint_{\cu_m} 
\nabla \phi_p \cdot \a\nabla \phi_p
\,.
\end{equation*}
Since~$\nabla \phi_p$ vanishes in~$\widetilde{A}$ for each~$A \in \mathcal{C}_m^\circ(\omega)$, we have that 
\begin{equation*}
\sum_{A \in\mathcal{C}_m^\circ(\omega)} \int_{\widetilde{A}} \nabla \phi_p \cdot \a\nabla \phi_p = 0\,.
\end{equation*}
Since~$\a(x) = \Id$ in~$B \coloneqq \cu_m \setminus \cup\{ \widetilde{A}\,:\, A \in \mathcal{C}_m(\omega)\}$ and~$\a(x) \le \Lambda \Id$ otherwise, we deduce that 
\begin{equation*}
\fint_{\cu_m} 
\nabla \phi_p \cdot \a\nabla \phi_p
=
\frac{1}{|\cu_m|}
\int_{B} |\nabla \phi_p|^2
+
\frac{\Lambda }{|\cu_m|}
\sum_{A\in \mathcal{C}_m^b(\omega)} 
\int_{\widetilde{A}} |\nabla \phi_p|^2
\,.
\end{equation*}
Using the properties of~$\nabla \phi_p$, we have 
\begin{align*}
\frac{1}{|\cu_m|}
\int_{B} |\nabla \phi_p|^2
\leq 
|p|^2 + 
\frac{1}{|\cu_m|} 
\sum_{A \in \mathcal{C}_m(\omega)} 
400 |p|^2 |A|^2|\breve{A} \setminus \widetilde{A}|
&
\leq
|p|^2 + 
\frac{800|p|^2}{|\cu_m|} 
\sum_{A \in \mathcal{C}_m(\omega)} 
|A|^3
\notag \\ & 
\leq 
|p|^2 \Bigl( 1 + 
800 \avsum_{z\in \cu_m} 
N_{m+1}(z)^2  \Bigr)
\end{align*}
and, similarly, 
\begin{align*}
\frac{\Lambda }{|\cu_m|}
\sum_{A\in \mathcal{C}_m^b(\omega)} 
\int_{\widetilde{A}} |\nabla \phi_p|^2
\leq 
\frac{400\Lambda |p|^2}{|\cu_m|} 
\sum_{A\in \mathcal{C}_m^b(\omega)} 
|A|^2
&
\leq
\frac{400\Lambda |p|^2}{|\cu_m|} 
\sum_{z\in \partial \cu_m} 
N_{m+1}(z)^2 \,.
\notag \\ & 
\leq 400|p|^2  3^{-m} \Lambda 
\avsum_{z\in \partial \cu_m} 
N_{m+1}(z)^2 \,.
\end{align*}
Next, let~$n \in \N$ with~$n < m$ and  observe that 
\begin{equation*}
\avsum_{z\in \cu_m} 
N_{m+1}(z)^2 
\leq
\avsum_{z\in \cu_m} 
N_n(z)^2 
+
\avsum_{z\in \cu_m} 
\indc_{E_{3^n}(z)} N_{m+1}(z)^2
\,.
\end{equation*}
Using~\eqref{e.Nm.SIB} and~\eqref{e.Gammasigma.concentration}, and noting that the mean of~$N_n(z)$ is at most~$C \left|\log \rho \right|^{-2}$ by~\eqref{e.Nm.SIB}, we find that, for every~$t\geq C \left|\log \rho \right|^{-2}$, 
\begin{equation*}
\P \biggl[ \avsum_{z\in \cu_m} 
N_n(z)^2  > t \biggr] 
\leq 
3^{d(m-n)} \exp\bigl( -c 3^{\frac d2(m-n)} t^{\frac12} \bigr) 
+
\exp\bigl( - c3^{d(m-n)} t^2 \bigr) 
\,.
\end{equation*}
This yields, for~$t\geq C \left|\log \rho \right|^{-2}$ and~$\rho$ sufficiently small, 
\begin{align*}
\P \biggl[ \avsum_{z\in \cu_m} 
N_n(z)^2  > t \biggr] 
&
\leq 
\exp\bigl( C (m-n) - ct^{\frac12}  3^{\frac d2(m-n)}  \bigr) 
+
\exp\bigl( - c t^2 3^{d(m-n)} \bigr)
\leq
\exp\bigl(  - ct^{\frac12} 3^{\frac d2(m-n)} \bigr) 
\,.
\end{align*}
By~\eqref{e.Nm.AFRD}, we have that 
\begin{align*}
\P \biggl[ \avsum_{z\in \cu_m} 
\indc_{E_{3^n}(z)}  
N_{m+1} (z)^2  \neq 0 \biggr] 
&
\leq
\P \bigl[ \exists z\in \cu_m\,, E_{3^n} (z) > 3^n \bigr]
\leq 
\exp\bigl( Cm - c 3^n \left|\log \rho\right| \bigr)
\,.
\end{align*}
Combining the above yields, for every~$t \geq C \left|\log \rho \right|^{-2}$, 
\begin{equation*}
\P \biggl[ \avsum_{z\in \cu_m} N_{m+1}(z)^2 > t \biggr] 
\leq 
\exp\bigl(  - ct^{\frac12} 3^{\frac d2(m-n)}  \bigr) 
+
\exp\bigl( Cm - c 3^n \left|\log \rho\right| \bigr)\,.
\end{equation*}
By a very similar computation, we find that 
\begin{equation*}
\P \biggl[ \avsum_{z\in \partial \cu_m} N_{m+1}(z)^2 > t \biggr] 
\leq 
\exp\bigl(  - ct^{\frac12} 3^{\frac {d-1}2(m-n)}  \bigr) 
+
\exp\bigl( Cm - c 3^n \left|\log \rho\right| \bigr)\,.
\end{equation*}
Putting these together, we obtain, for every~$t \geq C \left|\log \rho \right|^{-2}$ and~$m,n\in\N$ with~$n<m$ and~$3^n\geq \Lambda$, 
\begin{equation}
\label{e.Poisson.a.cu.bound} 
\P \bigl[ \a(\cu_m) \not\leq (1+ t ) \Id \bigr] 
\leq
\exp\bigl(  - ct^{\frac12} 3^{\frac d2(m-n)}  \bigr) 
+
\exp\bigl( Cm - c 3^n \left|\log \rho\right| \bigr)
\,.
\end{equation}
By essentially the same argument, we also obtain an estimate for~$\a_*^{-1}(\cu_m)$, which states that, for every~$t \geq C \left|\log \rho \right|^{-2}$ and~$m,n\in\N$ with~$n<m$ and~$3^n\geq \Lambda \vee \lambda^{-1}$,
\begin{equation}
\label{e.Poisson.a.star.cu.bound} 
\P \bigl[ \a_*^{-1} (\cu_m)  \not\leq (1+ t ) \Id \bigr] 
\leq
\exp\bigl(  - ct^{\frac12} 3^{\frac d2(m-n)}  \bigr) 
+
\exp\bigl( Cm - c 3^n \left|\log \rho\right| \bigr)
\,.
\end{equation}
To prove~\eqref{e.Poisson.a.star.cu.bound}, we start from the variational formula
\begin{equation*}
q \cdot \a_*^{-1} (\cu_m) q = \min_{\g \in q + L^2_{\mathrm{sol},0}(\cu_m)} \fint_{\cu_m} \g \cdot \a^{-1} \g \,.
\end{equation*}
We test this formula with a divergence-free field~$\h_q$ on~$\cu_m$ which has the following properties: for a constant~$C(d)<\infty$, 
\begin{itemize}
\item $\h_q$ vanishes on~$\widetilde{A}$, for every~$A \in \mathcal{C}_m^\circ(\omega)$.

\item $\h_q = q$ in $\cu_m \setminus \cup\{ \breve{A}\,:\, A \in \mathcal{C}_m^\circ(\omega)\}$.

\item For each~$A \in \mathcal{C}_m^\circ(\omega)$, we have~$| \h_q  | \leq C |q| |A|$ in~$\breve{A} \setminus \widetilde{A}$.
\end{itemize}
Such a divergence-free field is relatively straightforward to construct (even if it is less obvious that than for the analogous gradient field above). The argument for~\eqref{e.Poisson.a.star.cu.bound} then follows nearly identically to that of~\eqref{e.Poisson.a.cu.bound}, with~$\h_q$ in place of~$\nabla \phi_p$.  

By combining~\eqref{e.Poisson.a.cu.bound} and~\eqref{e.Poisson.a.star.cu.bound} we obtain, for every~$t \geq C \left|\log \rho \right|^{-2}$ and~$m,n\in\N$ with~$n<m$ and~$3^n\geq \Lambda \vee \lambda^{-1}$,
\begin{equation*}
\P \bigl[ \bfA(\cu_m) \not\leq (1+ t ) \Itwod \bigr] 
\leq
\exp\bigl(  - ct^{\frac12} 3^{\frac d2(m-n)}  \bigr) 
+
\exp\bigl( Cm - c 3^n \left|\log \rho\right| \bigr)\,.
\end{equation*}
Optimizing the parameter~$n$ leads to choose~$n$ so that~$3^{(\frac d2 +1)n} \simeq t^{\frac12} 3^{\frac d2 m}$. We obtain, for every~$t \geq C \left|\log \rho \right|^{-2}$ and~$m\in\N$ with~$3^{m} \geq (\Lambda \vee \lambda^{-1} )^{1+\frac 2d}$,
\begin{equation*}
\P \bigl[ \bfA(\cu_m) \not\leq (1+ t ) \Itwod \bigr] 
\leq
\exp\bigl(  - ct^{\frac 1{d+2}} 3^{\frac d{d+2}m}  \bigr) 
\,.
\end{equation*}
Let~$n_0 \in \N$ be the smallest integer such that~$3^{n_0} \geq (\Lambda \vee \lambda^{-1} )^{1+\frac 2d}  \vee \exp(C \gamma^{-1})$. A union bound and the above display now give us 
\begin{align}
\label{e.union.bound.this.for.PPC}
\lefteqn{
\P \Bigl[ 
\exists n \in \N \cap [n_0 , m ]\,, \; z \in 3^n \Zd \cap\cu_m ,\
\bfA(z+\cu_n) \not\leq 3^{\gamma(m-n)} (1+ C \left|\log \rho \right|^{-2} ) \Itwod \Bigr] 
} \qquad\qquad\qquad &
\notag \\ &
\leq
\sum_{n=n_0}^m 
\sum_{z \in 3^n \Zd \cap\cu_m}
\P \Bigl[ 
\bfA(z+\cu_n) \not\leq 3^{\gamma(m-n)} (1+ C \left|\log \rho \right|^{-2} ) \Itwod \Bigr] 
\notag \\ &
\leq
\sum_{n=n_0}^m 
\exp\bigl( C(m-n) - c3^{\frac {\gamma}{d+2}(m-n)} 3^{\frac d{d+2}n}  \bigr) 
\notag \\ &
\leq
\exp\bigl( - c  3^{\frac {\gamma}{d+2}m + \frac{d-\gamma}{d+2} n_0}  \bigr) 
\,.
\end{align}
On the other hand, we have the quenched bound
\begin{equation*} 
3^{-\gamma(m-n_0)} \bfA(z+\cu_n)  
\leq
C 3^{\gamma n_0} (\Lambda \vee \lambda^{-1})  3^{-\gamma m} \Itwod
\leq C \bigl(\Lambda \vee \lambda^{-1} \bigr)^{1+\gamma(1+\frac 2d)} 3^{-\gamma m} \Itwod
\,,
\end{equation*}
so that if~$m_0\in\N$ is the smallest integer such that
\begin{equation*} 
3^{\gamma m_0} \geq C (\Lambda \vee \lambda^{-1} )^{1+ \gamma(1+\frac 2d)} 
\,,
\end{equation*}
we have
\begin{equation} 
\label{e.small.scales.for.PPC}
m\geq m_0\,, \quad n \leq n_0 
\quad\implies \quad
\bfA(z+\cu_n)  \leq 3^{\gamma(m-n)} \Itwod\quad \forall z \in 3^n \Zd \cap \cu_m
\,.
\end{equation}
Defining now, for~$\bfE \coloneqq  (1+ C \left|\log \rho \right|^{-2} ) \Itwod$, the minimal scale~$\mathcal{S}$ by
\begin{equation*}
\mathcal{S}  \coloneqq  \sup \biggl\{ 3^{m+1} \,:\, m\in \N \cap [m_0 ,\infty) \,, \; 
\sup_{ n \in \N \cap (-\infty, m ]} 
\,
\sup_{z\in 3^n\Zd \cap  \cu_m} 
\bfA(z+\cu_n) \not\leq 3^{\gamma(m-n)} \bfE
\biggr\}
\,,
\end{equation*}
we obtain
\begin{equation*} 
3^m \geq \S \implies 
\bfA(z+\cu_n) \leq 3^{\gamma(m-n)} \bfE \quad \forall n \in \Z \cap (-\infty,m]\,, \quad
\forall z\in 3^n\Zd \cap  \cu_m
\,.
\end{equation*}
By~\eqref{e.union.bound.this.for.PPC},~\eqref{e.small.scales.for.PPC} and a union bound we have 
\begin{equation*} 
\S \leq \O_{\Psi_\S}(1) 
\quad \mbox{with} \quad
\Psi_{\S}(t) = \exp\Bigl( c  \bigl(3^{-m_0} t\bigr)^{\frac{\gamma}{d+2}} -1 \Bigr)
\,.
\end{equation*}
By a direct computation, we also deduce that
\begin{equation*} 
K_{\Psi_\S}  \coloneqq  \Bigl( 1 + C \gamma^{-1}3^{\frac{\gamma}{d+2}m_0 }    \Bigr)^{\! \frac
{d+2}{\gamma}}
\quad \implies \quad 
t \Psi_{\S}(t) \leq \Psi(K_{\Psi_\S} t) \quad \forall t \in [1,\infty)
\,,
\end{equation*}
which yields the statement. 
\end{proof}

\subsection{Fractional Gaussian fields}
\label{ss.FGF}

In this section, we review some basic facts about fractional Gaussian fields and verify the claim made in the introduction that these fields give rise to examples of random elliptic coefficient fields satisfying our hypotheses.

\subsubsection{Definition of fractional Gaussian fields}

We begin with the definition and basic properties of fractional Gaussian fields. Many of the facts presented here can be found in~\cite{LSSW}, but we include proofs and full details of the computations for the reader's convenience. 

We denote by~$W$ a standard Gaussian white noise process on~$\Rd$. 
It is a random distribution on~$\Rd$, that is, a random element of~$\mathscr{S}'(\Rd)$, the dual of the space~$\mathscr{S}(\Rd)$ of Schwarz functions on~$\Rd$.
The field~$W$ is characterized by two properties: first, that~$W(\psi)$ is a Gaussian random variable for each~$\mathscr{S}(\Rd)$; and second, that the following covariance formula is satisfied:
\begin{equation}
\label{e.cov.W}
\cov[ W(\psi_1) , W(\psi_2) ] = \int_{\Rd} \psi_1(x)\psi_2(x)\,dx \,, \quad \forall \psi_1,\psi_2 \in \mathscr{S}(\Rd)\,.
\end{equation}
The distribution~$W$ almost surely belongs to~$H^{-\nicefrac d2-\ep}_{\mathrm{loc}}(\Rd)$, for every~$\ep>0$, but not~$H^{-\nicefrac d2}_{\mathrm{loc}}(\Rd)$. 
Immediate from the covariance formula is the following scaling invariance for~$W$: for every~$\lambda>0$, 
\begin{equation}
\label{e.W.scaling}
\lambda^{\nicefrac d2} 
W(\lambda  \cdot) 
\quad \mbox{has the same law as} \ W\,.
\end{equation}
Proof of these facts, as well as an explicit construction of~$W$, can be found in~\cite[Chapter 5]{AKMbook}. In what follows, we abuse notation by informally writing~$\int_{\Rd} \psi(x) W(x)\,dx$ in place of~$W(\psi)$. 

White noise is an example of a self-similar fractional Gaussian process. These fields are typically indexed by the \emph{Hurst parameter}, which is roughly the regularity of the field. The white noise field~$W$ has Hurst parameter~$-\nicefrac d2$. 

\smallskip

We denote the self-similar fractional Gaussian field with Hurst parameter~$-\sigma$, with~$\sigma\in (0,\nicefrac d2)$, by~$F_\sigma$.  It is characterized by the fact that~$F_\sigma(\psi)$ is a Gaussian random variable for each fixed test function~$\psi\in \mathscr{S}(\Rd)$, and the covariance formula (cf.~\cite[Theorem 3.3]{LSSW})
\begin{equation*}
\cov[ F_\sigma(\psi_1) , F_\sigma(\psi_2) ] = C(\sigma,d) \int_{\Rd}\int_{\Rd} |x-y|^{-2\sigma} \psi_1(x)\psi_2(y)\,dx\,dy \,, \quad \forall \psi_1,\psi_2 \in \mathscr{S}(\Rd)\,,
\end{equation*}
where~$C(\sigma,d)>0$ is the special constant defined by
\begin{equation}
\label{e.C.sigma.d}
C(\sigma,d)\coloneqq \frac{2^{2\sigma -d} \pi^{-\nicefrac d2} \gammafun(\sigma)}{\gammafun(\nicefrac d2-\sigma)}
\,.
\end{equation}
Note that~$\gammafun(\sigma)$ is of order~$\sigma^{-1}$ as~$\sigma \to 0$, and thus so is~$C(\sigma,d)$. 

\smallskip

The field~$F_\sigma$ can be constructed explicitly in terms of~$W$, in fact, as a deterministic function of~$W$. It is (in)formally the convolution
\begin{equation*}
F_\sigma  \coloneqq  2^\sigma (2\pi)^{-\nicefrac d2} |x|^{-(\frac d2 +\sigma)} \ast W
\,.
\end{equation*}
This convolution is, however, not well-defined. Making sense of it can be accomplished in various ways. Here we use the following integral identity: for every~$q\in(0,\infty)$ and~$x\in\Rd \setminus\{ 0 \}$, 
\begin{equation}
\label{e.xq.formula}
|x|^{-q}
=
\frac{(4\pi)^{\nicefrac d2} }{ 2^{q}\gammafun(\nicefrac q2)}
\int_0^\infty 
t^{-\frac12 ( 2 - d + q) } 
\Phi(t,x) \,dt
\,.
\end{equation}
Here~$\gammafun$ is the gamma function defined by~$\gammafun(\alpha) \coloneqq  \int_0^\infty t^{-1+\alpha} \exp(-t)\,dt$, and~$\Phi$ denotes the standard heat kernel on~$\R^{d}$, defined by
\begin{equation*}
\Phi(t,x) \coloneqq (4\pi t)^{-\nicefrac d2} \exp \biggl(-\frac{|x|^2}{4t} \biggr)\,, \quad (t,x) \in (0,\infty) \times \Rd \,.
\end{equation*}
We define~$F_\sigma$ for every~$\sigma\in (0,\nicefrac d2)$ as
\begin{equation}
\label{e.F.sigma.def}
F_\sigma(\psi) 
 \coloneqq  
\frac1{\gammafun(\nicefrac d4 - \nicefrac \sigma 2)}
\int_0^\infty 
t^{-1+\frac12(\frac d2 -\sigma)}
\int_{\Rd} (\Phi(t,\cdot) \ast \psi)(x) W(x)\,dx\,dt\,,
\qquad 
\psi \in \mathscr{S}(\Rd)\,.
\end{equation}
It is clear that~\eqref{e.F.sigma.def} defines~$F_\sigma$ as a Gaussian random distribution. 
To check that this definition yields a fractional Gaussian field with Hurst parameter~$\sigma$, it therefore suffices to check the covariance formula. By polarization, we just need to check the variance formula
\begin{equation}
\label{e.check.the.var.sigma}
\var[ F_\sigma(\psi)] = C(\sigma,d) \int_{\Rd}\int_{\Rd} |x-y|^{-2\sigma} \psi(x)\psi(y)\,dx\,dy \,, \quad \forall \psi \in \mathscr{S}(\Rd)\,,
\end{equation}
where~$C(\sigma,d)$ is as defined in~\eqref{e.C.sigma.d}.

\smallskip

To check~\eqref{e.check.the.var.sigma} we straightforwardly compute
\begin{align}
\label{e.var.H.sigma.psi}
\lefteqn{
\gammafun(\nicefrac d4 - \nicefrac \sigma 2)^{2} 
\var \bigl[ F_\sigma(\psi) \bigr] 
} \ &
\notag \\ &
=
\E \biggl[
\int_0^\infty\!
\int_0^\infty\! 
\int_{\Rd}\!
\int_{\Rd}\!
(st)^{-1+\frac12(\frac d2 -\sigma)}
( \Phi(t,\cdot) \ast \psi ) (x)
( \Phi(s,\cdot) \ast \psi ) (y)
W(x) W(y)  
\, dx \, dy \, dt \, ds
\biggr]
\notag \\ &
=
\int_0^\infty\!
\int_0^\infty\! 
\int_{\Rd}\!
(st)^{-1+\frac12(\frac d2 -\sigma)}
( \Phi(t,\cdot) \ast \psi ) (x)
( \Phi(s,\cdot) \ast \psi ) (x)
\, dx \, dt \, ds
\biggr]
\notag \\ &
=
\int_0^\infty 
\int_0^\infty 
\int_{\Rd} 
\int_{\Rd} 
(st)^{-1+\frac12(\frac d2 -\sigma)}
\Phi(t+s,x-y) 
\psi (x) \psi (y) 
\,dx \,dy \, dt \, ds 
\,.
\end{align}
In the above display, we used~\eqref{e.cov.W} to get the second equality and the semigroup property of the heat kernel to get the third equality.

\smallskip

To evaluate the expression on the last line of~\eqref{e.var.H.sigma.psi} side, we change variables by setting~$t = \frac{1}{4T}-s$ for a new variable~$T$, and then reverse the order of integration between the variables~$s$ and~$T$ then set~$s \coloneqq  S/4T$. After some computations, we get that the last line of~\eqref{e.var.H.sigma.psi} is equal to
\begin{equation}
\mbox{\eqref{e.var.H.sigma.psi}}
=
4^{s-d} \pi^{-\nicefrac d2} \!
\int_0^1 
(S-S^2)^{-1+\frac 12(\frac d2-\sigma)}
\,dS
\int_0^\infty\!\!
T^{\sigma -1} \exp ( -T|x-y|^2) \,dT\!
\int_{\Rd} \!
\int_{\Rd} \!
\psi(x) \psi(y)\,dx\,dy
\,.
\label{e.var.H.sigma.psi.cont}
\end{equation}
For the first integral factor, we have that 
\begin{equation}
\label{e.beta.identity}
\int_0^1
(S-S^2)^{-1 + \frac12(\frac d2-\sigma) } \,dS
= \frac{\gammafun( \frac d4-\frac\sigma2)^2}{\gammafun(\frac d2-\sigma)}
\,.
\end{equation}
Indeed, the integral on the right side is equal (by definition) to~$B(\frac12(\frac d2-\sigma),\frac12(\frac d2-\sigma))$, where~$B(\cdot,\cdot)$ is the \emph{beta function}. The beta function can be written in terms of the gamma function by the formula~$B(s_1,s_2) = \gammafun(s_1) \gammafun(s_2) / \gammafun(s_1+s_2)$, a proof of which can be found in~\cite[p.~18-19]{Artin} or in the Wikipedia on the beta function.
This yields~\eqref{e.beta.identity}.  
By a simple change of variables, we can relate the second integral factor on the right side of~\eqref{e.var.H.sigma.psi.cont} to the gamma function: we have
\begin{equation*}
\int_0^\infty\!\!
T^{\sigma -1} \exp ( -T|x-y|^2) \,dT
=
\gammafun (\sigma)|x-y|^{-2\sigma} 
\,.
\end{equation*}
We therefore obtain
\begin{equation*}
\gammafun(\nicefrac d4 - \nicefrac \sigma 2)^{2} 
\var \bigl[ F_\sigma(\psi) \bigr] 
=
4^{\sigma-\nicefrac d2} \pi^{-\nicefrac d2} 
\frac{\gammafun( \frac d4-\frac\sigma2)^2}{\gammafun(\frac d2-\sigma)}
\gammafun (\sigma)
\int_{\Rd}
\int_{\Rd}
|x-y|^{-2\sigma} 
\psi(x)\psi(y)
\,dx\,dy\,.
\end{equation*}
This completes the proof of~\eqref{e.check.the.var.sigma}.

\subsubsection{Finite range decomposition of fractional Gaussian fields}
\label{ss.FGF.decomp}

In this subsection, we provide an explicit decomposition of the fractional Gaussian field~$F_\sigma$ defined in~\eqref{e.F.sigma.def} of the form
\begin{equation*}
F_\sigma = \sum_{n\in\Z} F_{\sigma,n} \,,
\end{equation*}
where~$\{ F_{\sigma,n} \}_{n\in\Z}$ is a sequence of Gaussian random fields such that each~$F_{\sigma,n}$ which is defined pointwise, is locally smooth and has a range of dependence proportional to~$3^n$. 

\smallskip

We select a partition of unity~$\{ \eta_n\}_{n\in\Z}$ on~$\Rd \setminus \{0\}$ satisfying the following:
\begin{itemize}

\item
For every~$n\in\Z$, the function~$\eta_n$ belongs to~$C^\infty_c ( \Rd)$,
\begin{equation}
\label{e.eta.n.indc}
\indc_{B_{3^{n}} \setminus B_{3^{n-1}}}
\leq 
\eta_n 
\leq 
\indc_{B_{\frac 32\cdot 3^{n}} \setminus B_{\frac 23 \cdot 3^{n-1}}}
\end{equation}
and 
\begin{equation}
\label{e.eta.n.bounds}
3^n \| \nabla \eta_n \|_{L^\infty(\Rd)} 
+
3^{2n} \| \nabla^2 \eta_n \|_{L^\infty(\Rd)} 
\leq
100\,.
\end{equation}

\item 
For every $n\in\Z$ and~$x\in\Rd$, 
\begin{equation}
\label{e.eta.n.dilation}
\eta_n (x) = \eta_0(3^{-n} x)\,.
\end{equation}

\item For every~$x\in\Rd \setminus \{0\}$, 
\begin{equation}
\label{e.eta.n.partition}
\sum_{n\in\Z} \eta_n(x) = 1
\,.
\end{equation}
\end{itemize}
We can construct~$\{ \eta_n\}_{n\in\N}$ by taking the indicator function~$\indc_{B_{1} \setminus B_{\nicefrac13}}$ and mollifying the near the inner boundary~$\partial B_{\nicefrac13}$ with a smooth, radial function~$\zeta$ which is supported in~$B_{\nicefrac 19}$ and has unit mass, and then mollifying near the outer boundary~$\partial B_{1}$ with~$3^{-d} \zeta(3^{-1}\cdot)$. This defines~$\eta_0$, and we can then define~$\eta_n$ for~$n\in\Z\setminus\{0\}$ using the scaling relation~\eqref{e.eta.n.dilation}. 
The other properties of~$\eta_n$ are then immediate from the construction. 

\smallskip

With an eye toward~\eqref{e.F.sigma.def}, we define~$F_{\sigma,n}$ as follows:
\begin{equation*}
F_{\sigma,n}(x)  \coloneqq  
\frac1{\gammafun(\nicefrac d4 {-} \nicefrac \sigma 2)}
\int_0^\infty 
t^{-1+\frac12(\frac d2 -\sigma)}
\int_{\Rd} (\Phi(t,\cdot) \ast \eta_n(\cdot-x))(y) W(y)\,dy\,dt\,,
\quad x \in \Rd\,.
\end{equation*}
By~\eqref{e.xq.formula}, we have that 
\begin{equation}
\label{e.Hn.convo}
F_{\sigma,n}(x) = 2^\sigma \pi^{-\nicefrac d2} 
\int_{\Rd} 
\eta_n(y-x) |y-x|^{-(\frac d2 +\sigma)}
W(y) \,dy 
\,.
\end{equation}
It follows that~$F_{\sigma,n}$ is an~$\Rd$--stationary Gaussian field with zero mean, and 
\begin{align*}
\var \bigl[ F_{\sigma,n}(0) \bigr]
&
\leq 
4^{\sigma} \pi^{-d}
\int_{\Rd} 
\eta_n^2(x) |x|^{-(d+2\sigma)}\,dx
\notag \\ & 
\leq 
4^{\sigma}\pi^{-d}
\int_{B_{\frac 32\cdot 3^{n}} \setminus B_{\frac 23 \cdot 3^{n-1}}} 
|x|^{-(d+2\sigma)}\,dx
\notag \\ & 
=
4^{\sigma}\pi^{-d} 
| \partial B_1| 
\int_{\frac 23 \cdot 3^{n-1}}^{\frac 32\cdot 3^{n}}
r^{-(1+2\sigma)} 
\,dr
\leq
\frac{162^{\sigma}}{2\sigma} \pi^{-d} 
| \partial B_1| 
3^{-2n\sigma}
\,.
\end{align*}
Since~$\sigma < \nicefrac d2$, we obtain that, for a constant~$C(d)<\infty$, 
\begin{equation*}
\var \bigl[ F_{\sigma,n}(0) \bigr]
\leq 
C\sigma^{-1} 
3^{-2n\sigma}\,.
\end{equation*}
Since~$F_{\sigma,n}(0)$ is Gaussian, we have that 
\begin{equation}
\label{e.F.sigman.size}
\bigl| F_{\sigma,n}(0) \bigr|
=
\O_{\Gamma_2} \bigl( C\sigma^{-\nicefrac12} 
3^{-n\sigma} \bigr)\,.
\end{equation}
By a similar computation, using~\eqref{e.eta.n.bounds}, we also have that 
\begin{equation}
\label{e.nab.F.sigman.size}
\bigl| \nabla F_{\sigma,n}(0) \bigr|
=
\O_{\Gamma_2} \bigl( C\sigma^{-\nicefrac12} 
3^{-n(1+\sigma)} \bigr)\,.
\end{equation}
We observe next from~\eqref{e.eta.n.indc},~\eqref{e.Hn.convo} and the independence properties of white noise that
\begin{equation*}
\mbox{$F_{\sigma,n}$ has range of dependence at most~$\frac 32 \cdot 3^n$}
\end{equation*}
and
\begin{equation*}
\mbox{for every \ $m,n\in\N$ \  with \ $|m-n| \geq 2$, \ the fields  $F_{\sigma,n}$ and~$F_{\sigma,m}$ are independent.}
\end{equation*}
By~\eqref{e.eta.n.dilation} and the scaling invariance of white noise in~\eqref{e.W.scaling}, it is immediate that, for every~$n\in\Z$, the field~$F_{\sigma,n}$ has the same law as~$3^{-\sigma} F_{\sigma,0}(3^{-n}\cdot)$. 

\smallskip

Finally, by~\eqref{e.eta.n.partition}, we obtain that~$F_{\sigma}$ defined in~\eqref{e.F.sigma.def} satisfies
\begin{equation*}
F_{\sigma}(\psi) 
=
\sum_{n\in\Z} \int_{\Rd} F_{\sigma,n} (x) \psi(x)\,dx
\,, \qquad \psi \in \mathscr{S}(\Rd)\,.
\end{equation*}
What needs to be justified is that the sum on the right side is convergent, but this is straightforward to obtain from the bounds~\eqref{e.F.sigman.size} and~\eqref{e.nab.F.sigman.size}, above.

\subsubsection{{Proof of Proposition~\ref{p.example.advdiff}}}
\label{ss.example.advdiff}

We next present the proof of Proposition~\ref{p.example.advdiff}. In fact, we will obtain the following more general statement.\footnote{While Proposition~\ref{p.example.advdiff.gen} is mostly about checking the ellipticity assumption~\ref{a.ellipticity}, it also implies that~\ref{a.CFS} is satisfied with~$\beta = 1-\frac{2\sigma}{d}$ and~$\Psi(t)=\Gamma_2(c (\tfrac d2-\sigma) t)$, see~\cite[Chapter 3]{AK.Book}.}

\begin{proposition}
\label{p.example.advdiff.gen}
Consider the case in which 
\begin{equation*}
\a(x) = \lambda \Id 
+
\k(x) \,,
\end{equation*}
where the~$\k(\cdot)$ is a~$\Zd$--stationary random field valued in the~$d$-by-$d$ anti-symmetric matrices with real entries, and which admits the following decomposition: 
\begin{equation*}
\k(x) = \sum_{j=0}^\infty \k_j(x)\,,
\end{equation*}
where the sequence~$\{ \k_j \}_{j\in\N}$ satisfies the following:
\begin{itemize}

\item For each~$j\in\N$, the field~$\k_j$ is a~$\Zd$--stationary random field valued in the~$d$-by-$d$ anti-symmetric matrices;

\item For each~$j\in\N$, the range of dependence of~$\k_j$ is at most~$3^j$; 

\item There exists~$K_0 \in (0,\infty)$ and~$\sigma \in(0,\nicefrac d2)$ such that, for each~$j\in\N$,
\begin{equation}
\label{e.kj.bound.ass}
3^j \| \nabla \k_j \|_{L^\infty(\cu_j)}
+
\| \k_j \|_{L^\infty( \cu_j)} 
\leq 
\O_{\Gamma_2} 
\bigl( K_0 3^{-\sigma j} \bigr) 
\,.
\end{equation}
\end{itemize}
Then there exists~$C(d)<\infty$ such that, for every~$\gamma \in (0,\sigma \wedge 1)$, the ellipticity condition~\ref{a.ellipticity.dagger} is satisfied with the parameters
\begin{equation*}
\bfE = 
\begin{pmatrix} 
2(\lambda + C\lambda^{-1} K_0^2\sigma^{-2} )\Id
& 0
\\ 0
& 2\lambda^{-1} \Id
\end{pmatrix}
\quad \mbox{and} \quad
\Psi_\S(t)  = (\sigma-\gamma) \exp \bigl( C^{-1}t^\gamma -C \gamma^{-1} | \log \gamma| \bigr) 
\,.
\end{equation*}
\end{proposition}

In view of the discussion in Sections~\ref{ss.dagger} and~\ref{ss.FGF.decomp}, Proposition~\ref{p.example.advdiff.gen} implies Proposition~\ref{p.example.advdiff} since the assumptions of the latter imply those of the former with~$K_0 = C(d) \sigma^{-\nicefrac12}$ in~\eqref{e.kj.bound.ass}.

\begin{proof}[{Proof of Proposition~\ref{p.example.advdiff.gen}}]
Fix~$\gamma\in (0,1)$ and~$\rho \in (0,\alpha)$ with~$\gamma < 2\rho$. 
Let~$j,n,m\in\N$ with~$j\leq n \leq m$.
By~\eqref{e.kj.bound.ass} and Lemma~\ref{l.exp.concentration}, there exists a constant~$C(d)<\infty$ such that, for every~$z\in\Rd$,   
\begin{align*}
\| \k_j \|_{\underline{L}^2(z+\cu_n)}^2
&
=
\avsum_{z'\in z+3^j\Zd \cap \cu_n}
\| \k_j \|_{\underline{L}^2(z'+\cu_n)}^2
\\ \notag & 
=  
\E\bigl[ \| \k_j \|_{\underline{L}^2(\cu_n)}^2 \bigr]
+
\avsum_{z'\in z+3^j\Zd \cap \cu_n}
\bigl(
\| \k_j \|_{\underline{L}^2(z'+\cu_n)}^2
-
\E\bigl[ \| \k_j \|_{\underline{L}^2(\cu_n)}^2 \bigr]
\bigr)
\\ \notag &
\leq 
C K_0^2 3^{-2\sigma j} 
+
\O_{\Gamma_{1}}\bigl( C K_0^2 3^{-2\sigma j} 3^{-\frac d2(n-j)}
\bigr) \,.
\end{align*}
By the Cauchy-Schwarz inequality, we have that 
\begin{equation*}
\| \k_1+\ldots+\k_n \|_{\underline{L}^2(\cu_n)}^2
\leq 
\frac{C}{\sigma} 
\sum_{j=0}^n 
3^{\sigma j} 
\| \k_j \|_{\underline{L}^2(\cu_n)}^2\,.
\end{equation*}
Combining these, we obtain
\begin{align*}
\| \k_1+\ldots+\k_n \|_{\underline{L}^2(\cu_n)}^2
&
\leq 
\frac{CK_0^2}{\sigma} 
\sum_{j=0}^n
C 3^{-\sigma j} 
+
\O_{\Gamma_{1}}\biggl( 
\frac{CK_0^2}{\sigma} 
\sum_{j=0}^n
3^{-\sigma j} 3^{-\frac d2(n-j)}
\biggr) 
\notag \\ & 
\leq 
\frac{CK_0^2}{\sigma^2}
+
\O_{\Gamma_{1}}\biggl( 
\frac{CK_0^23^{-\sigma n}}{\sigma} 
\biggr)
\,.
\end{align*}
On the other hand, by~\eqref{e.kj.bound.ass} and the generalized triangle inequality, 
\begin{equation*}
\biggl\| \sum_{j=n+1}^\infty \k_j \biggr\|_{\underline{L}^2(\cu_n)}
\leq 
\sum_{j=n+1}^\infty \| \k_j \|_{L^\infty(\cu_{j})} 
\leq 
\O_{\Gamma_2} 
\biggl( \frac{CK_0}{\sigma} 3^{-n\sigma} \biggr) 
\,.
\end{equation*}
Combining the previous two displays yields, for every~$n\in\N$, 
\begin{equation*}
\| \k \|_{\underline{L}^2(\cu_n)}^2
\leq 
\frac{CK_0^2}{\sigma^2}
+
\O_{\Gamma_{1}}\biggl( 
\frac{CK_0^23^{-\sigma n}}{\sigma} 
\biggr)
\,.
\end{equation*}
For~$n\in\Z$ with~$n<0$, the assumption~\eqref{e.kj.bound.ass} already gives a better bound; combining this with the above estimate yields, for every~$n\in\Z$,
\begin{equation*}
\| \k \|_{\underline{L}^2(\cu_n)}^2
\leq 
\frac{CK_0^2}{\sigma^2}
+
\O_{\Gamma_{1}}\biggl( 
\frac{CK_0^23^{-\sigma (n\vee 0)}}{\sigma} 
\biggr)
\,.
\end{equation*}
By the Markov inequality and a union bound, we deduce, for every~$m,n\in\N$ with~$n\leq m$,
\begin{align*}
\P \biggl[
\max_{z\in 3^n\Zd \cap \cu_m} 
\| \k \|_{\underline{L}^2(z+\cu_n)}^2
>
\frac{CK_0^2}{\sigma^2} 3^{\gamma(m-n)}
\biggr]
&
\leq 
\sum_{z\in 3^n\Zd \cap \cu_m}
\P \biggl[
\| \k \|_{\underline{L}^2(z+\cu_n)}^2
>
\frac{CK_0^2}{\sigma^2} 3^{\gamma(m-n)}
\biggr]
\notag \\ & 
\leq 
3^{d(m-n)}
\P \biggl[
\| \k \|_{\underline{L}^2(\cu_n)}^2
>
\frac{CK_0^2}{\sigma^2} 3^{\gamma(m-n)}
\biggr]
\notag \\ & 
\leq 
3^{d(m-n)} \exp \Bigl( - c \sigma^{-1} 3^{\gamma(m-n) + \sigma (n \vee 0)} \Bigr)
\,. 
\end{align*}
By another union bound, we obtain, for every~$m\in\N$,
\begin{align*}
\lefteqn{
\P \biggl[ \exists n \in \Z \cap (-\infty,  m] , \ 
\max_{z\in 3^n\Zd \cap \cu_m} 
\| \k \|_{\underline{L}^2(z+\cu_n)}^2
>
\frac{CK_0^2}{\sigma^2} 3^{\gamma(m-n)}
\biggr]
} \qquad & 
\notag \\  &
\leq 
\sum_{n=-\infty}^m
 \P \biggl[
\max_{z\in 3^n\Zd \cap \cu_m} 
\| \k \|_{\underline{L}^2(z+\cu_n)}^2
>
\frac{CK_0^2}{\sigma^2} 3^{\gamma(m-n)}
\biggr]
\notag \\  &
\leq 
\sum_{n=-\infty}^m
3^{d(m-n)} \exp \Bigl( - c \sigma^{-1} 3^{\gamma(m-n) + \sigma (n \vee 0)} \Bigr)
\leq
\frac{1}{\sigma-\gamma}
\exp\biggl( 
\frac{C\left|\log \gamma\right|}{\gamma}
\biggr)
\exp\bigl( 
-c 3^{\gamma m}
\bigr)
\,.
\end{align*}
We next define
\begin{equation*}
\S  \coloneqq  \sup \biggl\{ 
3^{m+1} \,:\, m\in\N\,, \ 
\exists n \in \Z \cap (-\infty,  m] , \ 
\max_{z\in 3^n\Zd \cap \cu_m} 
\| \k \|_{\underline{L}^2(z+\cu_n)}^2
>
\frac{CK_0^2}{\sigma^2} 3^{\gamma(m-n)}
\biggr\} \,.
\end{equation*}
The previous display says that~$\S$ has the integrability claimed in the statement. Its definition implies that~\eqref{e.ellipticity.bfE} holds for~$\bfE$ given in the statement of the proposition since by the definition of~$\S$ we have 
\begin{align*}
3^m \geq \S 
& \implies
\max_{n\in\Z\cap (-\infty , m]}  
\max_{z\in 3^n\Zd \cap \cu_m} 
\| \k \|_{\underline{L}^2(z+\cu_n)}^2
\leq
\frac{CK_0^2}{\sigma^2} 3^{\gamma(m-n)}
\notag \\ &  \implies
\left\{
\begin{aligned}
&\bfA(z+\cu_n)
\leq 
\begin{pmatrix} 
\bigl(2\lambda + C\lambda^{-1}K_0^2\sigma^{-2} 3^{\gamma(m-n)}\bigr)\Id
& 0
\\ 0
& 2\lambda^{-1}\Id
\end{pmatrix}\,,
\\ 
& \forall 
n\in\Z\cap (-\infty , m]\,, \ z\in 3^n\Zd \cap \cu_m\,.
\end{aligned}
\right.
\end{align*}
This completes the proof. 
\end{proof}

\subsection{Log-normal fields}
\label{ss.log-normal}

In this section, we prove Proposition~\ref{p.example.log-normal}. We first consider the case of a log-normal random field with a finite range of dependence. Notice that we make no symmetry assumption on~$\g$, nor do we assume that the symmetric part of~$\g$ is nonnegative. (The symmetric part of~$\a(\cdot)$ will however be positive.)

\begin{proposition}[Log-normal field with finite range of dependence]
\label{p.log-normal.frd}
Suppose~$\a(\cdot)$ is given by
\begin{equation*}
\a(x) = \exp( \g(x)) \,, 
\end{equation*}
where~$L\geq 1$, $h>0$ and~$\g(\cdot)$ is a random field valued in the~$\R^{d\times d}$ matrices satisfying
\begin{equation}
\label{e.assump.g.FRD}
\left\{
\begin{aligned}
& 
\text{$\g$ is~$\Zd$--stationary,} \\ &
\text{$\g$ has range of dependence at most~$L$,} \\ & 
\text{$\| \g \|_{L^\infty(\cu_0)} \leq \O_{\Gamma_2} (h)$.}  
\end{aligned}
\right.
\end{equation}
Then there exists~$C(d)<\infty$ such that, for every~$\gamma \in (0,1)$, the field~$\a(\cdot)$ satisfies assumption~\ref{a.ellipticity.dagger} with parameters~$\gamma$, 
\begin{equation*}
\bfE = C \exp( 18 h^2 ) \Itwod ,
\end{equation*}
and minimal scale~$\S$ satisfying 
\begin{equation*}
\P \bigl[ \S > L t \bigr] 
\leq 
\exp \bigl( Ch^2 \gamma^{-2} - ch^{-2} \log^2 t \bigr) \,.
\end{equation*}
\end{proposition}

The assumption~\eqref{e.assump.g.FRD} is satisfied, for example, if~$\g(\cdot)$ is a~$\R^{d\times d}$-valued stationary Gaussian field with a compactly supported covariance function.
Another example is if~$\g(\cdot)$ is given by the convolution of a bounded, deterministic and compactly supported~$\R^{d\times d}$--valued function on~$\Rd$ and a Poisson point process on~$\Rd$. 
Note that, in both of these cases, the distributions of~$\| \a \|_{L^\infty(\cu_0)}$ and~$\| \a^{-1} \|_{L^\infty(\cu_0)}$ have tails which are as fat as those of a log-normal random variable. More generally, under the assumption~\eqref{e.assump.g.FRD}, we have
\begin{equation*}
\P \bigl[ \| \a \|_{L^\infty(\cu_0)} > t \bigr] 
\leq 
\exp \biggl( - \frac12 \Bigl( \frac{\log t}{h} \Bigr)^{\!2}  \biggr)
\,,\quad \forall t >0\,,
\end{equation*}
with the same bound holding also for~$\a^{-1}$ in place of~$\a$.

\smallskip

It is clear that the field~$\a(\cdot)$ satisfies the stationarity assumption~\ref{a.stationarity}. The validity of~\ref{a.CFS} with~$\beta = 0$ and~$\Psi = \Gamma_2(\cdot/CL)$  for some~$C(d)<\infty$ follows from the finite range of dependence assumption; see~\cite[Section 3.2.1]{AK.Book}.

\begin{proof}[{Proof of Proposition~\ref{p.log-normal.frd}}]
We will assume without loss of generality that~$L=1$.
Let us decompose~$\g$ into pieces by writing
\begin{equation*}
\g = \sum_{j=0}^\infty \g_j 
\,, \quad \mbox{where} \quad 
\g_0(x)  \coloneqq  \g(x) \indc_{ | \g(x)| <1 \} }
\quad \mbox{and} \quad
\g_j(x)  \coloneqq  \g(x) \indc_{ \{ 2^{j-1} \leq | \g(x)| <2^j \} }
\,, \quad \forall j \geq 1\,.
\end{equation*}
Note that for each~$x$, exactly one of the~$\g_j(x)=\g(x)$ and the rest are zero. 
For every~$m\in\N$ and~$\lambda\geq 1$,  
\begin{equation*}
\fint_{\cu_m}
|\a(x)|^{\lambda/h}\,dx
=
\fint_{\cu_m} 
\bigl| \exp( \lambda h^{-1}  \g(x) ) \bigr|
\,dx 
=
1 + \sum_{j=0}^\infty
\fint_{\cu_m} \bigl( |\exp( \lambda h^{-1}  \g_j(x) )| -1 \bigr) \,dx
\,.
\end{equation*}
Let~$N_{m,j}$ denote the number of distinct~$z \in \Zd \cap \cu_m$ such that~$\g_j$ does not vanish in~$z+\cu_0$. The distribution of~$N_{m,j}$ is essentially a binomial with parameters~$3^{dm}$ (the number of unit cubes) and~$p_j$, the probability of~$\g_j \not\equiv 0$ in~$\cu_0$, which satisfies the upper bound
\begin{equation*}
p_j = \P \bigl[ \g_j \not\equiv 0 \ \mbox{in} \ \cu_0 \bigr] 
\leq 
\P \bigl[ h^{-1} \| \g \|_{L^\infty(\cu_0)} \geq 2^{j-1} \bigr] 
\leq 
\exp \bigl( - 2^{2(j-1)} \bigr)\,.
\end{equation*}
This is not quite right since neighboring unit cubes are not independent; it would be more accurate to say that~$N_{m,j}$ is bounded by~$3^d$ many identically distributed copies of a binomial distribution with parameters~$3^{(d-1)m}$ and~$p_j$. To see this, we partition the collection of unit cubes in~$\cu_m$ into~$3^d$ different subcollections, each of which contains cubes that are separated by a distance of at least one.  
From standard tail estimates on the binomial distribution (or just apply Hoeffding's inequality), we have the bound
\begin{equation*}
\P \bigl[ N_{m,j} \geq 3^{md} \exp \bigl( - 2^{2(j-1)} \bigr) + 3^d t \bigr]  
\leq 
\exp( - 2\cdot 3^{-(d-1)m} t^2  ) 
\,.
\end{equation*}
That is, 
\begin{equation}
\label{e.Nmj.Hoeffding}
N_{m,j} \leq 3^{md} \exp \bigl( - 2^{2(j-1)} \bigr)
+
\O_{\Gamma_2} \bigl( C 3^{\frac d2 m}\bigr)
\,.
\end{equation}
Alternatively, when~$j$ is relatively large, it is better to use a crude estimate obtained from a simple union bound, 
\begin{equation}
\label{e.Nmj.do.or.die}
\P \bigl[ N_{m,j} \neq 0 \bigr] 
\leq 3^{md} \P \bigl[ h^{-1} \| \g \|_{L^\infty(\cu_0)} > 2^{j-1} \bigr] 
\leq 
3^{md} \exp \bigl( - 2^{2(j-1)} \bigr)\,.
\end{equation}
Applying~\eqref{e.Nmj.Hoeffding} yields the estimate
\begin{align*}
\fint_{\cu_m} \exp \bigl( \lambda h^{-1}  \g_j(x) \bigr) 
\,dx
&
\leq
1+
\frac{N_{m,j}}{|\cu_m|} \exp \bigl( \lambda 2^j \bigr)
\notag \\ & 
\leq
1+
\exp \bigl( \lambda 2^j - 2^{2(j-1)} \bigr)
+
\O_{\Gamma_2} \bigl( C 3^{-\frac d2 m} \exp \bigl( \lambda 2^j \bigr)
\bigr)
\,.
\end{align*}
This inequality will be used for small values of~$j$, namely those satisfying~$\exp(\lambda 2^j) \leq t 3^{\delta m}$ for parameters~$\delta\in (0,\tfrac12]$ and~$t\geq 1$, to be selected.
Summing over this range of~$j$'s, we get 
\begin{align*}
\lefteqn{
\sum_{j \in\N \,:\, \exp(\lambda 2^j) \leq t3^{\delta m}}
\fint_{\cu_m}
\bigl( |\exp( \lambda h^{-1} \g_j(x) )| - 1 \bigr)
\,dx
} \qquad & 
\notag \\ &
\leq
\sum_{j=0}^\infty\exp \bigl( \lambda 2^j - 2^{2(j-1)} \bigr)
+
\sum_{j \in\N\,:\, \exp(\lambda 2^j) \leq t3^{\delta m}}
\O_{\Gamma_2} \bigl( C 3^{-\frac d2 m} \exp \bigl( \lambda 2^j \bigr)
\bigr)
\notag \\ & 
\leq
\sum_{j=0}^\infty
\exp \bigl( 2\lambda^2 + 2^{2j-3} - 2^{2j-2} \bigr)
+
\O_{\Gamma_2} \bigl( Ct 3^{-(\frac d2-\delta) m} \bigr)
\notag \\ & 
\leq
C \exp ( 2\lambda^{2})
+
\O_{\Gamma_2} \bigl( Ct 3^{-(\frac d2-\delta) m} \bigr)
\,.
\end{align*}
For~$j$'s larger than this, we simply hope that~$N_{m,j} = 0$, or else we give up. 
Using~\eqref{e.Nmj.do.or.die}, we have  
\begin{align*}
\P \biggl[ \exists j \in \N\,, \ \exp(\lambda 2^j) > t3^{\delta m}\,, \
N_{m,j} \neq 0
\biggr]
&
\leq 
\sum_{j\in\N \,:\, \exp(\lambda 2^j) > t 3^{\delta m}}
3^{md} \exp \bigl( - 2^{2 j} \bigr)
\notag  \\ &
\leq
C3^{md} 
\exp \bigl( - \bigl( \log 3^{\frac{\delta m}{\lambda}} \bigr)^2  \bigr)
\notag \\ &
\leq
\exp \bigl( - \bigl( \lambda^{-1} \log ct3^{\delta m} \bigr)^2  \bigr)
\,.
\end{align*}
We deduce that 
\begin{equation*}
\P \biggl[
\fint_{\cu_m} 
\bigl|\exp( \lambda h^{-1} \g(x) ) \bigr|
\,dx 
\geq 
C \exp ( 2\lambda^{2})
+ t
\biggr]
\leq
\exp \bigl( -c 3^{(d-2\delta) m}t \bigr)
+
\exp \bigl( - \bigl( \lambda^{-1} \log ct3^{\delta m} \bigr)^2  \bigr)
\,.
\end{equation*}
Since~$t\geq 1$, the second term on the right is larger than the first. Taking~$\delta=\nicefrac12$, we obtain, for constants~$C(d)<\infty$ and~$c(d)>0$, 
\begin{equation*}
\P \biggl[
\fint_{\cu_m} 
\bigl|\exp( \lambda h^{-1} \g(x) )\bigr|
\,dx 
\geq 
C \exp ( 2\lambda^{2})
+ t
\biggr]
\leq
C\exp \bigl( - c\lambda^{-2}\bigl( \log t3^{m} \bigr)^2  \bigr)
\,.
\end{equation*}
Fix now a parameter~$\gamma \in (0,1]$. 
By performing a union bound over a mesoscale represented by~$n\in\N$, $n\leq m$, we obtain
\begin{align*}
\lefteqn{
\P \biggl[\sup_{z\in 3^n\Zd\cap\cu_m}\fint_{z+\cu_n} \bigl|\exp( \lambda h^{-1} \g(x) )\bigr|
\,dx \geq  C3^{\gamma(m-n)} \exp ( 2\lambda^{2})  \biggr]
} \qquad\qquad &
\notag \\ & 
\leq
3^{d(m-n)}
\P \biggl[
\fint_{\cu_n} 
\bigl|\exp( \lambda h^{-1} \g(x) )\bigr|
\,dx 
\geq 
C 3^{\gamma(m-n)}\exp ( 4\lambda^{2})
\biggr]
\notag \\ &
\leq
C3^{d(m-n)}\exp \bigl( - c\lambda^{-2}\bigl( \log 3^{\gamma(m-n)} + \log 3^{m} \bigr)^2  \bigr)
\notag \\ &
\leq 
C3^{d(m-n)}\exp \bigl( - c\lambda^{-2}\bigl( \log 3^{\gamma(m-n)} \bigr)^2 \bigr) 
\exp( - c\lambda^{-2} \bigl( \log 3^{m} \bigr)^2  \bigr)
\,.
\end{align*}
Taking another union bound, we obtain
\begin{align*}
\lefteqn{
\P \biggl[ \sup_{n \in \{0,\ldots,m\} }\sup_{z\in 3^n\Zd\cap\cu_m}\fint_{z+\cu_n} \bigl|\exp( \lambda h^{-1} \g(x) )\bigr|
\,dx \geq  C3^{\gamma(m-n)} \exp ( 2\lambda^{2})  \biggr]
} \qquad\qquad\qquad &
\notag \\ &
\leq 
C\sum_{n=0}^m 
\biggl( 
3^{d(m-n)}\exp \bigl( - c\lambda^{-2}\bigl( \log 3^{\gamma(m-n)} \bigr)^2 \bigr) 
\biggr)
\exp( - c\lambda^{-2} \bigl( \log 3^{m} \bigr)^2  \bigr)
\notag \\ &
\leq 
C \exp \bigl( C(\lambda \gamma^{-1})^{2}\bigr)
\exp( - c\lambda^{-2} ( \log 3^{m} )^2  \bigr)
\,.
\end{align*}
Taking yet another union bound, we get, for every~$k\in\N$, 
\begin{align}
\label{e.kill.S}
\lefteqn{
\P \biggl[ \sup_{ m \geq k } \sup_{n \in \{0,\ldots,m\} }3^{-\gamma(m-n)} \sup_{z\in 3^n\Zd\cap\cu_m}\fint_{z+\cu_n} \bigl|\exp( \lambda h^{-1}  \g(x) )\bigr|
\,dx \geq  C\exp ( 2\lambda^{2})  \biggr]
} \qquad\qquad\qquad\qquad\qquad\qquad\qquad\qquad &
\notag \\ &
\leq 
\exp \bigl( C (\lambda \gamma^{-1})^{2}\bigr)\sum_{m=k}^\infty
\exp( - c\lambda^{-2} \bigl( \log 3^{m} \bigr)^2  \bigr)
\notag \\ &
\leq
\exp \bigl( C(\lambda \gamma^{-1})^{2}\bigr)
\exp( - c\lambda^{-2} \bigl( \log 3^{k} \bigr)^2  \bigr)
\,.
\end{align}
For the cubes smaller than the unit cubes, we use the bound 
\begin{align*}
\P \bigl[ \| \exp ( \lambda h^{-1}  \g) \|_{L^\infty(\cu_m)} >1+ t \bigr] 
&
\leq 
3^{dm} \P \bigl[ \| \exp ( \lambda h^{-1}  \g) \|_{L^\infty(\cu_0)} > 1+t \bigr] 
\notag \\ & 
\leq 
3^{dm} \exp ( - \lambda^{-2} \log^2 (1+  t ))\,.
\end{align*}
From this we obtain
\begin{align}
\label{e.kill.S.smallsuff}
\lefteqn{
\P \biggl[ \sup_{ m \geq k }\sup_{n \in -\N } 3^{-\gamma(m-n)}  \sup_{z\in 3^n\Zd\cap\cu_m}\fint_{z+\cu_n} \bigl|\exp( \lambda h^{-1} \g(x) )\bigr|
\,dx \geq 
C \biggr]
} \qquad\qquad\qquad\qquad\qquad\qquad\qquad\qquad &
\notag \\ &
\leq
\P \biggl[ \sup_{ m \geq k } 3^{-\gamma m}  \bigl\|\exp( \lambda h^{-1} \g )\bigr\|_{L^\infty(\cu_m)} \geq 
C \biggr]
\notag \\ &
\leq 
\exp \bigl( C \lambda^{2}\bigr)\sum_{m=k}^\infty
\exp( - c\lambda^{-2} \bigl( \log 3^{\gamma m} \bigr)^2  \bigr)
\notag \\ &
\leq
\exp \bigl( C(\lambda \gamma^{-1})^{2}\bigr)
\exp( - c\lambda^{-2} ( \log 3^{k} )^2  \bigr)\,,
\end{align}
as above. 
If we define the minimal scale~$\S$ by
\begin{equation}
\label{e.Slambda.def}
\S_\lambda \coloneqq  \sup \biggl\{ 3^m \,:\, 
\sup_{n \in \Z \cap (-\infty,m] }\sup_{z\in 3^n\Zd\cap\cu_m}\fint_{z+\cu_n} \bigl|\exp(\lambda h^{-1} \g(x) )\bigr|
\,dx \geq C \exp ( 2\lambda^{2} )
 3^{\gamma(m-n)}  \biggr\}\,,
\end{equation}
then~\eqref{e.kill.S} and~\eqref{e.kill.S.smallsuff} imply that 
\begin{equation}
\label{e.S.integrability.FRD}
\P \bigl[ \S_\lambda > 3^k \bigr]
\leq 
\exp \bigl( C(\lambda \gamma^{-1})^{2}\bigr)
\exp( - c\lambda^{-2} \bigl( \log 3^{k} \bigr)^2  \bigr)
\,,
\end{equation}
and hence 
\begin{equation*}
k \geq C\lambda^{2} \gamma^{-2}
\implies
\P \bigl[ \S_\lambda > 3^k \bigr]
\leq 
\exp( - c\lambda^{-2} \bigl( \log 3^{k} \bigr)^2  \bigr)
\,.
\end{equation*}
In other words, 
\begin{equation*}
\S_\lambda 
=
\exp ( C \lambda^{2} \gamma^{-2} ) 
+ 
\O_{\Psi_{C\lambda}}(C\lambda)\,.
\end{equation*}
It is clear from its definition in~\eqref{e.Slambda.def} that  
\begin{equation*}
3^m \geq \S_\lambda \implies
\fint_{z+\cu_n}
|\a(x)|^{\lambda/h} \,dx
\leq 
C3^{\gamma(m-n)} \exp ( 2\lambda^2 )\,,
\qquad \forall n\in \Z \cap (-\infty,m] \,, \ z \in 3^n\Zd\cap \cu_m
\,.
\end{equation*}
The same argument gives a similar minimal scale for~$\a^{-1}$ in place of~$\a$, so by taking the maximum of these, we may suppose that~$\S$ satisfies~\eqref{e.S.integrability.FRD} and 
\begin{equation}
\label{e.S.for.cubes}
3^m \geq \S_\lambda \implies
\left\{
\begin{aligned}
& \fint_{z+\cu_n}
\bigl( |\a(x)|^{\lambda/h} + |\a^{-1}(x)|^{\lambda/h} \bigr) \,dx
\leq 
C3^{\gamma(m-n)} \exp ( 2\lambda^{2} )
\,,
\\ &
\qquad\qquad\qquad\qquad\qquad
\forall n\in \Z \cap (-\infty,m] \,, \ z \in 3^n\Zd\cap \cu_m
\,.
\end{aligned}
\right.
\end{equation}
Since~$|\bfA(x)| \leq (1+|\a(x)|^2) |\a^{-1}(x)|$, we deduce from~\eqref{e.S.for.cubes} and H\"older's inequality that 
\begin{align*}
3^m \geq \S_{3h} 
&
\implies
\fint_{z+\cu_n}
|\bfA(x)| \,dx
\leq 
C 3^{\gamma(m-n)}\exp \bigl(  18h^2 \bigr)
\,,
\qquad \forall n\in \Z \cap (-\infty,m] \,, \ z \in 3^n\Zd\cap \cu_m
\notag \\ &
\implies 
\bfA(z+\cu_n)
\leq 
C 3^{\gamma(m-n)}\exp \bigl( 18h^2 \bigr)
\Itwod\,,
\qquad \forall n\in \Z \cap (-\infty,m] \,, \ z \in 3^n\Zd\cap \cu_m
\,.
\end{align*}
Note that scales below the unit scale can be taken care of immediately from the third line of~\eqref{e.assump.g.FRD} and a union bound. 
This completes the proof.
\end{proof}

We turn to the proof of Proposition~\ref{p.example.log-normal}. We will prove the following more general statement.\footnote{While Proposition~\ref{p.log-normal.gen} is mostly about checking the ellipticity assumption~\ref{a.ellipticity.dagger}, it also implies that~\ref{a.CFS} is satisfied with~$\beta = 1-\frac{2\sigma}{d}$ and~$\Psi(t)=\Gamma_2(c (\tfrac d2-\sigma) t)$, see~\cite[Chapter 3]{AK.Book}.}

\begin{proposition}
\label{p.log-normal.gen}
Suppose that~$\a(\cdot)$ is given by~$\a(x) = \exp( \g(x))$ where~$\g(\cdot)$ is an~$\R^{d\times d}$--valued random field which admits the decomposition
\begin{equation*}
\g(x) = \sum_{j=0}^\infty \g_j(x)\,,
\end{equation*}
where the sequence~$\{ \g_j \}_{j\in\N}$ satisfies the following:
\begin{itemize}

\item For each~$j\in\N$, the field~$\g_j$ is a~$\Zd$--stationary random field valued in the~$d$-by-$d$ matrices;

\item For each~$j\in\N$, the range of dependence of~$\g_j$ is at most~$3^j$; 

\item There exists~$h \in (0,\infty)$ and~$\sigma \in(0,\nicefrac d2)$ such that, for each~$j\in\N$,
\begin{equation*}
3^j \| \nabla \g_j \|_{L^\infty(\cu_j)}
+
\| \g_j \|_{L^\infty( \cu_j)} 
\leq 
\O_{\Gamma_2} 
\bigl( h 3^{-\sigma j} \bigr) 
\,.
\end{equation*}
\end{itemize}
Then there exists~$C(d)<\infty$ such that, for every~$\gamma \in (0,1)$, the field~$\a(\cdot)$ satisfies assumption~\ref{a.ellipticity.dagger} with parameters~$\gamma$, 
\begin{equation*}
\bfE = \exp\bigl( C h^2  \sigma^{-2} \bigr) \Itwod \,,
\end{equation*}
and minimal scale~$\S$ satisfying 
\begin{equation*}
\P \bigl[ \S > t \bigr] 
\leq 
\exp \bigl( Ch^2 \sigma^{-2} \gamma^{-2} - ch^{-2} \sigma^{2} \log^2 t  \bigr) \,.
\end{equation*}
\end{proposition}
\begin{proof}
Here, we approximate by finite range fields and apply the previous result. Denote, for each~$k\in\N$, 
\begin{equation*}
\hat{\g}_k \coloneqq \sum_{j=0}^k \g_j 
\end{equation*}
and observe that, for each~$k,m\in\N$ with~$k\leq m$,
\begin{equation*}
\| \g - \hat{\g}_k \|_{L^\infty(\cu_k)}
\leq \O_{\Gamma_2} \bigl( Ch \sigma^{-1} 3^{-\sigma k} \bigr)
\end{equation*}
and hence 
\begin{equation*}
\P \bigl[ \| \g - \hat{\g}_k \|_{L^\infty(\cu_m)} > t \bigr] 
\leq 
3^{d(m-k)} 
\exp \bigl( - ch^{-2} \sigma^2 3^{2\sigma k} t^2 \bigr) \,.
\end{equation*}
Taking~$k = \lceil \nicefrac m2 \rceil$ yields, for every~$t \geq C h\sigma^{-2}$, 
\begin{equation}
\label{e.highrese}
\P \bigl[ \| \g - \hat{\g}_{\lceil \nicefrac m2 \rceil}  \|_{L^\infty(\cu_m)} > t \bigr] 
\leq 
\exp \bigl( - ch^{-2} \sigma^2 3^{\sigma m} t^2 \bigr) \,.
\end{equation}
Applying the above result for finite range fields to~$\hat{\a}_{\lceil \nicefrac m2 \rceil} \coloneqq  \exp( h \hat{\g}_{\lceil \nicefrac m2 \rceil})$, using that 
\begin{equation*}
\hat{\g}_{\lceil \nicefrac m2 \rceil} = \O_{\Gamma_2} (Ch\sigma^{-1}) 
\end{equation*}
gives, for every~$m\geq C h^2 \sigma^{-2} \gamma^{-2}$,  
\begin{multline}
\label{e.lowrese}
\P \biggl[
\exists n\in \Z \cap (-\infty,m] \,, \ z \in 3^n\Zd\cap \cu_m\,, \ 
\fint_{z+\cu_n}
|\hat\bfA_{\lceil \nicefrac m2 \rceil}(x)| \,dx
\geq 
3^{\gamma(m-n)}\exp \bigl(  Ch^2 \sigma^{-2} h^2 \bigr)
\biggr] 
\\
\leq 
\exp\bigl( - c h^{-2} \sigma^{2} \bigl( \log 3^{m} \bigr)^2  \bigr)
\,.
\end{multline}
Combining~\eqref{e.highrese} and~\eqref{e.lowrese} yields, for every~$m\in\N$ with~$m\geq C h^2 \sigma^{-2} \gamma^{-2}$, 
\begin{multline*}
\P \biggl[
\exists n\in \Z \cap (-\infty,m] \,, \ z \in 3^n\Zd\cap \cu_m\,, \ 
\fint_{z+\cu_n}
|\bfA (x)| \,dx
\geq 
3^{\gamma(m-n)}\exp \bigl(  Ch^2  \sigma^{-2} \bigr)
\biggr] 
\\
\leq 
\exp\bigl( - c h^{-2} \sigma^{2} \bigl( \log 3^{m} \bigr)^2  \bigr)
\,.
\end{multline*}
This completes the proof. 
\end{proof}

\subsubsection*{\bf Acknowledgments}
S.A. was supported by NSF grants DMS-1954357 and DMS-2000200 and by the Simons Programme at IHES.
T.K. was supported by the Academy of Finland and by the European Research Council (ERC) under the European Union's Horizon 2020 research and innovation programme (grant agreement No 818437). We thank Ahmed Bou-Rabee, Aidan Lau and Heikki Lohi for carefully reading a previous version of this manuscript and correcting some typos and imprecisions.

{\small
\bibliographystyle{alpha}
\bibliography{highcontrast}

\newcommand{\noop}[1]{} \def\cprime{$'$}
\begin{thebibliography}{BMdlP23}

\bibitem[ABRK24]{ABK.SD}
S.~Armstrong, A.~Bou-Rabee, and T.~Kuusi.
\newblock Superdiffusive central limit theorem for a {B}rownian particle in a
  critically-correlated incompressible random drift, 2024.
\newblock arXiv:2404.01115.

\bibitem[AD18]{AD}
S.~Armstrong and P.~Dario.
\newblock Elliptic regularity and quantitative homogenization on percolation
  clusters.
\newblock {\em Comm. Pure Appl. Math.}, 71(9):1717--1849, 2018.

\bibitem[AK24]{AK.Book}
S.~Armstrong and T.~Kuusi.
\newblock Elliptic homogenization from qualitative to quantitative, 2024.
\newblock arXiv:2210.06488v2.

\bibitem[AKM19]{AKMbook}
S.~Armstrong, T.~Kuusi, and J.-C. Mourrat.
\newblock {\em Quantitative stochastic homogenization and large-scale
  regularity}, volume 352 of {\em Grundlehren der mathematischen Wissenschaften
  [Fundamental Principles of Mathematical Sciences]}.
\newblock Springer, Cham, 2019.

\bibitem[AM16]{AM}
S.~N. Armstrong and J.-C. Mourrat.
\newblock Lipschitz regularity for elliptic equations with random coefficients.
\newblock {\em Arch. Ration. Mech. Anal.}, 219(1):255--348, 2016.

\bibitem[And78]{Ando}
T.~Ando.
\newblock {\em Topics on operator inequalities}.
\newblock Hokkaido University, Research Institute of Applied Electricity,
  Division of Applied Mathematics, Sapporo, 1978.

\bibitem[Art64]{Artin}
E.~Artin.
\newblock {\em The gamma function}.
\newblock Holt, Rinehart and Winston, New York-Toronto-London, 1964.

\bibitem[AS16]{AS}
S.~N. Armstrong and C.~K. Smart.
\newblock Quantitative stochastic homogenization of convex integral
  functionals.
\newblock {\em Ann. Sci. \'Ec. Norm. Sup\'er. (4)}, 49(2):423--481, 2016.

\bibitem[Aya23]{Aya}
T.~Aya.
\newblock Quantitative stochastic homogenization of elliptic equations with
  unbounded coefficients, 2023.
\newblock arXiv:2302.00822.

\bibitem[BF91]{Brezzi}
F.~Brezzi and M.~Fortin.
\newblock {\em Mixed and hybrid finite element methods}, volume~15 of {\em
  Springer Ser. Comput. Math.}
\newblock New York etc.: Springer-Verlag, 1991.

\bibitem[BK24]{BellaK}
P.~Bella and M.~Kniely.
\newblock Regularity of random elliptic operators with degenerate coefficients
  and applications to stochastic homogenization.
\newblock {\em Stoch. Partial Differ. Equ. Anal. Comput.}, 2024.
\newblock (in press).

\bibitem[BMdlP23]{BMP}
M.~Bakhshizadeh, A.~Maleki, and V.~H. de~la Pena.
\newblock Sharp concentration results for heavy-tailed distributions.
\newblock {\em Inf. Inference}, 12(3):Paper No. iaad011, 31, 2023.

\bibitem[CGQ24]{CGQ}
N.~Clozeau, A.~Gloria, and S.~Qi.
\newblock Quantitative homogenization for log-normal coefficients, 2024.
\newblock arXiv:2403.00168.

\bibitem[Dar21]{D}
P.~Dario.
\newblock Optimal corrector estimates on percolation cluster.
\newblock {\em Ann. Appl. Probab.}, 31(1):377--431, 2021.

\bibitem[DCKT21]{DKT}
H.~Duminil-Copin, G.~Kozma, and V.~Tassion.
\newblock Upper bounds on the percolation correlation length.
\newblock In {\em In and out of equilibrium 3. {C}elebrating {V}ladas
  {S}idoravicius}, volume~77 of {\em Progr. Probab.}, pages 347--369.
  Birkh\"{a}user/Springer, Cham, [2021] \copyright 2021.

\bibitem[DG21]{DG}
P.~Dario and C.~Gu.
\newblock Quantitative homogenization of the parabolic and elliptic {G}reen's
  functions on percolation clusters.
\newblock {\em Ann. Probab.}, 49(2):556--636, 2021.

\bibitem[DG22]{DGl}
M.~Duerinckx and A.~Gloria.
\newblock Quantitative homogenization theory for random suspensions in steady
  {S}tokes flow.
\newblock {\em J. \'{E}c. polytech. Math.}, 9:1183--1244, 2022.

\bibitem[DLV18]{DLV}
B.~Dyda, J.~Lehrb\"ack, and A.V. V\"ah\"akangas.
\newblock Fractional {H}ardy-{S}obolev type inequalities for half spaces and
  {J}ohn domains.
\newblock {\em Proc. Amer. Math. Soc.}, 146(8):3393--3402, 2018.

\bibitem[FP97]{FP}
M.~Fiedler and V.~Pt\'{a}k.
\newblock A new positive definite geometric mean of two positive definite
  matrices.
\newblock {\em Linear Algebra Appl.}, 251:1--20, 1997.

\bibitem[GNO20]{GNO.reg}
A.~Gloria, S.~Neukamm, and F.~Otto.
\newblock A regularity theory for random elliptic operators.
\newblock {\em Milan J. Math.}, 88(1):99--170, 2020.

\bibitem[GO11]{GO1}
A.~Gloria and F.~Otto.
\newblock An optimal variance estimate in stochastic homogenization of discrete
  elliptic equations.
\newblock {\em Ann. Probab.}, 39(3):779--856, 2011.

\bibitem[GO12]{GO2}
A.~Gloria and F.~Otto.
\newblock An optimal error estimate in stochastic homogenization of discrete
  elliptic equations.
\newblock {\em Ann. Appl. Probab.}, 22(1):1--28, 2012.

\bibitem[Har90]{Hara}
T.~Hara.
\newblock Mean-field critical behaviour for correlation length for percolation
  in high dimensions.
\newblock {\em Probab. Theory Related Fields}, 86(3):337--385, 1990.

\bibitem[JK95]{JeKe}
D.~Jerison and C.~E. Kenig.
\newblock The inhomogeneous {D}irichlet problem in {L}ipschitz domains.
\newblock {\em J. Funct. Anal.}, 130(1):161--219, 1995.

\bibitem[KO84]{KO84}
A.~Kufner and B.~Opic.
\newblock How to define reasonably weighted {S}obolev spaces.
\newblock {\em Comment. Math. Univ. Carolin.}, 25(3):537--554, 1984.

\bibitem[LSSW16]{LSSW}
A.~Lodhia, S.~Sheffield, X.~Sun, and S.~S. Watson.
\newblock {Fractional Gaussian fields: A survey}.
\newblock {\em Probab. Surv.}, 13:1 -- 56, 2016.

\bibitem[Sho97]{Showalter}
R.~E. Showalter.
\newblock {\em Monotone operators in {B}anach space and nonlinear partial
  differential equations}, volume~49 of {\em Mathematical Surveys and
  Monographs}.
\newblock American Mathematical Society, Providence, RI, 1997.

\bibitem[SW01]{SW}
S.~Smirnov and W.~Werner.
\newblock Critical exponents for two-dimensional percolation.
\newblock {\em Math. Res. Lett.}, 8(5-6):729--744, 2001.

\bibitem[SZ90]{ScottZhang}
L.R. Scott and S.~Zhang.
\newblock Finite element interpolation of nonsmooth functions satisfying
  boundary conditions.
\newblock {\em Math. Comput.}, 54(190):483--493, 1990.

\bibitem[Wei83]{Weinberg}
S.~Weinberg.
\newblock Why the renormalization group is a good thing.
\newblock In {\em Asymptotic Realms of Physics}, pages 1--19. MIT Press, 1983.

\end{thebibliography}
}

\end{document}